\documentclass[a4paper,twoside,12pt]{book}
\usepackage[english]{babel}

\usepackage{amsmath,amsfonts,amssymb,amsthm}
\newtheorem{teo}{Theorem}
\newtheorem{lem}{Lemma}
\newtheorem{deff}{Definition}
\newtheorem{remark}{Remark}

\newtheorem{ex}{Example}

\newtheorem{corol}{Corollary}

 \title{{\bf    Theory of   sectorial  operators  and its application    in \\Fractional calculus      }}

\author{Maksim \,V.~Kukushkin   \\ \\
 \small  \textit{National Research University Higher School of Economics, 101000,  Moscow, Russia}\\
\small\textit{Russian Academy of Sciences,  Kabardino-Balkarian Scientific Center,}\\
 \small\textit{Institute of Applied Mathematics and Automation, 360000,  Nalchik, Russia}\\
 }

\date{}
\begin{document}

\maketitle
\tableofcontents

\section{Preliminaries}

Let    $ C,C_{i} ,\;i\in \mathbb{N}_{0}$ be     positive constants. We   assume   that  a  value of $C$    can be different in   various formulas and parts of formulas  but   values of $C_{i} $ are  certain. Denote by $ \mathrm{Fr}\,M$   the set of boundary points of the set $M.$    Everywhere further, if the contrary is not stated, we consider   linear    densely defined operators acting on a separable complex  Hilbert space $\mathfrak{H}$. Denote by $ \mathcal{B} (\mathfrak{H})$    the set of linear bounded operators   on    $\mathfrak{H}.$  Denote by
    $\tilde{L}$   the  closure of an  operator $L.$ We establish the following agreement on using  symbols $\tilde{L}^{i}:= (\tilde{L})^{i},$ where $i$ is an arbitrary symbol.  Denote by    $    \mathrm{D}   (L),\,   \mathrm{R}   (L),\,\mathrm{N}(L)$      the  {\it domain of definition}, the {\it range},  and the {\it kernel} or {\it null space}  of an  operator $L,$ respectively. The deficiency (codimension) of $\mathrm{R}(L),$ dimension of $\mathrm{N}(L)$ are denoted by $\mathrm{def}\, L,\;\mathrm{nul}\,L$ respectively. Assume that $L$ is a closed   operator acting on $\mathfrak{H},\,\mathrm{N}(L)=0,$  let us define a Hilbert space
$
 \mathfrak{H}_{L}:= \big \{f,g\in \mathrm{D}(L),\,(f,g)_{ \mathfrak{H}_{L}}=(Lf,Lg)_{\mathfrak{H} } \big\}.
$
Consider a pair of complex Hilbert spaces $\mathfrak{H},\mathfrak{H}_{+},$ the notation
$
\mathfrak{H}_{+}\subset\subset\mathfrak{ H}
$
   means that $\mathfrak{H}_{+}$ is dense in $\mathfrak{H}$ as a set of    elements and we have a bounded embedding provided by the inequality
$
\|f\|_{\mathfrak{H}}\leq C_{0}\|f\|_{\mathfrak{H}_{+}},\,C_{0}>0,\;f\in \mathfrak{H}_{+},
$
moreover   any  bounded  set with respect to the norm $\mathfrak{H}_{+}$ is compact with respect to the norm $\mathfrak{H}.$
  Let $L$ be a closed operator, for any closable operator $S$ such that
$\tilde{S} = L,$ its domain $\mathrm{D} (S)$ will be called a core of $L.$ Denote by $\mathrm{D}_{0}(L)$ a core of a closeable operator $L.$ Let    $\mathrm{P}(L)$ be  the resolvent set of an operator $L$ and
     $ R_{L}(\zeta),\,\zeta\in \mathrm{P}(L),\,[R_{L} :=R_{L}(0)]$ denotes      the resolvent of an  operator $L.$ Denote by $\lambda_{i}(L),\,i\in \mathbb{N} $ the eigenvalues of an operator $L.$
 Suppose $L$ is  a compact operator and  $N:=(L^{\ast}L)^{1/2},\,r(N):={\rm dim}\,  \mathrm{R}  (N);$ then   the eigenvalues of the operator $N$ are called   the {\it singular  numbers} ({\it s-numbers}) of the operator $L$ and are denoted by $s_{i}(L),\,i=1,\,2,...\,,r(N).$ If $r(N)<\infty,$ then we put by definition     $s_{i}=0,\,i=r(N)+1,2,...\,.$
 Let  $\nu(L)$ denotes   the sum of all algebraic multiplicities of an  operator $L.$ Denote by $n(r)$ a function equals to the quantity of the elements of the sequence $\{a_{n}\}_{1}^{\infty},\,|a_{n}|\uparrow\infty$ within the circle $|z|<r.$ Let $L$ be a compact operator, denote by  $n_{L}(r)$ or   $n(r)$ {\it counting function}   a function  corresponding to the sequence  $\{s^{-1}_{i}(L)\}_{1}^{\infty},$  in some cases we will also  use this notation for the counting function of the absolute values of the operator characteristic numbers.
 Let $\mathfrak{S}_{\sigma}(\mathfrak{H}),\,0<\sigma<\infty$ be a Schatten-von Neumann  class (Schatten class) and $\mathfrak{S}_{\infty}(\mathfrak{H})$ be a set of compact operators, by definition put
 $$
 \mathfrak{S}_{\sigma}(\mathfrak{H}):=\left\{L:\mathfrak{H}\rightarrow \mathfrak{H},\;\sum\limits_{n=1}^{\infty}s^{\sigma}_{n}(L)<\infty,\,0<\sigma<\infty\right\}.
 $$
   Denote by $  \mathfrak{S}^{\star}_{\sigma}(\mathfrak{H}),\,0\leq\sigma<\infty$ the class of the operators such that
$
 B\in \mathfrak{S}^{\star}_{\sigma}(\mathfrak{H}) \Rightarrow\{B\in\mathfrak{S}_{\sigma+\varepsilon},\,B \overline{\in} \,\mathfrak{S}_{\sigma},\,\forall\varepsilon>0 \}.
$
In accordance with \cite{firstab_lit:1kukushkin2021}, we will call it   {\it Schatten-von Neumann    class of the convergence exponent}.
Suppose  $L$ is  an   operator with a compact resolvent and
$s_{n}(R_{L})\leq   C \,n^{-\mu},\,n\in \mathbb{N},\,0\leq\mu< \infty;$ then
 we
 denote by  $\mu(L) $   order of the     operator $L$  (see \cite{firstab_lit:Shkalikov A.}).
 Denote by  $ \mathfrak{Re} L  := \left(L+L^{*}\right)/2,\, \mathfrak{Im} L  := \left(L-L^{*}\right)/2 i$
  the  real  and   imaginary Hermitian  components    of an  operator $L$  respectively.
In accordance with  the terminology of the monograph  \cite{firstab_lit:kato1980}, the set $\Theta(L):=\{z\in \mathbb{C}: z=(Lf,f)_{\mathfrak{H}},\,f\in  \mathrm{D} (L),\,\|f\|_{\mathfrak{H}}=1\}$ is called the  {\it numerical range}  of an   operator $L.$
  An  operator $L$ is called    {\it sectorial}    if its  numerical range   belongs to a  closed
sector     $\mathfrak{ L}_{\iota}(\theta):=\{\zeta:\,|\arg(\zeta-\iota)|\leq\theta<\pi/2\} ,$ where      $\iota$ is the vertex   and  $ \theta$ is the semi-angle of the sector   $\mathfrak{ L}_{\iota}(\theta).$ If we want to stress the  correspondence  between $\iota$ and $\theta,$  then   we will write $\theta_{\iota}.$
 An operator $L$ is called  {\it bounded from below}   if the following relation  holds  $\mathrm{Re}(Lf,f)_{\mathfrak{H}}\geq \gamma_{L}\|f\|^{2}_{\mathfrak{H}},\,f\in  \mathrm{D} (L),\,\gamma_{L}\in \mathbb{R},$  where $\gamma_{L}$ is called a lower bound of $L.$ An operator $L$ is called  {\it   accretive}   if  $\gamma_{L}=0.$
 An operator $L$ is called  {\it strictly  accretive}   if  $\gamma_{L}>0.$      An  operator $L$ is called    {\it m-accretive}     if the next relation  holds $(A+\zeta)^{-1}\in \mathcal{B}(\mathfrak{H}),\,\|(A+\zeta)^{-1}\| \leq   (\mathrm{Re}\zeta)^{-1},\,\mathrm{Re}\zeta>0. $
    An operator $L$ is called     {\it symmetric}     if one is densely defined and the following  equality  holds $(Lf,g)_{\mathfrak{H}}=(f,Lg)_{\mathfrak{H}},\,f,g\in   \mathrm{D}  (L).$  Consider a   sesquilinear form   $ s [\cdot,\cdot]$ (see \cite{firstab_lit:kato1980} )
defined on a linear manifold  of the Hilbert space $\mathfrak{H}.$ Let $L$ be a bounded operator acting in $\mathfrak{H},$ and assume that $\{\varphi_{n}\}_{1}^{\infty},\,\{\psi_{n}\}_{1}^{\infty}$ a pair of orthonormal bases in $\mathfrak{H}.$ Define the {\it absolute operator norm} as follows
$$
\|L\|_{2}:=\left(\sum\limits_{n,k=1}^{\infty}|(L\varphi_{n},\psi_{k})_{\mathfrak{H}}|^{2}\right)^{1/2}<\infty.
$$

Let   $  \mathfrak{h}=( s + s ^{\ast})/2,\, \mathfrak{k}   =( s - s ^{\ast})/2i$
   be a   real  and    imaginary component     of the   form $  s $ respectively, where $ s^{\ast}[u,v]=s \overline{[v,u]},\;\mathrm{D}(s ^{\ast})=\mathrm{D}(s).$ Denote by $   s  [\cdot ]$ the  quadratic form corresponding to the sesquilinear form $ s  [\cdot,\cdot].$ According to these definitions, we have $
 \mathfrak{h}[\cdot]=\mathrm{Re}\,s[\cdot],\,  \mathfrak{k}[\cdot]=\mathrm{Im}\,s[\cdot].$ Denote by $\tilde{s}$ the  closure   of a   form $s.$  The range of a quadratic form
  $ s [f],\,f\in \mathrm{D}(s),\,\|f\|_{\mathfrak{H}}=1$ is called   the {\it range} of the sesquilinear form  $s $ and is denoted by $\Theta(s).$
 A  form $s$ is called    {\it sectorial}    if  its    range  belongs to   a sector  having  a vertex $\iota$  situated at the real axis and a semi-angle $0\leq\theta<\pi/2.$   Due to Theorem 2.7 \cite[p.323]{firstab_lit:kato1980}  there exist unique    m-sectorial operators  $T_{s},T_{ \mathfrak{h}} $  associated  with   the  closed sectorial   forms $s,  \mathfrak{h}$   respectively.   The operator  $T_{\mathfrak{h}}$ is called  a {\it real part} of the operator $T_{s}$ and is denoted by  $\mathrm{Re}\, T_{s}.$

Using     notations of the paper     \cite{firstab_lit:kipriyanov1960}, we assume that $\Omega$ is a  convex domain of the  $n$ -  dimensional Euclidean space $\mathbb{E}^{n}$, $P$ is a fixed point of the boundary $\partial\Omega,$
$Q(r,\mathbf{e})$ is an arbitrary point of $\Omega.$  Let  $\mathrm{d}:=\mathrm{diam}\Omega,$ we denote by $\mathbf{e}$   a unit vector having a direction from  $P$ to $Q,$ denote by $r=|P-Q|$   the Euclidean distance between the points $P,Q,$ and   use the shorthand notation    $T:=P+\mathbf{e}t,\,t\in \mathbb{R}.$
We   consider the Lebesgue  classes   $L_{p}(\Omega),\;1\leq p<\infty $ of  complex valued functions.  For the function $f\in L_{p}(\Omega),$    we have
\begin{equation}\label{1.1}
\int\limits_{\Omega}|f(Q)|^{p}dQ=\int\limits_{\omega}d\chi\int\limits_{0}^{d(\mathbf{e})}|f(Q)|^{p}r^{n-1}dr<\infty,
\end{equation}
where $d\chi$   is an element of   solid angle of
the unit sphere  surface (the unit sphere belongs to $\mathbb{E}^{n}$)  and $\omega$ is a  surface of this sphere,   $d:=d(\mathbf{e})$  is the  length of the  segment of the  ray going from the point $P$ in the direction
$\mathbf{e}$ within the domain $\Omega.$
Without loss of  generality, we consider only those directions of $\mathbf{e}$ for which the inner integral on the right-hand  side of equality \eqref{1.1} exists and is finite. It is  the well-known fact that  these are almost all directions.
 We use a shorthand  notation  $P\cdot Q=P^{i}Q_{i}=\sum^{n}_{i=1}P_{i}Q_{i}$ for the inner product of the points $P=(P_{1},P_{2},...,P_{n}),\,Q=(Q_{1},Q_{2},...,Q_{n})$ which     belong to  $\mathbb{E}^{n}.$
     Denote by  $D_{i}f$  a weak partial derivative of the function $f$ with respect to a coordinate variable with index   $1\leq i\leq n.$
We  assume that all functions have  a zero extension outside  of $\bar{\Omega}.$
Everywhere further,   unless  otherwise  stated,  we   use  notations of the papers   \cite{firstab_lit:1Gohberg1965},  \cite{firstab_lit:kato1980},  \cite{firstab_lit:kipriyanov1960}, \cite{firstab_lit:1kipriyanov1960},
\cite{firstab_lit:samko1987}.

\chapter{Properties of fractional differential operators}\label{Ch1}

\section{Multidimensional integro-differential operators }

Accepting  the   notation  of the paper  \cite{firstab_lit:kipriyanov1960} we assume that $\Omega$ is a  convex domain of the  $n$ -  dimensional Euclidean space $\mathbb{E}^{n}$, $P$ is a fixed point of the boundary $\partial\Omega,$
$Q(r,\mathbf{e})$ is an arbitrary point of $\Omega;$ we denote by $\mathbf{e}$   a unit vector having the direction from  $P$ to $Q,$ denote by $r=|P-Q|$   the Euclidean distance between the points $P$ and $Q.$ We use the shorthand notation    $T:=P+\mathbf{e}t,\,t\in \mathbb{R}.$
We   consider the Lebesgue  classes   $L_{p}(\Omega),\;1\leq p<\infty $ of  complex valued functions.  For the function $f\in L_{p}(\Omega),$    we have
\begin{equation}\label{1.2}
\int\limits_{\Omega}|f(Q)|^{p}dQ=\int\limits_{\omega}d\chi\int\limits_{0}^{d(\mathbf{e})}|f(Q)|^{p}r^{n-1}dr<\infty,
\end{equation}
where $d\chi$   is the element of the solid angle of
the unit sphere  surface (the unit sphere belongs to $\mathbb{E}^{n}$)  and $\omega$ is a  surface of this sphere,   $d:=d(\mathbf{e})$  is the length of the  segment of the  ray going from the point $P$ in the direction
$\mathbf{e}$ within the domain $\Omega.$
Without lose of   generality, we consider only those directions of $\mathbf{e}$ for which the inner integral on the right side of equality \eqref{1.2} exists and is finite. It is  the well-known fact that  these are almost all directions.
We denote by   ${\rm Lip}\, \mu,\;(0<\mu\leq1) $  the set of functions satisfying the Holder-Lipschitz condition
$$
{\rm Lip}\, \lambda:=\left\{\rho(Q):\;|\rho(Q)-\rho(P)|\leq M r^{\lambda},\,P,Q\in \bar{\Omega}\right\}.
$$
Consider the  Kipriyanov  fractional differential operator     defined in  the paper \cite{firstab_lit:1kipriyanov1960}  by  the formal expression
\begin{equation*}
\mathfrak{D}^{\alpha}(Q)=\frac{\alpha}{\Gamma(1-\alpha)}\int\limits_{0}^{r} \frac{[f(Q)-f(T)]}{(r - t)^{\alpha+1}} \left(\frac{t}{r} \right) ^{n-1} dt+
C^{(\alpha)}_{n} f(Q) r ^{ -\alpha}\!,\, P\in\partial\Omega,
\end{equation*}
where
$
C^{(\alpha)}_{n} = (n-1)!/\Gamma(n-\alpha).
$
In accordance with Theorem 2   \cite{firstab_lit:1kipriyanov1960}, under the assumptions
\begin{equation}\label{1.3}
 lp\leq n,\;0<\alpha<l- \frac{n}{p} +\frac{n}{q}, \,q>p,
\end{equation}
we have that
     for sufficiently small $\delta>0$ the following inequality holds
\begin{equation}\label{1.4}
\|\mathfrak{D}^{\alpha}f\|_{L_{q}(\Omega)}\leq \frac{K}{\delta^{\nu}}\|f\|_{L_{p}(\Omega)}+\delta^{1-\nu}\|f\|_{L^{l}_{p}(\Omega)},\, f\in\dot{W}_{p}^{\,l}  (\Omega),
\end{equation}
where
\begin{equation*}
 \nu=\frac{n}{l}\left(\frac{1}{p}-\frac{1}{q} \right)+\frac{\alpha+\beta}{l}.
\end{equation*}
The constant  $K$ does not depend on $\delta,\,f;$   the point $P\in\partial\Omega ;\;\beta$ is an arbitrarily small fixed positive number.
Further, we assume that $\alpha \in (0,1).$
Using the notation   of the paper  \cite{firstab_lit:samko1987}, we   denote  by $I_{a+}^{\alpha}(L_{p} ),\;I_{b-}^{\alpha}(L_{p} ),\;1\leq p\leq\infty$ the left-side, right-side  classes of functions representable by the fractional integral on the segment $[a,b]$ respectively.  Let $\mathrm{d}:={\rm diam}\,\Omega ;\;C,C_{i}={\rm const},\,i\in \mathbb{N}_{0}.$ We use a shorthand  notation  $P\cdot Q=P^{i}Q_{i}=\sum^{n}_{i=1}P_{i}Q_{i}$ for the inner product of the points $P=(P_{1},P_{2},...,P_{n}),\,Q=(Q_{1},Q_{2},...,Q_{n})$ which     belong to  $\mathbb{E}^{n}.$
     Denote by  $D_{i}u$  the week  derivative of the function $u$ with respect to a coordinate variable with index   $1\leq i\leq n.$
We  assume that all functions have  a zero extension outside  of $\bar{\Omega}.$  Denote by  $  \mathrm{D} (L), \mathrm{R} (L)$    the domain of definition, range of values of the operator $L$ respectively.
Everywhere further,   unless  otherwise  stated,  we   use the notations of the papers    \cite{firstab_lit:kipriyanov1960}, \cite{firstab_lit:1kipriyanov1960},
\cite{firstab_lit:samko1987}.
Let us  define the operators
$$
 (\mathfrak{I}^{\alpha}_{0+}g)(Q  ):=\frac{1}{\Gamma(\alpha)} \int\limits^{r}_{0}\frac{g (T)}{( r-t)^{1-\alpha}}\left(\frac{t}{r}\right)^{n-1}dt,\,(\mathfrak{I}^{\alpha}_{d-}g)(Q  ):=\frac{1}{\Gamma(\alpha)} \int\limits_{r}^{d }\frac{g (T)}{(t-r)^{1-\alpha}}dt,
$$
$$
\;g\in L_{p}(\Omega),\;1\leq p\leq\infty.
$$
These  operators we call respectively   the left-side, right-side directional fractional integral.
We introduce   the classes of functions representable by the directional fractional integrals.
 \begin{equation*}
  \mathfrak{I}^{\alpha}_{0+}(L_{p}  ):=\left\{ u:\,u(Q)=(\mathfrak{I}^{\alpha}_{0+}g)(Q  ),\, g\in L_{p}(\Omega),\,1\leq p\leq\infty \right\},
 \end{equation*}
\begin{equation*}
  \mathfrak{I }  ^{\alpha}_{ d   -} (L_{p}  ) =\left\{ u:\,u(Q)=(\mathfrak{I}^{\alpha}_{d-}g)(Q  ),\;g\in L_{p}(\Omega),\;1\leq p\leq\infty  \right\}.
 \end{equation*}
Define the operators  $\psi^{+}_{\varepsilon },\; \psi^{-}_{\varepsilon }$ depended on the parameter $\varepsilon>0.$  In the left-side case
 \begin{equation}\label{1.5}
(\psi^{+}_{  \varepsilon }f)(Q)=  \left\{ \begin{aligned}
 \int\limits_{0}^{r-\varepsilon }\frac{ f (Q)r^{n-1}- f(T)t^{n-1}}{(  r-t)^{\alpha +1}r^{n-1}}  dt,\;\varepsilon\leq r\leq d  ,\\
   \frac{ f(Q)}{\alpha} \left(\frac{1}{\varepsilon^{\alpha}}-\frac{1}{ r ^{\alpha} }    \right),\;\;\;\;\;\;\;\;\;\;\;\;\;\;\; 0\leq r <\varepsilon .\\
\end{aligned}
 \right.
\end{equation}
In the right-side case
\begin{equation*}
 (\psi^{-}_{  \varepsilon }f)(Q)=  \left\{ \begin{aligned}
 \int\limits_{r+\varepsilon }^{d }\frac{ f (Q)- f(T)}{( t-r)^{\alpha +1}} dt,\;0\leq r\leq d -\varepsilon,\\
   \frac{ f(Q)}{\alpha} \left(\frac{1}{\varepsilon^{\alpha}}-\frac{1}{(d -r)^{\alpha} }    \right),\;\;\;d -\varepsilon <r \leq d ,\\
\end{aligned}
 \right.
 \end{equation*}
 where $\mathrm{D}(\psi^{+}_{  \varepsilon }),\mathrm{D}(\psi^{-}_{  \varepsilon })\subset L_{p}(\Omega).$
Using the definitions of the monograph  \cite[p.181]{firstab_lit:samko1987}  we consider the following operators.  In the left-side case
 \begin{equation}\label{1.6}
 ( \mathfrak{D} ^{\alpha}_{0+\!,\,\varepsilon}f)(Q)=\frac{1}{\Gamma(1-\alpha)}f(Q) r ^{-\alpha}+\frac{\alpha}{\Gamma(1-\alpha)}(\psi^{+}_{  \varepsilon }f)(Q).
 \end{equation}
 In the right-side case
 \begin{equation*}
 ( \mathfrak{D }^{\alpha}_{d-\!,\,\varepsilon}f)(Q)=\frac{1}{\Gamma(1-\alpha)}f(Q)(d-r)^{-\alpha}+\frac{\alpha}{\Gamma(1-\alpha)}(\psi^{-}_{  \varepsilon }f)(Q).
 \end{equation*}
 The left-side  and  right-side fractional derivatives  are  understood  respectively  as the  following  limits with respect to the norm  $L_{p}(\Omega),\,(1\leq p<\infty)$
 \begin{equation}\label{1.7}
 \mathfrak{D }^{\alpha}_{0+}f=\lim\limits_{\stackrel{\varepsilon\rightarrow 0}{ (L_{p}) }} \mathfrak{D }^{\alpha}_{0+\!,\,\varepsilon} f  ,\; \mathfrak{D }^{\alpha}_{d-}f=\lim\limits_{\stackrel{\varepsilon\rightarrow 0}{ (L_{p}) }} \mathfrak{D }^{\alpha}_{d-\!,\,\varepsilon} f .
\end{equation}
 We   need   auxiliary propositions, which  are  presented  in the next section.

 \section{Mapping and representation  theorems}

We have the following theorem on boundedness of the directional fractional integral operators.
 \begin{teo}\label{T1.1} The directional fractional integral operators are bounded in $L_{p}(\Omega),$ $1\leq p<\infty,$ the following estimates holds
  \begin{equation}\label{1.8}
 \| \mathfrak{I}^{\alpha}_{0 +}u\|_{L_{p}(\Omega)}\leq C\|u \|_{L_{p}(\Omega)},\;\|   \mathfrak{I} ^{\alpha}_{d -}u\|_{L_{p}(\Omega)}\leq C\|u \|_{L_{p}(\Omega)},\;C=  \mathrm{d} ^{\alpha}/ \Gamma(\alpha+1)  .
 \end{equation}
 \end{teo}
 \begin{proof}
Let us prove   first estimate  \eqref{1.8}, the proof of the second one   is absolutely analogous. Using the generalized Minkowski inequality, we have
$$
 \| \mathfrak{I}^{\alpha}_{0 +}u\|_{L_{p}(\Omega)}=\frac{1}{\Gamma(\alpha)}\left( \int\limits_{\Omega}\left|  \int\limits^{r}_{0}\frac{g (T)}{( r-t)^{1-\alpha}}\left(\frac{t}{r}\right)^{n-1}\!\!\!dt  \right|^{p}dQ\right)^{1/p}
$$
$$
 =\frac{1}{\Gamma(\alpha)}\left( \int\limits_{\Omega}\left|  \int\limits^{r}_{0}\frac{g (Q-\tau  \mathbf{e})}{\tau^{1-\alpha}}\left(\frac{r-\tau}{r}\right)^{n-1}\!\!\!d\tau  \right|^{p}dQ\right)^{1/p}
$$
$$
  \leq\frac{1}{\Gamma(\alpha)}\left( \int\limits_{\Omega}\left(  \int\limits^{\mathrm{d}}_{0}\frac{|g (Q-\tau  \mathbf{e})|}{\tau^{1-\alpha}}d\tau  \right)^{p}dQ\right)^{1/p}
  $$
  $$
\leq \frac{1}{\Gamma(\alpha)}  \int\limits^{\mathrm{d}}_{0}\tau^{\alpha-1} d\tau \left( \int\limits_{\Omega}  |g (Q-\tau  \mathbf{e})|^{p} dQ  \right)^{1/p}\!\!\leq
 \frac{ \mathrm{d} ^{\alpha}}{\Gamma(\alpha+1)}\,  \| u\|_{L_{p}(\Omega)}.
$$\end{proof}

\begin{teo}\label{T1.2}
Suppose $f\in L_{p}(\Omega),$  there  exists   $\lim\limits_{\varepsilon\rightarrow  0}\psi^{+}_{  \varepsilon }f$ or $\lim\limits_{\varepsilon\rightarrow  0}\psi^{-}_{  \varepsilon }f$ with respect to the norm $L_{p}(\Omega),\,(1\leq p<\infty);$ then   $f\in \mathfrak{I} ^{\alpha}_{0 +}(L_{p}) $ or $f\in \mathfrak{I }^{\alpha}_{d -}(L_{p})$ respectively.
\end{teo}
\begin{proof}
 Let $f\in L_{p}(\Omega)$ and $\lim\limits_{\stackrel{\varepsilon\rightarrow 0}{ (L_{p}) }}\psi^{+}_{  \varepsilon }f=\psi.$ Consider the function
$$
(\varphi^{+}_{ \varepsilon}f)(Q)=\frac{1}{\Gamma(1-\alpha)}\left\{\frac{f(Q)}{ r ^{\alpha}}+\alpha (\psi^{+}_{ \varepsilon }f)(Q) \right\}.
 $$
 Taking into account     \eqref{1.5}, we can easily prove  that $\varphi^{+}_{  \varepsilon }f\in L_{p}(\Omega).$ Obviously,    there  exists the limit
 $\varphi^{+}_{  \varepsilon }f\rightarrow \varphi\in L_{p}(\Omega),\,\varepsilon\downarrow 0.$ Taking into account Theorem \ref{T1.1}, we can  complete the proof,  if we  show that
 \begin{equation}\label{1.9}
 \mathfrak{I}^{\alpha}_{0+}\varphi^{+}_{ \varepsilon }f  \stackrel{L_{p}}{\rightarrow} f,\,\varepsilon\downarrow0.
 \end{equation}
 In the case  $(\varepsilon\leq r\leq d),$ we have
$$
(\mathfrak{I}^{\alpha}_{0 +}\varphi^{+}_{ \varepsilon }f)(Q)\frac{\pi r^{n-1}}{\sin\alpha\pi}=
 \int\limits_{\varepsilon}^{r}\frac{f (P+y\mathbf{e})y ^{n-1-\alpha}}{( r-y)^{1-\alpha} } dy
$$
$$
  +\alpha\int\limits_{\varepsilon}^{r}( r-y)^{\alpha-1}  dy\int\limits_{0 }^{y-\varepsilon }\frac{ f (P+y\mathbf{e})y^{n-1}- f(T)t^{n-1}}{( y-t )^{\alpha +1}} dt
$$
$$
 +\frac{1}{\varepsilon^{\alpha}}\int\limits_{0}^{\varepsilon  }f (P+y\mathbf{e})( r-y)^{\alpha-1} y^{n-1}  dy =I.
$$
By direct calculation, we obtain
 \begin{equation}\label{1.10}
I=   \frac{1}{\varepsilon^{\alpha}}\int\limits_{0}^{r  }f (P+y\mathbf{e})( r-y)^{\alpha-1}y^{n-1}   dy  -
 \alpha\int\limits_{\varepsilon}^{r}( r-y)^{\alpha-1} dy\int\limits_{0 }^{y-\varepsilon }\frac{  f(T)}{(  y-t)^{\alpha +1}}t^{n-1} dt .
\end{equation}
 Changing the variable of integration in  the second  integral,   we have
$$
 \alpha\int\limits_{\varepsilon}^{r}( r-y)^{\alpha-1} dy\int\limits_{0 }^{y-\varepsilon }\frac{  f(T)}{(  y-t)^{\alpha +1}}t^{n-1} dt
$$
$$
 =\alpha\int\limits_{0}^{r-\varepsilon}( r-y-\varepsilon)^{\alpha-1} dy\int\limits_{0 }^{y  }\frac{  f(T)}{(  y+\varepsilon-t)^{\alpha +1}}t^{n-1} dt
$$
\begin{equation}\label{1.11}
=\alpha\int\limits_{0}^{r-\varepsilon}f(T)t^{n-1} dt\int\limits_{t }^{r-\varepsilon  }\frac{( r-y-\varepsilon)^{\alpha-1}   }{(  y+\varepsilon-t)^{\alpha +1}}dy
$$
$$
=\alpha\int\limits_{0}^{r-\varepsilon}f(T)t^{n-1} dt\int\limits_{t +\varepsilon}^{r  } ( r-y )^{\alpha-1}   (  y -t)^{-\alpha -1} dy .
\end{equation}
Applying formula   (13.18) \cite[p.184]{firstab_lit:samko1987}, we get
\begin{equation}\label{1.12}
 \int\limits_{t +\varepsilon}^{r  } ( r-y )^{\alpha-1}   (  y -t)^{-\alpha -1} dy=
  \frac{1}{\alpha \varepsilon^{\alpha}}\cdot\frac{(r-t-\varepsilon)^{\alpha}}{ r-t }.
\end{equation}
 Combining    relations \eqref{1.10},\eqref{1.11},\eqref{1.12}, using the change of the variable     $t=r-\varepsilon\tau, $ we get
$$
(\mathfrak{I}^{\alpha}_{0 +}\varphi^{+}_{ \varepsilon }f)(Q)\frac{\pi r^{n-1}}{\sin\alpha\pi}
$$
$$
=\frac{1}{\varepsilon^{\alpha}}\left\{ \int\limits_{0}^{r  }f (P+y\mathbf{e})(r-y)^{\alpha-1}y^{n-1}   dy- \int\limits_{0}^{r-\varepsilon  } \frac{f(T)(r-t-\varepsilon)^{\alpha}}{ r-t } t^{n-1}dt \right\}
$$
\begin{equation}\label{1.13}
=\frac{1}{ \varepsilon^{\alpha}} \int\limits_{0 }^{r  } \frac{f(T)\left[(r -t)^{\alpha}-(r-t-\varepsilon)_{+}^{\alpha}\right]}{ r-t }t^{n-1}dt
$$
$$
=\int\limits_{0 }^{r/\varepsilon   }\frac{\tau^{\alpha}-(\tau-1)_{+}^{\alpha}}{ \tau } f(P+[r-\varepsilon \tau ]\mathbf{e})(r-\varepsilon \tau)^{n-1} d\tau,\;\tau_{+}=\left\{\begin{array}{cc}\tau,\;\tau\geq0;\\[0,25cm] 0,\;\tau<0\,.\end{array}\right. .
 \end{equation}
Consider the auxiliary function $\mathcal{K}$ defined in the paper   \cite[p.105]{firstab_lit:samko1987}
 \begin{equation}\label{1.14}
 \mathcal{K}(t)= \frac{\sin\alpha\pi}{\pi }\cdot\frac{ t_{+}^{\alpha}-(t-1)_{+}^{\alpha}}{ t }\in L_{p}(\mathbb{R}^{1});\; \int\limits_{0 }^{\infty  }\mathcal{K}(t)dt=1;\;\mathcal{K}(t)>0.
\end{equation}
  Combining  \eqref{1.13},\eqref{1.14} and taking into account that    $f$ has the zero extension outside of  $\bar{\Omega},$ we obtain
 \begin{equation}\label{1.15}
 (\mathfrak{I}^{\alpha}_{0+}\varphi^{+}_{ \varepsilon }f)(Q)-f(Q) =  \int\limits_{0 }^{\infty  }\mathcal{K}(t) \left\{f(P+[r-\varepsilon t]\mathbf{e})(1-\varepsilon t/r)_{+}^{n-1}-f(P+ r \mathbf{e})  \right\}dt.
\end{equation}
Consider the case  $(0\leq  r <\varepsilon).$ Taking into account \eqref{1.5}, we get
 \begin{equation}\label{1.16}
(\mathfrak{I}^{\alpha}_{0+}\varphi^{+}_{ \varepsilon }f)(Q)-f (Q) =
\frac{\sin\alpha\pi}{\pi\varepsilon^{\alpha}} \int\limits_{0}^{r }\frac{f (T)}{(r-t)^{1-\alpha} } \left(\frac{t}{r} \right)^{n-1} dt-f (Q)
$$
$$
=\frac{\sin\alpha\pi}{\pi\varepsilon^{\alpha}} \int\limits_{0}^{r }\frac{f (P+[r-t]\mathbf{e})}{t^{1-\alpha} }\left(\frac{r-t}{r} \right)^{n-1} dt-f (Q).
\end{equation}
Consider the domains
 \begin{equation*}
 \Omega_{\varepsilon} :=\{Q\in\Omega,\,d(\mathbf{e})\geq\varepsilon \},\;\tilde{\Omega}_{  \varepsilon }=\Omega\setminus \Omega_{\varepsilon}.
 \end{equation*}
In accordance with    this definition  we can  divide the surface $\omega$ into two  parts $ \omega_{\varepsilon}$ and  $\tilde{\omega}_{ \varepsilon },$ where $ \omega_{\varepsilon}$ is  the subset of  $\omega$ such that   $d(\mathbf{e})\geq\varepsilon$ and $\tilde{\omega}_{ \varepsilon }$ is  the subset of  $\omega$ such that  $d(\mathbf{e}) <\varepsilon.$
Using  \eqref{1.15},\eqref{1.16}, we get
\begin{equation*}
  \|(\mathfrak{I}^{\alpha}_{0+}\varphi^{+}_{ \varepsilon }f) -f\|^{p}_{L_{p}(\Omega)}
$$
$$
=   \int\limits_{\omega_{\varepsilon}}d\chi\int\limits_{\varepsilon}^{ d}
  \left|\int\limits_{0 }^{\infty  }\mathcal{K}(t)[f(Q- \varepsilon t \mathbf{e})(1-\varepsilon t/r)_{+}^{n-1}-f(Q)]dt\right|^{p}r^{n-1}dr
$$
$$
+ \int\limits_{\omega_{\varepsilon}}d\chi\int\limits_{0}^{\varepsilon }
\left|  \frac{\sin\alpha\pi}{\pi\varepsilon^{\alpha}}\int\limits_{0}^{r }
\frac{f (P+[r-t]\mathbf{e})}{t^{1-\alpha} }\left(\frac{r-t}{r} \right)^{n-1} dt-f (Q)\right|^{p}r^{n-1}dr
$$
$$
+  \int\limits_{\tilde{\omega}_{\varepsilon} }d\chi\int\limits_{0}^{d }
\left| \frac{\sin\alpha\pi}{\pi\varepsilon^{\alpha}}\int\limits_{0}^{r }\frac{f (P+[r-t]\mathbf{e})}{t^{1-\alpha} }\left(\frac{r-t}{r} \right)^{n-1} dt-f (Q) \right|^{p}r^{n-1}dr =I_{1}+I_{2}+I_{3}.
 \end{equation*}
Consider $ I_{1},$ using the generalized Minkovski  inequality,  we get
$$
  I^{\frac{1 }{p}}_{1} \leq  \int\limits_{0 }^{\infty  }\mathcal{K}(t)
  \left(\int\limits_{\omega_{\varepsilon} }d\chi\int\limits_{\varepsilon}^{ d}
  |f(Q- \varepsilon t \mathbf{e})(1-\varepsilon t/r)_{+}^{n-1}-f(Q)|^{p}r^{n-1}dr \right)^{\frac{1 }{p}} dt.
$$
We use the following  notation
$$
 h(\varepsilon,t):= \mathcal{K}(t)\left(\int\limits_{\omega_{\varepsilon} }d\chi\int\limits_{\varepsilon}^{ d}
  |f(Q- \varepsilon t \mathbf{e})(1-\varepsilon t/r)_{+}^{n-1}-f(Q)|^{p}r^{n-1}dr \right)^{\frac{1 }{p}} dt.
$$
It can easily be checked that
\begin{equation}\label{1.17}
  |h(\varepsilon,t)|\leq 2\mathcal{K}(t) \| f\|_{L_{p}(\Omega)},\;\forall\varepsilon>0;
\end{equation}
  \begin{equation*}
  |h(\varepsilon,t)|\leq  \left(\int\limits_{\omega_{\varepsilon} }d\chi\int\limits_{\varepsilon}^{ d}
 \left |(1-\varepsilon t/r)_{+}^{n-1}[f(Q- \varepsilon t \mathbf{e})-f(Q)]\right|^{p}r^{n-1}dr \right)^{\frac{1 }{p}} dt
 $$
 $$+\left(\int\limits_{\omega_{\varepsilon} }d\chi\int\limits_{0}^{ d}
  \left|f(Q)  [1-(1-\varepsilon t/r)_{+}^{n-1}]\right|^{p}r^{n-1}dr \right)^{\frac{1 }{p}} dt=I_{11}+I_{12}.
\end{equation*}
By virtue of the average continuity property  in  $L_{p}(\Omega),$  we have $\forall t>0:\, I_{11}\rightarrow 0,\;\varepsilon\downarrow 0.$
Consider $I_{12}$ and  let us define the function
$$
h_{1}(\varepsilon,t,r):=\left|f(Q)  \left[1-(1-\varepsilon t/r)_{+}^{n-1}\right]\right|.
$$ Obviously, the following relations hold  almost everywhere  in $\Omega$
$$
\forall t>0,\,h_{1}(\varepsilon,t,r) \leq |f(Q)|,\;h_{1}(\varepsilon,t,r)\rightarrow 0,\;\varepsilon \downarrow0.
$$
Applying the Lebesgue  dominated convergence theorem, we get $I_{12}\rightarrow 0,\;\varepsilon\downarrow 0.$
It implies that
\begin{equation}\label{1.18}
\forall t>0,\,\lim\limits_{\varepsilon\rightarrow 0} h(\varepsilon,t)=0.
\end{equation}
Taking into account  \eqref{1.17}, \eqref{1.18} and    applying  the Lebesgue  dominated convergence theorem again, we obtain
$$
 I_{1}\rightarrow 0,\;\;\varepsilon\downarrow0  .
$$
Consider $I_{2},$ using the Mincovski  inequality, we  get
$$
I^{\frac{1 }{p}}_{2} \leq \frac{\sin\alpha\pi}{\pi\varepsilon^{\alpha}}\left( \int\limits_{\omega_{\varepsilon} }d\chi\int\limits_{0}^{\varepsilon }
\left| \int\limits_{0}^{r }\frac{f (Q-t\mathbf{e})}{t^{1-\alpha} }\left(\frac{r-t}{r} \right)^{n-1} dt\right|^{p}r^{n-1}dr\right)^{\frac{1 }{p}}
$$
$$
+\left(\int\limits_{\omega_{\varepsilon} }d\chi\int\limits_{0}^{\varepsilon}
\left| f (Q) \right|^{p}r^{n-1}dr \right)^{\frac{1 }{p}}=I_{21} +I_{22}.
$$
Applying the generalized  Mincovski  inequality, we obtain
$$
I_{21}\frac{\pi}{\sin\alpha\pi}=\frac{1}{\varepsilon^{\alpha}}\left(\int\limits_{\omega_{\varepsilon} }d\chi
\int\limits_{0}^{\varepsilon }
\left|  \int\limits_{0}^{r }\frac{f (Q-t\mathbf{e})}{t^{1-\alpha} }\left(\frac{r-t}{r} \right)^{n-1} \!\!\! dt\right|^{p}r^{n-1}\! dr \right)^{\frac{1 }{p}}
$$
$$
\leq\frac{1}{\varepsilon^{\alpha}}\left\{\int\limits_{\omega_{\varepsilon} }\!\!\left[\int\limits_{0}^{\varepsilon }\!\!t ^{\alpha-1 }\!\!
\left( \int\limits_{t}^{\varepsilon }\!\!|f (Q -t \mathbf{e})|^{p}\!\left(\frac{r-t}{r} \right)^{\!\!\!(p-1)(n-1)}\!\!\!(r-t)^{n-1} \!dr  \right)^{\frac{1 }{p}}\!\!dt\right]^{p}\!\!d\chi \right\}^{\frac{1 }{p}}
 $$
 $$
\leq\frac{1}{\varepsilon^{\alpha}}\left\{\int\limits_{\omega_{\varepsilon} }\left[\int\limits_{0}^{\varepsilon } t ^{\alpha-1 }
\left( \int\limits_{t}^{\varepsilon } \left|f (P+[r-t]\mathbf{e})\right|^{p}  (r-t)^{n-1}dr  \right)^{\frac{1 }{p}}\!\!dt\right]^{p}\!\!d\chi \right\}^{\frac{1 }{p}}
 $$
  $$
\leq\frac{1}{\varepsilon^{\alpha}}\left\{\int\limits_{\omega_{\varepsilon} }\left[\int\limits_{0}^{\varepsilon } t ^{\alpha-1 }
\left( \int\limits_{0}^{\varepsilon } |f (P +r \mathbf{e})|^{p}  r^{n-1}dr  \right)^{\frac{1 }{p}}\!\!dt\right]^{p}\!\!d\chi \right\}^{\frac{1 }{p}}\!\!=\frac{1}{\alpha}\| f\| _{L_{p}( \Delta_{\varepsilon})},\;
$$
$$
\Delta_{\varepsilon} :=\{Q\in\Omega_{\varepsilon},\,r<\varepsilon \} .
 $$
Note that   ${\rm mess}\, \Delta_{\varepsilon}\rightarrow 0,\,\varepsilon\downarrow0,$
hence  $I_{21},I_{22}\rightarrow 0,\,\varepsilon\downarrow0 .$ It follows  that $I_{2 }\rightarrow 0,\,\varepsilon\downarrow0.$ In the same way, we obtain   $I_{3 }\rightarrow 0,\,\varepsilon\downarrow 0.$ Since we proved that $I_{1},I_{2},I_{3}\rightarrow 0,\,\varepsilon\downarrow0,$ then relation \eqref{1.9} holds.
 This completes the proof corresponding to the left-side case.
 The proof corresponding to the right-side case is absolutely analogous.
\end{proof}
\begin{teo}\label{T1.3}
Suppose  $f=\mathfrak{I}^{\alpha}_{0+} \psi$ or $f= \mathfrak{I} ^{\alpha}_{d -} \psi ,\;\psi\in L_{p}(\Omega),\;1\leq p<\infty;$ then
$
\, \mathfrak{D }^{\alpha}_{0+}f =\psi
$
or
$
 \mathfrak{D} ^{\alpha}_{d-}f =\psi
$
respectively.
 \end{teo}
\begin{proof}
Consider
$$
r^{n-1}f(Q)-(r-\tau)^{n-1}f(Q-\tau\mathbf{e})
$$
$$
=\frac{1}{\Gamma(\alpha)} \int\limits_{0}^{ r}\frac{\psi (Q-t\mathbf{e})}{t^{1-\alpha}}\left( r-t \right)^{n-1}dt-\frac{1}{\Gamma(\alpha)} \int\limits_{\tau}^{ r}\frac{\psi (Q-t\mathbf{e})}{(t-\tau)^{1-\alpha}}\left( r-t \right)^{n-1}dt
$$
$$
 =\tau^{\alpha-1}  \int\limits_{0}^{ r} \psi (Q-t\mathbf{e}) k\left(\frac{t}{\tau}\right)(r-t)^{n-1}dt,\;k(t)= \frac{1}{\Gamma(\alpha)}\left\{\begin{array}{cc}t^{\alpha-1},\;0<t<1;\\[0,25cm] t^{\alpha-1}-(t-1)^{\alpha-1},\;t>1.\end{array}\right.
$$
Hence in the case    $(\varepsilon\leq r\leq d),$ we have
$$
(\psi^{+}_{  \varepsilon }f)(Q)=\int\limits_{ \varepsilon }^{ r }\frac{r^{n-1} f(Q)-(r-\tau)^{n-1}f(Q-\tau\mathbf{e})}{ r^{n-1}\tau ^{\alpha +1}} d\tau
$$
$$
=\int\limits_{ \varepsilon }^{ r } \tau^{-2} d\tau\int\limits_{0}^{ r} \psi (Q-t\mathbf{e}) k\left(\frac{t}{\tau}\right)\left( 1-t/r \right)^{n-1}dt
$$
$$
=\int\limits_{ 0 }^{ r }\psi (Q-t\mathbf{e}) \left( 1-t/r \right)^{n-1}  dt\int\limits_{\varepsilon}^{ r}  k\left(\frac{t}{\tau}\right) \tau^{-2}d \tau
$$
$$
=\int\limits_{ 0 }^{ r }\psi (Q-t\mathbf{e}) \left( 1-t/r \right)^{n-1}t^{-1} dt\int\limits_{ t/ r  }^{t/\varepsilon}  k (s )ds.
$$
Applying   formula    (6.12) \cite[p.106]{firstab_lit:samko1987}, we get
$$
(\psi^{+}_{  \varepsilon }f)(Q)\cdot\frac{\alpha}{\Gamma(1-\alpha)}=\int\limits_{ 0 }^{ r }\psi (Q-t\mathbf{e})
 \left( 1-t/r \right)^{n-1}\left[ \frac{1}{\varepsilon}\mathcal{K}\left(\frac{t}{\varepsilon}\right) - \frac{1}{ r}\mathcal{K}\left(\frac{t}{ r}\right) \right]dt.
$$
Since in accordance with \eqref{1.14},  we have
 $$
\mathcal{K}\left(\frac{t}{ r}\right)=\left\{\Gamma(1-\alpha)\Gamma(\alpha) \right\}^{-1} \left(\frac{t}{ r}\right)^{\alpha-1}\!\!\!,
$$
then
$$
(\psi^{+}_{  \varepsilon }f)(Q)\cdot\frac{\alpha}{\Gamma(1-\alpha)}=\int\limits_{ 0 }^{ r /\varepsilon }\mathcal{K} ( t   ) \psi (Q-\varepsilon t\mathbf{e})\left( 1-\varepsilon t/r \right)^{n-1}
  dt-\frac{f(Q)}{\Gamma(1-\alpha) r ^{ \alpha}}.
$$
 Taking into account \eqref{1.6},\eqref{1.14}, and that the function     $\psi(Q)$ has the zero  extension   outside of   $\bar{\Omega},$   we obtain
$$
 ( \mathfrak{D} ^{\alpha}_{0+,\varepsilon}f)(Q)-\psi(Q) =\int\limits_{ 0 }^{\infty}
\mathcal{K} ( t   ) \left[\psi (Q - \varepsilon t \mathbf{e})(1-\varepsilon t/r)_{+}^{n-1}-\psi (Q) \right]dt,\;\varepsilon\leq r\leq d  .
$$
Consider the case    $ ( 0\leq r<\varepsilon).$ In accordance with \eqref{1.5},  we have
$$
( \mathfrak{D} ^{\alpha}_{0+,\varepsilon}f)(Q)-\psi(Q)=\frac{f(Q)}{\varepsilon^{\alpha}\Gamma(1-\alpha)}-\psi(Q).
$$
Using the generalized Mincovski  inequality, we   get
$$
\|( \mathfrak{D} ^{\alpha}_{0+,\varepsilon}f)(Q)-\psi(Q)\|_{L_{p}(\Omega)}\leq\!\! \int\limits_{ 0 }^{\infty}
\!\! \mathcal{K}(t)   \| \psi(Q - \varepsilon t \mathbf{e})(1-\varepsilon t/r)_{+}^{n-1}-\psi (Q)\|_{L_{p}(\Omega)}dt
$$
$$
+\frac{1}{ \Gamma(1-\alpha)\varepsilon^{\alpha}}\, \|f\|_{L_{p}(\Delta'_{\varepsilon})}+\|\psi\|_{L_{p}(\Delta'_{\varepsilon})},\;\Delta'_{\varepsilon}=\Delta_{\varepsilon}\cup \tilde{\Omega}_{\varepsilon},
$$
here we use the denotations that were used in Theorem \ref{T1.2}. Arguing as above (see Theorem \ref{T1.2}), we see that
    all three  summands of the right-hand  side  of the previous  inequality tend to zero, when  $\varepsilon\downarrow 0.$
\end{proof}

\begin{teo} \label{T1.4}  Suppose $\rho\in {\rm Lip}\,\lambda,\;\alpha<\lambda\leq 1,\;f\in H_{0}^{1}(\Omega) ;$ then $\rho f\in  \mathfrak{I} ^{\alpha}_{\,0 +}(L_{2} ) \cap \mathfrak{I} ^{\alpha}_{d -}(L_{2} ).$
\end{teo}

\begin{proof}
We   provide a proof only for the left-side   case,   the proof corresponding to the right-side case is absolutely analogous.     First,  assume that  $f\in C_{0}^{\infty}(\Omega) .$
Using  the denotations that were used in Theorem \ref{T1.2},
  we have
\begin{equation}\label{1.19}
\|\psi^{+}_{ \varepsilon_1}f  - \psi^{+}_{ \varepsilon_2}f \|_{L_{2}(\Omega )}\leq\!\!\|\psi^{+}_{ \varepsilon_1}f  - \psi^{+}_{ \varepsilon_2}f \|_{L_{2}( \Omega_{\varepsilon_{1}})}+\|\psi^{+}_{ \varepsilon_1}f  - \psi^{+}_{ \varepsilon_2}f \|_{L_{2}(\tilde{\Omega}_{ \varepsilon_{1}})},
\end{equation}
where $\varepsilon_{1}>\varepsilon_{2}>0.$
   We have the following reasoning
$$\|\psi^{+}_{ \varepsilon_1}f  - \psi^{+}_{ \varepsilon_2}f \|_{L_{2}(\Omega_{\varepsilon_{1}})}\leq
 \left(\int\limits_{\omega_{\varepsilon_{1}}}d\chi\int\limits_{\varepsilon_{1}}^{d }\left|\int\limits_{r-\varepsilon_{1}}^{r-\varepsilon_{2}}\frac{(\rho f)(Q)r^{n-1}-
 (\rho f)(T)t^{n-1}}{r^{n-1}(r-t)^{\alpha +1}} dt \right|^{2}r^{n-1} dr    \right)^{\frac{1}{2}}
$$
$$
  +\!\!\left(\int\limits_{\omega_{\varepsilon_{1}}}d\chi\int\limits_{ \varepsilon_{2}}^{ \varepsilon_{1}}\!\!\left|\int\limits_{0}^{r- \varepsilon_{1}}\frac{(\rho f)(Q)r^{n-1} }{r^{n-1}(r-t)^{\alpha +1}} dt-\!\!
 \int\limits_{ 0}^{r-\varepsilon_{2}}\frac{(\rho f)(Q)r^{n-1}-(\rho f)(T)t^{n-1}}{r^{n-1}(r-t)^{\alpha +1}} dt\right|^{2}r^{n-1} dr    \right)^{\frac{1}{2}}
 $$
 $$
 +\left(\int\limits_{\omega_{\varepsilon_{1}}}d\chi\int\limits_{0}^{\varepsilon_{2} }\left|\int\limits_{ 0}^{r-\varepsilon_{1}}
 \frac{(\rho f)(Q)r^{n-1} }{r^{n-1}(r-t)^{\alpha +1}} dt-
 \int\limits_{0}^{r- \varepsilon_{2}}\frac{(\rho f)(Q)r^{n-1} }{r^{n-1}(r-t)^{\alpha +1}} dt\right|^{2}r^{n-1} dr    \right)^{\frac{1}{2}}
$$
$$
 =I_{1}+I_{2}+I_{3}.
$$
  Since  $f\in C_{0}^{\infty}(\Omega),$ then for sufficiently  small   $\varepsilon_{1}>0$ we have  $f(Q)=0,\,r<\varepsilon_{1}.$ This implies that   $I_{2}=I_{3}=0$ and that the second summand of the right side of   inequality   \eqref{1.19} equals zero.
Making the change the variable in $I_{1},$ then using the generalized Minkowski  inequality, we get
$$
 I_{1}  =\left(\int\limits_{\omega_{\varepsilon_{1}}}d\chi\int\limits_{\varepsilon_{1}}^{d}\left|\int\limits_{ \varepsilon_{1}}^{ \varepsilon_{2}}
 \frac{(\rho f)(Q)r^{n-1}-
 (\rho  f )(Q-\mathbf{e} t )(r-t)^{n-1}}{ r^{n-1}t ^{\alpha +1}} dt\right|^{2}r^{n-1} dr    \right)^{\frac{1}{2}}
$$
$$ \leq\int\limits_{ \varepsilon_{2}}^{ \varepsilon_{1}}t^{-\alpha -1}\!\!\!\left(\int\limits_{\omega_{\varepsilon_{1}}}d \chi\int\limits_{\varepsilon_{1}}^{d}\left| (\rho  f )(Q)-(1-t/r)^{n-1}(\rho  f )(Q-\mathbf{e}t) \right|^{2}  r^{n-1} dr    \right)^{\frac{1}{2}}\!\! dt
$$
$$
  \leq\int\limits_{ \varepsilon_{2}}^{ \varepsilon_{1}}t^{-\alpha -1}\!\!\!\left(\int\limits_{\omega_{\varepsilon_{1}}}d \chi\int\limits_{\varepsilon_{1}}^{d}\left| (\rho f )(Q)- (\rho f )(Q-\mathbf{e}t) \right|^{2}  r^{n-1} dr    \right)^{\frac{1}{2}} \!\!dt
$$
$$
+\int\limits_{ \varepsilon_{2}}^{ \varepsilon_{1}}t^{-\alpha -1}\left(\int\limits_{\omega_{\varepsilon_{1}}}d \chi\int\limits_{\varepsilon_{1}}^{d}\left[ 1-( 1-t/r )^{n-1}\right]\left| (\rho f )(Q-\mathbf{e}t) \right|^{2}  r^{n-1} dr    \right)^{\frac{1}{2}}\!\!dt
$$
$$
\leq C_{1}\!\!\int\limits_{ \varepsilon_{2}}^{ \varepsilon_{1}}t^{\lambda-\alpha-1 }d t+\int\limits_{ \varepsilon_{2}}^{ \varepsilon_{1}}t^{-\alpha }\left(\int\limits_{\omega_{\varepsilon_{1}}}d \chi\int\limits_{\varepsilon_{1}}^{d}
\left|\frac{1}{r} \sum\limits_{i=0}^{n-2}\left( \frac{t}{r}\right)^{i}(\rho f )(Q-\mathbf{e}t) \right|^{2}  r^{n-1} dr    \right)^{\!\!\frac{1}{2}}\!\!dt.
$$
Using  the  function $f$    property, we see that there  exists a constant $\delta$ such that $  f  (Q-\mathbf{e}t )=0,\;r<\delta.$  In accordance with the above reasoning,      we have
 $$
 I_{1} \leq
   C_{1}\frac{  \varepsilon^{  \lambda-\alpha}_{1}-\varepsilon^{\lambda-\alpha}_{2}   }{\!\!\!\!\lambda-\alpha  }+
   (n-1)\frac{   \varepsilon^{  1- \alpha}_{1}-\varepsilon^{1- \alpha}_{2}    }{\!\!\!\delta(1-\alpha)  } \,\|f\|_{L_{2}(\Omega)} .
$$
Applying Theorem \ref{T1.1}, we complete  the     proof for the case    $ ( f\in C_{0}^{\infty}(\Omega)).$
Now assume that  $f\in H^1_{0}(\Omega),$ then there   exists the sequence $\{f_{n}\} \subset C_{0}^{\infty}(\Omega),\;   f_{n}\stackrel{ H^1_{0}}{\longrightarrow} f.$ It is easy to prove that $   \rho f_{n}\stackrel{ L_{2}}{\longrightarrow} \rho f.$
In accordance with the proven  above fact, we have $\rho f_{n}= \mathfrak{I} ^{\alpha}_{0+}\varphi_{n},\;\{\varphi_{n}\}\in L_{2}(\Omega),$ therefore
\begin{equation}\label{1.20}
 \mathfrak{I} ^{\alpha}_{0+}\varphi_{n}\stackrel{L_{2} }{\longrightarrow} \rho f.
\end{equation}
To conclude the proof, it is sufficient to show that   $\varphi_{n}\stackrel{L_{2} }{\longrightarrow}\varphi\in L_{2}(\Omega).$  Note that by virtue of  Theorem  \ref{T1.2} we have  $\mathfrak{ D} ^{\alpha}_{0+}\rho f_{n}=\varphi_{n}.$
Let  $ c_{n,m}:=f_{n+m}-f_{n},$  we have
$$
\|\varphi_{n+m}-\varphi_{n}\|_{L_{2}(\Omega)}\leq\frac{\alpha}{\Gamma(1-\alpha)}\left(\int\limits_{\Omega} \left|\int\limits_{0 }^{r}\frac{(\rho c_{n,m})(Q)r^{n-1}-
(\rho c_{n,m})(T)t^{n-1}}{r^{n-1}( t-r)^{\alpha +1}} dt\right|^{2} dQ    \right)^{\frac{1}{2}}
$$
$$
+\frac{1}{\Gamma(1-\alpha)}\left(\int\limits_{\Omega}\left| \frac{(\rho c_{n,m})(Q) }{ r ^{\alpha }} \right|^{2}dQ    \right)^{\frac{1}{2}}=I_3 +I_4.
$$
Consider $I_{3}.$ It can be shown in the usual way that
$$
\frac{\Gamma(1-\alpha)}{\alpha} I_3  \leq
  \left\{\int\limits_{\Omega}\left|\int\limits_{0 }^{ r}\frac{  (\rho c_{n,m}) (Q)-  (\rho c_{n,m}) (Q-\mathbf{e}t) }{ t^{\alpha +1}} dt\right|^{2} dQ    \right\}^{\frac{1}{2}}
$$
$$
  + \left\{\int\limits_{\Omega} \left|\int\limits_{0 }^{ r}\frac{(\rho c_{n,m})(Q-\mathbf{e}t)[1-(1- t/r)^{n-1} ]  }{ t^{1+\alpha   }} dt\right|^{2} dQ    \right\}^{\frac{1}{2}}=
    I_{01}+I_{02};
 $$
 $$
I_{01}\leq \sup\limits_{Q\in \Omega}|\rho(Q)|\left\{\int\limits_{\Omega}\left(\int\limits_{0 }^{ r}\frac{  |  c_{n,m} (Q)-    c_{n,m} (Q-\mathbf{e}t) |}{ t^{\alpha +1}} dt\right)^{2} dQ    \right\}^{\frac{1}{2}}
$$
$$
+\left\{\int\limits_{\Omega}\left|\int\limits_{0 }^{ r}\frac{ c_{n,m} (Q-\mathbf{e}t)  [  \rho (Q)-    \rho (Q-\mathbf{e}t)] }{ t^{\alpha +1}} dt\right|^{2} dQ    \right\}^{\frac{1}{2}}=
 I_{11}+I_{21} .
$$
Applying  the generalized Minkowski  inequality, then representing the function under the inner integral by the directional derivative,   we get
$$
I_{11}
 \leq C_{1}\int\limits_{ 0}^{ \mathrm{d}}t^{-\alpha -1}\left(\int\limits_{\Omega} \left| c_{n,m}(Q)-c_{n,m}(Q-\mathbf{e}t) \right|^{2}  dQ   \right)^{\frac{1}{2}} dt
$$
$$
=C_{1}\int\limits_{ 0}^{ \mathrm{d}}t^{-\alpha -1}\left(\int\limits_{\Omega} \left| \int\limits_{0}^{t} c'_{n,m} (Q-\mathbf{e}\tau) d\tau\right|^{2}   dQ\right)^{\frac{1}{2}} dt.
$$
Using the Cauchy-Schwarz inequality, the Fubini  theorem, we have
$$
I_{11} \leq C_{1}\int\limits_{ 0}^{ \mathrm{d}}t^{-\alpha -1}\left(\int\limits_{\Omega}dQ \int\limits_{0}^{t} \left|c'_{n,m} (Q-\mathbf{e}\tau)\right|^{2}d\tau \int\limits_{0}^{t}d\tau    \right)^{\frac{1}{2}} dt
$$

$$
  =C_{1}\int\limits_{ 0}^{ \mathrm{d}}t^{-\alpha -1/2 }\left( \int\limits_{0}^{t}d\tau \int\limits_{\Omega}\left|c'_{n,m} (Q-\mathbf{e}\tau)\right|^{2} dQ   \right)^{\frac{1}{2}} dt\leq C_{1}
 \frac{ \mathrm{d}^{1-\alpha}  }{1-\alpha  } \,  \|c'_{n,m} \|_{L_{2}(\Omega)}.
$$
 Arguing as above, using the Holder  property of the function $\rho,$ we see that
$$
I_{21}\leq M \int\limits_{ 0}^{ \mathrm{d}}t^{\lambda-\alpha -1}\left(\int\limits_{\Omega} \left|  c_{n,m}(Q-\mathbf{e}t) \right|^{2}  dQ    \right)^{\frac{1}{2}} dt\leq
M\frac{ \mathrm{d}^{\lambda-\alpha}  }{\lambda-\alpha } \,  \|c _{n,m} \|_{L_{2}(\Omega)}.
$$
It can be shown in the usual way that
$$
I_{02}\leq C_{1}\left\{\int\limits_{\Omega} \left|\int\limits_{0 }^{ r}  | c_{n,m}(Q-\mathbf{e}t)| \sum\limits_{i=0}^{n-2}\left(\frac{t}{r}\right)^{i}  r^{-1} t^{ -\alpha   }  dt\right|^{2} dQ    \right\}^{\frac{1}{2}}
$$
 $$
\leq C_{2}\left\{\int\limits_{\Omega} \left(\int\limits_{0 }^{ r} t^{ -\alpha   }dt \int\limits_{t}^{r}\left| c'_{n,m}(Q-\mathbf{e}\tau)\right| d\tau      \right)^{2}r^{-2}  dQ    \right\}^{\frac{1}{2}}
$$
$$
= C_{2}\left\{\int\limits_{\Omega} \left(\int\limits_{0 }^{ r}\left| c'_{n,m}(Q-\mathbf{e}\tau)\right| d\tau  \int\limits_{0}^{\tau}t^{ -\alpha   }dt      \right)^{2}r^{-2}  dQ    \right\}^{\frac{1}{2}}
$$
 $$
 \leq\frac{C_{2}}{1-\alpha}\left\{\int\limits_{\Omega} \left(\int\limits_{0 }^{ r}  \left| c'_{n,m}(Q-\mathbf{e}\tau)\right| \tau^{ -\alpha} d\tau      \right)^{2}   dQ    \right\}^{\frac{1}{2}}.
$$
Applying  the generalized Minkowski  inequality, we have
$$
I_{02}\leq C_{3}\int\limits_{0 }^{ \mathrm{d}} \tau^{ -\alpha} d\tau  \left(\int\limits_{\Omega} \left| c'_{n,m}(Q-\mathbf{e}\tau)\right|^{2} dQ     \right)^{\frac{1}{2}}   \leq
C_{3}\frac{\mathrm{d}^{1-\alpha}}{1-\alpha} \|c'_{n,m}\|_{L_{2}(\Omega)}.
$$
Consider   $I_{2},$  we have
$$
I_2  \leq\frac{C_{ 1}}{\Gamma(1-\alpha)}\left(\int\limits_{\Omega} \left| c_{n,m}(Q)\right|^{2}   r ^{-2\alpha  } dQ\right)^{ \frac{1}{2}}
=\frac{C_{ 1}}{\Gamma(1-\alpha)} \left(\int\limits_{\Omega}  r ^{-2\alpha  } \left|\int\limits_{0}^{r} c'_{n,m}(Q-\mathbf{e}t)dt\right|^{2}dQ     \right)^{\frac{1}{2}}
 $$
 $$
 \leq\frac{C_{ 1}}{\Gamma(1-\alpha)} \left(\int\limits_{\Omega}    \left|\int\limits_{0}^{r} c'_{n,m}(Q-\mathbf{e}t)t^{-\alpha}dt\right|^{2}dQ     \right)^{\frac{1}{2}}.
$$
Using the generalized Minkowski  inequality, then applying the  trivial estimates,  we get
$$
 I_2\leq  C_{ 4} \left\{\int\limits_{\omega}\left[\int\limits_{0}^{d} t^{-\alpha} dt  \left(\int\limits_{t}^{d}|c'_{n,m}(Q-\mathbf{e}t)|^{2}   r^{ n-1 }dr\right)^{\frac{1}{2}} \right]^{2}d\chi  \right\}^{\frac{1}{2}}
$$
$$
  \leq  C_{ 4} \left\{\int\limits_{\omega}\left[\int\limits_{0}^{\mathrm{d}} t^{-\alpha} dt  \left(\int\limits_{0}^{d}|c'_{n,m}(Q-\mathbf{e}t)|^{2}   r^{ n-1 }dr\right)^{\frac{1}{2}} \right]^{2}d\chi  \right\}^{\frac{1}{2}}
$$
$$
  =  C_{ 4}\int\limits_{0}^{\mathrm{d}} t^{-\alpha}   dt  \left(\int\limits_{\omega}d\chi\int\limits_{0}^{d}|c'_{n,m}(Q-\mathbf{e}t)|^{2}   r^{ n-1 }dr\right)^{\frac{1}{2}}       \leq  \frac{C_{ 4}\mathrm{d}^{1-\alpha}}{1-\alpha}  \|c'_{n,m}\|_{L_{2}(\Omega)}.
$$
  Taking into account that the sequences    $\{f_{n}\},\{f'_{n}\}$ are fundamental, we obtain       $I_{1},I_{2}\rightarrow 0.$ Hence the  sequence  $\{\varphi_{n}\} $ is fundamental and  $\varphi_{n}\stackrel{L_{2} }{\longrightarrow}\varphi\in L_{2}(\Omega).$  Note that by virtue  of  Theorem \ref{T1.1} the  directional fractional integral   operator  is bounded on the space $L_{2}(\Omega).$ Hence
$$
 \mathfrak{I }^{\alpha}_{0 +}\varphi_{n}\stackrel{L_{2} }{\longrightarrow}  \mathfrak{I} ^{\alpha}_{0+}\varphi.
$$
Combining this fact with   \eqref{1.20}, we have   $\rho f=  \mathfrak{I} ^{\alpha}_{0+}\varphi.$

\end{proof}
 \begin{lem}\label{L1.1}
 The operator $\mathfrak{D}^{ \alpha }$ is a restriction  of the  operator $\mathfrak{D}^{ \alpha }_{0+}.$
\end{lem}
\begin{proof}
It suffices  to show that the  next equality holds
\begin{equation}\label{1.21}
 (\mathfrak{D}^{ \alpha }f)(Q)= \left(\mathfrak{D}^{ \alpha }_{0+} f\right)(Q), \,f\in \dot{W}_{p}^{\,l}  (\Omega).
  \end{equation}
Using simple reasonings, we get
$$
 \mathfrak{D}^{ \alpha }v=\frac{\alpha}{\Gamma(1-\alpha)}\int\limits_{0}^{r} \frac{v(Q)-v(T)}{(r- t)^{\alpha+1}}  \left(\frac{t}{r}\right)  ^{n-1} dt+
  C_{n}^{(\alpha)}    v(Q)   r ^{ -\alpha}
$$
$$
= \frac{\alpha}{\Gamma(1-\alpha)}\int\limits_{0}^{r} \frac{r^{n-1}v(Q)-t^{n-1}v(T)}{r^{n-1}(r- t)^{\alpha+1}}dt-\frac{\alpha\,v(Q)}{\Gamma(1-\alpha)}\int\limits_{0}^{r}
\frac{ r^{n-1} -t^{n-1}  }{r^{n-1}(r- t)^{\alpha+1}}   dt
$$
$$
+C_{n}^{(\alpha)}   v(Q)   r ^{  -\alpha}=
  ( \mathfrak{D} ^{\alpha}_{0+}v)(Q)-
\frac{\alpha \, v(Q) }{\Gamma(1-\alpha)} \sum\limits_{i=0}^{n-2}  r ^{ -1-i}\int\limits_{0}^{r} \frac{t^{i}}{(r- t)^{\alpha }}   dt
$$
\begin{equation*}
+C_{n}^{(\alpha)}  v(Q)   r ^{  -\alpha}-\frac{v(Q)   r ^{  -\alpha}}{\Gamma(1-\alpha)}   =  ( \mathfrak{D} ^{\alpha}_{0+}v)(Q)- I_1 +I_2 -I_3.
\end{equation*}
Applying  the formula of the fractional integral of the power  function (2.44) \cite[p.47]{firstab_lit:samko1987}, we get
$$
I_1 =
\frac{\alpha\, v(Q)\,r ^{-1 } }{\Gamma(1-\alpha)}\int\limits_{0}^{r} \frac{dt}{(r- t)^{\alpha }}      +\frac{\alpha\, v(Q) }{\Gamma(1-\alpha)} \sum\limits_{i=1}^{n-2}
r ^{-1-i}\int\limits_{0}^{r} \frac{t^{i}}{(r- t)^{\alpha }}   dt
$$
$$
  =
 v(Q)\frac{\alpha }{ \Gamma(2-\alpha)}r ^{ -\alpha }    +  v(Q)  \alpha \sum\limits_{i=1}^{n-2}
r ^{-1-i} (I^{1-\alpha}_{0+}t^{i})(r)
$$
$$
=v(Q)\frac{\alpha }{ \Gamma(2-\alpha)}r ^{ -\alpha }    +  v(Q)  \alpha \sum\limits_{i=1}^{n-2}
r ^{ -\alpha}\frac{i!}{\Gamma(2-\alpha+i)}.
$$
Hence
$$
  I_1 +I_3
=\frac{v(Q)r ^{ -\alpha }}{ \Gamma(2-\alpha)}    +   v(Q)r ^{ -\alpha }  \alpha \sum\limits_{i=1}^{n-2}
 \frac{i!}{\Gamma(2-\alpha+i)}=
 \frac{2 v(Q)r ^{ -\alpha }}{ \Gamma(3-\alpha)}
 $$
 $$
 +   v(Q)r ^{ -\alpha } \alpha \sum\limits_{i=2}^{n-2}
 \frac{i!}{\Gamma(2-\alpha+i)}=
 \frac{3!v(Q)r ^{ -\alpha }}{ \Gamma(4-\alpha)}     +    v(Q)r ^{ -\alpha } \alpha \sum\limits_{i=3}^{n-2}
 \frac{i!}{\Gamma(2-\alpha+i)}
$$
\begin{equation*}
 =\frac{(n-2)!v(Q)r ^{ -\alpha }}{ \Gamma(n-1-\alpha)}     +    v(Q)r ^{ -\alpha } \alpha
 \frac{(n-2)!}{\Gamma(n-\alpha )}=C_{n}^{(\alpha)}v(Q)r ^{ -\alpha }.
\end{equation*}
Therefore $I_{2}-I_{1}-I_{3}=0$
and we  obtain   equality \eqref{1.21}. Let us prove that the considered operators do  not coincide with each other. For this purpose consider the function
  $f= \mathfrak{I}^{\alpha}_{0+}  \varphi,\;\varphi \in L_{p}(\Omega) ,$  then in accordance with  Theorem \ref{T1.2}, we have $ \mathfrak{D} _{0+} ^{ \alpha }  \mathfrak{I}^{\alpha}_{0+}   \varphi  =\varphi.$ Hence  $\mathfrak{I}^{\alpha}_{0+} \left( L_{p}\right)\subset\mathrm{D}\left(\mathfrak{D}^{ \alpha }_{0+}\right).$ Now,  it  suffices  to note  that
$
 \exists f\in \mathfrak{I}^{\alpha}_{0+} \left( L_{p}\right) ,\;f(\Lambda)\neq 0,
$
where $\Lambda \subset\partial\Omega,\;{\rm mess}\, \Lambda\neq 0.$
On the other hand, we know that
$
f(\partial \Omega)= 0\;a.e.,\;\forall f\in \mathrm{D}\left(\mathfrak{D}^{ \alpha }  \right) .
$

\end{proof}
\begin{lem}\label{L1.2} The following identity  holds
\begin{equation*}
   \mathfrak{D} _{0+} ^{ \alpha \ast }   = \mathfrak{D }^{\alpha}_{d-},
\end{equation*}
where    limits \eqref{1.7} are understood  in the sense of  $L_{s},\,s=p,q,\,1/p+1/q=1$ norm  respectively.
 \end{lem}
\begin{proof}
Let us  show that the following  relation holds
\begin{equation}\label{1.22}
 (\mathfrak{D}^{ \alpha }_{0+} f ,  g  )_{L_{2}(\Omega)}
=(f , \mathfrak{D }^{\alpha}_{d-} g )_{L_{2}(\Omega)},
\end{equation}
$$
 f\in \mathfrak{I}^{\alpha}_{0+} \left( L_{p}\right),\,g\in  \mathfrak{I } ^{\alpha}_{d-}\left(  L_{q}\right) .
$$
Note that by virtue of Theorem  \ref{T1.3}, we have    $\,\mathfrak{D}^{ \alpha }_{0+} \mathfrak{I}^{\alpha}_{0+} \varphi   =\varphi ,\,
\mathfrak{D }^{\alpha}_{d-}  \mathfrak{I } ^{\alpha}_{d-}   \psi  =\psi  ,$ where $\psi,\psi\in L_{s}(\Omega).$
Hence, using  Theorem \ref{T1.1}, we have  that the expressions at the   left-hand and right-hand sides of \eqref{1.22} are finite. Therefore, the conditions of the Fubini  theorem are satisfied  which  application gives us
$$
(\mathfrak{D}^{ \alpha }_{0+} f ,  g  )_{L_{2}(\Omega)}=\int\limits_{\omega}d\chi \int\limits_{0}^{d}\varphi(Q)
\overline{\left(\mathfrak{I } ^{\alpha}_{d-}\psi\right)(Q)}r^{n-1}dr
$$
$$
 =\frac{1}{\Gamma(\alpha)}\int\limits_{\omega}d\chi\int\limits_{0}^{d}\varphi(Q)r^{n-1}dr\int\limits_{r}^{d}\frac{\overline{\psi(T)}}{(t-r)^{1-\alpha}}\,dt
$$
\begin{equation*}
=\frac{1}{\Gamma(\alpha)}\int\limits_{\omega}d\chi\int\limits_{0}^{d}\overline{\psi(T)}t^{n-1}dt\int\limits_{0}^{t}\frac{\varphi(Q)}{(t-r)^{1-\alpha}}\left( \frac{r}{t}\right)^{n-1}dr
$$
$$
=\int\limits_{\Omega} \left(\mathfrak{I} ^{\alpha}_{0+} \varphi\right)(Q)\,\overline{\psi(Q)} \, dQ=(f , \mathfrak{D }^{\alpha}_{d-} g )_{L_{2}(\Omega)}.
\end{equation*}
Thus, inequality \eqref{1.22} is proved. It follows that   $ \mathfrak{D }^{\alpha}_{d-} \subset   \mathfrak{D} _{0+} ^{ \alpha \ast } .$  To establish the coincidence, in accordance with the definition of the  adjoint  operator, consider
\begin{equation*}
\left(\mathfrak{D}^{ \alpha }_{0+} f,g \right)_{L_{2}(\Omega)}=\left(   f,   g^{\ast}   \right)_{L_{2}(\Omega)},\,f\in \mathrm{D}(\mathfrak{D}^{ \alpha }_{0+}),\,g^{\ast}\in L_{q}(\Omega),
\end{equation*}
assuming that $g= \mathfrak{I } ^{\alpha}_{d-}g^{\ast},$ we have $g^{\ast}\in \mathrm{R}(\mathfrak{D }^{\alpha}_{d-}).$ Since  $\mathrm{R}(\mathfrak{D}^{ \alpha }_{0+})=L_{p}$ then     $\ker \mathfrak{D} _{0+} ^{ \alpha \ast }=0,$ the latter fact proves the coincidence.
\end{proof}

\section{ Strictly accretive  property}

It is remarkable that the term accretive,  which applicable to a linear operator $T$ acting in   Hilbert space $\mathfrak{H},$ is introduced by Friedrichs in the paper \cite{firstab_lit:fridrichs1958}, and means that the operator $T$ has the following property:   the numerical range $\Theta(T)$ (see \cite[p.335]{firstab_lit:kato1966}) is a subset of the right half-plane i.e.
    $$
    {\rm Re}\left( Tu,u\right)_{\mathfrak{H}}\geq0,\;u\in  \mathrm{D} (T),
    $$
the particular  case  corresponding to a much stronger condition given bellow
  $$
    {\rm Re}\left( Tu,u\right)_{\mathfrak{H}}\geq C\|u\|_{\mathfrak{H}}^{2},\;u\in  \mathrm{D} (T)
    $$
is known as a {\it strictly accretive} property  (see \cite[p. 352]{firstab_lit:kato1966}). The  following theorem establishes the strictly accretive  property    of the    Kipriyanov  fractional differential operator.\\

\begin{teo}\label{T1.5}
  Suppose  $\rho(Q)$ is a real non-negative function,   $\rho\in{\rm Lip}\, \lambda,\; \lambda>\alpha;$
  then the following inequality    holds
\begin{equation}\label{1.23}
 {\rm Re} ( f,\mathfrak{D}^{\alpha}f)_{L_2(\Omega,\rho)}\geq\mu\|f\|^{2}_{L_2(\Omega,\rho)},\;f\in H^{1}_{0} (\Omega),
\end{equation}
where
$$
\mu=     \frac{\Gamma^{-1}(1-\alpha) +C_{n}^{(\alpha)}}{2\mathrm{d}^{ \alpha}} -
\frac{\alpha M  \mathrm{d}^{\lambda-\alpha}  }{2\Gamma(1-\alpha)(\lambda-\alpha)\inf  \rho}\,.
$$
Moreover, if  we have in additional that for any fixed direction $\mathbf{e}$ the function  $\rho$ is   monotonically  non-increasing,   then
$$
\mu=\frac{\Gamma^{-1}(1-\alpha) +C_{n}^{(\alpha)}}{2\mathrm{d}^{ \alpha}}.
$$

\end{teo}
\begin{proof}
Consider a real  case and let $f\in C_{0}^{\infty}(\Omega),$ we have
\begin{equation*}
\rho(Q)f(Q)( \mathfrak{D}^{\alpha}f )(Q)=\frac{1}{2}( \mathfrak{D}^{\alpha}\rho f^{2} )(Q)
$$
$$
+\frac{\alpha}{2\Gamma(1-\alpha)} \int\limits_{0}^{r} \frac{\rho(Q)[f (Q)- f(T)]^{2}}{(r - t)^{\alpha+1}}
\left(\frac{t}{r}\right)^{n-1}\!\!\!dt
$$
$$
+\frac{\alpha}{2\Gamma(1-\alpha)} \int\limits_{0}^{r} \frac{f^{2}(T)[\rho (T)- \rho(Q)] }{(r - t)^{\alpha+1}}
\left(\frac{t}{r}\right)^{n-1}\!\!\!dt+ \frac{C_{n}^{\alpha}}{2}(\rho f^{2})(Q)r^{-\alpha}
$$
$$
=I_{0}(Q)+I_{1}(Q)+I_{2}(Q)+I_{3}(Q).
\end{equation*}
Applying Theorem \ref{T1.4}, we have
\begin{equation}\label{1.24}
\int\limits_{\Omega} I_{0}(Q)dQ=\frac{1}{2}\int\limits_{\Omega}(\mathfrak{D}^{\alpha}_{d-}1)(Q)   (\rho f^{2} )(Q)dQ
$$
$$
=\frac{1}{2\Gamma(1-\alpha)}\int\limits_{\Omega}(d(\mathbf{e})-r)^{-\alpha}    (\rho f^{2} )(Q)dQ\geq \frac{\mathrm{d}^{-\alpha}}{2\Gamma(1-\alpha)}\|f\|^{2}_{L_{2}(\Omega,\rho)}.
\end{equation}

Using the Fubini theorem, it can be shown in the usual way that
\begin{equation}\label{1.25}
\left|\int\limits_{\Omega} I_{2}(Q)dQ\right|\leq\frac{\alpha}{2\Gamma(1-\alpha)}\int\limits_{\omega}d \chi \int\limits_{0}^{d(\mathbf{e})}r^{n-1}dr \int\limits_{0}^{r} \frac{f^{2}(T)|\rho (T)- \rho(Q)| }{(r - t)^{\alpha+1}}
\left(\frac{t}{r}\right)^{n-1}\!\!\!dt
$$
$$
=\frac{\alpha}{2\Gamma(1-\alpha)}\int\limits_{\omega}d \chi \int\limits_{0}^{d(\mathbf{e})}f^{2}(T) t^{n-1}dt \int\limits_{t}^{d(\mathbf{e})} \frac{|\rho (T)- \rho(Q)| }{(r - t)^{\alpha+1}}
 dr
$$
$$
=\frac{\alpha}{2\Gamma(1-\alpha)}\int\limits_{\omega}d \chi \int\limits_{0}^{d(\mathbf{e})}f^{2}(T) t^{n-1}dt \int\limits_{0}^{d(\mathbf{e})-t}
\frac{|\rho (Q-\tau \mathbf{e})- \rho(Q)| }{\tau^{\alpha+1}}
 d\tau
$$
$$
\leq\frac{\alpha M}{2\Gamma(1-\alpha)}\int\limits_{\omega}d \chi \int\limits_{0}^{d(\mathbf{e})}f^{2}(T) t^{n-1}dt \int\limits_{0}^{d(\mathbf{e})-t} \tau^{\lambda-\alpha+1}
 d\tau
$$
$$
\leq\frac{\alpha M  \mathrm{d}^{\lambda-\alpha}}{2\Gamma(1-\alpha)(\lambda-\alpha)}\|f\|^{2}_{L_{2}(\Omega)} .
\end{equation}
Consider
\begin{equation}\label{1.26}
\int\limits_{\Omega} I_{3}(Q)dQ=C^{(\alpha)}_{n}\int\limits_{\Omega}(\rho f^{2})(Q)r^{-\alpha}dQ\geq\frac{ C^{(\alpha)}_{n} \mathrm{d}^{-\alpha}}{2}\|f\|^{2}_{L_{2}(\Omega,\rho)}.
\end{equation}
 Combining  \eqref{1.24},\eqref{1.25},\eqref{1.26}, and the fact that  $I_{1} $ is  non-negative, we obtain
\begin{equation}\label{1.27}
 ( f,\mathfrak{D}^{\alpha}f)_{L_2(\Omega,\rho)}\geq\mu\|f\|^{2}_{L_2(\Omega,\rho)},\;f\in C^{\infty}_{0} (\Omega).
\end{equation}
In the case when  for any fixed direction $\mathbf{e}$ the function  $\rho$ is   monotonically  non-increasing, we have    $I_{2}\geq 0.$ Hence \eqref{1.27} is   fulfilled. Now assume    that   $f\in H_{0}^{1} (\Omega),$ then  there exists   a sequence $\{f_k\}\in C^{\infty}_{0}(\Omega),\,
f_k\stackrel {H_{0}^1 }{\longrightarrow}f.
$
Using this fact,  it is not hard to prove that
$
f_k\stackrel {L_{2}(\Omega,\rho) }{\longrightarrow}f.
$
Using   inequality \eqref{1.4}, we prove  that
$
\|\mathfrak{D}^{\alpha} f \|_{L_{2}(\Omega,\rho)}   \leq C  \| f \|_{H_0 ^1(\Omega)}.
$
Therefore
$
 \mathfrak{D}^{\alpha}f_k\stackrel {L_{2}(\Omega,\rho) }{\longrightarrow} \mathfrak{D}^{\alpha}f.
$
Hence using the continuity property of the inner product, we get
$$
( f_k,\mathfrak{D}^{\alpha}f_k)_{L_{2}(\Omega,\rho)}\rightarrow ( f ,\mathfrak{D}^{\alpha}f )_{L_{2}(\Omega,\rho)} .
$$
Passing to the limit on the left and right side of inequality \eqref{1.27}, we obtain
\begin{equation}\label{1.28}
   ( f,\mathfrak{D}^{\alpha}f)_{L_2(\Omega,\rho)}\geq\mu\|f\|^{2}_{L_2(\Omega,\rho)},\;f\in H^{1}_{0} (\Omega).
\end{equation}
 Now let us  consider the  complex  case. Note that the following   equality is true
\begin{equation}\label{1.29}
{\rm Re} ( f,\mathfrak{D}^{\alpha}f)_{L_2(\Omega,\rho)}= ( u,\mathfrak{D}^{\alpha}u)_{L_2(\Omega,\rho)}+
( v,\mathfrak{D}^{\alpha}v)_{L_2(\Omega,\rho)},\;u={\rm Re}f,\,v={\rm Im}f.
\end{equation}

Combining   \eqref{1.29}, \eqref{1.28},  we obtain  \eqref{1.23}.
 \end{proof}
 \section{  Sectorial property }
Consider a uniformly elliptic operator with real   coefficients and the Kipriyanov fractional derivative    in the final term
\begin{equation*}
 Lu:=-  D_{j} ( a^{ij} D_{i}u)  +\rho\, \mathfrak{D}^{ \alpha }u,\;\;  i,j=1,2,...,n\,  ,
 \end{equation*}
 $$
 \; \mathrm{D}(L)=H^{2}(\Omega)\cap H^{1}_{0}(\Omega),
 $$
 \begin{equation*}
 a^{ij}(Q)\in C^{1}(\bar{\Omega})  ,\,a^{ij}\xi _{i}  \xi _{j}  \geq a_{0} |\xi|^{2} ,\,a_{0}>0,
\end{equation*}
 \begin{equation*}
\rho(Q)>0,\;\rho(Q)\in {\rm Lip\,\lambda},\,\alpha<\lambda\leq1.
\end{equation*}
We assume in additional that $\mu>0,$ here we use  the denotation  that is used in Theorem \eqref{T1.5}.
We also   consider the formal adjoint  operator
\begin{equation*}
 L^{+}u:=-  D_{i} ( a^{ij} D_{j}u)  + \mathfrak{D}  ^{\alpha}_{ d -}\rho u  ,\,\mathrm{D}(L^{+})=\mathrm{D}(L),
 \end{equation*}
and   the   operator
\begin{equation*}
  H=\frac{1}{2}(L+L^{+}).
 \end{equation*}
 We   use  a special case of the Green  formula
 \begin{equation}\label{1.30}
-\int\limits_{\Omega}D_{j}(a^{ij}D_{i}u)\,\bar{v}\, dQ=\int\limits_{\Omega}a^{ij}D_{i}u\, \overline{D_{j}v}\,  dQ\,,\;u\in H^{2}(\Omega),v\in H_{0}^{1}(\Omega) .
\end{equation}
\begin{remark}\label{R1.1}
  The operators $L,L^{\!+}\!,H$  are closeable. We can easily check this fact,  if we apply    Theorem 3.4 \cite[p.337]{firstab_lit:kato1966}.
\end{remark}
We have the following lemma.
\begin{teo}
\label{T1.6}
 The operators $\tilde{L},\,\tilde{L}^{+}$  are  strictly  accretive, their  numerical range  belongs to  the sector
\begin{equation*}
   \mathfrak{S}:= \{\zeta\in\mathbb{C}: \,|{\rm arg }\,(\zeta-\gamma)|\leq\theta\},
\end{equation*}
 where $\theta$ and $\gamma$ are defined by the coefficients of the operator $L.$
\end{teo}
\begin{proof}
Consider the operator $L .$ It is not hard to prove that
\begin{equation*}
-{\rm Re} \left( D_{j} [a^{ij} D_{i}f] ,f\right)_{L_{2}(\Omega)}\geq a_{0}\|   f  \|^{2}_{L^{1}_{2}(\Omega)},\;f\in \mathrm{D}(L).
\end{equation*}
Hence
\begin{equation}\label{1.31}
  {\rm Re}  (f_{n}, L f_{n}  )_{L_{2}(\Omega)}\geq a_{0}\|   f_{n} \|^{2}_{L^{1}_{2}(\Omega)}+ {\rm Re}(f_{n},\mathfrak{D}^{\alpha}f_{n})_{L_2(\Omega,\rho)} ,\;\{f_{n}\}\subset\mathrm{D}(L).
\end{equation}
Assume that  $ f\in \mathrm{D}(\tilde{L}).$ In accordance with the definition, there  exists a sequence $\{f_{n}\}\subset\mathrm{D}(L),\,f_{n}\xrightarrow[  L ]{}f.$  By virtue of  \eqref{1.31},  we easily  prove that  $f\in H_{0}^{1}(\Omega).$  Using the continuity property of the inner product, we pass to the limit on the left and right side of inequality \eqref{1.31}. Thus, we have
\begin{equation}\label{1.32}
 {\rm Re}( f , \tilde{L} f   )_{L_{2}(\Omega)}\geq a_{0}\|   f  \|^{2}_{L^{1}_{2}(\Omega)}+ {\rm Re}( f ,\mathfrak{D}^{\alpha}f )_{L_2(\Omega,\rho)} ,\; f  \in\mathrm{D}(\tilde{L}).
\end{equation}
 By virtue of Theorem \ref{T1.5}, we can rewrite the previous inequality as follows
 \begin{equation*}
  {\rm Re}( f,\tilde{L}f  )_{L_{2}(\Omega)}\geq a_{0}\|f\|^{2}_{L^{1}_{2}(\Omega)}+\mu\|f\|^{2}_{L_{2}(\Omega,\rho)} ,\;f\in \mathrm{D}(\tilde{L}).
\end{equation*}
Applying the     Friedrichs inequality    to the first summand of   the right side, we  get
  \begin{equation}\label{1.33}
  {\rm Re}  ( f,\tilde{L}f )_{L_{2}(\Omega)}\geq \mu_{1}\|f\|^{2}_{L_{2}(\Omega)},\;
 f  \in\mathrm{D}(\tilde{L}),\; \mu_{1}=a_{0}+\mu\inf\rho(Q)  .
 \end{equation}
Consider the  imaginary component   of the form, generated by the operator $L$
\begin{equation}\label{1.34}
\left|{\rm Im} ( f,Lf    )_{L_{2}(\Omega)}\right|\leq   \left|\int\limits_{\Omega}
\left(a^{ij}  D_{i}u D_{j}v-a^{ij}  D_{i}v D_{j}u\right)dQ\right|
$$
$$
+\left|  ( u,\mathfrak{D}^{\alpha}v )_{L_{2}(\Omega,\rho)} -( v,\mathfrak{D}^{\alpha}u   )_{L_{2}(\Omega,\rho)}\right|= I_{1}+I_{2}.
 \end{equation}
 Using the Cauchy Schwarz inequality for sums, the Young inequality,  we have
\begin{equation}\label{1.35}
a^{ij}  D_{i}u D_{j}v\leq a  |Du| |Dv| \leq
 \frac{a}{2}  \left(
|Du|^{2} +|Dv|^{2} \right),a(Q)\!=\!\left(\sum\limits_{i,j=1}^{n}
|a_{ij}(Q)|^{2} \right)^{\!\!1/2}\!\!.
\end{equation}
Hence
\begin{equation}\label{1.36}
 I_{1}\leq   a_{1}\|f\|^{2}_{L^{1}_{2}(\Omega)},\;a_{1}=\sup  a(Q)  .
\end{equation}
 Applying    inequality \eqref{1.4}, the Young inequality,    we get
\begin{equation}\label{1.37}
\left|( u,\mathfrak{D}^{\alpha}v )_{L_{2}(\Omega,p)}\right|\leq C_{1}\|u\|_{L_{2}(\Omega)}\|\mathfrak{D}^{\alpha}v\|_{L_{q}(\Omega)}
$$
$$
 \leq
C_{1} \|u\|_{L_{2}(\Omega)}\left\{\frac{K}{\delta^{\nu}}\|v\|_{L_{2}(\Omega)}+\delta^{1-\nu}\|v\|_{L^{1}_{2}(\Omega)} \right\}
$$
$$
\leq\frac{1}{\varepsilon} \|u\|^{2}_{L_{2}(\Omega)} + \varepsilon\left(\frac{ KC_{1}}{\sqrt{2}\delta^{\nu}}\right)^{2} \|v\|^{2}_{L_{2}(\Omega)}  +
 \frac{\varepsilon}{2}\left( C_{1}\delta^{1-\nu}\right)^{2} \|v\|^{2}_{L^{1}_{2}(\Omega)},
\end{equation}
where $ 2<q< 2n/(2\alpha-2+n).$
Hence
$$
 I_2   \leq\left|( u,\mathfrak{D}^{\alpha}v )_{L_{2}(\Omega,\rho)}\right|+\left|( v,\mathfrak{D}^{\alpha}u )_{L_{2}(\Omega,\rho)}\right|\leq \frac{1}{\varepsilon}\left(\|u\|^{2}_{L_{2}(\Omega)}+\|v\|^{2}_{L_{2}(\Omega)}\right)
 $$
 $$
 +\varepsilon\left(\frac{ KC_{1}}{\sqrt{2}\delta^{\nu}}\right)^{2}\left(\|u\|^{2}_{L_{2}(\Omega)}+\|v\|^{2}_{L_{2}(\Omega)} \right)+\frac{\varepsilon}{2}\left( C_{1} \delta^{1-\nu}\right)^{2}\left(\|u\|^{2}_{L^{1}_{2}(\Omega)}+\|v\|^{2}_{L^{1}_{2}(\Omega)} \right)
 $$
\begin{equation}\label{1.38}
= \left(\varepsilon \delta^{-2\nu}C_{2}  +\frac{1}{\varepsilon}\right) \|f\|^{2}_{L_{2}(\Omega)} +
 \varepsilon \delta^{2-2\nu} C_{3} \|f\|^{2}_{L^{1}_{2}(\Omega)}  .
\end{equation}
Taking into account  \eqref{1.34} and combining  \eqref{1.36},  \eqref{1.38}, we easily prove that
$$
\left|{\rm Im} ( f,\tilde{L}f    )_{L_{2}(\Omega)}\right|\leq
 \left(\varepsilon \delta^{-2\nu}C_{2}  +\frac{1}{\varepsilon}\right)   \|f\|^{2}_{L_{2}(\Omega)} + \left(\varepsilon \delta^{2-2\nu} C_{3}+a_{1}\right)   \|  f\|^{2}_{L_{2}^{1}(\Omega)}, f\in \mathrm{D}(\tilde{L}).
 $$
Thus  by virtue of    \eqref{1.33}   for an arbitrary number  $k>0,$  the next inequality holds
$$
{\rm Re}( f,\tilde{L}f    )_{L_{2}(\Omega)}-k \left|{\rm Im} ( f,\tilde{L}f    )_{L_{2}(\Omega)}\right|\geq
 \left(a_{0}-k    [\varepsilon\delta^{2-2\nu} C_{3}+a_{1}] \right)\|  f\|^{2}_{L_{2}^{1}(\Omega)}
 $$
 $$
 +\left( \mu \inf \rho(Q)- k\left[\varepsilon \delta^{-2\nu}C_{2}  +\frac{1}{\varepsilon}\right] \right)\|f\|^{2}_{L_{2}(\Omega)}.
$$
 Choose  $k= a_{0}\left(\varepsilon \delta^{2-2\nu} C_{3}+ a_{1} \right)^{-1}\!\!,$ we get
\begin{equation}\label{1.39}
\left|{\rm Im} ( f,(\tilde{L}-\gamma) f    )_{L_{2}(\Omega)}\right|   \leq \frac{1}{k}{\rm Re}( f,(\tilde{L}-\gamma)f    )_{L_{2}(\Omega)},
$$
$$
\;\gamma= \mu \inf \rho(Q)- k\left[\varepsilon \delta^{-2\nu}C_{2}  +\frac{1}{\varepsilon}\right].
\end{equation}
This  inequality shows  that the numerical range    $\Theta(\tilde{L})$ belongs to the sector with the top   $\gamma$ and the semi-angle $\theta=\arctan(1/k).$
The prove corresponding  to the operator $\tilde{L}^{+}$ is analogous.
 \end{proof}
We do not   study in detail   the conditions under which  $\gamma>0,$  but we   just   note that relation     \eqref{1.39} gives us an opportunity to formulate them in an easy way. Further,     we assume that the coefficients of the operator $L$ such that $\gamma>0.$

 \begin{teo}\label{T1.7}
   The operators $\tilde{L},\tilde{L}^{+},\tilde{H}$   is m-sectorial, the  operator $\tilde{H}$ is selfadjoint.
\end{teo}
\begin{proof} By virtue  of  Theorem \ref{T1.6} we have  that the operator $\tilde{L}$ is sectorial i.e.
 $\Theta(L)\subset  \mathfrak{S}.$  Applying Theorem 3.2   \cite[p. 336]{firstab_lit:kato1966}
 we   conclude that $\mathrm{R}(\tilde{L}-\zeta)$ is a closed  space  for any $\zeta\in \mathbb{C}\setminus\mathfrak{S}$ and that the next relation     holds
\begin{equation}\label{1.40}
  {\rm def}(\tilde{L}-\zeta)=\eta,\; \eta={\rm const} .
 \end{equation} Using   \eqref{1.33}, it is not hard to prove that   $
  \|\tilde{L}f\|_{L_{2}(\Omega)}\geq \sqrt{\mu_{1}}\|f\|_{L_{2}(\Omega)},\,
 f  \in\mathrm{D}(\tilde{L}).
$ Hence the inverse operator $(\tilde{L}+\zeta)^{-1}$ is defined  on the subspace $\mathrm{R}(\tilde{L}+\zeta),\,{\rm Re}\zeta>0.$     In accordance with    condition (3.38) \cite[p.350]{firstab_lit:kato1966},   we need to show  that
\begin{equation}\label{1.41}
{\rm def}(\tilde{L}+\zeta)=0,\;\|(\tilde{L}+\zeta)^{-1}\|\leq ({\rm Re}\zeta)^{-1},\;{\rm Re}\zeta>0 .
\end{equation}
Since $\gamma>0,$  then  the left half-plane is  included    in the the set $\mathbb{C}\setminus \mathfrak{S}.$
Note that by virtue of inequality  \eqref{1.33}, we have
 \begin{equation}\label{1.42}
  {\rm Re}  ( f,(\tilde{L}-\zeta )f  )_{L_{2}(\Omega)}\geq  (\mu- {\rm Re} \zeta ) \|f\|^{2}_{L_{2}(\Omega)} .
 \end{equation}
 Let $\zeta_{0}\in\mathbb{C}\setminus \mathfrak{S},\;{\rm Re}\zeta_{0} <0.$
  Since the  operator $\tilde{L}-\zeta_{0}$ has a closed range     $\mathrm{R} (\tilde{L}-\zeta_{0}),$ then we have
\begin{equation*}
 L_{2}(\Omega)=\mathrm{R} (\tilde{L}-\zeta_{0})\oplus \mathrm{R} (\tilde{L}-\zeta_{0})^{\perp} .
 \end{equation*}
Note that   $  C^{\infty}_{0}(\Omega)\cap \mathrm{R} (\tilde{L}-\zeta_{0})^{\perp}=0,$   because if we assume   the contrary, then applying inequality  \eqref{1.42}  for any element
 $u\in C^{\infty}_{0}(\Omega)\cap \mathrm{R}  (\tilde{L}-\zeta_{0})^{\perp},$    we get
 \begin{equation*}
(\mu-{\rm Re}\zeta_{0}) \|u\|^{2}_{L_{2}(\Omega)} \leq {\rm Re} ( u,(\tilde{L}-\zeta_{0})u  )_{L_{2}(\Omega)}=0,
 \end{equation*}
hence $u=0.$ Thus  this fact   implies that
$$
\left(g,v\right)_{L_{2}(\Omega)}=0,\,g\in  \mathrm{R}  (\tilde{L}-\zeta_{0})^{\perp},\, \in C^{\infty}_{0}(\Omega).
$$
Since $  C^{\infty}_{0}(\Omega)$  is a dense set in $L_{2}(\Omega),$ then $\mathrm{R}  (\tilde{L}-\zeta_{0})^{\perp}=0.$ It follows that  ${\rm def} (\tilde{L}-\zeta_{0}) =0.$ Now if we note \eqref{1.40}
then we  came to the conclusion that ${\rm def} (\tilde{L}-\zeta )=0,\;\zeta\in \mathbb{C}\setminus\mathfrak{S}.$  Hence ${\rm def} (\tilde{L}+\zeta )=0,\;{\rm Re}\zeta>0.$   Thus the proof  of  the first relation of  \eqref{1.41} is complete.
To prove the second relation \eqref{1.41} we should  note that
 \begin{equation*}
(\mu +{\rm Re}\zeta ) \|f\|^{2}_{L_{2}(\Omega)} \leq {\rm Re} ( f,(\tilde{L}+\zeta )f  )_{L_{2}(\Omega)}\leq \|f\|_{L_{2}(\Omega)}\|(\tilde{L}+\zeta )\|_{L_{2}(\Omega)},
 \end{equation*}
 $$
 \;f\in \mathrm{D}(\tilde{L}),\;
{\rm Re}\zeta>0 .
 $$
Using    first relation   \eqref{1.41}, we have
 \begin{equation*}
\|(\tilde{L}+\zeta )^{-1}g\|_{L_{2}(\Omega)}     \leq(\mu  +{\rm Re}\,\zeta ) ^{-1} \|g\|_{L_{2}(\Omega)}\leq ( {\rm Re}\,\zeta ) ^{-1} \|g\|_{L_{2}(\Omega)},\,g\in L_{2}(\Omega).
 \end{equation*}
 This implies that
$$
\|(\tilde{L}+\zeta )^{-1} \| \leq( {\rm Re}\,\zeta ) ^{-1},\,{\rm Re}\zeta>0.
$$
This concludes the proof  corresponding to the operator  $\tilde{L}.$  The proof  corresponding to the operator  $\tilde{L}^{+}$ is     analogous.
Consider the  operator $\tilde{H}.$ It is obvious that   $ \tilde{H}  $ is a symmetric operator. Hence   $ \Theta(\tilde{H})\subset \mathbb{R}.  $ Using \eqref{1.31} and arguing as above, we see that
 \begin{equation}\label{1.43}
    ( f,\tilde{H}f  )_{L_{2}(\Omega)}\geq  \mu_{1}  \|f\|^{2}_{L_{2}(\Omega)} .
 \end{equation}
 Continuing the used above  line of reasoning  and applying Theorem 3.2 \cite[p.336]{firstab_lit:kato1966},  we see that
  \begin{equation}\label{1.44}
{\rm def}(\tilde{H}-\zeta )=0,\,{\rm Im }\zeta\neq 0;
\end{equation}
  \begin{equation}\label{1.45}
   {\rm def}(\tilde{H}+\zeta)=0,\;\|(\tilde{H}+\zeta)^{-1}\|\leq ({\rm Re}\zeta)^{-1},\;{\rm Re}\zeta>0 .
\end{equation}
Combining \eqref{1.44} with   Theorem 3.16  \cite[p.340]{firstab_lit:kato1966}, we conclude that  the operator $\tilde{H}$ is selfadjoint.
Finally, note that in accordance with the definition,   relation  \eqref{1.45} implies   that the  operator $\tilde{H}$ is m-accretive.
  Since we  already know  that the operators $\tilde{L},\tilde{L}^{+},\tilde{H}$ are sectorial and m-accretive, then  in accordance with the definition they are m-sectorial.
 \end{proof}

\section{Compactness of the resolvent}

In this section we need using the theory of sesquilinear forms. If it is not stated otherwise, we use the definitions and  the  notation  of the monograph
\cite{firstab_lit:kato1966}.
Consider the forms
\begin{equation*}
t[u,v] = \int\limits_{\Omega} a^{ij}D_{i}u\, \overline{D_{j}v}dQ+ \int\limits_{\Omega}  \rho\,\mathfrak{D}^{\alpha }u \, \bar{v } dQ, \; u,v\in H^{1}_{0}(\Omega) ,
\end{equation*}
  \begin{equation*}
t^{*}[u,v]: =\overline{t [v,u]} =\int\limits_{\Omega} a^{ij}D_{j}u\, \overline{D_{i}v}dQ+ \int\limits_{\Omega}   u  \rho\,  \overline{\mathfrak{D}^{\alpha }v }dQ,
\end{equation*}
$$
\mathfrak{Re} t :=\frac{1}{2}(t+t^{*}).
$$
For convenience, we use the shorthand notation $h:=\mathfrak{Re} t.$

 \begin{lem}\label{L1.3}
The form $t$ is a closed sectorial form, moreover  $t=\mathfrak{\tilde{f}}, $ where
 \begin{equation*}
 \mathfrak{ f }[u,v]=(\tilde{L}u,v)_{L_{2}(\Omega)},\;u,v\in \mathrm{D}(\tilde{L}).
 \end{equation*}
\end{lem}
 \begin{proof}
First we shall show  that the following    inequality   holds
\begin{equation}\label{1.46}
C_{0}\|f\|^{2} _{H^{1}_{0}(\Omega)}\leq  \left|t[f ]\right|\leq C_{1}\|f\|^{2} _{H^{1}_{0}(\Omega)},\,f\in H^{1}_{0}(\Omega).
\end{equation}
Using \eqref{1.32}, Theorem \ref{T1.5},   we obtain
\begin{equation}\label{1.47}
C_{0}\|f\|^{2} _{H^{1}_{0}(\Omega)}\leq   {\rm Re} t[f ] \leq\left|t[f ]\right|,\;f\in H^{1}_{0} (\Omega).
\end{equation}
Applying \eqref{1.35},\eqref{1.37}, we get
\begin{equation}\label{1.48}
|t[f]|\leq\left|\left(a^{ij}D_if,D_jf\right)_{\!L_{2}(\Omega)}\!\right|+\left|\left(\rho \,\mathfrak{D}^{\alpha} f, f\right)_{\!L_{2}(\Omega)}\!\right| \leq C_{1}\|f\|^{2}_{H^{1}_{0}(\Omega)},\,f\in H^{1}_{0} (\Omega).
\end{equation}
Note  that $ H^{1}_{0}(\Omega) \subset \mathrm{D}( \tilde{t}) .$  If $f\in \mathrm{D}( \tilde{t} ),$ then in accordance with the definition,  there exists  a sequence
$
\{f_{n}\}\subset \mathrm{D}( t),\, f_{n}\xrightarrow[t ]{ }f.
$
Applying   \eqref{1.46}, we get  $ f_{n}\xrightarrow[  ]{ H^{1}_{0}}f .$
Since the space $H^{1}_{0}(\Omega)$ is complete, then   $\mathrm{D}( \tilde{t}) \subset H^{1}_{0}(\Omega).$ It implies that $\mathrm{D}( \tilde{t})=\mathrm{D}( t).$ Hence   $t$ is a closed form.
The proof of the sectorial property contains in the proof of Theorem \ref{T1.6}. Let us prove that $t=\mathfrak{\tilde{f}}.$  First, we shall    show that
\begin{equation}\label{1.49}
\mathfrak{ \mathfrak{f} }[u,v]=t[u,v],\;u,v\in \mathrm{D}(\mathfrak{f}).
\end{equation}
 Using   formula \eqref{1.30}, we have
\begin{equation}\label{1.50}
( L u ,v )_{L_{2}(\Omega)}=t[u ,v ],\;u,v\in \mathrm{D}(L).
 \end{equation}
 Hence  we can rewrite    relation \eqref{1.46} in the following form
\begin{equation}\label{1.51}
C_{0}\|f\|^{2} _{H^{1}_{0}(\Omega)}\leq  \left|( L f,f)_{L_{2}(\Omega)}\right|\leq C_{1}\|f\|^{2} _{H^{1}_{0}(\Omega)},\;f\in \mathrm{D}(L).
\end{equation}
 Assume that $f \in \mathrm{D}(\tilde{L}),$ then there exists a sequence $\{f_{n}\}\in \mathrm{D}( L ),\,f_{n}\xrightarrow[L]{}f.$   Combining   \eqref{1.51},\eqref{1.46}, we obtain
   $f_{n}\xrightarrow[t]{}f.$   These facts give us an opportunity to pass to the limit on the left and right side of \eqref{1.50}. Thus, we obtain \eqref{1.49}. Combining  \eqref{1.49},\eqref{1.46}, we get
\begin{equation*}
C_{0}\|f\|^{2} _{H^{1}_{0}(\Omega)}\leq  \left|\mathfrak{ f }[f ]\right|\leq C_{1}\|f\|^{2} _{H^{1}_{0}(\Omega)},\;f\in \mathrm{D}(\mathfrak{f}).
\end{equation*}
Note that by virtue  of Theorem \ref{T1.6} the operator $\tilde{L}$ is sectorial, hence due to Theorem 1.27 \cite[p.399] {firstab_lit:kato1966} the form $\mathfrak{f}$ is closable. Using the facts established above, Theorem 1.17 \cite[p.395] {firstab_lit:kato1966},   passing to the limit  on the left and right side of inequality \eqref{1.49}, we get
\begin{equation*}
\mathfrak{ \mathfrak{\tilde{f}} }[u,v]=t[u,v],\;u,v\in H_{0}^{1}(\Omega).
\end{equation*}
This concludes the proof.
 \end{proof}

\begin{lem}\label{L1.4}
The form h is a closed symmetric sectorial form, moreover  $h=\mathfrak{\tilde{k}}, $ where
 \begin{equation*}
 \mathfrak{ k }[u,v]=(\tilde{H}u,v)_{L_{2}(\Omega)},\;u,v\in \mathrm{D}(\tilde{H}).
 \end{equation*}
\end{lem}
\begin{proof}
To prove  the  symmetric  property   (see(1.5) \cite[p.387] {firstab_lit:kato1966}) of the form $h,$  it is  sufficient to note that
$$
h[u,v]=\frac{1}{2}\left(t[u,v]+ \overline{t[v,u]}  \right)=\frac{1}{2}\overline{\left( t[v,u] +\overline{t[u,v]}\right)}=\overline{h[v,u]},\;u,v\in \mathrm{D}(h).
$$
Obviously, we have
$
h[f]= {\rm Re}\,t[f].
$
Hence applying    \eqref{1.47},   \eqref{1.48},   we have
\begin{equation}\label{1.52}
C_{0}\|f\| _{H^{1}_{0}(\Omega)}\leq  h[f ] \leq C_{1}\|f\| _{H^{1}_{0}(\Omega)},\;f\in H^{1}_{0}(\Omega).
\end{equation}
 Arguing as above, using \eqref{1.52}, it is easy to  prove   that $\mathrm{D}(\tilde{h})=H_{0}^{1}(\Omega). $ Hence the form $h$ is a closed form. The  proof of the sectorial property of the form $h$ can be implemented due to the scheme of reasonings represented in   Theorem \ref{T1.6}, thus we left the technical repetition to the reader.
Let us  prove that $h=\mathfrak{\tilde{k}}.$ We shall  show that
\begin{equation}\label{1.53}
\mathfrak{ \mathfrak{k} }[u,v]=h[u,v],\,u,v\in \mathrm{D}(\mathfrak{k}).
\end{equation}
Applying \ref{L1.1}, Lemma \ref{L1.2}, we have
$$
(\rho\,\mathfrak{D}^{\alpha}f,g)_{L_{2}(\Omega)}=(f,\mathfrak{D}_{d-}^{\alpha}\rho g)_{L_{2}(\Omega)},\,f,g\in H_{0}^{1}(\Omega).
$$
Combining this fact with formula \eqref{1.30}, it is not hard to prove that
\begin{equation}\label{1.54}
( H u,v)_{L_{2}(\Omega)}=h[u,v],\;u,v\in \mathrm{D}(H).
\end{equation}
Using \eqref{1.54}, we can rewrite  estimate \eqref{1.52}  as  follows
\begin{equation}\label{1.55}
C_{0}\|f\| _{H^{1}_{0}(\Omega)}\leq    ( H f,f)_{L_{2}(\Omega)} \leq C_{1}\|f\| _{H^{1}_{0}(\Omega)},\;f\in \mathrm{D}(H).
\end{equation}
Note that in consequence of  Remark  \ref{R1.1} the  operator $H$ is closeable. Assume that  $f \in \mathrm{D}(\tilde{H}),$    then there  exists a sequence  $\{f_{n}\}\subset \mathrm{D}(H),\,f_{n}\xrightarrow[H]{}f .$   Combining  \eqref{1.55},\eqref{1.52}, we obtain
 $f_{n}\xrightarrow[h]{}f.$    Passing to the limit on the left and right side of \eqref{1.54}, we get \eqref{1.53}.
 Combining  \eqref{1.53},\eqref{1.52}, we obtain
\begin{equation*}
C_{0}\|f\| _{H^{1}_{0}(\Omega)}\leq   \mathfrak{ k }[f ] \leq C_{1}\|f\| _{H^{1}_{0}(\Omega)},\;f\in \mathrm{D}(\mathfrak{k}).
\end{equation*}
 Note that in consequence of Theorem \ref{T1.6} the operator $\tilde{H}$ is sectorial. Hence by virtue of    Theorem 1.27 \cite[p.399] {firstab_lit:kato1966} the form $\mathfrak{k}$ is closable.
Using the proven  above facts, Theorem 1.17 \cite[p.395] {firstab_lit:kato1966},  passing to the limits on the left and right side of inequality \eqref{1.53}, we get
\begin{equation*}
 \mathfrak{ \mathfrak{\tilde{k}} }[u,v]  = h[u,v],\,u,v\in H_{0}^{1}(\Omega).
\end{equation*}
This completes the proof.
\end{proof}

\begin{teo}\label{T1.8}
The operator $\tilde{H}$ has a compact resolvent,    the following estimate   holds
\begin{equation}\label{1.56}
\lambda_{n}(L_{0})\leq\lambda_{n}(\tilde{H})\leq\lambda_{n}(L_{1}),\,n\in\mathbb{N},
\end{equation}
where $\lambda_{n}(L_{k}),\,k=0,1$ are respectively  the eigenvalues of the following   operators with real  constant coefficients
\begin{equation*}
L_{k}f=-  a_{ k }^{ij} D _{j}D_{i}f +\rho_{k}f,\,\mathrm{D}(L_{k})=\mathrm{D}(L),
$$
$$
a_{ k }^{ij}\xi_{i}\xi_{j}>0,\,\rho_{k}>0.
\end{equation*}
\end{teo}
\begin{proof}
  First, we shall prove  the following      propositions \\
i) {\it The  operators $\tilde{H},L_{k}$ are positive-definite.} Using the fact that   the operator $H$ is selfadjoint, relation \eqref{1.43}, we conclude that
the  operator  $\tilde{H}$ is positive-definite. Using the definition, we can easily prove    that the operators $L_{k}$ are   positive-definite.   \\
ii) {\it The space $H^{1}_{0}(\Omega)$ coincides  with the energetic spaces  $\mathfrak{H}_{\tilde{H}},\mathfrak{H}_{L_{k}}$  as a set of   elements.} Using   Lemma \ref{L1.4}, we have
  \begin{equation}\label{1.57}
 \|f\| _{\mathfrak{H}_{\tilde{H}}}=\tilde{\mathfrak{k}}[f]=h[f],\;f\in H_{0}^{1}(\Omega).
\end{equation}
Hence the space  $\mathfrak{H}_{\tilde{H}}$ coincides with $H^{1}_{0}(\Omega)$ as a set of   elements. Using this fact,
we obtain   the coincidence of the spaces $H^{1}_{0}(\Omega)$ and $\mathfrak{H}_{L_{k}}$ as the particular case. \\
iii){\it We have the following estimates
\begin{equation}\label{1.58}
 \|f\| _{\mathfrak{H} _{L_{0}}}\leq \|f\| _{\mathfrak{H}_{\tilde{H}}}\leq  \|f\| _{\mathfrak{H} _{L_{1}}},\,f\in H^{1}_{0}(\Omega).
\end{equation}
}
We obtain the equivalence  of the   norms $\|\cdot\|_{H_{0}^{1}}$ and $\|\cdot\|_{\mathfrak{H} _{L_{k}}}$ as the particular case of relation \eqref{1.46}.
It is obvious that there   exist  such   operators  $L_{k}$    that the next  inequalities hold
\begin{equation}\label{1.59}
\|f\| _{\mathfrak{H} _{L_{0}}}\leq C_{0}\|f\| _{H^{1}_{0}(\Omega)},\;C_{1}\|f\| _{H^{1}_{0}(\Omega)}\leq\|f\| _{\mathfrak{H} _{L_{1}}},\,f\in H^{1}_{0}(\Omega).
\end{equation}
Combining \eqref{1.52},\eqref{1.57},\eqref{1.59}, we get   \eqref{1.58}.

 Now we can prove the   proposal  of this theorem.  Note that   the operators $\tilde{H},$ $L_{k}$  are positive-definite,    the   norms
 $\|\cdot\|_{H_{0}^{1}}, \|\cdot\|_{\mathfrak{H} _{L_{k}}}, \|\cdot\|_{\mathfrak{H}_{\tilde{H}}}$     are equivalent.   Applying  the  Rellich-Kondrashov  theorem, we have that the energetic  spaces
  $\mathfrak{H}_{\tilde{H}},\;\mathfrak{H} _{L_{k}}$  are compactly embedded into $L_{2}(\Omega).$ Using Theorem 3 \cite[p.216]{firstab_lit:mihlin1970}, we obtain  the fact  that  the operators $ L_{0} ,L_{1},\tilde{H}$ have  a discrete   spectrum.
 Taking into account   (i),(ii),(iii), in accordance with the definition \cite[p.225]{firstab_lit:mihlin1970}, we have
$$
L_{0}\leq \tilde{H} \leq L_{1}.
$$
 Applying     Theorem 1 \cite[p.225]{firstab_lit:mihlin1970},   we obtain  \eqref{1.56}.
 Note that by virtue of Theorem \ref{T1.7} the operator $\tilde{H}$ is m-accretive. Hence $0\in P(\tilde{H}).$ Due to    Theorem 5 \cite[p.222]{firstab_lit:mihlin1970} the operator $\tilde{H}$ has a compact resolvent at the point zero.
 Applying    Theorem 6.29 \cite[p.237]{firstab_lit:kato1966}, we conclude that the  operator $\tilde{H}$ has a compact resolvent.

\end{proof}

\begin{teo}\label{T1.9}
Operator $\tilde{L}$ has a compact resolvent,   discrete spectrum.
\end{teo}
\begin{proof}
Note that in accordance with Theorem \ref{T1.7} the operators $\tilde{L},\tilde{H}$ are m-sectorial, the operator $\tilde{H}$  is self-adjoint. Applying Lemma \ref{L1.3}, Lemma \ref{L1.4}, Theorem 2.9 \cite[p.409]{firstab_lit:kato1966}, we get   $T_{t}=\tilde{L},\;T_{h}=\tilde{H},$ where $T_{t},T_{h}$ are the Fridrichs  extensions    of the operators $\tilde{L}, \tilde{H}$ (see \cite[p.409]{firstab_lit:kato1966}) respectively. Since in accordance
with  the   definition \cite[p.424]{firstab_lit:kato1966} the operator $\tilde{H}$ is a real part of the operator $\tilde{L},$ then due to  Theorem
\ref{T1.8}, Theorem 3.3 \cite[p.424]{firstab_lit:kato1966} the operator $\tilde{L}$ has a compact resolvent. Applying  Theorem 6.29 \cite[p.237]{firstab_lit:kato1966}, we conclude that the operator $\tilde{L}$ has a     discrete   spectrum.
\end{proof}

\section{Existence   theorems via the Lax-Milgram method}\label{S1.6}

Consider  a boundary value problem for a differential  equation of the fractional order, containing an  uniformly elliptic operator with real-valued   coefficients  in the left-hand side and fractional derivative in the Kipriyanov  sense of   in   lower terms
\begin{equation}\label{1.60}
  Lu:=-  D_{j} ( a^{ij} D_{i}u)  +p\, \mathfrak{D}^{ \alpha }u=f\in L_{2}(\Omega),\;\;i,j=1,2,...,n,
 \end{equation}
 \begin{equation}\label{1.61}
 \;  u\in H^{2}(\Omega)\cap H^{1}_{0}(\Omega),
 \end{equation}
 \begin{equation}\label{1.62}
 a^{ij}(Q)\in C^{1}(\bar{\Omega})  ,\;a^{ij}\xi _{i}  \xi _{j}  \geq a_{0}  |\xi|^{2},\,a_{0}>0,\;p(Q)>0,\;p(Q)\in {\rm Lip\,\lambda},\, \lambda>\alpha.
\end{equation}
We will use  a special case of the Green's formula
\begin{equation}\label{1.63}
-\int\limits_{\Omega}v\,\overline{D_{j}(a^{ij}D_{i}u)}\, dQ=\int\limits_{\Omega}a^{ij}D_{j}v\, \overline{D_{i}u}\,  dQ\,,\;u\in H^{2}(\Omega),v\in H_{0}^{1}(\Omega) .
\end{equation}
Further,   we  need the  following lemma.
\begin{lem}\label{L1.5}
Let $u,v\in L_{2}(\Omega),\;  {\rm dist}\,({\rm supp}\,v,\,\partial\Omega)>2|h|,$ then we have the following formula
 \begin{equation}\label{1.64}
\int\limits_{\Omega}\triangle^{ h}_{k}v\,\overline{u}\,dQ =
-\int\limits_{\Omega}v\, \triangle^{ -h}_{k}\overline{u} \,dQ.
\end{equation}
\end{lem}
\begin{proof}
Under the lemma  assumptions, we have the following reasonings
$$
\int\limits_{\Omega}\triangle^{ h}_{k}v\,\overline{u}\,dQ=\frac{1}{h}\int\limits_{\Omega} \left[v(Q+e_{k}h)-v(Q)\right]\,\overline{u(Q)}\,dQ =
 $$
$$
= \frac{1}{h}\int\limits_{\omega}d\chi \int\limits_{0}^{r}  v(P'+\mathbf{\bar{e}}r ) \,\overline{u(P'+\mathbf{\bar{e}}r-e_{k}h)}\,r^{n-1}dr-\frac{1}{h}\int\limits_{\Omega}  v(Q) \,\overline{u(Q)}\,dQ=
$$
$$
=\frac{1}{h}\int\limits_{\Omega'}   v(Q' ) \,\overline{u(Q'-e_{k}h)}\,dQ'-\frac{1}{h}\int\limits_{\Omega}  v(Q) \,\overline{u(Q)}\,dQ,\;P'=P+e_{k}h,\;Q'=P'+\mathbf{\bar{e}}r,
$$
where $\Omega'$ shift of the domain $\Omega$ on the distance $h$ in the direction $e_{k}.$
Note that   in consequence of  the  condition imposed upon    the set  ${\rm supp}\,u,$ we have: ${\rm supp}\,u_{1}\subset \Omega\cap \Omega',\;u_{1}(Q')=u(Q'-e_{k}h) .$ Hence, we can rewrite the last relation as follows
 \begin{equation*}
\int\limits_{\Omega}\triangle^{ h}_{k}v\,\overline{u}\,dQ=\frac{1}{h}\int\limits_{\Omega}   v(Q ) \overline{\left[ u(Q-e_{k}h)  -u(Q) \right]}\,dQ=
-\int\limits_{\Omega}v\,\triangle^{ -h}_{k}\overline{u}\,dQ.
\end{equation*}
\end{proof}

The  existence and uniqueness theorems    proved further   based upon the results obtained in the papers \cite{firstab_lit:1kukushkin2018}, \cite{firstab_lit:2kukushkin2017}.\\

Consider the boundary value problem \eqref{1.60},\eqref{1.61}.
The proved   strictly  accretive property of    fractional differential operators  allows  by application of the  Lax-Milgram theorem     to prove  an existence and uniqueness  theorem for the  problem. Before   formulating the main statement, consider the following definition

\begin{deff}\label{D1.1}
We will call the element $z\in H^{1}_{0}(\Omega) $ by a generalized solution of the boundary value problem \eqref{1.60},\eqref{1.61} if the following integral identity holds
\begin{equation}\label{1.65}
 B(v,z)= (v,f)_{L_{2}(\Omega)}  ,\;\forall v\in H^{1}_{0}(\Omega),
 \end{equation}
 where
\begin{equation*}
 B (v,u)= \int\limits_{\Omega} \left[ a^{ij}D_{j}v \overline{D_{i}u}  +      (\mathfrak{D}^{\alpha}_{d-}p\,v)  \, \overline{u} \right]\,dQ ,\;u,v\in H^{1}_{0}(\Omega).
\end{equation*}
\end{deff}

\begin{teo}\label{T1.10}

 There exists a  unique   generalized solution  of the   boundary value problem   \eqref{1.60},\eqref{1.61}.
\end{teo}
\begin{proof}
We will show that   form \eqref{1.65} satisfies the conditions of the Lax-Milgram theorem, particulary    we will show that  the next inequalities hold
  \begin{equation}\label{1.66}
  |B (v,u)|\leq K_{1}\|v\|_{H^{1}_{0}(\Omega)}\|u\|_{H^{1}_{0}(\Omega)} ,\;\;{\rm Re}\,B (v,v)\geq  K_{2}\|v\|^{2}_{H^{1}_{0}(\Omega) },\;u,v\in H^{1}_{0}(\Omega),
 \end{equation}
where  $K_{1}>0,\;K_{2}>0$ are constants independent on  the real functions  $u,v.$\\
Let us prove the first inequality   \eqref{1.66}.
 Using the Cauchy-Schwarz inequality for a sum,  we have
\begin{equation*}
a^{ij}  D_{j}v \overline{D_{i}u}\leq a(Q) |Dv||Du| ,\;a(Q)=\left(\sum\limits_{i,j=1}^{n}
|a_{ij}(Q)|^{2} \right)^{1/2}.
\end{equation*}
Hence
\begin{equation}\label{1.67}
 \left| \int\limits_{\Omega} a^{ij} D_{j}v \overline{D_{i}u}\, dQ\right|\leq    P   \|v\|_{H^{1}_{0}(\Omega)}\|u\|_{H^{1}_{0}(\Omega)},\;P=\sup\limits_{Q\in \Omega}|a(Q)|.
\end{equation}
In consequence of Lemma 1 \cite{firstab_lit:1kukushkin2018},  Lemma 2 \cite{firstab_lit:1kukushkin2018},    we have
\begin{equation}\label{1.68}
 (\mathfrak{D} ^{\alpha}_{d-}p\, v ,u    )_{L_{2}(\Omega )} =  ( v,\mathfrak{D}^{ \alpha }u)_{L_{2}(\Omega,p)},\;u,v\in H^{1}_{0}(\Omega)  .
\end{equation}
 Applying   inequality \eqref{1.3}, then  Jung's inequality    we get
\begin{equation*}
\left|( v,\mathfrak{D}^{ \alpha }u)_{L_{2}(\Omega,p)}\right|     \leq C_{0}
  \|v\|_{L_{2}(\Omega)}\|\mathfrak{D}^{\alpha}u\|_{L_{q}(\Omega)} \leq
C_{0}  \|v\|_{L_{2}(\Omega)}\left\{\frac{K}{\delta^{\nu}}\|u\|_{L_{2}(\Omega)}+\delta^{1-\nu}\|u\|_{L^{1}_{2}(\Omega)} \right\} \leq
$$
$$
\leq\frac{1}{\varepsilon} \|v\|^{2}_{L_{2}(\Omega)} + \varepsilon\left(\frac{ KC_{0} }{\sqrt{2}\delta^{\nu}}\right)^{2} \|u\|^{2}_{L_{2}(\Omega)}  +
 \frac{\varepsilon}{2}\left( C_{0} \delta^{1-\nu}\right)^{2} \|u\|^{2}_{L^{1}_{2}(\Omega)}  ,
\end{equation*}
\begin{equation*}
 2<q<\frac{2n}{2\alpha-2+n},\;C_{0} = ({\rm mess}\,\Omega)^{\frac{q-2}{q}}\sup\limits_{Q\in\Omega}p(Q).
 \end{equation*}
Applying the Friedrichs inequality, finally we have the  following estimate
\begin{equation}\label{1.69}
| (\mathfrak{D} ^{\alpha}_{d-}p\, v ,u    )_{L_{2}(\Omega )} |\leq C\|v\|_{H^{1}_{0}}\|u\|_{H^{1}_{0}}.
 \end{equation}
Note that   the first inequality   \eqref{1.66} follows   from inequalities  \eqref{1.67},\eqref{1.69}.   Using  inequalities (28)\,\cite{firstab_lit:1kukushkin2018}, (36)\,\cite{firstab_lit:1kukushkin2018}, we have
\begin{equation}\label{1.70}
{\rm Re}\,B(v,v)\geq a_{0}\|v\|^{2}_{L_{2}^{1}(\Omega)}+\lambda^{-2}\|v\|^{2}_{L_{2}(\Omega,p)}\geq a_{0}\|v\|^{2}_{L_{2}^{1}(\Omega)}+\lambda^{-2}p_{0}\|v\|^{2}_{L_{2}(\Omega)},\;p_{0}=\inf\limits_{Q\in \Omega} p(Q).
\end{equation}
It is obvious  that
\begin{equation*}
a_{0}\|v\|^{2}_{L_{2}^{1}(\Omega)}+\lambda^{-2}p_{0}\|v\|^{2}_{L_{2}(\Omega)}\geq K_{2}\left(\|v\|^{2}_{L_{2}^{1}(\Omega)}+
\|v\|^{2}_{L_{2}(\Omega)} \right)=
\end{equation*}
\begin{equation}\label{1.71}
=K_{2}\left(\int\limits_{\Omega}\sum\limits_{i=1}^{n}|D_{i}v|^{2}dQ+\int\limits_{\Omega} | v|^{2}dQ\right)=K_{2}\|v\|^{2}_{H_{0}^{1}(\Omega)},\;K_{2}=\min\{a_{0},\lambda^{-2}p_{0}\}.
\end{equation}
Hence  the second inequality   \eqref{1.66} follows from    inequalities \eqref{1.70}, \eqref{1.71}.

Since the conditions of Lax-Milgram theorem holds, then for all bounded on $H^{1}_{0}(\Omega) $ functional  $F,$  exists a  unique element $z\in H^{1}_{0}(\Omega) $ such that
 \begin{equation}\label{1.72}
   B (v,z)= F(v)  ,\;\forall v\in H^{1}_{0}(\Omega).
 \end{equation}
 Consider a functional
 \begin{equation}\label{1.73}
F(v)=(v,f)_{L_{2}(\Omega)},\,f\in L_{2}(\Omega),\,v\in H^{1}_{0}(\Omega).
 \end{equation}
Applying the  Cauchy-Schwarz inequality, we get
\begin{equation*}
|F(v)|= |( v,f)_{L_{2}(\Omega)}|\leq \|f\|_{L_{2}(\Omega)}\|v\| _{H_{0}^{1}(\Omega)}.
\end{equation*}
Hence  functional \eqref{1.73} is bounded on $H^{1}_{0}(\Omega),$ then in accordance with \eqref{1.72} we have the  equality
 \begin{equation}\label{1.74}
   B ( v,z)= ( v,f)_{L_{2}(\Omega)}  ,\;\forall v\in H^{1}_{0}(\Omega).
 \end{equation}
 Therefore, in accordance with Definition  \ref{D1.1}  the element  $z$ is a  unique generalized solution of the   boundary value problem \eqref{1.60},\eqref{1.61}.
\end{proof}

The theorem \ref{T1.10} allows to prove the existence and uniqueness theorem for the boundary value problem   \eqref{1.60},\eqref{1.61}.

 \begin{teo}\label{T1.11}
 There exists a  unique strong solution  of the boundary value  problem \eqref{1.60},\eqref{1.61}.
\end{teo}
\begin{proof}

In consequence of  Theorem \ref{T1.10} there  exists  a unique element $z\in H^{1}_{0}(\Omega),$ so that equality  \eqref{1.74} holds.
Note that if the generalized  solution   of boundary value problem \eqref{1.60},\eqref{1.61} belongs to Sobolev space $H^{2}(\Omega),$ then
applying formulas \eqref{1.63},\eqref{1.68} we get
\begin{equation*}
 (  v,Lz )_{L_{2}(\Omega )} = B( v,z)=( v,f)_{L_{2}(\Omega )},\;\forall v\in C_{0}^{\infty}(\Omega),
\end{equation*}
hence
\begin{equation*}
 (  v,Lz-f )_{L_{2}(\Omega )} = 0,\;\forall v\in C_{0}^{\infty}(\Omega).
\end{equation*}
Using  the well-known fact that  there does not exist  a  non-zero element in the Hilbert space  orthogonal to a dense  manifold, we conclude that    $z$ is solution of the boundary value problem\eqref{1.60},\eqref{1.61}.

Let's prove that $z\in H^{2} (\Omega).$ Choose the function $v$ in \eqref{1.74} so that $\overline{({\rm supp} \,v)}     \subset \Omega,$   implementing  easy calculation, using equality \eqref{1.68}, we get
\begin{equation}\label{1.75}
 \int\limits_{\Omega} a^{ij}  D_{j} v  \overline{D_{i} z}\, dQ = \int\limits_{\Omega}   v \overline{q}\, dQ ,\; \forall v\in H^{1}_{0}(\Omega),\,   \overline{({\rm supp} \,v)}     \subset \Omega,
\end{equation}
where $q=(f-p\,  \mathfrak{D}^{ \alpha }z).$
For $2|h|<  {\rm dist}\,({\rm supp}\,v,\,\partial\Omega),$ let us  change the function $v$ on its  difference attitude $\triangle^{-h} v=\triangle^{-h}_{k}v $ for some $1\leq k\leq n.$ Using \eqref{1.64},   \eqref{1.75}, we get
\begin{equation*}
 \int\limits_{\Omega}  D_{j} v\,\overline{\triangle^{ h}\left( a^{ij}  D_{i} z\right)} dQ=-\int\limits_{\Omega}  (D_{j}\triangle^{-h} v)\overline{a^{ij}D_{i} z}\, dQ=-\int\limits_{\Omega} a^{ij} (D_{j}\triangle^{-h} v)\overline{D_{i} z}\, dQ= -\int\limits_{\Omega}(\triangle^{-h}v)\, \overline{q} \, dQ.
\end{equation*}
Using elementary calculation, we get
$$
\triangle^{ h}\left( a^{ij}  D_{i} z\right)(Q)=a^{ij}(Q+h \,\vec{\mathrm{e}}_{k})  (D_{i} \triangle^{ h} z )(Q)+[\triangle^{ h} a^{ij}(Q) ]( D_{i} z)(Q),
$$
hence
\begin{equation*}
 \int\limits_{\Omega}D_{j} v\, \overline{a^{ij}(Q+h \,\vec{\mathrm{e}}_{k})  (D_{i}\triangle^{ h}  z)} \, dQ=      -\int\limits_{\Omega}    Dv \cdot g +(\triangle^{-h}v)\,\overline{q}  \, dQ,
\end{equation*}
where $g=(g_{1},g_{2},...,g_{n}),\;g_{j}=(\triangle^{ h} a^{ij})   D_{i} z.$
Using the  last relation,   the Cauchy Schwarz inequality, finiteness property of the function $v,$ Lemma 7.23 \cite[p.164]{firstab_lit:gilbarg}, we have
\begin{equation}\label{1.76}
\left|\int\limits_{\Omega}a^{ij}(Q+h \,\vec{\mathrm{e}}_{k}) D_{j} v\, \overline{  (D_{i}\triangle^{ h}  z)}\, dQ \right| = \left|\int\limits_{\Omega} D_{j} v\, \overline{a^{ij}(Q+h \,\vec{\mathrm{e}}_{k})  (D_{i}\triangle^{ h}  z)}\, dQ \right| \leq
 $$
 $$
\leq\|Dv\|_{L_{2}(\Omega)} \|g\|_{L_{2}(\Omega)}+\|\triangle^{-h}v\|_{L_{2}(\Omega)}\|q\|_{L_{2}(\Omega)} \leq\|Dv\|_{L_{2}(\Omega)}\left(\|g\|_{L_{2}(\Omega)}+\|q\|_{L_{2}(\Omega)}\right)  .
\end{equation}
Applying  the Cauchy Schwarz inequality for   sums and integrals, it is easy to see that
$$
\|g\|_{L_{2}(\Omega)}=\left(\int\limits_{\Omega}\sum\limits_{j=1}^{n}|(\triangle^{ h} a^{ij})   D_{i} z|^{2}dQ\right)^{1/2}\leq \left(\int\limits_{\Omega}|Dz|^{2}\sum\limits_{i,j=1}^{n}| \triangle^{ h} a^{ij} |^{2}dQ\right)^{1/2}\leq
$$
$$
\leq \sup\limits_{Q\in \Omega}
\left(\sum\limits_{i,j=1}^{n}\left| \triangle^{ h} a^{ij}(Q) \right|^{2}\right)^{1/2}\left(\int\limits_{\Omega}|Dz|^{2}dQ\right)^{1/2}\leq C_{1} \|z\|_{H^{1}(\Omega)}.
$$
  Using \eqref{1.3}, we have
$$
\|q\|_{L_{2}(\Omega)}\leq \|f\|_{L_{2}(\Omega)}+   \|p\,\mathfrak{D}^{ \alpha }z\|_{L_{2}(\Omega)} \leq \|f\|_{L_{2}(\Omega)}+C_{2}
  \|z\|_{H^{1} (\Omega)}.
 $$
In accordance with the  above, using  \eqref{1.76}, we get
\begin{equation}\label{1.77}
 \left|\int\limits_{\Omega}a^{ij}(Q+h \,\vec{\mathrm{e}}_{k}) D_{j} v\, \overline{  (D_{i}\triangle^{ h}  z)}\, dQ \right| \leq C\left(\|z\|_{H^{1} (\Omega)}+\|f\|_{L_{2}(\Omega)}   \right)\|D  v\|_{L_{2}(\Omega)}.
\end{equation}
Applying condition  \eqref{1.62}, we obtain the following estimate
\begin{equation}\label{1.78}
\left| \int\limits_{\Omega}  a^{ij} \xi_{j}\overline{\xi_{i}}\, dQ\right| =
\left|\int\limits_{\Omega}  a^{ij}(  {\rm Re}\xi_{j}\,{\rm Re}\xi_{i}+  {\rm Im} \xi_{j} \,{\rm Im} \xi_{i} )\, dQ+
\imath \int\limits_{\Omega}  a^{ij}(  {\rm Re}\xi_{i}\,{\rm Im}\xi_{j}-  {\rm Re} \xi_{j} \,{\rm Im} \xi_{i} )\, dQ\right| =
$$
$$
=\left\{\left(\int\limits_{\Omega}  a^{ij}(  {\rm Re}\xi_{j}\,{\rm Re}\xi_{i}+  {\rm Im} \xi_{j} \,{\rm Im} \xi_{i} )\, dQ\right)^{2}+
 \left( \int\limits_{\Omega}  a^{ij}(  {\rm Re}\xi_{i}\,{\rm Im}\xi_{j}-  {\rm Re} \xi_{j} \,{\rm Im} \xi_{i} )\, dQ\right)^{2}\right\}^{1/2}\geq
 $$
 $$
 \geq \int\limits_{\Omega}  a^{ij}(  {\rm Re}\xi_{j}\,{\rm Re}\xi_{i}+  {\rm Im} \xi_{j} \,{\rm Im} \xi_{i} )\, dQ\geq
 k_{0} \int\limits_{\Omega}
 \left | \xi \right|^{2}\, dQ.
\end{equation}
Define the function $\chi,$ so that $  {\rm dist}\,({\rm supp}\,\chi,\,\partial\Omega)>2|h|,$
\begin{equation*}
\chi(Q)=\left\{ \begin{aligned}
 1,\;\;\;\;\;\;\;Q\in {\rm supp  }\,\chi,\\
  0,\;Q\in  \bar{\Omega} \setminus {\rm supp  }\,\chi .\\
\end{aligned}
 \right.
 \end{equation*}
Suppose   $v=\chi\triangle^{h}z,$ using relations \eqref{1.77}, \eqref{1.78}, we have an  estimate
\begin{equation}\label{1.79}
k_{0}\|\chi\triangle^{ h} D   z\|^{2}_{L_{2}(\Omega)}\leq\left| \int\limits_{\Omega}\chi  a^{ij}(Q+h \,\vec{\mathrm{e}}_{k}) \triangle^{ h}  D_{j} z \overline{\triangle^{ h}D_{i}  z}\, dQ\right|=
$$
$$
 =\left| \int\limits_{\Omega} a^{ij}(Q+h \,\vec{\mathrm{e}}_{k})     D_{j} (\chi\triangle^{h}z)\, \overline{ (D_{i}\triangle^{ h}  z)}\, dQ\right|\leq
$$
$$
\leq
C\left(\|z\|_{H^{1} (\Omega)}+\|f\|_{L_{2}(\Omega)}   \right)\|\chi\triangle^{ h} D   z\|_{L_{2}(\Omega)}.
\end{equation}
Using the  Jung's inequality, for all positive $k,$ we  get an estimate
\begin{equation*}
2\left(\|z\|_{H^{1} (\Omega)}+\|f\|_{L_{2}(\Omega)}   \right)\|\chi\triangle^{ h} D   z\|_{L_{2}(\Omega)}\leq
\frac{1}{k} \left(\|z\|_{H^{1} (\Omega)}+\|f\|_{L_{2}(\Omega)}   \right)^{2}+ k \|\chi\triangle^{ h} D   z\|^{2}_{L_{2}(\Omega)} .
\end{equation*}
Choosing $k<2k_{0}C^{-1},$ we can represent  inequality \eqref{1.79} as follows
\begin{equation*}
 \|\chi\triangle^{ h} D   z\|^{2}_{L_{2}(\Omega)}\leq
C_{1}\left(\|z\|_{H^{1} (\Omega)}+\|f\|_{L_{2}(\Omega)}   \right)^{2} .
\end{equation*}
It implies that for a domain $\Omega'$ such that  ${\rm dist}(\Omega',\partial \Omega)>2|h|,$ we have
\begin{equation*}
\|\triangle^{ h}_{i}  D_{j}z\|_{L_{2}(\Omega ') } \leq C_{2}\left(\|z\|_{H^{1} (\Omega)}+\|f\|_{L_{2}(\Omega)}   \right),
\;  i,j=1,2,...,n.
\end{equation*}
In consequence of Lemma 7.24 \cite[p.165]{firstab_lit:gilbarg}, we have that there  exists a generalized derivative $D_{i}  D_{j}  z$ satisfying    the condition
\begin{equation*}
 \|D_{i}  D_{j}  z\|_{L_{2}(\Omega  ) } \leq C_{2}\left(\|z\|_{H^{1} (\Omega)}+\|f\|_{L_{2}(\Omega)}   \right),
\;  i,j=1,2,...,n.
\end{equation*}
Hence $z\in H^{2}(\Omega).$
\end{proof}

\section{Remarks}

  The chapter   represents  the  results  obtained as a prerequisite to  the spectral theory of  fractional differential  operators. However, some    propositions  are of   the independent interest from the point of view of the fractional calculus theory.
  The  new  concept  of the introduced  multidimensional directional fractional integral represents an effective tool by virtue of a    simplest  construction in comparison with its known analogs. Some key aspects of the classical one-dimensional fractional calculus theory were considered for the multidimensional case, for instance the sufficient conditions   of the representability   by the directional  fractional integral were formulated.   Auxiliary propositions  for the  fractional differential equation  theory  such as  the inclusion of  the Sobolev space to  the class  of   functions     representable  by the directional  fractional integral were proved.  Note that in spite of the fact that  the technique of the  proofs  is analogous to  the one-dimensional case, we can claim  that it has its own peculiarities and  is of particular interest. It also should be noted that the  extension of the Kipriyanov fractional differential  operator  was obtained in the natural way that had been  dictated by the operator theory point  of view.
  The proved strictly accretive property, that is itself can be claimed as a  significant result  in the framework of the operator theory,  plays an important role in   further development of the spectral theory. These results create   a rather complete description  reflecting   qualitative properties of  fractional differential operators  establishing a base for further study in the framework of the spectral theory.

    Apparently,   the proved propositions can be spread to unbounded domains, for instance the strictly accretive property was obtained in the one-dimensional case for the real axis. It is worth noticing  that the application of the  sesquilinear forms theory,  as a tool to study the second order differential operators with a fractional derivative in the final term,  gives  an opportunity to analyze  the major role  of the senior term  from the operator theory point of view. This technique   can be used for  studying  the spectrum of  perturbed  fractional differential operators. Therefore, the idea of the proof may be of interest regardless of the results. It can easily be checked   that the Kypriaynov operator   is reduced  to the Marchaud operator in the one-dimensional case. At the same time,
 the most of proved  propositions   are only true   for  the  dimensions more than one.
 However,   using   Corollary 1 \cite{firstab_lit:1kukushkin2018}, which establishes the strictly  accretive property  of the Marchaud  operator, we can apply the obtained  technique to  the one-dimensional case.

In   section \ref{1.7},
  we were guided by the well-known, in the classical case corresponding to the  differential  equations
of the positive-integer order,    idea  of connection between the solvability of a boundary-value problem and properties of the corresponding quadratic
functional.  The idea   to use the same approach in the fractional case    required a some  technique of the fractional calculus theory, in particular we used a strictly   accretive  property of the fractional differential
operators.

Applying  the  Lax-Milgram  theorem  we  proved the existence of a generalized solution of the boundary value problem for  the differential equation of the  fractional order. The method allowing to establish the fact that the found  generalized solution belonging  to   the  Sobolev space is  the very  strong solution was elaborated applicably to the fractional differential equations. Although   the  method  is not novel  in the theory of partial differential equations, the surpassed difficulties related to the fractional nature of the objects give us a significant  complement to the general  theory.

\chapter{Spectral properties of the sectorial operators}\label{Ch2}

\section{Historical review}
It is remarkable that initially the perturbation theory of selfadjoint operators was born in  the  works of M. Keldysh
\cite{firstab_lit:1keldysh 1951}-\cite{firstab_lit:3keldysh 1971} and had been  motivated by  the   works of  famous scientists such as T. Carleman \cite{firstab_lit:1Carleman} and Ya.  Tamarkin  \cite{firstab_lit:1Tamarkin}.   Many papers were published within the framework of this theory over time, for instance    F. Browder \cite{firstab_lit:1Browder},  M. Livshits \cite{firstab_lit:1Livshits}, B. Mukminov \cite{firstab_lit:1Mukminov}, I. Glazman \cite{firstab_lit:1Glazman}, M. Krein \cite{firstab_lit:1Krein}, B. Lidsky \cite{firstab_lit:1Lidskii},  A. Marcus \cite{firstab_lit:1Markus},\cite{firstab_lit:2Markus}, V. Matsaev \cite{firstab_lit:1Matsaev}-\cite{firstab_lit:3Matsaev}, S. Agmon \cite{firstab_lit:2Agmon}, V. Katznelson \cite{firstab_lit:1Katsnelson}, N. Okazawa \cite{firstab_lit:3Okazawa}.
Nowadays there  exists   a  huge amount of theoretical results formulated  in the work of A. Shkalikov    \cite{firstab_lit:Shkalikov A.}. However for applying  these results for a concrete operator $W$ we must have  a representation   of it   by  a sum of operators $W=T+A.$
  It is essential   that $T$ must be an operator of a special type either a  selfadjoint or   normal operator. If we consider a case  where  in the representation the operator  $T$ is  neither selfadjoint nor normal and we cannot approach the  required representation in an  obvious  way, then it is possible to   use  another technique  based on    properties  of the real component of the initial operator. Note that in this case the made assumptions related to the initial operator $W$ allow us to consider  a m-accretive operator class which was thoroughly studied by mathematicians such as T. Kato \cite{firstab_lit:1.1kato1961}, N. Okazawa \cite{firstab_lit:1Okazawa},\cite{firstab_lit:2Okazawa}.    This is a subject to consider  in the second  section. In the third  section   we
demonstrate the  significance   of the obtained abstract results  and consider  concrete operators. Note that the  relevance   of such consideration  is based on the following.
 The eigenvalue problem is still relevant for  the  second order fractional  differential operators.  Many papers were  devoted to this question, for instance the papers
 \cite{kukushkin2019}, \cite{firstab_lit:1Nakhushev1977}.    We would like to study spectral properties of some class of  non-selfadjoint operators in   the abstract case.  Via obtained results  we  study  a  multidimensional case corresponding to  the  second order fractional  differential operator, this case   can be reduced to the
  cases   considered in the papers listed above.
 We consider a
Kipriyanov fractional differential operator, considered in detail in the papers \cite{firstab_lit:kipriyanov1960}-\cite{firstab_lit:2.2kipriyanov1960}, which presents itself as  a  fractional derivative in a weaker sense   with respect
to the approach classically known with the name of the Riemann-Liouville derivative. More precisely, in the one dimensional case  the Kipriaynov operator  coincides with the Marchaud operator  which relationship  with   the Weyl and  Riemann-Liouville     operators is well known
 \cite{firstab_lit:1arx Ferrari},\cite{firstab_lit:samko1987}.

\section{ Special operator class}

  Further,    if it is not     stated otherwise   we  use   the  notations   of  the  monographs   \cite{firstab_lit:1Gohberg1965},\cite{firstab_lit:kato1980},\cite{firstab_lit:samko1987}.
  Consider a pair of  complex  separable Hilbert spaces $\mathfrak{H},\mathfrak{H}_{+}$ such that
\begin{equation}\label{2.1}
\mathfrak{H}_{+}\subset\subset\mathfrak{ H} .
\end{equation}
This denotation  implies  that we have a bounded embedding provided by the inequality
\begin{equation}\label{2.2}
\|f\|_{\mathfrak{H}}\leq \|f\|_{\mathfrak{H}_{+}},\;f\in \mathfrak{H}_{+},
\end{equation}
moreover   any  bounded  set in the space   $\mathfrak{H}_{+} $ is a compact set in   the space $\mathfrak{H}.$
We  also assume that $\mathfrak{H}_{+}$ is a dense set  in $\mathfrak{H}.$
We  consider   non-selfadjoint operators that can be represented by a sum $W=T+A,$  where the operators $T$ and $A$     act   on  $\mathfrak{H}.$     We   assume  that:  there  exists a linear manifold  $\mathfrak{M}\subset \mathfrak{H}_{+} $    that   is  dense in $ \mathfrak{H}_{+},$   the operators $T,A$ and their adjoint operators are   defined on $\mathfrak{M}.$   Further,  we  assume that      $\mathrm{D}(W)=\mathfrak{M}.$ These   give  us the opportunity to  claim      that  $\mathrm{D}(W )\subset \mathrm{D}(W^{\ast})$ thus, by virtue of this fact,  the real component of $W$ is defined on $\mathfrak{M}.$
  Suppose the  operator $W^{+}$ is the   restriction of   $W^{\ast}$ to $ \mathrm{D}(W );$   then the operator $W^{+}$ is called a {\it formal adjoint} operator with respect to $W,$  it is clear that we need not impose more restrictions to guaranty its closeness sing the adjoint operator is closed.
 Denote  by $\tilde{W}^{+} $ the closure of the operator $W^{+}.$
 Further, we assume  that the following conditions are fulfilled
\begin{equation}\label{2.3}
\mathrm{i})\,\mathrm{Re}(Tf,f)_{\mathfrak{H}}\geq C_{0}\|f\|^{2}_{\mathfrak{H}_{+}}\!,\;\mathrm{ii})\,
\left| (Tf,g)_{\mathfrak{H}}\right|\leq C_{1}\|f\|_{\mathfrak{H}_{+}}\|g\|_{\mathfrak{H}_{+}},
$$
$$
 \mathrm{iii})\,\mathrm{Re}(Af,f)_{\mathfrak{H}}\geq C_{2} \|f\|^{2}_{\mathfrak{H}}, \; \mathrm{iv})\,|(Af,g)_{\mathfrak{H}}|\leq C_{3}\|f\|_{\mathfrak{H}_{+}}\|g\|_{\mathfrak{H} },\,f,g\in  \mathfrak{M}.
\end{equation}
 Due to these conditions  it is easy to prove  that the operators $W,\mathfrak{Re}W $ are closeable (see   Theorem 3.4 \cite[p.268]{firstab_lit:kato1980}).
  To make  some formulas readable we also use the following form of notation
 $$V:=   \mathfrak{Re}R_{\tilde{W}} ,\, \mathcal{H}:= \mathfrak{Re}W,\,H:= Re \tilde{W}.
 $$

In this section  we   formulate abstract theorems that  are generalizations of some particular  results obtained by the author. First,  we generalize Theorem 4.2 \cite{firstab_lit:1kukushkin2018} establishing the sectorial property of the second order fractional differential operator.
\begin{lem}\label{L2.1}
 The numerical range of the operators $\tilde{W},\,\tilde{W}^{+}$ belongs to    a positive sector.
\end{lem}
\begin{proof}  Due to    inequalities   \eqref{2.2},\eqref{2.3}   we   conclude that the  operator $W$ is strictly accretive, i.e.
\begin{equation}\label{2.4}
\mathrm{Re}(Wf,f)_{\mathfrak{H} }\geq C_{0} \|f\|^{2}_{\mathfrak{H} } ,\;f\in \mathrm{D}(W).
\end{equation}
 Let us prove that the operator $\tilde{W}$ is canonical sectorial. Combining   \eqref{2.3}  (ii) and  \eqref{2.3} (iii), we get
\begin{equation}\label{2.5}
\mathrm{Re}(Wf,f)_{\mathfrak{H}}= \mathrm{Re}(Tf,f)_{\mathfrak{H}}+\mathrm{Re}(Af,f)_{\mathfrak{H}}\geq  C_{0}\|f\|_{\mathfrak{H}_{+}}+C_{2}\|f\|_{\mathfrak{H}},\,f\in \mathrm{D}(W) .
\end{equation}
  Obviously    we can extend the previous  inequality  to
\begin{equation}\label{2.6}
\mathrm{Re}(\tilde{W}f,f)_{\mathfrak{H}} \geq  C_{0}\|f\|_{\mathfrak{H}_{+}}+C_{2}\|f\|_{\mathfrak{H}} ,\,f\in \mathrm{D}(\tilde{W}).
\end{equation}
By virtue of \eqref{2.6}, we obtain  $\mathrm{D}(\tilde{W})\subset\mathfrak{H}_{+}. $
Note that we have the estimate
\begin{equation*}
|\mathrm{Im}(Wf,f)_{\mathfrak{H}}|\leq \left|\mathrm{Im}(Tf,f)_{\mathfrak{H}}\right|+\left|\mathrm{Im}(Af,f)_{\mathfrak{H}}\right|=
I_{1}+I_{2},\,f\in \mathrm{D}(W) .
\end{equation*}
Using        inequality \eqref{2.3} (ii),  the    Young  inequality,   we get
$$
I_{1}=\left| (Tv,u)_{\mathfrak{H}}- (Tu,v)_{\mathfrak{H}}\right|\leq  \left| (Tv,u)_{\mathfrak{H}}\right|+\left| (Tu,v)_{\mathfrak{H}}\right|\leq   2 C_{1}\|u\|_{\mathfrak{H}_{+}}\|v\|_{\mathfrak{H}_{+}}\leq  C_{1} \|f\|^{2}_{\mathfrak{H}_{+}},
$$
where $f=u+i\, v.$ Consider $I_{2}.$
Applying  the  Cauchy Schwartz inequality and inequality \eqref{2.3} (iv),  we obtain for arbitrary positive $\varepsilon$
$$
 \left| (Av,u)_{\mathfrak{H}} \right|\leq C_{3}\|v\|_{\mathfrak{H}_{+}}\|u\|_{\mathfrak{H}} \leq
\frac{C_{3}}{2}\left\{\frac{1}{\varepsilon}\|u\|^{2}_{\mathfrak{H}}+\varepsilon\|v\|^{2}_{\mathfrak{H}_{+}}\right\}\,;
$$
$$
 \left| (Au,v)_{\mathfrak{H}}\right|\leq
\frac{C_{3}}{2}\left\{\frac{1}{\varepsilon}\|v\|^{2}_{\mathfrak{H}}+\varepsilon\|u\|^{2}_{\mathfrak{H}_{+}}\right\}.
$$
 Hence
$$
I_{2}=\left| (Av,u)_{\mathfrak{H}} -(Au,v)_{\mathfrak{H}}\right|\leq \left| (Av,u)_{\mathfrak{H}}| +|(Au,v)_{\mathfrak{H}}\right|\leq \frac{C_{3}}{2}\left\{\frac{1}{\varepsilon}\|f\|^{2}_{\mathfrak{H}}+\varepsilon\|f\|^{2}_{\mathfrak{H}_{+}}\right\}.
$$
  Finally,   we have the following estimate
$$
|\mathrm{Im}(Wf,f)_{\mathfrak{H}}|  \leq  \frac{C_{3}}{2}\, \varepsilon^{-1} \|f\|^{2}_{\mathfrak{H}}+\left(\frac{C_{3}}{2}\,\varepsilon+C_{1}\right)\|f\|^{2}_{\mathfrak{H}_{+}} ,\,f\in \mathrm{D}(W).
 $$
Thus,  we   conclude that      the next inequality holds for   arbitrary $k>0$
\begin{equation*}
{\rm Re}( Wf, f    )_{\mathfrak{H}}-k \left|{\rm Im} ( Wf, f    )_{\mathfrak{H}}\right|\geq
$$
$$
\geq\left[C_{0}-k\left(\frac{C_{3}}{2}\,\varepsilon+C_{1}\right)\right] \|f\|^{2}_{\mathfrak{H}_{+}} +\left(C_{2}-k\,\frac{C_{3}}{2}\, \varepsilon^{-1}\right)\|f\|^{2}_{\mathfrak{H}},\,f\in \mathrm{D}(W) .
\end{equation*}
  Using the continuity property of the inner product, we can extend the previous  inequality  to  the set  $\mathrm{D}(\tilde{W}).$
It follows easily that
\begin{equation}\label{2.7}
\left|{\rm Im} \left( [\tilde{W}-\gamma(\varepsilon)]f, f    \right)_{\mathfrak{H}}\right|   \leq \frac{1}{k(\varepsilon)}
{\rm Re}\left( [\tilde{W}-\gamma(\varepsilon)]f, f    \right)_{\mathfrak{H}}\! ,\,f\in \mathrm{D}(\tilde{W}) ,
$$
$$
k(\varepsilon)= C_{0}\left(\frac{ C_{3} }{2} \,\varepsilon+ C_{1}\right)^{-1},
\;\gamma(\varepsilon)=C_{2}-  k(\varepsilon) \, \frac{  C_{3} }{2}\,  \varepsilon^{-1} .
\end{equation}
The previous  inequality implies that the numerical range   of the operator $ \tilde{W} $ belongs to the sector $\mathfrak{L}_{\gamma}(\theta)$ with the vertex situated   at the point  $\gamma$ and the semi-angle $\theta=\arctan(1/k).$
Solving system of equations  \eqref{2.7} relative to $\varepsilon$ we obtain  the positive root $\xi$ corresponding to  the value $\gamma=0$ and the following description  for the coordinates of  the   sector vertex $\gamma$
\begin{equation*}
\gamma:=\left\{ \begin{aligned}
  \gamma<0,\;\varepsilon\in(0,\xi)  ,\\
  \gamma\geq 0,\; \varepsilon\in [\xi,\infty)   \\
\end{aligned}
 \right., \xi=\sqrt{\left(\frac{C_{1}}{ C_{3}}\right)^{2}+ \frac{C_{0}}{C_{2}}   }-\frac{C_{1}}{ C_{3}}.
\end{equation*}
 It follows that the operator $\tilde{W}$ has a positive sector.   The proof corresponding to    the operator  $ \tilde{W}^{+} $ follows   from the  reasoning  given  above if we note  that   $W ^{+}$    is formal adjoint with respect to   $W.$
  \end{proof}
  \begin{lem}\label{L2.2}
  The operators $\tilde{W},\tilde{W}^{+}$   are m-accretive, their  resolvent sets contain  the half-plane $\{\zeta:\,\zeta\in \mathbb{C},\,\mathrm{Re}\,\zeta<C_{0}\}.$
\end{lem}
\begin{proof}    Due to Lemma \ref{L2.1}  we know that the operator $\tilde{W}$ has a positive sector,  i.e. the numerical range of    $\tilde{W}$   belongs to the  sector
 $\mathfrak{L}_{\gamma}(\theta),\,\gamma>0.$  In consequence of Theorem 3.2   \cite[p.268]{firstab_lit:kato1980},
 we have   $\forall\zeta\in \mathbb{C}\setminus\mathfrak{L}_{\gamma}(\theta),$ the set  $\mathrm{R}(\tilde{W}-\zeta)$ is a closed  space,   and the next relation     holds
\begin{equation*}
  {\rm def}(\tilde{W}-\zeta)=\eta,\; \eta={\rm const} .
 \end{equation*}
Due to  Theorem 3.2   \cite[p.268]{firstab_lit:kato1980} the inverse operator $(\tilde{W}+\zeta)^{-1}$ is defined  on the subspace $\mathrm{R}(\tilde{W}+\zeta),\,{\rm Re}\zeta>0.$
In accordance with the definition of m-accretive operator  given in the monograph   \cite[p.279]{firstab_lit:kato1980}   we need to show  that
\begin{equation*}
{\rm def}(\tilde{W}+\zeta)=0,\;\|(\tilde{W}+\zeta)^{-1}\|\leq ({\rm Re}\zeta)^{-1},\,{\rm Re}\zeta>0.
\end{equation*}
  For this purpose assume that
 $\zeta_{0} \in\mathbb{C}\setminus \mathfrak{L}_{\gamma}(\theta),\;{\rm Re}\zeta_{0}  <0.$
Using  \eqref{2.4}, we get
 \begin{equation}\label{2.8}
  {\rm Re}  \left( f,[\tilde{W}-\zeta_{0} ]f  \right)_{\!\!\mathfrak{H}}\!\!\geq  (C_{0}- {\rm Re} \zeta_{0} ) \|f\|^{2}_{\mathfrak{H}},\;f\in \mathrm{D}(\tilde{W}).
 \end{equation}
Since the  operator $\tilde{W}-\zeta_{0}$    has the closed range    $\mathrm{R} (\tilde{W}-\zeta_{0}),$ it follows that
\begin{equation*}
 \mathfrak{H}=\mathrm{R} (\tilde{W}-\zeta_{0})\oplus \mathrm{R} (\tilde{W}-\zeta_{0})^{\perp} .
 \end{equation*}
Note that the  intersection of the  sets  $  \mathfrak{M}$ and $\mathrm{R} (\tilde{W}-\zeta_{0})^{\perp}$ is   zero. If we assume   otherwise, then applying  inequality  \eqref{2.8}   for any element
 $u\in \mathfrak{M}\cap \mathrm{R}  (\tilde{W}-\zeta_{0})^{\perp}$    we get
 \begin{equation*}
(C_{0}-{\rm Re}\zeta_{0}) \|u\|^{2}_{\mathfrak{H}} \leq {\rm Re} \left( u,[\tilde{W}-\zeta_{0}]u  \right)_{\mathfrak{H}}=0,
 \end{equation*}
hence $u=0.$ Thus  the intersection of the sets  $  \mathfrak{M}$ and $\mathrm{R} (\tilde{W}-\zeta_{0})^{\perp}$ is   zero. It implies that
$$
\left(g,v\right)_{\mathfrak{H}}=0,\;\forall g\in  \mathrm{R}  (\tilde{W}-\zeta_{0})^{\perp},\;\forall v\in \mathfrak{M}.
$$
Since $ \mathfrak{M}$   is    a dense set in $\mathfrak{H}_{+},$  then   taking into account  \eqref{2.2}, we obtain that      $ \mathfrak{M}$  is  a  dense  set  in $\mathfrak{H}.$ Hence    $\mathrm{R}  (\tilde{W}-\zeta_{0})^{\perp}=0,\,{\rm def} (\tilde{W}-\zeta_{0}) =0.$ Combining this fact with Theorem 3.2   \cite[p.268]{firstab_lit:kato1980}, we get   ${\rm def} (\tilde{W}-\zeta )=0,\;\zeta\in \mathbb{C}\setminus\mathfrak{L}_{\gamma}(\theta).$  It is clear that
${\rm def} (\tilde{W}+\zeta )=0,\,\forall\zeta,\,{\rm Re}\zeta>0.$  Let us prove that $\|(\tilde{W}+\zeta)^{-1}\|\leq ({\rm Re}\zeta)^{-1},\,\forall\zeta,\,{\rm Re}\zeta>0.$
  We must   notice that
 \begin{equation*}
(C_{0}+{\rm Re}\zeta ) \|f\|^{2}_{\mathfrak{H}} \leq {\rm Re} \left( f,[\tilde{W}+\zeta ]f  \right)_{\mathfrak{H}}\leq \|f\|_{\mathfrak{H}}\|(\tilde{W}+\zeta )f\|_{\mathfrak{H}},\;f\in \mathrm{D}(\tilde{W}),\;
{\rm Re}\zeta>0 .
 \end{equation*}
By virtue of the fact   ${\rm def} (\tilde{W}+\zeta )=0,\;\forall\zeta,\,{\rm Re}\zeta>0$  we   know   that the  resolvent   is   defined. Therefore
 \begin{equation*}
\|(\tilde{W}+\zeta )^{-1}f\|_{\mathfrak{H}}     \leq(C_{0}+{\rm Re}\,\zeta ) ^{-1} \|f\|_{\mathfrak{H}}\leq ( {\rm Re}\,\zeta ) ^{-1} \|f\|_{\mathfrak{H}},\;f\in \mathfrak{H}.
 \end{equation*}
This implies that
$$
\|(\tilde{W}+\zeta )^{-1} \| \leq( {\rm Re}\,\zeta ) ^{-1},\;\forall\zeta,\,{\rm Re}\zeta>0.
$$
If we combine inequality \eqref{2.6}  with     Theorem 3.2 \cite[p.268]{firstab_lit:kato1980}, we get    $\mathrm{P}(\tilde{W})\supset\{\zeta:\,\zeta\in \mathbb{C},\,\mathrm{Re}\,\zeta<C_{0}\}.$
The  proof corresponding to the operator  $\tilde{W}^{+}$    is absolutely analogous.
\end{proof}

 \begin{lem}\label{L2.3}
 The operator $\tilde{\mathcal{H}}$ is strictly accretive,  m-accretive,  selfadjoint.
\end{lem}
\begin{proof}
  It is  obvious  that $   \mathcal{H}   $ is a symmetric operator. Due to the continuity  property of the inner product we can conclude that
   $\tilde{\mathcal{H}}  $ is   symmetric too.  Hence  $\Theta(\tilde{\mathcal{H}})\subset \mathbb{R}.$ By virtue of  \eqref{2.5}, we have
 \begin{equation*}
( \mathcal{H} f,f)_{\mathfrak{H}}\geq C_{0}\|f\|^{2}_{\mathfrak{H}_{+}},\,f\in \mathrm{D}(W).
\end{equation*}
Using inequality \eqref{2.2} and  the  continuity property of  the   inner product, we obtain
 \begin{equation}\label{2.9}
(\tilde{\mathcal{H}}f,f)_{\mathfrak{H}}\geq C_{0}\|f\|^{2}_{\mathfrak{H}_{+}}\geq C_{0}\|f\|^{2}_{\mathfrak{H}},\,f\in \mathrm{D}(\tilde{\mathcal{H}}).
\end{equation}
It implies that $\tilde{\mathcal{H}}$ is strictly accretive.
  In the same way as in the proof of Lemma \ref{L2.2}   we   come  to conclusion that
$\tilde{\mathcal{H}}$ is m-accretive. Moreover, we obtain the   relation ${\rm def}(\tilde{\mathcal{H}}-\zeta )=0,\,  {\rm Im }\zeta\neq 0 .$
Hence by virtue  of Theorem 3.16  \cite[p.271]{firstab_lit:kato1980} the operator $\tilde{\mathcal{H}}$ is selfadjoint.
\end{proof}
 \begin{teo}\label{T2.1}
 The  operators $\tilde{\mathcal{H}},\tilde{W},\tilde{W}^{+}$   have    compact resolvents.
\end{teo}
\begin{proof}
First note that due to Lemma \ref{L2.3} the operator $\tilde{\mathcal{H}}$  is selfadjoint.  Using \eqref{2.9},  we obtain    the estimates
$$
\|f\|_{ H }\geq  \sqrt{C_{0} }\|f\|_{\mathfrak{H}_{+}}\geq \sqrt{C_{0} } \|f\|_{\mathfrak{H} }, \;f\in \mathfrak{H}_{H},
$$
where $H:=\tilde{\mathcal{H}}.$
Since  $\mathfrak{H}_{+}\subset\subset\mathfrak{H} ,$ then  we   conclude that each   set   bounded     with respect to the energetic norm generated by the operator $\tilde{\mathcal{H}}$ is  compact with respect to the norm $\|\cdot\|_{\mathfrak{H}}.$     Hence in accordance  with Theorem   \cite[p.216]{firstab_lit:mihlin1970} we conclude that   $\tilde{\mathcal{H}}$ has a discrete spectrum.  Note that  in consequence of  Theorem 5 \cite[p.222]{firstab_lit:mihlin1970} we have that    a selfadjoint strictly accretive    operator with  discrete spectrum has a compact inverse operator. Thus using   Theorem 6.29 \cite[p.187]{firstab_lit:kato1980} we obtain  that $\tilde{\mathcal{H}}$ has a compact resolvent.

 Further, we need the  technique of the   sesquilinear forms theory  stated    in \cite{firstab_lit:kato1980}.
Consider the sesquilinear forms
$$ \mathfrak{t} [f,g]=(\tilde{W}f,g)_{\mathfrak{H}},\,f,g\in \mathrm{D}(\tilde{W}),\;  \mathfrak{h} [f,g]=(\tilde{\mathcal{H}}f,g)_{\mathfrak{H}},\,f,g\in \mathrm{D}(\tilde{\mathcal{H}}).
$$
 Recall  that due to   inequality \eqref{2.6} we came to the  conclusion that $\mathrm{D}(\tilde{W})\subset \mathfrak{H}_{+}.$ In the same way  we can deduce  that  $\mathrm{D}(\tilde{\mathcal{H}})\subset \mathfrak{H}_{+}.$
By virtue   of   Lemma  \ref{L2.1},  Lemma \ref{L2.3}, it is easy to prove that the   sesquilinear forms $\mathfrak{t},\mathfrak{h}$  are sectorial. Applying   Theorem 1.27 \cite[p.318]{firstab_lit:kato1980}  we get that these  forms  are closable. Now note that   $\mathfrak{Re}\, \tilde{\mathfrak{t}}$   is a sum of two closed sectorial forms. Hence   in consequence of Theorem 1.31 \cite[p.319]{firstab_lit:kato1980},  we have that   $\mathfrak{Re}\,\tilde{\mathfrak{t}}$   is a closed form. Let us show that $\mathfrak{Re}\,\tilde{\mathfrak{t}}=\tilde{\mathfrak{h}}.$    First  note that this equality is true   on the  elements of the linear manifold $\mathfrak{M} \subset\mathfrak{H}_{+}.$   This  fact can be obtained directly from the following
$$\tilde{\mathfrak{t}}  [f,g]=(  W   f,g)_{\mathfrak{H}},
\;  \overline{\tilde{\mathfrak{t}}[g,f]} = ( W ^{+}\! f,g)_{\mathfrak{H}},\,f,g\in \mathfrak{M}.
$$
On the other hand
$$\tilde{\mathfrak{h}}[f,g]=(\tilde{\mathcal{H}}f,g)_{\mathfrak{H}}=( \mathcal{H} f,g)_{\mathfrak{H}},\,f,g\in \mathfrak{M}.
$$
Hence
\begin{equation}\label{2.10}
\mathfrak{Re}\,\tilde{\mathfrak{t}}[f,g]=\tilde{\mathfrak{h}}[f,g],\;f,g\in \mathfrak{M}.
\end{equation}
 Using  \eqref{2.3}, we get
\begin{equation}\label{2.11}
 C_{0}\|f\|^{2}_{\mathfrak{H}_{+}}\leq \mathrm{Re}\,\tilde{\mathfrak{t}}[f] \leq C_{4} \|f\|^{2}_{\mathfrak{H}_{+}},\,C_{0} \|f\|^{2}_{\mathfrak{H}_{+}}\leq\tilde{\mathfrak{h}}[f]\leq  C_{4}\|f\|^{2}_{\mathfrak{H}_{+}},\,f\in \mathfrak{M},
\end{equation}
where $C_{4}=C_{1}+C_{3}.$  Since $ \mathfrak{Re} \,\tilde{\mathfrak{t}}[f]=\mathrm{ Re } \,\tilde{\mathfrak{t}}[f],\,f\in \mathfrak{M},$   the sesquilinear forms  $\mathfrak{Re} \,\tilde{\mathfrak{t}},\tilde{\mathfrak{h}}$ are closed forms, then using \eqref{2.11} it is easy to prove that $\mathrm{D}(\mathfrak{Re}\,\tilde{\mathfrak{t}})=\mathrm{D}(\tilde{\mathfrak{h}})=\mathfrak{H}_{+}.$ Using   estimates \eqref{2.11},   it is not hard to prove  that   $\mathfrak{M}$ is a core of the forms $\mathfrak{Re} \,\tilde{\mathfrak{t}},\tilde{\mathfrak{h}}.$ Hence using \eqref{2.10}, we obtain
$\mathfrak{Re}\,\tilde{\mathfrak{t}}[f] =\tilde{\mathfrak{h}}[f],\,f\in \mathfrak{H}_{+}.$ In accordance with    the   polarization principle (see (1.1) \cite[p.309]{firstab_lit:kato1980}),  we have     $\mathfrak{Re}\,\tilde{\mathfrak{t}}=\tilde{\mathfrak{h}}.$
Now   recall      that the     forms
  $\tilde{\mathfrak{t}},\tilde{\mathfrak{h}}$ are generated    by the operators  $\tilde{W}, \tilde{\mathcal{H}}$ respectively. Note that    in consequence of
  Lemmas \ref{L2.1}-\ref{L2.3}  these operators   are m-sectorial. Hence     by virtue  of    Theorem 2.9 \cite[p.326]{firstab_lit:kato1980}, we get  $T_{\tilde{\mathfrak{t}}}=\tilde{W},  T_{\tilde{\mathfrak{h}}}=\tilde{\mathcal{H}} .$    Since we have proved that  $\mathfrak{Re}\,\tilde{\mathfrak{t}}=\tilde{\mathfrak{h}}$, then $T_{\mathfrak{Re}\,\tilde{\mathfrak{t}}}=\tilde{\mathcal{H}}.$ Therefore,   by definition   we have  that the operator $\tilde{\mathcal{H}}$ is  the real part of the m-sectorial operator $\tilde{W},$   by symbol $\tilde{\mathcal{H}}=Re  \tilde{W}.$   Since we   proved above  that $\tilde{\mathcal{H}}$ has a compact resolvent, then
  using  Theorem 3.3 \cite[p.337]{firstab_lit:kato1980}  we conclude that the  operator $\tilde{W}$ has a compact resolvent. The proof corresponding   to the operator $\tilde{W}^{+}$ is absolutely  analogous.
\end{proof}

\section{Asymptotic equivalence}

It is remarkable that     we obtain the equality $\tilde{\mathcal{H}}=Re  \tilde{W}$ in the proof of  Theorem \ref{T2.1}. This fact     is however  worth considering itself and  one may find a comprehensive analysis in the section Remarks. Thus,  since further we prefer standing at the operator theory point of view that is harmoniously connected with the sesqulinear forms theory, we deal with the operator $H:=Re \tilde{W}.$

\begin{teo}\label{T2.2}
The following relation   holds
\begin{equation}\label{2.12}
  \lambda_{i}(R_{H})\asymp \lambda_{i}\left(V\right).
\end{equation}
 \end{teo}
\begin{proof}
  Note that the properties established in  Lemma \ref{L2.1}, Lemma \ref{L2.2} give an opportunity to apply   Theorem 3.2 \cite[p.337]{firstab_lit:kato1980} according to which   there exist  the selfadjoint   operators  $B_{i}:=\{B_{i}\in\mathcal{B} (\mathfrak{H}),\,\|B_{i}\|\leq \tan \theta \},\,i=1,2 $
(where $\theta$ is the  semi-angle of the sector $\mathfrak{L}_{0}(\theta)\supset \Theta(\tilde{W})$) such  that
\begin{equation}\label{2.13}
\tilde{W}=H^{\frac{1}{2}}(I+i B_{1}) H^{\frac{1}{2}},\;\tilde{W}^{+}=H^{\frac{1}{2}}(I+i B_{2}) H^{\frac{1}{2}}.
\end{equation}
Since the set of  linear operators  generates  ring,   it follows that
\begin{equation*}
  Hf\! =\!\frac{1}{2}\left[H^{\frac{1}{2}}(I+i B_{1})  +H^{\frac{1}{2}}(I+i B_{2})\right]H^{\frac{1}{2}}\!  =
  $$
  $$
  =\! \frac{1}{2}\left\{H^{\frac{1}{2}}\left[(I+i B_{1})  + (I+i B_{2})\right]\right\}H^{\frac{1}{2}}\!=
$$
$$
= \! H f +
 \frac{i}{2} H^{\frac{1}{2}}\left(B_{1}+B_{2}\right)  H^{\frac{1}{2}}f  ,\;f\in \mathfrak{M}.
\end{equation*}
 Consequently
\begin{equation}\label{2.14}
H^{\frac{1}{2}}\left(B_{1}+B_{2}\right) H^{\frac{1}{2}}f=0  ,\;f\in \mathfrak{M}.
\end{equation}
Let us show that $B_{1}=-B_{2}.$
 In  accordance with  Lemma \ref{L2.3} the operator $H$ is m-accretive, hence we have
$
(H+\zeta)^{-1}\in  \mathcal{B }(\mathfrak{H}),\,\mathrm{Re}\,\zeta>0.
$
Using this fact, we get
\begin{equation}\label{2.15}
{\rm Re}\left([H+\zeta]^{-1}Hf,f\right)_{\mathfrak{H}}={\rm Re}\left([H+\zeta]^{-1}[H+\zeta]  f,f\right)_{\mathfrak{H}}-{\rm Re}\left(\zeta\,[H+\zeta]^{-1}   f,f\right)_{\mathfrak{H}}\geq
$$
$$
\geq \|f\|^{2}_{\mathfrak{H}}-|\zeta|\cdot\|(H+\zeta)^{-1}\|\cdot\|f\|^{2}_{\mathfrak{H}}=\|f\|^{2}_{\mathfrak{H}} \left(1-|\zeta|\cdot \|(H+\zeta)^{-1}\|\right),
$$
$$
\,\mathrm{Re}\,\zeta>0,\,f\in \mathrm{D}(H).
\end{equation}
Applying    inequality \eqref{2.9},   we obtain
$$
\| f\| _{\mathfrak{H}}\|(H+\zeta)^{-1}f\| _{\mathfrak{H}}\geq|(f,[H+\zeta]^{-1}f)|
 \geq ( {\rm Re} \zeta+C_{0} )\|(H+\zeta)^{-1}f\|^{2}_{\mathfrak{H}},\;f\in \mathfrak{H}.
$$
It implies that
$$
\|(H+\zeta)^{-1}\|\leq ({\rm Re}\zeta +C_{0}  )^{-1},\;{\rm Re}\zeta>0.
$$
  Combining    this  estimate  and  \eqref{2.15}, we have
$$
{\rm Re}\left([H+\zeta]^{-1}Hf,f\right)_{\!\mathfrak{H}}\geq \|f\|_{\mathfrak{H}}^{2}\left(1-\frac{|\zeta|}{{\rm Re}\zeta +C_{0}}\right),\,{\rm Re}\zeta>0,\,f\in \mathrm{D}(H).
$$
Applying    formula (3.45) \cite[p.282]{firstab_lit:kato1980} and  taking into account that $H^{\frac{1}{2}}$ is selfadjoint, we get
\begin{equation}\label{2.16}
 \left(H^{\frac{1}{2} }f,f\right)_{\!\mathfrak{H}}=\frac{1}{\pi}\int\limits_{0}^{\infty}\zeta^{-1/2}{\rm Re}\left([H+\zeta]^{-1}Hf,f\right)_{\!\mathfrak{H}}d\zeta\geq
$$
$$
 \geq\|f\|^{2}_{\mathfrak{H}}\cdot \frac{C_{0}}{\pi }\int\limits_{0}^{\infty}\frac{\zeta^{-1/2}}{\zeta+C_{0}}  d\zeta=
 \sqrt{C_{0}}  \|f\|^{2}_{\mathfrak{H}},\,f\in  \mathrm{D} (H) .
\end{equation}
Since in accordance with   Theorem 3.35 \cite[p.281]{firstab_lit:kato1980} the set      $\mathrm{D}(H)$ is  the core of the operator $H^{ \frac{1}{2}},$    then we can extend \eqref{2.16} to
  \begin{equation}\label{2.17}
 \left(H^{\frac{1}{2}}f,f\right)_{\!\!\mathfrak{H}}\geq \sqrt{C_{0}} \|f\|^{2}_{\mathfrak{H}},\;f\in \mathrm{D} (H^{\frac{1 }{2}}).
\end{equation}
Hence $\mathrm{N}(H^{\frac{1}{2} })=0.$ Combining this fact and  \eqref{2.14}, we obtain
\begin{equation}\label{2.18}
 \left(B_{1}+B_{2}\right)  H^{ \frac{1}{2}}f=0  ,\;f\in \mathfrak{M}.
\end{equation}
Let us show that the set $\mathfrak{M}$ is a core of the operator $H^{ \frac{1}{2}}.$ Note that due to Theorem 3.35 \cite[p.281]{firstab_lit:kato1980}
 the operator  $H^{ \frac{1}{2}}$
is selfadjoint and $\mathrm{D}(H)$ is a core of the  operator $H^{ \frac{1}{2}}.$ Hence   we have the representation
\begin{equation}\label{2.19}
 \|H^{ \frac{1}{2}}f \|^{2}_{\mathfrak{H}}=(Hf,f)_{ \mathfrak{H}},\,f\in \mathrm{D}(H).
\end{equation}
To achieve    our aim, it is
sufficient  to show  the following
\begin{equation}\label{2.20}
\forall\,f_{0}\in \mathrm{D}(H^{ \frac{1}{2}}),\,\exists\, \{f_{n}\}_{1}^{\infty}\subset \mathfrak{M}:\; f_{n}\stackrel{\mathfrak{H}}{\longrightarrow} f_{0},\,H^{ \frac{1}{2}}f_{n}\stackrel{\mathfrak{H}}{\longrightarrow} H^{ \frac{1}{2}}f_{0}.
 \end{equation}
Since  in accordance with the definition the set $\mathfrak{M}$ is a core of $H$, then we can extend   second relation \eqref{2.11} to
$
\sqrt{C_{0}} \|f\| _{\mathfrak{H}_{+}}\leq (Hf,f)_{\mathfrak{H}}\leq \sqrt{C_{4}} \|f\| _{\mathfrak{H}_{+}},\,f\in \mathrm{D}(H).
$
Applying    \eqref{2.19}, we can write
\begin{equation}\label{2.21}
\sqrt{C_{0}} \|f\| _{\mathfrak{H}_{+}}\leq \|H^{ \frac{1}{2}}f \| _{\mathfrak{H}}\leq \sqrt{C_{4}} \|f\| _{\mathfrak{H}_{+}},\,f\in \mathrm{D}(H).
 \end{equation}
Using    lower estimate \eqref{2.21} and the fact that  $\mathrm{D}(H)$    is a  core of    $H^{ \frac{1}{2}},$ it is not hard to prove that  $\mathrm{D}(H^{   \frac{1}{2} })\subset\mathfrak{H}_{+}.$
Taking into account this fact and using         upper estimate \eqref{2.21}, we obtain \eqref{2.20}. It implies that  $\mathfrak{M}$ is a core of $H^{ \frac{1}{2}}.$ Note that  in accordance with    Theorem 3.35   \cite[p.281]{firstab_lit:kato1980} the operator $ H ^{ \frac{1}{2}} $ is m-accretive. Hence  combining  Theorem 3.2 \cite[p.268]{firstab_lit:kato1980}        with \eqref{2.17},
  we obtain   $\mathrm{R}(H ^{ \frac{1}{2}})=\mathfrak{H}.$ Taking into account that      $\mathfrak{M}$ is a  core   of the operator $H^{ \frac{1}{2}},$  we   conclude that
$\mathrm{R}(\check{H} ^{\frac{1}{2} })$ is dense in $\mathfrak{H},$ where $\check{H} ^{ \frac{1}{2} }$ is the  restriction of  the operator $H ^{ \frac{1}{2}}$ to $\mathfrak{M}.$ Finally, by virtue of \eqref{2.18}, we have that the sum   $B_{1}+B_{2}$   equal to zero   on the dense subset  of $\mathfrak{H}.$ Since these operators are   defined on $\mathfrak{H}$ and  bounded,  then   $B_{1}=-B_{2}.$   Further,   we    use the denotation  $B_{1}:=B.$

Note that due to Lemma \ref{L2.2} there exist the  operators $R_{\tilde{W}},R_{\tilde{W}^{+}}.$      Using the  properties  of the operator $B,$    we get
$\|(I\pm iB)f\|_{\mathfrak{H} }\|f\|_{\mathfrak{H} }\geq\mathrm{Re }\left([I\pm iB]f,f\right)_{\mathfrak{H}} =\|f\|^{2}_{\mathfrak{H}},\,f\in \mathfrak{H}.$ Hence
\begin{equation*}
\|(I\pm iB)f\|_{\mathfrak{H} } \geq \|f\|_{\mathfrak{H}},\,f\in \mathfrak{H}.
\end{equation*}
 It implies that the operators  $I\pm iB$ are invertible.
Since it was proved above that  $ \mathrm{R} (H^{ \frac{1}{2}})=\mathfrak{H},\,\mathrm{N}(H^{\frac{1}{2}})=0$,    then  there exists an operator $H^{-\frac{1}{2}}$ defined on $\mathfrak{H}.$
  Using    representation   \eqref{2.13}  and  taking into account  the   reasonings given above,  we obtain
\begin{equation}\label{2.22}
R_{\tilde{W}}=H^{-\frac{1 }{2}}(I+iB )^{-1} H^{- \frac{1}{2}},\;R_{\tilde{W}^{+}}=H^{-\frac{1 }{2}}(I-iB )^{-1} H^{- \frac{1}{2}}.
\end{equation}
Note that the following equality can be proved easily  $R^{\ast}_{ \tilde{W}}=R^{\,}_{\tilde{W}^{+}}.$ Hence  we have
\begin{equation}\label{2.23}
V=\frac{1}{2}\left(R_{\tilde{W}}+R_{\tilde{W}^{+}}\right).
\end{equation}
Combining  \eqref{2.22},\eqref{2.23},  we get
\begin{equation}\label{2.24}
V=\frac{1}{2}\,H^{-\frac{1 }{2}}\left[(I+iB )^{-1}+(I-iB )^{-1} \right]H^{- \frac{1}{2}}.
\end{equation}
 Using the obvious identity
$
(I+B^{2})=(I+iB ) (I-iB )= (I-iB )(I+iB ),
$
 by  direct calculation we   get
\begin{equation}\label{2.25}
(I+iB )^{-1}+(I-iB )^{-1}=(I+B^{2})^{-1}.
\end{equation}
Combining \eqref{2.24},\eqref{2.25}, we obtain
\begin{equation}\label{2.26}
V=\frac{1}{2}\,H^{-\frac{1 }{2}}  (I+B^{2} )^{-1}  H^{- \frac{1}{2}}.
\end{equation}
Let us evaluate the form $\left(V  f,f\right)_{\mathfrak{H}}.$  Note  that  there exists the operator $R_{H}$ (see Lemma \ref{L2.3}). Since  $H$ is selfadjoint (see Lemma \ref{L2.3}), then  due to Theorem 3 \cite[p.136]{firstab_lit: Ahiezer1966}     $R_{H}$ is selfadjoint.  It is clear that $R_{H}$  is positive because  $H$ is positive.   Hence by virtue of  the  well-known theorem (see \cite[p.174]{firstab_lit:Krasnoselskii M.A.})   there exists a unique  square root of the operator $ R_{H} ,$ the   selfadjoint operator $ \hat{R} $ such that
  $\hat{R} \hat{R}  =R_{H}.$     Using the    decomposition $H=H^{\frac{1}{2}}H^{\frac{1}{2}},$   we get   $H^{-\frac{1}{2}}H^{-\frac{1}{2}}H=I.$ Hence
$R_{H}\subset H^{-\frac{1}{2}}H^{-\frac{1}{2}},$ but   $ \mathrm{D} (R_{H})=\mathfrak{H}.$ It implies that $R_{H}=H^{-\frac{1}{2}}H^{-\frac{1}{2}}.$  Using  the  uniqueness property  of  the  square root  we   obtain       $H^{-\frac{1}{2}}= \hat{R}.$    Let us use the shorthand notation  $S:=I+B^{2}.$  Note that due to the obvious inequality   $\left(\|Sf\|_{\mathfrak{H}}  \geq\|f\|_{\mathfrak{H}},\,f\in \mathfrak{H}\right)$ the  operator $S^{-1}$ is bounded on the set $\mathrm{R}(S).$ Taking into account the reasoning given above, we get
 $$
\left(V  f,f\right)_{\mathfrak{H}}=\left(H^{-\frac{1 }{2}} S^{-1}     H^{- \frac{1}{2}}   f,f\right)_{\mathfrak{H}}=
\left( S^{-1}     H^{- \frac{1}{2}}   f,H^{-\frac{1 }{2}}f\right)_{\mathfrak{H}}\leq
$$
$$
\leq \|S^{-1}     H^{- \frac{1}{2}}   f \| _{\mathfrak{H}}  \|H^{- \frac{1}{2}}   f \|_{\mathfrak{H}}\leq   \|S^{- 1}\|  \cdot  \|   H^{- \frac{1}{2}}   f \|^{2}_{\mathfrak{H}}
=
  \|S^{- 1}       \|\cdot\left(R_{H }  f,f\right)_{\mathfrak{H}},\;f\in \mathfrak{H}.
$$
On the other hand,  it is easy to see that  $(S^{-1}f,f)_{\mathfrak{H}}\geq \|S^{-1}f\|^{2}_{\mathfrak{H}},\,f\in \mathrm{R}(S).$ At the same time   it is obvious that   $S$ is bounded and we have   $\|S^{-1}f\|_{\mathfrak{H}}\geq \|S\|^{-1} \|f\|_{\mathfrak{H}},\,f\in \mathrm{R}(S).$ Using   these estimates, we have
$$
\left(V f,f\right)_{\mathfrak{H}}=\left( S^{-1}     H^{- \frac{1}{2}}   f,H^{-\frac{1 }{2}}f\right)_{\mathfrak{H}}\geq  \|S^{-1}     H^{- \frac{1}{2}}   f \|^{2}_{\mathfrak{H}}\geq
$$
$$
\geq  \| S   \|^{-2} \cdot  \|   H^{- \frac{1}{2}}   f \|^{2}_{\mathfrak{H}}= \| S   \|^{-2}  \cdot\left(R_{H }  f,f\right)_{\mathfrak{H}},\;f\in \mathfrak{H}.
$$
  Note that due to Theorem \ref{T2.1} the operator $R_{H}$ is compact. Combining \eqref{2.23} with Theorem \ref{T2.1}, we get that the operator $V$ is compact.  Taking into account these facts and   using   Lemma 1.1 \cite[p.45]{firstab_lit:1Gohberg1965}, we obtain
 \eqref{2.12}.
\end{proof}

\section{Singular numbers     and completeness of the root vectors}

The following   theorem   is  formulated in   terms of      order   $\mu:=\mu(H)$    and  devoted to the Schatten-von Neumann classification of     the  operator     $R_{\tilde{W}}.$

\begin{teo}\label{T2.3} We have the following classification

\begin{equation*}
R_{\tilde{W}}\in  \mathfrak{S}_{p},\,p= \left\{ \begin{aligned}
\!l,\,l>2/\mu,\,\mu\leq1,\\
   1,\,\mu>1    \\
\end{aligned}
 \right.\;.
\end{equation*}
Moreover   under  the  assumption   $ \lambda_{n}(R_{H})\geq  C \,n^{-\mu},\,n\in \mathbb{N},$  we have
$$
 R_{ \tilde{W}}\in\mathfrak{S}_{p}\;  \Rightarrow \;\mu p>1,\;1\leq p<\infty,
$$
where $\mu:=\mu(H).$
 \end{teo}
\begin{proof}
  Consider the case $(\mu\leq1).$   Since we    already know  that   $R_{  \tilde{W }   }^{*}=R_{ \tilde{W}^{+} }^{\!},$ then it can easily be checked that
the operator $ R_{  \tilde{W }  }^{*}  R_{ \tilde{W }}^{\, } $   is a selfadjoint positive compact operator. Due to the well-known fact   \cite[p.174]{firstab_lit:Krasnoselskii M.A.} there exists the operator $|R_{\tilde{W }}  |.$
   By virtue  of Theorem 9.2 \cite[p.178]{firstab_lit:Krasnoselskii M.A.} the operator $|R_{\tilde{W }}  |$ is compact.
Since $\mathrm{N}(|R_{\tilde{W }}|^{2}) =0,$  it follows that $\mathrm{N}(|R_{\tilde{W }}|)=0.$ Hence      applying   Theorem \cite[p.189]{firstab_lit: Ahiezer1966}, we get that the operator
$|R_{\tilde{W }}|$ has an infinite  set of the eigenvalues.   Using  condition \eqref{2.3} (iii), we get
$$
{\rm Re}(R_{ \tilde{W}}f,f)_{\mathfrak{H}}\geq C_{0}\|R_{ \tilde{W}}f\|^{2}_{\mathfrak{H}},\;f\in \mathfrak{H}.
$$
Hence
$$
(|R_{\tilde{W }}|^{2} f,f)_{\mathfrak{H}}=\|R_{ \tilde{W}}f\|^{2}_{\mathfrak{H}}\leq C^{-1}_{0}{\rm Re}(R_{ \tilde{W}}f,f)_{\mathfrak{H}}= C^{-1}_{0}(V f,f)_{\mathfrak{H}},\,V:=  \left(R_{\tilde{W}}\right)_{\mathfrak{R}}.
$$
Since we already know that the operators $|R_{\tilde{W }}|^{2},V$ are compact, then  using  Lemma 1.1 \cite[p.45]{firstab_lit:1Gohberg1965}, Theorem \ref{T2.2},  we get
\begin{equation}\label{2.27}
 \lambda_{i} (|R_{\tilde{W }}|^{2} )\leq C^{-1}_{0} \,\lambda_{i}(V)\leq  C  i^{-\mu },\;i\in \mathbb{N}.
\end{equation}
   Recall  that by  definition  we have     $s_{i}(R_{ \tilde{W}})= \lambda_{ i }( |R_{\tilde{W }}|  ).$ Note that the operators $|R_{\tilde{W }}|,|R_{\tilde{W }}|^{2}$ have the  same eigenvectors.
This fact can be easily proved if we note   the obvious  relation
$
|R_{\tilde{W }}|^{2}f_{i}=|\lambda _{i} (|R_{\tilde{W }}|)|^{2} f_{i},\, i\in \mathbb{N}
$
and
   the  spectral representation for the  square root of a  selfadjoint positive compact operator
$$
|R_{\tilde{W }}|f=\sum\limits_{i=1}^{\infty}\sqrt{\lambda _{ i }(|R_{\tilde{W }}|^{2})}  \left(f,\varphi_{i}\right)  \varphi_{i}, \,f\in \mathfrak{H},
$$
where  $f_{i}\,, \varphi_{i}$ are  the  eigenvectors  of the  operators $|R_{\tilde{W }}|,|R_{\tilde{W }}|^{2}$ respectively  (see (10.25) \cite[p.201]{firstab_lit:Krasnoselskii M.A.}).  Hence
$ \lambda_{ i }( |R_{\tilde{W }}|  ) =
\sqrt{\lambda _{ i }( |R_{\tilde{W }}|^{2}  )} ,\,i\in \mathbb{N}.$
Combining this fact with  \eqref{2.27}, we get
$$
\sum\limits_{i=1}^{\infty}s^{p}_{i}(R_{ \tilde{W}})=
\sum\limits_{i=1}^{\infty}\lambda_{i}^{\frac{p}{2} }( |R_{\tilde{W }}|^{2}  )\leq C \sum\limits_{i=1}^{\infty} i^{-\frac{\mu p }{2}}.
$$
This completes the proof    for the case $(\mu\leq1).$

Consider the case $(\mu>1).$ It follows from
  \eqref{2.23} that  the  operator $V$ is positive     and   bounded. Hence   by virtue  of Lemma 8.1 \cite[p.126]{firstab_lit:1Gohberg1965}, we have that for any  orthonormal basis $\{\psi_{i}\}_{1}^{\infty}\subset \mathfrak{H}$ the following equalities  hold
\begin{equation}\label{2.28}
\sum\limits_{i=1}^{\infty}{\rm Re}(R_{\tilde{W}}\psi_{i},\psi_{i})_{\mathfrak{H}}=\sum\limits_{i=1}^{\infty}  (V \psi_{i},\psi_{i})_{\mathfrak{H}}=\sum\limits_{i=1}^{\infty}  (V\, \varphi_{i},\varphi_{i})_{\mathfrak{H}} ,
\end{equation}
where $\{\varphi_{i}\}_{1}^{\infty}$ is the orthonormal  basis of  the  eigenvectors of the operator $V.$
Due to  Theorem   \ref{T2.2}, we get
$$
  \sum\limits_{i=1}^{\infty}  (V \varphi_{i},\varphi_{i})_{\mathfrak{H}} =\sum\limits_{i=1}^{\infty}
 s_{i}(V) \leq C  \sum\limits_{i=1}^{\infty}  i^{-\mu }  .
$$
By virtue of  Lemma \ref{L2.1}, we get $ |{\rm Im} (R_{\tilde{W}}\psi_{i},\psi_{i})_{\mathfrak{H}}|\leq k^{-1} (\xi)\, {\rm Re}(R_{\tilde{W}}\psi_{i},\psi_{i})_{\mathfrak{H}}.$ Combining this fact with \eqref{2.28}, we get that  the    following series is convergent
$$
\sum\limits_{i=1}^{\infty} (R_{\tilde{W}}\psi_{i},\psi_{i})_{\mathfrak{H}}<\infty.
$$
Hence by definition \cite[p.125]{firstab_lit:1Gohberg1965} the operator $R_{\tilde{W}}$ has  a finite matrix trace.
 Using    Theorem 8.1 \cite[p.127]{firstab_lit:1Gohberg1965}, we get   $R_{ \tilde{W}}\in \mathfrak{S}_{1}.$ This completes the    proof for the case   $(\mu>1).$

Now, assume that  $ \lambda_{n}(R_{H})\geq  C \,n^{-\mu},\,n\in \mathbb{N},\,0\leq\mu<\infty.$     Let us show that the   operator $V$ has the complete orthonormal  system of the  eigenvectors.   Using   formula \eqref{2.26},   we get
$$
V^{\!^{-1}}=2H^ \frac{1 }{2}  (I+B^{2})      H^{ \frac{1}{2}},\; \mathrm{D} (V^{\!^{-1}})= \mathrm{R} (V).
$$
Let us prove that   $\mathrm{D}(V^{\!^{-1}})\subset \mathrm{D}(H).$    Note that   the set $\mathrm{D}(V^{^{\!-\!1}})$ consists  of the  elements $f+g,$ where $f\in \mathrm{D}(\tilde{W}),\,g\in \mathrm{D}(\tilde{W}^{+}).$    Using  representation \eqref{2.13}, it is easy to prove that
 $ \mathrm{D}(\tilde{W})\subset\mathrm{D}(H),\,\mathrm{D}(\tilde{W}^{+}) \subset\mathrm{D}(H) .$ This gives the desired result.
Taking into account the facts proven  above, we get
\begin{equation}\label{2.29}
 (V^{^{\!-\!1}}\!\!f,f)_{\mathfrak{H}}  =2(  S      H^{ \frac{1}{2}}f,H^ \frac{1 }{2}f)_{\mathfrak{H}}\geq 2\|H^ \frac{1 }{2} f\|^{2}_{\mathfrak{H}}=2(Hf,f)_{\mathfrak{H}}  ,\,f\in \mathrm{D}(V^{^{\!-\!1}}),
\end{equation}
where $S=I+B^{2}.$
 Since  $V$ is selfadjoint, then   due to  Theorem 3 \cite[p.136]{firstab_lit: Ahiezer1966} the operator    $V^{^{\!-\!1}}$ is selfadjoint.
  Combining  \eqref{2.29} with
 Lemma \ref{L2.3} we get that    $V^{^{\!-\!1}}$ is strictly accretive.
Using these facts   we can write
\begin{equation}\label{2.30}
\|f\| _{V^{^{ -\!1}}} \geq C\|f\|_{H  } ,\,f\in  \mathfrak{H}_{ V^{^{ -\!1}}} .
\end{equation}
Since  the operator $H$ has a discrete spectrum (see Theorem 5.3 \cite{firstab_lit:1kukushkin2018}), then any set  bounded   with respect to the      norm $\mathfrak{H}_{H}$ is a compact set   with respect to the norm    $\mathfrak{H}$ (see Theorem 4 \cite[p.220]{firstab_lit:mihlin1970}). Combining this fact with \eqref{2.30}, Theorem 3 \cite[p.216]{firstab_lit:mihlin1970}, we get
  that the operator $V^{^{\!-\!1}}$ has a discrete spectrum, i.e.   it has   the infinite set of  the eigenvalues
$\lambda_{1}\leq\lambda_{2}\leq...\leq\lambda_{i}\leq..., \, \lambda_{i} \rightarrow \infty  ,\,i\rightarrow \infty$ and the   complete orthonormal system of the  eigenvectors.
 Now  note that the operators $V,\,V^{^{\!-\!1}}$ have the same eigenvectors.  Therefore   the operator   $V$ has the  complete orthonormal  system of the  eigenvectors. Recall  that any  complete orthonormal system  is a  basis in  separable Hilbert space. Hence the complete orthonormal  system of the   eigenvectors of the operator $V$ is
  a basis in the space $\mathfrak{H}.$
 Let  $\{\varphi_{i}\}_{1}^{\infty}$ be the  complete orthonormal system  of the   eigenvectors of the  operator  $V$   and suppose    $
R_{ \tilde{W}}\in\mathfrak{S}_{p};$ then by virtue of   inequalities  (7.9) \cite[p.123]{firstab_lit:1Gohberg1965}, Theorem \ref{T2.2},  we get
\begin{equation}\label{2.31}
\sum\limits_{i=1}^{\infty}|s_{i} (R_{\tilde{W}} )|^{p}\geq\sum\limits_{i=1}^{\infty}| (R_{\tilde{W}}\varphi_{i},\varphi_{i})_{\mathfrak{H}}|^{p}\geq\sum\limits_{i=1}^{\infty}|{\rm Re}(R_{\tilde{W}}\varphi_{i},\varphi_{i})_{\mathfrak{H}}|^{p}=
$$
$$
=\sum\limits_{i=1}^{\infty}| (V \varphi_{i},\varphi_{i})_{\mathfrak{H}}|^{p}=\sum\limits_{i=1}^{\infty}
|\lambda_{i}(V)|^{p}\geq C   \sum\limits_{i=1}^{\infty}  i^{- \mu p }  .
\end{equation}
We claim   that $\mu p>1.$   Assuming the converse in the previous   inequality,  we come   to   contradiction with the condition  $R_{ \tilde{W}}\in\mathfrak{S}_{p}.$
This completes the proof.
\end{proof}

  The following theorem establishes  the  completeness property of  the  system  of   root vectors   of the operator $R_{\tilde{W}}.$

\begin{teo}\label{T2.4}
Suppose  $\theta< \pi \mu/2;$   then the  system  of   root vectors of the operator $R_{ \tilde{W} }$ is complete, where $\theta$ is   the   semi-angle of the     sector $ \mathfrak{L}_{0}(\theta)\supset \Theta (\tilde{W}),\,\mu:=\mu(H).$
 \end{teo}
\begin{proof}
Using    Lemma \ref{L2.1}, we have
\begin{equation}\label{2.32}
|{\rm Im} (R_{\tilde{W}}f,f)_{\mathfrak{H}}|\leq k^{-1}(\xi) \, {\rm Re}(R_{\tilde{W}}f,f)_{\mathfrak{H}},\,f\in \mathfrak{H}.
\end{equation}
Therefore   $ \overline{\Theta (R_{\tilde{W}})}\subset \mathfrak{L}_{0}(\theta).$  Note that the map
     $z: \mathbb{C}\rightarrow \mathbb{C},\;  z=1/\zeta$ takes each    eigenvalue  of the operator $R_{\tilde{W}}$  to the eigenvalue  of the operator $\tilde{W}.$ It is also clear    that   $z:\mathfrak{L}_{0}(\theta)\rightarrow \mathfrak{L}_{0}(\theta).$
  Using   the definition  \cite[p.302]{firstab_lit:1Gohberg1965} let us consider the following set
\begin{equation*}
   \mathfrak{P}:=\left\{z:\,z=t\, \xi,\,\xi\in  \overline{\Theta (R_{\tilde{W}})},\, 0\leq t<\infty\right\}.
\end{equation*}
It is easy to see  that $  \mathfrak{P}$ coincides with a closed sector of the complex plane  with the  vertex situated  at the point zero. Let us denote  by
 $\vartheta(R_{\tilde{W}})$ the  angle of this sector. It is obvious that $  \mathfrak{P}\subset \mathfrak{ L }_{0}(\theta).$ Therefore    $ 0 \leq\vartheta(R_{\tilde{W}}) \leq  2\theta.$ Let us prove that     $0 <\vartheta(R_{\tilde{W}}),$ i.e.    the strict inequality holds.    If we assume  that $\vartheta(R_{\tilde{W}})=0,$ then we get  $ e^{-i \mathrm{arg} z} =\varsigma,\,\forall z\in  \mathfrak{P}\setminus 0, $
 where $\varsigma$  is a constant independent on $z.$ In consequence of this fact   we have $\mathrm{Im}\, \Theta (\varsigma R_{\tilde{W}})=0.$ Hence the operator $\varsigma R_{\tilde{W}}$ is symmetric (see Problem  3.9 \cite[p.269]{firstab_lit:kato1980}) and by virtue of the fact   $ \mathrm{D} (\varsigma R_{\tilde{W}})=\mathfrak{H}$ one is selfadjoint. On the other hand, taking into account   the      equality  $R^{\ast}_{ \tilde{W}}=R^{\,}_{ \tilde{W}^{+}}$   (see the proof of   Theorem \ref{T2.2}), we have $(\varsigma R_{\tilde{W}}f,g)_{\mathfrak{H}}=( f,\bar{\varsigma} R_{\tilde{W}^{+}}g)_{\mathfrak{H}},\,  f,g\in \mathfrak{H}.$ Hence
 $\varsigma R_{\tilde{W}}=\bar{\varsigma} R_{\tilde{W}^{+}}.$ In the particular case we have $\forall f\in \mathfrak{H},\,\mathrm{Im}f=0:\, \mathrm{Re}\,\varsigma\,R_{\tilde{W}}f= \mathrm{Re}\,\varsigma\,R_{\tilde{W}^{+}}f,\,\mathrm{Im}\,\varsigma\,R_{\tilde{W}}f= -\mathrm{Im}\,\varsigma\,R_{\tilde{W}^{+}}f  .  $ It implies that  $\mathrm{N}(R_{\tilde{W}})\neq 0.$  This contradiction   concludes  the proof of the fact $ \vartheta(R_{\tilde{W}})>0.$
    Let us use   Theorem 6.2 \cite[p.305]{firstab_lit:1Gohberg1965}    according to  which   we have the following. If the following two conditions  (a) and (b) are fulfilled, then   the  system  of   root vectors  of the operator $R_{\tilde{W}}$ is complete.  \\

\noindent  a) $\vartheta(R_{\tilde{W}})=  \pi/ d ,$ where $d>1,$\\

\noindent b) for some $\beta,$ the  operator $B:=\mathfrak{Im}\left(e^{ i \beta}R_{\tilde{W}} \right) :\;s_{i}(B)=o(i^{-1/d }),\,i\rightarrow \infty .$\\

\noindent Let us show that conditions (a) and (b) are fulfilled. Note that due to Lemma \ref{L2.1} we have  $0\leq\theta<\pi/2.$
Hence $ 0<\vartheta(R_{\tilde{W}})<\pi.$ It implies that  there exists  $1<d<\infty$ such that
$\vartheta(R_{\tilde{W}})=  \pi/ d.$  Thus   condition  (a) is fulfilled.
  Let us choose the  certain value  $\beta=\pi/2$ in   condition (b) and notice  that $\mathfrak{Im}\left(e^{ i \pi/2}R_{\tilde{W}} \right) =\mathfrak{Re} R_{\tilde{W}} . $ Since the operator $V:= \mathfrak{Re} R_{\tilde{W}}  $ is  selfadjoint, then we have   $s_{i}(V)=\lambda_{i}(V),\,i\in \mathbb{N}.$
  In consequence of Theorem \ref{T2.2}, we obtain
\begin{equation*}
 s_{i}(V)\, i^{1/d} = s_{i}(V)\, i^{ \mu} \cdot i^{1/d-\mu}    \leq C\cdot i^{1/d-\mu} ,\,i\in \mathbb{N}.
\end{equation*}
 Hence to achieve  condition  (b),      it is sufficient  to show that  $d>\mu^{-1}.$
By virtue of the conditions $ \vartheta(R_{\tilde{W}}) \leq 2 \theta,\,\theta< \pi \mu/2,$ we have $d=\pi/\vartheta(R_{\tilde{W}})\geq \pi/2\theta>\mu^{-1}.$ Hence we obtain  $s_{i}(V)=o(i^{-1/d }).$
 Since   both     conditions (a),(b) are fulfilled, then using   Theorem 6.2 \cite[p.305]{firstab_lit:1Gohberg1965} we complete the proof.
\end{proof}

  Theorem \ref{T2.3} is  devoted to    the  description  of $s$-numbers    behavior  but     questions related with   asymptotic of the eigenvalues $\lambda_{i}(R_{\tilde{W}}),\,i\in \mathbb{N}$ are still relevant in our work.  It is a well-known fact  that for any  bounded  operator with the compact  imaginary component    there is a relationship between $s$-numbers of the  imaginary component  and the  eigenvalues    (see \cite {firstab_lit:1Gohberg1965}).  Similarly using  the information on    $s$-numbers of the real component,    we  can obtain  an asymptotic formula for the eigenvalues $\lambda_{i}(R_{\tilde{W}}),\,i\in \mathbb{N}.$      This idea is realized in the following theorem.
\begin{teo}\label{T2.5}
 The following inequality   holds
 \begin{equation}\label{2.33}
\sum\limits_{i=1}^{n}|\lambda_{i}(R_{\tilde{W}})|^{p}\leq
  \sec ^{p} \theta \,\left\|S^{- 1}      \right\| \sum\limits_{i=1}^{n } \,\lambda^{ p}_{i}(R_{H}),
\end{equation}
$$
 n=1,2,...,\, \nu(R_{\tilde{W}}),\,1\leq p<\infty.
$$
Moreover   if  $\nu(R_{\tilde{W}})=\infty$ and the  order $\mu(H)\neq0,$  then  the following asymptotic formula  holds
\begin{equation}\label{2.34}
|\lambda_{i}(R_{\tilde{W}})|=  o\left(i^{-\mu+\varepsilon}    \right)\!,\,i\rightarrow \infty,\;\forall \varepsilon>0.
\end{equation}
\end{teo}

\begin{proof}
Let $L$ be a bounded    operator with  a compact imaginary component.  Note that  according to Theorem 6.1 \cite[p.81]{firstab_lit:1Gohberg1965},   we have
 \begin{equation}\label{2.35}
\sum\limits_{m=1}^{k}|{\rm Im}\,\lambda_{m}(L)|^{p}\leq \sum\limits_{m=1}^{k}
|s_{m} (\mathfrak{Im} L    )|^{p},\;\left(k=1,\,2,...,\,\nu_{\,\mathfrak{I}}(L)\right),\,1\leq p<\infty,
\end{equation}
where    $ \nu_{\,\mathfrak{I}}(L) \leq \infty $ is   the sum of   all  algebraic multiplicities corresponding to   the not real  eigenvalues   of the operator $L$ (see  \cite[p.79]{firstab_lit:1Gohberg1965}).    It can easily be checked that
 \begin{equation}\label{2.36}
\mathfrak{Im}(iL) = \mathfrak{Re}L ,\; \mathrm{Im }\, \lambda_{m} (i\,L)= \mathrm{Re}\,  \lambda_{m}(L),\;m\in \mathbb{N}.
\end{equation}
By virtue of  \eqref{2.32}, we have ${\rm Re}\,\lambda_{m}(R_{\tilde{W}})>0,\,m=1,2,...,\nu\left(  R_{\tilde{W}} \right) .$ Combining this fact with     \eqref{2.36}, we get
 $\nu_{\mathfrak{I}}(i R_{\tilde{W}} )=\nu\left(R_{\tilde{W}}\right).$ Taking  into account the previous  equality  and    combining \eqref{2.35},\eqref{2.36},       we obtain
 \begin{equation}\label{2.37}
\sum\limits_{m=1}^{k}|{\rm Re}\,\lambda_{m}(R_{\tilde{W}})|^{p}\leq \sum\limits_{m=1}^{k}
|s_{m} (V )|^{p} ,\; \left(k=1,\,2,...\,,\nu  (R_{\tilde{W}})\right),\,V:= \mathfrak{Re }R_{\tilde{W}} .
\end{equation}
 Note that by virtue of \eqref{2.32}, we have
$$
|{\rm Im}\,\lambda_{m}(R_{\tilde{W}})|\leq  \tan \theta \,{\rm Re}\,\lambda_{m}\,(R_{\tilde{W}}),\,m\in \mathbb{N}.
$$
Hence
 \begin{equation}\label{2.38}
| \lambda_{m}(R_{\tilde{W}})|= \sqrt{|{\rm Im}\,\lambda_{m}(R_{\tilde{W}})|^{2}+|{\rm Re}\,\lambda_{m}(R_{\tilde{W}})|^{2}}\leq
$$
$$
  \leq \sqrt{\tan^{2}\theta+1} \;|{\rm Re}\,\lambda_{m}(R_{\tilde{W}})|
=\sec \theta\, |{\rm Re}\,\lambda_{m}(R_{\tilde{W}})|,\,m\in \mathbb{N}.
\end{equation}
Combining \eqref{2.37},\eqref{2.38}, we get
 \begin{equation*}
\sum\limits_{m=1}^{k}| \lambda_{m}(R_{\tilde{W}})|^{p}\leq \sec^{p}\!\theta\,\sum\limits_{m=1}^{k}
|s_{m} (V )|^{p} ,\; \left(k=1,\,2,...\,,\nu  (R_{\tilde{W}})\right).
\end{equation*}
Using   \eqref{2.12}, we complete     the proof of   inequality \eqref{2.33}.

Suppose $  \nu(R_{\tilde{W}})=\infty,\, \mu(H)\neq0 $ and let us prove \eqref{2.34}.   Note   that for    $\mu>0$ and for any $ \varepsilon>0,$ we can choose  $p$  so that
 $\mu p>1,\, \mu-\varepsilon<1/p .$
Using the condition $\mu p>1,$  we obtain  convergence of the series on  the left side of    \eqref{2.33}. It implies that
 \begin{equation}\label{2.39}
|\lambda_{i}(R_{\tilde{W}})| i^{1/p}\rightarrow 0,\;i\rightarrow\infty.
\end{equation}
It is obvious that
$$
|\lambda_{i}(R_{\tilde{W}})| i^{ \mu-\varepsilon}<|\lambda_{i}(R_{\tilde{W}})| i^{ 1/p} ,\,i\in \mathbb{N}.
$$
Taking into account \eqref{2.39}, we obtain  \eqref{2.34}.
\end{proof}

\section{Comparison analysis with the  subordination concept }

\subsection{Counterarguments }
 We begin with definitions.
Suppose  $\Omega$
is a convex domain of the n-dimensional Euclidian space with the  sufficient smooth
boundary,
$L_{2}(\Omega)$ is  a  complex Lebesgue space  of summable with square functions,   $H^{2}(\Omega),$ $H^{1}(\Omega)$
 are  complex Sobolev spaces, $D_{i}f:= \partial f/\partial x_{i} ,\,1 \leq i \leq n$ is the    weak   partial derivatives of  of the function $f.$   Consider a sum of a  uniformly elliptic operator
and the extension of the Kipriyanov fractional differential operator of   order $0 <\alpha < 1$ (see Lemma
2.5 \cite{firstab_lit:1kukushkin2018})
\begin{equation*}
 Lu:=-  D_{j} ( a^{ij} D_{i}f) \, +   \mathfrak{D}  ^{ \alpha }_{0+}f ,
 $$
 $$
   \mathrm{D}(L)=H^{2}(\Omega)\cap H^{1}_{0}(\Omega),
  \end{equation*}
 with the following  assumptions relative to the  real-valued   coefficients
\begin{equation*}
 a^{ij}(Q) \in C^{1}(\bar{\Omega}),\,   a^{ij}\xi _{i}  \xi _{j}  \geq   a  |\xi|^{2} ,\,  a  >0 .
 \end{equation*}
 It was proved in the  paper  \cite{firstab_lit:1kukushkin2018} that the operator
 $   L^{+}f:=-D_{i} ( a^{ij} D_{j}f)+ \mathfrak{D}   ^{ \alpha }_{d-}f,$
 $
 \mathrm{D}   (L^{+})=\mathrm{D}(L)
 $  is formal adjoint with respect  to $L.$
Note that  in accordance with Theorem 2    \cite{firstab_lit(JFCA)kukushkin2018} we have $\mathrm{R}(L) = \mathrm{R}(L^{+}) = L_{2}(\Omega),$ due to the reasonings  of  Theorem 3.1 \cite{firstab_lit:1.1kukushkin2017} the operators
 $L,L^{+}$ are strictly accretive. Taking into account  these facts we can conclude that the operators $L,L^{+}$ are closed
(see problem 5.15 \cite[p.165]{firstab_lit:kato1980}). Consider the operator $\mathfrak{Re}L.$  Having made the   absolutely analogous
reasonings   as in the previous case,  we  conclude that the operator $\mathfrak{Re}L$ is closed. Applying
the   reasonings of  Theorem 4.3 \cite{firstab_lit:1kukushkin2018}, we obtain that the operator  $\mathfrak{Re}L$ is   selfadjoint and strictly accretive.
Recall  that to apply    the methods    described  in the paper  \cite{firstab_lit:Shkalikov A.} we must have some decomposition of the initial
operator $L$ on a sum $L=\mathcal{T}+\mathcal{A},$     where
$\mathcal{T}$ must be an  operator of a special type   either a  selfadjoint or  a normal operator.
Note that the uniformly elliptic operator of  second order  is
neither selfadjoint  no normal  in the general case.   To demonstrate  the   significance of the method obtained in this paper,   we would like to note  that    a search  for  a convenient decomposition of $L$ on a sum
of a selfadjoint operator and  some operator does not  seem  to be a reasonable  way. Now  to justify this  claim  we consider one of  possible decompositions of $L$ on  a sum.   Consider  a selfadjoint strictly accretive operator
$  \mathcal{T}  :\mathfrak{H}\rightarrow \mathfrak{H}. $

\begin{deff}\label{D2.1}
In accordance with the   definition of the paper  \cite{firstab_lit:Shkalikov A.},  a quadratic form $\mathfrak{a}: = \mathfrak{a}[f]$ is called a $\mathcal{T}$ - subordinated form if the following condition holds
\begin{equation}\label{2.40}
\left|\mathfrak{a}[f]\right| \leq b\, \mathfrak{t}[f] +M\|f\|^{2}_{\mathfrak{H}},\;
  \mathrm{D}(\mathfrak{a}) \supset \mathrm{D}(\mathfrak{t}),\; b < 1,\; M > 0,
\end{equation}
where $\mathfrak{t}[f]=\|\mathcal{T}^{\frac{1}{2}}\|^{2}_{\mathfrak{H}},\,f\in \mathrm{D}(\mathcal{T}^{\frac{1}{2}}).$
 The form $\mathfrak{a}$  is called  a completely $\mathcal{T}$ - subordinated form    if besides   of \eqref{2.40}  we have the following additional condition  $\forall \varepsilon>0 \,\exists b,M>0:\, b <\varepsilon.$
\end{deff}
Let us consider the trivial  decomposition  of the operator $L$ on the sum $L=2 \mathfrak{Re}L-L^{+}$ and let us use the notation $\mathcal{T}:=2\mathfrak{Re}L,\, \mathcal{A}:=-L^{+}.$ Then we have $L=\mathcal{T}+\mathcal{A}.$
Due to the sectorial property
proven in   Theorem 4.2 \cite{firstab_lit:1kukushkin2018}  we have
\begin{equation}\label{2.41}
 |(\mathcal{A}f,f)_{L_{2} }| \!=\!\sec \theta_{f}\, |\mathrm{Re}(\mathcal{A}f,f)_{L_{2} }|\! = \!\sec \theta_{f}\, \frac{1}{2}(\mathcal{T}f,f)_{L_{2} },\,f\in \mathrm{D}(\mathcal{T}),
\end{equation}
 where $ 0\leq\theta_{f}\leq \theta,\, \theta_{f}:=\left|\mathrm{arg}  (L^{+}f,f)_{L_{2} }\right|,\, L_{2}:=L_{2}(\Omega)  $   and $\theta$ is the semi-angle   corresponding to the sector $\mathfrak{L}_{0}(\theta).$
  Due to Theorem 4.3 \cite{firstab_lit:1kukushkin2018}  the operator $\mathcal{T}$ is m-accretive. Hence in consequence of Theorem
3.35  \cite[p.281]{firstab_lit:kato1980} we have that  $\mathrm{D}(\mathcal{T} )$   is a core of the operator $\mathcal{T}^{\frac{1}{2}}.$ It implies that we can extend   relation \eqref{2.41} to
\begin{equation}\label{2.42}
 \frac{1}{2} \, \mathbf{\mathfrak{t}}[f]\leq|\mathbf{\mathfrak{a }}[f]|\leq \sec \theta \frac{1}{2} \, \mathbf{\mathfrak{t}}[f],\, f\in \mathrm{D} (\mathbf{\mathfrak{t}}),
\end{equation}
where $\mathfrak{a}$ is a quadratic form generated by $\mathcal{A}$ and $\mathfrak{t}[f] = \|\mathcal{T}^{\frac{1}{2}}f\|^{2}_{\mathfrak{H}}.$
 If we consider
the case $0<\theta<\pi/3,$ then it is obvious  that there exist  constants  $b < 1$ and $M > 0$ such that
the following inequality holds
$$
|\mathfrak{a}[f]| \leq b\, \mathfrak{t}[f] +M\|f\|^{2}_{L_{2}},\;
 f\in \mathrm{D}(\mathfrak{t}).
$$
Hence the form $\mathfrak{a}$ is a $\mathcal{T}$ - subordinated form.
In accordance with  the  definition given in the paper \cite{firstab_lit:Shkalikov A.} it means $\mathcal{T}$ - subordination of the operator $\mathcal{A}$  in the sense of   form. Assume
that $\forall \varepsilon>0\,\exists b,M>0:\,  b <\varepsilon.$  Using  inequality \eqref{2.42},  we get
$$
   \frac{1}{2}\mathfrak{t}[f]   \leq \varepsilon \, \mathfrak{t}[f] + M(\varepsilon)\|f\|^{2}_{L_{2} };\; \mathfrak{t}[f]\leq \frac{2 M(\varepsilon)}{(1-2\varepsilon)}\|f\|^{2}_{L_{2} },  \,f\in  \mathrm{D}(\mathfrak{t}),\, \varepsilon < 1/2.
$$
Using the strictly accretive property of the operator $L$ (see inequality  (4.9) \cite{firstab_lit:1kukushkin2018}), we   obtain
$$
C\|f\|^{2}_{H^{1}_{0} }  \leq  \mathfrak{t}[f],  \,f\in  \mathrm{D}(\mathfrak{t}).
$$
On the other hand, using the results of the paper  \cite{firstab_lit:1kukushkin2018},  it is easy to prove that $H^{1}_{0}(\Omega)\subset\mathrm{D}(\mathfrak{t}).$
Taking into account the facts considered above,   we get
$$
\|f\|_{H^{1}_{0} }   \leq  C \|f\|_{L_{2}},  \,f\in H^{1}_{0}(\Omega)  ,
$$
   but as it is well known this inequality is not true. This contradiction shows us that the form $\mathfrak{a}$ is not    a  completely $\mathcal{T}$ - subordinated form. It implies that we cannot  use     Theorem 8.4 \cite{firstab_lit:Shkalikov A.}   which could give us an opportunity to describe the spectral properties of the operator  $L.$
Note that the reasonings corresponding to  another trivial decomposition of $L$ on a sum  is analogous.

     This rather particular example does not aim to show the inability of using   remarkable
methods considered in the paper  \cite{firstab_lit:Shkalikov A.}  but only creates   prerequisite for    some   value  of    another
method  based on  using spectral properties of the real component of the  initial operator $L.$

\subsection{Arguments}

 In this paragraph, we would like to demonstrate   the    effectiveness    of   the elaborated   method.

  Suppose   $\mathfrak{H} := L_{2}(\Omega
),\, \mathfrak{H}^{+} := H^{1}_{0}(\Omega),\, Tf :=  -D_{j}(a^{ij}D_{i}f),\, Af :=
   \mathfrak{D}^{ \alpha }_{0+}f,\; \mathrm{D}(T),\mathrm{D}(A) =H^{2}(\Omega)\cap H_{0}^{1}(\Omega);$ then    due to the Rellich-Kondrachov  theorem we have that  condition
\eqref{2.1} is fulfilled. Due to  the  results obtained in the paper  \cite{firstab_lit:1kukushkin2018} we have that   condition  \eqref{2.3} is  fulfilled. Applying the results obtained in the paper  \cite{firstab_lit:1kukushkin2018}
we  conclude that the operator $\mathfrak{Re}L$ has   non-zero order. Hence we can apply the abstract results of
this paper to the   operator $L.$ In fact, Theorems \ref{T2.3}-\ref{T2.5} describe the  spectral properties of the operator $L.$

 Let us  provide  one more example.     We deal with the differential operator acting in the complex Sobolev  space and defined by the following
expression
$$
\mathcal{L}f := (c_{k}f^{(k)})^{(k)} + (c_{k-1}f^{(k-1)})^{(k-1)}+...+  c_{0}f,
$$
$$
\mathrm{D}(\mathcal{L}) = H^{2k}(I)\cap H_{0}^{k}(I),\,k\in \mathbb{N},
$$
where    $I: = (a, b) \subset \mathbb{R},$ the   complex-valued coefficients
$c_{j}(x)\in C^{(j)}(\bar{I})$ satisfy the condition $  {\rm sign} (\mathrm{Re} c_{j}) = (-1)^{j} ,\, j = 1, 2, ..., k.$
It is easy
to see that
$$
\mathrm{Re}(\mathcal{L}f,f)_{L_{2}(I)}\geq\sum\limits_{j=0}^{k}|\mathrm{Re} c_{j}|\,\|f^{(j)}\|^{2}_{L_{2}(I)}\geq C \|f^{(j)}\|^{2}_{H_{0}^{k}(I)},\;f\in \mathrm{D}(\mathcal{L}).
$$
On the other hand
$$
|(\mathcal{L}f,f)_{L_{2}(I)}|=\left|\sum\limits_{j=0}^{k}(-1)^{j}(c_{j}f^{(j)},g^{(j)} )_{L_{2}(I)}\right|\leq
\sum\limits_{j=0}^{k}\left|(c_{j}f^{(j)},g^{(j)} )_{L_{2}(I)}\right|\leq
$$
$$
\leq C \sum\limits_{j=0}^{k} \|f^{(j)}\| _{L_{2}(I)}\|g^{(j)}\| _{L_{2}(I)}\leq
\|f \| _{H^{k}_{0}(I)}\|g \| _{H^{k}_{0}(I)},\;f\in \mathrm{D}(\mathcal{L}).
$$
Consider the  Riemann-Liouville   operators of fractional differentiation of   arbitrary non-negative
order $\alpha$ (see \cite[p.44]{firstab_lit:samko1987})  defined by the expressions
\begin{equation*}
 D_{a+}^{\alpha}f=\left(\frac{d}{dx}\right)^{[\alpha]+1}\!\!\!\!I_{a+}^{1-\{\alpha\}}f;\;
 D_{b-}^{\alpha}f=\left(-\frac{d}{dx}\right)^{[\alpha]+1}\!\!\!\!I_{b-}^{1-\{\alpha\}}f,
\end{equation*}
where the fractional integrals of      arbitrary positive order  $\alpha$ defined by
$$
\left(I_{a+}^{\alpha}f\right)\!(x)=\frac{1}{\Gamma(\alpha)}\int\limits_{a}^{x}\frac{f(t)}{(x-t)^{1-\alpha}}dt,
 \left(I_{b-}^{\alpha}f\right)\!(x)=\frac{1}{\Gamma(\alpha)}\int\limits_{x}^{b}\frac{f(t)}{(t-x)^{1-\alpha}}dt
, f\in L_{1}(I).
$$
Suppose  $0<\alpha<1,\, f\in AC^{l+1}(\bar{I}),\,f^{(j)}(a)=f^{(j)}(b)=0,\,j=0,1,...,l;$ then the next formulas follows
from   Theorem 2.2 \cite[p.46]{firstab_lit:samko1987}
\begin{equation}\label{2.43}
 D_{a+}^{\alpha+l}f= I_{a+}^{1- \alpha }f^{(l+1)},\;
 D_{b-}^{\alpha+l}f= (-1)^{l+1}I_{b-}^{1- \alpha }f^{(l+1)}.
\end{equation}
  Further, we need  the following inequalities    (see  \cite{firstab_lit:(vlad)kukushkin2016})
\begin{equation}\label{2.44}
\mathrm{Re} (D_{a+}^{\alpha}f,f)_{L_{2}(I)}\geq C\|f\|^{2}_{L_{2}(I)},\,f\in I_{a+}^{\alpha}(L_{2}),\;
$$
$$
\mathrm{Re} (D_{b-}^{\alpha}f,f)_{L_{2}(I)}\geq C\|f\|^{2}_{L_{2}(I)},\,f\in I_{b-}^{\alpha}(L_{2}),
\end{equation}
where $I_{a+}^{\alpha}(L_{2}),I_{b-}^{\alpha}(L_{2})$ are the  classes of  the  functions representable by the fractional integrals (see\cite{firstab_lit:samko1987}).
  Consider the following operator  with the  constant  real-valued  coefficients
$$
\mathcal{D}f:=p_{n}D_{a+}^{\alpha_{n}}+q_{n}D_{b-}^{\beta_{n}}+p_{n-1}D_{a+}^{\alpha_{n-1}}+q_{n-1}D_{b-}^{\beta_{n-1}}+...+
p_{0}D_{a+}^{\alpha_{0}}+q_{0}D_{b-}^{\beta_{0}},
$$
$$
\mathrm{D}(\mathcal{D}) = H^{2k}(I)\cap H_{0}^{k}(I),\,n\in \mathbb{N},
$$
where $\alpha_{j},\beta_{j}\geq 0,\,0 \leq [\alpha_{j}],[\beta_{j}] < k,\, j = 0, 1, ..., n.,\;$
\begin{equation*}
 q_{j}\geq0,\;{\rm sign}\,p_{j}= \left\{ \begin{aligned}
  (-1)^{\frac{[\alpha_{j}]+1}{2}},\,[\alpha_{j}]=2m-1,\,m\in \mathbb{N},\\
\!\!\!\!\!\!\! \!\!\!\!(-1)^{\frac{[\alpha_{j}]}{2}},\;\,[\alpha_{j}]=2m,\,\,m\in \mathbb{N}_{0}   .\\
\end{aligned}
\right.
\end{equation*}
Using \eqref{2.43},\eqref{2.44},  we get
$$
(p_{j}D_{a+}^{\alpha_{j}}f,\!f)_{L_{2}(I)}\!=
\!p_{j}\!\left(\!\!\left(\frac{d}{dx}\right)^{\!\!\!m}\!\!D_{a+}^{m-1+\{\alpha_{j}\}}\!\!f,\!f   \!\right)_{\!\!L_{2}(I)}\!\!\!\!\!\! =
(\!-1)^{m}p_{j}\!\left(\! I_{a+}^{1-\{\alpha_{j}\}}\!\!f^{(m)}\!\!,\!f^{(m)}   \!\right)_{\!\!L_{2}(I)}\!\! \geq
$$
$$
\geq C\left\|I_{a+}^{1-\{\alpha_{j}\}}f^{(m)}\right\|^{2}_{L_{2}(I)}=
C\left\|D_{a+}^{\{\alpha_{j}\}}f^{(m-1)}\right\|^{2}_{L_{2}(I)}\geq C \left\| f^{(m-1)}\right\|^{2}_{L_{2}(I)},
$$
 where  $f \in \mathrm{D}(\mathcal{D})$ is    a real-valued function and   $ [\alpha_{j}]=2m-1,\,m\in \mathbb{N}.$
  Similarly,  we obtain for orders  $ [\alpha_{j}]=2m,\,m\in \mathbb{N}_{0}$
$$
(p_{j}D_{a+}^{\alpha_{j}}f,f)_{L_{2}(I)}=p_{j}\left( D_{a+}^{2m +\{\alpha_{j}\}}f,f   \right)_{L_{2}(I)}=(-1)^{m}p_{j}\left( D_{a+}^{m+\{\alpha_{j}\}}f ,f^{(m)}   \right)_{L_{2}(I)}=
$$
$$
=(-1)^{m}p_{j}\left( D_{a+}^{ \{\alpha_{j}\}}f^{(m)} ,f^{(m)}   \right)_{\!L_{2}(I)}\geq C \left\| f^{(m)}\right\|^{2}_{L_{2}(I)}.
$$
Thus in both cases  we have
$$
(p_{j}D_{a+}^{\alpha_{j}}f,f)_{L_{2}(I)}\geq C \left\| f^{(s)}\right\|^{2}_{L_{2}(I)},\;s= \big[[\alpha_{j}]/2\big] .
$$
 In the same way, we obtain the inequality
$$
(q_{j}D_{b-}^{\alpha_{j}}f,f)_{L_{2}(I)}\geq C \left\| f^{(s)}\right\|^{2}_{L_{2}(I)},\;s= \big[[\alpha_{j}]/2\big] .
$$
  Hence in the
complex case we have
$$
\mathrm{Re}(\mathcal{D}f,f)_{L_{2}(I)}\geq C \left\| f \right\|^{2}_{L_{2}(I)},\;f\in \mathrm{D}(\mathcal{D}).
$$
Combining   Theorem 2.6 \cite[p.53]{firstab_lit:samko1987}  with  \eqref{2.43}, we get
$$
\left\| p_{j}D_{a+}^{\alpha_{j}}f \right\| _{L_{2}(I)}=  \left\|  I_{a+}^{1-\{\alpha_{j}\}}f^{([\alpha_{j}]+1)} \right\| _{L_{2}(I)}
 \leq C   \left\|   f^{([\alpha_{j}]+1)} \right\|_{L_{2}(I)}\leq C   \left\|   f  \right\|_{H^{k}_{0}(I)};
 $$
 $$
 \;\left\|q_{j}D_{b-}^{\alpha_{j}}f \right\| _{L_{2}(I)}
\leq  C   \left\|   f  \right\|_{H^{k}_{0}(I)},\;f\in \mathrm{D}(\mathcal{D}).
$$
 Hence, we obtain
$$
\left\| \mathcal{D}f \right\| _{L_{2}(I)}\leq C \left\|  f \right\|_{H^{k}_{0}(I)},\;f\in \mathrm{D}(\mathcal{D}).
$$
Now we can formulate the main result. Consider the operator
$$
G=\mathcal{L}+\mathcal{D},
$$
$$
\mathrm{D}(G)=H^{2k}(I)\cap H_{0}^{k}(I).
$$
Suppose  $\mathfrak{H} := L_{2}(I),\, \mathfrak{H}^{+} := H_{0}^{k}(I),\, T := \mathcal{L},\, A := \mathcal{D};$ then due to the well-known fact of the  Sobolev spaces theory
      condition \eqref{2.1} is fulfilled, due to the reasonings given above condition  \eqref{2.3}
is fulfilled. Taking into account the equality
$$
\mathfrak{Re}\mathcal{L} f = ( \mathrm{Re}c_{k} f^{(k)})^{(k)} + ( \mathrm{Re}c_{k-1} f^{(k-1)})^{(k-1)} + ... +  \mathrm{Re}c_{0} f,\;f\in \mathrm{D}(\mathcal{D})
$$
and using  the method described
  in the paper  \cite{firstab_lit:2kukushkin2017}, we   can prove   that the closure of the  operator $\mathfrak{Re}\, G  $  has  a non-zero order.
Hence we can successfully  apply  the abstract results of this paper  to the   operator $G.$ Now it is easily seen that   Theorems \ref{T2.3}-\ref{T2.5} describe the  spectral properties of the operator $G.$

\section{Connection between  singular numbers asyptotics}\label{S2.6}

   We consider  statements   particularly represented in the previous section, however they   will be undergone to  a thorough study since our principal challenge is to obtain an accurate description of the  Schatten-von Neumann class index of a non-selfadjoint operator. Bellow, we produce a slight generalizations of the results represented in the previous section that  gives us a description of spectral properties of a non-selfadjoint operator $L$ acting in $\mathfrak{H}.$

   We have the following classification in terms of the operator order $\mu,$ where it is defined as follows $ \lambda_{n}(R_{H})=O  (n^{-\mu}),\,n\rightarrow\infty.$\\

\begin{teo}\label{T2.6}
Assume that $L$ is a non-sefadjoint operator acting in $\mathfrak{H},$  the following  conditions hold\\

 \noindent  $ (\mathrm{H}1) $ There  exists a Hilbert space $\mathfrak{H}_{+}\subset\subset\mathfrak{ H}$ and a linear manifold $\mathfrak{M}$ that is  dense in  $\mathfrak{H}_{+}.$ The operator $L$ is defined on $\mathfrak{M}.$    \\

 \noindent  $( \mathrm{H2} )  \,\left|(Lf,g)_{\mathfrak{H}}\right|\! \leq \! C_{1}\|f\|_{\mathfrak{H}_{+}}\|g\|_{\mathfrak{H}_{+}},\,
      \, \mathrm{Re}(Lf,f)_{\mathfrak{H}}\!\geq\! C_{2}\|f\|^{2}_{\mathfrak{H}_{+}} ,\,f,g\in  \mathfrak{M},\; C_{1},C_{2}>0.
$
\\

\noindent Let $W$ be  a closure of the   restriction of the operator  $L$ on the set $\mathfrak{M}.$  Then   the following  propositions are true.\\

\noindent $({\bf A})$  If $\mu\leq1,$ then   $\; R_{  W  }\in  \mathfrak{S}_{p},\,p>2/\mu.$ If  $\mu>1,$ then $R_{ W }\in  \mathfrak{S}_{1}.$ \\

 \noindent Moreover, under assumptions
$\lambda _{ n} \left(H \right)=O  (n^{ \mu}),\,\mu>0,$   the following implication holds $$R_{  W}\in\mathfrak{S}_{p},\,p\in [1,\infty),\Rightarrow \mu>1/p.$$

\noindent$({\bf B})$     In the case  $\nu(R_{ W })=\infty,\,\mu \neq0,$  the following relation  holds
$$
|\lambda_{n}(R_{W})|=  o\left(n^{-\tau}    \right)\!,\,n\rightarrow \infty,\;0<\tau<\mu.\\
$$

\noindent $({\bf C})$   Assume that   $\theta< \pi \mu/2\, ,$    where $\theta$ is   the   semi-angle of the     sector $ \mathfrak{L}_{0}(\theta)\supset \Theta (W).$
Then  the system of   root   vectors  of   $R_{ W }$ is complete in $\mathfrak{H}.$\\

\end{teo}

\begin{proof} Note that  the closeness of  the  operator  $W$ follows from
 the first  condition  H2    and   Theorem 3.4 \cite[p.268]{firstab_lit:kato1980}, the detailed reasonings are left to the reader.
Let us show that $W$ is sectorial. By virtue of condition H2, we get
$$
\mathrm{Re}(Wf,f)_{\mathfrak{H}}\geq C_{2}\|f\|^{2}_{\mathfrak{H_{+}}}\geq C_{2}\varepsilon\|f\|^{2}_{\mathfrak{H}_{+}}+\frac{C_{2}(1-\varepsilon)}{C_{0}}\|f\|^{2}_{\mathfrak{H}};
$$
$$
\mathrm{Re}(Wf,f)_{\mathfrak{H}} -k|\mathrm{Im}(Wf,f)_{\mathfrak{H}}|\geq (C_{2}\varepsilon-k C_{1})\|f\|^{2}_{\mathfrak{H}_{+}}  +\frac{C_{2}(1-\varepsilon)}{C_{0}}\|f\|^{2}_{\mathfrak{H}}=\frac{C_{2}(1-\varepsilon)}{C_{0}}\|f\|^{2}_{\mathfrak{H}},
$$
where
$
 k=  \varepsilon C_{2}/C_{1}.
$
Hence $\Theta(W)\subset \mathfrak{L}_{\gamma}(\theta),\,\gamma=C_{2}(1-\varepsilon)/C_{0}.$ Thus, the claim of Lemma \ref{L2.1}   is true regarding   the operator $W.$   Using this fact, we conclude that the claim of   Lemma   \ref{L2.2}   is true regarding the operator $W$ i.e.  $W$ is m-accretive.

Using the first representation theorem (Theorem 2.1 \cite[p.322]{firstab_lit:kato1980})  we have a one-to-one correspondence between m-sectorial operators and closed sectorial sesquilinear forms i.e. $W=T_{t}$ by symbol,  where $t$ is a sesquilinear form corresponding to the operator $W.$  Hence $H:=Re\, W$ is defined (see \cite[p.337]{firstab_lit:kato1980}). In accordance with Theorem 2.6 \cite[p.323]{firstab_lit:kato1980} the operator $H$ is selfadjoint, strictly accretive.

 A compact embedding   provided by the relation $\mathfrak{h}[f]\geq C_{2} \|f\|_{\mathfrak{H}_{+}} \geq C_{2}/C_{0}\|f\|_{\mathfrak{H}},\,f\in \mathrm{D}(h)$ proves that $R_{H}$ is compact (see proof of  Theorem \ref{T2.1}) and as a result   of  the application of  Theorem 3.3   \cite[p.337]{firstab_lit:kato1980}, we get $R_{W}$ is compact. Thus the  claim of  Theorem \ref{T2.1}   remains true regarding   the operators $R_{H},\,R_{W}.$

 In accordance with   Theorem 2.5 \cite[p.323]{firstab_lit:kato1980} , we get $ W ^{\ast}=T_{t^{\ast}}$  (since $W^{\ast}=W^{\ast}$).
Now if we denote $t_{1}:=t^{\ast},$  then  it is easy to calculate $\mathfrak{k} =-\mathfrak{k}_{1}.$   Since $t$ is sectorial, than $|\mathfrak{k}_{1}|\leq \tan \theta\cdot \mathfrak{h}.$ Hence, in accordance with Lemma 3.1 \cite[p.336]{firstab_lit:kato1980}, we get  $\mathfrak{k} [u,v]=(BH^{1/2}u,H^{1/2}v),\,\mathfrak{k}_{1}[u,v]=-(BH^{1/2}u,H^{1/2}v),\,u,v\in \mathrm{D}(H^{1/2}),$ where    $B\in  \mathcal{B}(\mathfrak{H}) $ is a symmetric operator. Let us prove that $B$ is selfadjoint.  Note that in accordance with  Lemma 3.1 \cite[p.336]{firstab_lit:kato1980} $\mathrm{D}(B)=\mathrm{R}(H^{1/2}),$    in accordance with   Theorem 2.1 \cite[p.322]{firstab_lit:kato1980}, we have   $(Hf,f)_{\mathfrak{H}}\geq C_{2}/C_{0}\|f\|_{\mathfrak{H}}^{2},\,f\in \mathrm{D}(H),$ using the reasonings of Theorem  \ref{T2.2}, we conclude that   $\mathrm{R}(H^{1/2})=\mathfrak{H}$ i.e. $\mathrm{D}(B)=\mathfrak{H}.$  Hence $B$ is selfadjoint.  Using Lemma 3.2 \cite[p.337]{firstab_lit:kato1980}, we obtain a representation $W=H^{1/2}(I+iB) H^{1/2},\,W^{\ast}=H^{1/2}(I-iB) H^{1/2}.$  Noting the fact   $\mathrm{D}(B)=\mathfrak{H},$  we can easily obtain $ (I\pm iB)^{\ast}= I \mp iB.$ Since $B$ is selfadjoint, then
$\mathrm{Re}([I\pm iB]f,f)_{\mathfrak{H}}=\|f\|^{2}_{\mathfrak{H}}.$ Using this fact and  applying Theorem 3.2 \cite[p.268]{firstab_lit:kato1980}, we conclude that $\mathrm{R}(I\pm iB)$ is a closed set. Since $\mathrm{N}(I\pm iB)=0,$ then $\mathrm{R}(I \mp iB)=\mathfrak{H}$ (see (3.2) \cite[p.267]{firstab_lit:kato1980}). Thus, we obtain   $(I\pm i B)^{-1}\in \mathcal{B}(\mathfrak{H}).$ Taking into account the above facts, we get $R_{W}=H^{-1/2}(I+iB)^{-1} H^{-1/2},\,R_{W^{\ast}}=H^{-1/2}(I-iB)^{-1} H^{-1/2}.$ In accordance with the well-known theorem (see Theorem 5 \cite[p.557]{firstab_lit:Smirnov5}), we have $R^{\ast}_{W }=R_{W^{\!\ast}}.$ Note that the relations   $(I\pm i B)\in \mathcal{B}(\mathfrak{H}),\,(I\pm i B)^{-1}\in \mathcal{B}(\mathfrak{H}),\,H^{-1/2}\in \mathcal{B}(\mathfrak{H})$  allow  as to obtain the following formula by direct calculations
$$
\mathfrak{Re} R _{W }    =\frac{1}{2}H^{-1/2}(I+ B^{2})^{-1} H^{-1/2}.
$$
  This formula is a crucial point of the matter, we can repeat the rest part of the proof of    Theorem  \ref{T2.2}     in terms $H:= Re \, W.$  By virtue of these facts  Theorems  \ref{T2.3}-\ref{T2.5}   can be  reformulated in terms  $H:= Re \, W ,$    since they are based on Lemmas  \ref{L2.1},  \ref{L2.3}, Theorems  \ref{T2.1},  \ref{T2.2}.\\
\end{proof}
 Observe  that the given above classification is far from the exact description of the Schatten-von Neumann class index. However, having analyzed  the above   implications,  we can say that   it makes a prerequisite to establish  a hypotheses $ R_{ W }\in  \mathfrak{S}_{p},\,\inf p=1/\mu. $
The following narrative is devoted to its verification.

 Let  us undergone the technical tools involved in the proof of the statement to   the thorough analysis in order to absorb and contemplate them.
Consider the statement (A),  if    $\mu\leq1,$ then   $\; R_{  W  }\in  \mathfrak{S}_{p},\,\inf p\leq2/\mu.$ The main result on which it is based is in the asymptotic equivalence between the the inverse of the real component and the real component of the resolvent, the latter due to the technical tool makes the result, i.e.
$$
(|R_{ W  }|^{2} f,f)_{\mathfrak{H}}=\|R_{  W }f\|^{2}_{\mathfrak{H}}\leq C \cdot{\rm Re}(R_{ W }f,f)_{\mathfrak{H}}= C\cdot \left( \mathfrak{Re}R_{ W }f,f\right)_{\mathfrak{H}}.
$$

Consider the statement, if  $\lambda_{n}(R_{H})\geq  C \,n^{-\mu},\,0\leq\mu<\infty,$    then  the following implication holds $R_{  W}\in\mathfrak{S}_{p},\,p\in [1,\infty),\Rightarrow \mu>1/p.$ The main results that guaranty the fulfilment of the latter are  inequality  (7.9) \cite[p.123]{firstab_lit:1Gohberg1965}, Theorem 3.5 \cite{kukushkin2021a}, in accordance with which,   we get
$$
\sum\limits_{i=1}^{\infty}|s_{i} (R_{W} )|^{p}\geq\sum\limits_{i=1}^{\infty}| (R_{W}\varphi_{i},\varphi_{i})_{\mathfrak{H}}|^{p}\geq\sum\limits_{i=1}^{\infty}|{\rm Re}(R_{W}\varphi_{i},\varphi_{i})_{\mathfrak{H}}|^{p}=
$$
$$
=\sum\limits_{i=1}^{\infty}| \left(\mathfrak{Re}R_{ W } \varphi_{i},\varphi_{i}\right)_{\mathfrak{H}}|^{p}=\sum\limits_{i=1}^{\infty}
|\lambda_{i}\left(\mathfrak{Re}R_{ W }\right)|^{p}\geq C   \sum\limits_{i=1}^{\infty}  i^{- \mu p },\,p\geq1.
$$

Bellow, we observe  the   statement (B), where the peculiar result related to the asymptotics of the absolute value of the eigenvalue is given. It is based upon the Theorem 6.1 \cite[p.81]{firstab_lit:1Gohberg1965},  in accordance with which, we have
 \begin{equation*}\label{3.31}
\sum\limits_{m=1}^{k}|{\rm Im}\,\lambda_{m}(B)|^{p}\leq \sum\limits_{m=1}^{k}
|\lambda_{m} (\mathfrak{Im} B   )|^{p},\;\left(k=1,\,2,...,\,\nu_{\,\mathfrak{I}}(B)\right),\,1\leq p<\infty,
\end{equation*}
where    $ \nu_{\,\mathfrak{I}}(B) \leq \infty $ is   the sum of   all  algebraic multiplicities corresponding to   the not real  eigenvalues   of the bounded operator $B,\,\mathfrak{Im} B \in \mathfrak{S}_{\infty}$ (see  \cite[p.79]{firstab_lit:1Gohberg1965}).

Note that the statement (B)  will allow us to arrange brackets in the series that converges in the Abel-Lidskii sense  what would be an advantageous achievement in the later constructed theory.  However, it may be interesting if we do not have the  exact index of the Schatten class  for in this case, we obtain the obvious
$$
R_{W}\in \mathfrak{S}_{p},\Rightarrow s_{n}=o(n^{-1/p}),
$$
hence in accordance with the connection of the asymptotics (see Chapter II, \S3 \cite{firstab_lit:1Gohberg1965} ), we get $|\lambda_{n}(R_{W})|=o\left(n^{-1/p}    \right)$ that is the same if we have $p>1/\mu.$ Thus, along the mentioned above implication $R_{  W}\in\mathfrak{S}_{p},\,p\in [1,\infty),\Rightarrow p>1/\mu$ it makes the prerequisite to observe  the hypotheses $\inf p = 1/\mu. $

 Apparently, the used technicalities appeal to the so-called non-direct estimates for singular numbers realized due to the series estimates. As we will see further, the main advantage of the series estimation is the absence of the conditions imposed on the type of the asymptotics, it may be not one of the power type. However, we will show that under the restriction imposed on the type of the asymptotics, assuming that one is of the power type, we can obtain direct estimates for singular numbers. In the reminder, let us note that classes of differential operators have the asymptotics of the power type what make the issue quite relevant.

\subsection{The completion  of the proposition A}

The reasonings produced bellow appeals to a compact operator $B$ what represents a most general case in the framework of the decomposition on the root vectors theory, however to obtain more peculiar results we are compelled to deploy some restricting conditions. In this regard we involve hypotheses H1,H2 if it is necessary.
The result represented bellow gives us  the upper estimate for the singular numbers it is based on the result  by  Ky Fan \cite{firstab_lit:Fan} which  can be found as a corollary of the well-known Allakhverdiyev theorem,  see   Corollary  2.2 \cite{firstab_lit:1Gohberg1965}.

\begin{lem}\label{L2.4} Assume that  $B$ is a compact sectorial operator with the vertex situated at the point zero, then
$$
s_{2m-1}(B)\leq \sqrt{2}\,\sec  \theta \cdot \lambda_{m} (\mathfrak{Re}  B),\;\;s_{2m}(B)\leq \sqrt{2}\,\sec  \theta \cdot  \lambda_{m} (\mathfrak{Re}  B), \;\;m =1,2,...\,  .
$$
\end{lem}
\begin{proof}
 Consider the Hermitian components
$$
\mathfrak{Re} B:=\frac{B+B^{\ast}}{2},\;\;\mathfrak{Im} B:=\frac{B-B^{\ast}}{2i},
$$
it is clear that they are compact selfadjoint operators, since $B$ is compact and due to the technicalities of the given algebraic constructions.
Note that the  following relation can be established by direct calculation
$$
\mathfrak{Re}^{2} \!B+ \mathfrak{Im}^{2} \!B=\frac{B^{\ast}B+BB^{\ast}}{2},
$$
from what follows the inequality
\begin{equation}\label{2.45}
 \frac{1}{2} \cdot B^{\ast}B \leq     \mathfrak{Re}^{2} \!B+  \mathfrak{Im}^{2}B .
\end{equation}
Having analyzed the latter formula, we see that  it is rather reasonable to think over the opportunity of applying the corollary of the  minimax principle pursuing the aim to estimate the singular numbers of the operator $B.$ For the purpose of implementing the latter concept,   consider the following relation
$
\mathfrak{Re}^{2} \!B \,f_{n}=\lambda^{2}_{n} f_{n},
$
where $f_{n},\lambda_{n}$ are the eigenvectors and the eigenvalues of the operator $\mathfrak{Re}   B$ respectively.  Since the operator $\mathfrak{Re}   B$ is selfadjoint and compact then its set of eigenvalues form a basis in $\overline{\mathrm{R}(\mathfrak{Re}   B)}.$   Assume that there exists a non-zero  eigenvalue of the operator $ \mathfrak{Re}^{2} \!B $ that is different from  $\{\lambda^{2}_{n}\}_{1}^{\infty},$
then, in accordance with the well-known fact of the operator theory,  the corresponding eigenvector is orthogonal to the eigenvectors of the operator $\mathfrak{Re}   B.$  Taking into account the fact that the latter form a basis in $\overline{\mathrm{R}(\mathfrak{Re}   B)},$ we come to the conclusion that the eigenvector does not belong to  $\overline{\mathrm{R}(\mathfrak{Re}   B)}.$   Thus, the obtained contradiction proves the fact
$
\lambda_{n}(\mathfrak{Re}^{2} \!B)=\lambda^{2}_{n}(\mathfrak{Re}   B).
$
Implementing the same reasonings, we obtain
$
\lambda_{n}(\mathfrak{Im}^{2} \!B)=\lambda^{2}_{n}(\mathfrak{Im}   B).
$

Further, we   need a result by Ky Fan \cite{firstab_lit:Fan}  see   Corollary  2.2 \cite{firstab_lit:1Gohberg1965} (Chapter II, $\S$ 2.3), in accordance with which, we have
$$
s_{m+n-1}(\mathfrak{Re}^{2} \!B+ \mathfrak{Im}^{2}B)\leq \lambda_{m}(\mathfrak{Re}^{2} \!B)+\lambda_{n}(\mathfrak{Im}^{2}B),\;\;m,n=1,2,...\,.
$$
Choosing $n=m$ and $n=m+1,$  we obtain respectively
$$
s_{2m-1}(\mathfrak{Re}^{2} \!B+ \mathfrak{Im}^{2}B)\leq \lambda_{m}(\mathfrak{Re}^{2} \!B)+\lambda_{m}(\mathfrak{Im}^{2}B),
$$
$$
\, s_{2m}(\mathfrak{Re}^{2} \!B+ \mathfrak{Im}^{2}B)\leq \lambda_{m}(\mathfrak{Re}^{2} \!B)+\lambda_{m+1}(\mathfrak{Im}^{2}B) \;\;m =1,2,...\,.
$$
At this stage of the reasonings we need involve  the sectorial property $\Theta(B)\subset \mathfrak{L}_{0}(\theta)$ which  gives us $|\mathrm{Im}( Bf,f)|\leq \tan \theta \,\mathrm{Re}(   Bf,f).$  Applying the corollary of the minimax principle to the latter relation, we   get $|\lambda_{n}(\mathfrak{Im}   B)|\leq \tan \theta \, \lambda_{n}(\mathfrak{Re}   B).$
Therefore
$$
s_{2m-1}(\mathfrak{Re}^{2} \!B+ \mathfrak{Im}^{2}B)\leq \lambda_{m}(\mathfrak{Re}^{2} \!B)+\lambda_{m}(\mathfrak{Im}^{2}B)\leq\sec^{2}\!\theta \cdot \lambda_{m}^{2}(\mathfrak{Re}  B),
$$
$$
\, s_{2m}(\mathfrak{Re}^{2} \!B+ \mathfrak{Im}^{2}B)\leq \sec^{2}\!\theta  \cdot\lambda_{m}^{2}(\mathfrak{Re}  B) \;\;m =1,2,...\,.
$$
Applying the minimax principle to the formula \eqref{2.45}, we get
$$
s_{2m-1}(B)\leq \sqrt{2}\sec  \theta \cdot  \lambda_{m} (\mathfrak{Re}  B),\;\;s_{2m}(B)\leq \sqrt{2}\sec  \theta \cdot  \lambda_{m} (\mathfrak{Re}  B), \;\;m =1,2,...\,  .
$$
This   gives us the upper estimate for the singular values of the operator $B.$
\end{proof}

However, to      obtain the lower estimate we need involve Lemma 3.1 \cite[p.336]{firstab_lit:kato1980}, Theorem 3.2 \cite[p.337]{firstab_lit:kato1980}.
 Consider an unbounded   operator $T,\, \Theta(T)\subset \mathfrak{L}_{0}(\theta),$  in accordance with the first representation theorem   \cite[p. 322]{firstab_lit:kato1980}, we can consider its Friedrichs extension the m-sectorial operator $W,$ in its own turn   due to the results \cite[p.337]{firstab_lit:kato1980}, it has a real part $H$ which coincides with the Hermitian  real component if we deal with a bounded operator. Note that by virtue of the sectorial property the operator $H$ is non-negotive. Further, we consider the case $\mathrm{N}(H)=0$ it follows that $\mathrm{N}(H^{\frac{1}{2}})  =0.$ To prove this fact we should note that $\mathrm{def}H=0,$ considering inner product with the element belonging to $\mathrm{N}(H^{\frac{1}{2}})$ we obtain easily the fact  that it must equal to zero.
Having analyzed the proof of Theorem 3.2 \cite[p.337]{firstab_lit:kato1980}, we see that its statement remains true in the modified  form even in the case if we lift the m-accretive condition, thus under the  sectorial condition imposed upon the closed densely defined  operator $T,$ we get the following inclusion
$$
T\subset H^{1/2}(I+iG)H^{1/2},
$$
here   symbol $G$ denotes    a bounded selfadjoint operator in $\mathfrak{H}.$ However, to obtain the asymptotic formula established in Theorem 5 \cite{firstab_lit(arXiv non-self)kukushkin2018} we cannot be satisfied by the made assumptions but require   the existence of the resolvent at the point zero and its compactness. In spite of the fact that   we can  proceed our narrative under the weakened conditions regarding the operator $W$ in comparison with H1,H2, we can   claim that the statement of  Theorem 5 \cite{firstab_lit(arXiv non-self)kukushkin2018} remains true under the assumptions made above, we prefer to deploy H1,H2 what guarantees the conditions we need and at the same time provides a description of the matter under the natural point of view.

\begin{lem}\label{L2} Assume that the conditions  H1,H2  hold   for the operator $W,$ moreover
$$
 \|\mathfrak{Im}W  /\mathfrak{Re} W\|_{2}< 1,
$$
then $$
  \lambda^{-1}_{2n} \left(\mathfrak{Re}W \right)\leq C s_{n}(R_{W}),\;n\in \mathbb{N}.
$$
\end{lem}
\begin{proof}
Firstly, let us show that $\mathrm{D}(W^{2})$ is a dense set in $\mathfrak{H}_{+}.$ Since the operator $W$ is closed and strictly accretive, we have $\mathrm{R}(W)=\mathfrak{H},$ hence there exists the preimage of the set $\mathfrak{M},$ let us denote it by $\mathfrak{M}'.$ Consider an arbitrary element $x_{0}\in \mathfrak{H}$ and denote its preimage by $x'_{0},$ we  have
$$
\|W(x'_{0}-x'_{n})\|_{\mathfrak{H}}\geq C\|x'_{0}-x'_{n}\|_{\mathfrak{H}_{+}},
$$
where $\{x_{n}\}_{1}^{\infty}\subset \mathfrak{M}'.$ Hence, the set $\mathfrak{M}'$ is dense in $\mathrm{D}(W)$  in  the sense of the norm $\mathfrak{H}_{+},$ hence it is dense in $\mathfrak{H}_{+}$ and consequently the set  $\mathrm{D}(W^{2})$   is dense  in $\mathfrak{H}_{+}$ since $\mathfrak{M}'\subset \mathrm{D}(W^{2}).$ Here, we should note that we have proved the  fulfilment of the condition H1 for the operator $W^{2}$ with respect to the same pair of Hilbert spaces.

Note that    under the   assumptions H1,H2, using the reasonings of Theorem 3.2 \cite{firstab_lit:kato1980}, we have the following representation
$$
W=H^{1/2}(I+iG)H^{1/2},\;W^{\ast}=H^{1/2}(I-iG)H^{1/2}.
$$
It follows easily from this formula that the Hermitian components of the operator $W$ are defined, we have
$
\mathfrak{Re}W =H,\;\;\mathfrak{Im}W=H^{1/2}GH^{1/2}.
$
Using the decomposition
$
W=\mathfrak{Re}W +i\mathfrak{Im}W,\;W^{\ast}=\mathfrak{Re}W -i\mathfrak{Im}W,
$
we get easily
$$
 \left(  \frac{W^{2}+W^{\ast\,2}}{2}\, f,f\right)_{\!\!\mathfrak{H}}=\left\|\mathfrak{Re}W f\right\|^{2}_{\mathfrak{H}}-\left\|\mathfrak{Im}W f\right\|^{2}_{\mathfrak{H}};\,
$$
$$
  \left(  \frac{W^{2}-W^{\ast\,2}}{2i}\, f,f\right)_{\!\!\mathfrak{H}} =  (\mathfrak{Im}W \,\mathfrak{Re}Wf,f)_{\mathfrak{H}} +(\mathfrak{Re}W\, \mathfrak{Im}Wf,f)_{\mathfrak{H}},\;f\in \mathrm{D}(W^{2}).
$$
Using simple reasonings,  we can rewrite the above formulas in terms of Theorem 3.2 \cite{firstab_lit:kato1980}, we have
\begin{equation}\label{2.46}
\mathrm{Re}(W^{2}f,f)_{\mathfrak{H}}=\|Hf\|^{2}_{\mathfrak{H}}-\|H^{1/2}GH^{1/2}f\|^{2}_{\mathfrak{H}},\;\mathrm{Im} (W^{2}f,f)_{\mathfrak{H}}=\mathrm{Re}(H^{1/2}GH^{1/2}f,H f)_{\mathfrak{H}},
$$
$$
\,f\in \mathrm{D}(W^{2}).
\end{equation}
Consider a set of eigenvalues $\{\lambda_{n}\}_{1}^{\infty}$ and a complete system of orthonormal vectors $\{e_{n}\}_{1}^{\infty}$ of the operator $H,$  the conditions H1,H2 guarantee existing of the latter since  $R_{H}$ is compact (see Theorem 3 \cite{kukushkin2021a}), using the matrix form of the operator $G,$     we have
$$
\|Hf\|^{2}_{\mathfrak{H}}=\sum\limits_{n=1}^{\infty}  |\lambda _{n}|^{2}|f_{n}|^{2},\;\|H^{1/2}GH^{1/2}f\|^{2}_{\mathfrak{H}}=\sum\limits_{n=1}^{\infty} \lambda_{n} \left| \sum\limits_{k=1}^{\infty} b_{nk}\sqrt{\lambda_{k}}f_{k}\right|^{2},
$$
$$
\mathrm{Re}(H^{1/2}GH^{1/2}f,H f)_{\mathfrak{H}}=\mathrm{Re}\left(\sum\limits_{n=1}^{\infty} \lambda^{3/2}_{n}f_{n}  \sum\limits_{k=1}^{\infty} b_{nk}\sqrt{\lambda_{k}}\bar{f_{k}}\right),
$$
where $b_{nk}$  are the matrix coefficients of the operator $G.$
Applying the Cauchy-Swarcz inequality, we get
$$
 \|H^{1/2}GH^{1/2}f\|^{2}_{\mathfrak{H}}\leq\sum\limits_{n=1}^{\infty} \lambda_{n} \left|\sum\limits_{k=1}^{\infty} |\lambda_{k} f_{k}|^{2} \sum\limits_{k=1}^{\infty}
  |b_{ nk}|^{2} /\lambda_{k}\right|\leq \|Hf\|^{2}_{\mathfrak{H}} \sum\limits_{  n,k=1}^{\infty}|b_{ nk}|^{2}\lambda_{n}/\lambda_{k} ;
$$
$$
 |\mathrm{Re}(H^{1/2}GH^{1/2}f,H f)_{\mathfrak{H}}|\leq \|Hf\|_{\mathfrak{H}} \left(\sum\limits_{n=1}^{\infty}   \left|\sum\limits_{k=1}^{\infty}  \bar{b}_{ nk}   \sqrt{\lambda_{n}\lambda_{k}}f_{k}\right|^{2}
      \right)^{1/2} \leq\|Hf\|^{2}_{\mathfrak{H}} \left(\sum\limits_{ n,k=1}^{\infty}    |b_{nk}|^{2}\lambda_{n}/\lambda_{k} \right)^{1/2} .
$$
In accordance with the definition of the sectorial property, we require
$$
|\mathrm{Im}(W^{2}f,f)_{\mathfrak{H}}|\leq \tan\theta \cdot\mathrm{Re}(W^{2}f,f)_{\mathfrak{H}},\; 0<\theta<\pi/2.
$$
Therefore, the sufficient conditions of the  sectorial property can be expressed as follows
$$
   \|Hf\|^{2}_{\mathfrak{H}}\left(\sum\limits_{ n,k=1}^{\infty}    |b_{nk}|^{2}/\lambda_{k} \right)^{1/2}\!\!\!  \leq   \|Hf\|^{2}_{\mathfrak{H}} \left(1-\sum\limits_{  n,k=1}^{\infty}|b_{ nk}|^{2}\lambda_{n}/\lambda_{k}\right)\tan\theta   ;
 $$
 $$
 \sum\limits_{  n,k=1}^{\infty}|b_{nk}|^{2}\lambda_{n}/\lambda_{k}  +  \cot \theta\left(\sum\limits_{ n,k=1}^{\infty}   |b_{nk}|^{2}\lambda_{n}/\lambda_{k} \right)^{1/2}\leq 1,
$$
where $\theta$ is the  semi-angle of the supposed sector.
Solving the corresponding quadratic equation, we obtain the desired estimate
\begin{equation}\label{2.47}
 \left(\sum\limits_{ n,k=1}^{\infty} |b_{nk}|^{2}\lambda_{n}/\lambda_{k}\right)^{1/2}<\frac{1}{2}\left\{\sqrt{  \cot^{2} \theta  +4}- \cot \theta \right\} .
\end{equation}
 Having noticed the fact that    the right hand side of \eqref{2.47} tends to one from below when $\theta$ tends to $\pi/2,$ we obtain the condition of the sectorial property expressed in terms of  the absolute norm
 \begin{equation}\label{2.48}
  \|H^{1/2}GH^{-1/2}\| _{2}:=\left(\sum\limits_{ n,k=1}^{\infty} |b_{nk}|^{2}\lambda_{n}/\lambda_{k}\right)^{1/2}<1,
 \end{equation}
  in this case, we we can choose the semi-angle of the sector using the following relation
  $$
   \tan\theta = \frac{N}{1-N^{2}}+\varepsilon,\;N:=\|H^{1/2}GH^{-1/2}\| _{2},
  $$
  where $\varepsilon$ is an arbitrary small positive number. Thus, we can resume that in the value of the absolute norm less than one than the operator $W^{2}$ is sectorial and the value of the absolute norm defines the semi-angle.
  Note that coefficients  $  b_{nk} \sqrt{\lambda_{n}/\lambda_{k}},\;\overline{b_{kn}} \sqrt{\lambda_{n}/\lambda_{k}} $ correspond to the matrices of the operators respectively
$$
H^{1/2}GH^{-1/2}f=\sum\limits_{n=1}^{\infty}  \lambda^{1/2}_{n} e_{n}  \sum\limits_{k=1}^{\infty} b_{nk}\lambda^{-1/2}_{k} f_{k},\;H^{-1/2}GH^{1/2}f=\sum\limits_{n=1}^{\infty}  \lambda^{-1/2}_{n} e_{n}  \sum\limits_{k=1}^{\infty} b_{nk}\lambda^{1/2}_{k} f_{k}   .
$$
Thus, if the absolute operator norm exists, i.e.
$$
\|H^{1/2}GH^{-1/2}\| _{2} <\infty,
$$
then both of them belong to the so-called Hilbert-Schmidt class simultaneously, but it is clear without involving the   absolute norm  since the above operators are adjoint.   It is remarkable that,   we can write formally the obtained estimate in terms of the Hermitian components of the operator, i.e.
$$
\|\mathfrak{Im}W /\mathfrak{Re}W\|_{2}< 1.
$$
Below, for a convenient form of writing, we will use a short-hand notation $A:=R_{ W },$ where it is necessary. The next step is to establish the asymptotic formula
\begin{equation}\label{2.49}
\lambda_{n}\left( \frac{A^{2}+A^{ 2\ast}}{2}\right)\asymp \lambda^{-1}_{n} \left(\mathfrak{Re}W^{2} \right),\;n\rightarrow\infty.
\end{equation}
However, we cannot apply directly  Theorem 5 \cite{firstab_lit(arXiv non-self)kukushkin2018} to the operator $W^{2},$  thus we   are compelled to modify the proof having taken into account weaker conditions and the additional condition \eqref{2.48}.

 Let us observe that the compactness of  the operator $R_{W}(\lambda),\,\lambda \in \mathrm{P}(W)$ gives us the compactness of the operator $W^{-2}.$ Since the latter is sectorial, it follows easily that  $R_{W^{2}}(\lambda),\,\lambda \in \mathrm{P}(W^{2})$ is compact, since the outside of the sector belongs to the resolvent set and the resolvent compact at least at one point is compact everywhere on the resolvent set.  Note that due to the reasonings given above the following relation holds
\begin{equation}\label{2.50}
\mathrm{Re}(W^{2}f,f)_{\mathfrak{H}} \geq C \|Hf\|^{2}_{\mathfrak{H}} \geq C\|f\|^{2}_{\mathfrak{H}_{+}},\,f\in \mathrm{D}(W^{2}),
\end{equation}
the latter inequality can be obtained easily (see (28) \cite{firstab_lit(arXiv non-self)kukushkin2018}). Thus, we obtain the fact that the operator $W^{2}$ is sectorial, strictly accretive operator, hence falls in the scope of the first representation theorem in accordance with which there exists one to one correspondence between the closed densely defined sectorial forms and m-sectorial operators. Using this fact, we can claim that the real part $H_{1}:=\mathrm{Re} W^{2}$ is defined and the following relations hold in accordance with the representation theorem i.e., Theorem 3.2 \cite[p.337]{firstab_lit:kato1980}, we get
$$
W^{2}=H_{1}^{1/2}(I+iG_{1})H_{1}^{1/2},\;W^{ 2\ast}=H_{1}^{1/2}(I+iG_{2})H_{1}^{1/2},
$$
where $G_{1},G_{2}$ are selfadjoint bounded operators. Now by direct calculation, we can verify   that $H_{1}=\mathfrak{Re}W^{2},$ we should also note that $\mathrm{D}(W^{2})$ is a core of  the corresponding closed densely defined sectorial form $\mathfrak{h}$ put in correspondence to the operator $H_{1}$ by the first representation theorem, i.e. $\mathrm{D}_{0}(\mathfrak{h})=\mathrm{D}(W^{2}).$
Let us show that $G_{1}=-G_{2}.$ We have
\begin{equation*}
  H_{1}f\! =\!\frac{1}{2}\left[H_{1}^{\frac{1}{2}}(I+i G_{1})  +H_{1}^{\frac{1}{2}}(I+i G_{2})\right]H_{1}^{\frac{1}{2}}\!  =
  $$
  $$
  = \! H_{1} f +
 \frac{i}{2} H_{1}^{\frac{1}{2}}\left(G_{1}+G_{2}\right)  H_{1}^{\frac{1}{2}}f  ,\;f\in \mathfrak{M}'.
\end{equation*}
By virtue of  inequality \eqref{2.50}, we see that the operator $H_{1}$ is strictly accretive, therefore
  $\mathrm{N}(H_{1})=0 ;\;( G_{1}+G_{2})H_{1}^{1/2}=0.$ Since
   $$
   \mathfrak{H}=\overline{\mathrm{R}(H_{1}^{1/2})}\oplus \mathrm{N}(H_{1}^{1/2}),
   $$
   then $G_{1}=G_{2}=:G'.$ Applying the reasonings represented in Theorem 5 \cite{firstab_lit(arXiv non-self)kukushkin2018}, we obtain the fact that    $H_{1}^{-1/2}$ is a bounded operator  defined on $\mathfrak{H}.$
   Using the  properties  of the operator $G',$    we get
$\|(I+ iG')f\|_{\mathfrak{H}} \cdot\|f\|_{\mathfrak{H}} \geq\mathrm{Re }\left([I+ iG']f,f\right)_{\mathfrak{H}}  =\|f\|^{2}_{\mathfrak{H}} ,\,f\in \mathfrak{H}.$ Hence
$
\|(I+ iG')f\|_{\mathfrak{H}}  \geq \|f\|_{\mathfrak{H}} ,\,f\in \mathfrak{H}.
$
 It implies that the operators  $I+ iG'$ are invertible. The reasonings   corresponding to the operator $I-iG'$ are absolutely analogous.
   Therefore
  \begin{equation}\label{2.51}
A^{2}=H_{1}^{-\frac{1 }{2}}(I+iG'  )^{-1} H_{1}^{- \frac{1}{2}},\;A^{2\ast}=H_{1}^{-\frac{1 }{2}}(I-iG'  )^{-1} H_{1}^{- \frac{1}{2}}.
\end{equation}
Using simple calculation based upon the operator properties established above, we get
\begin{equation}\label{2.52}
\mathfrak{Re} A^{2} =\frac{1}{2}\,H_{1}^{-\frac{1 }{2}}  (I+G'^{2} )^{-1}  H_{1}^{- \frac{1}{2}}.
\end{equation}
Therefore
 $$
\left(\mathfrak{Re} A^{2}  f,f\right)_{\mathfrak{H}} =\left(H_{1}^{-\frac{1 }{2}} (I+G'^{2} )^{-1}     H_{1}^{- \frac{1}{2}}   f,f\right)_{\!\!\mathfrak{H}}  \leq
  \|(I+G'^{2} )^{-1}      \| \cdot\left(R_{H_{1}}  f,f\right)_{\mathfrak{H}} ,\;f\in \mathfrak{H}.
$$
On the other hand,  it is easy to see that  $ ((I+G'^{2} )^{-1}f,f)_{\mathfrak{H}} \geq \|(I+G'^{2} )^{-1}f\|^{2}_{\mathfrak{H}} .$ At the same time   it is obvious that   $ I+G'^{2} $ is bounded and we have   $\|(I+G'^{2} )^{-1}f\|_{\mathfrak{H}} \geq \| I+G'^{2}   \|^{-1} \|f\|_{\mathfrak{H}} .$ Using   these estimates, we have
$$
\left(\mathfrak{Re} A^{2}  f,f\right)_{\mathfrak{H}} =\left( (I+G'^{2} )^{-1}     H_{1}^{- \frac{1}{2}}   f,H_{1}^{-\frac{1 }{2}}f\right)_{\mathfrak{H}} \geq
\|(I+G'^{2} )^{-1}     H_{1}^{- \frac{1}{2}}   f \|^{2}_{\mathfrak{H}} \geq
$$
$$
\geq  \|  I+G'^{2}    \|^{-2} \cdot   \left(R_{H_{1} }  f,f\right)_{\mathfrak{H}},\;f\in \mathfrak{H}.
$$
Using relation \eqref{2.50}, we obtain the fact that the resolvent  $R_{H_{1}}$ is compact, the fact that $\mathfrak{Re} A^{2}$ is compact is obvious.
 Thus, analogously to the reasonings of  Theorems 5 \cite{firstab_lit(arXiv non-self)kukushkin2018}  applying the minimax principle we obtain the desired asymptotic formula \eqref{2.49}.
 Further, we will use  the following formula obtained due to the  positiveness of the  squared Hermitian imaginary  component of the operator $A,$ we have
$$
\frac{A^{2}+A^{  2\ast}}{2}=\frac{A^{2}+A^{\ast 2}}{2}\leq A^{\ast}A+AA^{\ast}.
$$
Applying the   corollary of the well-known Allakhverdiyev theorem   (Ky Fan  \cite{firstab_lit:Fan}),   see   Corollary  2.2 \cite{firstab_lit:1Gohberg1965} (Chapter II, $\S$ 2.3),   we have
$$
\lambda_{2n}\left(  A^{\ast}A+AA^{\ast}\right)\leq \lambda_{n}(A^{\ast}A)+\lambda_{n}(AA^{\ast}),\; n\in \mathbb{N}.
$$
Taking into account the fact $s_{n}(A)=s_{n}(A^{\ast}),$ using the minimax principle, we obtain the estimate
$$
s^{2}_{n}(A)\geq C \lambda_{2n}\left( \frac{A+A^{  2\ast}}{2}\right),\;n\in \mathbb{N},
$$
applying \eqref{2.49}, we obtain
$$
s^{2}_{n}(A)\geq C \lambda^{-1}_{2n} \left(\mathfrak{Re}W^{2} \right),\;n\in \mathbb{N}.
$$
Here, it is rather reasonable to apply formula \eqref{2.46} which  gives us
$$
\|f\|^{2}_{\mathfrak{H}}\leq \|f\|^{2}_{\mathfrak{H}_{+}} \leq\left(\mathfrak{Re}W^{2}f,f\right)_{\mathfrak{H}}\leq\left(H f, Hf\right)_{\mathfrak{H}} ,\;f\in \mathrm{D}(W^{2}),
$$
what in its own turn, collaboratively with the minimax principle leads  us to the theorem statement.
\end{proof}
\begin{remark}\label{R2.1} It is remarkable that the central point of  the above proof is the representation theorems, in accordance with the first one we have a plain  construction of the operator real part equaling the Hermitian real component. These allow  us to implement the simplified scheme of reasonings represented in  \cite{firstab_lit(arXiv non-self)kukushkin2018}.
\end{remark} In application to a rather  wide operator class  including the operators  having the   asymptotics of the resolvent  singular numbers   or one of the real component  eigenvalues  of the power type, i.e. $C_{1}n^{\mu}\leq \lambda_{n}\leq C_{2}n^{\mu},$
the results given above can  be reformulated in the following stylistically convenient form.
\begin{teo}\label{T2.7}
Assume that the hypotheses H1,H2 hold  for the operator $W,$      moreover  $$
\|\mathfrak{Im}W /\mathfrak{Re}W\|_{2}< 1,
$$
then
$$
s_{n}(R_{W})\asymp   \lambda^{-1}_{n} \left(\mathfrak{Re} W \right).
$$
\end{teo}
\begin{proof}
  Since conditions  H1,H2  hold then the resolvent $R_{W}$ is a compact sectorial  operator with the vertex situated  at the point zero (see Theorem 3 \cite{kukushkin2021a}).  The estimates from the above and below  for the singular numbers follow  from the application of Lemmas \ref{L2.4},\ref{L2} respectively, here we should take into account the fact $(Cn)^{\gamma}\asymp n^{\gamma},\,\gamma\in \mathbb{R}$ and the fact $\lambda_{n}(\mathfrak{Re}R_{W})\asymp \lambda^{-1}_{n}(\mathfrak{Re}W)$ that is the claim of Theorem 5 \cite{firstab_lit(arXiv non-self)kukushkin2018}.
\end{proof}

\subsection{The low bound for the Schatten index  of the  perturbed   differential  operators}

\noindent{\bf 1.} Trying to show an application of Lemma \ref{L2.4}, we produce  an example  of a non-selfadjoin operator that is not completely subordinated in the sense of forms (see \cite{firstab_lit:Shkalikov A.}, \cite{firstab_lit(arXiv non-self)kukushkin2018}). The pointed out fact means that, we cannot deal with the operator applying methods \cite{firstab_lit:Shkalikov A.} for they do not work.
   Consider a  differential operator acting in the complex Sobolev  space
$$
\mathcal{L}f := (c_{k}f^{(k)})^{(k)} + (c_{k-1}f^{(k-1)})^{(k-1)}+...+  c_{0}f,
$$
$$
\mathrm{D}(\mathcal{L}) = H^{2k}(I)\cap H_{0}^{k}(I),\,k\in \mathbb{N},
$$
where    $I: = (a, b) \subset \mathbb{R},$ the   complex-valued coefficients
$c_{j}(x)\in C^{(j)}(\bar{I})$ satisfy the condition $  {\rm sign} (\mathrm{Re} c_{j}) = (-1)^{j} ,\, j = 1, 2, ..., k.$
 Consider  a linear combination of   the  Riemann-Liouville  fractional differential   operators
  (see \cite[p.44]{firstab_lit:samko1987})   with the  constant  real-valued  coefficients
$$
\mathcal{D}f:=p_{n}D_{a+}^{\alpha_{n}}+q_{n}D_{b-}^{\beta_{n}}+p_{n-1}D_{a+}^{\alpha_{n-1}}+q_{n-1}D_{b-}^{\beta_{n-1}}+...+
p_{0}D_{a+}^{\alpha_{0}}+q_{0}D_{b-}^{\beta_{0}},
$$
$$
\mathrm{D}(\mathcal{D}) = H^{2k}(I)\cap H_{0}^{k}(I),\,n\in \mathbb{N},
$$
where $\alpha_{j},\beta_{j}\geq 0,\,0 \leq [\alpha_{j}],[\beta_{j}] < k,\, j = 0, 1, ..., n.,\;$
\begin{equation*}
 q_{j}\geq0,\;{\rm sign}\,p_{j}= \left\{ \begin{aligned}
  (-1)^{\frac{[\alpha_{j}]+1}{2}},\,[\alpha_{j}]=2m-1,\,m\in \mathbb{N},\\
\!\!\!\!\!\!\! \!\!\!\!(-1)^{\frac{[\alpha_{j}]}{2}},\;\,[\alpha_{j}]=2m,\,\,m\in \mathbb{N}_{0}   .\\
\end{aligned}
\right.
\end{equation*}
The following result is represented in the paper  \cite{firstab_lit(arXiv non-self)kukushkin2018},    consider the operator
$$
G=\mathcal{L}+\mathcal{D},
$$
$$
\mathrm{D}(G)=H^{2k}(I)\cap H_{0}^{k}(I).
$$
It is  clear that it is an operator with a compact resolvent, however for the accuracy  we will prove this fact, moreover we will produce a pair of Hilbert spaces  so that conditions H1,H2 holds. It follows that the resolvent is compact, thus we can observe the problem of calculating Schatten index. Apparently, it may happen that the direct calculation  of the  singular numbers or  their estimation is rather complicated   since we have the following construction
$$
GG^{\ast}\supset  (\mathcal{L}+\mathcal{D})(\mathcal{L}^{\ast}+\mathcal{D}^{\ast})\supset  \mathcal{L}\mathcal{L}^{\ast}+ \mathcal{D}\mathcal{L}^{\ast}+\mathcal{L}\mathcal{D}^{\ast}+\mathcal{D}\mathcal{D}^{\ast}
$$
where inclusions must satisfy some conditions connected with the core of the operator form for in other case we have a risk to lose some singular numbers. In spite of the fact that    the shown difficulties in many cases  can be  eliminated the offered method of singular numbers estimation becomes apparently relevant.

Let us prove the fulfilment of the conditions H1,H2 under the assumptions  $\mathfrak{H} := L_{2}(I),\, \mathfrak{H}^{+} := H_{0}^{k}(I),\,\mathfrak{M}:=C_{0}^{\infty}(I).$
The fulfillment of the condition H1 is obvious, let us show the fulfilment of the condition  H2. It is easy
to see that
$$
\mathrm{Re}(\mathcal{L}f,f)_{L_{2}(I)}\geq\sum\limits_{j=0}^{k}|\mathrm{Re} c_{j}|\,\|f^{(j)}\|^{2}_{L_{2}(I)}\geq C \|f^{(j)}\|^{2}_{H_{0}^{k}(I)},\;f\in \mathrm{D}(\mathcal{L}).
$$
On the other hand
$$
|(\mathcal{L}f,f)_{L_{2}(I)}|=\left|\sum\limits_{j=0}^{k}(-1)^{j}(c_{j}f^{(j)},g^{(j)} )_{L_{2}(I)}\right|\leq
\sum\limits_{j=0}^{k}\left|(c_{j}f^{(j)},g^{(j)} )_{L_{2}(I)}\right|\leq
$$
$$
\leq C \sum\limits_{j=0}^{k} \|f^{(j)}\| _{L_{2}(I)}\|g^{(j)}\| _{L_{2}(I)}\leq
\|f \| _{H^{k}_{0}(I)}\|g \| _{H^{k}_{0}(I)},\;f\in \mathrm{D}(\mathcal{L}).
$$
Consider  fractional differential   Riemann-Liouville   operators  of   arbitrary non-negative
order $\alpha$ (see \cite[p.44]{firstab_lit:samko1987})  defined by the expressions
\begin{equation*}
 D_{a+}^{\alpha}f=\left(\frac{d}{dx}\right)^{[\alpha]+1}\!\!\!\!I_{a+}^{1-\{\alpha\}}f;\;
 D_{b-}^{\alpha}f=\left(-\frac{d}{dx}\right)^{[\alpha]+1}\!\!\!\!I_{b-}^{1-\{\alpha\}}f,
\end{equation*}
where the fractional integrals of      arbitrary positive order  $\alpha$ defined by
$$
\left(I_{a+}^{\alpha}f\right)\!(x)=\frac{1}{\Gamma(\alpha)}\int\limits_{a}^{x}\frac{f(t)}{(x-t)^{1-\alpha}}dt,
 \left(I_{b-}^{\alpha}f\right)\!(x)=\frac{1}{\Gamma(\alpha)}\int\limits_{x}^{b}\frac{f(t)}{(t-x)^{1-\alpha}}dt
, f\in L_{1}(I).
$$
Suppose  $0<\alpha<1,\, f\in AC^{l+1}(\bar{I}),\,f^{(j)}(a)=f^{(j)}(b)=0,\,j=0,1,...,l;$ then the next formulas follows
from   Theorem 2.2 \cite[p.46]{firstab_lit:samko1987}
\begin{equation}\label{2.53}
 D_{a+}^{\alpha+l}f= I_{a+}^{1- \alpha }f^{(l+1)},\;
 D_{b-}^{\alpha+l}f= (-1)^{l+1}I_{b-}^{1- \alpha }f^{(l+1)}.
\end{equation}
  Further, we need  the following inequalities    (see  \cite{firstab_lit:1kukushkin2018})
\begin{equation}\label{2.54}
\mathrm{Re} (D_{a+}^{\alpha}f,f)_{L_{2}(I)}\geq C\|f\|^{2}_{L_{2}(I)},\,f\in I_{a+}^{\alpha}(L_{2}),\;
$$
$$
\mathrm{Re} (D_{b-}^{\alpha}f,f)_{L_{2}(I)}\geq C\|f\|^{2}_{L_{2}(I)},\,f\in I_{b-}^{\alpha}(L_{2}),
\end{equation}
where $I_{a+}^{\alpha}(L_{2}),I_{b-}^{\alpha}(L_{2})$ are the  classes of  the  functions representable by the fractional integrals (see\cite{firstab_lit:samko1987}).
  Consider the following operator  with the  constant  real-valued  coefficients
$$
\mathcal{D}f:=p_{n}D_{a+}^{\alpha_{n}}+q_{n}D_{b-}^{\beta_{n}}+p_{n-1}D_{a+}^{\alpha_{n-1}}+q_{n-1}D_{b-}^{\beta_{n-1}}+...+
p_{0}D_{a+}^{\alpha_{0}}+q_{0}D_{b-}^{\beta_{0}},
$$
$$
\mathrm{D}(\mathcal{D}) = H^{2k}(I)\cap H_{0}^{k}(I),\,n\in \mathbb{N},
$$
where $\alpha_{j},\beta_{j}\geq 0,\,0 \leq [\alpha_{j}],[\beta_{j}] < k,\, j = 0, 1, ..., n.,\;$
\begin{equation*}
 q_{j}\geq0,\;{\rm sign}\,p_{j}= \left\{ \begin{aligned}
  (-1)^{\frac{[\alpha_{j}]+1}{2}},\,[\alpha_{j}]=2m-1,\,m\in \mathbb{N},\\
\!\!\!\!\!\!\! \!\!\!\!(-1)^{\frac{[\alpha_{j}]}{2}},\;\,[\alpha_{j}]=2m,\,\,m\in \mathbb{N}_{0}   .\\
\end{aligned}
\right.
\end{equation*}
Using \eqref{2.53},\eqref{2.54},  we get
$$
(p_{j}D_{a+}^{\alpha_{j}}f,\!f)_{L_{2}(I)}\!=
\!p_{j}\!\left(\!\!\left(\frac{d}{dx}\right)^{\!\!\!m}\!\!D_{a+}^{m-1+\{\alpha_{j}\}}\!\!f,\!f   \!\right)_{\!\!L_{2}(I)}\!\!\!\!\!\! =
(\!-1)^{m}p_{j}\!\left(\! I_{a+}^{1-\{\alpha_{j}\}}\!\!f^{(m)}\!\!,\!f^{(m)}   \!\right)_{\!\!L_{2}(I)}\!\! \geq
$$
$$
\geq C\left\|I_{a+}^{1-\{\alpha_{j}\}}f^{(m)}\right\|^{2}_{L_{2}(I)}=
C\left\|D_{a+}^{\{\alpha_{j}\}}f^{(m-1)}\right\|^{2}_{L_{2}(I)}\geq C \left\| f^{(m-1)}\right\|^{2}_{L_{2}(I)},
$$
 where  $f \in \mathrm{D}(\mathcal{D})$ is    a real-valued function and   $ [\alpha_{j}]=2m-1,\,m\in \mathbb{N}.$
  Similarly,  we obtain for orders  $ [\alpha_{j}]=2m,\,m\in \mathbb{N}_{0}$
$$
(p_{j}D_{a+}^{\alpha_{j}}f,f)_{L_{2}(I)}=p_{j}\left( D_{a+}^{2m +\{\alpha_{j}\}}f,f   \right)_{L_{2}(I)}=(-1)^{m}p_{j}\left( D_{a+}^{m+\{\alpha_{j}\}}f ,f^{(m)}   \right)_{L_{2}(I)}=
$$
$$
=(-1)^{m}p_{j}\left( D_{a+}^{ \{\alpha_{j}\}}f^{(m)} ,f^{(m)}   \right)_{\!L_{2}(I)}\geq C \left\| f^{(m)}\right\|^{2}_{L_{2}(I)}.
$$
Thus in both cases,  we have
$$
(p_{j}D_{a+}^{\alpha_{j}}f,f)_{L_{2}(I)}\geq C \left\| f^{(s)}\right\|^{2}_{L_{2}(I)},\;s= \big[[\alpha_{j}]/2\big] .
$$
 In the same way, we obtain the inequality
$$
(q_{j}D_{b-}^{\alpha_{j}}f,f)_{L_{2}(I)}\geq C \left\| f^{(s)}\right\|^{2}_{L_{2}(I)},\;s= \big[[\alpha_{j}]/2\big] .
$$
  Hence in the
complex case we have
$$
\mathrm{Re}(\mathcal{D}f,f)_{L_{2}(I)}\geq C \left\| f \right\|^{2}_{L_{2}(I)},\;f\in \mathrm{D}(\mathcal{D}).
$$
Combining   Theorem 2.6 \cite[p.53]{firstab_lit:samko1987}  with  \eqref{2.53}, we get
$$
\left\| p_{j}D_{a+}^{\alpha_{j}}f \right\| _{L_{2}(I)}=  \left\|  I_{a+}^{1-\{\alpha_{j}\}}f^{([\alpha_{j}]+1)} \right\| _{L_{2}(I)}
 \leq C   \left\|   f^{([\alpha_{j}]+1)} \right\|_{L_{2}(I)}\leq C   \left\|   f  \right\|_{H^{k}_{0}(I)};
 $$
 $$
 \;\left\|q_{j}D_{b-}^{\alpha_{j}}f \right\| _{L_{2}(I)}
\leq  C   \left\|   f  \right\|_{H^{k}_{0}(I)},\;f\in \mathrm{D}(\mathcal{D}).
$$
 Hence, we obtain
$$
\left\| \mathcal{D}f \right\| _{L_{2}(I)}\leq C \left\|  f \right\|_{H^{k}_{0}(I)},\;f\in \mathrm{D}(\mathcal{D}).
$$
Taking into account the relation
$$
\left\|f\right\| _{L_{2}(I)}\leq C \left\|f\right\| _{H^{k}_{0}(I)},\,f\in H^{k}_{0}(I),
$$
Combining the above estimates, we get
$$
\mathrm{Re} (Gf,f)_{L_{2}(I)}\geq C\|f\|^{2}_{H^{k}_{0}(I)},\; |(Gf,g)_{L_{2}(I)}|\leq \|f \| _{H^{k}_{0}(I)}\|g \| _{H^{k}_{0}(I)},\,f,g\in C^{\infty}_{0}(I).
$$
Thus, we have obtained the desired result.

To deploy  the  minimax principle for  eigenvalues estimating, we come to the following relation
$$
C_{1}\|f \|^{2} _{H^{k}_{0}(I)}\leq(\mathfrak{Re} G f,f)_{L_{2}(I)}\leq C_{2}\|f \|^{2} _{H^{k}_{0}(I)},
$$
from what follows  easily, due to the asymptotic formulas for  the selfadjoint operators  eigenvalues (see \cite{firstab_lit:Rosenblum}), the fact
$$
\lambda_{n}(\mathfrak{Re} G)\asymp n^{2k},\;n\in  \mathbb{N},
$$
therefore applying Lemma \ref{L2.4} collaboratively with the asymptotic equivalence formula (see Theorem 5 \cite{firstab_lit(arXiv non-self)kukushkin2018})
$$
\lambda^{-1}_{n}(\mathfrak{Re} G)\asymp \lambda_{n}(\mathfrak{Re}R_{ G}),\;n\in  \mathbb{N},
$$
 we obtain the fact
$$
R_{G}\in \mathfrak{S}_{p},\,\inf p\leq 1/2k.
$$
Thus, it gives us an opportunity to establish the range of  the Schatten index.\\

\noindent{\bf 2.} Let us show the application of Lemma \ref{L2}, firstly consider the following reasonings
$$
\| \mathfrak{Im}W H^{-1} \|_{2}=\| H^{-1}\mathfrak{Im}W\|_{2}=\sum\limits_{n,k=1}^{\infty}\left|(\mathfrak{Im} We_{n},H^{-1}e_{k})_{\mathfrak{H}}\right|^{2}=
\sum\limits_{n,k=1}^{\infty}\lambda^{-2}_{n}(H)\left|(e_{n}, \mathfrak{Im} W e_{k})_{\mathfrak{H}}\right|^{2}=
$$
$$
=\sum\limits_{n=1}^{\infty}\lambda^{-2}_{n}(H)||  \mathfrak{Im} W e_{n} ||_{\mathfrak{H}}^{2},
$$
where $\{e_{n}\}_{1}^{\infty}$ is the orthonormal  set of the eigenvectors of the operator $H.$ Thus, we obtain the following condition
\begin{equation}\label{2.55}
\sum\limits_{n=1}^{\infty}\lambda^{-2}_{n}(H)||  \mathfrak{Im} W e_{n} ||_{\mathfrak{H}}^{2}<1,
\end{equation}
which guaranties the fulfilment of  the   conditions regarding  the absolute norm  in Lemma \ref{L2}. It is remarkable that this form of the condition is quiet convenient if we consider perturbations of differential operators.
   Below we observe  a simplified case of the operator considered in the previous paragraph. Consider
  $$
 Lf := - f''  + \xi D_{0+}^{\alpha }f ,\;
\mathrm{D}(L) = H^{2 }(I)\cap H_{0}^{1}(I) ,\;I=(0,\pi),\,\alpha\in (0,1/2),\,\xi\in \mathbb{R},
$$
then
$$
C_{0}( L_{1}f,f)_{L_{2}(I)}\leq(\mathfrak{Re}Lf,f)_{L_{2}(I)}\leq C_{1}( L_{1}f,f)_{L_{2}(I)},\; L_{1}f:=-f'',\;\mathrm{D}(L_{1})=\mathrm{D}(L).
$$
It is well-known fact that
$$
\lambda_{n}(L_{1})=n^{2},\;e_{n}=\sin nx.
$$
It is also clear that
$$
\mathfrak{Im }L \supset \xi(D_{0+}^{\alpha }-D_{\pi-}^{\alpha })/2i.
$$
In accordance with the first representation theorem $H^{2 }(I)\cap H_{0}^{1}(I)$  is a core of the form corresponding to the operator $L^{\ast},$ hence
$$
\mathfrak{Im }L = \xi(D_{0+}^{\alpha }-D_{\pi-}^{\alpha })/2i.
$$
Note that
$$
\left(D_{0+}^{\alpha }e_{n}\right)(x)=\frac{n}{\Gamma(1-\alpha)}\int\limits_{0}^{x}(x-t)^{-\alpha}\cos n t\, dt
$$
Applying the generalized Minkovskii inequality, we get
$$
\left(\int\limits_{0}^{\pi}\left|(D_{a+}^{\alpha }e_{n})(x)\right|^{2} dx\right)^{1/2}= \frac{n }{\Gamma(1-\alpha)}\left(\int\limits_{0}^{\pi}\left|\int\limits_{0}^{x}(x-t)^{ -\alpha}\cos n t\, dt\right|^{2}\right)^{1/2}\leq
$$
$$
\leq  \frac{n }{\Gamma(1-\alpha)}\int\limits_{0}^{\pi}\cos n t\, dt \left(\int\limits_{t}^{\pi}(x-t)^{ -2\alpha}d x \right)^{1/2}=
 \frac{n }{\sqrt{(1-2\alpha)}\Gamma(1-\alpha)}\int\limits_{0}^{\pi}(\pi-t)^{1/2- \alpha}\cos n t\, dt\leq
$$
$$
\leq
 \frac{n \pi^{1/2-\alpha}}{\sqrt{(1-2\alpha)}\Gamma(1-\alpha)}.
$$
Analogously, we obtain
$$
\left(\int\limits_{0}^{\pi}\left|(D_{\pi-}^{\alpha }e_{n})(x)\right|^{2} dx\right)^{1/2}\leq  \frac{n \pi^{1/2-\alpha}}{\sqrt{(1-2\alpha)}\Gamma(1-\alpha)}\,.
$$
Hence
$$
\|\mathfrak{Im }L e_{n}\|\leq \frac{n \xi\pi^{1/2-\alpha}}{\sqrt{(1-2\alpha)}\Gamma(1-\alpha)}.
$$
Therefore
$$
\sum\limits_{n=1}^{\infty}\lambda^{-2}_{n}(\mathfrak{Re}L)||  \mathfrak{Im} Le_{n} ||^{2}< \frac{  \xi^{2}\pi^{1 -2\alpha}}{ (1-2\alpha) \Gamma^{2}(1-\alpha)}\sum\limits_{n=1}^{\infty}\frac{1}{n^{2}}=\frac{  \xi^{2}\pi^{3 -2\alpha } }{6 (1-2\alpha) \Gamma^{2}(1-\alpha)}.
$$
Using this relation, we can obviously impose a condition on $\xi$ that guarantees the fulfilment of   relation \eqref{2.55}, i.e.
$$
 \xi  <\frac{    \sqrt{6 (1-2\alpha)} \Gamma (1-\alpha)   }{ \pi^{3/2 - \alpha }}.
$$
In accordance with Theorem \ref{T2.7}, the latter condition follows that
$$
s_{n}^{-1}(R_{L})\asymp n^{2},\;R_{L}\in \mathfrak{S}_{p},\,\inf p=1/2.
$$

\noindent {\bf3.} To demonstrate the main result of the section, we produce an example dealing with   well-known operators. Consider a rectangular domain in the space $\mathbb{R}^{n},\;$ defined as follows $\Omega:=\{x_{j}\in [0,\pi],\,j=1,2,...,n\}$  and
consider the  Kipriyanov  fractional differential operator     defined in  the paper \cite{firstab_lit:1kipriyanov1960}  by  the formal expression
\begin{equation*}
\mathfrak{D}^{\beta}f(Q)=\frac{\beta}{\Gamma(1-\beta)}\int\limits_{0}^{r} \frac{[f(Q)-f(T)]}{(r - t)^{\beta+1}} \left(\frac{t}{r} \right) ^{n-1} dt+
(n-1)!f(Q)r ^ {-\beta} /\Gamma(n-\beta),
$$
$$
\beta\in(0,1),\, P\in\partial\Omega,
\end{equation*}
where $Q:=P+\mathbf{e} r,\;P:=P+\mathbf{e}t,\;\mathbf{e}$ is a unit vector having a direction from the fixed point of the boundary $P$ to an arbitrary point $Q$ belonging to $\Omega.$

  Consider the perturbation of the Laplace operator by the Kipriyanov operator
$$
L :=   D^{2k} + \xi \mathfrak{D}^{\beta} ,\,\mathrm{D}(L)= H_{0}^{k}(\Omega)\cap H ^{2k}(\Omega),
$$
where  $\xi>0,$
$$
D^{2k}f=(-1)^{k}\sum\limits_{j=1}^{n} \mathcal{D}_{j}^{2k}f.
$$
It was proved in the paper \cite{kukushkin2021a} that
 $$
 C_{0} (D^{2k} f,f)_{L_{2}(\Omega)}\leq(\mathfrak{Re}Lf,f)_{L_{2}(\Omega)}\leq  C_{1} (D^{2k} f,f)_{L_{2}(\Omega)},\;f\in \mathrm{D}(L).
$$
Therefore
$$
\lambda_{n}(\mathfrak{Re} L) \asymp n^{2k/n}.
$$
On the other hand, we have the following eigenfunctions of $D^{2k}$ in the rectangular
$$
e_{\bar{l}} =\prod\limits_{j=1}^{n}\sin l_{j}x_{j},\;\bar{l}:=\{l_{1},l_{2},...,\,l_{ n }\},\,l_{s}\in \mathbb{N},\,s=1,2,...,n.
$$
It is clear that
$$
\;  D^{2k} e_{\bar{l}} =\lambda_{\bar{l}}\,e_{\bar{l}},\;\lambda_{\bar{l}} =\sum\limits_{j=1}^{n}l^{2k}_{j}.
$$
Since the search  of the given below  information in   literature    (however, it is a well-known fact)  can bring  some difficulties, we would like to represent it. Let us prove that the system
$
\left\{e_{\bar{l}} \right\}
$
is complete in the Hilbert space  $L_{2}(\Omega).$ We will show it if we prove that the element that is orthogonal to every element of the system is a zero.
Assume that
$$
 \int\limits_{0}^{\pi}\sin l_{1}x_{1}dx_{1}\int\limits_{0}^{\pi} \sin l_{2}x_{2}dx_{2}...\int\limits_{0}^{\pi} f(x_{1},x_{2},...,x_{n})\sin l_{n}x_{n}dx_{n}=  (e_{\bar{l}},f)_{L_{2}(\Omega)}=0.
$$
In accordance with the fact that the system $\{\sin m x\}_{1}^{\infty}$ is a compleat system in $L_{2}(0,\pi),$ we conclude that
$$
 \int\limits_{0}^{\pi} \sin l_{2}x_{2}dx_{2}...\int\limits_{0}^{\pi} f(x_{1},x_{2},...,x_{n})\sin l_{n}x_{n}dx_{n} =0.
$$
Having repeated the same reasonings step by step, we obtain the desired result. Taking into account the following   inequality (see \cite{kukushkin2021a} ) and the embedding theorems, we get
\begin{equation}\label{2.56}
\|\mathfrak{D}^{\beta}f\|_{L_{2}(\Omega)}\leq C_{\beta}\|f\|_{H_{0}^{1}(\Omega)}\leq C_{\beta,k,n} \|f\|_{H_{0}^{k}(\Omega)},
\end{equation}
where the following constant $C_{\beta}$ is defined through  the infinitesimal generator  $J$ of the corresponding semigroup of contraction (shift semigroup in the direction) (9) \cite{kukushkin2021a}. Now it is clear  that the conditions H1,H2 are satisfied, where $\mathfrak{H}:=L_{2}(\Omega),\,\mathfrak{H}_{+}:=H_{0}^{k}(\Omega),\,\mathfrak{M}:=C_{0}^{\infty}(\Omega).$
Using the intermediate inequality \eqref{2.56}, by  direct calculation, we get
$$
\sum\limits_{l_{1},l_{2},...l_{n}=1}^{\infty} \lambda^{-2}_{\bar{l}}(\mathfrak{Re}L)_{L_{2}(\Omega)} \,\|\mathfrak{Im} L e_{\bar{l}}\|^{2}_{L_{2}(\Omega)}  \leq (\xi
C_{\beta})^{2}\!\!\!\sum\limits_{l_{1},l_{2},...l_{n}=1}^{\infty}\frac{\lambda _{\bar{l}}(  D^{2 })}{\lambda^{ 2}_{\bar{l}}(  D^{2k})}.
$$
Therefore, if the following condition holds
\begin{equation}\label{2.57}
 \sum\limits_{l_{1},l_{2},...l_{n}=1}^{\infty}\frac{l^{2}_{1}+l^{2}_{2}+...+l^{2}_{n} }{(l^{2k}_{1}+l^{2k}_{2}+...+l^{2k}_{n})^{2}}< (\xi
C_{\beta})^{-2},
\end{equation}
then the conditions of Lemma \ref{L2} are satisfied. Applying Lemma \ref{L2}, we can    consider the values of the parameters $k,n$ such that the last series is convergent and at the same time $R_{L}\in \mathfrak{S}_{p},\,\inf p=n/2k>1.$ The latter gives us the argument showing  relevance of Lemma \ref{L2} since we can find the range of $p.$  Below, we demonstrate  the corresponding reasoning.

Assume that the following  condition holds
$$
\frac{n}{2}+1<2k<n .
$$
Consider the vector function
$$
\psi(\bar{l})=\frac{(l^{2k}_{1}+l^{2k}_{2}+...+l^{2k}_{n})^{2}   }{ l^{2}_{1}+l^{2}_{2}+...+l^{2}_{n} },
$$
then $\psi(\bar{t})=nt^{2(2k-1)},\,\bar{t}=\{t,t,...t\}.$ It is clear that the number $s$ of values   $\psi(\bar{l}),\,l_{i}\leq t$ equals to $t^{n},$ i.e. $s=t^{n}.$ Therefore
$$
\psi(\bar{t})=ns^{\frac{2(2k-1)}{n}},\;\psi( \overline{t-1} )=n(s^{1/n}-1)^{2(2k-1) };
$$
$$
n(s^{1/n}-1)^{2(2k-1) } \leq\psi(\bar{l})\leq ns^{\frac{2(2k-1)}{n}},\; t-1\leq l_{i}\leq t,\;i=1,2,...,n.
$$
Having arranged the values in the order corresponding to their absolute value  increasing, we get
$$
 n(s^{1/n}-1)^{2(2k-1) }  \leq\psi_{j} \leq ns^{\frac{2(2k-1)}{n}},\;(s^{1/n}-1)^{n}<j< s.
$$
Therefore
$$
\frac{(s^{1/n}-1)^{2(2k-1) }}{s^{\frac{2(2k-1)}{n}}}<\frac{\psi_{j} }{nj^{\frac{2(2k-1)}{n}}}<\frac{ s^{\frac{2(2k-1)}{n}}}{(s^{1/n}-1)^{2(2k-1) } },
$$
from what follows the convergence of the following series, since if we take into account the condition  $n/2+1<2k,$ we get
$$
\sum\limits_{j=1}^{\infty} \psi^{-1}_{j}<\infty
$$
what gives us the desired result, i.e. the conditions of Lemma \ref{L2} are satisfied.

\section{Remarks}\label{S2.7}

Consider  a condition  $\mathfrak{M}\subset \mathrm{D}( W ^{\ast}),$ in this case the real Hermitian component  $\mathcal{H}:=\mathfrak{Re }\,W$ of the operator is defined on $\mathfrak{M},$ the fact is that $\tilde{\mathcal{H}}$ is selfadjoint,    bounded  from bellow (see Lemma  \ref{L2.3}), where $H=Re W.$  Hence a corresponding sesquilinear  form (denote this form by $h$) is symmetric and  bounded from bellow also (see Theorem 2.6 \cite[p.323]{firstab_lit:kato1980}). It can be easily shown  that $h\subset   \mathfrak{h},$  but using this fact    we cannot claim in general that $\tilde{\mathcal{H}}\subset H$ (see \cite[p.330]{firstab_lit:kato1980}). We just have an inclusion   $\tilde{\mathcal{H}}^{1/2}\subset H^{1/2}$     (see \cite[p.332]{firstab_lit:kato1980}). Note that the fact $\tilde{\mathcal{H}}\subset H$ follows from a condition $ \mathrm{D}_{0}(\mathfrak{h})\subset \mathrm{D}(h) $ (see Corollary 2.4 \cite[p.323]{firstab_lit:kato1980}).
 However, it is proved (see proof of Theorem  \ref{T2.1}) that relation H2 guaranties that $\tilde{\mathcal{H}}=H.$ Note that the last relation is very useful in applications, since in most concrete cases we can find a concrete form of the operator $\mathcal{H}.$

\chapter{Semigroup approach}\label{Ch3}
\section{Historical review}

To write this chapter, we were firstly motivated by  the  boundary value problems of the  Sturm-Liouville type   for fractional
differential equations. Many authors devoted their attention to the topic, nevertheless   this kind of problems are relevant for today. First of all, it is connected  with the fact that they model various physical -
chemical processes: filtration of liquid and gas in highly porous fractal   medium; heat exchange processes in medium  with fractal structure and memory; casual walks of a point particle that starts moving from the origin
by self-similar fractal set; oscillator motion under the action of
elastic forces which is  characteristic for  viscoelastic media, etc.
  In particular,  we would like   to  study  the  eigenvalue problem for a   differential operator  with a fractional derivative in  final terms, in this connection such operators as a Kipriyanov fractional differential operator, Riesz potential,  difference operator are involved.

   In the case corresponding to a selfadjoint senior term we can partially solve the problem  having applied the results of the   perturbation theory,   within the framework of which  the following papers are well-known   \cite{firstab_lit:1Katsnelson}, \cite{firstab_lit:1Krein},   \cite{firstab_lit:2Markus},
  \cite{firstab_lit:3Matsaev},\cite{firstab_lit:Markus Matsaev},
 \cite{firstab_lit:Shkalikov A.}. Generally, to apply the last paper results for a concrete operator $L$ we must be able to represent  it  by  a sum   $L=T+A,$ where the senior term
    $T$ must be either a selfadjoint or normal operator. In other cases we can use methods of the papers
    \cite{kukushkin2019},\cite{firstab_lit(arXiv non-self)kukushkin2018}  which are relevant if we deal with non-selfadjoint operators and allow us  to study spectral properties of   operators  whether we have the mentioned above  representation or not.  We should add that the results of  the paper \cite{firstab_lit:Markus Matsaev}  can be  also applied to study non-selfadjoin operators (see a detailed remark in \cite{firstab_lit:Shkalikov A.}).

In many papers    the eigenvalue problem  was studied by methods of a theory of functions and it is remarkable that  special properties of the fractional derivative were used in these papers, bellow  we present a brief  review.

However,  we  deal with a more general operator --- a  differential operator    with a fractional  integro-differential operator   composition  in   final terms, which covers the  operator mentioned above. Note that   several types of   compositions of fractional integro-differential operators were studied by such mathematicians as
 Prabhakar T.R. \cite{firstab_lit:1Prabhakar}, Love E.R. \cite{firstab_lit:5Love}, Erdelyi A. \cite{firstab_lit:15Erdelyi}, McBride A. \cite{firstab_lit:9McBride},
  Dimovski I.H., Kiryakova V.S. \cite{firstab_lit:2Dim-Kir}, Nakhushev A.M. \cite{firstab_lit:nakh2003}.

 The central idea of this paper  is to built a   model that gives us a representation  of a  composition of  fractional differential operators   in terms of the semigroup theory.    For instance we can    represent a second order differential operator as some kind of  a  transform of   the infinitesimal generator of a shift semigroup. Continuing this line of reasonings we generalize a   differential operator with a fractional integro-differential composition  in final terms   to some transform of the corresponding  infinitesimal generator   and introduce  a class of  transforms of   m-accretive operators. Further,   we  use   methods obtained in the papers
\cite{firstab_lit(arXiv non-self)kukushkin2018},\cite{kukushkin2019} to  study spectral properties  of non-selfadjoint operators acting  in  a complex  separable Hilbert space, these methods alow us to   obtain an  asymptotic equivalence between   the
real component of the resolvent and the resolvent of the   real component of an operator. Due  to such an approach we  obtain relevant  results since  an asymptotic formula  for   the operator  real component  can be  established in many cases
(see \cite{firstab_lit:2Agranovich2011}, \cite{firstab_lit:Rosenblum}). Thus,   a classification  in accordance with  resolvent  belonging     to  the  Schatten-von Neumann  class is obtained,   a sufficient condition of completeness of the root vectors system is   formulated. As the most significant result  we obtain  an   asymptotic formula for the   eigenvalues.

 \section{  Transform}

\subsection{Accretive property}

  Let $f_{t} :I\rightarrow \mathfrak{H},\,t\in I:=[a,b],\,-\infty< a <b<\infty.$ The following integral is understood in the Riemann  sense as a limit of partial sums
\begin{equation}\label{3.1}
\sum\limits_{i=0}^{n}f_{\xi_{i}}\Delta t_{i}  \stackrel{\mathfrak{H}}{ \longrightarrow}  \int\limits_{I}f_{t}dt,\,\lambda\rightarrow 0,
\end{equation}
where $(a=t_{0}<t_{1}<...<t_{n}=b)$ is an arbitrary splitting of the segment $I,\;\lambda:=\max\limits_{i}(t_{i+1}-t_{i}),\;\xi_{i}$ is an arbitrary point belonging to $[t_{i},t_{i+1}].$
The sufficient condition of the last integral existence is a continuous property (see\cite[p.248]{firstab_lit:Krasnoselskii M.A.}) i.e.
$
f_{t}\stackrel{\mathfrak{H}}{ \longrightarrow}f_{t_{0}},\,t\rightarrow t_{0},\;\forall t_{0}\in I.
$
The improper integral is understood as a limit
\begin{equation}\label{3.2}
 \int\limits_{a}^{b}f_{t}dt\stackrel{\mathfrak{H}}{ \longrightarrow} \int\limits_{a}^{c}f_{t}dt,\,b\rightarrow c,\,c\in [-\infty,\infty].
\end{equation}

In this paragraph  we present propositions devoted to properties of  accretive operators and related questions.
For a reader convenience, we would like to establish well-known facts  of the operator theory under  a  point of view  that  is necessary    for the following reasonings.
\begin{lem}\label{L3.1}
Assume that   $A$ is  a  closed densely defined  operator, the following    condition  holds
\begin{equation}\label{3.3}
\|(A+t)^{-1}\|_{\mathrm{R} \rightarrow \mathfrak{H}}\leq\frac{1}{ t},\,t>0,
\end{equation}
where a notation  $\mathrm{R}:=\mathrm{R}(A+t)$ is used. Then the operators  $A,A^{\ast}$ are m-accretive.
\end{lem}
\begin{proof}
Using \eqref{3.3} consider
$$
\|f\|^{2}_{\mathfrak{H}}\leq  \frac{1}{t^{2}}  \|(A+t)f \|^{2}_{ \mathfrak{H}};\,
 \|f\|^{2}_{\mathfrak{H}}\leq  \frac{1}{ t^{2}} \left\{ \| A f\|^{2}_{ \mathfrak{H}}+2t \mathrm{Re}(Af,f)_{ \mathfrak{H}}+t^{2}\|   f \|^{2}_{ \mathfrak{H}}\right\} ;
$$
 $$
t^{-1} \| A  f \|^{2}_{ \mathfrak{H}}+2 \mathrm{Re}(A f,f)_{ \mathfrak{H}}\geq0,\,f\in  \mathrm{D} (A).
$$
Let $t$ be tended to infinity, then we obtain
\begin{equation}\label{3.4}\mathrm{Re}(A f,f)_{ \mathfrak{H}}\geq0 ,\,f\in \mathrm{D}(A).
\end{equation}
It means  that   the  operator $A$ has an accretive property.
Due to \eqref{3.4}, we have   $\{\lambda \in \mathbb{C}:\,\mathrm{Re}\lambda<0\}\subset \Delta(A),  $ where $\Delta(A)=\mathbb{C}\setminus   \overline{\Theta  (A)}.$  Applying Theorem 3.2 \cite[p.268]{firstab_lit:kato1980}, we obtain that $A-\lambda$ has a closed range and  $\mathrm{nul} (A-\lambda)=0,\,\mathrm{def} (A-\lambda)=\mathrm{const},\,\forall\lambda\in \Delta(A).$
Let $\lambda_{0}\in \Delta(A) ,\;{\rm Re}\lambda_{0} <0.$
Note that in consequence of inequality  \eqref{3.4}, we have
 \begin{equation}\label{3.5}
  {\rm Re} ( f,(A-\lambda )f  )_{\mathfrak{H}}\geq   - {\rm Re} \lambda   \|f\|^{2}_{\mathfrak{H}},\,f\in \mathrm{D}(A).
 \end{equation}
  Since the  operator $A-\lambda_{0}$ has a closed range, then
\begin{equation*}
 \mathfrak{H}=\mathrm{R} (A-\lambda_{0})\oplus \mathrm{R} (A-\lambda_{0})^{\perp} .
 \end{equation*}
We remark  that the intersection of the sets  $\mathrm{D}(A)$ and $\mathrm{R} (A-\lambda_{0})^{\perp}$ is  zero, because if we assume  the contrary,   then applying inequality  \eqref{3.5},  for arbitrary element
 $f\in \mathrm{D}(A)\cap \mathrm{R}  (A-\lambda_{0})^{\perp}$    we get
 \begin{equation*}
 - {\rm Re} \lambda_{0}  \|f\|^{2}_{\mathfrak{H}} \leq  {\rm Re} ( f,[A-\lambda_{0} ]f  )_{\mathfrak{H}}=0,
 \end{equation*}
hence $f=0.$   It implies that
$$
\left(f,g\right)_{\mathfrak{H}}=0,\;\forall f\in  \mathrm{R}  (A-\lambda_{0})^{\perp},\;\forall g\in \mathrm{D}(A).
$$
Since $  \mathrm{D}(A)$  is a dense set in $\mathfrak{H},$ then $\mathrm{R}  (A-\lambda_{0})^{\perp}=0.$ It implies that  ${\rm def} (A-\lambda_{0}) =0$ and if we take into account  Theorem 3.2 \cite[p.268]{firstab_lit:kato1980}, then
  we come to the conclusion that ${\rm def} (A-\lambda )=0,\;\forall\lambda\in \Delta(A),$ hence the operator $A$ is m-accretive.

Now assume that the operator $A$ is m-accretive.
 Since it is proved that $\mathrm{def}(A+\lambda)=0,\,\lambda> 0,$ then $ \mathrm{nul}(A+\lambda)^{\ast}=0,\,\lambda> 0$ (see (3.1) \cite[p.267]{firstab_lit:kato1980}).   In accordance with the  well-known fact, we have $\left([\lambda  +A]^{-1}\right)^{\ast}=[\left(\lambda  +A\right)^{\ast}]^{-1}.$ Using the obvious relation $\lambda+A^{\ast}=\left(\lambda  +A\right)^{\ast},$   we can deduce  $(\lambda+A^{\ast})^{-1}=[\left(\lambda  +A\right)^{\ast}]^{-1}.$ Also it is obvious that $\left\|( \lambda  +A)^{-1}\right\| =\left\| [( \lambda  +A)^{-1}]^{\ast}\right\|,$ since both operators are bounded. Hence
$$
\| \left(\lambda  +A^{\ast}\right) ^{-1}f  \|_{\mathfrak{H}}=\|[\left(\lambda  +A\right)^{\ast}]^{-1}f  \|_{\mathfrak{H}}=\|\left([\lambda  +A]^{-1}\right)^{\ast}f\|_{\mathfrak{H}} \leq\frac{1}{ \lambda}\|f\|_{\mathfrak{H}},
 f\in \mathrm{R}(\lambda  +A^{\ast}),\;\lambda>0.
$$
This relation can be rewritten in the following form
\begin{equation*}
\|\left(\lambda  +A^{\ast}\right)^{-1}\|_{\mathrm{R}\rightarrow \mathfrak{H}}  \leq\frac{1}{ \lambda},\,\lambda>0.
\end{equation*}
Using the proved above fact, we   conclude that
\begin{equation}\label{3.6}
\|\left(\lambda  +A^{\ast}\right)^{-1}\| \leq\frac{1}{\mathrm{Re}\lambda},\,\mathrm{Re}\lambda>0.
\end{equation}
The proof is complete.
\end{proof}

In accordance with the definition given  in \cite{firstab_lit:Krasnoselskii M.A.} we can define  a positive and negative  fractional powers of a positive  operator $A$ as follows
\begin{equation}\label{3.7}
A^{\alpha}:=\frac{\sin\alpha \pi}{\pi}\int\limits_{0}^{\infty}\lambda^{\alpha-1}(\lambda  +A)^{-1} A \,d \lambda;\,\,A^{-\alpha}:=\frac{\sin\alpha \pi}{\pi}\int\limits_{0}^{\infty}\lambda^{-\alpha}(\lambda  +A)^{-1}  \,d \lambda,\,\alpha\in (0,1).
\end{equation}
This definition can be correctly extended  on m-accretive operators, the corresponding reasonings can be found in \cite{firstab_lit:kato1980}. Thus, further we define positive and negative fractional powers of m-accretive operators by formula \eqref{3.7}.

 \begin{lem}\label{L3.2}
 Assume that $\alpha\in(0,1),$ the operator   $J$ is m-accretive,    $  J^{-1} $ is bounded, then
\begin{equation}\label{3.8}
\|J^{-\alpha}f\|_{\mathfrak{H}}\leq C_{   1-\alpha}\|f\|_{\mathfrak{H}},\,f\in \mathfrak{H},
\end{equation}
where
$
C_{ 1-\alpha }= 2 (1-\alpha)^{-1} \|J^{-1}\|  + \alpha^{-1} .$
\end{lem}
\begin{proof}
 Consider
$$
J^{-\alpha}=\int\limits_{0}^{1}\lambda^{-\alpha }(\lambda+J)^{-1}d\lambda+\int\limits_{1}^{\infty}\lambda^{-\alpha}(\lambda+J)^{-1}d\lambda=I_{1}+I_{2}.
$$
Using definition of the integral \eqref{3.1},\eqref{3.2}  in a Hilbert  space  and the fact $J(\lambda+J)^{-1}f=(\lambda+J)^{-1}Jf,\,f\in \mathrm{D}(J),$ we can easily obtain
$$
\|I_{1}f\|_{\mathfrak{H}}=\left\|\int\limits_{0}^{1}\lambda^{-\alpha}J^{-1}J(\lambda+J)^{-1}fd\lambda\right\|_{\mathfrak{H}}\leq
\|J^{-1}\|  \cdot\left\|\int\limits_{0}^{1}\lambda^{-\alpha}J(\lambda+J)^{-1}fd\lambda\right\|_{\mathfrak{H}}\leq
$$
$$
\leq \|J^{-1}\|_{\mathrm{R} \rightarrow \mathfrak{H}} \cdot \left\{\left\| f \right\|_{\mathfrak{H}}  \int\limits_{0}^{1}\lambda^{-\alpha} fd\lambda  +  \left\|\int\limits_{0}^{1}\lambda^{1-\alpha} (\lambda+J)^{-1}fd\lambda\right\|_{\mathfrak{H}}   \right\}       \leq
 $$
 $$
 \leq2 \|J^{-1}\|\cdot\left\|  f \right\|_{\mathfrak{H}} \int\limits_{0}^{1}\lambda^{-\alpha} d\lambda
,\,f\in \mathrm{D}(J);
$$
$$
\|I_{2}f\|_{\mathfrak{H}}=\left\|\int\limits_{1}^{\infty}\lambda^{-\alpha} (\lambda+J)^{-1}fd\lambda\right\|_{\mathfrak{H}}\leq
\left\|  f \right\|_{\mathfrak{H}}    \int\limits_{1}^{\infty}\lambda^{-\alpha} \|(\lambda+J)^{-1}\| d\lambda  \leq\left\|  f \right\|_{\mathfrak{H}} \int\limits_{1}^{\infty}\lambda^{-\alpha-1}   d\lambda .
$$
Hence
$ J^{-\alpha}
$  is bounded on $\mathrm{D}(J).$  Since $\mathrm{D}(J)$ is dense in $\mathfrak{H},$ then $J^{-\alpha}$ is bounded on $\mathfrak{H}.$ Calculating the right-hand sides of the above estimates, we obtain \eqref{3.8}.
\end{proof}

\subsection{Main theorem}

Consider a transform of an m-accretive operator $J$ acting in $\mathfrak{H}$
\begin{equation}\label{3.9}
 Z^{\alpha}_{G,F}(J):= J^{\ast}GJ+FJ^{\alpha},\,\alpha\in [0,1),
\end{equation}
where symbols  $G,F$  denote  operators acting in $\mathfrak{H}.$    Further, using a relation $L= Z^{\alpha}_{G,F}(J)$ we mean that there exists an appropriate representation for the operator $L.$
\noindent The following theorem gives us a tool to describe spectral properties  of transform  \eqref{3.9},
  as it will be shown   further  it has an important  application in fractional calculus  since  allows   to represent fractional differential  operators as a transform of the infinitesimal  generator of a    semigroup.

\begin{teo}\label{T3.1}
 Assume that  the operator   $J$ is m-accretive,    $  J^{-1} $ is compact, $G$ is bounded, strictly accretive, with  a lower bound $\gamma_{G}> C_{ \alpha} \|J^{-1}\|\cdot \|F\|,\; \mathrm{D}(G)\supset \mathrm{R}(J),$   $F\in \mathcal{B}(\mathfrak{H}),$ where $C_{\alpha}$  is a   constant   \eqref{3.8}.
 Then   $Z^{\alpha}_{G,F}(J)$ satisfies  conditions  H1 -  H2.

\end{teo}
\begin{proof}
Since  $J$ is m-accretive, then  it  is    closed, densely defined (see \cite[p.279]{firstab_lit:kato1980},  using the fact that $(J+\lambda)^{-1},\,(\lambda>0)$ is a closed operator, we conclude that $J$ is closed also).
Firstly,  we want to check fulfilment of condition $\mathrm{H1}.$  Let us choose  a space $\mathfrak{H}_{J}$ as a space   $\mathfrak{H}_{+}.$        Since $J^{-1}$ is compact, then   we conclude that the following   relation holds
 $\|f\|_{\mathfrak{H}}\leq \|J^{-1}\| \cdot \|Jf\|_{\mathfrak{H}},\,f\in \mathrm{D}(J) $ and    the embedding provided by this   inequality is compact.    Thus condition H1 is satisfied.

Let us prove that  $\mathrm{D}(J^{\ast}GJ)$ is a core of $J.$
Consider a space $\mathfrak{H}_{J}$ and a sesquilinear form
$$
 l_{G}(u,v):=(GJu,Jv)_{\mathfrak{H}} ,\;u,v\in \mathrm{D}(J).
$$
Observe that this form is a bounded functional on $\mathfrak{H}_{J},$ since we have
$
 |(GJu,Jv)_{\mathfrak{H}}|\leq \|G\|\cdot \|Ju\|_{\mathfrak{H}}  \|Jv\|_{\mathfrak{H}}.
$
Hence using the Riesz representation theorem, we have
$$
\forall z\in \mathrm{D}(J),\,   \exists f\in \mathrm{D}(J):\,(GJz,Jv)_{\mathfrak{H}}=(Jf,Jv)_{\mathfrak{H}}.
$$
On the other hand,   due to the properties of the operator $G,$ it is clear that  the conditions of the Lax-Milgram theorem are satisfied  i.e.
$
|(GJu,Jv)_{\mathfrak{H}}|\leq \|G\|\cdot\|Ju\|_{\mathfrak{H}} \|Jv\|_{\mathfrak{H}},\;|(GJu,Ju)_{\mathfrak{H}}|\geq \gamma_{G} \|Ju\|^{2}_{\mathfrak{H}}.
$
 Note that,   in accordance with  Theorem 3.24 \cite[p.275]{firstab_lit:kato1980} the set
$  \mathrm{D} (J^{\ast}J)$ is a core of $J$  i.e.
$$
\forall f\in \mathrm{D}(J),\,\exists \{f_{n}\}_{1}^{\infty}\subset \mathrm{D} (J^{\ast}J):   \,f_{n}\xrightarrow[ J ]{}f.
$$
 Using the Lax-Milgram theorem, in the previously used terms, we get
$$
   \forall f_{n},\, n\in \mathbb{N},\,    \exists z_{n}\in \mathrm{D}(J):\,(GJz_{n},Jv)_{\mathfrak{H}}=(Jf_{n},Jv)_{\mathfrak{H}}.
$$
Combining the above relations, we obtain
$$
(GJ\xi_{n},Jv)_{\mathfrak{H}}=(J\psi_{n},Jv)_{\mathfrak{H}},\,
$$
where $\xi_{n}:=z-z_{n},\,\psi_{n}:=f-f_{n}.$
  Using the strictly accretive property of the operator $G,$ we have
$$
\|J\xi_{n}\|_{\mathfrak{H}}^{2}\gamma_{G} \leq|(GJ\xi_{n},J\xi_{n})_{\mathfrak{H}}|=|(J\psi_{n},J\xi_{n})_{\mathfrak{H}}|\leq\|J\psi_{n}\|_{\mathfrak{H}}\|J\xi_{n}\|_{\mathfrak{H}} .
$$
Taking into account that $J^{-1}$ is bounded, we obtain
$$
K_{1}\| \xi_{n}\|_{\mathfrak{H}}\leq\|J\xi_{n}\|_{\mathfrak{H}} \leq K_{2}\|J\psi_{n}\|_{\mathfrak{H}},\; K_{1},K_{2}>0,
$$
from what follows that
$$
  Jz_{n} \xrightarrow[   ]{\mathfrak{H}}Jz .
$$
On the other hand, we have
$$
(GJz_{n},Jv)_{\mathfrak{H}}=(Jf_{n}, Jv)_{\mathfrak{H}}=(J^{\ast}Jf_{n}, v)_{\mathfrak{H}},\,v\in \mathrm{D}(J).
$$
Hence $\{z_{n}\}_{1}^{\infty}\subset \mathrm{D}(J^{\ast}GJ).$
Taking into account  the above reasonings, we conclude that $\mathrm{D}(J^{\ast}GJ)$ is a core of $J.$ Thus, we have obtained the desired result.

   Note that $\mathrm{D}_{0}(J)$ is dense in $\mathfrak{H},$ since $J$ is densely defined. We have proved above
$$
\mathrm{Re}\left(J^{\ast}GJf,f\right)_{\mathfrak{H}}=\mathrm{Re}\left( GJf,Jf\right)_{\mathfrak{H}}\geq  \gamma_{G}\|f\|^{2}_{\mathfrak{H}_{J}},\,
$$
$$
 \left|\left(J^{\ast}GJf,g\right)_{\mathfrak{H}}\right|=\left|\left( GJf,Jg\right)_{\mathfrak{H}}\right|\leq \|G\|\cdot\|Jf\|_{\mathfrak{H}}\|Jg\|_{\mathfrak{H}},\,f,g\in  \mathrm{D}_{0}(J).
$$
Similarly, we get
\begin{equation}\label{3.10}
|(FJ^{\alpha}f,g)_{\mathfrak{H}}|\leq \|FJ^{\alpha}f\|_{\mathfrak{H}}\|g\|_{\mathfrak{H}}\leq \|J^{-1}\|\cdot\|F\|\cdot\| J^{\alpha}f\|_{\mathfrak{H}}\|Jg\|_{\mathfrak{H}},\,f,g\in \mathrm{D}_{0}(J) .
\end{equation}
In accordance with \eqref{3.7}, we have $J^{\alpha-1}J\subset J^{\alpha}.$
 Therefore, using Lemma \ref{L3.2}, we obtain
\begin{equation}\label{3.11}
\|J^{\alpha}f\|_{\mathfrak{H}}= \|J^{\alpha-1}Jf\|_{\mathfrak{H}}\leq C_{  \alpha  }\| Jf\|_{\mathfrak{H}},\,f\in \mathrm{D}_{0}(J).
\end{equation}
  Combining this fact with \eqref{3.10},  we obtain
\begin{equation*}
|(FJ^{\alpha}f,g)_{\mathfrak{H}}| \leq  C_{  \alpha }\|J^{-1}\|\cdot \|F\|\cdot \|  f\|_{\mathfrak{H}_{J}}\|g\|_{\mathfrak{H}_{J}},\,f,g\in \mathrm{D}_{0}(J),
\end{equation*}
(the  case corresponding to $ \alpha=0 $ is trivial, since the operator $J^{-1}$ is bounded).
 It follows that
$$
\mathrm{Re}(FJ^{\alpha}f,f)\geq- C_{ \alpha}\|J^{-1}\|\cdot \|F\|\cdot \|  f\|^{2}_{\mathfrak{H}_{J}},\,f\in \mathrm{D}_{0}(J).
$$
Combining the above facts, we obtain fulfillment of  condition  $ \mathrm{H2}.$

\end{proof}

\begin{deff}\label{D3.1}
  Define an operator class $\mathfrak{G_{\alpha}} :=\{W:\,W\!= Z^{\alpha}_{G,F}(J) \},$ where $G,F,J$ satisfy  the conditions   of Theorem \ref{T3.1}.
\end{deff}

\section{Model}

In this section we consider various   operators  acting in a complex separable  Hilbert space  for which   Theorem \ref{T2.6} can be applied,  the given bellow results also   cover  a case $\alpha=0$ after   minor changes  which are omitted due to simplicity. Recall, that in accordance with the section  \ref{S2.7}  the conditions of Theorem \ref{T2.6} allow to claim that $\tilde{\mathcal{H}}=H$ what creates a convenient tool  in applications due to the opportunity of using the   minimax principle   for establishing the order of $H.$      \\

\subsection{Kipriyanov operator}\label{Sub3.3.1}

Here, we study a case $\alpha\in (0,1).$ Assume that  $\Omega\subset \mathbb{E}^{n}$ is  a convex domain, with a sufficient smooth boundary ($ C ^{3}$ class)   of the n-dimensional Euclidian space. For the sake of the simplicity we consider that $\Omega$ is bounded, but  the results  can be extended     to some type of    unbounded domains.
In accordance with the definition given in  the paper  \cite{firstab_lit:1kukushkin2018}, we consider the directional  fractional integrals.  By definition, put
$$
 (\mathfrak{I}^{\alpha}_{0+}f)(Q  ):=\frac{1}{\Gamma(\alpha)} \int\limits^{r}_{0}\frac{f (P+t \mathbf{e} )}{( r-t)^{1-\alpha}}\left(\frac{t}{r}\right)^{n-1}\!\!\!\!dt,\,(\mathfrak{I}^{\alpha}_{d-}f)(Q  ):=\frac{1}{\Gamma(\alpha)} \int\limits_{r}^{d }\frac{f (P+t\mathbf{e})}{(t-r)^{1-\alpha}}\,dt,
$$
$$
\;f\in L_{p}(\Omega),\;1\leq p\leq\infty.
$$
Also,     we   consider auxiliary operators,   the so-called   truncated directional  fractional derivatives    (see \cite{firstab_lit:1kukushkin2018}).  By definition, put
 \begin{equation*}
 ( \mathfrak{D} ^{\alpha}_{0+,\,\varepsilon}f)(Q)=\frac{\alpha}{\Gamma(1-\alpha)}\int\limits_{0}^{r-\varepsilon }\frac{ f (Q)r^{n-1}- f(P+\mathbf{e}t)t^{n-1}}{(  r-t)^{\alpha +1}r^{n-1}}   dt+\frac{f(Q)}{\Gamma(1-\alpha)} r ^{-\alpha},\;\varepsilon\leq r\leq d ,
 $$
 $$
 (\mathfrak{D}^{\alpha}_{0+,\,\varepsilon}f)(Q)=  \frac{f(Q)}{\varepsilon^{\alpha}}  ,\; 0\leq r <\varepsilon ;
\end{equation*}
\begin{equation*}
 ( \mathfrak{D }^{\alpha}_{d-,\,\varepsilon}f)(Q)=\frac{\alpha}{\Gamma(1-\alpha)}\int\limits_{r+\varepsilon }^{d }\frac{ f (Q)- f(P+\mathbf{e}t)}{( t-r)^{\alpha +1}} dt
 +\frac{f(Q)}{\Gamma(1-\alpha)}(d-r)^{-\alpha},\;0\leq r\leq d -\varepsilon,
 $$
 $$
  ( \mathfrak{D }^{\alpha}_{d-,\,\varepsilon}f)(Q)=      \frac{ f(Q)}{\alpha} \left(\frac{1}{\varepsilon^{\alpha}}-\frac{1}{(d -r)^{\alpha} }    \right),\; d -\varepsilon <r \leq d .
 \end{equation*}
  Now, we can  define  the directional   fractional derivatives as follows
 \begin{equation*}
 \mathfrak{D }^{\alpha}_{0+}f=\lim\limits_{\stackrel{\varepsilon\rightarrow 0}{ (L_{p}) }} \mathfrak{D }^{\alpha}_{0+,\varepsilon} f  ,\;
  \mathfrak{D }^{\alpha}_{d-}f=\lim\limits_{\stackrel{\varepsilon\rightarrow 0}{ (L_{p}) }} \mathfrak{D }^{\alpha}_{d-,\varepsilon} f ,\,1\leq p\leq\infty.
\end{equation*}
The properties of these operators  are  described  in detail in the paper  \cite{firstab_lit:1kukushkin2018}. Similarly to the monograph \cite{firstab_lit:samko1987} we consider   left-side  and  right-side cases. For instance, $\mathfrak{I}^{\alpha}_{0+}$ is called  a left-side directional  fractional integral  and $ \mathfrak{D }^{\alpha}_{d-}$ is called a right-side directional fractional derivative. We suppose  $\mathfrak{I}^{0}_{0+} =I.$ Nevertheless,   this    fact can be easily proved dy virtue of  the reasonings  corresponding to the one-dimensional case and   given in \cite{firstab_lit:samko1987}. We also consider integral operators with a weighted factor (see \cite[p.175]{firstab_lit:samko1987}) defined by the following formal construction
$$
 \left(\mathfrak{I}^{\alpha}_{0+}\mu f\right) (Q  ):=\frac{1}{\Gamma(\alpha)} \int\limits^{r}_{0}
 \frac{(\mu f) (P+t\mathbf{e})}{( r-t)^{1-\alpha}}\left(\frac{t}{r}\right)^{n-1}\!\!\!\!dt,
$$
where $\mu$ is a real-valued  function.

Consider a linear combination of an uniformly elliptic operator, which is written in the divergence form, and
  a composition of a   fractional integro-differential  operator, where the fractional  differential operator is understood as the adjoint  operator  regarding  the Kipriyanov operator  (see  \cite{firstab_lit:kipriyanov1960},\cite{firstab_lit:1kipriyanov1960},\cite{kukushkin2019})
\begin{equation*}
 L  :=-  \mathcal{T}  \, +\mathfrak{I}^{\sigma}_{ 0+}\rho\, \mathfrak{D}  ^{ \alpha }_{d-},
\; \sigma\in[0,1) ,
 $$
 $$
   \mathrm{D}( L )  =H^{2}(\Omega)\cap H^{1}_{0}( \Omega ),
  \end{equation*}
where
$\,\mathcal{T}:=D_{j} ( a^{ij} D_{i}\cdot),\,i,j=1,2,...,n,$
under    the following  assumptions regarding        coefficients
\begin{equation*}
     a^{ij}(Q) \in C^{2}(\bar{\Omega}),\, \mathrm{Re} a^{ij}\xi _{i}  \xi _{j}  \geq   \gamma_{a}  |\xi|^{2} ,\,  \gamma_{a}  >0,\,\mathrm{Im }a^{ij}=0 \;(n\geq2),\,
 \rho\in L_{\infty}(\Omega).
\end{equation*}
Note that in the one-dimensional case the operator $\mathfrak{I}^{\sigma }_{ 0+} \rho\, \mathfrak{D}  ^{ \alpha }_{d-}$ is reduced to   a  weighted fractional integro-differential operator  composition, which was studied properly  by many researchers (see introduction, \cite[p.175]{firstab_lit:samko1987}).
Consider a shift semigroup in a direction acting on $L_{2}(\Omega)$ and  defined as follows
$
T_{t}f(Q):=f(P+\mathbf{e}[r+ t])=f(Q+\mathbf{e}t).
$
We can formulate the following proposition.

\begin{lem}\label{L3.3}
The semigroup $T_{t}$ is a $C_{0}$  semigroup of contractions.
\end{lem}
 \begin{proof}
     By virtue of the continuous in average property, we conclude that $T_{t}$ is a strongly continuous semigroup. It can be easily established  due to the following reasonings, using the Minkowski inequality, we have
 $$
 \left\{\int\limits_{\Omega}|f(Q+\mathbf{e}t)-f(Q)|^{2}dQ\right\}^{\frac{1}{2}}\leq  \left\{\int\limits_{\Omega}|f(Q+\mathbf{e}t)-f_{m}(Q+\mathbf{e}t)|^{2}dQ\right\}^{\frac{1}{2}}+
 $$
 $$
 +\left\{\int\limits_{\Omega}|f(Q)-f_{m}(Q)|^{2}dQ\right\}^{\frac{1}{2}}+\left\{\int\limits_{\Omega}|f_{m}(Q)-f_{m}(Q+\mathbf{e}t)|^{2}dQ\right\}^{\frac{1}{2}}=
 $$
 $$
 =I_{1}+I_{2}+I_{3}<\varepsilon,
 $$
where $f\in L_{2}(\Omega),\,\left\{f_{n}\right\}_{1}^{\infty}\subset C_{0}^{\infty}(\Omega);$  $m$ is chosen so that $I_{1},I_{2}< \varepsilon/3 $ and $t$
is chosen so that $I_{3}< \varepsilon/3.$
Thus,  there exists such a positive  number $t_{0}$  that
$$
\|T_{t}f-f\|_{L_{2} }<\varepsilon,\,t<t_{0},
$$
for arbitrary small $\varepsilon>0.$  Using the assumption that  all functions have the zero extension   outside $\bar{\Omega},$    we have
$\|T_{t}\|  \leq 1.$ Hence  we conclude that $T_{t}$ is a $C_{0}$ semigroup of contractions (see \cite{Pasy}).
\end{proof}

\begin{lem}\label{L3.4}
Suppose  $\rho\in \mathrm{Lip} \lambda,\,\lambda>\alpha,\,0<\alpha<1;$ then
$$
\rho\cdot\mathfrak{I}^{\alpha}_{0+}(L_{2})= \mathfrak{I}^{\alpha}_{d-}(L_{2});\;\rho\cdot\mathfrak{I}^{\alpha}_{d-}(L_{2})= \mathfrak{I}^{\alpha}_{d-}(L_{2}).
$$
 \end{lem}
 \begin{proof}
  Consider an operator
  \begin{equation*}
(\psi^{+}_{  \varepsilon }f)(Q)=  \left\{ \begin{aligned}
 \int\limits_{0}^{r-\varepsilon }\frac{ f (Q)r^{n-1}- f(T)t^{n-1}}{(  r-t)^{\alpha +1}r^{n-1}}  dt,\;\varepsilon\leq r\leq d  ,\\
   \frac{ f(Q)}{\alpha} \left(\frac{1}{\varepsilon^{\alpha}}-\frac{1}{ r ^{\alpha} }    \right),\;\;\;\;\;\;\;\;\;\;\;\;\;\;\; 0\leq r <\varepsilon,\\
\end{aligned}
 \right.
\end{equation*}
where $T=P+\mathbf{e}t.$
We should prove that there exists a limit
$$
\psi^{+}_{  \varepsilon }\rho f\stackrel{L_{2}}{\longrightarrow} \psi   f,\,f\in \mathfrak{I}^{\alpha}_{0+}(L_{2}),
$$
where $\psi f$ is some function corresponding to $f.$ We have
$$
(\psi^{+}_{  \varepsilon }\rho f)(Q)=\int\limits_{0}^{r-\varepsilon }\frac{ \rho(Q)f (Q)r^{n-1}- \rho (T)f(T)t^{n-1}}{(  r-t)^{\alpha +1}r^{n-1}}  dt=
\rho(Q)\int\limits_{0}^{r-\varepsilon }\frac{ f (Q)r^{n-1}-  f(T)t^{n-1}}{(  r-t)^{\alpha +1}r^{n-1}}  dt+
$$
$$
+\int\limits_{0}^{r-\varepsilon }\frac{   f(T)[\rho(Q)-\rho(T)]}{(  r-t)^{\alpha +1}} \left(\frac{t}{r }\right)^{n-1}  \!\!\!dt=A_{\varepsilon}(Q)+B_{\varepsilon}(Q)
 ,\;\varepsilon\leq r\leq d;
$$
$$
(\psi^{+}_{  \varepsilon }\rho f)(Q)= \rho(Q)f(Q)  \frac{1}{\alpha}\left(\frac{1}{\varepsilon^{\alpha}}-\frac{1}{ r ^{\alpha} }    \right),\;  0\leq r <\varepsilon.
$$
Hence, we get
$$
\|\psi^{+}_{  \varepsilon_{n+1} }\rho f-\psi^{+}_{  \varepsilon_{n} }\rho f\|_{L_{2}(\Omega)}\leq \|\psi^{+}_{  \varepsilon_{n+1} }\rho f-\psi^{+}_{  \varepsilon_{n} }\rho f\|_{L_{2}(\Omega')}
 +\|\psi^{+}_{  \varepsilon_{n+1} }\rho f-\psi^{+}_{  \varepsilon_{n} }\rho f\|_{L_{2}(\Omega_{n})},
$$
where $\{\varepsilon_{n}\}_{1}^{\infty}\subset \mathbb{R}_{+}$ is a strictly decreasing sequence that is chosen in an arbitrary way, $\Omega_{n}:=  \omega\times \{0<r<\varepsilon_{n}\}, $
 $\Omega':=  \Omega\setminus \Omega_{n} .$
It is clear that
$$
\|A_{\varepsilon_{n+1}}-A_{\varepsilon_{n}}\|_{L_{2}(\Omega')}\leq \|\rho\|_{L_{\infty}(\Omega)}\|\psi^{+}_{  \varepsilon_{n+1} }f-\psi^{+}_{  \varepsilon_{n} }f\|_{L_{2}(\Omega')},
$$
Since in accordance with Theorem 2.3 \cite{firstab_lit:1kukushkin2018}  the sequence $ \psi^{+}_{  \varepsilon_{n} }f,\,(n=1,2,...) $ is fundamental for the defined function $f,$ with respect to the $L_{2}(\Omega)$ norm,    then the sequence  $ A_{\varepsilon_{n}} $ is also fundamental with respect to the $L_{2}(\Omega')$ norm.
 Having used the H\"{o}lder properties of $\rho,$   we have
$$
\|B_{\varepsilon_{n+1}}-B_{\varepsilon_{n}}\|_{L_{2}(\Omega')}\leq M \left\{\int\limits_{\Omega'}\left(\int\limits_{r-\varepsilon_{n}}^{r-\varepsilon_{n+1}}\frac{   |f(T)| }{(  r-t)^{\alpha +1-\lambda}} \left(\frac{t}{r }\right)^{n-1}  \!\!\!dt\right)^{2}dQ\right\}^{\frac{1}{2}}.
$$
Note that applying Theorem 2.3 \cite{firstab_lit:1kukushkin2018}, we have
$$
\left\{\int\limits_{\Omega}\left(\int\limits_{0}^{r}\frac{   |f(T)| }{(  r-t)^{\alpha +1-\lambda}} \left(\frac{t}{r }\right)^{n-1}  \!\!\!dt\right)^{2}dQ\right\}^{\frac{1}{2}}\leq C \|f\|_{L_{2}}.
$$
Hence the sequence $ \left\{B_{\varepsilon_{n}}\right\}_{1}^{\infty} $ is fundamental with respect to the $L_{2}(\Omega')$ norm.
Therefore
$$
\|\psi^{+}_{  \varepsilon_{n+1} }\rho f-\psi^{+}_{  \varepsilon_{n} }\rho f\|_{L_{2}(\Omega')}\rightarrow 0,\,n\rightarrow \infty.
$$
Consider
$$
\|\psi^{+}_{  \varepsilon_{n+1} }\rho f-\psi^{+}_{  \varepsilon_{n} }\rho f\|_{L_{2}(\Omega_{n})}\leq
\|\psi^{+}_{  \varepsilon_{n+1} }\rho f-\psi^{+}_{  \varepsilon_{n} }\rho f\|_{L_{2}(\Omega_{n+1})}+
$$
$$
+\left\{\int\limits_{\omega}d\chi  \int\limits_{\varepsilon_{n+1}}^{\varepsilon_{n}}
|A_{\varepsilon_{n+1}}(Q)+B_{\varepsilon_{n+1}}(Q)|^{2}rdr\right\}^{\frac{1}{2}}+
$$
$$
 +\frac{1}{\alpha}\left\{\int\limits_{\omega}d\chi  \int\limits_{\varepsilon_{n+1}}^{\varepsilon_{n}}
\left|\rho(Q)f(Q)  \left(\frac{1}{\varepsilon_{n}^{\alpha}}-\frac{1}{ r ^{\alpha} }    \right)\right|^{2}rdr\right\}^{\frac{1}{2}}=
I_{1}+I_{2}+I_{3}.
$$
We have
$$
I_{1}\leq \frac{1}{\alpha}\left(\frac{1}{\varepsilon_{n}^{\alpha}}-\frac{1}{\varepsilon_{n+1}^{\alpha}   }    \right)\|\rho\|_{L_{\infty}} \int\limits_{\omega}d\chi\int\limits_{0}^{\varepsilon_{n+1}}
   f(Q)r dr\leq
$$
$$
\leq\frac{1}{\alpha}\left(\frac{1}{\varepsilon_{n}^{\alpha}}-\frac{1}{\varepsilon_{n+1}^{\alpha}   }    \right)\|\rho\|_{L_{\infty}}  \int\limits_{\omega}\left\{\int\limits_{0}^{\varepsilon_{n+1}}
  |f(Q)|^{2}r dr \right\}^{\frac{1}{2}}\left\{\int\limits_{0}^{\varepsilon_{n+1}}
   r dr \right\}^{\frac{1}{2}}d\chi\leq
$$
$$
\leq\frac{1}{\sqrt{2}\alpha}\left(\frac{1}{\varepsilon_{n}^{\alpha}}-\frac{1}{\varepsilon_{n+1}^{\alpha}   }    \right) \varepsilon_{n+1}\|\rho\|_{L_{\infty}} \|f\|_{L_{2}}.
$$
Hence $I_{1}\rightarrow 0,\,n\rightarrow \infty.$
Using the  estimates used above, it is not hard to prove that $I_{2},I_{3}\rightarrow 0,\,n\rightarrow \infty.$ The proof is left to a reader.
Therefore
$$
\|\psi^{+}_{  \varepsilon_{n+1} }\rho f-\psi^{+}_{  \varepsilon_{n} }\rho f\|_{L_{2}(\Omega_{n})}\rightarrow 0,\,n\rightarrow \infty.
$$
Combining the obtained results, we have
$$
\|\psi^{+}_{  \varepsilon_{n+1} }\rho f-\psi^{+}_{  \varepsilon_{n} }\rho f\|_{L_{2}(\Omega)}\rightarrow 0,\,n\rightarrow \infty.
$$
Using  Theorem 2.2 \cite{firstab_lit:1kukushkin2018}, we obtain the desired result for the case corresponding  to the class $\mathfrak{I}^{\alpha}_{0+}(L_{2}).$ The proof corresponding to the class $\mathfrak{I}^{\alpha}_{d-}(L_{2})$ is absolutely analogous.
\end{proof}

The following theorem is formulated in terms of the infinitesimal  generator $-A$ of the semigroup $T_{t}.$

\begin{teo}\label{T3.2} We claim that  $L=Z^{\alpha}_{G,F}(A).$ Moreover  if  $ \gamma_{a} $ is sufficiently large in comparison  with $\|\rho\|_{L_{\infty}},$ then $L$ satisfies conditions H1-H2, where we put $\mathfrak{M}:=C_{0}^{\infty}(\Omega),$   if we additionally assume that $\rho \in \mathrm{Lip}\lambda,\,   \lambda>\alpha  ,$ then    $ \tilde{\mathcal{H}}=H.$
\end{teo}
\begin{proof}
     By virtue of   Corollary 3.6 \cite[p.11]{Pasy}, we have
\begin{equation} \label{3.12}
\|(\lambda+A)^{-1}\|  \leq \frac{1}{\mathrm{Re} \lambda },\,\mathrm{Re}\lambda>0.
\end{equation}
Inequality \eqref{3.12} implies that $A$ is m-accretive.
Using formula \eqref{3.7},  we can define positive fractional powers $\alpha\in (0,1)$ of the operator $A. $
Applying  the Balakrishnan formula, we obtain
 \begin{equation*}
A^{\alpha}f:=\frac{\sin\alpha \pi}{\pi}\int\limits_{0}^{\infty}\lambda^{\alpha-1}(\lambda  +A)^{-1} A f\,d \lambda=\frac{1}{\Gamma(-\alpha)}\int\limits_{0}^{\infty}\frac{T_{t}-I}{t^{\alpha+1}}fdt,\,f\in \mathrm{D}(A).
\end{equation*}
Hence, in the concrete  form of writing we have
 \begin{equation}\label{3.13}
A^{\alpha}f(Q)=\frac{1}{\Gamma(-\alpha)}\int\limits_{0}^{\infty}\frac{f(Q+\mathbf{e}t)-f(Q)}{t^{\alpha+1}}dt=
$$
$$
=\frac{\alpha}{\Gamma(1-\alpha)}\int\limits_{r}^{d   (\mathbf{e})  }\frac{f(Q)-f(P+\mathbf{e}t)}{(t-r)^{\alpha+1}}dt+ \frac{f (Q)}{\Gamma(1-\alpha)} \{d(\mathbf{e})-r\}^{-\alpha} =
\mathfrak{D}^{\alpha}_{d-}f(Q),\ f\in  \mathrm{D}  (A),
\end{equation}
where $d(\mathbf{e})$ is the distance from the point $P$ to the edge of $\Omega$ along the direction $\mathbf{e}.$ Note  that a relation  between positive  fractional powers of the operator $A$    and the  Riemann-Liouville  fractional derivative was demonstrated  in the one-dimensional case   in the paper    \cite{firstab_lit:1Ashyralyev}.

 Consider a restriction   $A_{0}\subset A,\,\mathrm{D}(A_{0})=C^{\infty}_{0}( \Omega )$ of the operator  $A.$ Note that,   since  the infinitesimal  generator  $-A$ is a closed operator (see \cite{Pasy}), then $A_{0}$  is closeable.
It is not hard to prove that $ \tilde{A}_{0}$ is an m-accretive operator. For this purpose, note that   since the operator $A$ is  m-accretive, then  by virtue of  \eqref{3.4}, we get
\begin{equation*} \mathrm{Re}(\tilde{A}_{0} f,f)_{ \mathfrak{H}}\geq0 ,\,f\in \mathrm{D}(\tilde{A}_{0}).
\end{equation*}
This gives us an opportunity to conclude that
$$
   \|f\|^{2}_{\mathfrak{H}}\leq  \frac{1}{ t^{2}} \left\{ \| \tilde{A}_{0} f\|^{2}_{ \mathfrak{H}}+2t \mathrm{Re}(\tilde{A}_{0}f,f)_{ \mathfrak{H}}+t^{2}\|   f \|^{2}_{ \mathfrak{H}}\right\}  ;\,\|f\|^{2}_{\mathfrak{H}}\leq  \frac{1}{t^{2}}  \|(\tilde{A}_{0}+t)f \|^{2}_{ \mathfrak{H}},\,t>0.
$$
Therefore
\begin{equation*}
\|(\tilde{A}_{0}+t)^{-1}\|_{\mathrm{R} \rightarrow \mathfrak{H}}\leq\frac{1}{ t},\,t>0,
\end{equation*}
where   $\mathrm{R}:=\mathrm{R}(\tilde{A}_{0}+t).$ Hence, in accordance with Lemma \ref{L3.1}, we obtain that the operator $\tilde{A}_{0}$ is m-accretive. Since there does not exist an accretive extension  of an m-accretive operator (see \cite[p.279]{firstab_lit:kato1980} ) and $\tilde{A}_{0}\subset A,$ then $\tilde{A}_{0}= A.$
It is easy to prove  that
\begin{equation}\label{3.14}
\|Af\|_{L_{2}}\leq C\|f\|_{H_{0}^{1}},\,f\in H_{0}^{1}(\Omega),
\end{equation}
 for this purpose we should establish  a representation $Af(Q) =-(\nabla f ,\mathbf{e})_{\mathbb{E}^{n}}, f\in C^{\infty}_{0}(\Omega)$  the rest of the proof   is  left to a reader.
Thus, we get   $H_{0}^{1}(\Omega) \subset \mathrm{D}(A),$ and as a result $A^{\alpha}f = \mathfrak{D}^{\alpha}_{d-}f ,\ f\in  H_{0}^{1}(\Omega) .$
Let us find a representation for the  operator $G.$
Consider an operator
 $$
 Bf(Q)=\!\int_{0}^{r}\!\!f(P+\mathbf{e}[r-t])dt,\,f\in L_{2}(\Omega).
 $$It is not hard to prove that  $B\in \mathcal{B}(L_{2}),$   applying the generalized Minkowski inequality, we get
 $$
 \|Bf\|_{L_{2} }\leq \int\limits_{0}^{\mathrm{diam\,\Omega}}dt  \left(\int\limits_{\Omega}|f(P+\mathbf{e}[r-t])|dQ\right)^{1/2}\leq C\|f\|_{L_{2} }.
 $$
    The fact $A^{-1}_{0}\subset B$   follows from properties of the  one-dimensional integral defined on smooth functions.
It is  a  well-known fact (see Theorem 2 \cite[p.555]{firstab_lit:Smirnov5}) that  since  $A_{0}$ is closeable and there exists a bounded operator $ A^{-1}_{0},$ then there exists a bounded operator $A^{-1}=\tilde{A}^{-1}_{0}=  \widetilde{A ^{-1}_{0}}  .$ Using  this relation we conclude that $A^{-1}\subset B.$ It is obvious that
\begin{equation}\label{3.15}
\int\limits_{\Omega }  A\left(B  \mathcal{T}  f \cdot g\right) dQ =\int\limits_{\Omega } AB\mathcal{T}f \cdot g\,  dQ +\int\limits_{\Omega }  B \mathcal{T} f \cdot Ag  \,dQ ,\, f\in C^{2}(\bar{\Omega}),g\in C^{\infty}_{0}( \Omega ).
\end{equation}
Using the   divergence  theorem, we get
\begin{equation}\label{3.16}
\int\limits_{\Omega }  A\left(B\mathcal{T}f \cdot g\right) \,dQ=\int\limits_{S}(\mathbf{e},\mathbf{n})_{\mathbb{E}^{n}}(B\mathcal{T}f\cdot  g)(\sigma)d\sigma,
\end{equation}
where $S$ is the surface of $\Omega.$
Taking into account   that $ g(S)=0$ and combining   \eqref{3.15},\eqref{3.16}, we get
\begin{equation}\label{3.17}
    -\int\limits_{\Omega }  AB\mathcal{T} f\cdot \bar{g} \, dQ=  \int\limits_{\Omega } B\mathcal{T} f\cdot   \overline{A  g} \, dQ,\, f\in C^{2}(\bar{\Omega}),g\in C^{\infty}_{0}( \Omega ).
\end{equation}
Suppose that $f\in H^{2}(\Omega),$ then    there exists a sequence $\{f_{n}\}_{1}^{\infty}\subset C^{2}(\bar{\Omega})$ such that
$ f_{n}\stackrel{ H^{2}}{\longrightarrow}  f$ (see \cite[p.346]{firstab_lit:Smirnov5}).  Using this fact, it is not hard to prove that
$\mathcal{T}f_{n}\stackrel{L_{2}}{\longrightarrow} \mathcal{T}f.$  Therefore  $AB\mathcal{T}f_{n}\stackrel{L_{2}}{\longrightarrow} \mathcal{T}f,$  since  $AB\mathcal{T}f_{n}=\mathcal{T}f_{n}.$ It is also clear that   $B\mathcal{T}f_{n}\stackrel{L_{2}}{\longrightarrow} B\mathcal{T}f,$ since $B$ is continuous.
Using these facts, we can extend relation \eqref{3.17} to the following
\begin{equation}\label{3.18}
  -\int\limits_{\Omega }  \mathcal{T} f \cdot  \bar{g} \, dQ= \int\limits_{\Omega } B\mathcal{T}f\,  \overline{Ag} \, dQ,\; f\in \mathrm{D}(L),\,g\in C_{0}^{\infty}(\Omega).
\end{equation}
It was previously proved that $H_{0}^{1}(\Omega) \subset \mathrm{D}(A),\, A^{-1} \subset B.$ Hence $G Af=B\mathcal{T} f,\, f\in  \mathrm{D}(L),$  where $G:=B\mathcal{T}B.$  Using  this fact    we can rewrite relation \eqref{3.18} in a   form
\begin{equation}\label{3.19}
  -\int\limits_{\Omega }  \mathcal{T} f \cdot  \bar{g} \, dQ= \int\limits_{\Omega } G Af\,  \overline{Ag} \, dQ,\; f\in \mathrm{D}(L),\,g\in C_{0}^{\infty}(\Omega).
\end{equation}
Note that   in accordance with the fact $A=\tilde{A}_{0},$ we have
$$
\forall g\in \mathrm{D}(A),\,\exists \{g_{n}\}_{1}^{\infty}\subset C^{\infty}_{0}( \Omega ),\,    g_{n}\xrightarrow[      A       ]{}g.
$$
Therefore, we can extend   relation \eqref{3.19} to the following
 \begin{equation}\label{3.20}
     -\int\limits_{\Omega }  \mathcal{T} f \cdot  \bar{g} \, dQ= \int\limits_{\Omega } G Af \, \overline{Ag} \, dQ,\; f\in \mathrm{D}(L),\,g\in \mathrm{D}(A).
\end{equation}
Relation \eqref{3.20} indicates that $G Af\in \mathrm{D}( A ^{\ast})$   and it is clear that $  -\mathcal{T}\subset A ^{\ast}GA.$ On the other hand in accordance with  Chapter VI, Theorem 1.2    \cite{firstab_lit: Berezansk1968}, we have that $-\mathcal{T}$ is a closed operator, hence in accordance with Lemma \ref{L3.1} the operator  $-\mathcal{T}$ is m-accretive. Therefore $-\mathcal{T}= A  ^{\ast}GA,$   since $A ^{\ast}GA$ is accretive.  Note that  by virtue of  Theorem 2.1 \cite{firstab_lit:1kukushkin2018}, we have     $(\mathfrak{I}^{\sigma }_{0+}\rho\, \cdot)\in \mathcal{B}(L_{2}).$
    It was previously proved that $\mathfrak{D}^{\alpha}_{d-}f= A^{\alpha}f,\,f\in H^{1}_{0}(\Omega).$         Thus, the  representation  $L=Z^{\alpha}_{GF}(A),$     where $G:=B\mathcal{T}B,\,F:=(\mathfrak{I}^{\sigma }_{0+}\rho\,\cdot)$   has been established.

 Let us prove that the operator $L$ satisfy   conditions H1--H2.    Choose the space  $L_{2}(\Omega)$ as a space $\mathfrak{H},$ the set  $C_{0}^{\infty}(\Omega)$ as a linear  manifold $\mathfrak{M},$ and the space  $H^{1}_{0}(\Omega)$ as a space $\mathfrak{H}_{+}.$ By virtue of the  Rellich-Kondrashov theorem, we have  $H_{0}^{1}(\Omega)\subset\subset L_{2}(\Omega).$  Thus, condition  H1  is fulfilled.
Using   simple  reasonings, we come to the following inequality
 \begin{equation*}
\left|\int\limits_{\Omega }  \mathcal{T} f \cdot  \bar{g} \, dQ\right|\leq C\|f\|_{H_{0}^{1}}\|g\|_{H_{0}^{1}},\; f,g\in C_{0}^{\infty}(\Omega).
\end{equation*}
     Let us prove that
\begin{equation}\label{3.21}
 |\left( \mathfrak{I}^{\sigma }_{0+}\rho\,\mathfrak{D}^{\alpha}_{d-}  f,g\right)_{ L _{2 }}| \leq K\|f\|_{H_{0}^{1}}\|g\|_{L_{2}},\,f,g\in C^{\infty}_{0}(\Omega),
\end{equation}
where $K=C\|\rho\|_{L_{\infty}}.$
Using  a fact that    the operator $(\mathfrak{I}^{\sigma }_{0+}\rho\,\cdot)$ is bounded, we obtain
\begin{equation}\label{3.22}
\|\mathfrak{I}^{\sigma }_{0+}\rho\,\mathfrak{D}^{\alpha}_{d-}  f\|_{L_{2}}\leq C \|\rho\|_{L_{\infty}} \| \mathfrak{D}^{\alpha}_{d-}  f\|_{L_{2}},\, f\in C_{0}^{\infty}(\Omega).  \end{equation}
Taking into account that $ A ^{-1} $ is bounded, $A$ is m-accretive, applying  Lemma \ref{L3.2} analogously to  \eqref{3.11}, we conclude that
$
\|A^{\alpha} f\|_{L_{2}}\leq C \|Af\|_{L_{2}},\,f\in \mathrm{D}(A).
$
Using \eqref{3.13},\eqref{3.14}, we get
$
\| \mathfrak{D}^{\alpha}_{d-}  f\|_{L_{2}}\leq C\|f\|_{H_{0}^{1}},\,f\in C_{0}^{\infty}(\Omega).
$
Combining this relation with \eqref{3.22}, we obtain
$$
\|\mathfrak{I}^{\sigma }_{0+}\rho\,\mathfrak{D}^{\alpha}_{d-}  f\|_{L_{2}}\leq K \|f\|_{H_{0}^{1}},\,f\in C_{0}^{\infty}(\Omega).
$$
Using this inequality, we can easily obtain \eqref{3.21}, from what follows that
$$
\mathrm{Re}(\mathfrak{I}^{\sigma }_{0+}\rho\,\mathfrak{D}^{\alpha}_{d-}  f,f)_{L_{2}}\geq -K\|f\|^{2}_{H^{1}_{0}},\, f \in C^{\infty}_{0}(\Omega).
$$
On the other hand, using a uniformly elliptic property of the operator $\mathcal{T}$ it is not hard to prove that
\begin{equation*}
-\mathrm{Re}(\mathcal{T} f,f)\geq \gamma_{a}\|f\|_{H_{0}^{1}},\,f\in C^{\infty}_{0}(\Omega),
\end{equation*}
the proof of this fact  is obvious and left to a reader (see  \cite{firstab_lit:1kukushkin2018}). Now, if we assume that $\gamma_{a}>K,$ then we obtain the fulfillment of condition
H2.

  Assume additionally that $\rho \in \mathrm{Lip}\lambda,\,\lambda>\alpha,$ let us prove that  $C_{0}^{\infty}(\Omega)\subset \mathrm{D}(L^{\ast}).$
Note that
 $$
 \int\limits_{\Omega}D_{j} ( a^{ij} D_{i}f)\,gdQ=\int\limits_{\Omega}f\,\overline{ D_{j} ( a^{ji} D_{i}g)}dQ,\,f\in \mathrm{D}(L),\, g\in C_{0}^{\infty}(\Omega).
 $$
Using this equality, we conclude that $(-\mathcal{T})^{\ast}$ is defined on $C_{0}^{\infty}(\Omega).$ Applying   the Fubini theorem,   Lemma \ref{L3.4}, Lemma 2.6 \cite{firstab_lit:1kukushkin2018},     we get
\begin{equation*}
 \left( \mathfrak{I}^{\sigma}_{0+}\rho\,\mathfrak{D}^{\alpha}_{d-}  f,g\right)_{L_{2}}= \left( \mathfrak{D}^{\alpha}_{d-}  f,\rho\,\mathfrak{I}^{\sigma}_{d-}g\right)_{L_{2}}= \left(   f,\mathfrak{D}^{\alpha}_{0+}\rho\,\mathfrak{I}^{\sigma}_{d-}g\right)_{L_{2}}\!\!, \,f\in \mathrm{D}(L),\, g\in C_{0}^{\infty}(\Omega).
\end{equation*}
Therefore the operator
$
\left(\mathfrak{I}^{\sigma}_{0+}\rho\,\mathfrak{D}^{\alpha}_{d-}\right)^{\ast}
$
is defined on $C_{0}^{\infty}(\Omega).$
Taking into account the above reasonings, we conclude that    $ C_{0}^{\infty}(\Omega) \subset \mathrm{D}( L^{\ast} ).$ Combining  this fact with relation H2, we obtain     $\tilde{\mathcal{H}}=H$  see paragraph  \ref{S2.7}.
\end{proof}

\begin{corol}\label{C3.1}
 Consider  a one-dimensional case,     we claim that  $L\in \mathfrak{G_{\alpha}}.$
\end{corol}
\begin{proof}
It is not hard to prove   that $ \| A_{0}f \|_{L_{2}}=\|f\|_{H_{0}^{1}},\,f\in C_{0}^{\infty}(\Omega).$ This relation can be extended to the following
\begin{equation*}
 \| A f \|_{L_{2}}=\|f\|_{H_{0}^{1}},\,f\in H_{0}^{1}(\Omega),
\end{equation*}
whence $\mathrm{D}(A)=H_{0}^{1}(\Omega).$
 Taking into account  the Rellich-Kondrashov theorem, we conclude that  $ A ^{-1}$ is compact. Thus, to show that conditions of Theorem \ref{T3.1} are fulfilled  we need  prove  that  the  operator $G:=B\mathcal{T}B$ is bounded and $\mathrm{R}( A )\subset \mathrm{D}(G).$ We can establish  the following relation  by direct calculations
$
GA_{0}f =B\mathcal{T}f =a^{11} A_{0}f ,\,f\in C_{0}^{\infty}(\Omega),
$
where $a^{11}=a^{ij},\;i,j=1.$ Using this equality, we can easily prove that
$
\|G  A f\|_{L_{2}}\leq C\| A f\|_{L_{2}},\, f\in \mathrm{D}( A ).
$
Thus, we obtain the desired  result.
\end{proof}

\subsection{ Riesz potential}

 Consider a   space $L_{2}(\Omega),\,\Omega:=(-\infty,\infty).$      We denote by $H^{2,\,\lambda}_{0}(\Omega)$ the completion of the set  $C^{\infty}_{0}(\Omega)$  with the norm
$$
\|f\|_{H^{2,\lambda}_{0}}=\left\{\|f\|^{2}_{L_{2}(\Omega) }+\|f''\|^{2}_{L_{2}(\Omega,\omega^{\lambda})} \right\}^{1/2},\,  \lambda\in \mathbb{R},
$$
where $\omega(x):=  1+|x|.$
Let us notice the following fact  (see Theorem 1 \cite{firstab_lit:1Adams}), if $\lambda>4,$ then
$
H^{2,\,\lambda}_{0}(\Omega)\subset\subset L_{2}(\Omega).
$
Consider a Riesz potential
$$
I^{\alpha}f(x)=B_{\alpha}\int\limits_{-\infty}^{\infty}f (s)|s-x|^{\alpha-1} ds,\,B_{\alpha}=\frac{1}{2\Gamma(\alpha)  \cos  \alpha \pi / 2   },\,\alpha\in (0,1),
$$
where $f$ is in $L_{p}(-\infty,\infty),\,1\leq p<1/\alpha.$
It is  obvious that
$
I^{\alpha}f= B_{\alpha}\Gamma(\alpha) (I^{\alpha}_{+}f+I^{\alpha}_{-}f),
$
where
$$
I^{\alpha}_{\pm}f(x)=\frac{1}{\Gamma(\alpha)}\int\limits_{0}^{\infty}f (s\pm x) s ^{\alpha-1} ds,
$$
the last operators are known as fractional integrals on a whole  real axis   (see \cite[p.94]{firstab_lit:samko1987}). Assume that the following  condition holds
 $ \sigma/2 + 3/4<\alpha<1 ,$ where $\sigma$ is a non-negative  constant. Following the idea of the   monograph \cite[p.176]{firstab_lit:samko1987}
 consider a sum of a differential operator and  a composition of    fractional integro-differential operators
$$
 L    := \tilde{\mathcal{T}}   +I^{\sigma}_{+}\,\rho \,  I^{2(1-\alpha)}\frac{d^{2}}{dx^{2}}  \,,
 $$
where  $$
\mathcal{T} := \frac{d^{2}}{dx^{2}}\left(a  \frac{d^{2}}{dx^{2}}  \cdot \right) ,\, \mathrm{D}(\mathcal{T})=C^{\infty}_{0}(\Omega) ,
$$
$$
\, \rho(x)\in L_{\infty}(\Omega) ,\,a(x)\in L_{\infty}(\Omega)\cap C^{ 2 }( \Omega ),\, \mathrm{Re}\,a(x) >\gamma_{a}(1+|x|)^{5},\,\gamma_{a}>0.
$$
Consider a family of operators
$$
T_{t}f(x)=(2\pi t)^{-1/2}\int\limits_{-\infty}^{\infty}e^{-(x-\tau)^{2}/2t}f(\tau)d\tau,\;t>0,\;
 T_{t}f(x)=f(x),\;t=0,\,f\in L_{2}(\Omega).
$$
\begin{lem}\label{L3.5}
$T_{t}$ is a $C_{0}$ semigroup of contractions.
\end{lem}
\begin{proof}
Let us establish the semigroup property,  by definition  we have $T_{0}=I.$ Consider the following formula, note that the interchange of the integration order can be easily substantiated
$$
T_{t}T_{t'}f(x)=\frac{1}{\sqrt{2\pi t}\sqrt{2\pi t'}}
\int\limits_{-\infty}^{\infty} e^{-\frac{(x-u)^{2}}{2t}}du\int\limits_{-\infty}^{\infty} e^{-\frac{(u-\tau)^{2}}{2t}}f(\tau)d\tau=
$$
$$
 =\frac{1}{\sqrt{2\pi t}\sqrt{2\pi t'}}
\int\limits_{-\infty}^{\infty} f(\tau)d\tau \int\limits_{-\infty}^{\infty}e^{-\frac{(x-u)^{2}}{2t}} e^{-\frac{(u-\tau)^{2}}{2t}}  du=\frac{1}{\sqrt{2\pi t}\sqrt{2\pi t'}}
\int\limits_{-\infty}^{\infty} f(\tau)d\tau \int\limits_{-\infty}^{\infty}e^{-\frac{(x-v-\tau)^{2}}{2t}} e^{-\frac{  v ^{2}}{2t}}  dv.
$$
On the other hand, in accordance with the formula \cite[p.325]{firstab_lit:Yosida}, we have
$$
\frac{1}{\sqrt{2\pi (t+t')} }e^{-\frac{(x-\tau)^{2}}{2t}}=\frac{1}{\sqrt{2\pi t}\sqrt{2\pi t'}}\int\limits_{-\infty}^{\infty}e^{-\frac{(x-\tau-v)^{2}}{2t}} e^{-\frac{ v ^{2}}{2t}}  dv.
$$
Hence
$$
\frac{1}{\sqrt{2\pi (t+t')} }\int\limits_{-\infty}^{\infty}e^{-\frac{(x-\tau)^{2}}{2t}}f(\tau) d\tau=\frac{1}{\sqrt{2\pi t}\sqrt{2\pi t'}}
\int\limits_{-\infty}^{\infty} f(\tau)d\tau \int\limits_{-\infty}^{\infty}e^{-\frac{(x-v-\tau)^{2}}{2t}} e^{-\frac{  v ^{2}}{2t}}  dv,
$$
from what immediately  follows  the fact   $T_{t}T_{t'}f=T_{t+t'}f.$
Let us show that $T_{t}$ is a $C_{0}$ semigroup of contractions.
Observe that
$$
(2\pi t)^{-1/2}\int\limits_{-\infty}^{\infty}e^{- \tau ^{2}/2t} d\tau =1.
$$
Therefore,  using the generalized Minkowski inequality (see (1.33) \cite[p.9]{firstab_lit:samko1987}), we get
$$
\|T_{t}f\|_{L_{2}}=\left(\int\limits_{-\infty}^{\infty} \left|\int\limits_{-\infty}^{\infty}f(x+s) N_{t}(s)ds  \right|^{2}dx \right)^{1/2}  \leq
  $$
  $$
  \leq\int\limits_{-\infty}^{\infty} N_{t}(s)ds \left(\int\limits_{-\infty}^{\infty}\left|f(x+s)   \right|^{2}dx \right)^{1/2}=\|f\|_{L_{2}},\,f\in C_{0}^{\infty}(\Omega),
$$
where $N_{t}(x):=(2\pi t)^{-1/2}e^{-x^{2}/2t}.$
It is clear that the last inequality can be extended to $L_{2}(\Omega),$ since $C_{0}^{\infty}(\Omega)$ is dense in $L_{2}(\Omega).$
Thus, we conclude that $T_{t}$ is a   semigroup of contractions.

Let us establish a strongly continuous property.  Assuming that  $z=(x-\tau)/\sqrt{t},$  we get in an obvious way
$$
\|T_{t}f-f\|_{L_{2}}=\left(\int\limits_{-\infty}^{\infty}\left|\int\limits_{-\infty}^{\infty}N_{1}(z)\left[ f(x-\sqrt{t}z)-f( x)\right] dz \right|^{2} dx\right)^{1/2}\!\!\!\leq
$$
$$
\leq\int\limits_{-\infty}^{\infty}N_{1}(z)\left(\int\limits_{-\infty}^{\infty}\left[ f(x-\sqrt{t}z)-f(x)\right]^{2} dx \right)^{1/2} \!\!\!dz,\, f\in L_{2}(\Omega),
$$
where $N_{1}=N_{t}|_{t=1}.$
Observe that, for arbitrary fixed $t,z$ we have
$$
N_{1}(z)\left(\int\limits_{-\infty}^{\infty}\left[ f(x-\sqrt{t}z)-f(x)\right]^{2} dx \right)^{1/2}\!\!\!\leq
 $$
 $$
 \leq N_{1}(z)\left(\int\limits_{-\infty}^{\infty}\left[ f(x-\sqrt{t}z) \right]^{2} dx \right)^{1/2}+N_{1}(z)\|f\|_{L_{2}}\leq 2 N_{1}(z)\|f\|_{L_{2}}.
$$
Applying the Fatou--Lebesgue theorem, we get
$$
 \overline{\lim\limits_{ t\rightarrow 0}}\int\limits_{\infty}^{\infty}N_{1}(z)\left(\int\limits_{\infty}^{\infty}\left[ f(x-\sqrt{t}z)-f(x)\right]^{2} dx \right)^{1/2} \!\!\!dz\leq
 \!\!\int\limits_{\infty}^{\infty}N_{1}(z)\,\overline{\lim\limits_{ t\rightarrow 0}}\left(\int\limits_{\infty}^{\infty}\left[ f(x-\sqrt{t}z)-f(x)\right]^{2} dx \right)^{\!\!1/2}\!\!\!\!=0,
$$
from what follows that $\|T_{t}f-f\|_{L_{2}}\rightarrow 0,\;t\rightarrow  0.$
Hence $T_{t}$ is a $C_{0}$ semigroup of contractions.

\end{proof}

 The following theorem is formulated in terms of     the infinitesimal generator   $-A$ of the semigroup $T_{t}.$

\begin{teo}\label{T3.3}  We claim that $L =Z^{\alpha}_{G,F}(A).$ Moreover,  if    $\min\{\gamma_{a},\delta\},\,(\delta>0)$ is  sufficiently large in comparison with   $\|\rho\|_{L_{\infty}},$ then
a  perturbation $L+\delta I$ satisfies  conditions  H1-H2,  where we put $\mathfrak{M}:=C_{0}^{\infty}(\Omega).$
 \end{teo}
\begin{proof}
Let us prove that
\begin{equation*}
Af =-\frac{1}{2}\frac{d^{\,2}f }{dx^{2}}\;\mathrm{a.e.},\,f\in \mathrm{D}(A).
\end{equation*}
Consider an operator $J_{n}=n(nI +A)^{-1}.$ It is clear that $AJ_{n}=n(I-J_{n}).$ Using the formula
$$
(nI+A)^{-1}f=\int\limits_{0}^{\infty}e^{-n t}T_{t}f dt,\,n>0,\,f\in L_{2}(\Omega),
$$
 we   easily obtain
$$
J_{n}f(x)=\frac{n}{\sqrt{2\pi }}\int\limits_{0}^{\infty}e^{-n t} t^{-1/2}  dt
\int\limits_{-\infty}^{\infty} e^{-\frac{(x-\tau)^{2}}{2t}}f(\tau)d\tau=\frac{n}{\sqrt{2\pi  }}\int\limits_{-\infty}^{\infty}f(\tau)d\tau
\int\limits_{0}^{\infty}   e^{-nt-\frac{(x-\tau)^{2}}{2t}}t^{-1/2} dt=
$$
$$
 =  \sqrt{\frac{2n}{\pi}} \int\limits_{-\infty}^{\infty}f(\tau)d\tau
\int\limits_{0}^{\infty}   e^{-\sigma^{2}-n\frac{(x-\tau)^{2}}{2\sigma^{2}}}d\sigma,\;t= \sigma^{2}/n.
$$
Applying  the  following   formula (see (3)  \cite[p.336]{firstab_lit:Yosida})
\begin{equation}\label{3.23}
\int\limits_{0}^{\infty}   e^{-(\sigma^{2}+  c^{2}/  \sigma^{2} )}d\sigma=\frac{\sqrt{\pi}}{2}e^{-2 |c| },
\end{equation}
we obtain
$$
 J_{n}f(x)=\sqrt{\frac{ n}{2}} \int\limits_{-\infty}^{\infty}f(\tau) e^{- \sqrt{2n}|x-\tau| }  d\tau = \sqrt{\frac{ n}{2}} \int\limits_{-\infty}^{x}f(\tau) e^{- \sqrt{2n} (x-\tau)  }  d\tau+\sqrt{\frac{ n}{2}} \int\limits_{x}^{\infty}f(\tau) e^{- \sqrt{2n} (\tau-x) }  d\tau=
 $$
 $$
=\sqrt{\frac{ n}{2}}e^{- \sqrt{2n} x} \int\limits_{-\infty}^{x}f(\tau) e^{ \sqrt{2n}   \tau   }  d\tau+\sqrt{\frac{ n}{2}}e^{ \sqrt{2n} x} \int\limits_{x}^{\infty}f(\tau) e^{ -\sqrt{2n}   \tau   }  d\tau,\,f\in L_{2}(\Omega).
 $$
 Consider
  $$
  I_{1}(x)= \int\limits_{-\infty}^{x}f(\tau) e^{ \sqrt{2n} \tau }  d\tau,\;I_{2}(x)=\int\limits_{x}^{\infty}f(\tau) e^{ -\sqrt{2n}   \tau   }  d\tau.
  $$
 Observe that the functions $ f(x) e^{ \sqrt{2n} x },\, f(x) e^{ -\sqrt{2n}x} $ have the same Lebesgue points, then in accordance with the known fact, we have
$ I'_{1}(x)=f(x) e^{ \sqrt{2n} x },\;I'_{2}(x)=-f(x) e^{ -\sqrt{2n}x},$ where $x$ is a  Lebesgue point. Using this result, we get
$$
 \left(J_{n}f(x)\right)'  =-n\int\limits_{-\infty}^{x}f(\tau) e^{- \sqrt{2n} ( x-\tau) }  d\tau+n   \int\limits_{x}^{\infty}f(\tau) e^{- \sqrt{2n} ( \tau-x)  }  d\tau \;\mathrm{a.e.}
 $$
  Analogously, we have almost everywhere
$$
 \left(J_{n}f(x)\right)''  =n  \left\{  \sqrt{2n}\int\limits_{-\infty}^{x}f(\tau) e^{- \sqrt{2n} (x-\tau) }  d\tau+\sqrt{2n}\int\limits_{x}^{\infty}f(\tau) e^{- \sqrt{2n} (\tau-x)  }  d\tau   -2f(x)   \right\}=
 $$
$$
=2n(J_{n}-I)f(x)=-2AJ_{n}f(x),
$$
  taking into account the fact  $\mathrm{R}(J_{n})=\mathrm{R}(R_{A}(n))=\mathrm{D}(A),$  we obtain the desired result.

In accordance with the reasonings of  \cite[p.336]{firstab_lit:Yosida}, we have $C(\Omega)\subset \mathrm{D}(A).$ Denote by $A_{0}$ a restriction of $A$ on $C^{\infty}_{0}(\Omega).$ Using Lemma \ref{L3.1}, we conclude that $\tilde{A}_{0}=A,$ since there does not exist an accretive extension of an m-accretive operator.  Now, it is clear that
\begin{equation}\label{3.24}
\| A f\|_{L_{2}}\leq \|f\|_{H^{2,\,5}_{0} },\,f\in H^{2,\,5}_{0}( \Omega),
\end{equation}
whence $H^{2,\,5}_{0}( \Omega)\subset \mathrm{D}( A).$
Let us establish the representation $L= Z^{\alpha}_{G,F}(J).$  Since     the operator  $A$ is m-accretive, then using formula  \eqref{3.7},  we can define positive fractional powers $\alpha\in (0,1)$ of the operator $ A . $
       Applying  the   relations  obtained  above, we can calculate
\begin{equation}\label{3.25}
(\lambda I+ A )^{-1} A f(x)=-\frac{1}{2\sqrt{2\pi  }}\int\limits_{0}^{\infty}e^{-\lambda t}  t^{-1/2} dt
\int\limits_{-\infty}^{\infty} e^{-\frac{(x-\tau)^{2}}{2t}}f''(\tau)d\tau=
$$
$$
=-\frac{1}{2\sqrt{2\pi  }}\int\limits_{-\infty}^{\infty}f''(\tau)d\tau
\int\limits_{0}^{\infty}   e^{-\lambda t-\frac{(x-\tau)^{2}}{2t}}t^{-1/2}dt,\;f\in C_{0}^{\infty}(\Omega).
\end{equation}
Here,   substantiation of the interchange of the integration order  can be easily obtained due to the  properties of the function. We have for arbitrary chosen $x,\lambda$
$$
\int\limits_{-A}^{A}f''(\tau)d\tau
\int\limits_{0}^{\infty}   e^{-\lambda t-\frac{(x-\tau)^{2}}{2t}}t^{-1/2}dt=\int\limits_{-A-x}^{A-x}f''(x+s)ds
\int\limits_{0}^{1}   e^{-\lambda t- s^{2}/2t }t^{-1/2}dt+
$$
$$
+\int\limits_{-A-x}^{A-x}f''(x+s)ds
\int\limits_{1}^{\infty}   e^{-\lambda t- s^{2}/2t }t^{-1/2}dt.
$$
Observe that  the inner integrals converge uniformly with respect to $s,$ it is also clear that the function under the integrals is continuous regarding to $s,t,$ except of the set of points $(s;t_{0}),\,t_{0}=0.$ Hence applying the well-known theorem of calculus, we obtain \eqref{3.25}. Consider

$$
\int\limits_{-\infty}^{\infty}f''(x+s)ds
\int\limits_{0}^{\infty}   e^{-\lambda t- s^{2}/2t }t^{-1/2}dt =2\lambda^{-1/2}\int\limits_{-\infty}^{\infty}f''(x+s)ds
\int\limits_{0}^{\infty}   e^{-\sigma^{2}- c^{2}/\sigma^{2} } d\sigma=I,
$$
where $c^{2}=s^{2}\lambda/2.$ Using   formula \eqref{3.23}, we obtain
$$
I=\sqrt{\pi}\lambda^{-1/2} \int\limits_{-\infty}^{\infty}f''(x+s) e^{-\sqrt{2\lambda}|s|}ds=\sqrt{\pi}\lambda^{-1/2} \int\limits_{0}^{\infty}f''(x+s) e^{-\sqrt{2\lambda} s }ds+
  \sqrt{\pi}\lambda^{-1/2} \int\limits_{0}^{\infty}f''(x-s) e^{-\sqrt{2\lambda}  s }ds.
  $$
  Thus, combining formulas \eqref{3.7},\eqref{3.25},  we conclude that
$$
 A  ^{\alpha}f(x) = - \frac{2^{-3/2} }{\Gamma(1-\alpha)\Gamma(\alpha)} \int\limits_{0}^{\infty}\lambda^{\alpha -3/2}d\lambda\int\limits_{-\infty}^{\infty}f''(x+s) e^{-\sqrt{2\lambda}|s|}ds,\,f\in C_{0}^{\infty}(\Omega).
$$
We  easily prove that
\begin{equation}\label{3.26}
\int\limits_{\varepsilon}^{\infty}f''(x+s)ds\int\limits_{0}^{\infty}\lambda^{\alpha -3/2}e^{-\sqrt{2\lambda} s }d\lambda =\int\limits_{0}^{\infty}\lambda^{\alpha -3/2}d\lambda\int\limits_{\varepsilon}^{\infty}f''(x+s)e^{-\sqrt{2\lambda} s }ds,\,f\in C_{0}^{\infty}(\Omega).
\end{equation}
Let us show that
\begin{equation}\label{3.27}
\int\limits_{0}^{\infty}\lambda^{\alpha -3/2}d\lambda\int\limits_{\varepsilon}^{\infty}f''(x+s)e^{-\sqrt{2\lambda} s }ds\rightarrow \int\limits_{0}^{\infty}\lambda^{\alpha -3/2}d\lambda\int\limits_{0}^{\infty}f''(x+s)e^{-\sqrt{2\lambda} s }ds,\,\varepsilon \rightarrow 0,
\end{equation}
we have
$$
\left|\int\limits_{0}^{\infty}\lambda^{\alpha -3/2}d\lambda\int\limits_{0}^{\varepsilon}f''(x+s)e^{-\sqrt{2\lambda} s }ds\right|\leq\|f''\|_{L_{\infty}}\int\limits_{0}^{\infty}\lambda^{\alpha -3/2}d\lambda\int\limits_{0}^{\varepsilon} e^{-\sqrt{2\lambda} s }ds=
$$
$$
= \frac{1}{\sqrt{2}}\|f''\|_{L_{\infty}}\int\limits_{0}^{\infty}\lambda^{\alpha - 2} \left( 1-e^{- \sqrt{2\lambda} \varepsilon }\right)d\lambda =
$$
$$
=
 \varepsilon^{2(1-\alpha)}2^{3/2-\alpha}\|f''\|_{L_{\infty}}\int\limits_{0}^{\infty}t^{2\alpha - 3} \left( 1-e^{- t }\right)dt
 \rightarrow 0,\,\varepsilon\rightarrow0,
$$
from what follows the desired result.
Using simple calculations, we get
\begin{equation}\label{3.28}
\int\limits_{0}^{\varepsilon}f''(x+s) ds\int\limits_{0}^{\infty}e^{-\sqrt{2\lambda} s }\lambda^{\alpha -3/2}d\lambda=
 $$
 $$
 =2^{3/2-\alpha}\Gamma(2\alpha-1) \int\limits_{0}^{\varepsilon}f''(x+s)s^{1-2\alpha} ds \leq C\|f''\|_{L_{\infty} }\varepsilon^{2(1-\alpha)}\rightarrow 0,\,\varepsilon\rightarrow0.
\end{equation}
In accordance with \eqref{3.26}, we can write
$$
\int\limits_{0}^{\infty}f''(x+s)ds\int\limits_{0}^{\infty}\lambda^{\alpha -3/2}e^{-\sqrt{2\lambda} s }d\lambda =\int\limits_{0}^{\infty}\lambda^{\alpha -3/2}d\lambda\int\limits_{\varepsilon}^{\infty}f''(x+s)e^{-\sqrt{2\lambda} s }ds+\int\limits_{0}^{\varepsilon}f''(x+s) ds\int\limits_{0}^{\infty}e^{-\sqrt{2\lambda} s }\lambda^{\alpha -3/2}d\lambda.
$$
  Passing to the limit at the right-hand side, using \eqref{3.27},\eqref{3.28},  we obtain
$$
 \int\limits_{0}^{\infty}\lambda^{\alpha -3/2}d\lambda\int\limits_{0}^{\infty}f''(x+s)e^{-\sqrt{2\lambda} s }ds=\int\limits_{0}^{\infty}f''(x+s)ds\int\limits_{0}^{\infty}\lambda^{\alpha -3/2}e^{-\sqrt{2\lambda} s }d\lambda=
  $$
  $$
  =2^{3/2-\alpha}\Gamma(2\alpha-1) \int\limits_{0}^{\infty}f''(x+s)s^{1-2\alpha} ds  .
$$
Taking into account the analogous reasonings, we conclude that
$$
 A ^{\alpha}f(x)= - \frac{\Gamma(2\alpha-1)}{2^{  \alpha}\Gamma(\alpha)\Gamma(1-\alpha)}  \int\limits_{-\infty}^{\infty}f''(x+s)|s|^{1-2\alpha} ds=  K_{\alpha} I^{2(1-\alpha)}f''(x),\;
 $$
 $$
 K_{\alpha}=-\frac{    \Gamma(2\alpha-1)\cos  \alpha \pi / 2 }{2^{  \alpha-1} \Gamma(1-\alpha)},\,f\in C_{0}^{\infty}(\Omega).
$$

Using   the  Hardy-Littlewood theorem with limiting exponent (see Theorem 5.3 \cite[p.103]{firstab_lit:samko1987}),  we get
\begin{equation}\label{3.29}
  \|A ^{\alpha}f\|_{L_{2}}\leq C\|I^{2(1-\alpha)}_{+}f''\|_{L_{2}}+C\|I^{2(1-\alpha)}_{-}f''\|_{L_{2}} \leq C \|f''\|_{L_{q}} ,\, f  \in C^{\infty}_{0}(\Omega),
\end{equation}
where $q=2/(5-4 \alpha ).$   Applying the H\"{o}lder inequality, we obtain
\begin{equation}\label{3.30}
\left(\int\limits_{-\infty}^{\infty}|f''(x)|^{q }(1+|x|)^{5q /2}(1+|x|)^{-5q /2}dx\right)^{1/q }\leq
$$
$$
\leq\left(\int\limits_{-\infty}^{\infty}|f''(x)|^{2}(1+|x|)^{5} dx\right)^{1/2}\left(\int\limits_{-\infty}^{\infty} (1+|x|)^{-5q \gamma/2}dx\right)^{1/q \gamma}\leq C\|f\|_{H_{0}^{2,5}},\,f \in C^{\infty}_{0}(\Omega),
\end{equation}
where $1<q<2,\,\gamma=2/(2 -q )>1.$
  Combining  \eqref{3.29},\eqref{3.30} and passing to the limit, we get
\begin{equation*}
  \|A ^{\alpha}f\|_{L_{2}}\leq C \|f\|_{H_{0}^{2,5}},\,f\in H_{0}^{2,5}(\Omega).
\end{equation*}
Hence  $ H_{0}^{2,5}(\Omega)\subset \mathrm{D}(A ^{\alpha}).$
Using   the  Hardy-Littlewood theorem with limiting exponent, we obtain
 $$
 \|I^{\sigma}_{+}\rho I^{2(1-\alpha)} f''\|_{L_{2}}\leq C \| \rho I^{2(1-\alpha)} f''\|_{L_{q_{1}}}\leq C_{\rho}\|     f''\|_{L_{q_{2}}},\,f\in C_{0}^{\infty}( \Omega),\,C_{\rho}=C\|\rho\|_{L_{\infty}},
 $$
 where $q_{1}=2/(1+2\sigma),\,q_{2}=q_{1}/(1+2q_{1}[1-\alpha]).$   We can rewrite  $q_{2}=2/(1+2\sigma +4 [1-\alpha]),$ thus  $1<q_{2}<2.$
Applying  formula \eqref{3.30} and passing to the limit, we get
 \begin{equation}\label{3.31}
 \|I^{\sigma}_{+}\rho I^{2(1-\alpha)} f''\|_{L_{2}}\ \leq C_{\rho}\|     f \|_{H^{2,5}_{0}},\,f\in H_{0}^{2,5}(\Omega).
 \end{equation}
Note that
 \begin{equation}\label{3.32}
 \int\limits_{\Omega}  \mathcal{T }f\, \bar{g}  dx= \int\limits_{\Omega}    a   f''\,  \overline{g''}  dx,\; f,g\in C_{0}^{\infty}( \Omega).
 \end{equation}
Therefore  $  \mathcal{T} $ is accretive,   applying  Lemma \ref{L3.1}   we  deduce that $\tilde{\mathcal{T}}$ is m-accretive.
Using  relation \eqref{3.24},\eqref{3.32} we can easily obtain
$
\|\tilde{\mathcal{T }}f\|_{L_{2}}\geq  \gamma_{a}\|f\|_{H_{0}^{2,5}}\geq C \|Af\|_{L_{2}},\; f \in \mathrm{D}(\tilde{\mathcal{T }}),
$
whence $\mathrm{D}(\tilde{\mathcal{T }})\subset H_{0}^{2,5}( \Omega) \subset \mathrm{D}(A).$
    Using simple reasonings, we can extend  relation \eqref{3.32}  and rewrite it in the following form
$$
 \int\limits_{\Omega}  \tilde{\mathcal{T }}f\,\bar{g}  dx= \frac{1}{4}\int\limits_{\Omega}    a   Af\,   \overline{Ag}   dx,\; f\in \mathrm{D}(\tilde{\mathcal{T}}),\,g\in \mathrm{D}(A),
 $$
whence $\tilde{\mathcal{T}}\subset A^{\ast}G A,$ where $G:=a/4.$ Since the operator $\tilde{\mathcal{T}}$ is m-accretive, $A^{\ast}G A$ is accretive, then $\tilde{\mathcal{T}}= A^{\ast}G A.$
Hence, taking into account the inclusion $\mathrm{D}(\tilde{\mathcal{T }})\subset H_{0}^{2,5}( \Omega),$ relation \eqref{3.31}, we conclude that
$
L= A^{\ast}G A+F A^{\alpha},
$
where $ F:=\rho I.$

Let us prove that the operator $L$ satisfies     conditions H1--H2.     Choose the space  $L_{2}(\Omega)$ as a space $\mathfrak{H},$ the set  $C_{0}^{\infty}(\Omega)$ as a linear  manifold $\mathfrak{M},$ and the space  $H^{2,5}_{0}(\Omega)$ as a space $\mathfrak{H}_{+}.$    By virtue of     Theorem 1 \cite{firstab_lit:1Adams}, we have  $H^{2,5}_{0}(\Omega)\subset\subset L_{2}(\Omega).$
Thus, condition  H1  is satisfied.

Using   simple  reasonings (the proof is omitted), we come to the following inequality
 \begin{equation*}
\left|\int\limits_{-\infty}^{\infty}  (\tilde{\mathcal{T}}+\delta I) f \cdot  \bar{g} \, dx\right|\leq C\|f\|_{H^{2,5}_{0}}\|g\|_{H^{2,5}_{0}},\; f,g\in C_{0}^{\infty}( \Omega).
\end{equation*}
Applying    the Cauchy Schwarz inequality, relation  \eqref{3.31},  we obtain
\begin{equation}\label{3.33}
 |\left(I^{\sigma}_{+}\rho I^{2(1-\alpha)} f'', g\right)_{ L _{2}}|  \leq C_{\rho} \|f\|_{H_{0}^{2,5}}\|g\|_{H_{0}^{2,5}},\,f ,\,g\in C^{\infty}_{0}(\Omega).
\end{equation}
On the other hand, using the conditions  imposed on the function $a(x),$  it is not hard to prove that
\begin{equation*}
 \mathrm{Re}( [\tilde{\mathcal{T}}+\delta I]f,f)\geq \min\{\gamma_{a},\delta\}\|f\|^{2}_{H^{2,5}_{0}},\,f\in C_{0}^{\infty}( \Omega).
\end{equation*}
Using relation \eqref{3.33}, we can easily obtain
$$
\mathrm{Re}( I^{\sigma}_{+}\rho I^{2(1-\alpha)} f'',f)\geq -C_{\rho}\|f\|^{2}_{H^{2,5}_{0}},\,f\in C_{0}^{\infty}( \Omega).
$$
Combining the above estimates,  we conclude that   if the condition   $\min\{\gamma_{a},\delta\}>C_{\rho}$ holds,  then   $\mathrm{Re}( L f,f)\geq C \|f\|^{2}_{H^{2,5}_{0}},\,f\in C_{0}^{\infty}( \Omega).$ Thus, condition H2 is satisfied.
\end{proof}

\subsection{Difference operator}

Consider a   space $L_{2}(\Omega),\,\Omega:=(-\infty,\infty),$   define a family of operators
$$
T_{t}f(x):=e^{-\lambda t}\sum\limits_{k=0}^{\infty}\frac{(\lambda t)^{k}}{k!}f(x-k\mu),\,f\in L_{2}(\Omega),\;\lambda,\mu>0,\; t\geq0,
$$
where convergence is understood in the sense of $L_{2}(\Omega)$ norm. It is not hard to prove that $T_{t}: L_{2}\rightarrow L_{2},$ for this purpose it is sufficient to note that
\begin{equation}\label{3.34}
\left\|\sum\limits_{k=n}^{n+p}\frac{(\lambda t)^{k}}{k!}f(\cdot-k\mu) \right\|_{L_{2}}\leq  \left\|  f \right\|_{L_{2}}\sum\limits_{k=n}^{n+p}\frac{(\lambda t)^{k}}{k!}.
\end{equation}
\begin{lem}\label{L3.6}
$T_{t}$ is a $C_{0}$ semigroup of contractions, the corresponding  infinitesimal generator and its adjoint operator are defined by the following expressions
$$
Af(x)=\lambda[f(x)-f(x-\mu)],\,A^{\ast}f(x)=\lambda[f(x)-f(x+\mu)],\,f\in L_{2}(\Omega).
$$
\end{lem}
\begin{proof}
Assume that $f\in L_{2}(\Omega).$  Analogously to  \eqref{3.34}, we easily  prove that $\|T_{t}f\|_{L_{2}}\leq \|f\|_{L_{2}}.$ Consider
$$
T_{s}T_{t}f(x)=e^{-\lambda s}\sum\limits_{n=0}^{\infty}\frac{(\lambda s)^{n}}{n!}\left[e^{-\lambda t}\sum\limits_{k=0}^{\infty}\frac{(\lambda t)^{k}}{k!}f(x-k\mu-n\mu)\right].
$$
Since we have
$$
 \left\|\sum\limits_{k=0}^{m}\frac{(\lambda t)^{k}}{k!}f(x-k\mu)\right\|_{L_{2}}\leq \|    f \|_{L_{2}}  \sum\limits_{k=0}^{m}\frac{(\lambda t)^{k}}{k!},
$$
then similarly  to the case corresponding to $C(\Omega)$ norm (the prove is based upon the properties of the absolutely convergent double series, see Example 3 \cite[p.327]{firstab_lit:Yosida} ), we conclude that
$$
T_{s}T_{t}f(x)=e^{-\lambda  s }\sum\limits_{n=0}^{\infty}\frac{(\lambda s)^{n}}{n!} \left[  e^{-\lambda  t } \sum\limits_{k=0}^{\infty}\frac{(\lambda t)^{k}}{k!} f(x-k\mu-n\mu )\right]=
$$
$$
=e^{-\lambda (s+t)}\sum\limits_{p=0}^{\infty}\frac{1}{p!} \left[p! \sum\limits_{n=0}^{p}\frac{(\lambda s)^{n}}{n!}\frac{(\lambda t)^{p-n}}{(p-n)!}f(x-p\mu )\right]=
$$
$$
 =e^{-\lambda (s+t)}\sum\limits_{p=0}^{\infty}\frac{1}{p!}  (\lambda s+\lambda t)^{p} f(x-p\mu ) =T_{s+t}f(x),
$$
where equality is understood in the sense of $L_{2}(\Omega)$  norm. Let us establish the strongly continuous property. For sufficiently small $t,$ we have
$$
\|T_{t}f-f\|_{L_{2}}\leq e^{-\lambda t}( e^{\lambda t}-1)\|f\|_{L_{2}}+e^{-\lambda t}\left\|\sum\limits_{k=1}^{\infty}\frac{(\lambda t)^{k}}{k!}f(\cdot\,-k\mu)\right\|_{L_{2}}\leq    t e^{-\lambda t}\|f\|_{L_{2}} \left\{C+\sum\limits_{k=0}^{\infty}\frac{(\lambda )^{k+1}t^{k}}{(k+1)!}\right\},
$$
from what  follows that
$$
\|T_{t}f-f\|_{L_{2}}\rightarrow 0,\,t\rightarrow 0.
$$
Taking into account the above facts, we conclude that $T_{t}$ is a $C_{0}$  semigroup of contractions. Let us show that
$$
Af(x)=\lambda[f(x)-f(x-\mu)],
$$
we have (the proof is omitted)
$$
\frac{(I-T_{t})f(x)}{t}=   \frac{1-e^{-\lambda t}}{t} f(x)- \lambda e^{-\lambda t}f(x-\mu)-te^{-\lambda t}\sum\limits_{k=2}^{\infty}\frac{ \lambda   ^{k}t^{k-2}}{k!}f(x-k\mu).
$$
Hence
$$
 \frac{(I-T_{t})f }{t} \stackrel{L_{2}}{ \longrightarrow} \lambda[f -f(\cdot\,-\mu)],\,t\downarrow0,
$$
thus, we have obtained the desired result.
 Using   change of    variables  in integral it is easy to show that
$$
\int\limits_{-\infty}^{\infty}Af(x)g(x)dx=\int\limits_{-\infty}^{\infty}f(x)\lambda[g(x)-g(x+\mu)]dx,\,f,g\in L_{2}(\Omega),
$$
hence  $A^{\ast}f(x)=\lambda[f(x)-f(x+\mu)],\,f\in L_{2}(\Omega).$  The proof is complete.
\end{proof}

It is remarkable that there are some difficulties to apply Theorem \ref{T2.6} to a transform  $Z^{\alpha}_{aI,bI}(A),$ where $a,b$ are functions,   and the main of them can be said as follows "it is not clear how we should    build a space $\mathfrak{H}_{+}$". However we can consider a rather  abstract  perturbation of the above transform in order to reveal its spectral properties.

\begin{teo}\label{T3.4} Assume that  $Q$ is a   closed operator acting in $L_{2}(\Omega),\,Q^{-1}\in \mathcal{K}(L_{2}),$ the operator $N$ is strictly accretive, bounded, $\mathrm{R}(Q)\subset \mathrm{D}(N).$ Then
a perturbation
$$
L:= Z^{\alpha}_{aI,bI}(A)+ Q^{\ast}N Q ,\;a,b\in L_{\infty}(\Omega),\,\alpha\in(0,1)
$$
   satisfies conditions  H1--H2, if  $\gamma_{N}>\sigma\|Q^{-1}\|^{2},$
where we put $\mathfrak{M}:=\mathrm{D}_{0}(Q),$
$$
 \sigma= 4\lambda\|a\|_{L_{\infty}}+\|b\|_{L_{\infty}}\frac{\alpha\lambda^{\alpha}   }{\Gamma(1 -\alpha)}
 \sum\limits_{k=0}^{\infty}\frac{ \Gamma(k -\alpha)}{k! }.
$$
\end{teo}

\begin{proof}
Let us find a representation for fractional powers of the operator $A.$ Using  formula  , we get
\begin{equation}\label{3.35}
   A^{\alpha}f=\sum\limits_{k=0}^{\infty}C_{k}f(x-k\mu), \,f\in C^{\infty}_{0}(\Omega),
\end{equation}
$$
   \,C_{k}=-\frac{\alpha \lambda^{k} }{k!\Gamma(1-\alpha)}\int\limits_{0}^{\infty}e^{-\lambda t}t^{k-1-\alpha}dt=-\frac{\alpha\Gamma(k -\alpha)}{k!\Gamma(1 -\alpha)}\lambda^{\alpha},\,k=0,1,2,...,\,.
$$
Let us  extend    relation \eqref{3.35} to $L_{2}(\Omega).$    We have almost everywhere
$$
\sum\limits_{k=0}^{\infty}C_{k} g (x-k\mu)  -\sum\limits_{k=0}^{\infty}C_{k}  f (x-k\mu)=\sum\limits_{k=0}^{\infty}C_{k}[g(x-k\mu)-f (x-k\mu)],\,g\in C_{0}^{\infty}(\Omega),\,f\in L_{2}(\Omega),
$$
since the first sum is a partial sum for a fixed $x\in \mathbb{R}.$
  In accordance with formula (1.66) \cite[p.17]{firstab_lit:samko1987}, we have    $|C_{k}|\leq C \,k^{-1-\alpha},$  hence
$$
\left\|\sum\limits_{k=0}^{\infty}C_{k}[g(\cdot-k\mu)-f (\cdot-k\mu)]\right\|_{L_{2}}\leq\|g-f\|_{L_{2}}\sum\limits_{k=0}^{\infty}|C_{k}| .
$$
Thus, we obtain
$$
\forall f\in L_{2}(\Omega),\,\exists \{f_{n}\}\in C_{0}^{\infty}(\Omega):     \,f_{n}\stackrel{L_{2}}{ \longrightarrow} f,\;A^{\alpha}f_{n}\stackrel{L_{2}}{ \longrightarrow} \sum\limits_{k=0}^{\infty}C_{k}  f (\cdot-k\mu).
$$
Since $A^{\alpha}$ is closed, then
 \begin{equation}\label{3.36}
   A^{\alpha}f=\sum\limits_{k=0}^{\infty}C_{k}f(x-k\mu),\, f\in L_{2}(\Omega).
\end{equation}
Moreover, it is clear that $C^{\infty}_{0}(\Omega)$ is a core of $A^{\alpha}.$
 On the other hand, applying formula \eqref{3.7}, using the notation $\eta(x)=\lambda[f(x)-f(x-\mu)],$ we get
$$
A^{\alpha}f(x)=\frac{\sin\alpha \pi}{\pi}\int\limits_{0}^{\infty}\xi^{\alpha-1}(\xi I+A)^{-1}Af(x) d\xi=\frac{\sin\alpha \pi}{\pi}\int\limits_{0}^{\infty}\xi^{\alpha-1}d \xi\int\limits_{0}^{\infty}e^{-\xi t}T_{t}\eta(x)dt=
$$
$$
 =\frac{\sin\alpha \pi}{\pi}\sum\limits_{k=0}^{\infty}\frac{ \lambda   ^{k}}{k!}\eta(x-k\mu)\int\limits_{0}^{\infty}\xi^{\alpha-1} d \xi\int\limits_{0}^{\infty} e^{-t(\xi+\lambda)   }t^{k} dt=
$$
$$
 =\frac{\sin\alpha \pi}{\pi}\sum\limits_{k=0}^{\infty}\frac{ \lambda   ^{k}}{k!}\eta(x-k\mu)\int\limits_{0}^{\infty}\xi^{\alpha-1}(\xi+\lambda)^{-k-1}d \xi\int\limits_{0}^{\infty} e^{-t   }t^{k} dt,\,f\in C_{0}^{\infty}(\Omega),
$$
  we can rewrite the previous relation as follows
\begin{equation}\label{3.37}
A^{\alpha}f(x)=\sum\limits_{k=0}^{\infty} C'_{k}[f(x-k\mu)-f(x-(k+1)\mu)],\,f\in C_{0 }^{\infty}(\Omega),
\end{equation}
$$
\, C'_{k}=\frac{\lambda   ^{k+1}\sin\alpha \pi}{\pi}\int\limits_{0}^{\infty}\xi^{\alpha-1}(\xi+\lambda)^{-k-1}d \xi.
$$
Note that analogously to \eqref{3.36} we can extend formula \eqref{3.37} to $L_{2}(\Omega).$
Comparing  formulas \eqref{3.35},\eqref{3.37} we   can check the results
   calculating directly, we  get
$$
C'_{k+1}-C'_{k}
=-\frac{\lambda   ^{k+1}\sin\alpha \pi}{\pi} \int\limits_{0}^{\infty}\xi^{\alpha }(\xi+\lambda)^{-k-2}d \xi=
-\frac{\alpha\Gamma(k+1 -\alpha)}{(k+1)!\Gamma(1 -\alpha)}\lambda^{\alpha}=C_{k+1},
 \,C'_{0}=C_{0},\,k\in \mathbb{N}_{0}.
$$
Observe that by virtue of the made assumptions regarding   $Q,$ we have $\mathfrak{H}_{Q}\subset\subset L_{2}(\Omega).$ Choose the space  $L_{2}(\Omega)$ as a space $\mathfrak{H}$ and the space  $\mathfrak{H}_{Q} $ as a space $\mathfrak{H}_{+}.$  Let $ S:=Z^{\alpha}_{aI,bI}(A),\,T:=Q^{\ast}N Q.$  Applying the reasonings of Theorem  \ref{T3.1}, we conclude that there exists a set $\mathfrak{M}:=\mathrm{D}_{0}(Q),$ which is dense in $\mathfrak{H}_{Q},$ such that the operators $S,T$ are defined on its elements.
        Thus, we obtain the fulfilment of condition H1.
Since the operator $N$ is bounded, then
$
|(Tf,g)|_{L_{2}}\leq\|N\| \cdot \|f\|_{\mathfrak{H}_{Q}}\|g\|_{\mathfrak{H}_{Q}}.
$
Using formula \eqref{3.36}, we can easily obtain
$
|(Sf,g)|_{L_{2}}\leq \sigma \|f\|_{L_{2}}\|g\|_{L_{2}}\leq \sigma \|Q^{-1}\|^{2}\cdot\|f\|_{\mathfrak{H}_{Q}}\|g\|_{\mathfrak{H}_{Q}},\,\sigma=4\lambda\|a\|_{L_{\infty}}+ \|b\|_{L_{\infty}}\sum _{k=0}^{\infty} |C_{k}|.
$
Using the strictly accretive property of the operator $N$ we get
$
\mathrm{Re}(Tf,f)\geq \gamma_{N}\|f\|^{2}_{\mathfrak{H}_{Q}}.
$
On the other hand
$
\mathrm{Re}(Sf,f)\geq -\sigma \|Q^{-1}\|^{2}\cdot\|f\|^{2}_{\mathfrak{H}_{Q}},
$
hence condition H2 is satisfied. The proof is complete.

\end{proof}

\section{Norm equivalence}

\subsection{Accretive operators}

The facts that have motivated us to write this paragraph  lie in the  fractional calculus theory. Basically, an event that a differential operator with a fractional derivative in final terms underwent  a careful study \cite{firstab_lit:1Nakhushev1977}, \cite{kukushkin2019}  have played an important role in our research. The main feature is that there exists various approaches to study the operator and one of them is based on an opportunity to represent it in a sum of a senior term and an a lower term, here we should note that this method works if the  senior term is selfadjoint or normal. Thus,
in the case corresponding to a selfadjoint senior term, we can partially solve the problem  having applied the results of the   perturbation theory,   within the framework of which  the following papers are well-known   \cite{firstab_lit:1Katsnelson}, \cite{firstab_lit:1Krein},   \cite{firstab_lit:2Markus},
  \cite{firstab_lit:3Matsaev}, \cite{firstab_lit:1Mukminov},
 \cite{firstab_lit:Shkalikov A.}. Note that  to apply the last paper results  we must have the mentioned above representation. In other cases we can use methods of the paper   \cite{firstab_lit(arXiv non-self)kukushkin2018},  which are relevant if we deal with non-selfadjoint operators and allow us  to study spectral properties of   operators.
In the  paper \cite{kukushkin2021a}  we  explore  a special  operator class for which    a number of  spectral theory theorems can be applied. Further, we construct an abstract  model of a  differential operator in terms of  {\it m}-accretive  operators and call it an {\it m}-accretive operator transform, we  find  such conditions that    being  imposed guaranty  that the transform   belongs to the class. One of them is a compact embedding of a space generated by an {\it m}-accretive  operator (infinitesimal generator) into the initial Hilbert space. Note that in the case corresponding to the second  order operator with the  Kiprianov operator in final terms we have obtained  the embedding mentioned above in the one-dimensional case only.   In this paragraph,  we try to reveal this problem  and the main result   is a theorem establishing    equivalence of norms in   function spaces  in consequence  of  which we have  a   compact embedding of a space generated by the infinitesimal generator of the shift semigroup in a direction into the Lebsgue space. We should note that this result do not give us a useful concrete  application in the built theory for it is more of  an abstract generalization.  However   this result, by virtue of popularity and well known applicability of the  Lebesgue spaces theory,    deserves    to be considered    itself.  As for relevance, that is more fundamental than applied, as it often occurs with such kind of results,  we should turn to the series of papers by Kiprianov I.A. devoted  to the alternative branch of fractional calculus theory \cite{firstab_lit:kipriyanov1960},\cite{firstab_lit:1kipriyanov1960},\cite{firstab_lit:2.2kipriyanov1960}. In this series of papers the author introduced a directional fractional derivative  afterwards represented in \cite{firstab_lit(arXiv non-self)kukushkin2018}   as a fractional power of the shift semigroup in a direction. The norm equivalence established  in this  paper allows to reformulate  results of the paper  \cite{kukushkin2021a} in terms of the infinitesimal generator of the shift semigroup in a direction. Thus, we may say that an opportunity to apply spectral theorems \cite{kukushkin2021a} in the natural way   becomes  relevant not only due to the application  part, but a better comprehension of the mathematical phenomenon.\\

Assume that  $\Omega\subset \mathbb{E}^{n}$ is  a convex domain, with a sufficient smooth boundary ($C^{3}$ class)  of   the  n-dimensional Euclidian space. For the sake of the simplicity we consider that $\Omega$ is bounded.
Consider the shift semigroup in a direction acting on $L_{2}(\Omega)$ and  defined as follows
$
T_{t}f(Q)=f(Q+\mathbf{e}t),
$
where $Q\in \Omega,\,Q=P+\mathbf{e}r.$   The following lemma establishes  a property of  the infinitesimal generator $-A$ of the semigroup $T_{t}.$
\begin{lem}\label{L3.7} We claim that
$
A=\tilde{A_{0}},\,\mathrm{N}(A)=0,
$
where $A_{0}$  is a restriction of $A$ on the set
   $ C^{\infty}_{0}( \Omega ).$
 \end{lem}
\begin{proof}
Let us show  that $T_{t}$ is a strongly continuous semigroup ($C_{0}$  semigroup). It can be easily established   due to the  continuous in average property. Using the Minkowskii inequality, we have
 $$
 \left\{\int\limits_{\Omega}|f(Q+\mathbf{e}t)-f(Q)|^{2}dQ\right\}^{\frac{1}{2}}\leq  \left\{\int\limits_{\Omega}|f(Q+\mathbf{e}t)-f_{m}(Q+\mathbf{e}t)|^{2}dQ\right\}^{\frac{1}{2}}+
 $$
 $$
 +\left\{\int\limits_{\Omega}|f(Q)-f_{m}(Q)|^{2}dQ\right\}^{\frac{1}{2}}+\left\{\int\limits_{\Omega}|f_{m}(Q)-f_{m}(Q+\mathbf{e}t)|^{2}dQ\right\}^{\frac{1}{2}}=
 $$
 $$
 =I_{1}+I_{2}+I_{3}<\varepsilon,
 $$
where $f\in L_{2}(\Omega),\,\left\{f_{n}\right\}_{1}^{\infty}\subset C_{0}^{\infty}(\Omega);$  $m$ is chosen so that $I_{1},I_{2}< \varepsilon/3 $ and $t$
is chosen so that $I_{3}< \varepsilon/3.$
  Thus,  there exists such a positive  number $t_{0}$    that
$$
\|T_{t}f-f\|_{L_{2} }<\varepsilon,\,t<t_{0},
$$
for arbitrary small $\varepsilon>0.$ Hence in accordance with the definition $T_{t}$ is  a $C_{0}$ semigroup.   Using the assumption that  all functions have the zero extension outside $\bar{\Omega},$   we have
$\|T_{t}\|  \leq 1.$ Hence  we conclude that $T_{t}$ is a $C_{0}$ semigroup of contractions (see \cite{Pasy}).
Hence   by virtue of   Corollary 3.6 \cite[p.11]{Pasy}, we have
\begin{equation} \label{3.38}
\|(\lambda+A)^{-1}\|  \leq \frac{1}{\mathrm{Re} \lambda },\,\mathrm{Re}\lambda>0.
\end{equation}
Inequality\eqref{3.38}implies that $A$ is {\it m}-accretive.
    It is the well-known fact that   an infinitesimal  generator    $-A$ is a closed operator, hence $A_{0}$  is closeable.
 It is not hard to prove that $ \tilde{A} _{0}$ is an {\it m}-accretive operator. For this purpose let us    rewrite relation\eqref{3.38}in the form
\begin{equation*}
\|(\lambda+ \tilde{A} _{0})^{-1}\|_{\mathrm{R}\rightarrow \mathfrak{H}}  \leq \frac{1}{\mathrm{Re} \lambda },\,\mathrm{Re}\lambda>0,
\end{equation*}
applying Lemma \ref{L3.1}, we obtain that $ \tilde{A} _{0}$ is an {\it m}-accretive operator. Note that there does not exist an  accretive extension of an {\it m}-accretive operator (see \cite{firstab_lit:kato1980}). On the other hand it is clear that $\tilde{A} _{0}\subset A.$  Thus we conclude that $\tilde{A} _{0}= A.$
Consider an operator
 $$
 B f(Q)=\!\int_{0}^{r}\!\!f(P+\mathbf{e }[r-t])dt,\,f\in L_{2}(\Omega).
 $$
 It is not hard to prove that  $B \in \mathcal{B}(L_{2}),$   applying the generalized Minkowskii inequality, we get
 $$
 \|B f\|_{L_{2} }\leq \int\limits_{0}^{\mathrm{diam\,\Omega}}dt  \left(\int\limits_{\Omega}|f(P+\mathbf{e }[r-t])|dQ\right)^{1/2}\leq C\|f\|_{L_{2} }.
 $$
Note that   the fact $A^{-1}_{ 0}\subset B ,$    follows from the properties of the  one-dimensional integral defined on smooth functions.
Using  Theorem 2 \cite[p.555]{firstab_lit:Smirnov5}, the proved above fact $\tilde{A} _{0}= A,$  we deduce that   $A^{-1} \subset B .$
 The proof is complete.
\end{proof}

\subsection{Multidimensional spaces}
Consider a linear  space
$
\mathbb{L}^{n}_{2}(\Omega):=\left\{f=(f_{1},f_{2},...,f_{n}),\,f_{i}\in L_{2}(\Omega)\right\},
$
endowed with the inner product
$$
(f,g)_{\mathbb{L}^{n}_{2}}=\int\limits_{\Omega} (f, g)_{\mathbb{E}^{n}} dQ,\,f,g\in \mathbb{L}^{n}_{2}(\Omega).
$$
It is clear that this  pair forms a Hilbert space and let us use the same  notation $\mathbb{L}^{n}_{2}(\Omega)$ for it.
Consider a    sesquilinear  form
$$
t(f,g):=\sum\limits_{i=1}^{n}\int\limits_{\Omega} (f ,\mathbf{e_{i}})_{\mathbb{E}^{n}}\overline{(g,\mathbf{e_{i}})}_{\mathbb{E}^{n}} dQ,\,f,g\in \mathbb{L} ^{n}_{2} (\Omega),
$$
where $\mathbf{e_{i}}$   corresponds to  $P_{i}\in \partial\Omega,\,\,i=1,2,...,n$   (i.e. $Q=P_{i}+ \mathbf{e_{i}}r     $).
\begin{lem}\label{L3.8}
The points $P_{i}\in \partial\Omega,\,i=1,2,...,n$ can be chosen so that the   form $t$ generates an inner product.
\end{lem}
\begin{proof}
It is clear that we should only establish an implication  $t(f,f)=0\,\Rightarrow f=0.$
 Since $\Omega\in \mathbb{E}^{n},$  then without lose of generality we can assume that   there exists $P_{i}\in \partial\Omega,\,i=1,2,...,n,$ such that
\begin{equation}\label{3.39}\Delta= \left|
\begin{array}{cccc}
P_{11}&P_{12}&...&P_{1n}\\
P_{21}&P_{22}&...&P_{2n}\\
...&...&...&...\\
P_{n1}&P_{n2}&...&P_{nn}
\end{array}
\right|\neq0,
\end{equation}
where $P_{i}=(P_{i1},P_{i2},...,P_{in}).$ It becomes clear if we remind that in the contrary case, for arbitrary set of points $P_{i}\in \partial\Omega,\,i=1,2,...,n,$  we   have
$$
P_{n}=\sum\limits_{k=1}^{n-1}c_{k}P_{k},\,c_{k}= \mathrm{const},
$$
from what follows  that we can consider $\Omega$ at least  as a subset of  $\mathbb{E}^{n-1}.$ Continuing this line of reasonings we can find such  a dimension $p$ that a corresponding $\Delta\neq0$ and further assume  that  $\Omega\in \mathbb{E}^{p}. $
 Consider a relation
\begin{equation*}
 \sum\limits_{i=1}^{n}\int\limits_{\Omega}| (\psi,\mathbf{e_{i}})_{\mathbb{E}^{n}}|^{2}   dQ  =0,\,\psi\in \mathbb{L}^{n}_{2}(\Omega).
\end{equation*}
It follows that  $\left(\psi(Q),\mathbf{e_{i}}\right)_{\mathbb{E}^{n}}=0$ a.e. $i=1,2,...,n.$  Note that every $P_{i}$ corresponds to the set
 $\vartheta_{i}:=\{Q\subset \vartheta_{i} :\;(\psi(Q),\mathbf{e_{i}})_{\mathbb{E}^{n}}\neq0  \}.$ Consider $\Omega'=\Omega\backslash\bigcup\limits_{i=1}^{n}\vartheta_{i},$ it is clear that $\mathrm{mess} \left(\bigcup\limits_{i=1}^{n}\vartheta_{i}\right)=0.$ Note that due to the made construction, we can reformulate the obtained above  relation in the coordinate form
\begin{equation*}
 \left\{ \begin{aligned}
(P_{11}-Q_{1})\psi_{1}(Q)+(P_{12}-Q_{2})\psi_{2}(Q)+...+(P_{1n}-Q_{n})\psi_{n}(Q)=0 \\
 (P_{21}-Q_{1})\psi_{1}(Q)+(P_{22}-Q_{2})\psi_{2}(Q)+ ... +(P_{2n}-Q_{n})\psi_{n}(Q)=0 \\
...\;\;\;\;\;\;\;\;\;\;\;\;\;\;\;\;\;\;\;\;\;\;\;\;\;\;\;\;...\;\;\;\;\;\;\;\;\;\;\;\;\;\;\;\;\;\;\;\;\;\;\;\;\;\;\;\; ...\;\;\;\;\;\;\;\;\;\;\;\;\;\;\;\;\;\;\;\;\;\;\;\;\;\;\;\; \\
 (P_{n1}-Q_{1})\psi_{1}(Q)+(P_{n2}-Q_{2})\psi_{2}(Q)+ ... +(P_{nn}-Q_{n})\psi_{n}(Q)=0
\end{aligned}
 \right.\;,
\end{equation*}
where $\psi=(\psi_{1},\psi_{2},...,\psi_{n}),\,Q=(Q_{1},Q_{2},...,Q_{n}),\, Q\in \Omega'.$ Therefore, if we prove that
$$\Lambda(Q)= \left|
\begin{array}{cccc}
P_{11}-Q_{1}&P_{12}-Q_{2}&...&P_{1n}-Q_{n}\\
P_{21}-Q_{1}&P_{22}-Q_{2}&...&P_{2n}-Q_{n}\\
...&...&...&...\\
P_{n1}-Q_{1}&P_{n2}-Q_{2}&...&P_{nn}-Q_{n}
\end{array}
\right|\neq0\;a.e.,
$$
then we obtain $\psi =0$ a.e.   Assume the contrary i.e. that there exists such a set $ \Upsilon \subset  \Omega ,\,\mathrm{mess}\,\Upsilon \neq0,$   that   $\Lambda(Q)=0,\,Q\in \Upsilon  .$    We have
$$ \left|
\begin{array}{cccc}
P_{11}-Q_{1}&P_{12}-Q_{2}&...&P_{1n}-Q_{n}\\
P_{21}-Q_{1}&P_{22}-Q_{2}&...&P_{2n}-Q_{n}\\
...&...&...&...\\
P_{n1}-Q_{1}&P_{n2}-Q_{2}&...&P_{nn}-Q_{n}
\end{array}
\right|=
\left|
\begin{array}{cccc}
P_{11} &P_{12} &...&P_{1n} \\
P_{21}-Q_{1}&P_{22}-Q_{2}&...&P_{2n}-Q_{n}\\
...&...&...&...\\
P_{n1}-Q_{1}&P_{n2}-Q_{2}&...&P_{nn}-Q_{n}
\end{array}
\right|-
$$
$$
-\left|
\begin{array}{cccc}
 Q_{1}& Q_{2}&...& Q_{n}\\
P_{21}-Q_{1}&P_{22}-Q_{2}&...&P_{2n}-Q_{n}\\
...&...&...&...\\
P_{n1}-Q_{1}&P_{n2}-Q_{2}&...&P_{nn}-Q_{n}
\end{array}
\right|=
\left|
\begin{array}{cccc}
P_{11} &P_{12} &...&P_{1n} \\
P_{21} &P_{22} &...&P_{2n} \\
...&...&...&...\\
P_{n1}-Q_{1}&P_{n2}-Q_{2}&...&P_{nn}-Q_{n}
\end{array}
\right|-
$$
$$
-\left|
\begin{array}{cccc}
P_{11} &P_{12} &...&P_{1n} \\
 Q_{1}& Q_{2}&...& Q_{n}\\
...&...&...&...\\
P_{n1}-Q_{1}&P_{n2}-Q_{2}&...&P_{nn}-Q_{n}
\end{array}
\right|
 -\left|
\begin{array}{cccc}
 Q_{1}& Q_{2}&...& Q_{n}\\
P_{21} &P_{22} &...&P_{2n} \\
...&...&...&...\\
P_{n1}-Q_{1}&P_{n2}-Q_{2}&...&P_{nn}-Q_{n}
\end{array}
\right|=
$$
$$
 =\left|
\begin{array}{cccc}
 P_{11} &P_{12} &...&P_{1n}\\
P_{21} &P_{22} &...&P_{2n} \\
...&...&...&...\\
P_{n1} &P_{n2} &...&P_{nn}
\end{array}
\right|-\sum\limits_{j=1}^{n} \Delta_{j}=0,
$$
where
$$
 \Delta_{j}=\left|
\begin{array}{cccc}
 P_{11} &P_{12} &...&P_{1n}\\
P_{21} &P_{22} &...&P_{2n} \\
...&...&...&...\\
P_{j-1\,1} &P_{j-1\,2} &...&P_{j-1\,n}\\
Q_{1}& Q_{2}&...& Q_{n}\\
P_{j+1\,1} &P_{j+1\,2} &...&P_{j+1\,n}\\
...&...&...&...\\
P_{n1} &P_{n2} &...&P_{nn}
\end{array}
\right|.
$$
Therefore,  we have
$$
\sum\limits_{j=1}^{n}  \Delta_{j}/ \Delta =1,
$$
since $\Delta\neq0.$
Hence, we can treat the above matrix constructions  in the way that gives us the following representation
$$
\sum\limits_{j=1}^{n}  \alpha_{j} P_{j}  =Q,\,\sum\limits_{j=1}^{n} \alpha_{j}   =1,\,\alpha_{j}=\Delta_{j}/ \Delta .
$$
Now, let us prove that $\Upsilon$ belongs to a hyperplane in $\mathbb{E}^{n},$ we have
$$  \left|
\begin{array}{cccc}
P_{11}-Q_{1}&P_{12}-Q_{2}&...&P_{1n}-Q_{n}\\
P_{21}-P_{11}&P_{22}-P_{12}&...&P_{2n}-P_{1n}\\
...&...&...&...\\
P_{n1}-P_{n-1\,1}&P_{n2}-P_{n-1\,2}&...&P_{nn}-P_{n-1\,n}
\end{array}
\right|= \left|
\begin{array}{cccc}
P_{11}&P_{12}&...&P_{1n}\\
P_{21}&P_{22}&...&P_{2n}\\
...&...&...&...\\
P_{n1}&P_{n2}&...&P_{nn}
\end{array}
\right|-
$$
$$
-\left|
\begin{array}{cccc}
 Q_{1}& Q_{2}&...& Q_{n}\\
P_{21}-P_{11}&P_{22}-P_{12}&...&P_{2n}-P_{1n}\\
...&...&...&...\\
P_{n1}-P_{n-1\,1}&P_{n2}-P_{n-1\,2}&...&P_{nn}-P_{n-1\,n}
\end{array}
\right| =
$$
$$
=
 \left|
\begin{array}{cccc}
P_{11}&P_{12}&...&P_{1n}\\
P_{21}&P_{22}&...&P_{2n}\\
...&...&...&...\\
P_{n1}&P_{n2}&...&P_{nn}
\end{array}
\right|-
\left|
\begin{array}{cccc}
\sum\limits_{j=1}^{n}  \alpha_{j} P_{j1}&\sum\limits_{j=1}^{n}  \alpha_{j} P_{j2}&...&\sum\limits_{j=1}^{n}  \alpha_{j} P_{jn}\\
P_{21}-P_{11}&P_{22}-P_{12}&...&P_{2n}-P_{1n}\\
...&...&...&...\\
P_{n1}-P_{n-1\,1}&P_{n2}-P_{n-1\,2}&...&P_{nn}-P_{n-1\,n}
\end{array}
\right| =
$$
$$
 =\left|
\begin{array}{cccc}
P_{11}&P_{12}&...&P_{1n}\\
P_{21}&P_{22}&...&P_{2n}\\
...&...&...&...\\
P_{n1}&P_{n2}&...&P_{nn}
\end{array}
\right|-\sum\limits_{j=1}^{n}  \alpha_{j}
 \left|
\begin{array}{cccc}
     P_{j1}&  P_{j2}&...&  P_{jn}\\
P_{21}-P_{11}&P_{22}-P_{12}&...&P_{2n}-P_{1n}\\
...&...&...&...\\
P_{n1}-P_{n-1\,1}&P_{n2}-P_{n-1\,2}&...&P_{nn}-P_{n-1\,n}
\end{array}
\right| =
$$
$$
 =\left|
\begin{array}{cccc}
P_{11}&P_{12}&...&P_{1n}\\
P_{21}&P_{22}&...&P_{2n}\\
...&...&...&...\\
P_{n1}&P_{n2}&...&P_{nn}
\end{array}
\right|-
 \left|
\begin{array}{cccc}
P_{11}&P_{12}&...&P_{1n}\\
P_{21}&P_{22}&...&P_{2n}\\
...&...&...&...\\
P_{n1}&P_{n2}&...&P_{nn}
\end{array}
\right|\sum\limits_{j=1}^{n}  \alpha_{j} =
$$
$$
 =\left|
\begin{array}{cccc}
P_{11}&P_{12}&...&P_{1n}\\
P_{21}&P_{22}&...&P_{2n}\\
...&...&...&...\\
P_{n1}&P_{n2}&...&P_{nn}
\end{array}
\right|-
 \left|
\begin{array}{cccc}
P_{11}&P_{12}&...&P_{1n}\\
P_{21}&P_{22}&...&P_{2n}\\
...&...&...&...\\
P_{n1}&P_{n2}&...&P_{nn}
\end{array}
\right|  =0.
$$
Hence $\Upsilon$ belongs to a hyperplane generated by the points $P_{i},\,i=1,2,...,n.$ Therefore $ \mathrm{mess} \Upsilon=0,$  and we obtain
  $ \psi =0$ a.e.
The proof is complete.
\end{proof}
Consider     a pre Hilbert  space $ \mathbf{L} ^{n}_{2}(\Omega):=\{f:\,f\in \mathbb{L}^{n}_{2}(\Omega)\}$ endowed with the inner product
$$
(f,g)_{\mathbf{L} ^{n}_{2} }:=\sum\limits_{i=1}^{n}\int\limits_{\Omega} (f ,\mathbf{e_{i}})_{\mathbb{E}^{n}}\overline{(g,\mathbf{e_{i}})}_{\mathbb{E}^{n}} dQ,\,f,g\in \mathbb{L} ^{n}_{2} (\Omega),
$$
where $\mathbf{e_{i}}$   corresponds to  $P_{i}\in \partial\Omega,\,\,i=1,2,...,n\,,$      condition\eqref{3.39} holds.
The following theorem establishes a norm equivalence.
\begin{teo}\label{T3.5}
The   norms $\|\cdot\|_{ \mathbb{L}^{n}_{2} }$ and $\|\cdot\|_{\mathbf{L} ^{n}_{2} } $ are equivalent.
\end{teo}
\begin{proof}
 Consider the space
$
\mathbb{L}^{n}_{2}(\Omega)
$
and a    functional
$
 \varphi(f):= \|f\|_{\mathbf{L} ^{n}_{2} } ,\,f\in \mathbb{L}^{n}_{2}(\Omega).
$
Let us prove that
$
  \varphi(f)\geq C,\,f\in \mathrm{U},
$
where $\mathrm{U}:=\{f\in \mathbb{L}^{n}_{2}(\Omega),\,\|f\|_{\mathbb{L}^{n}_{2}}=1\}.$
Assume the contrary,
then   there exists such a sequence $\{\psi_{k}\}_{1}^{\infty}\subset \mathrm{U},$    that    $\varphi(\psi_{k})\rightarrow0,\,k\rightarrow\infty.$   Since the sequence  $\{\psi_{k}\}_{1}^{\infty}$ is bounded, then we can extract a weekly  convergent subsequence $\{\psi_{k_{j}}\}_{1}^{\infty}$ and  claim that the week limit   $\psi$  of the sequence $\{\psi_{k_{j}}\}_{1}^{\infty}$ belongs to $\mathrm{U}.$  Consider a functional
$$
\mathcal{L}_{g}(f):=\sum\limits_{i=1}^{n}\int\limits_{\Omega} (f ,\mathbf{e_{i}})_{\mathbb{E}^{n}}\overline{(g,\mathbf{e_{i}})}_{\mathbb{E}^{n}} dQ,\;f,g\in \mathbb{L}^{n}_{2}(\Omega).
$$
 Due to the following  obvious  chain of the inequalities
 \begin{equation*}
|\mathcal{L}_{g}(f)|     \leq \sum\limits_{i=1}^{n}
  \left\{\int\limits_{\Omega}| (f  ,\mathbf{e_{i}})_{\mathbb{E}^{n}}|^{2}   dQ\right\}^{\frac{1}{2}}\left\{\int\limits_{\Omega}| (g  ,\mathbf{e_{i}})_{\mathbb{E}^{n}}|^{2}   dQ\right\}^{\frac{1}{2}}
  \leq
  $$
  $$
  \leq n \|f\|_{\mathbb{L}^{n}_{2}}\|g\|_{\mathbb{L}^{n}_{2}},  \; f,g\in \mathbb{L}^{n}_{2}(\Omega),
  \end{equation*}
  we see  that $\mathcal{L}_{g}$ is a linear bounded functional on $\mathbb{L}^{n}_{2}(\Omega).$
Therefore, by virtue of the weak convergence of the sequence $\{\psi_{k_{j}}\},$  we have $\mathcal{L}_{g}(\psi_{k_{j}})\rightarrow \mathcal{L}_{g}(\psi),\,k_{j}\rightarrow \infty.$ On the other hand, recall that since it was  supposed that $\varphi(\psi_{k})\rightarrow0,\,k\rightarrow\infty,$ then we have $\varphi(\psi_{k_{j}})\rightarrow0,\,k\rightarrow\infty.$ Hence applying \eqref{3.39}, we conclude that $\mathcal{L}_{g}(\psi_{k_{j}})\rightarrow 0,\,k_{j}\rightarrow \infty.$ Combining the given above results we obtain
\begin{equation}\label{3.40}
\mathcal{L}_{g}(\psi) =\sum\limits_{i=1}^{n}\int\limits_{\Omega} (\psi ,\mathbf{e_{i}})_{\mathbb{E}^{n}}\overline{(g,\mathbf{e_{i}})}_{\mathbb{E}^{n}} dQ=0,\,\forall g \in \mathbb{L}^{n}_{2}(\Omega).
\end{equation}
Taking into account \eqref{3.40}and using the ordinary properties of Hilbert space, we obtain
\begin{equation*}
  \sum\limits_{i=1}^{n}\int\limits_{\Omega}| (\psi,\mathbf{e_{i}})_{\mathbb{E}^{n}}|^{2}   dQ  =0.
\end{equation*}
Hence in accordance with Lemma \ref{L3.8}, we get  $ \psi =0$ a.e.
 Notice  that by virtue of this fact we come to the contradiction with the fact  $\|\psi\|_{\mathbb{L}^{n}_{2}}=1.$
  Hence the following estimate is true
$
  \varphi(f)\geq C,\,f\in \mathrm{U}.
$
Having applied the Cauchy Schwartz inequality to the Euclidian inner product, we can also easily obtain
$
 \varphi(f) \leq \sqrt{n}\|f\|_{\mathbb{L}^{n}_{2}},\,f\in \mathbb{L}^{n}_{2}(\Omega).
$
Combining the above inequalities,   we can rewrite these two estimates as follows
$
C_{0}\leq\varphi(f)\leq C_{1},\,f\in \mathrm{U}.
$
To make the issue clear,  we can  rewrite the previous inequality in the form
\begin{equation}\label{3.41}
C_{0}\|f\|_{\mathbb{L}^{n}_{2} }\leq\varphi(f)\leq C_{1}\|f\|_{\mathbb{L}^{n}_{2} },\,f\in  \mathbb{L}^{n}_{2}(\Omega),\;C_{0},C_{1}>0.
\end{equation}
The proof is complete.
\end{proof}
Consider a pre Hilbert  space
$$
\mathfrak{\widetilde{H}}^{n}_{ A }:= \big \{f,g\in C_{0}^{\infty}(\Omega),\,(f,g)_{\mathfrak{\widetilde{H}}^{n}_{ A }}=\sum\limits_{i=1}^{n}(A_{i} f,A_{i} g)_{L_{2} } \big\},
$$
where $-A_{i}$ is an infinitesimal generator corresponding to the point $P_{i}.$ Here, we should point out that the form $(\cdot,\cdot)_{\mathfrak{\widetilde{H}}^{n}_{ A }} $ generates an inner product due to the fact $\mathrm{N}(A_{i})=0,\,i=1,2,...,n$ proved in Lemma \ref{L3.7}.
  Let us denote a corresponding Hilbert space by $\mathfrak{H}^{n}_{A}.$
\begin{corol}\label{C3.2}
  The norms $\|\cdot\|_{\mathfrak{H}^{n}_{A}}$ and $\|\cdot\|_{H_{0}^{1}} $ are equivalent, we have a bounded compact  embedding
$$
\mathfrak{H}^{n}_{A}\subset\subset L_{2}(\Omega).
$$
\end{corol}
\begin{proof} Let us prove that
\begin{equation*}
 Af =                -(\nabla f ,\mathbf{e})_{\mathbb{E}^{n}},\,f\in  C^{\infty}_{0}( \Omega ).
\end{equation*}
Using the Lagrange  mean value  theorem, we have
$$
\int\limits_{\Omega}\left|  \left(\frac{T_{t}-I}{t}\right)f(Q)-(\nabla f ,\mathbf{e})_{\mathbb{E}^{n}}(Q )\right|^{2} dQ=\int\limits_{\Omega}\left|  (\nabla f ,\mathbf{e})_{\mathbb{E}^{n}}(Q_{\xi} )-(\nabla f ,\mathbf{e})_{\mathbb{E}^{n}}(Q )\right|^{2} dQ,
$$
where $Q_{\xi}=Q+ \mathbf{e}  \xi,\,0<\xi<t.$ Since the function $(\nabla f ,\mathbf{e})_{\mathbb{E}^{n}}$ is continuous on $\bar{\Omega},$ then it is uniformly continuous on $\bar{\Omega}.$ Thus,  for arbitrary $\varepsilon>0,$ a positive number $\delta>0$ can be chosen so that
$$
 \int\limits_{\Omega}\left|  (\nabla f ,\mathbf{e})_{\mathbb{E}^{n}}(Q_{\xi} )-(\nabla f ,\mathbf{e})_{\mathbb{E}^{n}}(Q ),\right|^{2} dQ<\varepsilon,\,t<\delta,
$$
from what follows the desired result. Taking it into account, we  obtain
 $$
 \|Af\|_{L_{2}}=\left\{\int\limits_{\Omega}|(\nabla f ,\mathbf{e})_{\mathbb{E}^{n}}|^{2}dQ\right\}^{1/2}\leq \left\{\int\limits_{\Omega} \|\mathbf{e}\|^{2}_{\mathbb{E}^{n}}\sum\limits_{i=1}^{n}|D_{i}f|^{2} dQ\right\}^{1/2} =\|f\|_{H_{0}^{1}},\,f\in C^{\infty}_{0}(\Omega).
 $$
Using this estimate, we easily obtain  $\|f\|_{\mathfrak{H}^{n}_{A}}\leq C \|f\|_{H_{0}^{1}},\,f\in  C_{0}^{\infty}(\Omega).$ On the other hand,
as a particular case of formula \eqref{3.41}, we obtain
$
C_{0}\|f\|_{H_{0}^{1}}\leq\|f\|_{\mathfrak{H}^{n}_{A}} ,\,f\in  C_{0}^{\infty}(\Omega).
$
Thus, we can combine  the previous inequalities and rewrite them as follows
$
C_{0}\|f\|_{H_{0}^{1}}\leq\|f\|_{\mathfrak{H}^{n}_{A}}\leq C \|f\|_{H_{0}^{1}},\,f\in  C_{0}^{\infty}(\Omega).
$
 Passing to the limit at the left-hand and right-hand side of the last inequality, we get
\begin{equation*}
C_{0}\|f\|_{H_{0}^{1}}\leq\| f\|_{\mathfrak{H}^{n}_{A}}\leq C \|f\|_{H_{0}^{1}},\,f\in  H_{0}^{1}(\Omega).
\end{equation*}
Combining the  fact
 $
 H_{0}^{1}(\Omega)\subset\subset L_{2}(\Omega),
 $
 (Rellich-Kondrashov theorem)
 with the  lower estimate in the previous inequality,
 we complete the proof.
\end{proof}
\vspace{0.5 cm}

\subsection{Connection with the semigroup approach }

In this section we aim to represent  some known operators in terms of the infinitesimal generator of the  shift semigroup in a direction and apply the obtained  results to the established representations. In this way we come to natural conditions in terms of the infinitesimal generator of the  shift semigroup in a direction that allows us to apply Theorem \ref{T2.6}.\\

\subsubsection{ Uniformly elliptic operator in the divergent form}

Consider a uniformly ecliptic operator
$$
\,-\mathcal{T}:=-D_{j} ( a^{ij} D_{i}\cdot),\, a^{ij}(Q) \in C^{2}(\bar{\Omega}),\,   a^{ij}\xi _{i}  \xi _{j}  \geq   \gamma_{a}  |\xi|^{2} ,\,  \gamma_{a}  >0,\,i,j=1,2,...,n,\;
$$
$$
\mathrm{D}( \mathcal{T} )  =H^{2}(\Omega)\cap H^{1}_{0}( \Omega ).
$$
The following theorem gives us a key to apply results of the paper \cite{kukushkin2021a} in accordance with  which a number of spectral theorems  can be applied to the operator $-\mathcal{T}.$ Moreover the conditions established bellow are formulated in terms of the operator $A,$ what reveals a mathematical nature of the operator $-\mathcal{T}.$
\begin{teo}\label{T3.6}
We claim that
\begin{equation}\label{3.42}
 -\mathcal{T}=\frac{1}{n}\sum\limits^{n}_{i=1} A_{i}^{\ast}G_{i}A_{i},
\end{equation}
the following relations hold
$$
-\mathrm{Re}(\mathcal{T}f,f)_{L_{2}}\geq C  \|f\|_{\mathfrak{H}^{n}_{A}};\, |(\mathcal{T}f,g)_{L_{2}}|\leq C \|f\|_{\mathfrak{H}^{n}_{A}}\|g\|_{\mathfrak{H}^{n}_{A}},\;f,g\in C_{0}^{\infty}(\Omega),
$$
where $G_{i}$ are some operators corresponding to the operators $A_{i}.$
\end{teo}
\begin{proof}
It is easy to prove  that
\begin{equation}\label{3.43}
\|A_{i}f\|_{L_{2}}\leq C\|f\|_{H_{0}^{1}},\,f\in H_{0}^{1}(\Omega),
\end{equation}
 for this purpose we should use a representation $A_{i}f(Q) =-(\nabla f ,\mathbf{e_{i}})_{\mathbb{E}^{n}}, f\in C^{\infty}_{0}(\Omega).$ Applying  the Cauchy Schwarz inequality, we get
 $$
 \|A_{i}f\|_{L_{2}}\leq\left\{\int\limits_{\Omega}|(\nabla f ,\mathbf{e_{i}})_{\mathbb{E}^{n}}|^{2}dQ\right\}^{1/2}\leq \left\{\int\limits_{\Omega}\|\nabla f\|^{2}_{\mathbb{E}^{n}}
 \|\mathbf{e_{i}}\|^{2}_{\mathbb{E}^{n}} dQ\right\}^{1/2}=\|f\|_{ H_{0}^{1}},\,f\in C^{\infty}_{0}(\Omega).
 $$
 Passing to the limit at the left-hand and right-hand side, we obtain \eqref{3.43}.
Thus, we get   $H_{0}^{1}(\Omega) \subset \mathrm{D}(A_{i}).$
Let us find a representation for the  operator $G_{i}.$
Consider the operators
 $$
 B_{i}f(Q)=\!\int_{0}^{r}\!\!f(P_{i}+\mathbf{e }[r-t])dt,\,f\in L_{2}(\Omega),\,i=1,2,...n.
 $$
It is obvious that
\begin{equation}\label{3.44}
\int\limits_{\Omega }  A_{i}\left(B_{i}  \mathcal{T}  f \cdot g\right) dQ =\int\limits_{\Omega } A_{i}B_{i}\mathcal{T}f \cdot g\,  dQ +\int\limits_{\Omega }  B_{i} \mathcal{T} f \cdot A_{i}g  \,dQ ,\, f\in C^{2}(\bar{\Omega}),g\in C^{\infty}_{0}( \Omega ).
\end{equation}
Using the   divergence  theorem, we get
\begin{equation}\label{3.45}
\int\limits_{\Omega }  A_{i}\left(B_{i}\mathcal{T}f \cdot g\right) \,dQ=\int\limits_{S}(\mathbf{e_{i}},\mathbf{n})_{\mathbb{E}^{n}}(B_{i}\mathcal{T}f\cdot  g)(\sigma)d\sigma,
\end{equation}
where $S$ is the surface of $\Omega.$
Taking into account   that $ g(S)=0$ and combining   \eqref{3.44},\eqref{3.45}, we get
\begin{equation}\label{3.46}
    -\int\limits_{\Omega }  A_{i}B_{i}\mathcal{T} f\cdot \bar{g} \, dQ=  \int\limits_{\Omega } B_{i}\mathcal{T} f\cdot   \overline{A_{i}  g} \, dQ,\, f\in C^{2}(\bar{\Omega}),g\in C^{\infty}_{0}( \Omega ).
\end{equation}
Suppose that $f\in H^{2}(\Omega),$ then    there exists a sequence $\{f_{n}\}_{1}^{\infty}\subset C^{2}(\bar{\Omega})$ such that
$ f_{n}\stackrel{ H^{2}}{\longrightarrow}  f$ (see \cite[p.346]{firstab_lit:Smirnov5}).  Using this fact, it is not hard to prove that
$\mathcal{T}f_{n}\stackrel{L_{2}}{\longrightarrow} \mathcal{T}f.$  Therefore  $A_{i}B_{i}\mathcal{T}f_{n}\stackrel{L_{2}}{\longrightarrow} \mathcal{T}f,$  since  $A_{i}B_{i}\mathcal{T}f_{n}=\mathcal{T}f_{n}.$ It is also clear that   $B_{i}\mathcal{T}f_{n}\stackrel{L_{2}}{\longrightarrow} B_{i}\mathcal{T}f,$ since $B_{i}$ is continuous (see proof of Lemma \ref{L3.7}).
Using these facts, we can extend relation \eqref{3.46} to the following
\begin{equation*}
  -\int\limits_{\Omega }  \mathcal{T} f \cdot  \bar{g} \, dQ= \int\limits_{\Omega } B_{i}\mathcal{T}f\,  \overline{A_{i}g} \, dQ,\; f\in \mathrm{D}(\mathcal{T}),\,g\in C_{0}^{\infty}(\Omega).
\end{equation*}
Note, that it was previously proved that $A_{i}^{-1} \subset B_{i}$ (see the proof of Lemma \ref{L3.7}),   $H_{0}^{1}(\Omega) \subset \mathrm{D}(A_{i}).$ Hence $G_{i} A_{i}f=B_{i}\mathcal{T} f,\, f\in  \mathrm{D}(\mathcal{T}),$  where $G_{i}:=B_{i}\mathcal{T}B_{i}.$  Using  this fact    we can rewrite relation \eqref{3.47} in a   form
\begin{equation}\label{3.47}
  -\int\limits_{\Omega }  \mathcal{T} f \cdot  \bar{g} \, dQ= \int\limits_{\Omega } G_{i} A_{i}f\,  \overline{A_{i}g} \, dQ,\; f\in \mathrm{D}(\mathcal{T}),\,g\in C_{0}^{\infty}(\Omega).
\end{equation}
Note that   in accordance with Lemma \ref{L3.7}, we have
$$
\forall g\in \mathrm{D}(A_{i}),\,\exists \{g_{n}\}_{1}^{\infty}\subset C^{\infty}_{0}( \Omega ),\,    g_{n}\xrightarrow[      A_{i}       ]{}g.
$$
Therefore, we can extend   relation \eqref{3.47} to the following
\begin{equation}\label{3.48}
     -\int\limits_{\Omega }  \mathcal{T} f \cdot  \bar{g} \, dQ= \int\limits_{\Omega } G_{i} A_{i}f \, \overline{A_{i}g} \, dQ,\; f\in \mathrm{D}(\mathcal{T}),\,g\in \mathrm{D}(A_{i}).
\end{equation}
Relation \eqref{3.48} indicates that $G_{i} A_{i}f\in \mathrm{D}( A_{i} ^{\ast})$   and it is clear that $  -\mathcal{T}\subset A_{i} ^{\ast}G_{i}A_{i}.$ On the other hand in accordance with  Chapter VI, Theorem 1.2    \cite{firstab_lit: Berezansk1968}, we have that $-\mathcal{T}$ is a closed operator.
Using the divergence theorem we get
$$
- \int\limits_{\Omega}D_{j} ( a^{ij} D_{i}f) \bar{g} dQ=
  \int\limits_{\Omega}    a^{ij}  D_{i}f\,     \overline{D_{j}g}    dQ,\,f\in C^{2}(\Omega),\,g\in C^{\infty}_{0}(\Omega).
$$
Passing to the limit at the left-hand and right-hand side of  the last inequality, we  can   extend it to the following
$$
- \int\limits_{\Omega}D_{j} ( a^{ij} D_{i}f)\, \bar{g} dQ=
  \int\limits_{\Omega}    a^{ij}  D_{i}f\,     \overline{D_{j}g}    dQ,\,f\in H^{2}(\Omega),\,g\in H^{1}_{0}(\Omega).
$$
Therefore, using the  uniformly elliptic property of the operator $-\mathcal{T},$ we get
 \begin{equation}\label{3.49}
-\mathrm{Re} \left(\mathcal{T}f,  f \right)_{L_{2}} \geq \gamma_{a}\int\limits_{\Omega}     \sum\limits_{i=1}^{n}   |  D_{i}f |^{2}    \,    dQ=
\gamma_{a}\|f\|^{2}_{H_{0}^{1}},\,f\in \mathrm{D}(\mathcal{T}).
\end{equation}
Using    the  Poincar\'{e}-Friedrichs inequality, we get
$
-\mathrm{Re} \left(\mathcal{T}f,  f \right)_{L_{2}}\geq C\|f\|^{2}_{L_{2}},\,f\in \mathrm{D}(\mathcal{T}),
$
Applying the Cauchy-Schwarz inequality to the left-hand side, we can easily deduce that the conditions of Lemma \ref{L3.1} are satisfied.
 Thus,   the operator  $-\mathcal{T}$ is {\it m}-accretive. In particular, it means that there does not exist an accretive extension of the operator $-\mathcal{T}.$ Let us prove that $A_{i} ^{\ast}G_{i}A_{i}$ is accretive, for this purpose combining  \eqref{3.47},\eqref{3.49}, we get
$
  \left(G_{i} A_{i}f , A_{i}f \right)_{L_{2}}  \geq 0,  f \in C_{0}^{\infty}  (\Omega).
$
Due to the relation $\tilde{A}_{0}=A,$ proved in Lemma \ref{L3.7}, the previous  inequality can be easily extended  to
$
 \left(G_{i} A_{i}f , A_{i}f \right)_{L_{2}}  \geq 0,  f \in \mathrm{D}(G_{i} A_{i}).
$
In its own    turn, it implies that
$
\left( A^{\ast}_{i}G_{i} A_{i}f ,f \right)_{L_{2}}  \geq 0,  f \in \mathrm{D}(A^{\ast}_{i}G_{i} A_{i}),
$
thus  we have obtained the desired result. Therefore, taking into account the facts given above,  we deduce that
  $-\mathcal{T}= A_{i}  ^{\ast}G_{i}A_{i},\,i=1,2,...\,n\,$ and     obtain   \eqref{3.42}.
Applying the Cauchy-Schwarz inequality to the inner sums, then using   Corollary \ref{C3.2}, we obtain
 \begin{equation*}
\left|\int\limits_{\Omega }  \mathcal{T} f \cdot  \bar{g} \, dQ\right|=\left|\int\limits_{\Omega}    a^{ij}  D_{i}f\,     \overline{D_{j}g}    dQ\right|\leq a_{1}
 \int\limits_{\Omega}      \|\nabla f\| _{\mathbb{E}^{n}}\,     \|\nabla g\| _{\mathbb{E}^{n}}    dQ \leq
 $$
 $$
 \leq a_{1}\|f\|_{H_{0}^{1}}\|g\|_{H_{0}^{1}}\leq C\|f\|_{\mathfrak{H}^{n}_{A}}\|g\| _{\mathfrak{H}^{n}_{A}} ,\; f,g\in C_{0}^{\infty}(\Omega),
 $$
 where
 $$
 a_{1}=\sup\limits_{Q\in \bar{\Omega}} \sqrt{\sum\limits_{i,j=1}^{n} |a^{ij}(Q)|^{2}}.
\end{equation*}
 On the other hand, applying  \eqref{3.43},\eqref{3.49} we get
\begin{equation*}
-\mathrm{Re}(\mathcal{T} f,f) \geq C\|f\|^{2}_{\mathfrak{H}^{n}_{A}},\,f\in C^{\infty}_{0}(\Omega).
\end{equation*}
  The proof is complete.
\end{proof}
Thus, by virtue of     Corollary \ref{C3.2} and Theorem \ref{T3.6}, we are able to   claim that   Theorem \ref{T2.6}  can be applied to the operator $-\mathcal{T}.$\\

\subsubsection{Fractional integro-differential operator}

In this paragraph we assume that $\alpha\in (0,1).$ In accordance with the definition given in  the paper  \cite{firstab_lit:1kukushkin2018}, we consider a directional  fractional integral.  By definition, put
$$
 (\mathfrak{I}^{\alpha}_{0+}f)(Q  ):=\frac{1}{\Gamma(\alpha)} \int\limits^{r}_{0}\frac{f (P+t \mathbf{e} )}{( r-t)^{1-\alpha}}\left(\frac{t}{r}\right)^{n-1}\!\!\!\!dt,
 \;f\in L_{p}(\Omega),\;1\leq p\leq\infty.
$$
Also,     we   consider an auxiliary operator,   the so-called   truncated directional  fractional derivative   (see \cite{firstab_lit:1kukushkin2018}).  By definition, put
\begin{equation*}
 ( \mathfrak{D }^{\alpha}_{d-,\,\varepsilon}f)(Q)=\frac{\alpha}{\Gamma(1-\alpha)}\int\limits_{r+\varepsilon }^{d }\frac{ f (Q)- f(P+\mathbf{e}t)}{( t-r)^{\alpha +1}} dt
 +\frac{f(Q)}{\Gamma(1-\alpha)}(d-r)^{-\alpha},\;0\leq r\leq d -\varepsilon,
 $$
 $$
  ( \mathfrak{D }^{\alpha}_{d-,\,\varepsilon}f)(Q)=      \frac{ f(Q)}{\alpha} \left(\frac{1}{\varepsilon^{\alpha}}-\frac{1}{(d -r)^{\alpha} }    \right),\; d -\varepsilon <r \leq d .
 \end{equation*}
  Now, we can  define  a directional   fractional derivative as follows
 \begin{equation*}
  \mathfrak{D }^{\alpha}_{d-}f=\lim\limits_{\stackrel{\varepsilon\rightarrow 0}{ (L_{p}) }} \mathfrak{D }^{\alpha}_{d-,\varepsilon} f ,\,1\leq p\leq\infty.
\end{equation*}
The properties of these operators  are  described  in detail in the paper  \cite{firstab_lit:1kukushkin2018}.   We suppose  $\mathfrak{I}^{0}_{0+} =I.$ Nevertheless,   this    fact can be easily established  dy virtue of  the reasonings  corresponding to the one-dimensional case and   given in \cite{firstab_lit:samko1987}. We also consider an integral operator with a weighted factor (see \cite[p.175]{firstab_lit:samko1987}) defined by the following formal construction
$$
 \left(\mathfrak{I}^{\alpha}_{0+}\mu f\right) (Q  ):=\frac{1}{\Gamma(\alpha)} \int\limits^{r}_{0}
 \frac{(\mu f) (P+t\mathbf{e})}{( r-t)^{1-\alpha}}\left(\frac{t}{r}\right)^{n-1}\!\!\!\!dt,
$$
where $\mu$ is a real-valued  function.

Consider a linear combination of an uniformly elliptic operator given in Theorem \ref{T3.6}  and
  a composition of a   fractional integro-differential  operator, where the fractional  differential operator is understood as the adjoint  operator  regarding  the Kipriyanov operator  (see  \cite{firstab_lit:kipriyanov1960},\cite{firstab_lit:1kipriyanov1960},\cite{kukushkin2019})
\begin{equation*}
 L  :=-  \mathcal{T}  \, +\mathfrak{I}^{\sigma}_{ 0+}\rho\, \mathfrak{D}  ^{ \alpha }_{d-}
 ,\,\rho\in L_{\infty}(\Omega),\; \sigma\in[0,1),
 $$
 $$
\mathrm{D}( L )  =H^{2}(\Omega)\cap H^{1}_{0}( \Omega ),
\end{equation*}
\begin{teo}\label{T3.7}
 We claim that
\begin{equation*}
L=\frac{1}{n}\sum\limits^{n}_{i=1} A_{i}^{\ast}G_{i}A_{i}+FA_{1}^{\alpha},
\end{equation*}
where $F $ is a bounded operator, $P_{1}:=P,$ and $G_{i}$ are the same as in Theorem \ref{T3.6}. Moreover  if  $ \gamma_{a} $ is sufficiently large in comparison  with $\|\rho\|_{L_{\infty}},$ then
the following relations hold
$$
 \mathrm{Re}(Lf,f)_{L_{2}}\geq C  \|f\|_{\mathfrak{H}^{n}_{A}};\, |(Lf,g)_{L_{2}}|\leq C \|f\|_{\mathfrak{H}^{n}_{A}}\|g\|_{\mathfrak{H}^{n}_{A}},\;f,g\in C_{0}^{\infty}(\Omega).
$$
\end{teo}
\begin{proof}
The proof follows obviously from Theorem 2, Theorem 3  \cite{kukushkin2021a}, Corollary \ref{C3.2}.
\end{proof}
Combining   the fact $\mathfrak{H}^{n}_{A}\subset\subset L_{2}(\Omega)$ established in  Corollary \ref{C3.2} and Theorem \ref{T3.7}, we     claim that  Theorems \ref{T2.6} can be applied to the operator $L.$\\

\section{Remarks}
In this chapter, we   studied  a true   mathematical nature of a differential operator   with a fractional derivative in final terms. We constructed a   model  in terms of the infinitesimal generator of a corresponding semigroup and    successfully applied   spectral theorems. Further, we generalized the obtained results to some class of transforms of  m-accretive operators, what can be treated as  an introduction to the fractional calculus of m-accretive operators. As a concrete theoretical achievement  of the offered  approach, we have  the following results: an  asymptotic equivalence between   the
real component of a  resolvent and the resolvent of the real component was established  for the  class; a  classification,   in accordance with     resolvent  belonging      to   the  Schatten-von Neumann  class, was obtained;
  a  sufficient condition of completeness of the root vectors system were formulated; an asymptotic formula for the eigenvalues was obtained. As an application,
 there were considered cases corresponding to a finite and infinite measure as well as various notions of fractional derivative under the semigroup  theory  point of view, such operators as a  Kipriyanov operator, Riesz potential,  difference operator were involved. The eigenvalue problem for
  a differential operator   with a composition of fractional integro-differential operators  in final terms was solved.

 In addition, note that  minor  results  are also worth noticing such as  a generalization of the well-known von Neumann theorem   (see  the proof of    Theorem  \ref{T3.1}).    In  paragraph \ref{Sub3.3.1},   it  might have been    possible to   consider an unbounded domain  $\Omega$ with some restriction imposed upon    a solid angle  containing  $\Omega,$ due to this natural  way we come to a   generalization of the Kipriyanov operator.   We should add  that  various   conditions, that may be imposed on the operator $F,$ are worth studying separately  since there is a number   of applications in  the theory of fractional differential equations.

\chapter{Root vectors  series expansion   of non-selfadjoint operators}\label{Ch4}

\section{Historical review}

  Generally, the concept origins from the  well-known fact that the   eigenvectors   of the compact selfadjoint operator form a basis in the closure of its range. The question what happens in the case when the operator is non-selfadjoint is rather complicated and  deserves to be considered as a separate part of the spectral theory. Basically, the aim of the mentioned  part of the spectral theory   are   propositions on the convergence of the root vector series in one or another sense to an element belonging to the closure of the operator range. Here, we should note when we  say a sense we mean   Bari, Riesz, Abel (Abel-Lidskii) senses of the series convergence  \cite{firstab_lit:2Agranovich1994},\cite{firstab_lit:1Gohberg1965}. A reasonable question that appears is about minimal conditions that guaranty the desired result, for instance in the mentioned papers    the authors  considered a domain of the parabolic type containing the spectrum of the operator. In the paper \cite{firstab_lit:2Agranovich1994}, non-salfadjoint operators with the special condition imposed on the numerical range of values are considered. The main advantage of this result is a   weak condition imposed upon the numerical range of values comparatively with the sectorial condition (see definition of the sectorial operator). Thus, the convergence in the Abel-Lidskii sense  was established for an operator class wider than the class of sectorial operators. Here, we  make a comparison between results devoted to operators with the discrete spectra and operators with the compact resolvent, for they can be easily reformulated from one to another realm.

  The central idea of this chapter  is to formulate sufficient conditions of the Abel-Lidskii basis property of the root functions system for a sectorial non-selfadjoint operator of the special type. Considering such an operator class, we strengthen a little the condition regarding the semi-angle of the sector, but weaken a great deal conditions regarding the involved parameters. Moreover, the central aim generates  some prerequisites to consider technical peculiarities such as a newly constructed sequence of contours of the power type on the contrary to the Lidskii results \cite{firstab_lit:1Lidskii}, where a sequence of the contours of the  exponential type was considered. Thus,  we clarify   results  \cite{firstab_lit:1Lidskii} devoted to the decomposition on the root vector system of the non-selfadjoint operator. We use  a technique of the entire function theory and introduce  a so-called  Schatten-von Neumann class of the convergence  exponent. Considering strictly accretive operators satisfying special conditions formulated in terms of the norm,   using a  sequence of contours of the power type,  we invent a peculiar  method how to calculate  a contour integral involved in the problem in its general statement.     Finally, we produce applications to differential equations in the abstract Hilbert space.
  In particular,     the existence and uniqueness theorems  for evolution equations   with the right-hand side -- a  differential operator  with a fractional derivative in  final terms are covered by the invented abstract method. In this regard such operator  as a Riemann-Liouville  fractional differential operator,    Kipriyanov operator, Riesz potential,  difference operator are involved.
Note that analysis of the required  conditions imposed upon the right-hand side of the    evolution equations that are in the scope  leads us to   relevance of the central idea of the paper. Here, we should note  a well-known fact  \cite{firstab_lit:Shkalikov A.},\cite{firstab_lit(arXiv non-self)kukushkin2018}   that a particular interest appears in the case when a senior term of the operator at least
    is not selfadjoint,  for in the contrary case there is a plenty of results devoted to the topic wherein  the following papers are well-known
\cite{firstab_lit:1Katsnelson},\cite{firstab_lit:1Krein},\cite{firstab_lit:Markus Matsaev},\cite{firstab_lit:2Markus},\cite{firstab_lit:Shkalikov A.}. The fact is that most of them deal with a decomposition of the  operator  on a sum,  where the senior term
     must be either a selfadjoint or normal operator. In other cases, the  methods of the papers
     \cite{kukushkin2019}, \cite{firstab_lit(arXiv non-self)kukushkin2018} become relevant   and allow us  to study spectral properties of   operators  whether we have the mentioned above  representation or not.   We should remark that the results of  the papers \cite{firstab_lit:2Agranovich1994},\cite{firstab_lit:Markus Matsaev}, applicable to study non-selfadjoin operators,  are   based on the sufficiently strong assumption regarding the numerical range of values of the operator. At the same time,
the methods  \cite{firstab_lit(arXiv non-self)kukushkin2018}   can be  used    in  the  natural way,  if we deal with more abstract constructions formulated in terms of the semigroup theory   \cite{kukushkin2021a}.  The  central challenge  of the latter  paper  is how  to create a model   representing     a  composition of  fractional differential operators   in terms of the semigroup theory.   We should note that motivation arouse in connection with the fact that
  a second order differential operator can be represented  as a some kind of  a  transform of   the infinitesimal generator of a shift semigroup. Here, we should stress   that
  the eigenvalue problem for the operator
     was previously  studied by methods of  theory of functions   \cite{firstab_lit:1Nakhushev1977}.
Thus, having been inspired by   novelty of the  idea  we generalize a   differential operator with a fractional integro-differential composition  in the final terms   to some transform of the corresponding  infinitesimal generator of the shift semigroup.
By virtue of the   methods obtained in the paper
\cite{firstab_lit(arXiv non-self)kukushkin2018}, we   managed  to  study spectral properties of the  infinitesimal generator  transform and obtain   an outstanding result --
   asymptotic equivalence between   the
real component of the resolvent and the resolvent of the   real component of the operator. The relevance is based on the fact that
   the  asymptotic formula  for   the operator  real component  can be  established in most  cases due to well-known asymptotic relations  for the regular differential operators as well as for the singular ones
 \cite{firstab_lit:Rosenblum}. It is remarkable that the results establishing spectral properties of non-selfadjoint operators  allow  us to implement a novel approach regarding   the problem  of the basis property of  root vectors.
   In its own turn, the application  of results connected with the basis property  covers  many   problems in the framework of the theory of evolution  equations.
  The abstract approach to the Cauchy problem for the fractional evolution equation was previously implemented in the papers \cite{firstab_lit:Bazhl},\cite{Ph. Cl}. At the same time, the  main advantage of this paper is the obtained formula for the solution of the evolution equation with the relatively wide conditions imposed upon the right-hand side,  where  the derivative at the left-hand side is supposed to be of the fractional order.
 This problem appeals to many ones that lie  in the framework of the theory of differential equations, for instance in the paper   \cite{L. Mor} the solution of the  evolution equation  can be obtained in the analytical way if we impose the conditions upon the right-hand side. We can also produce a number  of papers dealing  with differential equations which can be studied by this paper abstract methods \cite{firstab_lit:Mainardi F.}, \cite{firstab_lit:Mamchuev2017a}, \cite{firstab_lit:Mamchuev2017}, \cite{firstab_lit:Pskhu},       \cite{firstab_lit:Wyss}. The latter information gives us an opportunity to claim that the offered  approach is undoubtedly novel and relevant.

\section{ Some facts of the entire functions theory}

\subsection{ The growth scale}

To characterize the growth of an entire function $f(z),$ we introduce the functions
$$
M_{f}(r)=\max\limits_{|z|=r}|f(z)|,\,m_{f}(r)=\min\limits_{|z|=r}|f(z)|.
$$
An entire function $f(z)$ is said to be a function of finite order if there exists a positive constant $k$ such that the inequality
 $$
 M_{f}(r)<e^{r^{k}}
 $$
 is valid for all sufficiently large values of $r.$ The greatest lower bound of such numbers $k$ is called the {\it order} of the entire function $f(z).$\\

It follows from the definition that if $\varrho$ is the order of the entire function $f(z),$ and if $\varepsilon$ is an arbitrary positive number, then
\begin{equation*}
e^{r^{\varrho-\varepsilon}}<M_{f}(r)<e^{r^{\varrho+\varepsilon}},
\end{equation*}
where the inequality on the right-hand side is satisfied for all sufficiently large values of $r,$ and the inequality on the left-hand side holds for some sequence $\{r_{n}\}$ of values of $r,$ tending to infinity. It is easy to verify that the previous condition is equivalent to the equation
$$
\varrho=\overline{\lim\limits_{r\rightarrow \infty}}\,\frac{\ln\ln M_{f}(r)}{\ln r},
$$
which is taken as the definition of the order of the function. Further, an inequality that holds for all sufficiently large values of $r$ will be called an {\it asymptotic inequality.}

For functions of the given order a more precise characterization of the growth is given by the type of the function. By the {\it type} $\sigma$ of the entire  function $f(z)$ of the order $\varrho$ we mean the greatest lower bound of positive numbers $A$ for which the following relation  holds asymptotically
 $$
 M_{f}(r)<e^{Ar^{\varrho}}.
 $$
Analogously to the definition of the order it is easy to verify that the type $\sigma$ of a function $f(z)$ of order $\varrho$ is given by the relation
$$
\sigma=\overline{\lim\limits_{r\rightarrow\infty}}\frac{\ln M_{f}(r)}{r^{\varrho}}.
$$

\subsection{  Convergence exponent}

Here,    following  the monograph \cite{firstab_lit:Eb. Levin}, we introduce some notions and facts of the entire function theory. In this subsection, we   use the following notations
$$
G(z,p):=(1-z)e^{z+\frac{z^{2}}{2}+...+\frac{z^{p}}{p}},\,G(z,0):=(1-z).
$$
Consider   an entire function  that has zeros satisfying the following relation for some   $\lambda>0$
\begin{equation}\label{4.1}
 \sum\limits_{n=1}^{\infty}\frac{1}{|a_{n}|^{\lambda}}<\infty.
\end{equation}
In this case, we denote by $p$ the smallest integer number for which the following condition holds
\begin{equation}\label{4.2}
\sum\limits_{n=1}^{\infty}\frac{1}{|a_{n}|^{p+1}}<\infty .
\end{equation}
 It is clear that $0\leq p<\lambda.$ It is proved that under the assumption \eqref{4.1}  the   infinite product
 \begin{equation}\label{4.3}
  \prod\limits_{n=1}^{\infty} G\left(\frac{z}{a_{n}},p\right)
\end{equation}
 is uniformly convergent, we will call it a canonical product and call $p$ the genus of the canonical product.
By the   {\it convergence exponent} $\rho$ of the sequence
$
\{a_{n}\}_{1}^{\infty}\subset \mathbb{C},\,a_{n}\neq 0,\,a_{n}\rightarrow \infty
$
 we mean the greatest lower bound for such numbers $\lambda$ that the   series \eqref{4.1} converges.
 Note that if $\lambda$ equals to a convergence  exponent then series \eqref{4.1} may or  may not be convergent. For instance, the sequences $a_{n}= 1/n^{\lambda}$ and $1/(n\ln^{2} n)^{\lambda}$ have the same convergence exponent $\lambda=1,$ but in the first case the series \eqref{4.1} is divergent when $\lambda=1$ while in the second one it is convergent. In this paper, we have a special interest regarding the first case. Consider the following obvious relation between the convergence exponent $\rho$ and the genus $p$ of the corresponding canonical product $p \leq\rho\leq p+1.$ It is clear that if $\rho=p,$   then    the series  \eqref{4.1} diverges for $\lambda=\rho,$ while $\rho=p+1$ means that the series converges  (in accordance with the definition of $p$). In the monograph \cite{firstab_lit:Eb. Levin},   a more precise characteristic of the density of the sequence $\{a_{n}\}_{1}^{\infty}$ is considered than the convergence exponent. Thus, there is defined  a so-called counting function $n(r)$ equals to a number of points of the sequence in the circle $|z|<r.$ By upper density of the sequence, we call a number
 $$
 \Delta=\overline{\lim\limits_{r\rightarrow\infty}} n(r)/r^{\rho}.
 $$
 If a limit exists in the ordinary sense (not in the  sense of the upper limit),  then $\Delta$ is called the density. Note that it is proved in Lemma 1
  \cite{firstab_lit:Eb. Levin} that
 $$
   \lim\limits_{r\rightarrow\infty}  n(r)/r^{\rho+\varepsilon}\rightarrow 0,\,\varepsilon>0.
 $$
 We need the following fact (see \cite{firstab_lit:Eb. Levin} Lemma 3).
 \begin{lem}\label{L4.1}
 If the series \eqref{4.2} converges, then the corresponding infinite product \eqref{4.3}  satisfies  the following inequality in the entire complex plane
 $$
 \ln\left|  \prod\limits_{n=1}^{\infty} G\left(\frac{z}{a_{n}},p\right)\right|\leq Cr^{p}\left(\int\limits_{0}^{r}\frac{n(t)}{t^{p+1}}dt+r\int\limits_{r}^{\infty}\frac{n(t)}{t^{p+2}}dt\right),\,r:=|z|.
 $$
\end{lem}
Using this result, it is not hard to prove a relevant  fact mentioned in the monograph \cite{firstab_lit:Eb. Levin}. Since it has a principal role in the further narrative, then  we formulate it as a lemma  in terms of the   density.
\begin{lem}\label{L4.2}
Assume that the following series  converges for some values $\lambda>0,$ i.e.
$$
\sum\limits_{n=1}^{\infty}\frac{1}{|a_{n}|^{\lambda}}<\infty.
$$
Then  the following  relation holds
\begin{equation}\label{4.4}
\left|\prod\limits_{n=1}^{\infty} G\left(\frac{z}{a_{n}},p\right)\right|\leq e^{\beta(r)r^{\rho_{1}}},\,\beta(r)= r^{p-\rho_{1} }\left(\int\limits_{0}^{r}\frac{n(t)}{t^{p+1 }}dt+
r\int\limits_{r}^{\infty}\frac{n(t)}{t^{p+2  }}dt\right).
\end{equation}
  In the case $\rho_{1}=\rho,$ we have   $\beta(r)\rightarrow 0,$ if at   least  one of the following conditions holds:  the convergence exponent $\rho<\lambda$  is   non-integer       and the density equals to zero,  the convergence exponent $\rho=\lambda$  is   arbitrary.    In addition,   the  equality  $\rho=\lambda$ guaranties that the density equals to zero. In the case $\rho_{1}>\rho,$ we claim that $\beta(r)\rightarrow 0,$ without any additional conditions.
\end{lem}
\begin{proof}
Applying  Lemma \ref{L4.1}, we establish relation \eqref{4.4}. Consider a case   when  $\rho<\lambda$ is non-integer.
Taking into account the fact that the density equals to zero, using   L'H\^{o}pital's rule,    we easily obtain
\begin{equation}\label{4.5}
 r^{p-\rho}  \int\limits_{0}^{r}\frac{n(t)}{t^{p+1}}dt \rightarrow 0 ;\;r^{p+1-\rho } \int\limits_{r}^{\infty}\frac{n(t) }{t^{p+2}}dt \rightarrow 0,
\end{equation}
(here we should remark  that if $\rho$ is integer, then $p=\rho$).
Therefore  $\beta(r)\rightarrow 0.$
Consider the case when  $\rho=\lambda,$ then let us rewrite the series \eqref{4.1} in the form of the  Stieltjes integral, we have
$$
\sum\limits_{n=1}^{\infty}\frac{1}{|a_{n}|^{\lambda}}=\int\limits_{0}^{\infty}\frac{d n(t)}{t^{\rho}} .
$$
Using integration by parts formulae, we get
$$
\int\limits_{0}^{r}\frac{d n(t)}{t^{\rho}} =\frac{n(r)}{r^{\rho}} -\frac{n(\gamma)}{C^{\rho}}+\rho\int\limits_{0}^{r}\frac{n(t)}{t^{\rho+1}}dt.
$$
Here,  we should note that there exists a neighborhood of the point zero in which $n(t)=0.$
The latter representation shows us that the following  integral converges, i.e.
$$
\int\limits_{0}^{\infty}\frac{n(t)}{t^{\rho+1}}dt<\infty.
$$
In its own turn, it follows that
$$
\frac{n(r)}{r^{\rho}}=n(r)\rho\int\limits_{r}^{\infty}\frac{1}{t^{\rho+1}}dt<\rho\int\limits_{r}^{\infty}\frac{n(t)}{t^{\rho+1}}dt\rightarrow 0,\,r\rightarrow \infty.
$$
 Using this fact analogously to the above, applying L'H\^{o}pital's rule, we conclude  that  \eqref{4.5} holds if $\rho=\lambda$ is non-integer. If   $\rho=\lambda$ is  integer  then it is clear that  we have $\rho=p+1,$ here we should remind  that it is not possible to assume that $\rho=p$ due to the definition of $p.$ In the case $\rho=p+1,$    using the above  reasonings, we get
$$
 r^{-1}  \int\limits_{0}^{r}\frac{n(t)}{t^{p+1}}dt \rightarrow 0 ;\; \int\limits_{r}^{\infty}\frac{n(t) }{t^{p+2}}dt \rightarrow 0,
$$
from what follows  the fact that $\beta(r)\rightarrow0$. The reasonings related to the  case $\rho_{1}>\rho$ is absolutely analogous,     we  left  the proof to  the reader. The proof is complete.
\end{proof}
\begin{lem} \label{L4.3} We claim that the following implication holds
$$
\ln r\frac{n(r)}{r^{\rho_{1}}}\rightarrow0,\Longrightarrow \beta(r)\ln r \rightarrow 0,\,r\rightarrow\infty,
$$
where
$$
\beta(r)= r^{p-\rho_{1} }\left(\int\limits_{0}^{r}\frac{n(t)}{t^{p+1 }}dt+
r\int\limits_{r}^{\infty}\frac{n(t)}{t^{p+2  }}dt\right),\,\rho_{1}\neq p,  p+1.
$$
\end{lem}
\begin{proof}
 Let us define  auxiliary functions
$$
u_{1}(r):=\ln r\int\limits_{0}^{r}\frac{n(t)}{t^{p+1 }}dt,\,v_{1}(r):=\ln r \int\limits_{r}^{\infty}\frac{n(t)}{t^{p+2  }}dt,\,u_{2}(r):=r^{\rho_{1}-p},\,v_{2}(r):=r^{\rho_{1}-p-1}.
$$
It is clear that
$$
u_{1}'(r):=\frac{1}{r} \int\limits_{0}^{r}\frac{n(t)}{t^{p+1 }}dt +\ln r \frac{n(r)}{r^{p+1 }};\;
v'_{1}(r):=\frac{1}{r} \int\limits_{r}^{\infty}\frac{n(t)}{t^{p+2  }}dt+ \ln r\frac{n(r)}{r^{p+2  }}.
$$
Therefore
\begin{equation}\label{4.6}
\frac{u_{1}'(r)}{u_{2}'(r)}=Cr^{p-\rho_{1}} \int\limits_{0}^{r}\frac{n(t)}{t^{p+1 }}dt +C\ln r \frac{n(r)}{r^{\rho_{1}}};\;\frac{v_{1}'(r)}{v_{2}'(r)}=Cr^{p+1-\rho_{1}} \int\limits_{r}^{\infty}\frac{n(t)}{t^{p+2 }}dt +C\ln r \frac{n(r)}{r^{\rho_{1}}}.
\end{equation}
Having noticed  that
$\beta(r)\ln r=  u_{1} (r)/u_{2} (r)+v_{1} (r)/v_{2} (r)$
and   applying    L'H\^{o}pital's rule, we get
\begin{equation}\label{4.7}
  \beta(r)\ln r\sim  \frac{u_{1}'(r)}{u_{2}'(r)}+ \frac{v_{1}'(r)}{v_{2}'(r)},\,r\rightarrow\infty.
\end{equation}
In the analogous way, we obtain the following implication
\begin{equation}\label{4.8}
\frac{n(r)}{r^{\rho_{1}}} \rightarrow 0,\;\Longrightarrow\left\{r^{p-\rho_{1}} \int\limits_{0}^{r}\frac{n(t)}{t^{p+1 }}dt\rightarrow 0;\;r^{p+1-\rho_{1}} \int\limits_{r}^{\infty}\frac{n(t)}{t^{p+2 }}dt\rightarrow0 \right\}.
\end{equation}
Thus, taking into account the condition
$
\ln r \cdot  n(r)/r^{\rho_{1}} \rightarrow0,
$ combining \eqref{4.6}, \eqref{4.7}, \eqref{4.8},
we obtain the desired result.
\end{proof}

Regarding Lemma \ref{L4.3}, we can produce the following example that indicates the relevance of the issue itself.

\begin{ex}\label{Ex4.1}
  There exists a sequence $\{a_{n}\}_{1}^{\infty}$ such that  density equals to zero, moreover
$$
\sum\limits_{n=1}^{\infty}\frac{1}{|a_{n}|^{\rho}}=\infty,\; \beta(r)\ln r\rightarrow 0,r\rightarrow \infty,
$$
where
$$
\beta(r)= r^{p-\rho }\left(\int\limits_{0}^{r}\frac{n(t)}{t^{p+1 }}dt+
r\int\limits_{r}^{\infty}\frac{n(t)}{t^{p+2  }}dt\right).
$$
We can construct the required sequence supposing $n(r)\sim r^{\rho_{1}}(\ln r \cdot \ln\ln r )^{-1},\,\rho_{1}\in \mathbb{R}^{+} \setminus \mathbb{N}.
$
It is clear that we can represent partial sums of series \eqref{4.1} due to the   Stieltjes  integral
$$
 \sum\limits_{n=1}^{k}\frac{1}{|a_{n}|^{\lambda}}= \int\limits_{0}^{r(k) }\frac{d n(t)}{t^{\lambda}},\,\lambda\geq\rho_{1}.
$$
Thus,  the sequence $\{a_{n}\}_{1}^{\infty}$ is defined by the function $n(r).$
Applying the integration by parts formulae, we get
$$
 \int\limits_{0}^{r }\frac{d n(t)}{t^{\lambda}}  = \frac{n(r)}{r^{\lambda}} - \frac{1}{|a_{1}|^{\lambda}} +\lambda\int\limits_{0}^{r}\frac{n(t)}{t^{\lambda+1}}dt .
$$
  We can easily establish the fact  the last integral diverges  when $\lambda=\rho_{1} ,\;r\rightarrow\infty,$ we have
$$
 \int\limits_{0}^{r}\frac{n(t)}{t^{\rho_{1} +1}}dt\geq C \int\limits_{|a_{1}|}^{r}\frac{ dt}{t\ln t\cdot \ln\ln t }= \ln\ln\ln r-C.
$$
On the other hand, we have
$$
 \int\limits_{0}^{\infty}\frac{n(t)}{t^{\lambda+1}}dt\leq C \int\limits_{|a_{1}|}^{\infty}\frac{ dt}{t^{1+\lambda-\rho_{1}}\ln t\cdot \ln\ln t }< \infty ,\; \lambda>\rho_{1}.
$$
Taking into account the fact
$
   n(r)/r^{\rho_{1}}  \rightarrow0,\,r\rightarrow\infty,
$
we conclude that   the series \eqref{4.1} diverges if $\lambda=\rho_{1}$ and   converges  if $\lambda=\rho_{1}+\varepsilon,\,\varepsilon>0.$ Therefore, the convergence exponent equals to $\rho_{1},$ i.e. $\rho=\rho_{1},$     the density   equals to zero.
Let us  prove the fact   $\beta(r)\ln r\rightarrow 0,$  for this purpose in accordance with Lemma \ref{L4.3},
  it suffices to show that
$$
\ln r \frac{n(r)}{r^{\rho}}\rightarrow0,\;r\rightarrow\infty,
$$
by direct substitution, we get
$$
\ln r \frac{n(r)}{r^{\rho}}\leq C    (\ln\ln r)^{-1} \rightarrow0,\;r\rightarrow\infty.
$$
Thus, we   obtain  the desired result.
\end{ex}

Bellow, we refer to the Theorem 13 \cite{firstab_lit:Eb. Levin} (Chapter I, $\S$ 10) that gives us a representation of the entire function of the finite order. To avoid the any sort of inconvenient form of writing, we will also  call by a root a zero of the entire function.
\begin{teo}\label{T4.1} The entire function $f$ of the finite order $\varrho$ has the following representation
$$
f(z)=z^{m}e^{P(z)}\prod\limits_{n=1}^{\omega}G\left(\frac{z}{a_{n}};p\right),\;\omega\leq \infty,
$$
where $a_{n}$ are non-zero roots of the entire function, $p\leq\varrho,\;P(z)$ is a polynomial, $\mathrm{deg}\, P(z)\leq \varrho,\;m$ is a multiplicity of the zero root.
\end{teo}
The infinite product represented  in   Theorem \ref{T4.1} is called  a  canonical product of the entire function.

\subsection{Lower bounds for entire functions }

\subsubsection{Cartan estimate}
The theorem represented bellow is a consequence of the  Cartan estimate (see  \S 7, Chapter I \cite{firstab_lit:Eb. Levin}  ) which has an essential lack for it holds not for the involved set of the complex plane but for restricted set  except for the so-called exceptional sets of circles. One more disadvantage is a rather rough estimate which does not require any conditions on the distribution  of the zeros of an entire function but is far from creating a fundamental base for a peculiar result.

\begin{teo}\label{T4.2} Assume that the function of the complex variable $f(z)$ is holomorphic within the circle $|z|<2eR, \,f(0)=1$  and $\eta$ is an arbitrary positive number less than  or equal  to  $3e/2.$ Then inside the circle $|z|\leq R$ but outside the square covered by  the exceptional set of circulus with the sum of radii  less than $4\eta R,$ the following estimate holds
$$
\ln|f(z)|>-\left(2+\ln\frac{3e}{2\eta}\right)\ln M_{f}(2eR).
$$
\end{teo}

\subsubsection{Relation between maximum and minimum}

The following theorem  (Theorem 30, paragraph 18, Chapter I  \cite{firstab_lit:Eb. Levin}) gives us an instrument to  estimate entire functions from bellow of the order less than one, however  as a lack we can stress  inability  to describe a set of the complex plain where the estimate holds.

\begin{teo}\label{T4.3}Assume that the order $\varrho$ of an entire function less than one, then there exists such a sequence $\{r_{n}\},\;r_{n}\uparrow\infty$ that
$$
\forall\varepsilon>0,\,\exists N(\varepsilon):\;m_{f}(r_{n})>\left[M_{f}(r_{n})\right]^{\cos \pi \varrho-\varepsilon},\,n>N(\varepsilon).
$$
\end{teo}

\subsubsection{ Proximate order and angular density of zeros}

The scale of the growth admits   further clarifications. As a simplest generalization E.L.  Lindel\"{o}f made a comparison $M_{f}(r)$ with the functions of the type
$$
r^{\varrho}\ln^{\alpha_{1}}r \ln^{\alpha_{2}}_{2}r...\ln^{\alpha_{n}}_{n}r,
$$
where $\ln _{j}r=\ln\ln_{j-1}r,\;\alpha_{j}\in \mathbb{R},\,j=1,2,...,n.$ In order to make the further generalization, it is natural (see \cite{firstab_lit:Eb. Levin}) to define a class of the functions $L(r)$ having  {\it low growth}  and  compare $\ln M_{f}(r)$  with $r^{\varrho}L(r).$ Following the idea, G. Valiron introduced a notion of proximate order of the growth of the entire function $f,$ in accordance with which a function $\varrho(r),$  satisfying the following conditions
$$
\lim\limits_{r\rightarrow \infty}\varrho(r)=\varrho;\,\lim\limits_{r\rightarrow \infty}r\varrho'(r)\ln r=0,
$$
is said to be proximate order if the following relation holds
$$
\sigma_{f}=\overline{\lim \limits_{r\rightarrow \infty}} \,\frac{\ln M_{f}(r)}{r^{\varrho(r)}},\,0<\sigma_{f}<\infty.
$$
In this case the value $\sigma_{f}$ is said to be a type of the function $f$ under the proximate order $\varrho(r).$\\

To guaranty some technical results we need to consider a class of  entire functions  whose zero distributions have a certain type of regularity. We follow the monograph  \cite{firstab_lit:Eb. Levin} where the regularity of the distribution of the zeros is characterized by a certain type of density of the set of zeros.

We will say that the set $\Omega$ of the complex plane has an {\it angular density of index} $$\;\xi(r)\rightarrow\xi,\,r\rightarrow\infty,$$
if for an arbitrary set of values $\phi$ and $\psi\;(0<\phi<\psi\leq 2\pi),$ maybe except of   denumerable sets, there exists the limit
\begin{equation*}
\Delta(\phi,\psi)=\lim\limits_{r\rightarrow \infty}\frac{n(r,\phi,\psi)}{r^{\xi(r)}},
\end{equation*}
where $n(r,\phi,\psi)$ is the number of points of the set $\Omega$ within the sector $|z|\leq r,\;\phi< \mathrm{arg} z<\psi.$ The quantity $\Delta(\phi,\psi)$ will be called the angular density of the set  $\Omega$ within the sector $\phi< \mathrm{arg} z<\psi.$   For a fixed $\phi,$ the relation
$$
\Delta(\psi)-\Delta(\phi)=\Delta(\phi,\psi)
$$
determines, within the additive constant, a nondecreasing function $\Delta(\psi).$  This function is defined for all values of $\psi,$ may be except for a denumerable set of values. It is shown in the monograph \cite[p. 89]{firstab_lit:Eb. Levin}  that the exceptional values of $\phi$ and $\psi$ for which there does  not exist an angular density must be the points of discontinuity of the function $\Delta(\psi).$  A set will be said to be {\it regularly  distributed} relative to $\xi(r)$  if it has an angular density $\xi(r)$ with $\xi$ non-integer.

The asymptotic equalities which we will establish are related to the order of growth. By the asymptotic equation
$$
f(r)\approx \varphi(r)
$$
we will mean the fulfilment of the following condition
$$
[f(r)-\varphi(r)]/r^{\varrho(r)}\rightarrow0,\,r\rightarrow\infty.
$$

Consider the following conditions allowing us to solve technical problems related to estimation of contour integrals.    \\

\noindent $(\mathrm{I})$ There exists a value $d>0$ such that circles of  radii
$$
r_{n}=d|a_{n}|^{1-\frac{\varrho(|a_{n}|)}{2}}
$$
with the centers situated at the points $a_{n}$ do not intersect each other, where $a_{n}.$\\

\noindent $(\mathrm{II})$ The points $a_{n}$ lie inside angles with a common vertex at the origin but with no other points in common, which are such that if one arranges the points of the set $\{a_{n}\}$ within any one of these angles in the order of increasing moduli, then for all points which lie inside the same angle the following relation holds
$$
|a_{n+1}|-|a_{n}|>d|a_{n}|^{1-\varrho(|a_{n}|)},\,d>0.
$$
 The circles $|z-a_{n}|\leq r_{n}$ in the first case, and $|z-a_{n}|\leq d |a_{n}|^{1-\varrho(|a_{n}|)}$ in the second case, will be called the exceptional  circles.\\

The following theorem is a central point of the study. Bellow for the reader convenience, we present   the  Theorem 5 \cite{firstab_lit:Eb. Levin} (Chapter II, $\S$ 1) in the slightly changed form.

\begin{teo}\label{T4.4} Assume that the entire  function $f$  of the proximate order $\varrho(r),$    where $\varrho$ is not integer, is represented by its canonical product i.e.
$$
f(z)= \prod\limits_{n=1}^{\infty}G\left(\frac{z}{a_{n}};p\right),
$$
the set of zeros is regularly distributed relative to the proximate order and  satisfies one of the conditions $(\mathrm{I})$ or $(\mathrm{II}).$    Then outside of the exceptional set of circulus   the entire  function  satisfies the following  asymptotical inequality
$$
\ln |f(re^{i\psi})|\approx H(\psi)r^{\varrho(r)},
$$
where
$$
H(\psi):=\frac{\pi}{\sin \pi \varrho}\int\limits_{\psi-2 \pi}^{\psi}   \cos \varrho (\psi-\varphi-\pi)  d\Delta(\varphi).
$$
\end{teo}
The following lemma gives us a key for the technical part of being constructed theory. Although it does not contain implications of any subtle sort it is worth being presented in the expanded form for the reader convenience.
\begin{lem}\label{L4.4}
Assume that $ \varrho\in (0,1/2]$ then the  function $H(\psi)$ is positive if $ \psi\in(-\pi,\pi).$
\end{lem}
\begin{proof}
Taking  into account the facts $\cos \varrho (\psi-\varphi-\pi)=\cos \varrho (|\psi-\varphi|-\pi),\,\psi-2\pi<\varphi<\psi,\;\;\cos \varrho (|\psi-\varphi|-\pi)=\cos \varrho (|\psi-(\varphi+2\pi)|-\pi),$ we obtain the following form
$$
H(\psi):=\frac{\pi}{\sin \pi \varrho}\int\limits_{0}^{2 \pi}   \cos \varrho (|\psi-\varphi|-\pi)  d\Delta(\varphi).
$$
Having noticed the following correspondence between sets $\varphi\in [0,\psi]\Rightarrow\xi\in [\varrho(\psi-\pi),-\varrho\pi],$ $\varphi\in [\psi,\psi+\pi]\Rightarrow\xi\in [ -\varrho\pi,0],$ $\varphi\in [\psi+\pi, 2\pi]\Rightarrow\xi\in [  0,\varrho(\pi-\psi)],$ where $\xi:=\varrho (|\psi-\varphi|-\pi),$ we conclude that $\cos \varrho (|\psi-\varphi|-\pi) \geq 0,\,\varphi\in [0,2\pi].$ Taking into account the fact that the function $\Delta(\varphi)$ is non-decreasing, we obtain the desired result.
\end{proof}

\begin{lem}\label{L4.5}  Assume that the entire function $f$ is of the proximate   order $\varrho(r),\,\varrho\in (0,1/2],$    maps the ray  $\mathrm{arg}\,z=\theta_{0}$   within  a sector $\mathfrak{L}_{0}(\zeta),\,0<\zeta<\pi/2,$   the set of zeros is regularly distributed relative to the proximate order and  satisfies one of the conditions $(\mathrm{I})$ or $(\mathrm{II}),$   there exists $\varsigma>0$ such that  the    angle  $ \theta_{0}-\varsigma<arg\,z  <\theta_{0}+\varsigma  $    do not contain the zeros  with the sufficiently large absolute value.
 Then,  for a sufficiently large value $r,$ the following relation holds
$$
\forall \varepsilon>0,\,\exists N(\varepsilon):\mathrm{Re} f(z)>    e^{ (H(\theta_{0})-\varepsilon)r^{\varrho(r)}},\,r>N(\varepsilon),\, \mathrm{arg}\,z=\theta_{0}.
$$
\end{lem}
\begin{proof} Using Theorem \ref{T4.1}, we obtain the following representation
$$
f(z)=Cz^{m} \prod\limits_{n=1}^{\infty}G\left(\frac{z}{a_{n}};p\right) ,
$$
here we should remark that    $\mathrm{deg} P(z)=0.$   Let us show that the proximate order of the canonical product of the entire function is the same, we have
$$
M_{f}(r)=Cr^{m} M_{F}(r),\,F(z)=\prod\limits_{n=1}^{\infty}G\left(\frac{z}{a_{n}};p\right).
$$
Therefore in accordance with the definition of proximate order, we have
$$
 \overline{\lim \limits_{r\rightarrow \infty}}  \left\{\frac{m\ln r+\ln C }{r^{\varrho(r)}}  +\frac{\ln M_{F}(r)}{ r^{\varrho(r)}}\right\}  =\sigma_{f},\;0<\sigma_{f}<\infty,
$$
from what follows easily the fact $0<\sigma_{F}<\infty,$ moreover $\sigma_{F}=\sigma_{f}.$
 Note that due to the condition that guarantees that the image of the ray  $\mathrm{arg}\,z=\theta_{0}$ belongs to a sector in the right half-plane, we get
$$
\mathrm{Re}f(z)\geq (1+\tan \zeta)^{-1/2} |f(z)|,\,r=|z|,\; \mathrm{arg}\,z=\theta_{0}.
$$
Applying Theorem \ref{T4.4}  we conclude, that excluding  the intersection of the  exceptional set of circulus with the ray  $ \mathrm{arg}\,z=\theta_{0},$  the following relation holds for an arbitrary small $\varepsilon>0$ and the corresponding  sufficiently large values $r$
$$
|f(z)|=C r^{m} \left|\prod\limits_{n=1}^{\infty}G\left(\frac{z}{a_{n}};p\right)\right|\geq  r^{m}e^{ (H(\theta_{0})-\varepsilon)r^{\varrho(r)}},
$$
 where $H(\theta_{0})>0$ in accordance with Lemma \ref{L4.4}.  It is clear that if we show that the   intersection of the ray  $\mathrm{arg}\,z=\theta_{0}$ with the exceptional set of circulus  is empty, then we complete the proof. Note that the character of the zeros distribution   allows us to claim that is true. In accordance with the lemma conditions,  it suffices to  consider the neighborhoods  of the zeros defined as follows    $|z-a_{n}|< d |a_{n}|^{1-\varrho(|a_{n}|)},\,|z-a_{n}|< d |a_{n}|^{1-\varrho(|a_{n}|)/2}$ and note that  $0< \varrho(|a_{n}|)<1$ for a sufficiently large number $n\in \mathbb{N},$ since
 $
  \varrho(|a_{n}|)\rightarrow\varrho,\,n\rightarrow\infty.
 $
 Here, we ought to remind that the zeros are arranged in order with their absolute value growth.
Thus, using simple properties of the power function with the positive exponent less than one, we obtain the fact that the intersection of the  exceptional set of circulus with the ray  $ \mathrm{arg}\,z=\theta_{0}$ is empty for a sufficiently large $n\in \mathbb{N}.$
\end{proof}

\section{ Abell-Lidskii   Series expansion}

 In this subsection, we reformulate  results obtained by Lidskii \cite{firstab_lit:1Lidskii} in a more  convenient  form applicable to the reasonings of this paper.   However,  let us begin our narrative.  Throughout the paper we well consider a sectorial  operator $W$ with a discrete spectrum which inverse $B=W^{-1}$ (compact operator) belongs to the Schatten class $\mathfrak{S}_{p},\;0<p<\infty.$ We denote the eigenvalues $\lambda_{n}:=\lambda_{n}(W),\,\mu_{n}:=\mu_{n}(B),$ i.e. $\lambda_{n}=1/\mu_{n},\,n\in \mathbb{N}.$ In accordance with the Hilbert theorem
  (see \cite{firstab_lit:Riesz1955}, \cite[p.32]{firstab_lit:1Gohberg1965})   the spectrum of an arbitrary  compact operator $B$  consists of the so-called normal eigenvalues, it gives us an opportunity to consider a decomposition to a direct sum of subspaces
\begin{equation}\label{4.9}
 \mathfrak{H}=\mathfrak{N}_{q} \oplus  \mathfrak{M}_{q},
 \end{equation}
where both  summands are   invariant subspaces regarding the operator $B,$  the first one is  a finite dimensional root subspace corresponding to the eigenvalue $\mu_{q}$ and the second one is a subspace  wherein the operator  $B-\mu_{q} I$ is invertible.  Let $n_{q}$ is a dimension of $\mathfrak{N}_{q}$ and let $B_{q}$ is the operator induced in $\mathfrak{N}_{q}.$ We can choose a basis (Jordan basis) in $\mathfrak{N}_{q}$ that consists of Jordan chains of eigenvectors and root vectors  of the operator $B_{q}.$  Each chain $e_{q_{\xi}},e_{q_{\xi}+1},...,e_{q_{\xi}+k},\,k\in \mathbb{N}_{0},$ where $e_{q_{\xi}},\,\xi=1,2,...,m $ are the eigenvectors  corresponding   to the  eigenvalue $\mu_{q}$   and other terms are root vectors,   can be transformed by the operator $B$ in  accordance  with  the following formulas
\begin{equation}\label{4.10}
Be_{q_{\xi}}=\mu_{q}e_{q_{\xi}},\;Be_{q_{\xi}+1}=\mu_{q}e_{q_{\xi}+1}+e_{q_{\xi}},...,Be_{q_{\xi}+k}=\mu_{q}e_{q_{\xi}+k}+e_{q_{\xi}+k-1}.
\end{equation}
Considering the sequence $\{\mu_{q}\}_{1}^{\infty}$ of the eigenvalues of the operator $B$ and choosing a  Jordan basis in each corresponding  space $\mathfrak{N}_{q},$ we can arrange a system of vectors $\{e_{i}\}_{1}^{\infty}$ which we will call a system of the root vectors or following  Lidskii  a system of the major vectors of the operator $B.$
Assume that  $e_{1},e_{2},...,e_{n_{q}}$ is  the Jordan basis in the subspace $\mathfrak{N}_{q},$ let us prove that  (see \cite[p.14]{firstab_lit:1Lidskii})  there exists a  corresponding biorthogonal basis $g_{1},g_{2},...,g_{n_{q}}$ in the subspace $\mathfrak{M}_{q}^{\perp}.$ It is easy to prove that the subspace $\mathfrak{M}_{q}^{\perp}$ has the same dimension equals to  $n_{q}.$ For this purpose,
 assume that the vectors $f_{j}\in\mathfrak{M}_{q}^{\perp},\,j=1,2,...,l$ are linearly independent, then using the decomposition of the space $\mathfrak{H}$ to the direct sum
$f_{j}=a_{j}+b_{j},\,a_{j}\in \mathfrak{N}_{q},\,b_{j}\in \mathfrak{M}_{q},$  we get
 $$
 \sum\limits_{j=1}^{l}f_{j}\alpha_{j}=\sum\limits_{j=1}^{l}a_{j}\alpha_{j}+\sum\limits_{j=1}^{l}b_{j}\alpha_{j}\neq0,
 $$
where   $\{\alpha_{j}\}_{1}^{l}\subset \mathbb{C}$ is an arbitrary non-zero set.  It implies that
$$
\sum\limits_{j=1}^{l}a_{j}\alpha_{j}\neq0,
$$
for if we assume the contrary then we will  come to the contradiction.
 Hence $\dim \mathfrak{M}_{q}^{\perp}\!\leq \dim \mathfrak{N}_{q}.$ Using the same reasonings, we obtain the fact $\dim \mathfrak{M}_{q}^{\perp}\!\geq \dim \mathfrak{N}_{q}.$ Thus, we obtain the desired result. Now, let us choose an element $e_{p}\in\{e_{i}\}_{1}^{n_{q}}$ and consider $(n_{q}-1)$ -- dimensional space $\mathfrak{N}^{(p)}_{q}$ generated by the set $\{e_{i}\}_{1}^{n_{q}}\setminus e_{p},$ then let us choose an arbitrary element $g_{p}\neq0$ belonging to the orthogonal complement of the set
 $
 \mathfrak{N}^{(p)}_{q}\oplus \mathfrak{M}_{q}.
 $
It is clear that $g_{p}\in \mathfrak{M}_{q}^{\perp}$ since in accordance with the given definition the element $g_{p}$ is orthogonal to the set $\mathfrak{M}_{q}.$ It is also clear that in accordance with the definition $(e_{j},g_{p})_{\mathfrak{H}}=0,\,j\neq p.$ Note that  $(e_{p},g_{p})_{\mathfrak{H}}\neq0,$ for if we assume the contrary, then using the decomposition \eqref{4.9}, we get $(g_{p},f)_{\mathfrak{H}}=0,\,f\in \mathfrak{H}$ and as a result we obtain the contradiction, i.e. $g_{p}=0.$ It is clear that we can choose $g_{p}$ so that  $(e_{p},g_{p})_{\mathfrak{H}}=1.$  Let us show that the constructed in this way system of the elements $\{ g_{i}\}_{1}^{n_{q}}$ are linearly independent. It follows easily from the implication
$$
\sum\limits_{j=1}^{n_{q}}g_{j}\alpha_{j}=0,\Rightarrow \sum\limits_{j=1}^{n_{q}}\alpha_{j}(g_{j},e_{p})=0,\,\Rightarrow \alpha_{p}=0,\,  p\in\{1,2,...,n_{q}\}.
$$
Therefore, taking into account the proved above fact $\dim \mathfrak{M}_{q}^{\perp}=n_{q},$ we conclude that the system $\{ g_{i}\}_{1}^{n_{q}}$ is a basis in $  \mathfrak{M}_{q}^{\perp}.$ Let us show that the system   $\{ g_{i}\}_{1}^{n_{q}}$ consists of the Jordan chains of the operator $B^{\ast}$ which correspond  to the Jordan chains  \eqref{4.10}. Note that the space $\mathfrak{M}_{q}^{\perp}$  is an invariant subspace of the operator $B^{\ast}$ since it is orthogonal to the invariant subspace of the operator $B.$ Using the denotation $B^{(q)}$ for the operator $B$   restriction   on the invariant subspace $\mathfrak{N}_{q},$  let us denote by $B^{(q)\ast}$ a restriction of the operator $B^{\ast}$ on the subspace  $\mathfrak{M}_{q}^{\perp}.$ Assume that $\mathrm{B}(\mu_{q})$ is a matrix of the operator $B^{(q)}$ in the basis $\{ e_{i}\}_{1}^{n_{q}},$ then using conditions \eqref{4.10}, we conclude that it has a Jordan   form, i.e. it is  a block diagonal matrix, where each Jordan block is represented by a   matrix in the normal Jordan form, i.e.
$$
\mathrm{B}(\mu_{q})=  \begin{pmatrix} \mathrm{b}_{ q_{1} }(\mu_{q})&  ...&...&  ...&...&...\\
  ...&  \mathrm{b}_{ q_{2} }(\mu_{q})&...&  ...&...&...\\
...& ... & \mathrm{b}_{ q_{3} }(\mu_{q})&...&  ...&... \\
...\\
\\ ...&...&...&...&...& \mathrm{b}_{ q_{m} }(\mu_{q})
\end{pmatrix},\;
\mathrm{b}_{ q_{\xi} }(\mu_{q})=  \begin{pmatrix} \mu_{q}&  1&0&  0&0&...\\
  0&  \mu_{q}&1&  0&0&...\\
0& 0 &\mu_{q}&1&  0&... \\
...\\
\\ 0&...&...&0&0&\mu_{q}
\end{pmatrix},
$$
$$
\dim \mathrm{b}_{ q_{\xi} }(\mu_{q}) =k(q_{\xi})+1, \;\xi=1,2,...,m,
$$
where $m:=m(q)$ is a geometrical multiplicity of the $q$-th eigenvalue,  $k(q_{\xi})+1$ is a number of elements in the $q_{\xi}$-th Jordan chain.
Since we have   $B^{(q)}:\mathfrak{N}_{q}\rightarrow \mathfrak{N}_{q},\;B^{(q)\ast}:\mathfrak{M}_{q}^{\perp}\rightarrow \mathfrak{M}_{q}^{\perp},$ then
$$
B^{(q)}e_{i}=\sum\limits_{j=1}^{n_{q}}\alpha_{ji}e_{j},\;B^{(q)\ast}g_{j}=\sum\limits_{i=1}^{n_{q}}\gamma_{ij}e_{i},
$$
where $\{\alpha_{ji}\},\,\{\gamma_{ij}\},\,i,j=1,2,...,n_{q}$  are the matrices of the operators $B^{(q)},\,B^{(q)\ast}$ in the bases   $\{ e_{i}\}_{1}^{n_{q}},\,\{ g_{i}\}_{1}^{n_{q}}$ respectively. On the other hand, we have the obvious reasonings
$$
\alpha_{ji}=(B^{(q)}e_{i},e_{j})_{\mathfrak{H}}=(B e_{i},e_{j})_{\mathfrak{H}}=(  e_{i},B^{\ast}e_{j})_{\mathfrak{H}}=(  e_{i},B^{(q)\ast}e_{j})_{\mathfrak{H}}=\bar{\gamma}_{ij}.
$$
Therefore, we conclude that  the operator $B^{(q)\ast}$ is represented by a matrix
$$
\mathrm{B}^{\top}(\bar{\mu}_{q})=  \begin{pmatrix} \mathrm{b}^{\top}_{ q_{1} }(\bar{\mu}_{q})&  ...&...&  ...&...&...\\
  ...&  \mathrm{b}^{\top}_{ q_{2} }(\bar{\mu}_{q})&...&  ...&...&...\\
...& ... & \mathrm{b}^{\top}_{ q_{3} }(\bar{\mu}_{q})&...&  ...&... \\
...\\
\\ ...&...&...&...&...& \mathrm{b}^{\top}_{ q_{m} }(\bar{\mu}_{q})
\end{pmatrix}.
$$
Using this representation, we conclude that $\{ g_{i}\}_{1}^{n_{q}}$ consists of the Jordan chains of the operator $B^{\ast}$ which correspond to the Jordan chains  \eqref{4.10} due to the following formula
$$
B^{\ast}g_{q_{\xi}+k}= \bar{\mu}_{q} g_{q_{\xi}+k},\;B^{\ast}g_{q_{\xi}+k-1}= \bar{\mu}_{q} g_{q_{\xi}+k-1}+g_{q_{\xi}+k},...,
B^{\ast}g_{q_{\xi}}= \bar{\mu}_{q} g_{q_{\xi}}+g_{q_{\xi}+1}.
$$
 Let us show that   $\mathfrak{N}_{i}\subset \mathfrak{M}_{j},\,i\neq j$ for this purpose note that in accordance with  the property $P_{\mu_{i}}P_{\mu_{j}}=0,\,i\neq j,$ where $P_{\mu_{i}}$  is a Riesz projector (integral) corresponding to the eigenvalue $\mu_{i}$ (see \cite{firstab_lit:1Gohberg1965} Chapter I \S 1.3), and  the  property
   $P_{\mu_{i}}f=f,\,f\in\mathfrak{N}_{i},$   we have
$$
P_{\mu_{j}}f=P_{\mu_{j}}P_{\mu_{i}}f=0,\;f\in \mathfrak{N}_{i}.
$$
Combining this relation with the decomposition \eqref{4.9},   we obtain the desired result. Now, taking into account relation  \eqref{4.9}, we conclude that  the set  $\{g_{\nu}\}^{n_{j}}_{1},\,j\neq i$  is orthogonal to the set $ \{e_{\nu}\}_{1}^{n_{i}}.$  Gathering the sets $\{g_{\nu}\}^{n_{j}}_{1},\,j=1,2,...,$ we can obviously create a biorthogonal system $\{g_{n}\}_{1}^{\infty}$ with respect to the system of the major vectors of the operator $B.$ It is rather reasonable to call it as  a system of the major vectors of the operator $B^{\ast}.$ Note that if an element $f\in\mathfrak{H}$ allows a decomposition in the strong sense
$$
f=\sum\limits_{n=1}^{\infty}e_{n}c_{n},\,c_{n}\in \mathbb{C},
$$
then by virtue of  the biorthogonal  system existing, we can claim that such a representation is unique. Further, let us come to the previously made  agrement that the vectors in each Jordan chain are arranged in the same order as in \eqref{4.10}, i.e.  at the first place there stands an eigenvector. It is clear that under such an assumption,  we have
$$
c_{q_{\xi}+i}=\frac{(f,g_{q_{\xi}+k-i})}{(e_{q_{\xi}+i},g_{q_{\xi}+k-i})},\,0\leq i\leq k(q_{\xi}),
$$
where $k(q_{\xi})+1$ is a number of elements in the $q_{\xi}$-th Jourdan chain. In particular, if the vector $e_{q_{\xi}}$ is included to the major system solo, there does not exist a root vector corresponding to the corresponding eigenvalue, then
$$
c_{q_{\xi}}=\frac{(f,g_{q_{\xi}})}{(e_{q_{\xi}},g_{q_{\xi}})}.
$$
Note that in accordance with the property of the biorthogonal sequences, we can expect that the denominators equal to one in the previous two relations.
Define the operators
$$
\mathcal{P} _{q}(\alpha,t)f=\sum\limits_{\xi=1}^{m(q)}\sum\limits_{i=0}^{k(q_{\xi})}e_{q_{\xi}+i}c^{(\alpha)}_{q_{\xi}+i}(t),
$$
where   $k(q_{\xi})+1$ is a number of elements in the $q_{\xi}$-th Jourdan chain,  $m(q)$ is a geometrical multiplicity of the $q$-th eigenvalue,
\begin{equation*}
c^{(\alpha)}_{q_{\xi}+i}(t)=   e^{ - \lambda^{\alpha}_{q}   t}\sum\limits_{m=0}^{k(q_{\xi})-i}H_{m}(\alpha, \lambda_{q},t)c_{q_{\xi}+i+m},\,i=0,1,2,...,k(q_{\xi}),
\end{equation*}
$\lambda_{q}=1/\mu_{q}$ is a characteristic number corresponding to $e_{q_{\xi}},$
$$
H_{m}( \alpha,\lambda,t ):=  \frac{e^{ \lambda^{\alpha}  t}}{m!} \cdot\lim\limits_{\zeta\rightarrow 1/\lambda }\frac{d^{m}}{d\zeta^{\,m}}
\left\{ e^{- \zeta^{-\alpha} t}\right\} ,\;m=0,1,2,...\,,\,.
$$
 Using the fact $H_{m}( \alpha,\lambda,t )\rightarrow 0,\,t\rightarrow +0,\,m>0,$ we get
$
c_{q_{\xi}+j} (t)\rightarrow  c_{q_{\xi}+j},\,t\rightarrow +0.
$
Since we deal with the lacunae method elaborated by Lidskii V.B. we also use the following splitting
\begin{equation*}
 \mathcal{P} _{\nu}(\alpha,t)f:= \sum\limits_{q=N_{\nu}+1}^{N_{\nu+1}}\mathcal{P} _{q}(\alpha,t)f.
\end{equation*}

\begin{lem}\label{L4.3} Assume that   $B$ is a compact operator, then  in the pole $\lambda_{q}$ of the operator  $(I-\lambda B)^{-1},$ the residue of the vector  function $e^{-\lambda^{\alpha}t}B(I-\lambda B)^{-1}\!f,\,(f\in \mathfrak{H}),\,\alpha>0$  equals to
$$
-\sum\limits_{\xi=1}^{m(q)}\sum\limits_{i=0}^{k(q_{\xi})}e_{q_{\xi}+i}c^{(\alpha)}_{q_{\xi}+i}(t).
$$
\end{lem}
\begin{proof} Consider an integral
$$
 \mathfrak{I}=\frac{1}{2\pi i}\oint\limits_{\vartheta_{q}} e^{- \lambda^{\alpha} t}B(I-\lambda B)^{-1}fd\lambda ,\,f\in \mathrm{R}(B),
$$
where the interior of the contour $\vartheta_{q}$ does not contain any poles of the operator $(I-\lambda B)^{-1},$  except of $\lambda_{q}.$ Assume that  $\mathfrak{N}_{q}$ is a  root  space corresponding to $\lambda_{q}$ and  consider  a Jordan basis  \{$e_{q_{\xi}+i}\},\,i=0,1,...,k(q_{\xi}),\;\xi=1,2,...,m(q)$ in $\mathfrak{N}_{q}.$ Using decomposition of the Hilbert space in the direct sum \eqref{4.9}, we can represent an element
$$
f=f_{1}+f_{2},
$$
where $f_{1}\in \mathfrak{N}_{q},\;f_{2}\in \mathfrak{M}_{q}.$ Note that the operator function $e^{- \lambda^{\alpha} t}B(I-\lambda B)^{-1}f_{2}$ is regular in the interior of the contour $\vartheta_{q},$ it follows from the fact that $\lambda_{q}$ ia a normal eigenvalue (see the supplementary information). Hence, we have
$$
 \mathfrak{I}=\frac{1}{2\pi i}\oint\limits_{\vartheta_{q}} e^{- \lambda^{\alpha} t}B(I-\lambda B)^{-1}f_{1}d\lambda.
$$
Using the formula
$$
B(I-\lambda B)^{-1}=\left\{(I-\lambda B)^{-1}-I   \right\}\frac{1}{\lambda}=\left\{\left(\frac{1}{\lambda}I- B\right)^{-1}-\lambda I   \right\}\frac{1}{\lambda^{2}},
$$
we obtain
$$
\mathfrak{I}=-\frac{1}{2\pi i}\oint\limits_{\tilde{\vartheta}_{q}} e^{- \zeta^{-\alpha} t}B(\zeta I  -  B)^{-1}f_{1}d\zeta,\,\zeta=1/\lambda.
$$
Now, let us decompose the element $f_{1}$ on the corresponding Jordan basis, we have
\begin{equation}\label{4.11}
f_{1}=\sum\limits_{\xi=1}^{m(q)}\sum\limits_{i=0}^{k(q_{\xi})}e_{q_{\xi}+i}c_{q_{\xi}+i}.
\end{equation}
In accordance with the  relation \eqref{4.10}, we get
$$
Be_{q_{\xi}}=\mu_{q}e_{q_{\xi}},\;Be_{q_{\xi}+1}=\mu_{q}e_{q_{\xi}+1}+e_{q_{\xi}},...,Be_{q_{\xi}+k}=\mu_{q}e_{q_{\xi}+k}+e_{q_{\xi}+k-1}.
$$
Using this formula, we can prove the following relation
\begin{equation}\label{4.12}
(\zeta I-  B)^{-1}e_{q_{\xi}+i}=\sum\limits_{j=0}^{i}\frac{e_{q_{\xi}+j}}{(\zeta-\mu_{q})^{i-j+1}}.
\end{equation}
Note that the case $i=0$ is trivial. Consider a case, when $i>0,$ we have
$$
\frac{(\zeta I-  B)e_{q_{\xi}+j}}{(\zeta-\mu_{q})^{i-j+1}}=\frac{\zeta e_{q_{\xi}+j}-Be_{q_{\xi}+j}}{(\zeta-\mu_{q})^{i-j+1}} = \frac{ e_{q_{\xi}+j}}{(\zeta-\mu_{q})^{i-j }}-\frac{e_{q_{\xi}+j-1}}{(\zeta-\mu_{q})^{i-j+1}},\,j>0,
$$
$$
\frac{(\zeta I-  B)e_{q_{\xi} }}{(\zeta-\mu_{q})^{i +1}}=  \frac{ e_{q_{\xi} }}{(\zeta-\mu_{q})^{i }}.
$$
Using these formulas, we obtain
$$
\sum\limits_{j=0}^{i}\frac{(\zeta I-  B)e_{q_{\xi}+j}}{(\zeta-\mu_{q})^{i-j+1}}= \frac{ e_{q_{\xi} }}{(\zeta-\mu_{q})^{i }}+  \frac{ e_{q_{\xi}+1}}{(\zeta-\mu_{q})^{i-1 }}-\frac{e_{q_{\xi} }}{(\zeta-\mu_{q})^{i }}+...
$$
$$
+ \frac{ e_{q_{\xi}+i}}{(\zeta-\mu_{q})^{i-i }}-\frac{e_{q_{\xi}+i-1}}{(\zeta-\mu_{q})^{i-i+1}}=
 \frac{ e_{q_{\xi}+i}}{(\zeta-\mu_{q})^{i-i }},
$$
what gives us the desired result. Now, substituting  \eqref{4.11},\eqref{4.12}, we get
$$
\mathfrak{I}=- \frac{1}{2\pi i}\sum\limits_{\xi=1}^{m(q)}\sum\limits_{i=0}^{k(q_{\xi})}c_{q_{\xi}+i}\sum\limits_{j=0}^{i} e_{q_{\xi}+j} \oint\limits_{\tilde{\vartheta}_{q}}\frac{ e^{- \zeta^{-\alpha} t}}{(\zeta-\mu_{q})^{i-j+1}}d\zeta.
$$
Note that the function $ \zeta^{-\alpha} $ is analytic inside the interior of $\tilde{\vartheta}_{q},$ hence
$$
 \frac{1}{2\pi i}\oint\limits_{\tilde{\vartheta}_{q}}\frac{ e^{- \zeta^{-\alpha} t}}{(\zeta-\mu_{q})^{i-j+1}}d\zeta= \frac{1}{(i-j)!}\lim\limits_{\zeta\rightarrow\, \mu_{q}}\frac{d^{i-j}}{d\zeta^{\,i-j}}\left\{ e^{- \zeta^{-\alpha} t}\right\}= e^{- \lambda^{\alpha}_{q} t}H_{i-j}(\alpha,\lambda_{q},t ).
$$
Changing the indexes, we have
$$
\mathfrak{I}=-  \sum\limits_{\xi=1}^{m(q)}\sum\limits_{i=0}^{k(q_{\xi})}c_{q_{\xi}+i}e^{- \lambda^{\alpha}_{q} t}\sum\limits_{j=0}^{i} e_{q_{\xi}+j} H_{i-j}(\alpha,\lambda_{q},t )=
$$
$$
=-\sum\limits_{\xi=1}^{m(q)}\sum\limits_{j=0}^{k(q_{\xi})}e_{q_{\xi}+j}e^{-\lambda^{\alpha}_{q}t}\sum\limits_{m=0}^{k(q_{\xi})-j} c_{q_{\xi}+j+m} H_{m}(\alpha,\lambda_{q},t )=
$$
$$
=-\sum\limits_{\xi=1}^{m(q)}\sum\limits_{j=0}^{k(q_{\xi})}e_{q_{\xi}+j}   c^{(\alpha)}_{q_{\xi}+j} (t).
$$
The proof is complete.
\end{proof}

\section{  Convergence of the contour integral }

To establish the main results  we need the following lemmas by Lidskii. Note that in spite of the fact that we have rewritten the lemmas in the refined form main idea of the proof has not been changed  and can be found in the paper \cite{firstab_lit:1Lidskii}.
The main tool in  study the issues related to convergence of the root series is an integral on the complex plain along the contour going to the infinitely distant point.   Further, considering an arbitrary  compact operator $B: \mathfrak{H}\rightarrow \mathfrak{H}$ such that
$
\Theta(B)\subset \mathfrak{L}_{0}(\theta),\,-\pi<\theta<\pi,
$
we put the following contour   in correspondence to the operator
\begin{equation*}
\vartheta(B):=\left\{\lambda:\;|\lambda|=r>0,\,|\mathrm{arg} \lambda|\leq \theta+\varsigma\right\}\cup\left\{\lambda:\;|\lambda|>r,\; |\mathrm{arg} \lambda|=\theta+\varsigma\right\},
\end{equation*}
where $\varepsilon>0$ is an arbitrary small number, the number $r$ is chosen so that the operator  $ (I-\lambda B)^{-1} $ is regular within the corresponding closed circle. The latter fact follows from the representation $B(I-\lambda B)^{-1}=R_{W}(\lambda).$ Indeed, due to the compactness property of the operator $B$  its inverse   has a discrete spectrum with the limit point located at the  infinitely distant point. Thus, a finite number of eigenvalues are located in a circle with an arbitrary radios what gives us the desired claim.

\subsection{Estimates for the norm of the integral expression }

In this paragraph we study techniques related to estimation of the integral expression. It is reasonable to assume that they are mostly determined by location of the set containing  complex value $\lambda.$ The following lemma gives us a quite simple  result, however not so efficient and admit further improvement.

\begin{lem} \label{L4.7} Assume that $B$ is a compact  operator,  $\Theta(B)\subset \mathfrak{L}_{0}(\theta),\,-\pi<\theta<\pi,$ then on each ray $\zeta$ containing the point zero and not belonging to the sector $\mathfrak{L}_{0}(\theta)$ as well as the  real axis, we have
$$
\|(I-\lambda B)^{-1}\|\leq \frac{1}{\sin\varphi},\,\lambda\in \zeta,
$$
where $\,\varphi = \min \{|\mathrm{arg}\zeta -\theta|,|\mathrm{arg}\zeta +\theta|\}.$
\end{lem}
\begin{proof}Assume that $\lambda\in \zeta$ and
denote $h:=(I-\lambda B)^{-1}f,\,f\in \mathfrak{H},$ then $f=h-\lambda Bh,$ hence
$$
(f,h)=\|h\|^{2}-\lambda(Bh,h).
$$
Note that
$
 (Bh,h)/\|h\|^{2}\in \Theta(B)\subset \mathfrak{L}_{0}(\theta)
$
and
denote  $r:=d(M, 1/\lambda),$ where $M:=(Bh,h)/\|h\|^{2},$ it is clear that $r|\lambda| \geq \sin \varphi.$ Therefore
$$
\frac{\left|(f,h)\right|}{|\lambda|\cdot\|h\|^{2}}=\left|\frac{1}{\lambda}- \frac{(Bh,h)}{\|h\|^{2}} \right|\geq \frac{\sin \varphi}{|\lambda|},\;\;\|h\|^{2}\sin\varphi \leq |(f,h)|\leq \|f\|\cdot\|h\|.
$$
Hence
$$
\|(I-\lambda B)^{-1} f\|\leq \frac{1}{\sin\varphi} \cdot\|f\|,\;\lambda\in \zeta,
$$
the latter relation proves the desired result.
\end{proof}

\subsection{Fredholm  Determinant}

In this paragraph we consider methods of estimating the  norm of the resolvent in the case when $\lambda$ belongs to the arc inside the sector. The well-known technique used by Lidskii and others appeals to the notion of the Fredholm Determinant and due to this reason  we produce a complete description of the object. Having chosen an orthonormal basis $\{e_{n}\}$ consider the matrix $\{b_{ij}\}$ of the operator $B,$ where
$$
b_{ij}:=(Be_{i,}e_{j}).
$$
Using these terms, we can rewrite the  equation
$
(I-\lambda B)x=f
$
in the following form
\begin{equation}\label{4.13}
\sum\limits_{j=1}^{\infty}(\delta_{ij}-\lambda b_{ij})x_{j}=f_{i},\;f_{i}=(f,e_{i}).
\end{equation}
 Consider a formal decomposition of the determinant of the matrix with the infinite quantity of rows and columns
$$
\det \{\delta_{ij}-\lambda b_{ij}\}_{ij=1}^{\infty}=\sum\limits_{p=0}^{\infty}(-1)^{p}q_{p}\lambda^{p},
$$
where $q_{p},\,q_{0}=0, $ is a sum of all central minors of the matrix $\{b_{ij}\}$  of the order equals $p,$   formed from the columns and rows with $i_{1},i_{2},...,i_{p}$ numbers  i.e.
$$
q_{p}=\frac{1}{p!}\sum\limits_{i_{1}i_{2}...i_{p}}B\begin{pmatrix} i_{1}&  i_{2}&...&i_{p}\\
  i_{1}&  i_{2}&...&i_{p}
\end{pmatrix}.
$$
Now, consider a formal representation for the resolvent
\begin{equation}\label{4.14}
(I-\lambda B)^{-1}f =\sum\limits_{m=1}^{\infty}\left(\sum\limits_{l=1}^{\infty}(-1)^{l+m}\frac{\Delta^{lm}(\lambda)}{\Delta}f_{l}\right)e_{m},
\end{equation}
where
$$
\Delta^{lm}(\lambda)= 1+\sum\limits_{p=1}^{\infty}(-1)^{p}\lambda^{p}\!\!\!\!\!\sum\limits_{i_{1},i_{2},...,i_{p} =1}^{\infty}\!\!\!B\begin{pmatrix} i_{1}&  i_{2}&...&i_{p}\\
  i_{1}&  i_{2}&...&i_{p}
\end{pmatrix}_{lm},
$$
the  used formula in brackets means a minor formed from the columns and rows with $i_{1},i_{2},...,i_{p}$ numbers except for $l$ -th row and $m$ -th column. Using analogous form of writing, we denote the Fredholm determinant of the operator $B$ as follows
$$
\Delta (\lambda)=1- \lambda  \sum\limits_{i=1}^{\infty}B\begin{pmatrix} i\\
  i
\end{pmatrix} +\lambda^{2}  \sum\limits_{i_{1},i_{2} =1}^{\infty}B\begin{pmatrix} i_{1}&  i_{2}\\
  i_{1}&  i_{2}
\end{pmatrix} +...+(-1)^{p}\lambda^{p}\sum\limits_{i_{1},i_{2},...,i_{p} =1}^{\infty}B\begin{pmatrix} i_{1}&  i_{2}&...&i_{p}\\
  i_{1}&  i_{2}&...&i_{p}
\end{pmatrix} +...\,,
$$
where the  used formula in brackets means a minor formed from the columns and rows with $i_{1},i_{2},...,i_{p}$ numbers. Note that if $B$ belongs to the trace class then in accordance with the well-known theorems (see \cite{firstab_lit:1Gohberg1965}), we have
\begin{equation}\label{4.15}
\sum\limits_{n  =1}^{\infty}|b_{nn}| <\infty,\;\sum\limits_{n,m  =1}^{\infty}|b_{nm}|^{2} <\infty,
\end{equation}
where $ b_{nm} $ is the matrix coefficients of the operator $B.$ This follows easily from the properties of the trace class operators and Hilbert-Schmidtt class operators respectively. In accordance with the von Koch theorem the conditions \eqref{4.15} guaranty the absolute convergence of the series
$$
\sum\limits_{i_{1},i_{2},...,i_{p} =1}^{\infty}B\begin{pmatrix} i_{1}&  i_{2}&...&i_{p}\\
  i_{1}&  i_{2}&...&i_{p}
\end{pmatrix}.
$$
Moreover, the formal series $\Delta(\lambda)$ is convergent for arbitrary $\lambda\in \mathbb{C},$  hence it represents an entire function, the series \eqref{4.14} represents the solution the system \eqref{4.13}.

Since the main characteristic of the studied operators is their belonging to the Schatten-Von Neumann class then it is rather reasonable to consider a broad spectrum of its index.   The following lemmas give us   technical tools  to implement the latter idea.

\begin{lem}\label{L4.8}
Assume that $B$ is a compact operator, $P$ is an arbitrary orthogonal projector in $\mathfrak{H},$  then
$$
s_{n}(PBP)\leq s_{n}(B),\,n\in \mathbb{N}.
$$
\end{lem}
\begin{proof} We have the following obvious reasonings
$$
(\tilde{B}^{\ast}\tilde{B}f,f)=(PBPf,BPf)\leq (BPf,BPf)=(PB^{\ast} BPf,f),
$$
where $\tilde{B}:=PBP.$
Due to the minimax principle, we conclude
$$
s^{2}_{n}(\tilde{B})\leq \lambda_{n} (PB^{\ast} BP),\,n\in \mathbb{N}.
$$
  Let us show that
 $$
 \lambda_{n} (PB^{\ast} BP)\leq \lambda_{n} ( B^{\ast} B ),
 $$
 we obviously have
 $$
 \max\limits_{f\in \mathfrak{L}^{\bot}_{m}}(B^{\ast}BPf,Pf)\leq \max\limits_{f\in \mathfrak{L}^{\bot}_{m}}(Bf, f),
 $$
 where $\mathfrak{L}_{m}\subset \mathfrak{H},\,m\in \mathbb{N}_{0} $ is an arbitrary m-dimensional subspace of $\mathfrak{H}.$
Having noticed the fact
$$
\min_{\mathfrak{L}_{m}\in \mathfrak{H}}\max\limits_{  f\in \mathfrak{L}^{\bot}_{m}}(Bf, f)=\max\limits_{f\in \mathfrak{\tilde{L}}^{\bot}_{m}}(Bf, f),
$$
where $\tilde{\mathfrak{L}}^{\bot}_{m}$ is the m-dimensional eigenvector subspace corresponding to the eigenvalues $\lambda_{1},\lambda_{2},...,\lambda_{m}$ using the relation for the  eigenvalue given by the minimax principle,  we obtain
$$
\lambda_{m+1} ( B^{\ast} B )=\min_{\mathfrak{L}_{m}\in \mathfrak{H}}\max\limits_{  f\in \mathfrak{L}^{\bot}_{m}}(B^{\ast}Bf, f)=\max\limits_{f\in \mathfrak{\tilde{L}}^{\bot}_{m}}(B^{\ast}Bf, f)\geq \max\limits_{f\in \mathfrak{\tilde{L}}^{\bot}_{m}}(B^{\ast}BPf,Pf)\geq
$$
$$
\geq\min_{\mathfrak{L}_{m}\in \mathfrak{H}}\max\limits_{  f\in \mathfrak{L}^{\bot}_{m}}(B^{\ast}BPf, Pf)=\lambda_{m+1} ( PB^{\ast} PB ),\,m\in \mathbb{N}_{0}.
$$
The proof is complete.
\end{proof}
For the reader convenience we represent the following theorem belonging to Alfred Horn (Theorem 3 \cite{firstab_lit:H o r n}).

\begin{teo}\label{T4.5}
Let $A,B,C$ be compact operators, $A=BC,$ then
$$
\sum\limits_{n=1}^{\infty}f\left(s_{n}(A)\right)\leq \sum\limits_{n=1}^{\infty}f\left(s_{n}(B)s_{n}(C)\right),
$$
where $f(e^{x})$ is a convex increasing function of the argument $x.$
\end{teo}

The following lemma belongs to Lidskii  \cite{firstab_lit:1Lidskii}, we represent its statement supplied with the expended proof.
\begin{lem}\label{L4.9}  Assume that $B\in \mathfrak{S}_{q},\,0<q<\infty,$ then
$$
\sum\limits_{n=1}^{\infty}s_{n}^{\frac{q}{m} }(B^{m})\leq\sum\limits_{n=1}^{\infty}s_{n}^{ q }(B),\;m\in \mathbb{N}.
$$
\end{lem}
 \begin{proof}Let us implement the proof using the method of the mathematical induction. We have that the statement obviously  holds in the case $m=1,$ assume that the following inequality holds
$$
\sum\limits_{n=1}^{\infty}s_{n}^{\frac{q}{m-1} }(B^{m-1})\leq\sum\limits_{n=1}^{\infty}s_{n}^{ q }(B),\;m=2,3,...,\,.
$$
Since $B^{m}=B^{m-1}B,$  then in accordance with the Theorem \ref{T4.5} ( Horn A.), using the function $f(x):=x^{\frac{q}{m}},$ applying then H\"{o}lder inequality, we get
$$
\sum\limits_{n=1}^{\infty}s_{n}^{\frac{q}{m} }(B^{m})\leq\sum\limits_{n=1}^{\infty}s_{n}^{\frac{q}{m} }(B^{m-1})s_{n}^{ \frac{q}{m}}(B)\leq
\left\{\sum\limits_{n=1}^{\infty}s_{n}^{\frac{q  }{  m-1 } }(B^{m-1})\right\}^{\frac{m-1}{m}}
\left\{\sum\limits_{n=1}^{\infty} s_{n}^{   q }(B)\right\}^{\frac{1}{m}}\leq \sum\limits_{n=1}^{\infty} s_{n}^{   q }(B).
$$
The latter relation completes the proof.
 \end{proof}
\begin{lem}\label{L4.10}
Assume that  $B$ belongs to the trace class, then the following representation holds
$$
\Delta(\lambda)=\prod\limits_{n=1}^{\infty}\left\{1- \lambda \mu_{n}(B)    \right\}.
$$
\end{lem}
\begin{proof} Since the operator is in the trace class then the first condition \eqref{4.15} holds, since belonging to the trace class means that for an arbitrary orthonormal basis $\{\varphi_{n}\},$  we have
$$
\sum\limits_{n=1}^{\infty}(B\varphi_{n},\varphi_{n})_{\mathfrak{H}}<\infty.
$$
The arbitrary choice of the basis gives an opportunity to claim  that the series is convergent after an arbitrary transposition of the terms from what follows that the series is absolutely convergent. Hence the first condition \eqref{4.15} holds. To prove the second condition \eqref{4.15} we should note the identity between the trace class and $\mathfrak{S}_{1}$ class (this is why  the latter has the same name informally),  the inclusion $\mathfrak{S}_{2}\subset \mathfrak{S}_{1}$ and the fact that $\mathfrak{S}_{2}$ coincides with the so-called Schmidt class of the operators with the absolute norm $\|\cdot\|_{2}.$ The latter fact can c be established if we consider the completion of the orthonormal set $\{\varphi_{n}\}$ of the eigenvectors of the operator $B^{\ast}B$ to a basis $\{\psi_{n}\}$ in the Hilbert space. Then in accordance with the well-known decomposition formula (see \S 3, Chapter V,  \cite{firstab_lit:kato1980}), we get the orthogonal sum
$$
\mathfrak{H}=\mathrm{R}(B^{\ast}B)\,\dot{+}\,\mathrm{N}(B^{\ast}B),
$$
where $\{\varphi_{n}\}$ is a basis in  $\mathrm{R}(B^{\ast}B)$ in accordance with the general property of the compact selfadjoint operator. Therefore, the completion of the system $\{\varphi_{n}\}$ to the basis in $\mathfrak{H}$ belongs to $\mathrm{N}(B^{\ast}B).$ Therefore
$$
\|B\|^{2}_{2}=\sum\limits_{n,k=1}^{\infty}|(B\psi_{n},\psi_{k})|^{2} =\sum\limits_{n=1}^{\infty}\|B\psi_{n}\|^{2}=\sum\limits_{n=1}^{\infty}(B^{\ast}B\psi_{n}, \psi_{n})  =\sum\limits_{n=1}^{\infty}(B^{\ast}B\varphi_{n}, \varphi_{n})  =\sum\limits_{n=1}^{\infty}s^{2}_{n}.
$$
  Due to belonging to the trace class and the Schmidt class  the conditions \eqref{4.15} hold for an arbitrary chosen basis in the Hilbert space.

  Now consider an arbitrary basis $\{\varphi_{k}\}$ and consider an orthogonal  projector $P_{n}$ corresponding to the  subspace generated by the first $n$ basis vectors $\varphi_{1},\varphi_{2},...,\varphi_{n}.$ Consider a determinant
  $$
  \Delta^{(n)}(\lambda):=\det \{\delta_{ij}-\lambda b_{ij}\}_{ij=1}^{n} =\prod\limits_{k=1}^{n}\left\{1-  \lambda \mu_{k}( P_{n}BP_{n} )  \right\}
  $$
Using the Weil inequalities \cite{firstab_lit:1Gohberg1965}, we get
$$
  |\Delta^{(n)}(\lambda)| \leq\prod\limits_{k=1}^{n}\left\{1+  |\lambda|\cdot |\mu_{k}( P_{n}BP_{n} )|  \right\}\leq \prod\limits_{k=1}^{n}\left\{1+  |\lambda|\cdot |s_{k}( P_{n}BP_{n} )|  \right\}.
  $$
Applying Lemma \ref{L4.8}, we get
$$
  |\Delta^{(n)}(\lambda)|  \leq \prod\limits_{k=1}^{\infty}\left\{1+  |\lambda|\cdot |s_{k}(  B  )|  \right\}.
  $$
Passing to the limit while $n\rightarrow\infty,$ we get
$$
  |\Delta(\lambda)|  \leq \prod\limits_{k=1}^{\infty}\left\{1+  |\lambda|\cdot |s_{k}(  B  )|  \right\}.
  $$
It implies, if we observe Theorem 4 (Chapter I, \S 4) \cite{firstab_lit:Eb. Levin} that the entire function $\Delta (\lambda)$ is of the finite order, hence in accordance with the Theorem 13 (Chapter I, \S 10) \cite{firstab_lit:Eb. Levin}, it has a representation by the canonical product, we have
\begin{equation*}
\Delta (\lambda)= \prod\limits_{n=1}^{\infty}\left\{1- \lambda \mu_{n}(B )    \right\}.
\end{equation*}
The proof is complete.
\end{proof}

Bellow we represent   an adopted version of the propositions given in the paper \cite{firstab_lit:1Lidskii}, we consider a case when a compact operator  belongs to the Schatten-von Neumann  class $ \mathfrak{S} _{q}.$  Having taken into account the facts considered in the previous subsection, we can reformulate Lemma 2 \cite{firstab_lit:1Lidskii} in the refined form.

\begin{lem}\label{L4.11}   Assume that a compact operator $B$ satisfies   the  condition   $B\in \mathfrak{S}_{ q},\,0<q<\infty$       then
for arbitrary numbers  $R,\delta$   such that $R>0,\,0<\delta<1,$ there exists a  circle $|\lambda|=\tilde{R},\,(1-\delta)R<\tilde{R}<R,$ so that the following estimate holds
$$
\|(I-\lambda B )^{-1}\|_{\mathfrak{H}}\leq C e^{w(|\lambda|) }|\lambda|^{m},\,|\lambda|=\tilde{R},\,m=[q],
$$
where
$$
w(|\lambda|)= h ( |\lambda|^{m+1})  +(2+\ln\{4e/\delta\}) h\left( \frac{2e}{(1-\delta)}   |\lambda| ^{m+1}\right),\;h(r )=  \left(\int\limits_{0}^{r}\frac{n_{B^{m+1}}(t)dt}{t }+
r \int\limits_{r}^{\infty}\frac{n_{B^{m+1}}(t)dt}{t^{ 2  }}\right).
$$

\end{lem}
\begin{proof}
Let us prove the following relation
\begin{equation}\label{4.16}
|\Delta_{B^{m+1}} (\lambda^{m+1})|\cdot\|(I-\lambda B )^{-1}\|_{\mathfrak{H}}\leq C  |\lambda|^{m } \prod\limits_{n=1}^{\infty}\left\{1+ |\lambda^{m+1} \mu_{n}( B^{m+1}  )| \right\}.
\end{equation}
For this purpose, firstly consider the case $q=1.$  Let us chose an arbitrary element $f\in \mathfrak{H},$  and construct a new orthonormal basis having put $f$ as a first basis element. Note that the relations \eqref{4.15} hold for the matrix coefficients of the operator in the new basis, this fact follows from the well-known theorem for the trace class operator.  Thus, using the given  above representation for the resolvent \eqref{4.14}, we obtain the following relation
$$
 \left(D_{B}(\lambda)f,f\right)_{\mathfrak{H}}=  \Delta^{11} (\lambda).
$$
where $
D_{B}(\lambda):=\Delta (\lambda) (I-\lambda B )^{-1},
$
For a convenient form of writing, we have not used an  index indicating that the element $f$ is used as a first basis vector.
Let us observe the latter entire function more properly, we have
$$
\Delta^{11} (\lambda)=1- \lambda  \sum\limits_{i\neq 1}^{\infty}B\begin{pmatrix} i\\
  i
\end{pmatrix} +\lambda^{2}  \sum\limits_{i_{1},i_{2} \neq1}^{\infty}B\begin{pmatrix} i_{1}&  i_{2}\\
  i_{1}&  i_{2}
\end{pmatrix} +...+(-1)^{p}\lambda^{p}\sum\limits_{i_{1},i_{2},...,i_{p} \neq1}^{\infty}B\begin{pmatrix} i_{1}&  i_{2}&...&i_{p}\\
  i_{1}&  i_{2}&...&i_{p}
\end{pmatrix} +...\,.
$$
  The later construction  reveals the fact that it is the very  Fredholm determinant of the operator $P_{1}B_{k}P_{1},$ where $P_{1}$  is the projector into  orthogonal  complement of the element $f.$ Having applied  Lemma \ref{L4.8}, we obtain
$$
s_{n}(P_{1}B P_{1})\leq s_{n}( B  ),\,n=1,2,...\,.
$$
Therefore, applying Lemma \ref{L4.10},    we obtain the following  representation
$$
\Delta^{11}(\lambda)=\prod\limits_{n=1}^{\infty}\left\{1- \lambda \mu_{n}(P_{1}BP_{1})    \right\}.
$$
In accordance with   Corollary  3.1 \cite{firstab_lit:1Gohberg1965}, Chapter II, \S3  (Corollary of the  Weyl's majorant theorem), we have
$$
\prod\limits_{n=1}^{k}\left\{1+ |\lambda \mu_{n}(P_{1}B P_{1})| \right\}\leq \prod\limits_{n=1}^{k}\left\{1+ |\lambda s_{n}(P_{1}B P_{1})| \right\},\,k\in \mathbb{N}.
$$
Therefore, using the estimates given above, we obtain
$$
 |\left(D_{B}(\lambda)f,f\right)_{\mathfrak{H}}|\leq\prod\limits_{n=1}^{\infty}\left\{1+ |\lambda \mu_{n}(P_{1}B P_{1})| \right\}\leq \prod\limits_{n=1}^{\infty}\left\{1+ |\lambda s_{n}( B  )| \right\}.
$$
Since the right-hand side does not depend on $f,$ then using decomposition on the Hermitian components, we can  obtain the following relation
$$
\sup \limits_{\|f\|\leq 1 }\|D_{B }(\lambda)f\| \leq \sup \limits_{\|f\|\leq 1 }\|\mathfrak{Re} D_{B }(\lambda)f  \|+ \sup \limits_{\|f\|\leq 1 }\|  \mathfrak{Im} D_{B}(\lambda)f \|=
$$
$$
= \sup \limits_{\|f\|\leq 1 }\left| \mathrm{Re}(D_{B}  (\lambda)f,f )\right|+\sup \limits_{\|f\|\leq 1 }\left| \mathrm{Im}(D_{B} (\lambda)f,f )\right|\leq 2\sup \limits_{\|f\|\leq 1 }\left|  (D_{B}f,f )\right|\leq 2 \prod\limits_{n=1}^{\infty}\left\{1+ |\lambda s_{n}( B  )| \right\}.
$$
Thus, we have
$$
|\Delta (\lambda)|\cdot\|(I-\lambda B )^{-1}\|\leq 2 \prod\limits_{n=1}^{\infty}\left\{1+ |\lambda s_{n}( B  )| \right\}.
$$
It is easy to extend  the obtained result for the case $q>1,$ we should take into account the following obvious relation
 $$
 \|(I-\lambda B )^{-1}\|_{\mathfrak{H}}\leq\|(I-\lambda^{m+1}B^{m+1})^{-1}\|_{\mathfrak{H}} \cdot\|(I+\lambda B+\lambda^{2} B^{2}+...+\lambda^{m}B^{m})\|_{\mathfrak{H}}\leq
$$
$$
\leq\|(I-\lambda^{m+1}B^{m+1})^{-1}\|_{\mathfrak{H}} \cdot \frac{|\lambda|^{m+1}\|B\|^{m+1}-1}{|\lambda|\cdot\|B\|-1}\leq C |\lambda|^{m }\cdot\|(I-\lambda^{m+1}B^{m+1})^{-1}\|_{\mathfrak{H}},\,m=[q].
$$
In accordance with Lemma \ref{L4.9},   we have
$$
\sum\limits_{n=1}^{\infty}s^{\frac{q}{m+1}}_{n}( B^{m+1}  )\leq \sum\limits_{n=1}^{\infty}s^{ q }_{n}( B )<\infty,
$$
  Therefore, $B^{m+1}\in \mathfrak{S}_{q/(m+1)}\subset \mathfrak{S}_{1}$ and the problem is reduced to the previously solved one. More precisely, the following estimate is a conclusion of the obtained above formula
$$
|\Delta_{B^{m+1}} (\lambda^{m+1})|\cdot\|(I-\lambda B )^{-1}\|_{\mathfrak{H}}\leq C |\lambda|^{m } |\Delta_{B^{m+1}} (\lambda^{m+1})| \cdot\|(I-\lambda^{m+1}B^{m+1})^{-1}\|_{\mathfrak{H}}\leq
$$
$$
\leq C  |\lambda|^{m } \prod\limits_{n=1}^{\infty}\left\{1+ |\lambda^{m+1} s_{n}( B^{m+1}  )| \right\}.
$$
Thus, we have obtained \eqref{4.16}.
Using the properties of the canonical product, applying Lemma \ref{L4.2}, we obtain the estimate
$$
|\Delta_{B^{m+1}} (\lambda^{m+1})|\cdot\|(I-\lambda B )^{-1}\|_{\mathfrak{H}}\leq C |\lambda|^{m } e^{h(|\lambda|^{m+1}) }.
$$
Now, to obtained the lemma statement, we should estimate the absolute value of the  Fredholm determinant $|\Delta_{B^{m+1}} (\lambda^{m+1})|$ from bellow. For this purpose, let us notice that in accordance with Lemma \ref{L4.10} it is an entire function represented by the formula
$$
\Delta_{B^{m+1}}(\lambda^{m+1})=   \prod\limits_{n=1}^{\infty}\left\{1-  \lambda^{m+1} \mu_{n}( B^{m+1}  )  \right\}.
$$
Thus, in the simplified case,    we can use the Cartan estimate given in Theorem \ref{T4.2}, we have
$$
|\Delta_{B^{m+1}}(\lambda^{m+1})|\geq e^{-(2+\ln\{2e/3\eta\})\ln\xi_{m}},\;\xi_{m}= \!\!\!\max\limits_{\psi\in[0,2\pi /(m+1)]}|\Delta_{B^{m+1}}([2e R_{\nu}e^{i\psi}]^{m+1})|,\,|\lambda|\leq R _{\nu},
$$
except for the exceptional set of circles with the sum of radii less that $4\eta R_{\nu},$  where $\eta$ is an arbitrary number less than  $2e/3 .$  Thus  to find the desired circle $\lambda= e^{i\psi}\tilde{R}_{\nu}$ belonging to the ring, i.e. $R_{\nu}(1-\delta )<\tilde{R}_{\nu}<R_{\nu},$   we have to   choose $\eta$ satisfying the inequality
$$
4\eta R_{\nu}<R_{\nu}-R_{\nu}(1-\delta )=\delta R_{\nu};\;\eta< \delta /4,
$$
for instance let us choose $\eta=\delta/6.$  Under such assumptions, we can rewrite the above estimate in the form
$$
|\Delta_{B^{m+1}}(\lambda^{m+1})|\geq e^{-(2+\ln\{4e/\delta \})\ln\xi_{m}},\; |\lambda|=\tilde{R} _{\nu}.
$$
Applying Lemmas \ref{L4.2},\ref{L4.10}, implementing the given above scheme  of reasonings, we have
$$
|\Delta_{B^{m+1}}(\lambda)|\leq  \prod\limits_{n=1}^{\infty}\{1+|\lambda s_{n}( B^{m+1} )|\} \leq
 e^{ h (|\lambda| ) }.
$$
Using this estimate, we obtain
$$
\xi_{m}\leq e^{ h ([2e  R_{\nu}]^{m+1}) }\leq e^{ h \left([2e (1-\delta )^{-1}  \tilde{R}_{\nu}]^{m+1}\right)  }.
$$
Substituting, we get
$$
|\Delta_{B^{m+1}}(\lambda^{m+1})|^{-1} \leq e^{(2+\ln\{4e/\delta\})h \left([2e (1-\delta )^{-1}  \tilde{R}_{\nu}]^{m+1}\right) },\; |\lambda|=\tilde{R} _{\nu}.
$$
Combining the upper and lower estimates, we obtain the desired result.

\end{proof}

\subsection{Convergence with respect to the time variable}
It is rather essential question in the concept of the summation in the Abel-Lidskii sense that appeals to passing to the limit with respect to the time variable. We may  reasonably say that along with the series expansion of the contour integral it forms a major part of the concept. Here, we should referee to the restricted  Lidskii results and consequently obtained   Matsaev extension. Bellow, we represent  them respectively to the   order.

\begin{lem}\label{L4.12}  Assume that $B$ is a compact  operator,  $\Theta(B)\subset \mathfrak{L}_{0}(\theta),\, \theta<\pi/2\alpha,$ then
$$
\lim\limits_{t\rightarrow+0}\int\limits_{\vartheta(B)}e^{-\lambda^{\alpha}t}B(I-\lambda B)^{-1}fd\lambda=f,\;f\in \mathrm{R}(B),\,\alpha>0.
$$
\end{lem}
\begin{proof}
Using the formula
$$
B^{2}(I-\lambda B)^{-1}= \frac{1}{\lambda^{2}}\left\{\left(I- \lambda B\right)^{-1}-(I+\lambda B)   \right\},
$$
we obtain
$$
 \frac{1}{2 \pi i}\int\limits_{\vartheta(B)}e^{-\lambda^{\alpha} t}B(I-\lambda B)^{-1}fd\lambda=
 \frac{1}{2 \pi i}\int\limits_{\vartheta(B)}e^{-\lambda^{\alpha} t} \lambda^{-2}  \left(I- \lambda B\right)^{-1} v d\lambda -
$$
$$
- \frac{1}{2 \pi i}\int\limits_{\vartheta(B)}e^{-\lambda^{\alpha} t}\lambda^{-2} (I+\lambda B)  v d\lambda =I_{1}(t)+I_{2}(t),
$$
where $Bv=f.$Consider $I_{1}(t).$ Since this improper integral is uniformly convergent regarding $t,$ this fact can be established easily if we apply Lemma \ref{L4.7}, then using the theorem on the connection with the simultaneous limit and the repeated limit, we get
$$
\lim\limits_{t\rightarrow+0}I_{1}(t)= \frac{1}{2 \pi i}\int\limits_{\vartheta(B)}  \lambda^{-2}  \left(I- \lambda B\right)^{-1} vd\lambda.
$$
define a contour  $\vartheta_{R}(B):= \mathrm{Fr}\big\{\{\lambda:\, |\lambda|<R\}\setminus \mathrm{int }\,\vartheta(B) \}\big\}$ and let us prove that
\begin{equation}\label{4.17}
 \frac{1}{2\pi i}\oint\limits_{\vartheta_{R}(B)}  \lambda^{-2}  \left(I- \lambda B\right)^{-1} vd\lambda\rightarrow \frac{1}{2 \pi i}\int\limits_{\vartheta(B)}  \lambda^{-2}  \left(I- \lambda B\right)^{-1} vd\lambda,\;R\rightarrow \infty.
\end{equation}
Consider a decomposition of the contour $\vartheta_{R}(B)$ on terms   $\tilde{\vartheta}_{R}(B):=\{\lambda:\,|\lambda|=R,\, \theta +\varsigma\leq\mathrm{arg} \lambda \leq 2\pi- \theta -\varsigma  \},\,\hat{\vartheta}_{R}:= \{\lambda:\,|\lambda|=r,\,|\mathrm{arg} \lambda |\leq\theta +\varsigma\}\cup
\{\lambda:\,r<|\lambda|<R,\, \mathrm{arg} \lambda  =\theta +\varsigma\}\cup
\{\lambda:\,r<|\lambda|<R,\, \mathrm{arg} \lambda  =-\theta -\varsigma\}.$
 It is clear that
$$
\frac{1}{2\pi i}\oint\limits_{\vartheta_{R}(B)}  \lambda^{-2}  \left(I- \lambda B\right)^{-1} vd\lambda =
\frac{1}{2\pi i}\int\limits_{\tilde{\vartheta}_{R}(B)}  \lambda^{-2}  \left(I- \lambda B\right)^{-1} vd\lambda+
 $$
 $$
 +\frac{1}{2\pi i}\int\limits_{\hat{\vartheta}_{R}}  \lambda^{-2}  \left(I- \lambda B\right)^{-1} vd\lambda.
$$
Let us show that the first summand tends to zero when $R\rightarrow\infty,$ we have
$$
 \left\|\,\int\limits_{\tilde{\vartheta}_{R}(B)}  \lambda^{-2}  \left(I- \lambda B\right)^{-1} vd\lambda\right\|_{\mathfrak{H}}\leq  R^{-2}\int\limits_{\theta +\varsigma}^{2\pi - \theta -\varsigma}     \left\|\left(I\lambda^{-1}-  B\right)^{-1} v\right\|_{\mathfrak{H}}d\,  \mathrm{arg} \lambda.
$$
Applying Corollary 3.3,  Theorem 3.2  \cite[p.268]{firstab_lit:kato1980}, we have
$$
\left\|\left(I\lambda^{-1}-  B\right)^{-1} \right\|_{\mathfrak{H}}\leq R/\sin \varsigma ,\,\lambda \in \tilde{\vartheta}_{R}(B).
$$
Substituting this estimate to the last integral,
 we obtain the desired result. Thus, taking into account the fact
$$
\frac{1}{2\pi i}\int\limits_{\hat{\vartheta}_{R}}  \lambda^{-2}  \left(I- \lambda B\right)^{-1} vd\lambda\rightarrow \frac{1}{2\pi i}\int\limits_{ \vartheta(B)  }  \lambda^{-2}  \left(I- \lambda B\right)^{-1} vd\lambda ,\,R\rightarrow \infty,
$$
we obtain  \eqref{4.17}. Having noticed that the following integral can be calculated as a residue at the point zero, i.e.
$$
\frac{1}{2\pi i}\oint\limits_{\vartheta_{R}(B)}  \lambda^{-2}  \left(I- \lambda B\right)^{-1} vd\lambda =\lim\limits_{\lambda\rightarrow0}\frac{d (I-\lambda B)^{-1} }{d\lambda }\,v= f,
$$
we get
$$
\frac{1}{2 \pi i}\int\limits_{\vartheta(B)}  \lambda^{-2}  \left(I- \lambda B\right)^{-1} vd\lambda=f.
$$
Hence $I_{1}(t)\rightarrow f,\,t\rightarrow+0.$ Let us show that $I_{2}(t)=0.$ For this purpose, let us  consider  a  contour
  $\vartheta_{R}(B)=\tilde{\vartheta}_{ R}\cup \hat{\vartheta}_{R},$  where    $\tilde{\vartheta}_{ R}:=\{\lambda:\,|\lambda|=R,\,|\mathrm{arg} \lambda |\leq\theta +\varsigma\}$ and  $\hat{\vartheta}_{R}$ is previously defined. It is clear that
$$
\frac{1}{2\pi i}\oint\limits_{\vartheta_{R}(B)}  \lambda^{-2} e^{-\lambda^{\alpha}t} \left(I+ \lambda B\right)  vd\lambda=\frac{1}{2\pi i}\int\limits_{\tilde{\vartheta}_{ R}}  \lambda^{-2} e^{-\lambda^{\alpha}t} \left(I+ \lambda B\right)  vd\lambda+
 $$
 $$
+\frac{1}{2\pi i}\int\limits_{\hat{\vartheta}_{R}}  \lambda^{-2} e^{-\lambda^{\alpha}t} \left(I+ \lambda B\right)  vd\lambda.
$$
Considering the second term  having   taken  into account  the  definition of the improper integral,   we conclude    that if we show that
\begin{equation}\label{4.18}
\frac{1}{2\pi i}\int\limits_{\tilde{\vartheta}_{ R }}  \lambda^{-2} e^{-\lambda^{\alpha}t} \left(I+ \lambda B\right)  vd\lambda\rightarrow 0,\,R\rightarrow \infty,
\end{equation}
then we obtain
\begin{equation}\label{4.19}
 \frac{1}{2\pi i}\oint\limits_{\vartheta_{R }(B)}  \lambda^{-2} e^{-\lambda^{\alpha}t} \left(I+ \lambda B\right)  vd\lambda\rightarrow \frac{1}{2 \pi i}\int\limits_{\vartheta(B)}  \lambda^{-2} e^{-\lambda^{\alpha}t} \left(I+ \lambda B\right)  vd\lambda,\;R\rightarrow \infty.
\end{equation}

Using the lemma condition     $\,|\mathrm{arg} \lambda |<\pi/2\alpha,   $ we get
$$
\mathrm{Re }\lambda^{\alpha}\geq |\lambda|^{\alpha} \cos \left[(\pi/2\alpha-\delta)\alpha\right]= |\lambda|^{\alpha} \sin     \alpha \delta,
$$
 where $\delta$ is a sufficiently small number. Therefore
$$
 \left\|\,\int\limits_{\tilde{\vartheta}_{ R }}  \lambda^{-2} e^{-\lambda^{\alpha}t} \left(I+ \lambda B\right)  vd\lambda\right\|_{\mathfrak{H}}\leq C e^{- C t   R^{\alpha}  }\|v\|_{\mathfrak{H}} \int\limits_{-\theta-\varsigma}^{\theta+\varsigma}d\xi,
$$
from what follows \eqref{4.18}, \eqref{4.19}.
   Since the operator function under the integral is analytic, then
$$
\oint\limits_{\vartheta_{R }(B)}  \lambda^{-2} e^{-\lambda^{\alpha}t} \left(I+ \lambda B\right)  vd\lambda=0 .
$$
Combining this relation with \eqref{4.19}, we  obtain the fact $I_{2}(t)=0.$ The proof is complete.
\end{proof}

\subsection{Sector with an arbitrary small angle}\label{S4.4.4}

It is remarkable that we can choose a sequence of contours in various ways. For instance, a sequence of  contours of the exponential  type was considered in the paper \cite{firstab_lit:1Lidskii}.
In this paragraph, we produce an application of the  previous section results, we study a concrete operator class for which it is possible to choose a sequence of   contours of the power type. At the same time  having involved an additional condition we can spread the principal result of the paragraph  on a wider operator class. Note that using   condition  $\mathrm{H}2,$ see paragraph \ref{S2.6}, it is not hard to prove that
$
\mathrm{Re}( W f,f)_{\mathfrak{H}} -k|\mathrm{Im}( W f,f)_{\mathfrak{H}}|\geq (C_{2} -k C_{1})\|f\|^{2}_{\mathfrak{H}_{+}},\,k>0,
$
from what follows a fact  $\Theta(A)\subset \mathfrak{L}_{0}(\theta),\,\theta=\arctan (C_{1}/C_{2}).$ In general, the last relation gives us a range  of the semi-angle $\pi/4<\theta<\pi/2,$  thus the     conditions $\mathrm{H}1,\mathrm{H}2$ are not sufficient to guaranty a value of the semi-angle less than $\pi/4.$ However, we should remark that  some relevant   results can be obtained in the  case corresponding to  sufficiently small values of the semi-angle, this gives us a motivation to consider a more specific  additional  assumption in terms of the paragraph \ref{S2.6} \\

 \noindent  $( \mathrm{H3} )  \;  |\mathrm{Im}( L f,g)_{\mathfrak{H}}|\! \leq \! C_{3}\|f\|_{\mathfrak{H_{+}}}\|g\|_{\mathfrak{H}} ,\,
     ,\,f,g \in  \mathfrak{M},\;  C_{3}>0,
$\\

 \noindent here we should recall that we denote by $W$ the restriction of $L$ to the set $\mathfrak{M}.$  In this case, we have
$
\mathrm{Re}( W f,f)_{\mathfrak{H}} -k|\mathrm{Im}( W f,f)_{\mathfrak{H}}|\geq
 C_{2}\|f\|_{\mathfrak{H_{+}}}  -k C_{3}\left\{\varepsilon\|f\|^{2}_{\mathfrak{H_{+}}}/2+ \|f\|^{2}_{\mathfrak{H}}/2\varepsilon \right\}\geq
(C_{2}-k\varepsilon C_{3})|f\|^{2}_{\mathfrak{H_{+}}}/2+(C_{2}/C_{0}-k  C_{3}/\varepsilon)|f\|^{2}_{\mathfrak{H }}/2,\,k>0.
 $
Thus, choosing $\varepsilon=C_{2}/kC_{3},$ we get $\Theta( W )\subset \mathfrak{L}_{\iota}(\theta_{\iota}),$ where the abscissa of the vertex   $\iota=C_{2}/2C_{0}-(k  C_{3})^{2}/2C_{2},\,\theta_{\iota}=\arctan (1/k).$
This relation   guarantees  that having shifted the abscissa of the vertex to the left, we can choose a sufficiently small value of the semi-angle $\theta_{\iota}$  We
put the following contour in correspondence to an operator $B:=W^{-1}$ satisfying the additional condition $\mathrm{H}3$ along with the  perviously formulated $\mathrm{H}1,\mathrm{H}2$
$$
\Gamma(B):= \mathrm{Fr}\left\{ \left( \mathfrak{L}_{0}(\theta_{0}+\varsigma) \cap \mathfrak{L}_{\iota}(\theta_{\iota}+\varsigma) \right)  \setminus
\mathfrak{C}_{r}\right\},\,\iota<0,\;\mathfrak{C}_{r}:=\left\{\lambda:\;|\lambda|<r,\,|\mathrm{arg} \lambda|\leq \theta_{0} \right\},
$$
where $r$ is chosen so that the   operator $ (I-\lambda B)^{-1}  $ is regular within the corresponding  closed circle,  $\varsigma>0$ is    sufficiently small.

\begin{lem} \label{L4.13} Assume that the condition $\mathrm{H}3 $ holds, then
$$
\|(I-\lambda B)^{-1}\|_{\mathfrak{H}}\leq C,\,\lambda\in \mathrm{Fr}\left\{   \mathfrak{L}_{0}(\theta_{0}+\varsigma)\cap \mathfrak{L}_{\iota}(\theta_{\iota}+\varsigma)\right\},\,\iota<0,
$$
where $\iota=C_{2}/2C_{0}-(k  C_{3})^{2}/2C_{2},\,\theta_{\iota}=\arctan (1/k),\,\varsigma>0$ is an arbitrary small   number.
 \end{lem}
\begin{proof} Firstly, we should note that in accordance with the condition $\mathrm{H}3,$  for an arbitrary large value $k,$ we have
  $\Theta( W )\subset \mathfrak{L}_{\iota}(\theta_{\iota}),$ where $\iota=C_{2}/2C_{0}-(k  C_{3})^{2}/2C_{2},\,\theta_{\iota}=\arctan (1/k).$
 Hence $\Theta( W )\subset  \mathfrak{L}_{0}(\theta_{0} )\cap \mathfrak{L}_{\iota}(\theta_{\iota} ),$ where $\iota$ is arbitrary negative. Therefore
$\Theta(B)\subset  \mathfrak{L}_{0}(\theta_{0} )\cap \mathfrak{L}_{\iota}(\theta_{\iota} ),$ it can be verified directly due to the geometrical methods. Note that by virtue of Lemma \ref{L4.7}, we have
$
\|(I-\lambda B)^{-1}\|_{\mathfrak{H}}\leq C,\,\lambda\in \mathrm{Fr}\left\{   \mathfrak{L}_{0}(\theta_{0}+\varsigma)\right\}.
$
Thus to obtain the desired result it suffices to prove that
$
\|(I-\lambda B)^{-1}\|_{\mathfrak{H}}\leq C,\,\lambda\in \mathrm{Fr}\left\{   \mathfrak{L}_{\iota}(\theta_{\iota}+\varsigma)\right\},\,\mathrm{Re}\lambda\geq 0.
$
Note that in this case  $\lambda \in \mathrm{P}( W )$  and  we have a chain of   reasonings $\forall f\in \mathfrak{H}:\;( W -\lambda I)^{-1}f=h\in \mathrm{D}( W );$
$\; ( W -\lambda I)h =f;\;(f,h)_{\mathfrak{H}}=( W h,h)_{\mathfrak{H}}-\lambda \|h\|_{\mathfrak{H}}^{2}.$ Using the latter relation, we get
$
 |  (f,h)_{\mathfrak{H}}| /\|h\|_{\mathfrak{H}}^{2}   =  |  ( W h,h)_{\mathfrak{H}}/\|h\|_{\mathfrak{H}}^{2} -\lambda  |\geq |\lambda-\iota| \sin \varsigma.
$
Therefore, using the Cauchy-Schwartz inequality, we get
$$
 \|( W -\lambda I)^{-1}f\|_{\mathfrak{H}}\leq \frac{1}{|\lambda-\iota|\sin  \varsigma }\cdot\|f\|_{\mathfrak{H}},\;f\in \mathfrak{H}.
$$
Taking into account the fact
$
( W -\lambda I)^{-1} =(I-\lambda B )^{-1}B= \{(I-\lambda B )^{-1}-I \}/\lambda,
$
we get
$
\| (I-\lambda B )^{-1} \|_{\mathfrak{H}} -1\leq\| (I-\lambda B )^{-1}-I  \|_{\mathfrak{H}}\leq  |\lambda|/ |\lambda-\iota|\sin  \varsigma    ,
$
from what follows the desired result.
 \end{proof}

\begin{lem}\label{L4.14} Assume that  the condition $\mathrm{H}3$ holds, $f\in \mathrm{D}(W),$ then
$$
\lim\limits_{t\rightarrow+0}\int\limits_{\Gamma(B)}e^{-\lambda^{\alpha}t}B(I-\lambda B)^{-1}fd\lambda=f,\,\alpha>0.
$$
\end{lem}
\begin{proof} The proof is analogous to the proof of the Lemma 5 \cite{firstab_lit:1Lidskii} and  the only difference is in the following, we should use Lemma \ref{L4.13} instead of Lemma \ref{L4.7}.
\end{proof}

\section{Decreasing of the summation order to the convergent exponent  }

Recall that an arbitrary compact operator $B$ can be represented as a series on the basis vectors due to the so-called polar decomposition $B=U|B|,$ where $U$ is a concrete unitary operator, $|B|:=(B^{\ast}B)^{1/2},$ i.e.    using the system of eigenvectors  $\{e_{n}\}_{1}^{\infty},$ we have
\begin{equation}\label{4.20}
Bf=\sum\limits_{n=1}^{\infty}s_{n}(B)(f,e_{n})g_{n},
\end{equation}
where $ e_{n}, s_{n} $ are   the eigenvectors and eigenvalues of the operator $|B|$ respectively, $g_{n}=Ue_{n}.$ It is clear that the letter elements form a an orthonormal system due to the major property of the unitary operator. The main concept that lies in the framework of our consideration relates to the following statement.\\

\noindent $({\bf S 1}  )$ {\it Under the assumptions $B\in \mathfrak{S}_{p},\,0<\inf p<\infty,\,\Theta(B) \subset   \mathfrak{L}_{0}(\theta),\,\alpha>0$  a sequence of natural numbers $\{N_{\nu}\}_{0}^{\infty}$ can be chosen so that
$$
 \frac{1}{2\pi i}\int\limits_{\vartheta(B)}e^{-\lambda^{\alpha}t}B(I-\lambda B)^{-1}f d \lambda =\sum\limits_{\nu=0}^{\infty} \mathcal{P}_{\nu}( \alpha,t)f,
$$
where
$
\vartheta(B):=\left\{\lambda:\;|\lambda|=r>0,\,|\mathrm{arg} \lambda|\leq \theta+\varepsilon\right\}\cup\left\{\lambda:\;|\lambda|>r,\; |\mathrm{arg} \lambda|=\theta+\varepsilon\right\},
$
the latter series is absolutely convergent in the sense of the norm.
 Moreover
 $$
 \lim\limits_{t\rightarrow+0}\frac{1}{2\pi i}\int\limits_{\vartheta(B)}e^{-\lambda^{\alpha}t}B(I-\lambda B)^{-1}f d \lambda=f\in \mathrm{D}(W).
$$ }
Let us recall the definition  $  \mathfrak{S}^{\star}_{\sigma}(\mathfrak{H}),\,0\leq\sigma<\infty,$   the class of the operators such that
$$
 B\in \mathfrak{S}^{\star}_{\sigma}(\mathfrak{H}) \Rightarrow\{B\in\mathfrak{S}_{\sigma+\varepsilon},\,B \overline{\in} \,\mathfrak{S}_{\sigma},\,\forall\varepsilon>0 \}.
$$
Note that in accordance with the Lemma 3 \cite{firstab_lit:1kukushkin2021}, we have
\begin{equation}\label{4.21}
\ln r\frac{n(r)}{r^{\rho }}\rightarrow0,\Longrightarrow   \ln r \left(\int\limits_{0}^{r}\frac{n(t)}{t^{p+1 }}dt+
r\int\limits_{r}^{\infty}\frac{n(t)}{t^{p+2  }}dt\right)r^{p-\rho  } \rightarrow 0,\,r\rightarrow\infty,
\end{equation}
where $\rho$ is a convergence exponent (non integer) defined as follows  $\rho:=\inf \lambda,$
$$
 \sum\limits_{n=1}^{\infty}s_{n}^{\,\lambda}(B)<\infty,
$$
 $n(t)$ is a counting function  corresponding to the singular numbers of the operator $B,$ the number $\lambda=p+1$ is the smallest natural number for which the latter series is convergent. To produce an operator class with more subtle asymptotics, we can deal with representation \eqref{4.20} directly imposing conditions upon the singular numbers instead of considering Hermitian real  component. Assume that the sequence of singular numbers has a non integer  convergent exponent  $\rho$ and  consider the following condition
$$
 (\ln^{1+1/\rho}x)'_{s^{-1}_{m}(B)}  =o(  m^{-1/\rho}   ).
$$
This gives us
$$
\frac{m\ln s^{-1}_{m}(B)}{s^{-\rho}_{m}(B)}\leq C\cdot  \alpha_{m},\;\alpha_{m}\rightarrow 0,\,m\rightarrow\infty.
$$
Taking into account the facts
$
n(s^{-1}_{m})=m;\,\;n(r)= n(s^{-1}_{m}),\,s^{-1}_{m}<r<s^{-1}_{m+1},
$
using the monotonous property of the functions, we get
\begin{equation}\label{4.22}
\ln r\frac{n(r)}{r^{\rho}}  <  C\cdot  \alpha_{m},\;s^{-1}_{m}<r<s^{-1}_{m+1},
\end{equation}
i.e. we obtain the following implication

$$
 (\ln^{1+1/\rho}x)'_{s^{-1}_{m}(B)}  =o(  m^{-1/\rho}   ),\Longrightarrow \left(\ln r\frac{n(r)}{r^{\rho }}\rightarrow0\right),
$$
therefore the premise in \eqref{4.21} holds  and as a result the consequence holds. Absolutely analogously, we can obtain the implication
\begin{equation}\label{4.23}
  s^{-1}_{m}(B)   =o(  m^{-1/\rho}   ), \Longrightarrow \left(\frac{n(r)}{r^{\rho}}\rightarrow 0\right),
\end{equation}
here the corresponding  example is given by $s_{m}=(m \ln m)^{-1/\rho}$ and the reader can verify directly that $B\in \mathfrak{S}_{\rho}^{\star}.$   Using these facts, we can reformulate Theorem 2, 4 \cite{firstab_lit:1kukushkin2021} for an artificially constructed compact operator. In this paper, we represent   modified  proofs  since there are some difficulties that require refinement, moreover, we produce the variant  of the proof of Theorem 4 \cite{firstab_lit:1kukushkin2021} corresponding to the case      $\Theta(B) \subset   \mathfrak{L}_{0}(\theta),\,\theta<  \pi/2\alpha,$ that is not given in \cite{firstab_lit:1kukushkin2021}.

\begin{teo}\label{T4.6} Assume that $B$ is a compact operator, $\Theta(B) \subset   \mathfrak{L}_{0}(\theta),\,\theta< \pi/2\alpha ,\;B\in \mathfrak{S}^{\star} _{\alpha},$
$$
s_{n}(B)=o(n^{-1/\alpha}),
$$
where $\alpha$ is positive non integer.
Then the statement of   Theorem 2 \cite{firstab_lit:1kukushkin2021} remains true, i.e. statement S1 holds.
\end{teo}
\begin{proof}
Applying corollary of the well-known Allakhverdiyev theorem,  see   Corollary  2.2 \cite{firstab_lit:1Gohberg1965} (Chapter II, $\S$ 2.3), we get
$$
s_{(m+1)(k-1)+1}(B^{m+1})\leq s_{k}^{m+1}(B),\,m\in \mathbb{N}.
$$
Note that $n_{B}(r)=k,$ where $s_{k}^{-1}(B)<r<s_{k+1}^{-1}(B).$ Therefore
$$
s^{-1}_{(m+1)k+1}(B^{m+1})\geq s_{k+1}^{-m-1}(B)>r^{m+1},
$$
and we have
\begin{equation}\label{4.24}
n_{B^{m+1}}(r^{m+1})\leq (m+1)k= (m+1) n_{B}(r),
\end{equation}
hence using \eqref{4.23}, we obtain
\begin{equation*}
   \frac{n_{B^{m+1}}(r^{m+1})}{r^{\alpha} }\rightarrow 0,\,r\rightarrow\infty,\, m=[\alpha].
\end{equation*}
The rest part of the proof is absolutely analogous to Theorem 2 \cite{firstab_lit:1kukushkin2021}.
 \end{proof}
 Bellow, we use the accepted above notation $\lambda_{n}:=\mu^{-1}_{n}(B)$ in the specific case related to the considered throughout the chapter operator  $B,$ in the rest part we preserve the general notation for the eigenvalues, i.e. $\lambda_{n}(L)$ - the $n$ - th eigenvalue of an arbitrary operator $L.$
\begin{teo}\label{T4.7} Assume that $B$ is a compact operator, $\Theta(B) \subset   \mathfrak{L}_{0}(\theta),\,\theta< \pi/2\alpha ,\;B\in \mathfrak{S}^{\star} _{\alpha},$
$$
 (\ln^{1+1/\alpha}x)'_{s^{-1}_{n}(B)}  =o(  n^{-1/\alpha}   ).
$$
where $\alpha$ is positive non integer.
Then the statement of the Theorem 4 \cite{firstab_lit:1kukushkin2021} remains true, i.e. statement S1 holds, moreover we have the following gaps between the eigenvalues
$$
|\lambda_{N_{\nu}+1}| -|\lambda_{N_{\nu}}| \geq K |\lambda_{N_{\nu}+1}|^{  1-\alpha},\,K>0,
$$
and the   eigenvalues corresponding to the partial sums become united to the groups  due to the distance of the power type
$$
|\lambda_{N_{\nu}+k}| -|\lambda_{N_{\nu}+k-1}| \leq K |\lambda_{N_{\nu}+k}|^{ 1-\alpha},\;
 k=2,3,...,N_{\nu+1}-N_{\nu}.
$$
\end{teo}
\begin{proof}
Using the theorem conditions, we obtain easily
$$
 s _{n}(B)   =o(  n^{-1/\alpha}   ),\;n\rightarrow\infty,
$$
  In accordance with   Corollary  3.2 \cite{firstab_lit:1Gohberg1965} (Chapter II, $\S$ 3.3),  it gives us
\begin{equation*}\label{8z}
|\mu_{n}(B)|=o(  n^{-1/\alpha}   ),\;n\rightarrow\infty.
\end{equation*}
In its own turn it allows us to prove the fact that there exist a constant  $K>0$ and  such a sequence of natural numbers $\{N_{\nu}\}_{0}^{\infty}$ that
$$
 |\lambda_{N_{\nu}+1}|-|\lambda_{N_{\nu}}|\geq K |\lambda_{N_{\nu}+1}|^{1-\alpha},
$$
where $\alpha$ is a fixed positive number.
Assume the contrary, then we obviously have
$$
\lim\limits_{k\rightarrow\infty}(|\lambda_{N_{\nu}+1}|-|\lambda_{N_{\nu}}|)/|\lambda_{N_{k}+1}|^{1-\alpha} =0.
$$
Thus, under the assumption  if we prove the following implication
 $$
\lim\limits_{n\rightarrow\infty}(|\lambda_{n+1}|-|\lambda_{n}|)/|\lambda_{n+1}|^{(p-1)/p} =0,\;\Longrightarrow  \lim\limits_{n\rightarrow\infty} |\lambda_{n}|/n^{p}=0,\;p>0,
$$
 we will prove the desired  result through the obtained contradiction with the assumption   $|\lambda_{n}|n^{-p}\geq C.$
  Here, we referee the reader to    Lemma 2 \cite{firstab_lit:2Agranovich1994}, where the proof is represented.
Consider the case $p\geq1,$  making the auxiliary denotation $\xi_{n}:=|\lambda_{n}|^{1/p},$ after the proof of the fact
 $$
 \lim\limits_{k\rightarrow\infty}(|\lambda_{n+1}|-|\lambda_{n}|)/|\lambda_{n+1}|^{(p-1)/p} =0,\,\Longrightarrow  \lim\limits_{n\rightarrow\infty}|\lambda_{n}| /|\lambda_{n+1}|=1,\,
 $$
 we can rewrite the desired   implication   as follows
$$
\lim\limits_{k\rightarrow\infty}(\xi^{p}_{n+1}-\xi^{p}_{n})/\xi^{ p-1 }_{n}=0,\;\Longrightarrow  \lim\limits_{k\rightarrow\infty} \xi_{n}/n=0,
$$
denoting $\Delta\xi_{n}:=\xi _{n+1}-\xi _{n},$ we can easily obtain, due to the application of the mean value theorem (or Lagrange theorem), the estimate $\xi^{p}_{n+1}-\xi^{p}_{n}\geq p\, \xi^{p-1}_{n} \Delta\xi_{n}.$ Taking into account the latter formula, we see that the following implication holds
$$
\left\{\lim\limits_{n\rightarrow\infty} \Delta\xi_{n}=0,\Rightarrow  \lim\limits_{n\rightarrow\infty} \xi_{n}/n=0\right\}\Rightarrow\left\{\lim\limits_{n\rightarrow\infty}(\xi^{p}_{n+1}-\xi^{p}_{n})/\xi^{ p-1 }_{n}=0,\;\Longrightarrow  \lim\limits_{k\rightarrow\infty} \xi_{n}/n=0\right\},
$$
now we have in the reminder the proof of the premise of the latter implication, i.e. the fact
\begin{equation}\label{9z}
\lim\limits_{n\rightarrow\infty} \Delta\xi_{n}=0,\Rightarrow  \lim\limits_{n\rightarrow\infty} \xi_{n}/n=0.
\end{equation}
Consider the formula
$$
 \sum\limits_{k=1}^{n-1}\Delta\xi_{k} /n= (\xi_{n}-\xi_{2})/n,
$$
on the other hand, we have
$$
\forall\varepsilon>0,\,\exists\, m(\varepsilon):\;\sum\limits_{k=1}^{n-1}\Delta\xi_{k}<C_{m}+(n-1-m)\varepsilon/2,\,\forall n>m+1.
$$
Choosing $n$ so that $C_{m}/n<\varepsilon/2,$ we complete the proof of the
however  the fact that the implication in the first term holds   can be established with no difficulties.  The proof corresponding to the case $p<1$ is absolutely analogous, we just need notice
$$
\lim\limits_{n\rightarrow\infty}(\xi^{p}_{n+1}-\xi^{p}_{n})/\xi^{ p-1 }_{n}=0,\Rightarrow \lim\limits_{n\rightarrow\infty} \xi^{p}_{n} /\xi^{ p  }_{n+1}=1,
$$
and deal with the following implication \eqref{9z} and therefore complete the proof of the statement.

 Let us find $\delta_{\nu}$ from the condition $R_{\nu}=K|\lambda_{N_{\nu}}|^{1-\alpha}+|\lambda_{N_{\nu}}|,\,R_{\nu}(1-\delta_{\nu})=|\lambda_{N_{\nu}}|,$ then $\delta_{\nu}^{-1}=1+K^{-1}|\lambda_{N_{\nu}}|^{\alpha}.$ Note that by virtue of such a choice, we have $R_{\nu}<R_{\nu+1}(1-\delta_{\nu+1}).$

  In accordance with the theorem  condition,  established above relation \eqref{4.22}, we have   $  n_{B}(r)=o( r^{\alpha}/\ln r).$
 Consider a contour
$$
 \vartheta(B):= \left\{\lambda:\;|\lambda|=r_{0}>0,\,|\mathrm{arg} \lambda|\leq \theta+\varepsilon\right\}\cup\left\{\lambda:\;|\lambda|>r_{0},\; |\mathrm{arg} \lambda|=\theta+\varepsilon\right\},
 $$
 here the number $r$ is chosen so that the  circle  with the corresponding radios does not contain values $\mu_{n},\,n\in \mathbb{N}.$
Applying  Lemma 5  \cite{firstab_lit:1kukushkin2021},  we have that there exists such a sequence $\{\tilde{R}_{\nu}\}_{0}^{\infty}:\;(1-\delta_{\nu }) R_{\nu}<\tilde{R}_{\nu}<R_{\nu }$ that the following estimate holds
$$
\|(I-\lambda B )^{-1}\|_{\mathfrak{H}}\leq e^{\gamma (|\lambda|)|\lambda|^{\alpha}}|\lambda|^{m},\,m=[\alpha],\,|\lambda|=\tilde{R}_{\nu},
$$
where
$$
\gamma(|\lambda|)= \beta ( |\lambda|^{m+1})  +(2+\ln\{4e/\delta_{\nu}\}) \beta(\sigma |\lambda|  ^{m+1}) \,\sigma ^{\alpha/(m+1)},\,\sigma:=\left(\frac{2e}{1-\delta_{0}}\right)^{m+1},\,|\lambda|=\tilde{R}_{\nu},
$$
$$
\;\beta(r )= r^{ -\frac{\alpha}{m+1} }\left(\int\limits_{0}^{r}\frac{n_{B^{m+1}}(t)dt}{t }+
r \int\limits_{r}^{\infty}\frac{n_{B^{m+1}}(t)dt}{t^{ 2  }}\right).
$$
Here we should explain that that the opportunity to obtain this estimate is based   upon  the estimation of the Fredholm  determinant $\Delta_{B^{m+1}}(\lambda^{m+1})$  absolute value  from bellow (see \cite[p.8]{firstab_lit:1Lidskii}). In its own turn the latter  is implemented via  the general estimate from bellow of the absolute value of a holomorphic function (Theorem 11 \cite[p.33]{firstab_lit:Eb. Levin}).
  However, to avoid any kind of misunderstanding let us implement a scheme of reasonings of  Lemma 5 consistently.\\

 Using  the theorem condition, we have  $B\in \mathfrak{S}_{\alpha+\varepsilon},\,\varepsilon>0.$
Obviously, we have
\begin{equation}\label{8a}
(I-\lambda B )^{-1}=(I-\lambda^{m+1}B^{m+1})^{-1}(I+\lambda B+\lambda^{2} B^{2}+...+\lambda^{m}B^{m}).
\end{equation}
In accordance with Lemma 3 \cite{firstab_lit:1Lidskii}, for sufficiently small $\varepsilon>0,$ we have
$$
\sum\limits_{n=1}^{\infty}s^{\frac{\alpha+\varepsilon}{m+1}}_{n}(  B^{m+1} )\leq \sum\limits_{n=1}^{\infty}s^{ \alpha+\varepsilon }_{n}( B )<\infty,
$$
       Applying  inequality  (1.27) \cite[p.10]{firstab_lit:1Lidskii} (since $  \rho  /(m+1)<1$), we get
$$
\|\Delta_{B^{m+1}}(\lambda^{m+1})(I-\lambda^{m+1} B^{m+1})^{-1}\| \leq C\prod\limits_{i=1}^{\infty}\{1+|\lambda^{m+1} s_{i}( B^{m+1} )|\},
$$
where $\Delta_{B^{m+1}}(\lambda^{m+1})$ is a Fredholm  determinant of the operator $B^{m+1}$  (see \cite[p.8]{firstab_lit:1Lidskii}).
Rewriting the formulas in accordance with the terms of the entire functions theory, we get
$$
  \prod\limits_{n=1}^{\infty}\{1+|\lambda^{m+1} s_{n}( B^{m+1} )|\}=\prod\limits_{n=1}^{\infty} G\left(  |\lambda|^{m+1}/ a_{n} ,p\right),
$$
where $a_{n}:=-s^{-1}_{n}(B^{m+1}),\,p=[\alpha /(m+1)]=0,$
using Lemma 2  \cite{firstab_lit:1kukushkin2021},   we get
$$
 \prod\limits_{n=1}^{\infty} G\left(  r^{m+1}/ a_{n} ,p\right)\leq  e^{ \beta (r^{m+1})r^{\alpha}},\,r>0.
$$
In accordance with Theorem 11 \cite[p.33]{firstab_lit:Eb. Levin},   and regarding to the case,   the following estimate holds
$$
|\Delta_{B^{m+1}}(\lambda^{m+1})|\geq e^{-(2+\ln\{2e/3\eta\})\ln\xi_{m}},\;\xi_{m}= \!\!\!\max\limits_{\psi\in[0,2\pi /(m+1)]}|\Delta_{B^{m+1}}([2e R_{\nu}e^{i\psi}]^{m+1})|,\,|\lambda|\leq R _{\nu},
$$
except for the exceptional set of circles with the sum of radii less that $4\eta R_{\nu},$  where $\eta$ is an arbitrary number less than  $2e/3 .$  Thus  to find the desired circle $\lambda= e^{i\psi}\tilde{R}_{\nu}$ belonging to the ring, i.e. $R_{\nu}(1-\delta_{\nu})<\tilde{R}_{\nu}<R_{\nu},$   we have to   choose $\eta$ satisfying the inequality
$$
4\eta R_{\nu}<R_{\nu}-R_{\nu}(1-\delta_{\nu})=\delta_{\nu} R_{\nu};\;\eta< \delta_{\nu}/4,
$$
for instance let us choose $\eta_{\nu}=\delta_{\nu}/6,$ here we should note that $\delta_{\nu}$ tends to zero, this is why without loss of generality of reasonings we are free in such a choice since the inequality $\eta_{\nu}<3 e/2$ would be satisfied for a sufficiently large $\nu.$ Under such assumptions, we can rewrite the above estimate in the form
$$
|\Delta_{B^{m+1}}(\lambda^{m+1})|\geq e^{-(2+\ln\{4e/\delta_{\nu}\})\ln\xi_{m}},\; |\lambda|=\tilde{R} _{\nu}.
$$
Note that in accordance with the estimate (1.21) \cite[p.10]{firstab_lit:1Lidskii}, Lemma 2  \cite{firstab_lit:1kukushkin2021} we have
$$
|\Delta_{B^{m+1}}(\lambda)|\leq  \prod\limits_{i=1}^{\infty}\{1+|\lambda s_{i}( B^{m+1} )|\}=\prod\limits_{n=1}^{\infty} G\left(  |\lambda|/ a_{n} ,p\right)\leq
 e^{ \beta (|\lambda| )|\lambda|^{\alpha/(m+1)}}.
$$
Using this estimate, we obtain
$$
\xi_{m}\leq e^{ \beta ([2e  R_{\nu}]^{m+1})( 2e R_{\nu}  )^{\alpha}}.
$$
Substituting, we get
$$
|\Delta_{B^{m+1}}(\lambda^{m+1})|^{-1}\leq e^{(2+\ln\{4e/\delta_{\nu}\})\beta ([2e  R_{\nu}]^{m+1})( 2e R_{\nu}  )^{\rho}},\; |\lambda|=\tilde{R} _{\nu}.
$$
Having noticed the facts
$$
 1-\delta_{\nu}  = \frac{|\lambda_{N_{\nu}}|^{\alpha}}{K+|\lambda_{N_{\nu}}|^{\alpha}}=1-\frac{K}{K+|\lambda_{N_{\nu}}|^{\alpha}}\geq 1-\frac{K}{K+|\lambda_{N_{0}}|^{\alpha}}= 1-\delta_{0} ;
 $$
 $$
 \;R_{\nu}<\tilde{R}_{\nu}(1-\delta_{\nu})^{-1}\leq \tilde{R}_{\nu}(1-\delta_{0})^{-1},
$$
we obtain
$$
|\Delta_{B^{m+1}}(\lambda^{m+1})|^{-1}\leq e^{(2+\ln\{4e/\delta_{\nu}\})\beta (\sigma \tilde{R}^{m+1}_{\nu} )   \tilde{R}^{\alpha}_{\nu} \sigma^{\alpha/(m+1)}  },\; |\lambda|=\tilde{R} _{\nu}.
$$
Combining the above estimates, we get

$$
 \|\Delta_{B^{m+1}}(\lambda^{m+1})(I-\lambda^{m+1} B^{m+1})^{-1}\| =|\Delta_{B^{m+1}}(\lambda^{m+1})|\cdot\|(I-\lambda^{m+1} B^{m+1})^{-1}\| \leq e^{ \beta (|\lambda|^{m+1})|\lambda|^{\alpha}};
$$
$$
\|(I-\lambda^{m+1} B^{m+1})^{-1}\|\leq e^{ \beta (|\lambda|^{m+1})|\lambda|^{\alpha}}|\Delta_{B^{m+1}}(\lambda^{m+1})|^{-1}=
e^{\gamma (|\lambda|)|\lambda|^{\alpha}}  ,\,|\lambda|=\tilde{R}_{\nu}.
$$
Consider relation \eqref{8a}, we have
$$
 \|(I-\lambda B )^{-1}\| \leq\|(I-\lambda^{m+1}B^{m+1})^{-1}\|  \cdot\|(I+\lambda B+\lambda^{2} B^{2}+...+\lambda^{m}B^{m})\| \leq
$$
$$
\leq\|(I-\lambda^{m+1}B^{m+1})^{-1}\|  \cdot \frac{|\lambda|^{m+1}\|B\|^{m+1}-1}{|\lambda|\cdot\|B\|-1}\leq C e^{\gamma (|\lambda|)|\lambda|^{\alpha}}|\lambda|^{m},\,|\lambda|=\tilde{R}_{\nu}.
$$
Applying the obtained estimate for the norm,  we can  claim that  for a sufficiently   small $\delta>0,$   there exists an arch
$\tilde{\gamma}_{ \nu }:=\{\lambda:\;|\lambda|=\tilde{R}_{\nu},\,|\mathrm{arg } \lambda|< \theta +\varepsilon\}$    in the ring $(1-\delta_{\nu})R_{\nu}<|\lambda|<R_{\nu},$ on which the following estimate holds
$$
 J_{  \nu  }: =\left\|\,\int\limits_{\tilde{\gamma}_{ \nu }}e^{-\lambda^{\alpha}t}B(I-\lambda B)^{-1}f d \lambda\,\right\| \leq \,\int\limits_{\tilde{\gamma}_{ \nu }}e^{- t \mathrm{Re}\lambda^{\alpha}}\left\|B(I-\lambda B)^{-1}f \right\|  |d \lambda|\leq
$$
 $$
 \leq \|f\| C  e^{\gamma (|\lambda|)|\lambda|^{\alpha} }|\lambda|^{m }\int\limits_{-\theta-\varepsilon}^{\theta+\varepsilon} e^{- t \mathrm{Re}\lambda^{\alpha}} d \,\mathrm{arg} \lambda,\,|\lambda|=\tilde{R}_{\nu}.
$$
Using the theorem conditions,    we get     $\,|\mathrm{arg} \lambda |<\pi/2\alpha,\,\lambda\in \tilde{\gamma}_{ \nu },\,\nu=0,1,2,...\,,$ we get
$$
\mathrm{Re }\lambda^{\alpha}\geq |\lambda|^{\alpha} \cos \left[(\pi/2\alpha-\delta)\alpha\right]= |\lambda|^{\alpha} \sin     \alpha \delta,
$$
 where $\delta$ is a sufficiently small number.  Therefore
\begin{equation}\label{13b}
 J_{\nu} \leq C e^{|\lambda|^{\alpha}\{\gamma (|\lambda|)-t  \sin     \alpha \delta\}}|\lambda|^{m }
 ,\,m=[\alpha],\,|\lambda|=\tilde{R}_{\nu}.
\end{equation}
It is clear that within the contour $\vartheta(B),$ between the arches $\tilde{\gamma}_{ \nu },\tilde{\gamma}_{ \nu+1 }$ (we denote the boundary of this  domain by $\gamma_{ \nu}$) there  lie the eigenvalues only for which the following relation holds
$$
|\lambda_{N_{\nu}+k }|-|\lambda_{N_{\nu}+k-1 }|\leq K |\lambda_{N_{\nu}+k-1 }|^{1-\alpha},\;
 k=2,3,...,N_{\nu+1}-N_{\nu} .
$$
Using Lemma 8, \cite{firstab_lit:1kukushkin2021},  we   obtain a relation
$$
 \frac{1}{2\pi i}\int\limits_{\gamma_{ \nu}}e^{-\lambda^{\alpha}t}B(I-\lambda B)^{-1}f d \lambda
=     \sum\limits_{q=N_{\nu}+1}^{N_{\nu+1}}\sum\limits_{\xi=1}^{m(q)}\sum\limits_{i=0}^{k(q_{\xi})}e_{q_{\xi}+i}c_{q_{\xi}+i}(t).
$$
It is clear that to obtain the desired result, we should prove that the series composed  of the above terms converges. Consider the following auxiliary denotations originated from the reasonings $\gamma_{\nu_{+}}:=
\{\lambda:\, \tilde{R}_{\nu}<|\lambda|<\tilde{R}_{\nu+1},\, \mathrm{arg} \lambda  =\theta   +\varepsilon\},\,\gamma_{\nu_{-}}:=
\{\lambda:\,\tilde{R}_{\nu}<|\lambda|<\tilde{R}_{\nu+1},\, \mathrm{arg} \lambda  =-\theta   -\varepsilon\},$
$$
J^{+}_{\nu}: =\left\|\,\int\limits_{\gamma_{\nu_{+}}}e^{-\lambda^{\alpha}t}B(I-\lambda B)^{-1}f d \lambda\,\right\|,\; J^{-}_{\nu}:= \left\|\,\int\limits_{\gamma_{\nu_{-}}}e^{-\lambda^{\alpha}t}B(I-\lambda B)^{-1}f d \lambda\,\right\|,
$$
we have
$$
\left\| \int\limits_{\gamma_{ \nu}}e^{-\lambda^{\alpha}t}B(I-\lambda B)^{-1}f d \lambda\right\|\leq J_{\nu}+J_{\nu+1}+J^{+}_{\nu}+J^{-}_{\nu}.
 $$
Thus, it is clear that to establish S1, it suffices to establish the facts
\begin{equation}\label{eq16}
\sum\limits_{\nu=0}^{\infty}J_{\nu}<\infty,\;\sum\limits_{\nu=0}^{\infty}J^{+}_{\nu}<\infty,\;\sum\limits_{\nu=0}^{\infty}J^{-}_{\nu}<\infty.
\end{equation}
Consider the right-hand side of formula \eqref{13b}.
Substituting $\delta^{-1}_{\nu},$ we have
$$
\ln\{4e/\delta_{\nu}\} = 1+ \ln4 +\ln \{1+K^{-1}|\lambda_{N_{\nu}}|^{\alpha}\} \leq C \ln   |\lambda_{N_{\nu}}|     .
$$
Hence, to obtain the desired result we should prove that $ \beta (\sigma| \lambda_{N_{\nu}}| ^{m+1})\ln  |\lambda_{N_{\nu}}|\rightarrow 0,\,\nu\rightarrow\infty.$ Note that in accordance with relation \eqref{4.21}, we can prove the latter relation, if we show that
\begin{equation}\label{13a}
\ln r\frac{n_{B^{m+1}}(r^{m+1})}{r^{\alpha} }\rightarrow 0,\;r\rightarrow\infty.
\end{equation}
In accordance with \eqref{4.22}, we have
\begin{equation*}
(\ln^{1+1/\alpha}x)'_{s^{-1}_{m}(B)}  =o(  m^{-1/\alpha}   )
 , \Longrightarrow \left(\ln r\frac{n(r)}{r^{\rho}}\,\rightarrow 0,\,r\rightarrow\infty\right),
\end{equation*}
Applying \eqref{4.24}, we have
$$
n_{B^{m+1}}(r^{m+1})\leq (m+1) n_{B}(r),
$$
hence
$$
\ln r\frac{n_{B^{m+1}}(r^{m+1})}{r^{\alpha} }\leq (m+1)\ln r\frac{n_{B }(r )}{r^{\alpha} },
$$
what gives us the desired result, i.e. we compleat the proof of the first relation \eqref{eq16}.

In accordance with Lemma 6 \cite{firstab_lit:1kukushkin2021}, we can claim  that   on each ray $\zeta$ containing the point zero and not belonging to the sector $\mathfrak{L}_{0}(\theta)$ as well as the  real axis, we have
$$
\|(I-\lambda B)^{-1}\|\leq \frac{1}{\sin\psi},\,\lambda\in \zeta,
$$
where $\,\psi = \min \{|\mathrm{arg}\zeta -\theta|,|\mathrm{arg}\zeta +\theta|\}.$ Applying this estimate, the established above estimate $\mathrm{Re }\lambda^{\alpha}\geq |\lambda|^{\alpha} \sin     \alpha \delta,$ we get
$$
 J^{+}_{\nu} \leq C\|f\|  \!\!\! \int\limits_{ \tilde{R}_{\nu} }^{\tilde{R}_{\nu+1} }  e^{- t \mathrm{Re }\lambda^{\alpha}}   |d   \lambda|\leq C \|f\| e^{-t \tilde{R}_{\nu}  ^{\alpha} \sin     \alpha\delta}   \!\!\!\int\limits_{ \tilde{R}_{\nu} }^{\tilde{R}_{\nu+1} }   |d   \lambda|=
 $$
 $$
 = C\|f\| e^{-t \tilde{R}_{\nu}  ^{\alpha} \sin     \alpha\delta}( \tilde{R}_{\nu+1}-\tilde{R}_{\nu}) .
$$
It obvious, that the same estimate can be obtained for $J^{-}_{\nu}.$ Therefore, the second and the third relations \eqref{eq16} hold. The proof is complete.
 \end{proof}
\begin{remark}\label{Rem4.1} Application of the results established in paragraph 1,  under the imposed sectorial condition upon   the compact  operator $B,$   gives us
$$
\left\{(\ln^{1+1/\rho}x)'_{\lambda^{-1}_{m}(\mathfrak{Re}B)}  =o(  m^{-1/\rho}   )\right\}\Longrightarrow\left\{(\ln^{1+1/\rho}x)'_{s^{-1}_{m}(B)}  =o(  m^{-1/\rho}   )\right\},
$$
what becomes clear if we recall $s_{m}(B)\leq C \lambda_{m}(\mathfrak{Re}B).$ However to establish the equivalence
$$
B\in\mathfrak{S}^{\star}_{\rho}\Longleftrightarrow \mathfrak{Re} B\in\mathfrak{S}^{\star}_{\rho},
$$  we require the estimate from bellow $s_{m}(B)\geq C \lambda_{m}(\mathfrak{Re}B),$ which in its own turn can be established due to a more subtle technique and under  stronger  conditions regarding the operator $B.$ In particular the second representation theorem is involved and conditions upon the matrix coefficients of the operator $B$ are imposed in accordance with which they should decrease sufficiently fast. In addition,  the application of Theorem 5 \cite{firstab_lit(arXiv non-self)kukushkin2018}  (if $B$ satisfies its conditions)  gives us the relation
$
\lambda_{n}(\mathfrak{Re}   B)\asymp \lambda_{n}(H^{-1}),\,H=\mathfrak{Re}B^{-1},
$
the latter allows one to deal with unbounded operators reformulating Theorem \ref{T4.7} in terms of the  Hermitian real component. It should be noted that the made remark  remains true regarding the implication
$$
\left\{\lambda_{m}(\mathfrak{Re} B)=o(m^{-1/\rho})\right\}\Longrightarrow  \left\{s_{m}(B)=o(m^{-1/\rho})\right\},
$$
and since the scheme of the reasonings is the same  we left  them  to the reader.
\end{remark}

\begin{ex}\label{Ex4.2} Using decomposition \eqref{4.20},  let us construct artificially  the sectorial operator $B$ satisfying the Theorem \ref{T4.6} condition. Observe the produced above example $s_{m}(B)=(m \ln m)^{-1/\rho},$ as it was said above the latter condition guarantees $B\in \mathfrak{S}_{\rho}^{\star}.$ However, the problem is how to choose in the polar decomposition formula  the unitary operator  $U$ providing the sectorial property of the operator $B.$  The following approach is rather rough but  can be used to supply the desired example, in terms of  formula  \eqref{4.20}, we get
$$
 (Bf,f)=\sum\limits_{n=1}^{\infty}s_{n}(B)(f,e_{n})\overline{(f,g_{n})}.
 $$
 Therefore, if we impose the condition
 $$
 \left|\arg (f,e_{n})-   \arg (f,g_{n})\right|<\theta,\,n\in \mathbb{N},\,f\in \mathrm{R}(B),
 $$
 then we obtain the desired sectorial condition, where $\theta$ is a semi-angle of the sector. It is remarkable that that the selfadjoint case corresponds to the zero difference since in this case $U=I.$ However, we should remark that following to the classical mathematical style it is necessary to formulate conditions of the sectorial property in terms of the eigenvectors of the operator $|B|$ but it is not so easy matter requiring comprehensive analysis  of the operator $U$ properties provided with   concrete cases.
\end{ex}
\section{ Essential decreasing of the summation order}

Another method allows us to decrease the summation order essentially! In the paper \cite{firstab_lit:2Agranovich1994} Agranovich M.S. considered the condition
\begin{equation}\label{4.30}
 \overline{\lim\limits_{n\rightarrow\infty}}\lambda_{n}(\mathfrak{Re} W)n^{-p}>0,\,p>0.
\end{equation}
It can be shown easily that the assumption of the fulfilment (1.1),(1.2) \cite{firstab_lit:2Agranovich1994} are equivalent to conditions H1,H2 \cite{kukushkin2021a}, moreover   condition (1.2) contains tautology that can be diminished. Consider the additional condition imposed in the paper \cite{firstab_lit:2Agranovich1994}, the equivalent variant in terms of this monograph is represented bellow
\begin{equation}\label{4.31}
|(\mathfrak{Im}W f,f)|\leq C (\mathfrak{Re}f,f)^{2q},\,f\in \mathfrak{M},\;0\leq q<1.
\end{equation}
As we can see the condition \eqref{4.30} is more subtle than one of the power type, however the condition \eqref{4.31} can be treated as na existence of the parabolic-like domain containing the numerical range of values of the operator what is an essential weakness of conditions in comparison with the sectorial property. Note that the condition \eqref{4.30} guarantees   the existence of the subsequence  and $K>0$ so that
$$
  |\lambda_{N_{\nu}+1}| -|\lambda_{N_{\nu}}| \geq K |\lambda_{N_{\nu}}|^{ 1-1/p},\;\nu\in \mathbb{N}_{0}.
$$
Note that the   example   $\lambda_{n}=n^{p},\,p>1$ creates a prerequisite to consider a corresponding condition. Regarding this case, we can choose
$
 N_{\nu}=\nu^{q},\,q\in \mathbb{N};\,\lambda_{N_{\nu}}=\nu^{qp};\,\lambda_{N_{\nu}+1}=(\nu^{q}+1)^{p}.
$
The fact is that the previous estimate for the eigenvalues will be preserved, i.e.
$$
(\nu^{q}+1)^{p}-\nu^{qp}\geq p \nu^{qp(1-1/p)}.
$$
This gives us an opportunity to consider  the corresponding projectors $\mathcal{P}_{\nu}(\alpha,t)$ and put the brackets in the series in the definite way.

Now, consider a sectorial  operator $W$ with a discrete spectrum which inverse belongs to the Schatten class $\mathfrak{S}_{\sigma},\;0<\sigma<\infty.$ The  latter condition gives us the estimate
$$
|\lambda_{n}(W)|\geq Cn^{1/\sigma}.
$$
Indeed, the belonging to the Schatten class implies that (due to the test for a convergent series) $s_{n}(B)=o(n^{-1/\sigma}),$ taking into account the implication
(see \cite{firstab_lit:1Gohberg1965})

$$
s_{n}(B)=o(n^{-p}),\Rightarrow |\mu_{n}(B)|=o(n^{-p}),\,p>0,\,B\in  \mathfrak{S}_{\infty},
$$
we obtain the desired result.
Let us choose an arbitrary $q>0,\,K\in \mathbb{N}$ and extend the  previously defined $N_{\nu}$ to the real numbers so that $N_{\nu}:=K[\nu^{q}].$ Without loosing generality of reasonings we assume that $K=1.$  In accordance with the estimate for the eigenvalues established above, we have the following representation
$$
|\lambda_{n}|=C_{n}n^{1/\sigma},\,C_{n}>C,\,n\in \mathbb{N}.
$$
Using the Lagrange  mean value theorem, we get
$$
  |\lambda_{N_{\nu}+1}| -|\lambda_{N_{\nu}}|=  C_{N_{\nu}+1}([\nu^{q}]+1)^{1/\sigma}  - C_{N_{\nu}} [\nu^{q}]^{1/\sigma} =\sigma^{-1}\, \xi^{1/\sigma-1},
  $$
  $$
  \;\xi\in \left(C^{\sigma}_{N_{\nu}} [\nu^{q}],C^{\sigma}_{N_{\nu+1}} ([\nu^{q}]+1)    \right),
$$
here we used the fact that the eigenvalues arranged in order of their absolute value increasing,
so that, we have
$$
C_{n}n^{1/\sigma}>C_{n+1}(n+1)^{1/\sigma}.
$$
Therefore if $\sigma<1,$ we have
$$
  |\lambda_{N_{\nu}+1}| -|\lambda_{N_{\nu}}|\geq \sigma^{-1}\, \left\{C^{\sigma}_{N_{\nu}} [\nu^{q}]\right\}^{1/\sigma-1}= \sigma^{-1}\left\{C_{N_{\nu}} [\nu^{q}]^{1/\sigma}\right\}^{1-\sigma}=\sigma^{-1}|\lambda_{N_{\nu}}|^{1-\sigma}.
$$

It is remarkable that the considered phenomena gives us a  key to the method allowing to essentially decrease the summation order. To make the matter clear, firstly we consider a simplified case choosing
$$
N_{\nu}=\beta(\nu+1)^{\gamma},\; \gamma=\beta+\eta,\; \eta\in \mathbb{N},\;\beta/\eta \in \mathbb{N},
$$
in accordance with the made assumptions the values of the defined in this way function $N_{\nu}$ belongs to the subset of natural numbers. We have the following lemma.

\begin{lem}\label{L4.15} Consider a sectorial  operator $W$ with a discrete spectrum which inverse belongs to the Schatten class $\mathfrak{S}_{\sigma},\;0<\sigma<\infty.$
 The subsequence of the natural numbers $N_{\nu}$ can be chosen so that the following decomposition holds
$$
  |\lambda_{N_{\nu}+1}| -|\lambda_{N_{\nu}}|\geq \sigma^{-1}|\lambda_{N_{\nu}}|^{1-\sigma},\,\,N_{\nu}=\sum_{k=0}^{\nu^{\eta}}N_{k\nu  },
  $$
where
\begin{equation*}
N_{k\nu}= \left\{ \begin{aligned}
\!\gamma(\nu^{\beta}-k^{\beta/\eta}),\,\nu>k^{1/\eta}, \\
  0,\,\nu\leq k^{1/\eta}    \\
\end{aligned}
 \right.\;,\;N_{0\nu}=O^{\ast}(\nu^{\gamma-1}),
\end{equation*}
$$
 \gamma=\beta+ \eta ,\; \eta\in \mathbb{N},\;\beta/\eta \in \mathbb{N}.
 $$
\end{lem}
 \begin{proof}
Consider a representation
$
N_{\nu}= \beta(\nu+1)^{\gamma}
$
and  the following equation
$$
N_{\nu}=N_{0\nu}+\gamma\sum_{k=1}^{\nu^{ \eta}}(\nu^{\beta}-k^{\beta/\eta} ).
$$
where $N_{0\nu}$ is a solution - an unknown  subsequence of natural numbers having the index in  the asymptotic formula less than the given $N_{\nu},$ i.e.
$$
N_{0\nu}=O^{\ast}(\nu^{\xi}),\,\xi<\gamma.
$$
Let us solve the equation in the following way - find the index $\xi.$ For this purpose,   consider the estimate  between the given sum and a corresponding  definite integral, i.e. calculating the integral, we have on the one hand
$$
\sum\limits_{k=1}^{\nu^{ \eta}}k^{ \beta/\eta}\leq\int\limits_{1}^{\nu^{ \eta}+1}x^{\beta/\eta}dx=\frac{\eta(\nu^{ \eta}+1)^{\gamma/\eta}}{\gamma}+\frac{\eta}{\gamma },
$$
on the other hand
$$
\sum\limits_{k=2}^{\nu^{ \eta}}k^{ \beta/\eta}\geq \int\limits_{1}^{\nu^{ \eta} }x^{\beta/\eta}dx=\frac{\eta\nu^{\gamma }}{\gamma}+\frac{\eta}{\gamma}.
$$
Taking into account the fact
$$
\nu^{ \eta}+1\leq(\nu+1)^{ \eta},\, \eta\geq1,
$$
which can be easily proved due to the application of the Lagrange  mean value theorem
$$
(\nu+1)^{ \eta}-\nu^{ \eta}= \eta \cdot\xi^{ \eta-1},\,\xi\in (\nu,\nu+1),\,\nu\in \mathbb{N},
$$
we obtain the following two-sided estimate
$$
 \frac{\eta\nu ^{\gamma}}{\gamma}-\frac{\eta}{\gamma}   +1\leq\sum_{k=1}^{\nu^{ \eta}} k^{\beta/\eta} \leq \frac{\eta(\nu+1)^{\gamma}}{\gamma}  -\frac{\eta}{\gamma}.
$$
Having calculated $\eta/\gamma+\beta/\gamma=1,$ we get
$$
\gamma \nu ^{\gamma}   + \beta +  \beta \left\{ (\nu+1)^{\gamma}- \nu ^{\gamma}\right\}         \leq N_{\nu}+\gamma\sum_{k=1}^{\nu^{ \eta}} k^{\beta/\eta}  \leq   \gamma(\nu+1)^{\gamma}   - \eta  ,
 $$
since $\eta/\gamma+\beta/\gamma=1.$ Having noticed the fact
$$
\sum\limits_{k=1}^{\nu^{ \eta}}\nu^{  \beta}=\nu^{\gamma},
$$
we get
$$
  \beta +  \beta \left\{ (\nu+1)^{\gamma}- \nu ^{\gamma}\right\}         \leq N_{0\nu}  \leq   \gamma\{(\nu+1)^{\gamma}-\nu^{\gamma} \}  - \eta  .
 $$
Applying the Lagrange  mean value theorem, we get
$$
   \beta +  \gamma\beta \nu^{\gamma-1}        \leq N_{0\nu}   \leq    \gamma^{2} (\nu+1)^{\gamma-1}   - \eta  .
 $$
Hence
$$
N_{0\nu}= O^{\ast}(\nu^{\gamma-1}),\,\nu\rightarrow\infty.
$$
\end{proof} Now let us rearrange the sequence of the values (characteristic numbers of the operator $B$)  $\{\lambda_{n}\}_{1}^{\infty}$ in the   groups $\{\lambda_{k_{j}}\}_{1}^{\infty}$ corresponding to the numbers $N_{k\nu},\,k=0,1,...,$ so that in domain of the complex plane $\{z:\,|\lambda_{N_{\nu}}|<|z|\leq|\lambda_{N_{\nu+1}}|\} $ we have $N_{k\nu+1}-N_{k\nu}$  values of the $k$-th group and
$$
\{\lambda_{n}\}_{1}^{\infty}=\bigcup\limits_{k=1}^{\infty} \{\lambda_{k_{j}}\}_{1}^{\infty}.
$$
 Denote the counting function of the $k$-th group of the characteristic numbers  by $n_{k}(r).$ Using the lemma results, it is not hard to prove the following fact.
\begin{corol}\label{C4.1}   Under the lemma assumptions, we have
\begin{equation*}
\lim\limits_{r\rightarrow\infty} \frac{n_{k}(r)}{r^{\varphi_{k}}}=0,\;
 \varphi_{k}= \left\{ \begin{aligned}
\!\sigma(1-1/\gamma),\,k=0, \\
  \sigma \beta/\gamma,\,k=1,2,...\,,  \\
\end{aligned}
 \right.\; .
\end{equation*}
\end{corol}
\begin{proof}
Note that in accordance with the fact that the operator belongs to the Schatten class $\mathfrak{S}_{\sigma},$  we have
$$
\lim\limits_{r\rightarrow\infty} \frac{n(r)}{r^{\sigma}}=0,
$$
where $n(r)$ is a counting function of the characteristic numbers  of the operator $B.$
This fact obviously follows   from the implication
$$
s_{n}(B)=o(n^{-1/\sigma}) \Rightarrow |\mu_{n}(B)|=o(n^{-1/\sigma}),
$$
see \cite{firstab_lit:1Gohberg1965}. It gives us
$$
\frac{\nu^{\gamma}  }{|\lambda^{\sigma }_{N_{\nu} }|}\leq \frac{N _{\nu } }{|\lambda^{\sigma }_{N_{\nu} }|} \leq C.
$$
Therefore, using the estimate $N_{0\nu+1 }\leq C(\nu+1)^{\gamma-1},$ we get
$$
\frac{N_{0\nu+1 } }{|\lambda^{\sigma(1-1/\gamma) }_{N_{\nu} }|}\leq C \frac{(\nu+1)^{\gamma-1}  }{|\lambda^{\sigma (1-1/\gamma) }_{N_{\nu} }|}=
C\left\{\frac{\nu^{\gamma}  }{|\lambda^{\sigma }_{N_{\nu} }|}\right\}^{1-1/\gamma} \leq  C.
$$
Observe  the following  the fact
$
|\lambda_{N_{\nu } }|<|\lambda_{N_{0\nu+1} }|<|\lambda_{N_{\nu +1} }|,
$
it becomes clear if we recall that $\lambda_{N_{0\nu+1}}$ is one of the characteristic numbers
$$
  \lambda_{N_{\nu } +1},\lambda_{N_{\nu } +2},...,\lambda_{N_{\nu +1}},
$$
thus, we obtain
$$
\frac{N_{0\nu+1 } }{|\lambda^{\sigma(1-1/\gamma) }_{N_{0\nu+1} }|}\leq C,\;\nu\in \mathbb{N}_{0}.
$$
 In accordance with the asymptotic formula $N_{0\nu}=O^{\ast}(\nu^{\gamma-1}),$  we have the following relation
$$
C_{3}\leq N_{0\nu+1}/  N_{0\nu} \leq C_{4}.
$$
Hence, taking into account  the order of the eigenvalues, we get
$$
\frac{N_{0\nu}+k}{|\lambda^{ \varphi_{0}}_{N_{0\nu}+k}|}  <\frac{N_{0\nu+1} }{|\lambda^{\varphi_{0}}_{N_{0\nu} }|}<C_{4}\frac{N_{0\nu } }{|\lambda^{\varphi_{0}}_{N_{0\nu} }|}\leq C,\,\varphi_{0}=\sigma(1-1/\gamma).
$$
Thus, combining the obtained results, we have the following  relation for the counting function   of the characteristic numbers  $\{\lambda_{N_{0\nu} }\},$ we have
$$
\lim\limits_{r\rightarrow\infty} \frac{n_{0}(r)}{r^{\varphi_{0}}}=0,\,\varphi_{0}=\sigma(1-1/\gamma).
$$
  Absolutely analogously, we obtain
$$
 \lim\limits_{r\rightarrow\infty} \frac{n_{k}(r)}{r^{\varphi_{k}}}=0,\,\varphi_{k}=\sigma \beta/\gamma,\,k=1,2,...\,.
$$
\end{proof}

Having noticed the fact
$
(1-1/\gamma)\geq\beta/\gamma
$
that holds for the acceptable parameter values, we can invent a scheme of reasonings allowing to    decrease the summation  order  from $\alpha>\rho$  to $s>\rho(1-1/\gamma),$ where $\rho=\inf \sigma.$     Before formulating the main result, we need the following lemma.

 \begin{lem}\label{L4.16} Assume that $B$ is  a compact  operator $B\in \mathfrak{S}_{1},\,\Theta(B)\subset \mathfrak{L}_{0}(\theta),$     then the following estimate holds
$$
\prod\limits_{n=1}^{\infty}|1+\lambda\mu_{n}(Q_{1}BQ_{1})| \leq \prod\limits_{n=1}^{\infty}|1+\lambda\mu_{n}( B )|,
$$
where $Q_{1}$ is the orthogonal projector corresponding to the orthogonal complement of the one-dimensional  subspace generated by   an element $f\in \mathfrak{H}.$
\end{lem}
\begin{proof} In accordance with  Theorem 2.3 (Lidskii) Chapter V  \cite{firstab_lit:1Gohberg1965} the system of the root vectors of the operator $iB$ is compleat in $\mathfrak{H}.$ Indeed, we will prove it if we show that $iB$ is dissipative and  $iB\in \mathfrak{S}_{1}.$  Taking into account the  accretive  property of the operator $B,$ i.e. $\Theta(B) \subset \mathfrak{L}_{0}(\theta),$   we conclude that
$$
\mathrm{Im}(iBf,f)=\mathrm{Re}(Bf,f)\geq 0,\,f\in \mathfrak{H}.
$$
Therefore, the operator $iB$ is dissipative. It is clear that $s_{n} (iB)=s_{n} (B),\,n\in \mathbb{N},$ hence $iB\in \mathfrak{S}_{1}.$ Thus, we obtain the desired result.  Now, if we note that the operators $B$ and $iB$ have the same root vectors, we obtain the fact that the operator $B$ has a complete system of the root vectors.

The fact that the completeness  property is preserved under the orthogonal projection is described comprehensively in  Lemma 1.2   \cite{firstab_lit:1Gohberg1965},  Chapter V, however we will show that it follows directly by virtue of the sectorial  property of the operator.  For this purpose, note that
in accordance with   Lemma 1 \cite{firstab_lit:1Lidskii}, we have
\begin{equation}\label{4.32}
s_{n}(Q_{1}BQ_{1} )\leq s_{n}( B ),\,n=1,2,...\,,
\end{equation}
hence $Q_{1}BQ_{1}\in \mathfrak{S}_{1}.$  Note that by virtue of the accretive   property, we have
$$
   \mathrm{Re} (Q_{1}BQ_{1}f,f)=\mathrm{Re}(BQ_{1}f,Q_{1}f)  \geq 0 ,\,f\in \mathfrak{H},
$$
then using the above reasonings we obtain the fact that the system of the root vectors of the operator $Q_{1}BQ_{1}$ is complete in $\mathfrak{H}.$
  Consider the operators $\tilde{B}_{1}(\lambda):= B^{\ast}_{1}(\lambda)B_{1}(\lambda),\,B_{1}(\lambda):=(Q_{1}+\lambda  B_{1} ),\,B_{1}:=Q_{1}BQ_{1}$ and
$\tilde{B}(\lambda)=B^{\ast}(\lambda)B(\lambda),\,B(\lambda)=(I+\lambda  B ).$
 Note that
 $$
 \tilde{B}(\lambda)=I+C(\lambda),\,C(\lambda)= |\lambda B|^{2}+2\mathfrak{Re} (\lambda B),\;\tilde{B}_{1}(\lambda)=Q_{1}+C_{1}(\lambda),\,C_{1}(\lambda):= |\lambda B_{1}|^{2}+2\mathfrak{Re} (\lambda B_{1}),
 $$
 where $|B|^{2}:=B^{\ast}B.$ It is clear that  $C(\lambda)$ and $\tilde{B}(\lambda)$ have the same eigenvectors and since $C(\lambda)$ is compact selfadjoint then the set of the eigenvectors   is complete in $\overline{\mathrm{R}(C(\lambda))},$  and  we can choose a basis $\{e_{k}\}_{1}^{n}$ in $\overline{\mathrm{R}(C(\lambda))}$   such that the operator  matrix will have a diagonal form - the eigenvalues are situated on the major diagonal. It is clear that the same reasonings are true for the operator $C_{1}(\lambda).$ Observe that    $C(\lambda),C_{1}(\lambda)\in \mathfrak{S}_{1},$
  therefore,   applying Corollary 1.1 ($2^{\circ},3^{\circ}$), Chapter IV \cite{firstab_lit:1Gohberg1965}, we get
 $$
 \lim\limits_{n\rightarrow\infty}\det\{P_{n}\tilde{B}(\lambda)P_{n}\}=\lim\limits_{n\rightarrow\infty}\det\{P_{n}+P_{n}C(\lambda)P_{n}\}=\det\{I+ C(\lambda) \},
 $$
 where $P_{n}$ is an orthogonal projector into the subspace generated by the eigenvectors $\{e_{k}\}_{1}^{n}$ (corresponding to non-zero eigenvalues) of the operator $\tilde{B}(\lambda).$ Analogously, we get
 $$
 \lim\limits_{n\rightarrow\infty}\det\{P_{1n}+P_{1n}C_{1}(\lambda)P_{1n}\}=\det\{I+ C_{1}(\lambda) \},
 $$
 where $P_{1n}$ is an orthogonal projector into the subspace generated by the eigenvectors $\{e_{1k}\}_{1}^{n}$ of the operator $\tilde{B_{1}}(\lambda).$
Note that
$
 Q_{1}P_{1n} = P_{1n} ,
$
therefore
$$
\lim\limits_{n\rightarrow\infty}\det\{P_{1n}\tilde{B_{1}}(\lambda)P_{1n}\}=\lim\limits_{n\rightarrow\infty}\det\{P_{1n}+P_{1n}C_{1}(\lambda)P_{1n}\}=\det\{I+ C_{1}(\lambda) \}.
$$
Consider
$$
(C_{1}(\lambda)f,f)=(Q_{1}B^{\ast}Q_{1} B Q_{1}f,f)+2(\mathfrak{Re} (\lambda B_{1})f,f)\leq (B^{\ast}  B Q_{1}f,Q_{1}f)+2(\mathfrak{Re} (\lambda B )Q_{1}f,Q_{1}f)=
$$
$$
=(Q_{1}C (\lambda)Q_{1}f,f),
$$
here we used the obvious relations
$$
(Q_{1}B^{\ast}Q_{1} B Q_{1}f,f)\leq (B^{\ast} B Q_{1}f,Q_{1}f),\;\mathfrak{Re} (\lambda B_{1})=Q_{1}\mathfrak{Re} (\lambda B)Q_{1}.
$$
Since  $C (\lambda)$ is a compact non-negative selfadjoint operator, then by virtue  to the minimax principle for the eigenvalues  (see Courant theorem  \cite[p.120]{firstab_lit: Courant-Hilbert}), we get
$$
\mu_{n}( Q_{1}C (\lambda)Q_{1})\leq \mu_{n}(C (\lambda));\;
\mu_{n}(C_{1} (\lambda))\leq \mu_{n}(C (\lambda)).
$$
Therefore, taking into account the denotations given above, we get
$$
\det\{P_{1n}\tilde{B_{1}}(\lambda)P_{1n}\}\leq\det\{P_{n}\tilde{B}(\lambda)P_{n}\},\,n\in \mathbb{N},
$$
  it follows that
\begin{equation}\label{4.33}
\det\{I+C_{1}(\lambda) \}\leq\det\{I+C(\lambda) \}.
\end{equation}
On the other hand since the root vector systems of the operators $B,B_{1}$ are compleat in $\mathfrak{H}$ (in the second case we should also consider a zero eigenvalue and the corresponding eigenvector) then we can construct in both cases an orthogonal Schur basis (see Lemma 4.1 Chapter I \cite{firstab_lit:1Gohberg1965}) so that the operators matrices have a triangle form. Thus choosing  corresponding orthogonal projectors (they correspond to the bases), the property of the triangle determinant, we have
$$
\det\{P_{n}\tilde{B}(\lambda)P_{n}\}=\det\{P_{n} B (\lambda)P_{n}\}\det\{P_{n} B^{\ast} (\lambda)P_{n}\}=   \prod\limits_{k=1}^{n}|1+\lambda\mu_{k}( B )|^{2},\;
$$
$$
\det\{P_{1n}\tilde{B_{1}}(\lambda)P_{1n}\}=\det\{P_{1n} B_{1} (\lambda)P_{1n}\}\det\{P_{1n} B^{\ast}_{1} (\lambda)P_{1n}\}=\prod\limits_{k=1}^{n}|1+\lambda\mu_{k}(Q_{1}BQ_{1})|^{2}.
$$
Analogously to the above, we conclude that
$$
 \lim\limits_{n\rightarrow\infty}\det\{P_{n}\tilde{B}(\lambda)P_{n}\} =\det\{I+C(\lambda) \},\;\lim\limits_{n\rightarrow\infty}\det\{P_{1n}\tilde{B_{1}}(\lambda)P_{1n}\} =\det\{I+ C_{1}(\lambda) \}.
$$
Taking into account \eqref{4.33}, we get
$$
\prod\limits_{n=1}^{\infty}|1+\lambda\mu_{n}(Q_{1}BQ_{1})| \leq \prod\limits_{n=1}^{\infty}|1+\lambda\mu_{n}( B )|.
$$

\end{proof}

 We need the following lemma.

\begin{lem}\label{L4.17}Assume that $B$ is a  operator having a set of the eigenvalues $\{\mu_{j}\}_{1}^{\infty},$  then it induces a  compact restriction $B_{k}$    on the invariant subspace $\mathfrak{M}_{k}$ obtained as a closure of the root vectors corresponding to an arbitrary subset of the eigenvalues  $\{\mu_{k_{j}}\}_{1}^{\infty}\subset \{\mu_{j}\}_{1}^{\infty},$ the operator  $B_{k}$ has the eigenvalues $\{\mu_{k_{j}}\}_{1}^{\infty}$ only.
\end{lem}
\begin{proof}

Consider the set of the eigenvalues $\{\mu_{k_{l}}\}_{1}^{\infty}=\{\mu_{j}\}_{1}^{\infty}\setminus \{\mu_{k_{j}}\}_{1}^{\infty},$  then  in accordance with Theorem 6.17 Chapter III \cite{firstab_lit:kato1980} (we can omit the compactness property),    we  have the decomposition
$$
\mathfrak{H}=  \mathfrak{M}_{i}'\oplus  \mathfrak{M}_{i}'',\,i\in \mathbb{N},
$$
corresponding to the finite set $\{\mu_{k_{l}}\}_{1}^{i},$
where $ \mathfrak{M}_{i}'$ is   a finite dimensional invariant subspace of the operator $B$ and  $\mathfrak{M}_{i}''$ is its parallel complement respectively, i.e.  $P_{i}\mathfrak{H}=\mathfrak{M}_{i}',\;(I-P_{i})\mathfrak{H}=\mathfrak{M}_{i}''.$
Consider a subspace $ \mathfrak{M}_{k}$ obtained due to the closure of the root vectors linear combinations  corresponding to $\{\mu_{k_{j}}\}_{1}^{\infty}.$ Note that the closure of the root vectors linear combinations corresponding to $\{\mu_{j}\}_{1}^{\infty}\setminus\{\mu_{k_{l}}\}_{1}^{i}$  belongs to $\mathfrak{M}_{i}'',$ since
 if
 $$
 g_{k}\rightarrow g,\,k\rightarrow\infty,\,\{g_{k}\}_{1}^{\infty}\subset \mathfrak{M}_{i}'',
 $$
 then
$$
g_{k}=(I-P_{i})g_{k}\rightarrow g,\Rightarrow P_{i}g=0,\Rightarrow g\in \mathfrak{M}_{i}'',
$$
the latter implications  can be established easily due to the properties of the parallel projector.  Thus, we conclude that the space $\mathfrak{M}_{i}''$ is closed. Therefore $\mathfrak{M}_{i}''\supset \mathfrak{M}_{k},\,i\in \mathbb{N}.$  Hence
\begin{equation}\label{4.34}
\mathfrak{M}_{k}\subset \bigcap\limits_{i=1}^{\infty}\, \mathfrak{M}_{i}''.
\end{equation}

Consider an arbitrary eigenvalue of the operator $B_{k},$ we denote it $\mu_{k_{j}},$ it corresponds to the root vector subspace $\mathfrak{M}_{k_{j}}.$  Thus we can  put in correspondence to the operator $B_{k}$ a sequence of subspaces, where the index $j$ is counted. Let us show that the closures of the subspaces corresponding to the operators $B_{k}$ and $B_{m},\,m\neq k$ do not not intersect. For this purpose, let us notice the fact that the root vectors are linearly independent. It follows easily if we consider a Riesz projector and use the formula \eqref{4.12}
$$
P_{\Gamma_{q}}e_{q_{\xi}+i}=-\frac{1}{2\pi i}\oint_{\Gamma_{q}}B(I-\lambda B)^{-1}e_{q_{\xi}+i}d\lambda=-\frac{1}{2\pi i}\oint_{\Gamma_{q}}\frac{1}{\lambda^{2}}\left\{(\lambda^{-1}I-B)^{-1}-I\ \right\}e_{q_{\xi}+i} d\lambda=
$$
$$
 =-\frac{1}{2\pi i}\oint_{\Gamma_{q}}\frac{1}{\lambda^{2}} (\lambda^{-1}I-B)^{-1} e_{q_{\xi}+i} d\lambda=\frac{1}{2\pi i}\oint_{\Gamma'_{q}}  (\lambda' I-B)^{-1} e_{q_{\xi}+i} d\lambda'=
$$
$$
 =\frac{1}{2\pi i} \sum\limits_{j=0}^{i}\oint_{\Gamma'_{q}}  \frac{e_{q_{\xi}+j}}{(\lambda'-\mu_{q})^{i-j+1}}  d\lambda'=e_{q_{\xi}+i},\,
$$
where $\Gamma_{q}$ is a closed contour bounding a domain containing the characteristic number $\lambda_{q}$ only, we have made the change of the variables $\lambda^{-1}=\lambda'.$   In accordance with the Cauchy integral formula, we get easily $P_{\Gamma_{q_{1}}}P_{\Gamma_{q_{2}}}=0,\,q_{1}\neq q_{2}.$ From what follows linear independence between the root vectors corresponding to different eigenvalues. Indeed, assume that the root vectors corresponding to an eigenvalue are linearly dependent, then there exists a non-zero set $C_{j},\,j=0,1,...n$ so that
$$
 \sum\limits_{j=0}^{n}C_{j}e_{q_{\xi}+j}=0,
$$
where $n$ is a height of the eigenvalue.
Applying the operators $(B-\mu_{q})^{k},\,k=n-1, n-2,...,1$ to the above sum consistently, we obtain $C_{j}=0,\,j=0,1,...n$ what contradicts to the   given assumptions.
Thus, we have proved that the root vectors are linearly independent.

In accordance with the   made assumptions   the root vector system corresponding to the set of the eigenvalues $\{\mu_{k_{j}}\}_{1}^{\infty} $ belongs to
the root vector system of the constructed operator  $B_{k},$    let us show that they are coincided.  Assume the contrary, then we should admit that there exists a number $N$ and an eigenvalue
$
\mu\in \{\mu_{k_{l}}\}_{1}^{N},
$
so that $(B-\mu I)^{\xi}e=0,\,e \in \mathfrak{M}_{k},\,\xi\in \mathbb{N}$
 then taking into account the fact $B_{k}\subset B$ we conclude that there exists such a number $N\in \mathbb{N}$ that we have $\mathfrak{M}_{k}\bigcap \mathfrak{M}_{n}' \neq0,\,n>N$ but it contradicts the inclusion \eqref{4.34} in accordance to which $\mathfrak{M}_{k}\subset\mathfrak{M}_{n}''$ since $\mathfrak{M}'_{n}\cap\mathfrak{M}_{n}''=0.$ Let us show that the space $\mathfrak{M}_{k}$ is the invariant subspace of the operator $B.$ It is clear that the operator $B$ preserves linear combinations of the root vectors. Thus, the question is weather this holds for the limits of the root vectors linear combinations.   To prove the fact, consider an element $g$ such that
$$
Bf_{k}\rightarrow g ,\, f_{k}\rightarrow f \in \,\mathfrak{M}_{k},\, f_{k}:=\sum\limits_{l=0}^{k}e_{l}c_{lk},
$$
where $c_{lk}$ are complex valued coefficients, $e_{l}$  root vectors corresponding to the set $\{\mu_{k_{j}}\}.$ Using the decomposition $Be_{q_{j}}=\mu_{q}e_{q_{j}}+e_{q_{j-1}},\,j=1,2,...,k,$ where $k+1$ is the algebraic multiplicity of the eigenvalue $\mu_{q},$ we get
$$
Bf_{k}=\sum\limits_{l=0}^{k'}e'_{l}c'_{lk'},
$$
where $k',e'_{l},c'_{lk'}$ are the transformed  terms  $k,e_{l},c_{lk}$ due to the  applied decomposition  formula. Therefore, taking into account the closedness of $\mathfrak{M}_{k}$  we come to the conclusion that $g\in \mathfrak{M}_{k}.$ Hence the operator $B_{k}$ has a set of the eigenvalues  $\{\mu_{k_{j}}\}_{ 1}^{\infty} .$
Note   that the restriction    $B_{k}$ is compact, since $B$ is compact, and   in accordance with the Hilbert theorem its spectrum consists of normal  values. Thus, if a complex value is not an eigenvalue then it is a regular value. Since we have shown previously that   the   operator $B_{k}$ has the eigenvalues $\{\mu_{k_{j}}\}_{1}^{\infty}$ then the resolvent set  of $B_{k}$ contains  the rest part of eigenvalues of the operator $B.$

\end{proof}

\begin{teo} \label{T4.8}  Assume that $B$ is a compact operator, $B\in \mathfrak{S}_{\sigma}\,,\,0<\sigma<1,\,\Theta(B) \subset   \mathfrak{L}_{0}(\theta) ,\,\theta< \pi/4,\,\rho=\inf\sigma$ then
  the order of summation $\alpha>\rho$ can be essentially decreased to $ s>\rho(1-1/\gamma),\,\gamma=2,3,...\,.$

   Moreover, in each case corresponding to $\gamma$ the subsequence of the natural numbers can be  chosen  in the following way
$$
N_{\nu}=\beta(\nu+1)^{\gamma},\,\nu\in \mathbb{N}_{0},\; \gamma=\beta+ \eta ,\; \eta\in \mathbb{N},\;\beta/\eta \in \mathbb{N},
$$
 so that  there exists  the splitting  of the operators
 $$
P_{\nu}(s,t)=\sum\limits_{k=k_{1}(\nu)}^{k_{2}(\nu)}P^{(k)}(s,t),\,  k_{1}, k_{2} \in \mathbb{N}_{0},
$$
where
$$
P^{(k)}(s,t):=\sum\limits_{q\in \mathfrak{A}_{k}}  P_{q}(s,t),\,\mathfrak{A}_{k}\subset \mathbb{N},
$$
  the following  series is     convergent in the Abel-Lidskii sense, i.e.
$$
 \sum\limits_{k=0}^{\infty}  P^{(k)}(s,t)f \rightarrow f,\,t\rightarrow 0,\;f\in \mathfrak{H}.
$$

\end{teo}
\begin{proof} Consider the sequence of the   eigenvalues $\{\mu_{j}\}_{1}^{\infty}$ of the operator $B,$ it corresponds to the sequence of the characteristic numbers $\{\lambda_{j}\}_{1}^{\infty},\,\lambda_{j}=1/\mu_{j}.$  Note that in accordance with Lemma \ref{L4.15}, using decomposition
 $$
 N_{\nu}= \sum_{k=0}^{\nu^{ \eta}} N_{k\nu},
 $$
 we can rearrange the sequence of the eigenvalues  so that we obtain the subsequences  $\{\lambda_{k_{j}}\}_{1}^{\infty}$  having counting functions $n_{k}(r),\,k\in \mathbb{N}$ defined in Corollary \ref{C4.1}. In accordance with  Lemma \ref{L4.17}, they can be put in  correspondence  to the operators $B_{k}$ having properties described  in Lemma  \ref{L4.17}, i.e.  the operators $B_{k}$  are restrictions of the operator $B$ on the root vector subspaces corresponding to the eigenvalues values  $\{\mu_{k_{j}}\}_{1}^{\infty}$ which inverse values  $\{\lambda_{k_{j}}\}_{1}^{\infty}$ are counted by the counting functions $n_{k}(r)$ respectively. The smallest number $\nu\geq k^{1/\eta}$   indicates a position of the first group of the characteristic numbers  $\{\lambda_{k_{j}}\}_{1}^{\infty}$ corresponding to   $B_{k}$ belonging to the contour with the number $\nu.$ Here, we should recall that the condition holds
 $$
 |\lambda_{N_{\nu}+1}|-|\lambda_{N_{\nu}}|\geq K |\lambda_{N_{\nu}}|^{1-\sigma},
 $$
what gives us an opportunity to consider lacunaes between the  values $\lambda_{N_{\nu}},\lambda_{N_{\nu}+1}$ and define the closed contours $\vartheta_{\nu}$ containing the values $\lambda_{N_{\nu}+1},\lambda_{N_{\nu}+2},...,\lambda_{N_{\nu+1} }$ only, taking into account the fact that $N_{0}=\beta,$ we can artificially assume that $N_{-1}=0,$ then     $\vartheta_{-1}$ contains $\lambda_{1},\lambda_{2},...,\lambda_{\beta}=\lambda_{N_{0}}$ and does not contain the point zero.
 Consider a contour
$$
 \vartheta(B):= \left\{\lambda:\;|\lambda|=r_{0}>0,\,|\mathrm{arg} \lambda|\leq \theta+\varepsilon\right\}\cup\left\{\lambda:\;|\lambda|>r_{0},\; |\mathrm{arg} \lambda|=\theta+\varepsilon\right\},
 $$
 here the number $r$ is chosen so that the  circle  with the corresponding radios does not contain eigenvalues $\mu_{n}:=\mu_{n}(B),\,n\in \mathbb{N}.$
Consider the obtained by Lidskii representation that holds under the theorem assumptions
$$
\frac{1}{2\pi i}\oint\limits_{\vartheta }  e^{- \lambda^{\alpha}  t}B \left(I-\lambda B \right)^{-1}f d\lambda=-\sum\limits_{\nu=-1}^{\infty}
\sum\limits_{q=N_{\nu}+1}^{N_{\nu+1}} P_{q}(\alpha,t)f,\;f\in   \mathfrak{H}.
$$

The     obtained one to one correspondence between the set of the  root vector subspaces of the operator $B$ and the sequence of the operators $B_{k}$ in Lemmas \ref{L4.15}, \ref{L4.17}   gives us the essential decreasing of the summation  order since we can split operators $P_{\nu}$ and summarize  them in the different way.  The main  idea of the proof is to rearrange the brackets in the above series so that the newly obtained series will be convergent under the smaller value of the parameter $\alpha.$

Using the decomposition formula obtained in Lemma \ref{L4.15}
$$
N_{\nu}=\sum_{k=0}^{\nu^{\eta}}N_{k\nu  }
$$
consider the following representation
$$
\sum\limits_{q=N_{\nu}+1}^{N_{\nu+1}} P_{q}(s,t)f=\sum\limits_{k=0}^{\nu^{\eta}}\sum\limits_{\tilde{q}=N_{k\nu}+1}^{N_{k\nu+1}} P_{ q }(s,t)f,
$$
where we have a  correspondence between $\tilde{q}$ and $q$ defined  in the following way
$
q=\psi_{k}( \tilde{q}).
$
It is clear that this correspondence can be established since for the fixed number $k$ we have a natural  correspondence (inclusion) between the set $\{\lambda_{k_{j}}\}_{1}^{\infty}$ of the   values of the operator $B_{k}$  end the set $\{\lambda_{j}\}_{1}^{\infty}$ of the   values of the operator $B,$ i.e, $\{\lambda_{k_{j}}\}_{1}^{\infty}\subset\{\lambda_{j}\}_{1}^{\infty}.$

Therefore, we obtain the first part of the theorem claim  if we prove that the series
$$
\sum\limits_{\nu=0}^{\infty}\sum\limits_{k=0}^{\nu^{\eta}}\sum\limits_{\tilde{q}=N_{k\nu}+1}^{N_{k\nu+1}} P_{ q }(s,t)f
$$
is convergent. On the other hand the latter fact will be established if we prove the convergence of the following series
$$
\sum\limits_{k=0}^{\infty}\sum\limits_{\nu\geq k^{1/\eta}} \left\|\sum\limits_{\tilde{q}=N_{k\nu}+1}^{N_{k\nu+1}} P_{ q }(s,t)f\right\|,
$$
since in this case, we have
$$
\sum\limits_{\nu=0}^{\infty}
\left\|\sum\limits_{q=N_{\nu}+1}^{N_{\nu+1}} P_{q}(\alpha,t)f\right\|\leq\sum\limits_{\nu=0}^{\infty}\sum\limits_{k=0}^{\nu^{\eta}}\left\|\sum\limits_{\tilde{q}=N_{k\nu}+1}^{N_{k\nu+1}} P_{ q}(s,t)f\right\|=\sum\limits_{k=0}^{\infty}\sum\limits_{\nu\geq k^{1/\eta}} \left\|\sum\limits_{\tilde{q}=N_{k\nu}+1}^{N_{k\nu+1}} P_{q}(s,t)f\right\|<\infty.
$$
The latter relation gives us the second part of the theorem claim since we can rearrange the elements of the absolutely convergent series in an arbitrary way.

Now let us begin the proof.
In accordance with the Lemmas \ref{L4.3},\ref{L4.15},\ref{L4.17},  the last series can be rewritten in the form
\begin{equation}\label{4.35}
\sum\limits_{k=0}^{\infty}\sum\limits_{\nu\geq k^{1/\eta}} \left\|\sum\limits_{\tilde{q}=N_{k\nu}+1}^{N_{k\nu+1}} P_{ q }(s,t)f\right\|=\sum_{k=0}^{\infty} \sum_{\nu\geq k^{1/\eta} } \left\| \,\oint\limits_{\vartheta_{\nu}}  e^{- \lambda^{s}  t}B_{k}\left(I-\lambda B_{k}\right)^{-1}f d\lambda \right\|,
\end{equation}
where $s>\rho(1-1/\gamma)$ we recall that $\rho$ is the convergence exponent of the operator $B,$ the contours $\vartheta_{\nu}$ are defined above.

To proceed the next step, we need some theoretical basement appealing to the notion of the Fredholm determinant (see \cite{firstab_lit:1Lidskii}).
 Now, consider a formal representation for the resolvent
$$
(I-\lambda B)^{-1}f =\sum\limits_{m=1}^{\infty}\left(\sum\limits_{l=1}^{\infty}(-1)^{l+m}\frac{\Delta^{lm}(\lambda)}{\Delta}f_{l}\right)e_{m},
$$
where $\{e_{n}\}_{1}^{\infty}$ is an arbitrary orthonormal basis in $\mathfrak{H},$
$$
\Delta^{lm}(\lambda)= \sum\limits_{p=0}^{\infty}(-1)^{p}\lambda^{p}\!\!\!\!\!\sum\limits_{i_{1},i_{2},...,i_{p} =1}^{\infty}\!\!\!B\begin{pmatrix} i_{1}&  i_{2}&...&i_{p}\\
  i_{1}&  i_{2}&...&i_{p}
\end{pmatrix}_{lm},
$$
the  used formula in brackets means a minor formed from the columns and rows with $i_{1},i_{2},...,i_{p}$ numbers except for $l$ -th row and $m$ -th column. Using analogous form of writing, we denote the Fredholm determinant of the operator $B$ as follows
$$
\Delta (\lambda)=1- \lambda  \sum\limits_{i=1}^{\infty}B\begin{pmatrix} i\\
  i
\end{pmatrix} +\lambda^{2}  \sum\limits_{i_{1},i_{2} =1}^{\infty}B\begin{pmatrix} i_{1}&  i_{2}\\
  i_{1}&  i_{2}
\end{pmatrix} +...+(-1)^{p}\lambda^{p}\sum\limits_{i_{1},i_{2},...,i_{p} =1}^{\infty}B\begin{pmatrix} i_{1}&  i_{2}&...&i_{p}\\
  i_{1}&  i_{2}&...&i_{p}
\end{pmatrix} +...\,,
$$
where the  used formula in brackets means a minor formed from the columns and rows with $i_{1},i_{2},...,i_{p}$ numbers. Note that if $B$ belongs to the trace class then in accordance with the well-known theorems (see \cite{firstab_lit:1Gohberg1965}), we have
\begin{equation}\label{4.36}
\sum\limits_{n  =1}^{\infty}|b_{nn}| <\infty,\;\sum\limits_{n,k  =1}^{\infty}|b_{nm}|^{2} <\infty,
\end{equation}
where $ b_{nm} $ is the matrix coefficients of the operator $B.$ This follows easily from the properties of the trace class operators and Hilbert-Schmidtt class operators respectively. In accordance with the von Koch theorem the conditions \eqref{4.36} guaranty the absolute convergence of the series
$$
\sum\limits_{i_{1},i_{2},...,i_{p} =1}^{\infty}B\begin{pmatrix} i_{1}&  i_{2}&...&i_{p}\\
  i_{1}&  i_{2}&...&i_{p}
\end{pmatrix}.
$$
Moreover, the formal series $\Delta(\lambda)$ is convergent for arbitrary $\lambda\in \mathbb{C},$  hence it represents an entire function.
Note that in accordance with   Lemma 1 \cite{firstab_lit:Eb. Levin}, we have
$$
s_{n}(B_{k})\leq s_{n}( B ),\,k=0,1,...\,,
$$
since the operator $B_{k}$ admits the representation  $B=P_{k}BP_{k},$ where $P_{k}$ is orthogonal projector into the corresponding invariant subspace $\mathfrak{M}_{k}$ (constructed in Lemma \ref{L4.17}) of the operator $B_{k}.$ Thus, we have the implication
$$
B\in \mathfrak{S}_{\sigma},\Rightarrow B_{k}\in \mathfrak{S}_{\sigma}.
$$
Having denoted by $\Delta_{k}(\lambda),\,\Delta^{lm}_{k}(\lambda)$ the considered above constructions corresponding to the operator $B_{k},$ we come due to the reasonings represented in Lemma 1 \cite{firstab_lit:Eb. Levin} to the formula
$$
|\Delta_{k}(\lambda)|\leq \prod\limits_{n=1}^{\infty}\left\{1+|\lambda|s_{n}(B_{k})\right\}.
$$
It implies that the entire function $\Delta_{k}(\lambda)$ is of the finite order, hence in accordance with the Theorem 13 (Chapter I, \S 10) \cite{firstab_lit:Eb. Levin}, it has a representation by the canonical product, we have
$$
\Delta_{k}(\lambda)= \prod\limits_{n=1}^{\infty}\left\{1- \lambda \mu_{n}(B_{k})    \right\}.
$$
Now, let us chose an arbitrary element $f\in \mathfrak{M}_{k}$  and construct a new orthonormal basis having put $f$ as a first basis element. Note that the relations \eqref{4.36} hold for the matrix coefficients of the operator $B_{k}$ in a new basis, this fact follows from the well-known theorem for the trace class operator.  Thus, using the given  above representation for the resolvent, we obtain the following relation
$$
\Delta_{k}(\lambda)\left((I-\lambda B_{k})^{-1}f,f\right)_{\mathfrak{H}}=  \Delta^{11}_{k}(\lambda,f).
$$
Let us observe the latter entire function more properly, we have
$$
\Delta^{11}_{k}(\lambda,f)= \sum\limits_{p=0}^{\infty}(-1)^{p}\lambda^{p}\!\!\!\!\!\sum\limits_{i_{1},i_{2},...,i_{p} =1}^{\infty}\!\!\!B\begin{pmatrix} i_{1}&  i_{2}&...&i_{p}\\
  i_{1}&  i_{2}&...&i_{p}
\end{pmatrix}_{11},
$$
where
its construction reveals the fact that it is a Fredholm determinant of the operator $Q_{1}B_{k}Q_{1},$ where $Q_{1}$  is the projector into  orthogonal  complement of the element $f.$    Having applied the above reasonings (Lemma 1 \cite{firstab_lit:Eb. Levin}), we obtain
$$
s_{n}(Q_{1}B_{k}Q_{1})\leq s_{n}( B_{k} ),
$$
In the same way, we obtain that the entire function $\Delta^{11}_{k}(\lambda,f)$  is of the finite order and obtain the representation
$$
\Delta^{11}_{k}(\lambda,f)=\prod\limits_{n=1}^{\infty}\left\{1- \lambda \mu_{n}(Q_{1}B_{k}Q_{1})    \right\}.
$$
In accordance with Lemma \ref{L4.16}, we have
$$
\prod\limits_{n=1}^{\infty}|1+\lambda\mu_{n}(Q_{1}B_{k}Q_{1})| \leq \prod\limits_{n=1}^{\infty}|1+\lambda\mu_{n}( B_{k} )|,
$$
what gives us the desired result. Taking into account the fact
$$
\left|1-  \lambda_{1}\lambda_{2}\right|\leq \left|1+ \lambda_{1}\lambda_{2}\right|,\, \mathrm{arg}\lambda_{1}, \mathrm{arg}\lambda_{2}<\pi/4,
$$
Evaluating, we obtain
$$
|\Delta^{11}_{k}(\lambda,f)|\leq\prod\limits_{n=1}^{\infty}\left|1-  \lambda \mu_{n}(Q_{1}B_{k}Q_{1})\right|  \leq \prod\limits_{n=1}^{\infty}\left\{1+ |\lambda \mu_{n}( B_{k} )| \right\}.
$$
Since the right-hand side does not depend on $f,$ then using decomposition on the Hermitian components, we have  the following relation
\begin{equation}\label{4.37}
|\Delta_{k}(\lambda)|\cdot\|(I-\lambda B_{k})^{-1}\|\leq 2 \prod\limits_{n=1}^{\infty}\left\{1+ |\lambda \mu_{n}( B_{k} )| \right\}.
\end{equation}
To establish the letter relation Agaranovich M.S. used the polarization formula, however we can prove it using the offered method. Define the operator function
$$
D_{B_{k}}(\lambda)=\Delta_{k}(\lambda) (I-\lambda B_{k})^{-1},
$$
then $(D_{B_{k}} (\lambda)f,f )=\Delta^{11}_{k}(\lambda,f).$
Having involved ordinary   properties of selfadjoint operators, we get
$$
\sup \limits_{\|f\|\leq 1 }\|D_{B_{k}}(\lambda)f\|=\sup \limits_{\|f\|\leq 1 }\|\mathfrak{Re} D_{B_{k}}(\lambda)f+i\, \mathfrak{Im} D_{B_{k}}(\lambda)f \|\leq
$$
$$
\leq\sup \limits_{\|f\|\leq 1 }\|\mathfrak{Re} D_{B_{k}}(\lambda)f  \|+ \sup \limits_{\|f\|\leq 1 }\|  \mathfrak{Im} D_{B_{k}}(\lambda)f \|=
$$
$$
=\sup \limits_{\|f\|\leq 1 }\left| \mathrm{Re}(D_{B_{k}} (\lambda)f,f )\right|+\sup \limits_{\|f\|\leq 1 }\left| \mathrm{Im}(D_{B_{k}} (\lambda)f,f )\right|\leq 2 \prod\limits_{n=1}^{\infty}\left\{1+ |\lambda \mu_{n}( B_{k} )| \right\},
$$
thus, we obtain \eqref{4.37}. In accordance with the above   we have   $n_{k}(r)=o( r^{s}),\,s>\rho(1-1/\gamma),$ hence
$$
\sum\limits_{n=1}^{\infty} |\mu_{n}( B_{k} )|^{s} <\infty,
$$
therefore, using   estimate \eqref{4.4} for the canonical product, we obtain
$$
\|D_{B_{k}}(\lambda)\|\leq 2\prod\limits_{n=1}^{\infty}\left\{1+ |\lambda \mu_{n}( B_{k} )| \right\}\leq  e^{\beta_{k}(|\lambda|)|\lambda|^{s}},
$$
where
$$
\;\beta_{k}(r )= r^{ -s }\left(\;\int\limits_{0}^{r}\frac{n_{k}(t)dt}{t }+
r \int\limits_{r}^{\infty}\frac{n_{k}(t)dt}{t^{ 2  }}\right),\;k=0,1,... \,.
$$
   Taking into account Lemma \ref{L4.2}, we have
$
\beta_{k}(r )\rightarrow 0,\,r\rightarrow\infty.
$
Consider the entire function
$$
\Delta_{k}(\lambda)= \prod\limits_{n=1}^{\infty}\left\{1- \lambda \mu_{n}(B_{k})    \right\}.
$$
In accordance with the  Joseph Cartan concept, we can obtain the following estimate from the bellow for the entire function that  holds on the complex plane except may be the exceptional set of circulus. The latter cannot be found but and we are compelled to make an occlusion evaluating the measures. However, in the paper \cite{firstab_lit(frac2023)},  we   produce the method allowing to find exceptional set of   circulus.

We need an auxiliary construction, let us   find $\delta_{\nu},\,\nu\in \mathbb{N}_{0}$ from the condition $R_{\nu}=K|\lambda_{N_{\nu}}|^{1-\sigma}+|\lambda_{N_{\nu}}|,\,R_{\nu}(1-\delta_{\nu})=|\lambda_{N_{\nu}}|,$ then $\delta_{\nu}^{-1}=1+K^{-1}|\lambda_{N_{\nu}}|^{\sigma}.$ Note that by virtue of such a choice, we have $R_{\nu}<R_{\nu+1}(1-\delta_{\nu+1}).$ Now, applying Theorem 4 \cite[p.79]{firstab_lit:Eb. LevinE}, we have
$$
|\Delta_{k}(\lambda)|\geq e^{-(2+\ln\{4e/\delta_{\nu}\}) \beta_{k}(\varpi |\lambda|   )\varpi ^{1/s}} ,\,\varpi:=\frac{2e}{1-\delta_{0}},\,|\lambda|=\tilde{R}_{\nu},
$$
where $|\lambda_{N_{\nu}}|=(1-\delta_{\nu})R_{\nu}<\tilde{R}_{\nu}<R_{\nu}.$ Here the arch $|\lambda|=\tilde{R}_{\nu},\,|\arg \lambda|<\theta$ is chosen so that it belongs to the exceptional  set of circles.
 Therefore, in accordance with the obtained above relations, we get
 \begin{equation}\label{4.38}
 \|(I-\lambda B_{k})^{-1}\|= \frac{\|D_{B_{k}}(\lambda)\|}{|\Delta_{k}(\lambda)|}\leq e^{\gamma_{k}(|\lambda|)|\lambda|^{s}}|,\,|\lambda|=\tilde{R}_{\nu},
 \end{equation}
where
$$
\gamma_{k}(|\lambda|)= \beta_{k} ( |\lambda| )  +(2+\ln\{4e/\delta_{\nu}\}) \beta_{k}(\varpi |\lambda|   ) \,\varpi ^{1/s}.
$$
Recalling the formula for $\delta_{\nu},$ we have
$$
\ln\{4e/\delta_{\nu}\} = 1+ \ln4 +\ln \{1+K^{-1}|\lambda_{N_{\nu}}|^{\sigma}\} \leq C \ln   |\lambda_{N_{\nu}}| \leq C \ln   |\lambda_{N_{\nu}}| \leq C\ln \tilde{R}_{\nu}.
$$
Thus,
we get easily $\gamma_{k}( \tilde{R}_{\nu} )\rightarrow 0,\,\nu\rightarrow \infty,$ uniformly with respect to $k.$ Indeed,
to obtain a uniform estimate,  we can evaluate  the counting function and in this way evaluate the family $\beta_{k}(r ),\,k\in \mathbb{N}_{0}.$  For this purpose,
note that
$$
n_{k}(\lambda_{N_{\nu}+m})\leq n_{k}(\lambda_{N_{\nu+1}}),\;m=0,1,...,N_{\nu+1}-N_{\nu},\; k\in \mathbb{N}_{0};
$$
$$
 n_{k}(\lambda_{N_{\nu+1}})=(\nu+1)^{\beta}-k^{\beta/\eta} ,\,k>0.
$$
Hence, having noticed that $\beta\leq \gamma-1,$ we conclude that we  can  estimate   the family of counting   functions by a newly constructed step function  $\tilde{n}(r)=\gamma^{2}(\nu+1)^{\gamma-1},\,r= |\lambda_{N_{\nu}}|,$  we have $n_{k}(r)\leq \tilde{n}(r),\,k\in \mathbb{N}_{0}.$ Using the implication
$$
\lim\limits_{r\rightarrow\infty}\frac{n(r)}{r^{\sigma}}=0,\Rightarrow\lim\limits_{\nu\rightarrow\infty}\frac{\nu^{\gamma}  }{|\lambda^{\sigma }_{N_{\nu} }|}=0,\Rightarrow
\lim\limits_{\nu\rightarrow\infty}\frac{(\nu+1)^{\gamma-1}  }{|\lambda^{\sigma (1-1/\gamma) }_{N_{\nu} }|}=0,
$$
 we get
$$
\lim\limits_{r\rightarrow\infty}\frac{\tilde{n}(r)}{r^{s}}=0.
$$
It gives us the following relation if we apply the scheme of reasonings absolutely analogous to the one in accordance to which we obtained \eqref{4.5}, thus we get
 $$
\tilde{\beta}(r):=r^{ -s }\left(\int\limits_{0}^{r}\frac{\tilde{n} (t)dt}{t }+
r \int\limits_{r}^{\infty}\frac{\tilde{n} (t)dt}{t^{ 2  }}\right)\rightarrow 0,\,r\rightarrow\infty.
$$
Having involved inequality \eqref{4.38}, we get
$$
 \|(I-\lambda B_{k})^{-1}\| \leq e^{\tilde{\gamma} (|\lambda|)|\lambda|^{s}}|,\,|\lambda|=\tilde{R}_{\nu},\,\tilde{\gamma}(|\lambda|)= \tilde{\beta} ( |\lambda| )  +(2+\ln\{4e/\delta_{\nu}\}) \tilde{\beta}(\varpi |\lambda|   ) \,\varpi ^{1/s}.
$$
Let us estimate the inner sum \eqref{4.35}, for this purpose we want to  estimate termwise the following relation representing unified  reasonings for all values of $k\in  \mathbb{N}_{0}$
$$
 \left\| \,\int\limits_{\vartheta_{\nu}}  e^{- \lambda^{s}  t}B_{k}\left(I-\lambda B_{k}\right)^{-1}f d\lambda \right\|\leq J_{\nu}+J_{\nu+1}+J^{+}_{\nu}+J^{-}_{\nu}.
$$
where
$$
J^{+}_{\nu}: =\left\|\,\int\limits_{\vartheta_{\nu_{+}}} e^{- \lambda^{s}  t}B_{k}\left(I-\lambda B_{k}\right)^{-1}f d\lambda\,\right\|,\; J^{-}_{\nu}:= \left\|\,\int\limits_{\vartheta_{\nu_{-}}} e^{- \lambda^{s}  t}B_{k}\left(I-\lambda B_{k}\right)^{-1}f d\lambda\,\right\|,
$$
$$
\;J_{  \nu  }: =\left\|\,\int\limits_{\tilde{\vartheta}_{\nu}} e^{- \lambda^{s}  t}B_{k}\left(I-\lambda B_{k}\right)^{-1}f d\lambda\,\right\|,
$$
$\vartheta_{\nu_{+}}:=
\{\lambda:\, \tilde{R}_{\nu}<|\lambda|<\tilde{R}_{\nu+1},\, \mathrm{arg} \lambda  =\theta   +\varepsilon\},\,\vartheta_{\nu_{-}}:=
\{\lambda:\,\tilde{R}_{\nu}<|\lambda|<\tilde{R}_{\nu+1},\, \mathrm{arg} \lambda  =-\theta   -\varepsilon\},\,\tilde{\vartheta}_{ \nu }:=\{\lambda:\;|\lambda|=\tilde{R}_{\nu},\,|\mathrm{arg } \lambda|< \theta +\varepsilon\}.$
We have
$$
 J_{  \nu  } =\left\|\,\int\limits_{\tilde{\vartheta}_{\nu}}e^{-\lambda^{s}t}B_{k}\left(I-\lambda B_{k}\right)f d \lambda\,\right\| \leq \,\int\limits_{\tilde{\vartheta}_{ \nu }}e^{- t \mathrm{Re}\lambda^{s}}\left\|B_{k}\left(I-\lambda B_{k}\right)f \right\|  |d \lambda|\leq
$$
 $$
 \leq C\|f\|  e^{\tilde{\gamma}  (|\lambda|)|\lambda|^{s} } \int\limits_{-\theta-\varepsilon}^{\theta+\varepsilon} e^{- t \mathrm{Re}\lambda^{s}} d \,\mathrm{arg} \lambda,\,|\lambda|=\tilde{R}_{\nu}.
$$
Note that the condition
$
\theta<\pi/2\sigma < \pi/2s
$
gives us     $\,|\mathrm{arg} \lambda |<\pi/2s,\,\lambda\in \tilde{\vartheta}_{ \nu },\,\nu=0,1,2,...\,,$
$$
\mathrm{Re }\lambda^{s}\geq |\lambda|^{s} \cos  (\theta+\varepsilon) s \geq |\lambda|^{s} \cos \left[(\pi/2s-\delta)s\right]= |\lambda|^{s} \sin  s \delta,\,\lambda \in \tilde{\vartheta}_{ \nu },
$$
for a sufficiently small $\delta$ and $\varepsilon.$
Therefore, we get
$$
J_{\nu}\leq C\|f\|  e^{\tilde{R}_{\nu}^{s}\{\tilde{\gamma}  (\tilde{R}_{\nu})-t\sin     s \delta\} }.
$$
 To estimate other terms, we are rather satisfied with the estimate represented in Lemma 4 (Lidskii) \cite{firstab_lit:1Lidskii}
$$
\|(I-\lambda B_{k})^{-1}\|\leq \frac{1}{\sin\psi},\,\lambda\in \zeta,
$$
where   $\zeta$ is the ray  containing the point zero and not belonging to the sector $\mathfrak{L}_{0}(\theta),\,\psi = \min \{|\mathrm{arg}\zeta -\theta|,|\mathrm{arg}\zeta +\theta|\}.$ Absolutely analogously  to the reasonings represented in Lemma 7 (Lidskii) \cite{firstab_lit:1Lidskii}, we get
$$
 J^{+}_{\nu} \leq C\|f\|  \!\!\! \int\limits_{ \tilde{R}_{\nu} }^{\tilde{R}_{\nu+1} }  e^{- t \mathrm{Re }\lambda^{s}}   |d   \lambda|\leq C \|f\|   \!\!\!\int\limits_{ \tilde{R}_{\nu} }^{\tilde{R}_{\nu+1} } e^{-t |\lambda| ^{s} \sin     s\delta}   |d   \lambda|.
 $$
 Using integration by parts formulae, we can easily obtain the following estimate
$$
 J^{+}_{\nu} \leq C \|f\| \frac{ e^{-\tau \tilde{R}_{\nu}^{s}}\tilde{R}_{\nu}  +e^{-\tau \tilde{R}_{\nu+1}^{s} }\tilde{R}_{\nu+1}  }{\tau^{s+1}},
$$
where $\tau:=t\sin     s \delta.$
Thus, we have come to the relation
$$
 I_{ k,\nu}:=\left\| \,\int\limits_{\vartheta_{\nu}}  e^{- \lambda^{s}  t}B_{k}\left(I-\lambda B_{k}\right)^{-1}f d\lambda \right\|\leq C \|f\|\left\{e^{\tilde{R}_{\nu}^{s}\{\tilde{\gamma}  (\tilde{R}_{\nu})-\tau\} }+\tau ^{-s-1}e^{-\tau \tilde{R}_{\nu}^{s}}\tilde{R}_{\nu}\right\} .
$$
Hence,   we get
$$
\sum_{k=0}^{\infty} \sum_{\nu\geq k^{1/\eta} } \left\| \,\oint\limits_{\vartheta_{\nu}}  e^{- \lambda^{s}  t}B_{k}\left(I-\lambda B_{k}\right)^{-1}f d\lambda \right\|\leq
$$
$$
\leq C\left\{\sum_{k=0}^{\infty}\sum_{\nu\geq k^{1/\eta} } e^{\{\tilde{\gamma}  (\tilde{R}_{\nu})-\tau\}\tilde{R}_{\nu}^{s} }+\tau ^{-s-1}\sum_{k=0}^{\infty}\sum_{\nu\geq k^{1/\eta} }
e^{-\tau \tilde{R}_{\nu}^{s}}\tilde{R}_{\nu}\right\}.
$$
  Consider the following decomposition
$$
 \sum_{k=0}^{\infty} \sum_{\nu\geq k^{1/\eta} } e^{\{\tilde{\gamma}  (\tilde{R}_{\nu})-\tau\}\tilde{R}_{\nu}^{s} }=
 \sum_{k=0}^{m} \sum_{\nu\geq k^{1/\eta}}  e^{\{\tilde{\gamma}  (\tilde{R}_{\nu})-\tau\}\tilde{R}_{\nu}^{s} }+ \sum_{k=m+1} ^{\infty}\sum_{\nu\geq k^{1/\eta} } e^{\{\tilde{\gamma}  (\tilde{R}_{\nu})-\tau\}\tilde{R}_{\nu}^{s} }=I_{1}+I_{2},\;m=2,3,...\,.
$$
It is clear that
$$
I_{1}\leq m \sum_{\nu=1 }^{\infty}e^{\{\tilde{\gamma}  (\tilde{R}_{\nu})-\tau\}\tilde{R}_{\nu}^{s} }<\infty,
$$
the latter series is convergent since $\tilde{\gamma}  (\tilde{R}_{\nu})\rightarrow 0,\,\nu\rightarrow\infty.$  We can establish  easily that   for arbitrary $\tau>0$ and arbitrary small $\varepsilon>0$ there exists $m$ such that
$$
I_{2}\leq  \sum_{k= m+1}^{\infty} \sum_{\nu\geq k^{1/\eta} }e^{   -M\tilde{R}_{\nu}^{s} },\, \tau-\varepsilon<M<\tau.
$$
Taking in account the facts $\tilde{R}_{\nu}>|\lambda_{N_{\nu}}|> C N^{1/\sigma}_{\nu}>C \nu^{\gamma/\sigma},\,s>\sigma(1-1/\gamma)$ (more detailed see Corollary \ref{C4.1}),   we get
\begin{equation}\label{4.39}
I_{2}\leq \sum_{k= m+1}^{\infty} \sum_{\nu\geq k^{1/\eta} }e^{    -M|\lambda_{N_{\nu}}|^{s} }\leq \sum_{k= m+1}^{\infty} \sum_{\nu\geq k^{1/\eta} }e^{    -M \nu^{\gamma s /\sigma} }
\leq\sum_{k= m+1}^{\infty} \sum_{\nu\geq k^{1/\eta} }e^{     -M \nu^{ \gamma-1} }\leq \sum_{k= m+1}^{\infty} \sum_{\nu\geq k^{1/\eta} }e^{     -M \nu  }.
\end{equation}
The latter relation holds since $\gamma\geq2.$ To estimate it we need the following detailed observation. It is clear that in order to estimate the number of solutions of the equation $[k^{1/\eta}]=C$ with respect to $k\in \mathbb{N},$ we should estimate the difference $m_{2}-m_{1}.$ Here we denote    by $m_{1},m_{2}\in \mathbb{N}$ the numbers corresponding to $k$ so that $m^{1/\eta}_{1} \in \mathbb{N},\, m_{2}:=(m^{1/\eta}_{1}+1)^{\eta},\;m^{1/\eta}_{1}\leq k^{1/\eta}\leq m^{1/\eta}_{2}.$  Using the fact $m_{1}=k_{1}^{\eta},\,k_{1}\in \mathbb{N},$ we get
$$
m_{2}-m_{1}=(k_{1}+1)^{ \eta}- k_{1} ^{ \eta}\leq  \eta  k^{ \eta-1}_{1}\left(1+1/k_{1}\right)^{ \eta-1}\leq Ck^{ \eta-1}_{1}=Cm_{1}^{1-1/\eta}.
$$
Therefore
\begin{equation}\label{4.40}
\sum_{k= m+1}^{\infty} \sum_{\nu\geq k^{1/\eta} }e^{     -M \nu  }\leq C \sum_{k\in \mathfrak{P}} k^{1-1/\eta}\sum_{\nu= k^{1/\eta} }^{\infty}e^{     -M \nu  },\,\mathfrak{P}:=\{k: k=j^{ \eta},\,j\in \mathbb{N} \}.
\end{equation}
Using simple formulas for geometrical progression, we get
$$
I_{2}\leq C  \sum_{k=1}^{\infty}  k^{1-1/\eta}
\sum_{\nu= k }^{\infty}e^{     -M \nu  }
\leq \frac{e^{M}}{e^{M}-1}\sum_{k=1}^{\infty} k^{1-1/\eta} e^{     -M k   } <\infty.
$$
The fact that the following  series  is convergent is obvious due to the obtained above  estimate from the bellow for   $\tilde{R}_{\nu},$   analogously to \eqref{4.39}, \eqref{4.40}, we get  $$
\sum_{k=0}^{\infty}\sum_{\nu\geq k^{1/\eta} }
e^{-\tau \tilde{R}_{\nu}^{s}}\tilde{R}_{\nu}\leq C \sum_{k=0}^{\infty}\sum_{\nu\geq k^{1/\eta} }
e^{-M\tilde{R}_{\nu}^{s}}\leq
$$
$$
 \leq C \sum_{k=0}^{\infty}\sum_{\nu\geq k^{1/\eta} }
e^{-M|\lambda_{N_{\nu}}|^{s}} \leq C \sum_{k=0}^{\infty}\sum_{\nu\geq k^{1/\eta} }
e^{-M\nu}  \leq C \sum_{k\in \mathfrak{P}} k^{1-1/\eta}\sum_{\nu= k^{1/\eta} }^{\infty}e^{     -M \nu  }\leq
$$
$$
\leq C  \sum_{k=1}^{\infty}  k^{1-1/\eta}
\sum_{\nu= k }^{\infty}e^{     -M \nu  }
\leq \frac{e^{M}}{e^{M}-1}\sum_{k=1}^{\infty} k^{1-1/\eta} e^{     -M k   } <\infty.
$$
 Thus, we have proved that the series \eqref{4.35} is convergent. Therefore

\begin{equation}\label{4.41}
-\frac{1}{2\pi i}\oint\limits_{\vartheta }  e^{- \lambda^{s}  t}B \left(I-\lambda B \right)^{-1}f d\lambda=\sum\limits_{\nu=0}^{\infty}\sum\limits_{k=0}^{\nu^{\eta}}\sum\limits_{\tilde{q}=N_{k\nu}+1}^{N_{k\nu+1}} P_{ q }(s,t)f= \sum\limits_{k=0}^{\infty}  P^{(k)}(s,t)f,
$$
$$
 s>\rho(1-1/\gamma).
\end{equation}
Now, we have in the reminder the proof of the statement
$$
 \sum\limits_{k=0}^{\infty}  P^{(k)}(s,t)f \rightarrow f,\,t\rightarrow 0.
$$
 Using relation \eqref{4.41}, we can claim that  this fact has been  established by Lidskii V.B. in Lemma 5 \cite{firstab_lit:1Lidskii} in the case $f\in \mathrm{R}(B).$ However, the proof corresponding to the case $f\in \mathfrak{H}$ is represented in \cite{firstab_lit:3Agranovich1999}, in this case we should use the assumption  that $B$ is invertible $B^{-1}=W,$ it follows that $(W-\lambda I)^{-1}=B(I-\lambda B)^{-1}.$  The scheme of reasonings is also  represented in \cite{firstab_lit:Shkalikov A.}   Theorem 5.1. If we compare the conditions  of Theorem 5.1  \cite{firstab_lit:Shkalikov A.}  with   the theorem conditions  we will see that the following assumption  is required
 $$
 \|(W-\lambda I)^{-1}\|\leq C|\lambda|^{-1},
 $$
 where $\lambda$ belongs to $\{z\in \mathbb{C}:\,\mathrm{arg} z=\psi\},\,\theta<|\psi|<\pi/2 $ - the ray having the origin at the point zero and not belonging to  the sector $\mathfrak{L}_{0}(\theta).$
However,  it holds due to Theorem 3.2   \cite[p.268]{firstab_lit:kato1980}, since $\Theta(W)\subset \mathfrak{L}_{0}(\theta)$ and as a result
$$
 \|(W-\lambda I)^{-1}\|\leq \frac{1}{\mathrm{dist}(\lambda,\overline{\Theta(W)})} \leq\frac{1}{\mathrm{dist}(\lambda,\mathfrak{L}_{0}(\theta))}=  \frac{1}{\sin(|\psi|-\theta)|\lambda|},\;\lambda\in\{z\in \mathbb{C}:\,\mathrm{arg} z=\psi\}.
$$
 The proof is complete.

\end{proof}

\section{Remarks}

In this chapter, we  formulated the sufficient conditions of the Abel-Lidskii  basis property  for a sectorial non-selfadjoint operator of the special type. Having studied  such an operator class, we strengthened  the conditions  regarding  the semi-angle of the sector  and  weakened  a great deal conditions regarding the involved parameters. Thus,  we clarified   the results  by  Lidskii devoted to the decomposition on the root vector system of the non-selfadjoint operator. We used  a technique of the entire function theory and introduced  the so-called  Schatten-von Neumann class of the convergence  exponent. Having considered  strictly accretive operators satisfying special conditions formulated in terms of the norm and used  a  sequence of contours of the power type,  we invented  a peculiar  method how to calculate  a contour integral involved in the problem.

\chapter{Functional calculus of non-selfadjoint operators}\label{Ch5}

\section[ \hfil
  Operators with the asymptotics more subtle]{Operators with the asymptotics more subtle  than  one of the power type}

\subsection {Preliminaries and  prerequisites}

In this paragraph, we aim to produce an example of an operator which real part or Hermitian component if it is defined has more subtle asymptotics than one of the power type. Observe the following condition
\begin{equation}\label{5.1}
 (\ln^{1+\kappa}x)'_{\lambda_{n}(H)}  =o(  n^{-\kappa}   ),\; \kappa>0,
\end{equation}
where $H$  is  the  real part of the operator, we   consider a case when $H$  is an operator with a discrete spectrum.
In the paper \cite{firstab_lit:1kukushkin2021}, there was    considered an example of the sequence of the eigenvalues   that satisfy  condition \eqref{5.1} and at the same time the following relation holds
\begin{equation}\label{5.2}
\forall \varepsilon>0:\; n^{-\kappa-\varepsilon}< \frac{1}{\lambda_{n}(H)}  <\frac{C}{n^{\kappa}\ln^{\kappa} \lambda_{n}(H)}.
\end{equation}
 Here, we should recall that the operator order was   defined in \cite{firstab_lit:Shkalikov A.}   as a value $\mu>0$ so that $s_{n}(R_{H})\leq C n^{-\mu},$  we see that this definition is rather vague for we are compelled to consider $\inf \mu$ to obtain the correct information.   Thus,  we can contemplate   that  the notion of the order  in application to the operator $H $ satisfying condition \eqref{5.2}  as well as the asymptotics of the power type are spoiled since $s_{n}(R_{H}) = \lambda^{-1}_{n}(H).$ However, this kind of asymptototics \eqref{5.2} allows us to deploy fully the technicalities given by the Fredholm determinant.
Ostensibly, a remarkable    question appears  "Whether there exists an operator defined analytically which real part satisfies the condition \eqref{5.1}"?  It will be a crucial point of our narrative and we approach it from several  points of view.
Here,  we demonstrate the mentioned above example having drawn  the reader attention to the fact  that   we lift restrictions on $\kappa$ made in \cite{firstab_lit:1kukushkin2021} dictated by technicalities.

\begin{ex}\label{E5.1} Here we   produce an example of the sequence   $\{\lambda_{n} \}_{1}^{\infty}$  that satisfies   condition \eqref{5.1}, whereas
 $$
\sum\limits_{n=1}^{\infty}\frac{1}{|\lambda_{n}|^{1/\kappa}}=\infty.
$$
\end{ex}
Consider a sequence $\lambda_{n}=n^{\kappa}\ln^{\kappa} (n+q) \cdot \ln^{\kappa}\ln (n+q),\,q>e^{e}-1,\;n=1,2,...,\,.$  Using the integral test for convergence, we can easily see that the last series is divergent. At the same time substituting, we get
$$
\frac{\ln^{\kappa}\lambda_{n}}{\lambda_{n}} \leq
 \frac{ C\ln^{\kappa}(n+q)  }{n^{\kappa}\ln^{\kappa} (n+q) \cdot \ln^{\kappa}\ln (n+q)}= \frac{ C }{n^{\kappa}   \cdot \ln^{\kappa}\ln (n+q)},
$$
what gives us the fulfilment of the  condition.
This gives us the fact
$$
 H\bar{\in}\mathfrak{S}_{\kappa},\;H\in \mathfrak{S}_{p},\,\inf p=\kappa.
$$

\subsection{Operator function}

The following consideration  are not being reduced to a trivial finite-dimensional case since under assumptions  H1,H2  in consequence with Theorem \ref{T2.6} statement $(\mathbf{C})$ the operator $W$ has an infinite set of eigenvectors. It becomes clear if we notice that   since the algebraic multiplicities are finite-dimensional and the system of the root vectors is complete in $\mathfrak{H},$ then the set of the eigenvalues is infinite. Thus,   conditions   H1,H2 give us an opportunity  to avoid a trivial case. However, we can assume directly that the system of the eigenvectors is infinite    without worrying to lose  generality of reasonings. Bellow, we consider a sector $\mathfrak{L}_{0}(\theta_{0},\theta_{1}):=\{z\in \mathbb{C},\, \theta_{0}\leq arg z \leq\theta_{1}\},\,-\pi<\theta_{0}<\theta_{1}<\pi$ and use a short-hand notation $\mathfrak{L}_{0}(\theta):=\mathfrak{L}_{0}(-\theta,\theta).$  Although the reasonings represented bellow cover the case of a compact operator $B: \mathfrak{H}\rightarrow \mathfrak{H}$ such that
$
 \Theta(B) \subset \mathfrak{L}_{0}(\theta_{0},\theta_{1} ),
$
we assume that
$
 \Theta(B) \subset \mathfrak{L}_{0}(\theta ),
$
and   put the following contour   in correspondence to the operator
\begin{equation*}
\vartheta(B):=\left\{\lambda:\;|\lambda|=r>0,\, -\theta \leq\mathrm{arg} \lambda \leq \theta \right\}\cup\left\{\lambda:\;|\lambda|>r,\;
  \mathrm{arg} \lambda =-\theta ,\,\mathrm{arg} \lambda =\theta \right\},
\end{equation*}
where   the number $r$ is chosen so that the operator  $ (I-\lambda B)^{-1} $ is regular within the corresponding closed circle, more detailed see Chapter \ref{Ch4}.
  Assume that generally $\varphi$ is a function of the complex variable and
define an operator function as follows
$$
\frac{1}{2 \pi i}\int\limits_{\vartheta(B)}\varphi(\lambda) e^{-\varphi^{\alpha}(\lambda)  t} B(I-\lambda B)^{-1}fd\lambda= \varphi(W)u(t),
$$
where
$$
u(t)=\frac{1}{2 \pi i}\int\limits_{\vartheta(B)}e^{-\varphi^{\alpha}(\lambda) t}B(I-\lambda B)^{-1}fd\lambda,\,f\in \mathrm{D}(\varphi),
$$
the latter symbol denotes   a subset of the Hilbert space on which the given above integral constructions exist. The operator $W$ is called the operator argument. It is clear that the latter issue depends on both the properties of the operator argument and the properties of the operator function. Bellow, we produce sufficient conditions under which being imposed the operator function exists.

It is remarkable that   the given above definition corresponds to the closure of the operator function considered in   Lemma 3 \cite{firstab_lit(frac2023)}.\\

\noindent $(\mathrm{H}\mathrm{I})$ The operator $B$ is compact, $ \Theta(B)\subset \mathfrak{L}_{0}(\theta ),$   the entire function $\varphi$ of the order less than a half  maps the sector $\mathfrak{L}_{0}(\theta )$ into the sector $\mathfrak{L}_{0}(\varpi),\,\varpi<\pi/2\alpha,\,\alpha>0,$  its zeros with a sufficiently large absolute value   do not belong to the sector $\mathfrak{L}_{0}(\theta ),$ in the case $\alpha \neq  [\alpha].$

\begin{lem}\label{L5.1}  Assume that the condition $(\mathrm{H}\mathrm{I})$ holds,       the entire function $\varphi$ is of the order less than a half.  Then the following relation holds
\begin{equation}\label{5.3}
 \int\limits_{\vartheta(B)}\varphi(\lambda) e^{-\varphi^{\alpha}(\lambda)  t} B(I-\lambda B)^{-1}fd\lambda= \varphi(W)\!\!\!\int\limits_{\vartheta(B)}e^{-\varphi^{\alpha}(\lambda) t}B(I-\lambda B)^{-1}fd\lambda,\,f\in \mathrm{D}(W^{n}),\,\forall n\in \mathbb{N},
\end{equation}
moreover
\begin{equation}\label{5.4}
\lim\limits_{t\rightarrow+0}\frac{1}{2 \pi i}\int\limits_{\vartheta(B)}e^{-\varphi^{\alpha}(\lambda) t}B(I-\lambda B)^{-1}fd\lambda=f,\,f\in \mathrm{D}(W),
\end{equation}
where
$$
W=B^{-1},\,\varphi(z)=\sum\limits_{k=0}^{\infty}c_{k}z^{k},\;\varphi (W)=\sum\limits_{k=0}^{\infty}c_{k}W^{k},
$$
the latter series is assumed to be convergent pointwise in the sense of the norm of the Hilbert space.
\end{lem}
\begin{proof} Firstly, we should note that the conditions imposed upon the order of the function $\varphi$ alow us to claim that the latter integral converges for a fixed value of the parameter $t.$
Let us establish the formula
\begin{equation}\label{5.5}
 \int\limits_{\vartheta(B)}\varphi(\lambda) e^{-\varphi^{\alpha}(\lambda)  t} B(I-\lambda B)^{-1}f d\lambda=\sum\limits_{n=0}^{\infty}c_{n} \!\!
 \int\limits_{\vartheta(B)} e^{-\varphi^{\alpha}(\lambda)  t} \lambda^{n}B(I-\lambda B)^{-1} fd\lambda.
\end{equation}
To prove this fact, we should show  that the following relation holds
\begin{equation}\label{5.6}
 \int\limits_{\vartheta_{j}  (B)}\!\!\varphi(\lambda) e^{-\varphi^{\alpha}(\lambda)  t} B(I-\lambda B)^{-1}f d\lambda=\sum\limits_{n=0}^{\infty}c_{n} \!\! \int\limits_{\vartheta_{j}  (B)} e^{-\varphi^{\alpha}(\lambda)  t} \lambda^{n}B(I-\lambda B)^{-1} fd\lambda,
\end{equation}
where
$$
\vartheta_{j}(B):=\left\{\lambda:\;|\lambda|=r>0,\, \theta_{0} \leq\mathrm{arg} \lambda \leq \theta_{1} \right\}\cup\left\{\lambda:\;r<|\lambda|<r_{j},\,\mathrm{arg} \lambda =\theta_{0}   ,\,\mathrm{arg} \lambda =\theta_{1} \right\},
$$
$r_{j}\uparrow \infty.$
Note that in accordance with  Lemma 6 \cite{firstab_lit:1kukushkin2021}, we get
$$
\|(I-\lambda B)^{-1}\|\leq C,\,r<|\lambda|<r_{j},\,\mathrm{arg} \lambda =\theta_{0}   ,\,\mathrm{arg} \lambda =\theta_{1} .
$$
Using this estimate, we can easily obtain the fact
$$
\sum\limits_{n=0}^{\infty}|c_{n}|    | e^{-\varphi^{\alpha}(\lambda)  t}| |\lambda^{n}|\cdot\|B(I-\lambda B)^{-1} f\|  \leq C\|B\|\cdot \|f\| \sum\limits_{n=0}^{\infty}    |c_{n}|  |\lambda|^{n}      e^{- \mathrm{Re}\varphi^{\alpha}(\lambda)  t},\,\lambda\in \vartheta_{j}(B),
$$
where the latter series is convergent. Therefore, reformulating      the well-known theorem of calculus  on the absolutely convergent series in   terms of the norm, we obtain   \eqref{5.6}. Now, let us show that the series
\begin{equation}\label{5.7}
\sum\limits_{n=0}^{\infty}c_{n}\!\!\!\int\limits_{ \vartheta_{j}  (B)}  e^{-\varphi^{\alpha}(\lambda)  t} \lambda^{n}B(I-\lambda B)^{-1} fd\lambda
\end{equation}
is uniformly convergent with respect to $j\in \mathbb{N}.$
 Using Lemma 1 \cite{firstab_lit(axi2022)}, we get  a trivial inequality
$$
 \left\|\,\int\limits_{\vartheta_{j}  (B)} e^{-\varphi^{\alpha}(\lambda)  t} \lambda^{n}B(I-\lambda B)^{-1} fd\lambda \right\|_{\mathfrak{H}}\leq C\|f\|_{\mathfrak{H}} \!\!\!\int\limits_{\vartheta_{j}  (B)} e^{-\mathrm{Re} \varphi^{\alpha}(\lambda)  t  } |\lambda|^{n} |d\lambda|\leq
 $$
 $$
 \leq C\|f\|_{\mathfrak{H}} \int\limits_{\vartheta_{j}  (B)} e^{-C|\varphi(\lambda)|^{\alpha}t} |\lambda|^{n} |d\lambda|.
 $$
Here, we should note that to obtain the desired result one is satisfied with a rather rough estimate dictated by the estimate obtained in Lemma \ref{L4.5}, we get
$$
 \int\limits_{\vartheta_{j}  (B)} e^{-| \varphi(\lambda)|^{\alpha}  t  } |\lambda|^{n} |d\lambda|\leq    C\int\limits_{r}^{r_{j}} e^{-  x  t  } x^{n}  dx\leq C t^{-n}\Gamma(n+1) .
$$
Thus, we obtain
\begin{equation*}
\left\|\,\int\limits_{ \vartheta_{j}(B)} e^{-\varphi^{\alpha}(\lambda)  t} \lambda^{n}B(I-\lambda B)^{-1} fd\lambda \right\|_{\mathfrak{H}}\leq C t^{-n}n! \,.
\end{equation*}
Using the standard  formula   establishing the estimate for the  Taylor coefficients  of   the entire function, then applying    the Stirling formula, we get
 $$
|c_{n}|<\left(e \sigma \varrho \right)^{n/\varrho} n^{-n/\varrho}< (2\pi)^{1/2\varrho}(\sigma\varrho)^{n/\varrho}\left(\frac{\sqrt{n}}{n!}\right)^{1/\varrho}\!\!\!,
$$
where $0<\sigma<\infty$ is a type of the function $\varphi.$ Thus, we obtain
\begin{equation*}
 \sum\limits_{n=1}^{\infty}|c_{n}|\left\|\,\int\limits_{\vartheta_{j}(B)} e^{-\varphi^{\alpha}(\lambda)  t} \lambda^{n}B(I-\lambda B)^{-1} fd\lambda\right\|\leq  C\sum\limits_{n=1}^{\infty}(\sigma\varrho)^{n/\varrho}t^{-n}(n!)^{1-1/\varrho} n^{1/2\varrho}  .
\end{equation*}
The latter series is convergent for an arbitrary fixed  $t>0,$ what proves the  uniform convergence of the series \eqref{5.7} with respect to $j\,.$ Therefore, reformulating      the well-known theorem of calculus applicably to the norm of the Hilbert space,  taking into accounts the facts
\begin{equation*}
 \int\limits_{\vartheta_{j} (B)}\!\!\varphi(\lambda) e^{-\varphi^{\alpha}(\lambda)  t} B(I-\lambda B)^{-1}f d\lambda \stackrel{\mathfrak{H}}{\longrightarrow} \!\int\limits_{\vartheta   (B)}\!\!\varphi(\lambda) e^{-\varphi^{\alpha}(\lambda)  t} B(I-\lambda B)^{-1}f d\lambda,
 $$
 $$
 \int\limits_{\vartheta_{j}  (B)}\!\! e^{-\varphi^{\alpha}(\lambda)  t} \lambda^{n}B(I-\lambda B)^{-1} fd\lambda\stackrel{\mathfrak{H}}{\longrightarrow}\!\int\limits_{\vartheta   (B)} e^{-\varphi^{\alpha}(\lambda)  t} \lambda^{n}B(I-\lambda B)^{-1} fd\lambda,\;j\rightarrow\infty,
\end{equation*}
we obtain  formula \eqref{5.5}. Further, using the formula
\begin{equation*}
  \lambda^{k} B^{k}(I-\lambda B)^{-1}=(I-\lambda B)^{-1}-(I+\lambda B+...+\lambda^{k-1}B^{k-1}),\;k\in \mathbb{N},
\end{equation*}
taking into account the facts that the operators  $B^{k}$ and $(I-\lambda B)^{-1}$ commute,
we obtain
$$
 \int\limits_{\vartheta(B)}e^{-\varphi^{\alpha}(\lambda)  t}\lambda^{n}B(I-\lambda B)^{-1} fd\lambda=
$$
$$
= \int\limits_{\vartheta(B)}e^{-\varphi^{\alpha}(\lambda)  t}B(I-\lambda B)^{-1}W^{n}fd\lambda- \int\limits_{\vartheta(B)}e^{-\varphi^{\alpha}(\lambda)  t} \sum\limits_{k=0}^{n-1}\lambda^{k}B^{k+1} W^{n}fd\lambda=I_{1}(t)+I_{2}(t).
$$
Since the operators $W^{n}$ and $B(I-\lambda B)^{-1}$ commute,   this fact can be obtained  by direct calculation, we get
$$
I_{1}(t)=W^{n}\!\!\!\int\limits_{\vartheta(B)}e^{-\varphi^{\alpha}(\lambda)  t}B(I-\lambda B)^{-1} fd\lambda.
$$
Consider $I_{2}(t),$ using the technique applied in Lemma 5 \cite{firstab_lit(axi2022)} it is rather reasonable to consider the following representation
\begin{equation*}
 I_{2}(t):=
 -\sum\limits_{k=0}^{n-1 }\beta_{k}(t)B^{k-n+1}  f,\;
   \beta_{k}(t):=\!\!\int\limits_{\vartheta(B)}\!\!e^{-\varphi^{\alpha}(\lambda)t} \lambda^{k}d\lambda.
\end{equation*}
Analogously to the scheme of   reasonings of Lemma 5 \cite{firstab_lit(axi2022)}, we can show that $\beta_{k}(t)=0,$ under  the imposed  condition of the entire function growth regularity. Bellow, we produce a complete reasoning to avoid any kind of   misunderstanding.
Let us show that  $\beta_{k}(t)=0,$  define a contour  $\vartheta_{R}(B):= \mathrm{Fr}\left\{\mathrm{int }\,\vartheta(B) \,\cap \{\lambda:\,r<|\lambda|<R \}\right\}$ and let us prove that there exists such a sequence    $\{R_{n}\}_{1}^{\infty},\,R_{n}\uparrow \infty$ that
\begin{equation}\label{5.8}
  \oint\limits_{\vartheta_{R_{n}}(B)}\!\!\!\!e^{-\varphi^{\alpha}(\lambda)  t} \lambda^{k}d\lambda\rightarrow \beta_{k}(t),\;n\rightarrow \infty.
\end{equation}
 Consider a decomposition of the contour $\vartheta_{R}(B)$ on terms
$
\tilde{\vartheta}_{ R}:=\{\lambda:\,|\lambda|=R,\, \theta_{0}  \leq\mathrm{arg} \lambda  \leq\theta_{1}  \},
$
and
$
\hat{\vartheta}_{R}:= \{\lambda:\,|\lambda|=r,\,\theta_{0}  \leq\mathrm{arg} \lambda  \leq\theta_{1}\}\cup
\{\lambda:\,r<|\lambda|<R,\, \mathrm{arg} \lambda  =\theta_{0} ,\,\mathrm{arg} \lambda  = \theta_{1} \}.
$
We have
$$
 \oint\limits_{\vartheta_{R}(B)}e^{-\varphi^{\alpha}(\lambda)  t} \lambda^{k}d\lambda=
 \int\limits_{\tilde{\vartheta}_{ R}}e^{-\varphi^{\alpha}(\lambda)  t} \lambda^{k}d\lambda+
  \int\limits_{\hat{\vartheta}_{R}}e^{-\varphi^{\alpha}(\lambda)  t} \lambda^{k}d\lambda.
$$
Having noticed that the integral at the left-hand side of the last relation   equals to zero, since the   function under the integral is analytic   inside the contour, we  come to the conclusion that to obtain the desired result it suffices to show that
\begin{equation}\label{5.9}
 \int\limits_{\tilde{\vartheta}_{ R_{n}}}e^{-\varphi^{\alpha}(\lambda)  t} \lambda^{k}d\lambda\rightarrow 0,
\end{equation}
where     $\{R_{n}\}_{1}^{\infty},\,R_{n}\uparrow \infty.$
We have
\begin{equation*}
  \left|\,\int\limits_{\tilde{\vartheta}_{ R}}e^{-\varphi^{\alpha}(\lambda)  t}\lambda^{k}    d \lambda\,\right| \leq  R^{k}\int\limits_{\tilde{\vartheta}_{ R}}|e^{-\varphi^{\alpha}(\lambda)  t}| |d \lambda|\leq R^{k+1}\int\limits_{ \theta_{0} }^{\theta_{1}}  e^{-\mathrm{Re} \varphi^{\alpha}(\lambda)  t  } d \,\mathrm{arg} \lambda.
\end{equation*}
Consider a value   $\mathrm{Re}\,\varphi^{\alpha}(\lambda),\, \lambda\in  \tilde{\vartheta}_{ R}$  for a sufficiently large value $R.$
Using the condition imposed upon the order of the entire function and  applying  the Wieman theorem  (Theorem 30,  \S 18, Chapter I \cite{firstab_lit:Eb. Levin}),  we can claim that there exists such a sequence $\{R_{n}\}_{1}^{\infty},\,R_{n}\uparrow \infty$ that
\begin{equation*}
\forall \varepsilon>0,\,\exists N(\varepsilon):\,e^{- C|\varphi(\lambda )|^{\alpha}t}\leq e^{- C m^{\alpha}_{\varphi}(R_{n})t}\leq e^{- C t[M_{\varphi}(R_{n})]^{(\cos \pi \varrho-\varepsilon)\alpha}},\,\lambda \in \tilde{\vartheta}_{ R_{n}},\,n>N(\varepsilon),
\end{equation*}
where $\varrho$ is the order of the entire  function $\varphi.$ Applying this estimate, we obtain
$$
\int\limits_{ \theta_{0} }^{\theta_{1}}  e^{-\mathrm{Re} \varphi^{\alpha}(\lambda)  t  } d \,\mathrm{arg} \lambda \leq
  \int\limits_{ \theta_{0} }^{\theta_{1}}  e^{- C t |\varphi (\lambda)|^{\alpha} }  d \,\mathrm{arg} \lambda\leq   e^{- C t[M_{\varphi}(R_{n})]^{(\cos \pi \varrho-\varepsilon)\alpha}}\!\!\!\int\limits_{ \theta_{0} }^{\theta_{1}}    d \,\mathrm{arg} \lambda.
$$
The latter estimate gives us \eqref{5.9} from what follows \eqref{5.8}. Therefore $\beta_{k}(t)=0,$ hence $I_{2}(t)=0$ and we get
$$
 \int\limits_{\vartheta(B)}e^{-\varphi^{\alpha}(\lambda)  t}\lambda^{n}B(I-\lambda B)^{-1} fd\lambda=W^{n}\!\!\!\!\int\limits_{\vartheta(B)}e^{-\varphi^{\alpha}(\lambda) t}B(I-\lambda B)^{-1}fd\lambda.
$$
Substituting the latter relation into the formula \eqref{5.5}, we obtain the first statement of the lemma.

The scheme of the proof corresponding to the second statement  is absolutely analogous to the one presented in Lemma 4 \cite{firstab_lit(axi2022)}, we should just use Lemma \ref{L4.5} providing the estimates along the sides of the contour. Thus, the completion of the reasonings is due to the technical repetition of the Lemma 4 \cite{firstab_lit(axi2022)} reasonings, we left it to the reader.
\end{proof}
\begin{remark}\label{R5.1} Note that the lemma conditions guarantees the inclusion
$$
\mathrm{D}^{\infty}(W)\subset \mathrm{D}(\varphi),
$$
where
$$
\mathrm{D}^{\infty}(W):=\left\{ f\in \mathfrak{H}:\;f\in \mathrm{D}(W^{n}),\,n\in \mathbb{N}  \right\}.
$$
However, the hypotheses $ \mathrm{D}(\varphi)=\mathrm{D}^{\infty}(W)$ requires additional consideration. This is why for the sake of the simplicity and the reader convenience we restrict our reasonings by the latter equality put it as an artificial assumption.
\end{remark}

Choosing $\alpha=1$ and taking into account the obvious relation
$$
B(I-\lambda B)^{-1}=R_{W}(\lambda)(I-\lambda B)WB(I-\lambda B)^{-1}=R_{W}(\lambda),
$$ we get
\begin{equation}\label{5.10}
 \frac{1}{2 \pi i}\int\limits_{\vartheta(B)}e^{-\varphi  (\lambda) t}\varphi (\lambda)R_{W}(\lambda)f d\lambda= \varphi (W)\frac{1}{2 \pi i}\int\limits_{\vartheta(B)}e^{-\varphi  (\lambda) t} R_{W}(\lambda)f d\lambda,\;f\in \mathrm{D}(W^{n}),\,n\in \mathbb{N}.
\end{equation}
\begin{equation}\label{5.11}
\lim\limits_{t\rightarrow+0}\frac{1}{2 \pi i}\int\limits_{\vartheta(B)}e^{-\varphi  (\lambda) t} R_{W}(\lambda)fd\lambda=f,\,f\in \mathrm{D}(W).
\end{equation}
Consider a set
$$
\mathrm{D}_{t}(W):=\left\{f(t)=\frac{1}{2 \pi i}\int\limits_{\vartheta(B)}e^{-\varphi (\lambda) t}R_{W}(\lambda)fd\lambda,\,f\in \mathrm{D}(W)\right\},
$$
in accordance with relation \eqref{5.11}, we have that the set $\mathrm{D}_{t}(W)$ is dense in $\mathrm{D}(W).$ Obviously, it is natural to extend the  operator function $\varphi (W)$ to the closure  of the  the subset  of $\mathrm{D}_{t}(W),$ however there are some difficulties that may prevent the idea of the operator closedness   on the set $\mathrm{D}_{t}(W).$ At the same time, we can prove an analog of  closedness that is in the following. Assume that there exist  simultaneous   limits  $u^{(j)}_{k}(t)\rightarrow u^{(0)},\; \varphi(W)u^{(j)}_{k}(t)\rightarrow u^{(j)},\,k\rightarrow  \infty,\,t\rightarrow +0,\,j=1,2,$  then $u^{(1)}=u^{(2)}.$  Note that in accordance with \eqref{5.11}, for each $k\in \mathbb{N},$ we get $u^{(j)}_{k}(t)\rightarrow u^{(j)}_{k},\,t\rightarrow+0.$ Applying  the  theorem which gives the connection between simultaneous limits and repeated limits, we get $u^{(j)}_{k}\rightarrow u^{(0)},\,k\rightarrow \infty.$ Using   the simple estimating, we get
$$
 \left\|\, \int\limits_{\vartheta(B)}\varphi(\lambda)e^{-\varphi  (\lambda) t} R_{W}(\lambda) \left\{u^{(j)}_{k}-u^{(0)}\right\}d\lambda\right\|_{\mathfrak{H}}\leq C\left\|u^{(j)}_{k}-u^{(0)}\right\|_{\mathfrak{H}} \int\limits_{\vartheta(B)}|\varphi(\lambda)e^{-\varphi  (\lambda) t}|\cdot |d \lambda|\leq C\left\|u^{(j)}_{k}-u^{(0)}\right\|_{\mathfrak{H}}.
$$
Therefore, there exist coincident    limits
$$
  \varphi(W)u^{(j)}_{k}(t) \rightarrow \frac{1}{2\pi i}\int\limits_{\vartheta(B)}\varphi(\lambda)e^{-\lambda t} B(I-\lambda B)^{-1}  u^{(0)}d\lambda,\,k\rightarrow\infty.
$$
Applying the theorem on the connection between simultaneous limits and repeated limits, we get $u^{(1)}=u^{(2)}.$ Thus, we have proved that the operator $\varphi(W)$ is closeable in some sense and naturally come to the extension of the operator  function.
In particular if we have a pair of limits
 $$
\lim\limits_{t\rightarrow+0}\frac{1}{2 \pi i}\int\limits_{\vartheta(B)}e^{-\varphi (\lambda) t} R_{W}(\lambda)f d\lambda=f\in \mathrm{D}(W),\;\lim\limits_{t\rightarrow+0}\frac{1}{2 \pi i}\int\limits_{\vartheta(B)}\varphi (\lambda) e^{-\varphi (\lambda) t}R_{W}(\lambda)f d\lambda=h\in \mathfrak{H},
$$
then in accordance with the above, we have $\varphi (W)f=h.$\\

The following lemma represents conditions upon the function $\varphi$ under which the latter admits the extension.
\begin{lem}\label{L5.2} Suppose  $B$ is a compact  operator,  $\Theta(B)\subset \mathfrak{L}_{0}(\theta),$
\begin{equation}\label{5.12}
\varphi(z)=\sum\limits_{n=-\infty}^{s}c_{n}z^{n},\,z\in  \mathbb{C},\,s\in \mathbb{N},\; \max\limits_{n=0,1,...,s}(|\mathrm{arg} c_{n}|+n\theta)<\pi/2\alpha,\,\alpha\geq1,
\end{equation}
  then
\begin{equation}\label{5.13}
 \frac{1}{2\pi i}\int\limits_{\vartheta(B)}\varphi(\lambda) e^{-\varphi^{\alpha}(\lambda)  t} B(I-\lambda B)^{-1}fd\lambda= \varphi(W)u(t);\;\; \lim\limits_{t\rightarrow +0}\varphi(W)u(t)= \varphi(W)f,
\end{equation}
where
$$
u(t):=\frac{1}{2\pi i}\int\limits_{\vartheta(B)}  e^{-\varphi^{\alpha}(\lambda)  t} B(I-\lambda B)^{-1}fd\lambda,\,f\in \mathrm{D}(W^{s}).
$$
\end{lem}
\begin{proof}
Consider a decomposition of the  Laurent series on  two terms
$$
\varphi_{1}(z)=\sum\limits_{n=0}^{s}c_{n}z^{n};\;\varphi_{2}(z)=\sum\limits_{n=1}^{\infty}c_{-n}z^{-n}.
$$
Consider an obvious relation
\begin{equation}\label{5.14}
  \lambda^{k} B^{k}(E-\lambda B)^{-1}=(E-\lambda B)^{-1}-(E+\lambda B+...+\lambda^{k-1}B^{k-1}),\;k\in \mathbb{N}.
\end{equation}
It gives us the following representation
\begin{equation}\label{5.15}
\frac{1}{2\pi i}\int\limits_{\vartheta(B)} \lambda^{n} e^{-\varphi^{\alpha}(\lambda) t}B(I-\lambda B)^{-1}fd\lambda= I_{1n}(t)
+I_{2n}(t),\;n\in \mathbb{Z}^{-}\cup \{0,1,...,s\},
\end{equation}
 where
$$
I_{1n}:= \frac{1}{2\pi i}\int\limits_{\vartheta(B)}e^{-\varphi^{\alpha}(\lambda) t} (I-\lambda B)^{-1}W^{n-1}fd\lambda,\;I_{2n}(t):=0,\,n=0,
$$
\begin{equation*}
 I_{2n}(t):=  \left\{ \begin{aligned}
 -\sum\limits_{k=0}^{n-1 }\beta_{k}(t)B^{k-n+1}  f,\;n> 0,\\
 \sum\limits_{k=-1}^{n}\beta_{k}(t)B^{k-n+1}  f  ,\;  n<0\, \\
\end{aligned}
 \right.,  \;\; \beta_{k}(t):=\frac{1}{2\pi i}\int\limits_{\vartheta(B)}e^{-\varphi^{\alpha}(\lambda)t} \lambda^{k}d\lambda.
\end{equation*}
Let us show that  $\beta_{k}(t)=0,$  define a contour  $\vartheta_{R}(B):= \mathrm{Fr}\left\{\mathrm{int }\,\vartheta(B) \,\cap \{\lambda:\,r<|\lambda|<R \}\right\}$ and let us prove that
\begin{equation}\label{5.16}
I_{Rk}(t):= \frac{1}{2\pi i}\oint\limits_{\vartheta_{R}(B)}e^{-\varphi^{\alpha}(\lambda)t} \lambda^{k}d\lambda\rightarrow \beta_{k}(t),\;R\rightarrow \infty.
\end{equation}
Consider a decomposition of the contour $\vartheta_{R}(B)$ on terms   $\tilde{\vartheta}_{ R}:=\{\lambda:\,|\lambda|=R,\,|\mathrm{arg} \lambda |\leq\theta +\varsigma\}$ and
$\hat{\vartheta}_{R}:= \{\lambda:\,|\lambda|=r,\,|\mathrm{arg} \lambda |\leq\theta +\varsigma\}\cup
\{\lambda:\,r<|\lambda|<R,\, \mathrm{arg} \lambda  =\theta +\varsigma\}\cup
\{\lambda:\,r<|\lambda|<R,\, \mathrm{arg} \lambda  =-\theta -\varsigma\}.$
We have
$$
\frac{1}{2\pi i}\oint\limits_{\vartheta_{R}(B)}e^{-\varphi^{\alpha}(\lambda)t} \lambda^{k}d\lambda=
 \frac{1}{2\pi i}\int\limits_{\tilde{\vartheta}_{ R}}e^{-\varphi^{\alpha}(\lambda)t} \lambda^{k}d\lambda+
 \frac{1}{2\pi i}\int\limits_{\hat{\vartheta}_{R}}e^{-\varphi^{\alpha}(\lambda)t} \lambda^{k}d\lambda.
$$
Having noticed that $I_{Rk}(t)=0,$ since the operator function under the integral is analytic inside the contour, we  come to the conclusion that to obtain the desired result, we should show
\begin{equation}\label{5.17}
\frac{1}{2\pi i}\int\limits_{\tilde{\vartheta}_{ R}}e^{-\varphi^{\alpha}(\lambda)t} \lambda^{k}d\lambda\rightarrow 0,\;R\rightarrow \infty.
\end{equation}
We have
\begin{equation*}
  \left|\,\int\limits_{\tilde{\vartheta}_{ R}}e^{-\varphi^{\alpha}(\lambda)t}\lambda^{k}    d \lambda\,\right| \leq \,R^{k}\int\limits_{\tilde{\vartheta}_{ R}}|e^{-\varphi^{\alpha}(\lambda)t}| |d \lambda|\leq R^{k+1}\int\limits_{-\theta-\varsigma}^{\theta+\varsigma}  e^{- t \mathrm{Re}\,\varphi(\lambda) } d \,\mathrm{arg} \lambda.
\end{equation*}
Consider a value   $\mathrm{Re}\,\varphi(\lambda),\, \lambda\in  \tilde{\vartheta}_{ R}$  for a sufficiently large value $R.$ Using the property of the principal part of the Laurent series in is not hard to prove that $\forall\varepsilon>0,\,\exists N(\varepsilon):|\varphi_{2}(\lambda)|<\varepsilon,\,R>N(\varepsilon).$ It follows easily from the condition  \eqref{5.12} that $\mathrm{Re}\,\varphi^{\alpha}_{1}(\lambda)\geq C|\varphi_{1}(\lambda)|^{\alpha},\,\lambda \in \mathfrak{L}_{0}(\theta)$ for a sufficiently large value $R.$  It is clear that
$|\varphi_{1}(\lambda)|\sim |c_{s}| R^{s},\,R\rightarrow\infty.$ Thus, we have
\begin{equation}\label{5.18}
e^{- t \mathrm{Re}\,\varphi^{\alpha}(\lambda) }\leq    e^{- Ct |\varphi(\lambda)|^{\alpha} }\leq   e^{-Ct|\lambda|^{\alpha s}  },\,\lambda\in  \tilde{\vartheta}_{R}.
\end{equation}
Applying this estimate, we obtain
$$
\int\limits_{-\theta-\varsigma}^{\theta+\varsigma}  e^{- t \mathrm{Re}\,\varphi^{\alpha}(\lambda) } d \,\mathrm{arg} \lambda \leq
  \int\limits_{-\theta-\varsigma}^{\theta+\varsigma}  e^{- C t |\varphi(\lambda)|^{\alpha} }  d \,\mathrm{arg} \lambda\leq   e^{- C t  R^{\alpha s} }\int\limits_{-\theta-\varsigma}^{\theta+\varsigma}    d \,\mathrm{arg} \lambda.
$$
The latter estimate gives us \eqref{5.17} from what follows \eqref{5.16}. Therefore $\beta_{k}(t)=0$ and we obtain the fact  $I_{2n}(t)=0.$    Combining the fact of the operator $W$  closedness (see \cite[p.165]{firstab_lit:kato1980} ) with   the definition of the integral in the Riemann sense, we get easily
$$
 W^{n}u(t)=\frac{1}{2\pi i}\int\limits_{\vartheta(B)}e^{-\varphi^{\alpha}(\lambda)t} B(I-\lambda B)^{-1}W^{n}fd\lambda,\,n=0,1,...,s.
$$
Thus, using the formula \eqref{5.15}, we obtain
\begin{equation*}
 \frac{1}{2\pi i}\int\limits_{\vartheta(B)}\varphi_{1}(\lambda) e^{-\varphi^{\alpha}(\lambda)  t} B(I-\lambda B)^{-1}fd\lambda= \varphi_{1}(W)u(t).
\end{equation*}
Consider a principal part of the Laurent series. Using the formula \eqref{5.15},  we get for values $n\in \mathbb{N}$
$$
\frac{1}{2\pi i}\int\limits_{\vartheta(B)} \lambda^{-n} e^{-\varphi^{\alpha}(\lambda)t}B(I-\lambda B)^{-1}fd\lambda =B^{n}u(t).
$$
 Not that  by virtue of a character of the convergence of the series principal part, we have
$$
  \left\|\sum\limits_{n=1}^{\infty}  c_{-n}   e^{-\varphi^{\alpha}(\lambda) t}(I-\lambda B)^{-1}B^{n+1}f\,\right\|_{\mathfrak{H}}\leq C\|f\|_{\mathfrak{H}}\sum\limits_{n=1}^{\infty}  \left|c_{-n}\right| \cdot  \left\|B\right\|_{\mathfrak{H}}^{n+1} <\infty,\;\lambda\in \vartheta(B).
$$
Therefore
$$
 \int\limits_{\vartheta_{j}(B)} \varphi_{2}(\lambda) e^{-\varphi^{\alpha}(\lambda)t}B(I-\lambda B)^{-1}fd\lambda=\sum\limits_{n=1}^{\infty}c_{-n}\!\!   \int\limits_{\vartheta_{j}(B)}  e^{-\varphi^{\alpha}(\lambda) t}(I-\lambda B)^{-1}B^{n+1}fd \lambda,\,j\in \mathbb{N},
$$
where $$
\vartheta_{j}(B):=\left\{\lambda:\;|\lambda|=r>0,\,|\mathrm{arg} \lambda|\leq \theta+\varsigma\right\}\cup\left\{\lambda:\;r<|\lambda|<r_{j},\,r_{j}\uparrow \infty,\; |\mathrm{arg} \lambda|=\theta+\varsigma\right\}.
$$
Analogously to \eqref{5.18}, we can easily get
\begin{equation}\label{5.19}
 e^{-\mathrm{Re}\varphi^{\alpha}(\lambda) t}\leq e^{-C|\varphi(\lambda)|^{\alpha} t}\leq   e^{-C|\lambda|^{\alpha s} t},\,\lambda\in \vartheta(B).
\end{equation}
Applying this estimate, we obtain
$$
\left\|\sum\limits_{n=1}^{\infty}c_{-n}\!\!   \int\limits_{\vartheta_{j}(B)}  e^{-\varphi^{\alpha}(\lambda) t}(I-\lambda B)^{-1}B^{n+1}fd \lambda\right\|_{\mathfrak{H}}\leq C \|f\|_{\mathfrak{H}}\sum\limits_{n=1}^{\infty}    |c_{-n}| \cdot\|B^{n+1}\|\!\int\limits_{\vartheta_{j}(B)}  e^{- C|\lambda|^{s} t} |d \lambda|\leq
$$
$$
 \leq C \|f\|_{\mathfrak{H}}\sum\limits_{n=1}^{\infty}    |c_{-n}| \cdot\|B\|^{n+1}\int\limits_{\vartheta(B)}  e^{-C|\lambda|^{s} t} |d \lambda|<\infty.
$$
Note that the uniform convergence  of the series  in the left-hand side with respect to $j$ follows from the latter estimate.
Reformulating      the well-known theorem of calculus  on the absolutely convergent series in   terms of the norm, we have
\begin{equation}\label{5.20}
\frac{1}{2\pi i}\!\!\int\limits_{\vartheta(B)}\!\! \varphi_{2}(\lambda) e^{-\varphi^{\alpha}(\lambda)t}B(I-\lambda B)^{-1}fd\lambda=\frac{1}{2\pi i}\sum\limits_{n=1}^{\infty} c_{-n}\!\!\!  \int\limits_{\vartheta(B)}  e^{-\varphi^{\alpha}(\lambda) t} (I-\lambda B)^{-1}B^{n+1}fd \lambda=
 \varphi_{2}(W)u(t).
\end{equation}
Thus, we obtain the first relation \eqref{5.13}. Let us establish the second relation \eqref{5.13}. Using the formula \eqref{5.14}, we  obtain
$$
\frac{1}{2\pi i}\int\limits_{\vartheta(B)} \lambda^{n} e^{-\varphi^{\alpha}(\lambda) t}B(I-\lambda B)^{-1}fd\lambda= I_{1n}(t)
+I_{2n}(t),\;n\in \mathbb{Z}^{-}\cup \{0,1,...,s\},
$$
where
$$
I_{1n}(t):= \frac{1}{2\pi i}\int\limits_{\vartheta(B)}e^{-\varphi^{\alpha}(\lambda) t}\lambda^{-2}(I-\lambda B)^{-1}W^{n+1}fd\lambda,\;I_{2n}(t):=0,\,n=-2,
$$
\begin{equation*}
 I_{2n}(t):=  \left\{ \begin{aligned}
 -\sum\limits_{k=-2}^{n-1 }\beta_{k}(t)B^{k-n+1}  f,\;n> -2,\\
 \sum\limits_{k=-3}^{n}\beta_{k}(t)B^{k-n+1}  f  ,\;  n\leq-3\, \\
\end{aligned}
 \right. .
\end{equation*}
Using the proved above fact $\beta_{k}(t)=0,$ we have  $I_{2n}(t)=0.$ Since in consequence of Lemma \ref{L4.7}, inequality \eqref{5.19} for arbitrary $j\in \mathbb{N},f\in \mathrm{D}(W^{s}),$ we have
$$
 e^{-\varphi^{\alpha}(\lambda) t}\lambda^{-2}(I-\lambda B)^{-1}W^{n+1}f\rightarrow \lambda^{-2}(I-\lambda B)^{-1}W^{n+1}f,\,t\rightarrow +0 ,\,\lambda\in \vartheta_{j}(B),
$$
where convergence is uniform   with respect to the variable $\lambda,$    the  improper  integral $I_{1n}(t)$
 is uniformly convergent with respect to the variable $t,$ then
  we get
\begin{equation*}
I_{1n}(t)\rightarrow \frac{1}{2\pi i}\int\limits_{\vartheta(B)} \lambda^{-2}(I-\lambda B)^{-1}W^{n+1}fd\lambda,\,t\rightarrow+0,
\end{equation*}
Note that the last integral can be calculated as a residue, we have
\begin{equation}\label{5.21}
 \frac{1}{2\pi i}\int\limits_{\vartheta(B)} \lambda^{-2}(I-\lambda B)^{-1}W^{n+1}fd\lambda=
  \lim\limits_{\lambda\rightarrow0}\frac{d (I-\lambda B)^{-1} }{d\lambda }W^{n+1}f=W^{n}f,
$$
$$
 n\in \mathbb{Z}^{-}\cup \{0,1,...,s\}.
\end{equation}
It is obvious that using this  formula, we obtain the following relation
\begin{equation*}
\lim\limits_{t\rightarrow+0}\frac{1}{2\pi i}\int\limits_{\vartheta(B)}\varphi_{1}(\lambda) e^{-\varphi^{\alpha}(\lambda) t} B(I-\lambda B)^{-1}fd\lambda= \sum\limits_{n=0}^{s}c_{n}W^{n}f ,\;f\in \mathrm{D}(W^{s}) .
\end{equation*}
Consider a principal part of the Laurent series.
The following reasonings are analogous to the above,   we get
$$
\left\| \sum\limits_{n=1}^{\infty}   c_{-n} \!\!\int\limits_{\vartheta(B)}   e^{-\varphi^{\alpha}(\lambda) t}\lambda^{-2} (I-\lambda B)^{-1}B^{n-1}fd\lambda\right\|\leq C \|f\|_{\mathfrak{H}}\sum\limits_{n=1}^{\infty}    |c_{-n}| \cdot\|B^{n-1}\|\int\limits_{\vartheta(B)} |\lambda|^{-2}  e^{-C|\lambda|^{s} t} |d \lambda|\leq
$$
$$
 \leq C \|f\|_{\mathfrak{H}}\sum\limits_{n=1}^{\infty}    |c_{-n}| \cdot\|B\|^{n-1}\int\limits_{\vartheta(B)}  |\lambda|^{-2} |d \lambda|<\infty.
$$
It gives us  the uniform convergence  of the series with respect to $t$  at the left-hand side of the last relation.
Using the analog of   the well-known theorem of calculus  on the absolutely convergent series, we have
$$
 \sum\limits_{n=1}^{\infty} c_{-n} \!\!  \int\limits_{\vartheta(B)}   e^{-\varphi^{\alpha}(\lambda) t}\lambda^{-2} (I-\lambda B)^{-1}B^{n-1}fd\lambda\rightarrow
  \sum\limits_{n=1}^{\infty} c_{-n}\!\!   \int\limits_{\vartheta(B)}    \lambda^{-2} (I-\lambda B)^{-1}B^{n-1}fd\lambda,\;t\rightarrow+0.
$$
   Taking into account  \eqref{5.20}, \eqref{5.21},  we get
$$
\lim\limits_{t\rightarrow+0}\frac{1}{2\pi i}\int\limits_{\vartheta(B)}\varphi_{2}(\lambda) e^{-\varphi^{\alpha}(\lambda) t} B(I-\lambda B)^{-1}fd\lambda= \sum\limits_{n=1}^{\infty}c_{-n}B^{n}f ,\;f\in \mathfrak{H}.
$$
It is clear that  the second relation \eqref{5.13} holds. The proof is complete.
\end{proof}

\subsection{Convergence with respect to the time variable}

Apparently the  generalization of the Abbell-Lidskii concept  made in this chapter requires study of the convergence of the integral construction with respect to the time variable. We should remaind that it is a  rather essential question in the concept of the summation in the Abel-Lidskii sense that appeals to passing to the limit with respect to the time variable. It is considered in detail in Chapter \ref{Ch4} in the classical case.   Bellow, we  consider a generalized case corresponding to the technicalities being involved by   the operator function concept.

\begin{lem}\label{L5.3} Suppose the operator $B$ satisfies conditions of Lemma \ref{L4.7},  the entire function $\varphi$ of the order less than a half  maps the inside of the contour $\vartheta(B)$ into the sector $\mathfrak{L}_{0}(\varpi),\,\varpi<\pi/2\alpha$  for a sufficiently large value $|z|,\,z\in \mathrm{int}\, \vartheta(B).$ Then the following relation holds
\begin{equation*}
\lim\limits_{t\rightarrow+0}\frac{1}{2 \pi i}\int\limits_{\vartheta(B)}e^{-\varphi^{\alpha}(\lambda) t}B(I-\lambda B)^{-1}fd\lambda=f,\,f\in \mathrm{D}(W).
\end{equation*}
\end{lem}
\begin{proof}
Using the formula
$$
B^{2}(I-\lambda B)^{-1}= \frac{1}{\lambda^{2}}\left\{\left(I- \lambda B\right)^{-1}-(I+\lambda B)   \right\},
$$
we obtain
$$
 \frac{1}{2 \pi i}\int\limits_{\vartheta(B)}e^{-\varphi^{\alpha}(\lambda) t}B(I-\lambda B)^{-1}fd\lambda=
 \frac{1}{2 \pi i}\int\limits_{\vartheta(B)}e^{-\varphi^{\alpha}(\lambda) t} \lambda^{-2}  \left(I- \lambda B\right)^{-1} Wfd\lambda -
$$
$$
- \frac{1}{2 \pi i}\int\limits_{\vartheta(B)}e^{-\varphi^{\alpha}(\lambda) t}\lambda^{-2} (I+\lambda B)  Wfd\lambda =I_{1}(t)+I_{2}(t).
$$
Consider $I_{1}(t).$ Since this improper integral is uniformly convergent regarding $t,$ this fact can be established easily if we apply Lemma \ref{L4.7}, then using the theorem on the connection with the simultaneous limit and the repeated limit, we get
$$
\lim\limits_{t\rightarrow+0}I_{1}(t)= \frac{1}{2 \pi i}\int\limits_{\vartheta(B)}  \lambda^{-2}  \left(I- \lambda B\right)^{-1} Wfd\lambda.
$$
define a contour  $\vartheta_{R}(B):= \mathrm{Fr}\big\{\{\lambda:\, |\lambda|<R\}\setminus \mathrm{int }\,\vartheta(B) \}\big\}$ and let us prove that
\begin{equation*}\label{8a}
 \frac{1}{2\pi i}\oint\limits_{\vartheta_{R}(B)}  \lambda^{-2}  \left(I- \lambda B\right)^{-1} Wfd\lambda\rightarrow \frac{1}{2 \pi i}\int\limits_{\vartheta(B)}  \lambda^{-2}  \left(I- \lambda B\right)^{-1} Wfd\lambda,\;R\rightarrow \infty.
\end{equation*}
Consider a decomposition of the contour $\vartheta_{R}(B)$ on terms   $\tilde{\vartheta}_{R}(B):=\{\lambda:\,|\lambda|=R,\, \theta +\varsigma\leq\mathrm{arg} \lambda \leq 2\pi- \theta -\varsigma  \},\,\hat{\vartheta}_{R}:= \{\lambda:\,|\lambda|=r,\,|\mathrm{arg} \lambda |\leq\theta +\varsigma\}\cup
\{\lambda:\,r<|\lambda|<R,\, \mathrm{arg} \lambda  =\theta +\varsigma\}\cup
\{\lambda:\,r<|\lambda|<R,\, \mathrm{arg} \lambda  =-\theta -\varsigma\}.$
 It is clear that
$$
\frac{1}{2\pi i}\oint\limits_{\vartheta_{R}(B)}  \lambda^{-2}  \left(I- \lambda B\right)^{-1} Wfd\lambda =
\frac{1}{2\pi i}\int\limits_{\tilde{\vartheta}_{R}(B)}  \lambda^{-2}  \left(I- \lambda B\right)^{-1} Wfd\lambda+
 $$
 $$
 +\frac{1}{2\pi i}\int\limits_{\hat{\vartheta}_{R}}  \lambda^{-2}  \left(I- \lambda B\right)^{-1} Wfd\lambda.
$$
Let us show that the first summand tends to zero when $R\rightarrow\infty,$ we have
$$
 \left\|\,\int\limits_{\tilde{\vartheta}_{R}(B)}  \lambda^{-2}  \left(I- \lambda B\right)^{-1} Wfd\lambda\right\|_{\mathfrak{H}}\leq  R^{-2}\int\limits_{\theta +\varsigma}^{2\pi - \theta -\varsigma}     \left\|\left(I\lambda^{-1}-  B\right)^{-1} Wf\right\|_{\mathfrak{H}}d\,  \mathrm{arg} \lambda.
$$
Applying Corollary 3.3,  Theorem 3.2  \cite[p.268]{firstab_lit:kato1980}, we have
$$
\left\|\left(I\lambda^{-1}-  B\right)^{-1} \right\|_{\mathfrak{H}}\leq R/\sin \varsigma ,\,\lambda \in \tilde{\vartheta}_{R}(B).
$$
Substituting this estimate to the last integral,
 we obtain the desired result. Thus, taking into account the fact
$$
\frac{1}{2\pi i}\int\limits_{\hat{\vartheta}_{R}}  \lambda^{-2}  \left(I- \lambda B\right)^{-1} Wfd\lambda\rightarrow \frac{1}{2\pi i}\int\limits_{ \vartheta(B)  }  \lambda^{-2}  \left(I- \lambda B\right)^{-1} Wfd\lambda ,\,R\rightarrow \infty,
$$
we obtain  \eqref{8a}. Having noticed that the following integral can be calculated as a residue at the point zero, i.e.
$$
\frac{1}{2\pi i}\oint\limits_{\vartheta_{R}(B)}  \lambda^{-2}  \left(I- \lambda B\right)^{-1} Wfd\lambda =\lim\limits_{\lambda\rightarrow0}\frac{d (I-\lambda B)^{-1} }{d\lambda }W f= f,
$$
we get
$$
\frac{1}{2 \pi i}\int\limits_{\vartheta(B)}  \lambda^{-2}  \left(I- \lambda B\right)^{-1} Wfd\lambda=f.
$$
Hence $I_{1}(t)\rightarrow f,\,t\rightarrow+0.$ Let us show that $I_{2}(t)=0.$ For this purpose, let us  consider  a  contour
  $\vartheta_{R}(B)=\tilde{\vartheta}_{ R}\cup \hat{\vartheta}_{R},$  where    $\tilde{\vartheta}_{ R}:=\{\lambda:\,|\lambda|=R,\,|\mathrm{arg} \lambda |\leq\theta +\varsigma\}$ and  $\hat{\vartheta}_{R}$ is previously defined. It is clear that
$$
\frac{1}{2\pi i}\oint\limits_{\vartheta_{R}(B)}  \lambda^{-2} e^{-\varphi^{\alpha}(\lambda)t} \left(I+ \lambda B\right)  Wfd\lambda=\frac{1}{2\pi i}\int\limits_{\tilde{\vartheta}_{ R}}  \lambda^{-2} e^{-\varphi^{\alpha}(\lambda)t} \left(I+ \lambda B\right)  Wfd\lambda+
 $$
 $$
+\frac{1}{2\pi i}\int\limits_{\hat{\vartheta}_{R}}  \lambda^{-2} e^{-\varphi^{\alpha}(\lambda)t} \left(I+ \lambda B\right)  Wfd\lambda.
$$
Considering the second term  having   taken  into account  the  definition of the improper integral,   we conclude    that if we show that  there exists such a sequence $\{R_{n}\}_{1}^{\infty},\,R_{n}\uparrow\infty$  that
\begin{equation*}\label{9a}
\frac{1}{2\pi i}\int\limits_{\tilde{\vartheta}_{ R_{n}}}  \lambda^{-2} e^{-\varphi^{\alpha}(\lambda)t} \left(I+ \lambda B\right)  Wfd\lambda\rightarrow 0,\,n\rightarrow \infty,
\end{equation*}
then we obtain
\begin{equation*}\label{10a}
 \frac{1}{2\pi i}\oint\limits_{\vartheta_{R_{n}}(B)}  \lambda^{-2} e^{-\varphi^{\alpha}(\lambda)t} \left(I+ \lambda B\right)  Wfd\lambda\rightarrow \frac{1}{2 \pi i}\int\limits_{\vartheta(B)}  \lambda^{-2} e^{-\varphi^{\alpha}(\lambda)t} \left(I+ \lambda B\right)  Wfd\lambda,\;R\rightarrow \infty.
\end{equation*}
 Using the lemma conditions, we can accomplish the following estimation
 \begin{equation*}\label{11a}
|e^{-\varphi^{\alpha}(\lambda)t}|= e^{- \mathrm{Re}\,\varphi^{\alpha}(\lambda)t}\leq e^{- C|\varphi(\lambda)|^{\alpha}t},\, \lambda \in \tilde{\vartheta}_{ R},
\end{equation*}
where $R$ is sufficiently large.  Using the condition imposed upon the order of the entire function and  applying  the Wieman theorem  (Theorem 30 \S 18 Chapter I \cite{firstab_lit:Eb. Levin}),  we can claim that there exists such a sequence $\{R_{n}\}_{1}^{\infty},\,R_{n}\uparrow \infty$ that
\begin{equation*}
\forall \varepsilon>0,\,\exists N(\varepsilon):\,e^{- C|\varphi(\lambda )|^{\alpha}t}\leq e^{- C m^{\alpha}_{\varphi}(R_{n})t}\leq e^{- C t[M_{\varphi}(R_{n})]^{\alpha\cos \pi \phi-\varepsilon}},\,\lambda \in \tilde{\vartheta}_{ R_{n}},\,n>N(\varepsilon),
\end{equation*}
where $\phi$ is the order of the entire  function $\varphi.$
Using this estimate, we get
$$
 \left\|\,\int\limits_{\tilde{\vartheta}_{ R_{n}}}  \lambda^{-2} e^{-\varphi^{\alpha}(\lambda)t} \left(I+ \lambda B\right)  Wfd\lambda\right\|_{\mathfrak{H}}\leq C e^{- C t[M_{\varphi}(R_{n})]^{\cos \pi \phi-\varepsilon}}\|Wf\|_{\mathfrak{H}} \int\limits_{-\theta-\varsigma}^{\theta+\varsigma}d\xi.
$$
It is clear that if the order $\phi$ less than a half then we obtain the desired chain of reasonings.   Since the operator function under the integral is analytic, then
$$
\oint\limits_{\vartheta_{R_{n}}(B)}  \lambda^{-2} e^{-\varphi^{\alpha}(\lambda)t} \left(I+ \lambda B\right)  Wfd\lambda=0,\,n\in \mathbb{N}.
$$
Using this relation, we  obtain the fact $I_{2}(t)=0.$ The proof is complete.
\end{proof}
\begin{remark}Note that the statement of the lemma is not true if the order equals zero, in this case we cannot apply the Wieman theorem   (more detailed see the proof of the Theorem 30 \S 18 Chapter I \cite{firstab_lit:Eb. Levin}). At the same time the proof can be easily transformed for the case corresponding to a polynomial function. Here, we should note that the reasonings are the same, we  have to impose   conditions upon the polynomial to satisfy the lemma conditions and  establish an estimate analogous to the established by virtue of the Wieman theorem application.   Now assume that $\varphi(z)=c_{0}+c_{1}z+...+c_{n}z^{n},\,z\in \mathbb{C},$ by easy calculations we see that   the condition $$\max\limits_{k=0,1,...,n}(|\mathrm{arg} c_{k}|+k\theta)<\pi/2\alpha,$$ gives us  $|\mathrm{arg }\varphi(z)|<\pi/2\alpha,\,z\in \mathrm{int} \vartheta(B).$ Thus, we have the fulfilment of the estimate
$$
|e^{-\varphi^{\alpha}(\lambda)t}|= e^{- \mathrm{Re}\,\varphi^{\alpha}(\lambda)t}\leq e^{- C|\varphi(\lambda)|^{\alpha}t},\, \lambda \in \tilde{\vartheta}_{ R}.
$$
It can be established easily that $m_{\varphi}(|z|)\rightarrow \infty,\,|z|\rightarrow \infty.$ Combining this fact with the last estimate  and preserving the scheme of the reasonings, we obtain the lemma statement.
\end{remark}

\subsection{Operator function with more subtle asymptotics}

The approach implemented bellow is based on the ordinary properties of operators acting in the Hilbert space.
We denote by $\lambda_{n},e_{n},\;n\in \mathbb{N}$ the eigenvalues and eigenvectors of the sectorial operator $W$ respectively.
Consider the invariant space $\mathfrak{N}$ generated by eigenvectors of the operator, we mentioned above that it is an infinite dimensional space endowed with the stricture of the initial Hilbert space, hence we can consider a restriction of the operator  $R_{W}(\lambda)$ on the space $\mathfrak{N},$ where $\lambda$ does not take values of eigenvalues.
 Using simple reasonings involving properties of the resolvent, Cauchy integral formula, clock-wise direction, e.t.c., we get
$$
\varphi(W)e_{n}= \lim\limits_{t\rightarrow+0}\frac{1}{2 \pi i}\int\limits_{\vartheta(B)}e^{-\varphi (\lambda) t}\varphi (\lambda)R_{W}(\lambda)e_{n} d\lambda    =
 e_{n}\lim\limits_{t\rightarrow+0}\frac{1}{2 \pi i}\int\limits_{\vartheta(B)}e^{-\varphi (\lambda) t}\frac{\varphi (\lambda)}{\lambda_{n}-\lambda} d\lambda=
$$
$$
=e_{n}\lim\limits_{t\rightarrow+0}e^{-\varphi (\lambda_{n}) t}\varphi(\lambda_{n}) =e_{n}\varphi(\lambda_{n}) .
$$
This property can be taken as a concept for by  virtue of such an approach and uniqueness of the decomposition on basis vectors in the Hilbert space, we can represent the operator function  defined on elements of $\mathfrak{N}$ in the form of series on eigenvectors
\begin{equation}\label{5.22}
 \varphi (W)f=\sum\limits_{n=1}^{\infty} e_{n} \varphi(\lambda_{n})f_{n},\,f\in \mathrm{D}_{1}(\varphi),
\end{equation}
where
$$
\mathrm{D}_{1}(\varphi):=\left\{f\in \mathfrak{N}:\; \sum\limits_{n=1}^{\infty}  | \varphi(\lambda_{n})f_{n}|^{2}<\infty\right\}.
$$
Indeed, note that the resolvent $R_{W}(\lambda)$ is defined on $\mathfrak{N}$ and admits the following decomposition
$$
R_{W}(\lambda)f=\sum\limits_{n=1}^{\infty}\frac{f_{n}}{\lambda_{n}-\lambda}e_{n},\,f\in \mathfrak{N},
$$
therefore having substituted the latter relation to the formula of the operator function, we get
$$
\varphi(W)f=  \lim\limits_{t\rightarrow+0}\frac{1}{2 \pi i}\int\limits_{\vartheta(W)}e^{-\varphi (\lambda) t}\varphi (\lambda)\sum\limits_{n=1}^{\infty}\frac{f_{n}e_{n}}{\lambda_{n}-\lambda} d\lambda=
$$
$$
=\lim\limits_{t\rightarrow+0}\frac{1}{2 \pi i}\sum\limits_{n=1}^{\infty} f_{n}e_{n} \int\limits_{\vartheta(W)}\frac{ e^{-\varphi (\lambda) t}\varphi (\lambda)}{\lambda_{n}-\lambda} d\lambda=
 \lim\limits_{t\rightarrow+0}\sum\limits_{n=1}^{\infty}f_{n}e_{n} \!\mathop{\operatorname{res}}\limits_{z=\lambda_{n}}  \{  e^{-\varphi (z) t}\varphi (z)\}=
 $$
 $$
 =\lim\limits_{t\rightarrow+0}\sum\limits_{n=1}^{\infty}f_{n}e_{n}     e^{-\varphi (\lambda_{n}) t}\varphi (\lambda_{n})=\sum\limits_{n=1}^{\infty}f_{n}e_{n}     \varphi (\lambda_{n}),\,f\in \mathrm{D}_{1}(\varphi),
$$
 we obtain \eqref{5.22}.
Here, we ought to explain that we managed to pass to the limit considering the contour integrals by virtue of the growth regularity of the function, a complete scheme of  reasonings is represented in the derivation of formula \eqref{5.5}.
The justification of an opportunity to  integrate  the series termwise  is based upon the fact that the latter series is convergent in the sense of the norm. The passage to the limit when $t$ tends to zero is justified by the same fact.

 Thus, formula \eqref{5.22} gives us    another   definition of an operator function, compare with the one given in \cite{firstab_lit(frac2023)} represented by   formula \eqref{5.3}. It is clear that considering the set $ f(t),\;f\in \mathfrak{N},\;t>0,$ we can expand the domain of definition of the operator function $\varphi$ at the same time the extension remains closed in the mentioned above sense  as one can easily see it follows  from the above.

Suppose $\mathfrak{H}:=\mathfrak{N}$ and let us construct a space $\mathfrak{H}_{+}$ satisfying the condition of compact embedding $\mathfrak{H}\subset\subset \mathfrak{H}_{+}$ and suitable for spreading  the condition H2 upon the operator $\varphi(W).$ For this purpose, define
$$
\mathfrak{H}_{+}:=\left\{f\in  \mathfrak{N}:\;\|f\|^{2}_{\mathfrak{H}_{+}}= \sum\limits_{n=1}^{\infty}  | \varphi(\lambda_{n})||f_{n}|^{2}<\infty\right\},
$$
and let us prove the fact $\mathfrak{H}\subset\subset \mathfrak{H}_{+}.$ The idea of the proof is based on the application of the criterion of compactness in Banach spaces, let us involve the operator $B: \mathfrak{H}\rightarrow \mathfrak{H}$ defined as follows
$$
Bf=\sum\limits_{n=1}^{\infty}  | \varphi(\lambda_{n})|^{-1/2} f_{n} e_{n},
$$
here we are assuming without loss of generality that  $|\varphi(\lambda_{n})|\uparrow \infty.$ Note  that in any case, we can rearrange the sequence in the required way  having imposed  a condition of the growth regularity upon the operator function. Observe that if $\|f\|<K=\mathrm{const},$ then
$$
\|R_{k}Bf\|=\sum\limits_{n=k}^{\infty}  | \varphi(\lambda_{n})|^{-1 } |f_{n}|^{2}\leq\frac{\|f\|^{2}}{|\varphi (\lambda_{k})|} <\frac{K^{2}}{|\varphi (\lambda_{k})|},\,k\in \mathbb{N}.
$$
Therefore, in accordance with the compactness criterion in Banach spaces  the operator $B$ is compact. Now, consider a set bounded in the sense of the norm $\mathfrak{H}_{+},$ we will denote its elements by $f,$ thus in accordance with the above, we have
$$
\sum\limits_{n=1}^{\infty}  | \varphi(\lambda_{n})||f_{n}|^{2}<C.
$$
It is clear that the element $g:=\{|\varphi(\lambda_{n})|^{1/2} f_{n}\}_{1}^{\infty}$ belongs to $\mathfrak{H}$ and the set of elements from $\mathfrak{H}$ corresponding to the bounded set of elements from $\mathfrak{H}_{+}$ is bounded also. This is why the operator $B$ image of the set of elements $g$ is compact, but we have
$ Bg=f.$ The latter relation proves the fact that the set of elements $f$ bounded in the sense of the norm $\mathfrak{H}_{+}$ is a compact set in the sense of the norm $\mathfrak{H}.$ Thus, we have established the fulfilment of condition H1, it is obvious that we can choose a span of $\{e_{n}\}_{1}^{\infty}$ as the mentioned   linear manifold $\mathfrak{M}.$

The verification of the first   relation of  H2 is implemented due to direct application of the Cauchy-Schwarz inequality, we have
$$
|\left(\varphi (W)f,h\right)_{\mathfrak{H}}|=\left| \sum\limits_{n=1}^{\infty}  \varphi(\lambda_{n})f_{n} \overline{h}_{n} \right|\leq  \left(\sum\limits_{n=1}^{\infty}  |\varphi(\lambda_{n})||f_{n}|^{2}\right)^{1/2}  \left( \sum\limits_{n=1}^{\infty}  |\varphi(\lambda_{n})||h_{n}|^{2}\right)^{1/2}
$$
Let us verify the fulfilment of the  second condition, here we need impose a sectorial condition upon the analytic  function $\varphi$ i.e. it should preserve in the open right-half plain the closed sector belonging to the latter, then we have

$$
  \sum\limits_{n=1}^{\infty}    |\varphi(\lambda_{n})| |f_{n}|^{2}\leq  \sec \psi\cdot \mathrm{Re}\sum\limits_{n=1}^{\infty}    \varphi(\lambda_{n}) |f_{n}|^{2} ,
$$
where $\psi$ is the semi-angle of the sector contenting  the image of the analytic function $\varphi.$  Thus, the condition  H2 is fulfilled. Observe that
$$
 \varphi^{\ast}(W)f=\sum\limits_{n=1}^{\infty}   \overline{\varphi(\lambda_{n})}f_{n}e_{n},\;f\in \mathrm{D}_{1}(\varphi),
$$
it follows easily from the representation of the inner product in terms of the Fourier coefficients. Therefore, by virtue of the absolute convergence of the series, we get
$$
\mathfrak{Re}\varphi(W)f=\sum\limits_{n=1}^{\infty}  \mathrm{Re}\varphi(\lambda_{n})f_{n}e_{n},\;f\in \mathrm{D}_{1}(\varphi).
$$
 It is clear that
$$
\mathfrak{Re}\varphi(W)e_{n}=\mathrm{Re}\varphi(\lambda_{n}) e_{n},\;n\in \mathbb{N}.
$$
The fact   that there does not exist additional eigenvalues of the operator  $\mathfrak{Re}\varphi(W)$ follows from the fact that $\{e_{n}\}_{1}^{\infty}$ forms a basis in $\mathfrak{N}$ and can be established easily by implementing the scheme of reasonings applied above to the similar cases. Thus, we obtain
\begin{equation}\label{5.23}
\lambda_{n}\{\mathfrak{Re}\varphi(W)\}=\mathrm{Re}\varphi(\lambda_{n}  ),\,n\in \mathbb{N}.
\end{equation}

Note that condition \eqref{5.1} plays the distinguished role in the refinement  of the Lidskii results \cite{firstab_lit:1kukushkin2021} since  it guaranties   the equality of the convergence exponent and   the order of summation in the Abell-Lidskii sense. At the same time we can chose the sequence of contours, in the integral construction, of the power type. It is rather reasonable to expect that we are highly motivated to produce a concrete example of the operator satisfying the condition \eqref{5.1} for if we found it then it would stress the significance and novelty of the papers
 \cite{firstab_lit:1kukushkin2018,firstab_lit(arXiv non-self)kukushkin2018,kukushkin2019,firstab_lit:1kukushkin2021,kukushkin2021a,firstab_lit:2kukushkin2022,firstab_lit(axi2022),firstab_lit(frac2023)}.
Having been inspired  by the idea, by virtue  of the comparison test for series convergence, we can use the function considered in Example \ref{E5.1} as an indicator to find the desired operator function. Thus,  to satisfy condition $\eqref{5.1},$ having taken into account relation \eqref{5.23}, we can impose the following condition: for sufficiently large numbers $n\in \mathbb{N},$  the following   relation holds
\begin{equation}\label{5.24}
  C_{1}<\frac{(n \ln n\cdot \ln\ln n)^{\kappa}}{\mathrm{Re}  \varphi(\lambda_{n}  ) } <C_{2},\;\kappa>0,
\end{equation}
certainly we need assume that $\varphi$ preserves the sector containing the numerical range of values in the right half-plane.

Consider   a function   $\psi(z)=  z^{\xi}\ln z \cdot \ln\ln z ,\,0<\xi\leq1$ in the sector $|\arg z|\leq \theta,$ where for the simplicity of reasonings the   branch of the power function has been chosen so that it acts onto the sector and  we have chosen the branch of the logarithmic function  corresponding to the value   $\phi:=\arg z .$
Let us find the real and imaginary parts of the function $\psi(z),$ we have
$$
\psi(z)=z^{\xi}\ln z \cdot \ln\ln z= |z|^{\xi}e^{i\xi\phi} \left(\ln|z|+i\phi\right)\left(a+i \arctan \frac{\phi}{\ln|z|}\right),
$$
where, we denote $a:=\ln|\ln|z|+i\phi| .$ Thus, separating the real and imaginary parts of the function $\psi(z),$ we have
$$
|z|^{\xi}e^{i\xi\phi}\left\{a\ln|z|-\phi \arctan \frac{\phi}{\ln|z|}+i\left( a\phi+\ln|z|\arctan \frac{\phi}{\ln|z|}  \right) \right\}=
$$
$$
|z|^{\xi}\cos \xi\phi\left(a\ln|z|-\phi \arctan \frac{\phi}{\ln|z|} \right)-|z|^{\xi}\sin \xi\phi\left( a\phi+\ln|z|\arctan \frac{\phi}{\ln|z|} \right)+
$$
$$
+i\left\{|z|^{\xi}\sin \xi\phi\left( a\ln|z|-\phi \arctan \frac{\phi}{\ln|z|}\right)+ |z|^{\xi}\cos \xi\phi\left( a\phi+\ln|z|\arctan \frac{\phi}{\ln|z|} \right)\right\}.
$$
It gives us
$$
 \frac{\mathrm{Im} \,\psi(z) }{ \mathrm{Re }\,\psi(z)}\rightarrow \tan \xi \phi,\,|z|\rightarrow\infty,\;
$$
from what follows $\mathrm{arg}\, \psi(z)\rightarrow \xi \mathrm{arg} z,\;|z|\rightarrow\infty,$
and   leads us to the following estimate
$$
\frac{\mathrm{Im} \,\psi(z) }{ \mathrm{Re }\,\psi(z)} <
\frac{\tan \xi\theta\left( a\ln|z|-\phi \arctan  \ln^{-1}\!|z|^{1/\phi}  \right)+  \left( a\phi+\ln|z|\arctan  \ln^{-1}\!|z|^{1/\phi} \right)}{ \left(a\ln|z|-\phi \arctan  \ln^{-1}\!|z|^{1/\phi} \right)-\tan \xi\theta\left( a\phi+\ln|z|\arctan  \ln^{-1}\!|z|^{1/\phi} \right)}.
$$
The latter   gives us an opportunity to claim that for arbitrary $\varepsilon>0,$ there exists $R(\varepsilon)$ such that the following estimate holds
$$
 \frac{\mathrm{Im} \,\psi(z) }{ \mathrm{Re }\,\psi(z)}< (1+\varepsilon)\tan \xi\theta,\,|z|>R(\varepsilon).
$$
Apparently, we can claim that the function $\psi(z)$ nearly preserves the sector $|\arg z|\leq \theta,$ what is completely sufficient for our reasonings for we are dealing with the  neighborhood  of the infinitely distant point.   Let us calculate the absolute value, we have
$$
|\psi(z)|^{2}=  |z|^{2\xi}  \left|\left(\ln|z|+i\phi\right)\right|^{2}\cdot\left|\left(a+i \arctan \frac{\phi}{\ln|z|}\right)\right|^{2}
 = |z|^{2\xi}   (\ln^{2}|z|+\phi^{2})(a ^{2}+\arctan^{2}  \ln^{-1}\!|z|^{1/\phi}),
$$
  the latter relation establishes the growth regularity of the function $\psi(z).$
Note that   we obtain the following formulas
$$
\mathrm{Re}\,\psi^{\kappa} (z)=|z|^{\xi\kappa}    (\ln^{2}|z|+\phi^{2})^{\kappa/2}(\ln^{2}|\ln|z|+i\phi|+\arctan^{2}  \ln^{-1}\!|z|^{1/\phi})^{\kappa/2}
 \cos(\kappa\arg  \psi ) ,
$$
$$
\mathrm{Im}\,\psi^{\kappa} (z)=|z|^{\xi\kappa}   (\ln^{2}|z|+\phi^{2})^{\kappa/2}(\ln^{2}|\ln|z|+i\phi|+\arctan^{2}  \ln^{-1}\!|z|^{1/\phi})^{\kappa/2}
 \sin(\kappa\arg  \psi ),
$$
Here we should note the distinguished fact proved   above  $\arg\psi (z)\rightarrow \xi\arg z,\,|z|\rightarrow\infty$ and   having noticed that
$$
\psi(|z|)=  |z|^{\xi}\ln |z| \cdot \ln\ln |z|,
$$
we get
$$
   \mathrm{Re}\,\psi^{\kappa} (z) \sim  \psi^{\kappa}(|z|)   \cos(\xi\kappa \arg  z),\;\mathrm{Im}\,\psi^{\kappa} (z) \sim  \psi^{\kappa}(|z|)   \sin(\xi\kappa \arg  z),\,|z|\rightarrow\infty.
$$
Therefore the values of the function $\psi^{\kappa} (z)$ belong to the sector $\mathfrak{L}_{0}(\xi\!\kappa\, \theta+\varepsilon),$ where $\varepsilon$ is arbitrary small, for sufficiently large values of $|z|.$
Here, we should make a short narrative digression and remind that we pursue a rather particular aim to produce an example of an operator so we are free in some sense to choose an operator as an object for our needs. At the same time the given above reasonings    origin  from the fundamental scheme  and as a result allow to construct a fundamental theory. Define the function $\varphi(z):=\psi^{\kappa}(z)$ and consider the operator $W$  such that  $|\lambda_{n}(W )|\asymp n^{1/\xi}.$
Eventually, the given above reasonings lead us to the conclusion that     relation \eqref{5.24} holds and we obtain the desired operator with more subtle asymptotics of the real  component than the asymptotics of the power type. We are pleased to represent it in the special  forms due to the ordinary  properties of the exponential function
$$
\varphi(W)f=\lim\limits_{t\rightarrow+0}\frac{1}{2 \pi i}\int\limits_{\vartheta(W)}
 \left(\ln \lambda\right)^{ -t \varphi(\lambda)/\ln\ln\lambda }
\varphi(\lambda) \,R_{W}(\lambda)fd\lambda =
$$
$$
=\lim\limits_{t\rightarrow+0}\frac{1}{2 \pi i}\int\limits_{\vartheta(W)} \lambda^{ -t \varphi(\lambda)/   \ln  \lambda  }
\varphi(\lambda) \,R_{W}(\lambda)fd\lambda,\,f\in \mathrm{D}_{1}(\varphi).
$$
Consider a set
$$
\mathrm{D}_{1}(W):=\left\{f\in \mathfrak{N}:\; \sum\limits_{n=1}^{\infty}  |  \lambda_{n} f_{n}|^{2}<\infty\right\}.
$$
Note that  a simple comparison of the asymptotics  gives us the fact
 $\mathrm{D}_{1}(W)\subset\mathrm{D}_{1}(\varphi).$   Due to the sectorial property of the operator $W,$ we have $|\lambda_{n}(W)|\asymp \mathrm{Re} \lambda_{n}(W).$ Consider a restriction $W_{1}$ of the operator $W$ to the set $\mathrm{D}_{1}(W)$ taking into account the fact
$
 \mathrm{Re} \lambda_{n}(W_{1})=\lambda_{n}(\mathfrak{Re} W_{1} )  ,\,n\in \mathbb{N}
$
which can be proved due  to the analogous  reasonings   corresponding to   relation \eqref{5.23}, we get
 $$
 \lambda_{n}(\mathfrak{Re} W_{1} ) \asymp n^{1/\xi}.
 $$
  Note that the inverse chain of reasonings is obviously correct,  thus  we obtain a description in terms of asymptotics of the  real component  eigenvalues.

Now,  observe  benefits and  disadvantages of the idea to involve the restriction of the operators on the space $\mathfrak{N}.$ Apparently, the main disadvantage is the requirement in accordance with which we need deal with the expansion of the adjoint operator, since we have  $W_{1}\subset W \Rightarrow W^{\ast}\subset W^{\ast}_{1}.$ This follows that the orders of $\mathfrak{Re} W$ and $\mathfrak{Re} W_{1}$ may differentiate, what brings us the essential inconvenience if we want to provide a description in terms of the class $\mathfrak{S}^{\star}_{\xi}.$\\

\section[ \hfil
  Domain of definition of the operator function]{ Domain of definition of the operator function including well-known operators}
\noindent {\bf 1.} In this paragraph, we preserve the notation
$$
  \lambda_{n}(\mathfrak{Re} W_{1} ) \asymp n^{1/\xi}.
 $$
However,   let us consider a remarkably showing case  $\xi=1$ corresponding  to the so-called quasi-trace  operator class $\mathfrak{S}^{\star}_{1},$ here we should recall that a singular  number  of the normal operator  coincides  with  its  eigenvalue absolute value. Consider the  operator
$$
L:=-a_{2}\Delta+a_{0},
$$
 with a constant  complex coefficients   acting in $L_{2}(\Omega),$ here $\Omega\subset \mathbb{E}^{2}$ is a bounded domain with a sufficiently smooth boundary. It is clear that the operator is normal, hence for an arbitrary eigenvector of the operator, i.e  $Le=\lambda e,$ we have $Le=\overline{\lambda}\, e$ and the system of the eigenvectors is compleat in $\overline{\mathrm{R}(L)}$  it is well-known fact that under the conditions  imposed upon $\Omega,$ we have the following relation for   $\xi =n/m,$ where $m$ is the highest derivative and $n$ is  a dimension  of the Euclidian space. Thus,    $\xi=1.$ It is not hard to prove that $\mathrm{N}(L^{\ast})=0,$ since we have
 $$
   (Lf,f)_{L_{2}(\Omega)} =a_{2}\|f\|^{2}_{H_{0}^{1}(\Omega)}+a_{0}\|f\|^{2}_{L_{2}(\Omega)}.
 $$
Hence, by virtue of the well-known statement of the operator theory - orthogonal decomposition of the Hilbert space
$$
L_{2}(\Omega)=\mathrm{N}(L^{\ast})\oplus \overline{\mathrm{R}(L)},
$$
 we get $\overline{\mathrm{R}(L)}=L_{2}(\Omega).$ Using this fact, we can claim that $\mathfrak{Re}L$ and $L$ have the same eigenvectors. Indeed, it is clear due to the normal property of the operator that the system of eigenvectors of $L$ is a subsystem of eigenvectors of $\mathfrak{Re}L.$  The coincidence can established easily if we recall that
 the eigenvectors corresponding to different eigenvalues of the selfadjoint operator is orthogonal, indeed for an arbitrary operator $S=S^{\ast},$ we have
 $$
 \lambda_{1}(e_{1},e_{2})_{\mathfrak{H}}=(Se_{1},e_{2})_{\mathfrak{H}}=(e_{1},Se_{2})_{\mathfrak{H}}=\lambda_{2}(e_{1},e_{2})_{\mathfrak{H}},
 $$
since $\lambda_{1}\neq \lambda_{2},$ then  $(e_{1},e_{2})_{\mathfrak{H}}=0.$ Thus, if we assume that there exists one more eigenvector of the operator  $\mathfrak{Re}L$ then we obtain the fact that it is a zero element since it is orthogonal to the compleat system but the latter is impossible. The proved fact of eigenvectors coincidence allows us to claim that
$$
  \lambda_{n}(  \mathfrak{Re}L ) \asymp n^{1/\xi},\,\Rightarrow |\lambda_{n}(  L )| \asymp n^{1/\xi}.
 $$
  Hence we can implement the above scheme of reasonings and consider the operator function $\varphi(L),$ where $\varphi(z)=z^{\xi\kappa}\{\ln z \cdot \ln\ln z\}^{\kappa}$ for which condition \eqref{5.24} and as a result condition \eqref{5.1} holds. It is remarkable that   despite of  the fact that the hypotheses $\mathrm{H}1,\mathrm{H}2$ hold for the operator $L$ in the natural  way, the verification   is left to the reader, we can use the benefits of the scheme introduced above to construct the required pair of Hilbert spaces.

In addition, we want  to represent a concrete domain $\Omega:=\{x_{j}\in [0,\pi],\,j=1,2\}$ that gives us an opportunity to construct   a concrete compleat eigenvectors system of the operator $L.$
Consider the following functions
$$
e_{\bar{l}} = \sin l_{1}x_{1}\cdot\sin l_{2}x_{2},\;\bar{l}:=\{l_{1},l_{2}\},\;l_{1},l_{2}\in \mathbb{N}.
$$
It is clear that
$$
 L e_{\bar{l}} =\lambda_{\bar{l}}\,e_{\bar{l}},\;\lambda_{\bar{l}} = a_{1} (l^{2}_{1}+l^{2}_{2})+a_{0}.
$$
Let us show that the system
$
\left\{e_{\bar{l}} \right\}
$
is complete in the Hilbert space  $L_{2}(\Omega),$ we will show it if we prove that the element that is orthogonal to every element of the system is a zero.
Assume that
$$
 \int\limits_{0}^{\pi}\sin l_{1}x_{1}dx_{1}\int\limits_{0}^{\pi} \sin l_{2}x_{2}f(x_{1},x_{2})dx_{2} =  (e_{\bar{l}},f)_{L_{2}(\Omega)}=0.
$$
In accordance with the fact that the system $\{\sin l x\}_{1}^{\infty}$ is a compleat system in $L_{2}(0,\pi),$ we conclude that
$$
 \int\limits_{0}^{\pi} \sin l_{2}x_{2} f(x_{1},x_{2})dx_{2} =0.
$$
Having repeated the same reasonings, we obtain the desired result.\\

\noindent {\bf 2.} The next case, within  the scale  of most important ones, appeals to the so-called quasi-Hilbert-Schmidt class $\mathfrak{S}^{\star}_{2}.$  In this regard, let us consider the Sturm-Liouville operator $L,$  where the corresponding Euclidian space is one-dimensional. In order to make a clear demonstration let us consider a simplest case corresponding to the operator
$$
Lu:=-au'',\; u(0)=u(\pi  )=0,\;a\in \mathbb{C}
$$
acting in $L_{2}(I),\,I:=(0,\pi ).$ Recall that the general solution of the homogeneous equation $-u''-\lambda u=0,\,\lambda\in \mathbb{R}^{+}$ is given by the following formula
$
u=C_{1}\sin \sqrt{\lambda}x+C_{2}\cos\sqrt{\lambda}x.
$
Using the initial conditions, we find $C_{2}=0,\,\sin\sqrt{\lambda}\,\pi=0.$ Hence $\lambda_{n}(L)=an^{2},\, e_{n}(x)=\sin n  x,\;n\in \mathbb{N}$ are non-zero eigenvalues and  eigenvectors respectively. Note that the operator is normal and  the closure of the linear  span of the functions  $\sin n  x,\;n\in \mathbb{N}$ gives us the Hilbert space $L_{2}(I).$ Thus, we obtain $\xi=1/2$ since
$$
|\lambda_{n}(  L )| \asymp n^{2},
$$
  and    can implement the above scheme of reasonings and considering the operator function  $\varphi(z)=z^{\xi\kappa}\{\ln z \cdot \ln\ln z\}^{\kappa}$ for which condition \eqref{5.24} and as a result condition \eqref{5.1} holds. The verification  of  the fact that the hypotheses $\mathrm{H}1,\mathrm{H}2$ hold for the operator $L$ is too simple and left to the reader. In order to fulfill an exercise  we can also use the benefits of the scheme introduced above to construct the required pair of Hilbert spaces $\mathfrak{H}_{+}\subset\subset \mathfrak{H}.$

     The given above theory tells us that the operator function $\varphi(z)=z^{\xi\kappa}\{\ln z \cdot \ln\ln z\}^{\kappa}$ is defined on the $\mathrm{D}_{1}(L)$ which, in this case, coincides with the functions having a fourier coefficients with a  sufficiently large decrease so that the following series converges in the sense of $L_{2}(I)$ norm, i.e.
$$
\sum\limits_{n=1}^{\infty}n^{2}f_{n}\sin n  x\in L_{2}(I).
$$
In addition, we should add that we can consider an arbitrary non-selfadjoint differential and fractional differential operators assuming that the functional space is defined on the bounded  domain of an Euclidian space with a sufficiently   smooth boundary  (regular operators). In most of such  cases the minimax principle can be applied and we can obtain the asyptotics of the eigenvalues of the Hermitian real component, here we can referee a detailed description represented in the monograph by Rozenblyum  G.V. \cite{firstab_lit:Rosenblum}. The case corresponding to an unbounded domain (irregular operator) is also possible for study, in this regard  the Fefferman concept   covers such  problems \cite[p.47]{firstab_lit:Rosenblum}.
 The given above  theoretical results  can be applied to   the operator class  and we can construct in each case a corresponding operator function representing to the reader    the example of an operator with a more subtle asymptotics of the Hermitian real component eigenvalues than one of the power type. Bellow, we represent well-known non-selfadjoint operators which can be considered as an operator argument.\\

\subsection{ Kolmogorov operator}

The relevance of the considered operator is justified by resent results by Goldstein et al.   \cite{firstab_lit:Goldstein} where the following operator has been undergone to the rapt attention
$$
L=\Delta+\frac{\nabla\rho}{\rho}\cdot\nabla,
$$
here $\rho$ is a probability density on $\mathbb{R}^{N}$ satisfying $\rho\in C^{1+\alpha}_{lok}(\mathbb{R}^{N})$
  for some
$\alpha\in (0, 1),\; \rho(x) > 0$ for all $x\in \mathbb{R}^{N}.$

Apparently,  the results \cite{firstab_lit(arXiv non-self)kukushkin2018}, \cite{kukushkin2021a}, \cite{firstab_lit:1kukushkin2021} can be applied to the operator after an insignificant modification. A couple of words on the difficulties appearing while we study the operator composition.  Superficially, the problem  looks pretty well  but it is not so for the inverse operator (one need prove that it is a resolvent)  is a composition of an unbounded operator and a resolvent of the operator $W,$  indeed  since $R_{W}W=I,$ then formally, we have
$
L^{-1}f= R_{W}\rho f.
$
Most likely,   the general theory created in the papers \cite{firstab_lit(arXiv non-self)kukushkin2018}, \cite{kukushkin2021a} can be adopted to some operator composition but it is  a tremendous work. Instead of that, in the paper  \cite{firstab_lit(frac2023)}   we succeed   to  find a suitable pair of Hilbert spaces what allows us to apply theoretical results.\\

 \subsection{ The linear combination of the second order differential operator and the   Kipriyanov operator}

  Consider a linear combination of the uniformly elliptic operator, which is written in the divergence form, and
  a composition of a   fractional integro-differential  operator, where the fractional  differential operator is understood as the adjoint  operator  regarding  the Kipriyanov operator  (see  \cite{kukushkin2019,firstab_lit:kipriyanov1960,firstab_lit:1kipriyanov1960})
\begin{equation*}
 L  :=-  \mathcal{T}  \, +\mathfrak{I}^{\sigma}_{ 0+}\phi\, \mathfrak{D}  ^{ \beta  }_{d-},
\; \sigma\in[0,1) ,
 $$
 $$
   \mathrm{D}( L )  =H^{2}(\Omega)\cap H^{1}_{0}( \Omega ),
  \end{equation*}
where
$\,\mathcal{T}:=D_{j} ( a^{ij} D_{i}\cdot),\,i,j=1,2,...,m,$
under    the following  assumptions regarding        coefficients
\begin{equation} \label{12}
     a^{ij}(Q) \in C^{2}(\bar{\Omega}),\, \mathrm{Re} a^{ij}\xi _{i}  \xi _{j}  \geq   \gamma_{a}  |\xi|^{2} ,\,  \gamma_{a}  >0,\,\mathrm{Im }a^{ij}=0 \;(m\geq2),\,
 \phi\in L_{\infty}(\Omega).
\end{equation}
 Note that in the one-dimensional case, the operator $\mathfrak{I}^{\sigma }_{ 0+} \phi\, \mathfrak{D}  ^{ \beta  }_{d-}$ is reduced to   a  weighted fractional integro-differential operator  composition, which was studied properly  by many researchers
    \cite{firstab_lit:2Dim-Kir,firstab_lit:15Erdelyi,firstab_lit:9McBride,firstab_lit:nakh2003}, more detailed historical review  see  in \cite[p.175]{firstab_lit:samko1987}.        \\

\subsection{ The linear combination of the second order differential operator and the  Riesz potential}

Consider a   space $L_{2}(\Omega),\,\Omega:=(-\infty,\infty)$    and the Riesz potential
$$
I^{\beta}f(x)=B_{\beta}\int\limits_{-\infty}^{\infty}f (s)|s-x|^{\beta-1} ds,\,B_{\beta}=\frac{1}{2\Gamma(\beta)  \cos  (\beta \pi / 2)   },\,\beta\in (0,1),
$$
where $f$ is in $L_{p}(\Omega),\,1\leq p<1/\beta.$
It is  obvious that
$
I^{\beta}f= B_{\beta}\Gamma(\beta) (I^{\beta}_{+}f+I^{\beta}_{-}f),
$
where
$$
I^{\beta}_{\pm}f(x)=\frac{1}{\Gamma(\beta)}\int\limits_{0}^{\infty}f (s\mp x) s ^{\beta-1} ds,
$$
these operators are known as fractional integrals on the  whole  real axis   (see \cite[p.94]{firstab_lit:samko1987}). Assume that the following  condition holds
 $ \sigma/2 + 3/4<\beta<1 ,$ where $\sigma$ is a non-negative  constant.
 Following the idea of the   monograph \cite[p.176]{firstab_lit:samko1987},
 consider a sum of a differential operator and  a composition of    fractional integro-differential operators
\begin{equation*}
 W   :=  D^{2} a  D^{2}  +   I^{\sigma}_{+}\,\xi \,I^{2(1-\beta)}D^{2}+\delta I,\;\mathrm{D}(W)=C^{\infty}_{0}(\Omega),
\end{equation*}
where
$
\xi(x)\in L_{\infty}(\Omega),\, a(x)\in L_{\infty}(\Omega)\cap C^{ 2  }( \Omega ),\, \mathrm{Re}\,a(x) >\gamma_{a}(1+|x|)^{5}.
$

\subsection{ The perturbation of the  difference operator with the artificially constructed operator }

Consider a   space $L_{2}(\Omega),\,\Omega:=(-\infty,\infty),$   define a family of operators
$$
T_{t}f(x):=e^{-c t}\sum\limits_{k=0}^{\infty}\frac{(c \,t)^{k}}{k!}f(x-d\mu),\,f\in L_{2}(\Omega),\;c,d>0,\; t\geq0,
$$
where convergence is understood in the sense of $L_{2}(\Omega)$ norm. In accordance with the Lemma 6 \cite{kukushkin2021a}, we know that  $T_{t}$ is a $C_{0}$ semigroup of contractions, the corresponding  infinitesimal generator and its adjoint operator are defined by the following expressions
 $$
Yf(x)=c[f(x)-f(x-d)],\,Y^{\ast}f(x)=c[f(x)-f(x+d)],\,f\in L_{2}(\Omega).
$$
Let us find a representation for fractional powers of the operator $Y.$ Using formula  (45) \cite{kukushkin2021a}, we get
$$
   Y^{\beta}f=\sum\limits_{k=0}^{\infty}M_{k}f(x-kd), \,f\in L_{2}(\Omega),
   \,M_{k}= -\frac{\beta\Gamma(k -\beta)}{k!\Gamma(1 -\beta)}c^{\,\beta},\,\beta\in (0,1).
   $$
Consider the operator
$$
L:= Y^{\ast}\!a Y+b Y^{\beta}+ Q^{\ast}N Q,
$$
where   $a,b\in L_{\infty}(\Omega),\;Q$ is a   closed operator acting in $L_{2}(\Omega),\,Q^{-1}\in \mathfrak{S}_{\!\infty}(L_{2}),$ the operator $N$ is strictly accretive, bounded, $\mathrm{R}(Q)\subset \mathrm{D}(N).$ Note that   Theorem 14 \cite{kukushkin2021a}  gives us a compleat substantiation of the  given above theoretical approach application.

\chapter{Evolution equations in the abstract Hilbert space}

  This chapter is devoted to a  method allowing us to solve abstract evolution equations with the operator function in the second term. In this regard, we involve  a special technique providing a proof of convergence of   contour integrals,   a similar  scheme of reasonings was implemented in the papers \cite{firstab_lit:2kukushkin2022},\cite{firstab_lit(axi2022)}. At the same time, the behavior of the entire function in the neighborhood   of the point at infinity is the main obstacle to realize the scheme of reasonings. Thus, to overcome difficulties related to evaluation of improper contour integrals, we need   study more  comprehensive   innate properties of the  entire function. The property of the growth regularity is a key for the desired estimates for the involved integral constructions. However,   the lack of the latter approach is that the condition of the   growth regularity is supposed to be satisfied  within  the complex plane  except for the exceptional set of circles, the location of which in general cannot be described. On the other hand, we need not use the subtle   estimates for the Fredholm determinant established in \cite{firstab_lit:1Lidskii} for we  can be  completely satisfied  by  the application of the  Wieman  theorem in accordance with which we can    obtain the required estimate on the boundary of circle. We represent  a suitable algebraic reasonings  allowing   to involve   a fractional derivative in the first term. The  idea  gives an opportunity  to  to reformulate in the abstract form   many results in the framework of the theory of fractional differential equations to say nothing on previously unsolved problems.\\

\subsection{Preliminaries}

Denote by $\mathfrak{H}$   the abstract separable Hilbert space and    consider an invertible operator $B: \mathfrak{H}\rightarrow \mathfrak{H}$  with a dense range, we  use a notation $W:=B^{-1}.$ Note that such   agreements are justified by the significance of the operator with a compact resolvent, the detailed information on which spectral properties can be found in the papers   cited in the introduction section. Although the reasonings represented bellow cover the case of a compact operator $B: \mathfrak{H}\rightarrow \mathfrak{H}$ such that
$
 \Theta(B) \subset \mathfrak{L}_{0}(\theta_{0},\theta_{1} ),
$
we assume that
$
 \Theta(B) \subset \mathfrak{L}_{0}(\theta ),
$
and   put the following contour   in correspondence to the operator
\begin{equation*}
\vartheta(B):=\left\{\lambda:\;|\lambda|=r>0,\, -\theta \leq\mathrm{arg} \lambda \leq \theta \right\}\cup\left\{\lambda:\;|\lambda|>r,\;
  \mathrm{arg} \lambda =-\theta ,\,\mathrm{arg} \lambda =\theta \right\},
\end{equation*}
where   the number $r$ is chosen so that the operator  $ (I-\lambda B)^{-1} $ is regular within the corresponding closed circle, more detailed see Chapter \ref{Ch4}.
Remind that the operator function in general case was defined by   following expression in Chapter \ref{Ch5}.
$$
 \varphi(W)\!\!\int\limits_{\vartheta(B)}e^{-\varphi^{\alpha}(\lambda) t}B(I-\lambda B)^{-1}fd\lambda=\int\limits_{\vartheta(B)}\varphi(\lambda) e^{-\varphi^{\alpha}(\lambda)  t} B(I-\lambda B)^{-1}fd\lambda,\,f\in \mathrm{D}(\varphi),\,\alpha>0,
$$
where $\mathrm{D}(\varphi)$ is a subset of the Hilbert space on which the last integral construction is defined.
The operator $W$ is called the operator argument. In this section we consider a sectorial  operator operator argument $\Theta(W)\subset \mathfrak{L}_{0}(\theta)$ and suppose  that  the  function of the complex variable $\varphi$  maps the sector $\mathfrak{L}_{0}(\theta )$ into the sector $\mathfrak{L}_{0}(\varpi),\,\varpi<\pi/2\alpha.$ Such an assumption can be explained by technicalities guaranteeing  convergence of the improper integral  at the same time we produce a rather wide class of functions satisfying the condition.

Consider a  function of a complex variable     $\varphi$ that can be represented by its Laurent  series about the point zero.  Consider a formal construction
\begin{equation*}
I:=\sum\limits_{n=-\infty}^{\infty}  c_{n}  W^{n},
\end{equation*}
 where $c_{n}$ are the   coefficients corresponding to the function $\varphi.$ In  Chapter \ref{Ch5}
we established  conditions under which being imposed the latter series of operators   converges on some elements of the Hilbert space $\mathfrak{H},$ moreover  $I=\varphi(W),$      thus we have come to another definition of the operator function.

 In this chapter, we  consider      element-functions of the Hilbert space  $u:\mathbb{R}_{+}\rightarrow \mathfrak{H},\,u:=u(t),\,t\geq0$    assuming  that if $u$ belongs to $\mathfrak{H}$    then the fact  holds for all values of the variable $t.$   We understand such operations as differentiation and integration in the generalized sense that is caused by the topology of the Hilbert space $\mathfrak{H}.$ The derivative is understood as a    limit
$$
  \frac{u(t+\Delta t)-u(t)}{\Delta t}\stackrel{\mathfrak{H}}{ \longrightarrow}\frac{du}{dt} ,\,\Delta t\rightarrow 0.
$$
Let $t\in  J:=[a,b],\,0< a <b<\infty.$ The following integral is understood in the Riemann  sense as a limit of partial sums
\begin{equation*}
\sum\limits_{i=0}^{n}u(\xi_{i})\Delta t_{i}  \stackrel{\mathfrak{H}}{ \longrightarrow}  \int\limits_{ J }u(t)dt,\,\zeta\rightarrow 0,
\end{equation*}
where $(a=t_{0}<t_{1}<...<t_{n}=b)$ is an arbitrary splitting of the segment $ J ,\;\zeta:=\max\limits_{i}(t_{i+1}-t_{i}),\;\xi_{i}$ is an arbitrary point belonging to $[t_{i},t_{i+1}].$
The sufficient condition of the last integral existence is a continuous property (see\cite[p.248]{firstab_lit:Krasnoselskii M.A.}), i.e.
$
u(t)\stackrel{\mathfrak{H}}{ \longrightarrow}u(t_{0}),\,t\rightarrow t_{0},\;\forall t_{0}\in  J.
$
The improper integral is understood as a limit
\begin{equation*}
 \int\limits_{a}^{b}u(t)dt\stackrel{\mathfrak{H}}{ \longrightarrow} \int\limits_{a}^{c}u(t)dt,\,b\rightarrow c,\,c\in  [0,\infty].
\end{equation*}
Combining the operations we can consider a generalized  fractional derivative
  in the Riemann-Liouville sense (see \cite{firstab_lit:samko1987}),     in the formal form, we have
$$
   \mathfrak{D}^{1/\alpha}_{-}f(t):=-\frac{1}{\Gamma(1-1/\alpha)}\frac{d}{d t}\int\limits_{0}^{\infty}f(t+x)x^{-1/\alpha}dx,\;\alpha\geq1,
$$
here we should remark that
$$
\mathfrak{D}^{0}_{-}f(t)=-f(t),\;\mathfrak{D}^{1}_{-}f(t)= -\frac{df(t)}{dt}.
$$

In this chapter, under various assumptions regarding the operator function,  we    study   a Cauchy problem for the evolution equation of order $\alpha\geq1$  in the abstract Hilbert space
\begin{equation}\label{6.1}
   \mathfrak{D}^{1/\alpha}_{-}  u(t)- \varphi(W) u(t)=0 ,\;u(0)=f\in \mathrm{D}( \varphi ).
\end{equation}
where  $W$ is a closeable linear densely defined  sectorial operator acting in $\mathfrak{H}.$

Let $J:=(a,b)\subset \mathbb{R},\,\Omega:=[0,\infty),$ consider  functions  $u(t,x),\,t\in \Omega,\,x\in \bar{J}.$ In accordance with the   above, we   consider functional spaces with respect to the variable $x$ and we will assume that if $u$ belongs to a functional space then the fact  holds for all values of the variable $t.$ In the considered case of the functional  Hilbert space the defined above abstract fractional differential operator is reduced to the concept of
  Riemann-Liouville fractional integro-differential operators   acting in $L_{2}(J)$ with respect to the variable $x,$ see \cite{firstab_lit:samko1987}
 $$
 D^{\psi}_{a+}f(x):=\left(\frac{d}{dx}\right)^{[\psi]+1} \!\!\!\! I^{1-\{\psi\}}_{a+}\varphi   (x),\,\psi>0,\, D^{-\psi}_{a+}:=I^{\psi}_{a+},
 $$
 $$
    I^{\psi}_{a+}\varphi  (x):=\frac{1}{\Gamma(\psi)}\int\limits_{a}^{x}(x-t)^{\psi-1}\varphi(t)dt,\; \varphi \in L_{1}(J),\,\psi>0.
 $$

Along conditions  H1,H2 considered in detail in Chapters \ref{Ch2},\ref{Ch3}, we involve the additional condition\\

\noindent  $( \mathrm{H3} )  \;  |\mathrm{Im}( W f,g)_{\mathfrak{H}}|\! \leq \! C \|f\|_{\mathfrak{H_{+}}}\|g\|_{\mathfrak{H}} ,\,
     ,\,f,g \in  \mathfrak{M}.$ \\

The latter  guarantees the inclusion of the numerical range of values  into a sector with arbitrary small semi-angle, more detailed see Chapter \ref{Ch4}.

\section{General case}

\subsection{Spectral theory point of view}

The    results connected with application of  the basis property in the Abell-Lidskii sense  \cite{firstab_lit:1kukushkin2021} allow us to   solve  many   problems \cite{firstab_lit:2kukushkin2022} in the     theory of evolution  equations.  The central idea of this chapter is devoted to  an approach  allowing us to principally broaden conditions imposed upon the second term (the operator $W$)  of the evolution equation in the abstract Hilbert space \eqref{6.1}.

In this way we can obtain abstract results covering many applied problems to say nothing on the far-reaching   generalizations.
     We plan to implement the idea  having involved a notion of  an operator function. This is why one of the paper challenges is to find a harmonious way of reformulating the main principles of the spectral theorem having  taken  into account the peculiarities of the convergence in the Abel-Lidskii sense. However, our final goal is  an existence and uniqueness theorem for an abstract  evolution equation with an operator function in the second term, where  the derivative in the first term  is supposed to be of the integer order. The  peculiar result that is worth highlighted     is the obtained analytic  formula for the solution.   We should remind that involving a notion of the operator function, we broaden a great deal  an operator  class corresponding to the second term.   Thus, we can state   that the main issue of the paper is closely connected  with  the spectral theorem for  the non-selfadjoint unbounded  operator. Here, we should make a brief digression and consider a theoretical background that allows us to obtain such exotic results.

 In  Chapter \ref{Ch4}, we obtained the clarification of the results by   Lidskii  \cite{firstab_lit:1Lidskii}  on  the decomposition on the root vector system of the non-selfadjoint operator. We used a technique of the entire function theory and introduce  a so-called  Schatten-von Neumann class of the convergence  exponent. Considering strictly accretive operators satisfying special conditions formulated in terms of the norm, we constructed a   sequence of contours of the power type, on the  contrary to the results by  Lidskii    \cite{firstab_lit:1Lidskii}, where a sequence of contours of the  exponential  type was used.

 In the next paragraphs, we produce the  application of the   method elaborated in  Chapter \ref{Ch4} to evolution equations in the abstract Hilbert space  with the second term of the special type. Here, to show our motivation, we should  appeal to a plenty of applications to concrete differential equations  connected with  modeling  various   physical-chemical processes:  filtration of liquid and gas in the highly porous fractal   medium; heat exchange processes in a medium  with fractal structure and memory; casual walks of a point particle that starts moving from the origin
by self-similar fractal set; oscillator motion under the action of
elastic forces which is  characteristic for a viscoelastic medium, etc.
  In particular,  we would like   to  study  the existence and uniqueness theorems  for evolution equations   with the second term -- a  differential operator  with a fractional derivative in final terms. In this connection such operators as a Riemann-Liouville  fractional differential operator,    Kipriyanov operator, Riesz potential,  difference operator are involved.  Nowadays,  the concept of  the fractional integro-differentiation is efficiently   used in the study of  problems related with the approximate controllability of systems with damping \cite{firstab_lit:Sukavanam 2020}, approximate controllability of  semi-linear stochastic systems \cite{firstab_lit:Sukavanam 2017}, \cite{firstab_lit:Sukavanam 2018},  partial-approximate controllability of semi-linear systems \cite{firstab_lit:Haq  2022}, asymptotic stability of  stochastic differential equations in Banach spaces \cite{firstab_lit:Shukla 2020}. Note that analysis of the required  conditions imposed upon the second term of the studied class of evolution equations deserves to be mentioned. In this regard, we should note  a well-known fact discussed in Chapter \ref{Ch2}  (the initial source is  \cite{firstab_lit:Shkalikov A.})  that a particular interest appears in the case when a senior term of the operator  \cite{firstab_lit(arXiv non-self)kukushkin2018}   is not selfadjoint at least, for   in the contrary  case there is a plenty of results devoted to the topic, within the framework of which  the following papers are well-known
\cite{firstab_lit:1Katsnelson,firstab_lit:1Krein,firstab_lit:Markus Matsaev,firstab_lit:2Markus,firstab_lit:Shkalikov A.}. Indeed, most of them deal with a decomposition of the  operator  to a sum,  where the senior term
     must be either a selfadjoint or normal operator. In other cases, the  methods established in Chapter \ref{Ch2} initially in the papers
     \cite{kukushkin2019,firstab_lit(arXiv non-self)kukushkin2018}  become relevant   and allow us  to study spectral properties of   operators,  whether we have the mentioned above  representation or not. Here, we ought to    stress that the results of  the papers \cite{firstab_lit:2Agranovich1994,firstab_lit:Markus Matsaev}  can be  also applied to study non-selfadjoin operators but  based on the sufficiently strong assumption regarding the numerical range of values of the operator (the numerical range belongs to a parabolic domain).
The methods of \cite{firstab_lit(arXiv non-self)kukushkin2018} that are   applicable to study    sectorial   operators can be  used in the natural way,  if we deal with a more abstract construction -- the  infinitesimal generator of a semigroup of contraction, this  issue was discussed in detail in Chapter \ref{Ch3}, see also the initial source  \cite{kukushkin2021a}.  The  central challenge  of the latter  paper  is how  to create a model   representing     a  composition of  fractional differential operators   in terms of the semigroup theory. Here, we should note that motivation arouse in connection with the fact that
  a   second-order  differential operator can be presented  as a some kind of  a  transform of   the infinitesimal generator of a shift semigroup and stress that
  the eigenvalue problem for the operator
     was previously  studied by methods of  the theory of functions    \cite{firstab_lit:1Nakhushev1977}.
Having been inspired by   novelty of the  idea,  we generalize a   differential operator with a fractional integro-differential composition  in the final terms   to some transform of the corresponding  infinitesimal generator of the shift semigroup.
By virtue of the   methods elaborated  in Chapter \ref{Ch2}
\cite{firstab_lit(arXiv non-self)kukushkin2018}, we   managed  to  study spectral properties of the  infinitesimal generator  transform and obtained an outstanding result --
   asymptotic equivalence between   the
real component of the resolvent and the resolvent of the   real component of the operator. The relevance is based on the fact that
   the  asymptotic formula  for   the operator  real component  can be  established in most  cases due to well-known asymptotic relations  for the regular differential operators as well as for the singular ones
 \cite{firstab_lit:Rosenblum}. Thus,  we have theorems establishing spectral properties of some class of sectorial  operators which allow  us, jointly with the results  \cite{firstab_lit:1kukushkin2021}, to study the Cauchy  problem for the abstract  evolution equation by the functional analysis methods. Note that the abstract approach to the Cauchy problem for the fractional evolution equation was previously implemented  in \cite{firstab_lit:Bazhl,Ph. Cl}. However, the  main advantage of the results established in this Chapter is the obtained formula for the solution of the evolution equation with the relatively wide conditions imposed upon the second term,  wherein the derivative at the left-hand side is supposed to be of the real order.
 In the next paragraph we consider  the abstract  evolution equation  with the second term -- an operator function of the power  type. This problem appeals to many ones that lie  in the framework of the theory of differential   equations. For instance, in the paper   \cite{L. Mor} the solution of the  evolution equation modeling the switching kinetics of   ferroelectrics  in the injection mode  can be obtained in the analytical way, if we impose the conditions upon the second term. The following papers deal with equations which can be studied by the obtained in this chapter  abstract method \cite{firstab_lit:Mainardi F.,firstab_lit:Mamchuev2017a,firstab_lit:Mamchuev2017,firstab_lit:Pskhu,firstab_lit:Wyss}.

\subsection{  Series expansion  and its application to   Existence  theorems}

In this paragraph, we represent two theorems   valuable from  theoretical  and applied points of view respectively.    The first one is a generalization of the Lidskii method  this  is why  following the the classical approach we divide it into two statements that can be claimed separately. The first  statement establishes a character of the series convergence having a principal meaning within the whole concept. The second statement reflects the name of convergence - Abel-Lidskii since  the latter   can be connected with the definition of the series convergence in the Abel sense, more detailed  information can be found in the monograph by Hardy G.H. \cite{firstab_lit:Hardy}. The second theorem is a valuable  application of the first one, it is based upon   suitable algebraic reasonings having noticed by the author and allowing us to involve   a fractional derivative in the first term. We should note that previously,  a concept of an operator function represented in the second term  was realized  in the paper  \cite{firstab_lit:2kukushkin2022}, where  a case corresponding to  a function   represented by a Laurent series with a polynomial regular part was considered.  Bellow, we consider  a comparatively more difficult  case obviously  related to the infinite regular part of the Laurent series and therefore  requiring  a principally different method of study.

Using results of Chapter \ref{Ch4} consider a decomposition of the Hilbert space into a direct sum
$$
 \mathfrak{H}=\mathfrak{N}_{q} \oplus  \mathfrak{M}_{q}
$$
corresponding to the eigenvalue $\mu_{q}$ of the operator $B.$
  We can choose a  Jordan basis in   the finite dimensional root subspace $\mathfrak{N}_{q}$ corresponding to the eigenvalue $\mu_{q}$  that consists of Jordan chains of eigenvectors and root vectors  of the operator $B_{q}.$  Each chain has a set of numbers
\begin{equation}\label{6.2}
  e_{q_{\xi}},e_{q_{\xi}+1},...,e_{q_{\xi}+k},\,k\in \mathbb{N}_{0},
\end{equation}
  where $e_{q_{\xi}},\,\xi=1,2,...,m $ are the eigenvectors  corresponding   to the  eigenvalue $\mu_{q}.$
Considering the sequence $\{\mu_{q}\}_{1}^{\infty}$ of the eigenvalues of the operator $B$ and choosing a  Jordan basis in each corresponding  space $\mathfrak{N}_{q},$ we can arrange a system of vectors $\{e_{i}\}_{1}^{\infty}$ which we will call a system of the root vectors or following  Lidskii  a system of the major vectors of the operator $B.$
Assume that  $e_{1},e_{2},...,e_{n_{q}}$ is  the Jordan basis in the subspace $\mathfrak{N}_{q}.$  We can prove easily (see \cite[p.14]{firstab_lit:1Lidskii}) that     there exists a  corresponding biorthogonal basis $g_{1},g_{2},...,g_{n_{q}}$ in the subspace $\mathfrak{M}_{q}^{\perp},$ where  $\mathfrak{M}_{q},$ is a subspace  wherein the operator  $B-\mu_{q} I$ is invertible. Using the reasonings \cite{firstab_lit:1kukushkin2021},  we conclude that $\{ g_{i}\}_{1}^{n_{q}}$ consists of the Jordan chains of the operator $B^{\ast}$ which correspond to the Jordan chains  \eqref{6.2}, more detailed information can be found in \cite{firstab_lit:1kukushkin2021}.
It is not hard to prove   that  the set  $\{g_{\nu}\}^{n_{j}}_{1},\,j\neq i$  is orthogonal to the set $ \{e_{\nu}\}_{1}^{n_{i}}$ (see \cite{firstab_lit:1kukushkin2021}).  Gathering the sets $\{g_{\nu}\}^{n_{j}}_{1},\,j=1,2,...,$ we can obviously create a biorthogonal system $\{g_{n}\}_{1}^{\infty}$ with respect to the system of the major vectors of the operator $B.$

\begin{lem}\label{L6.1} Assume that   $B$ is a compact operator, $\varphi$ is an analytical function inside $\vartheta(B),$ then  in the pole $\lambda_{q}$ of the operator  $(I-\lambda B)^{-1},$ the residue of the vector  function $e^{-\varphi (\lambda) t}B(I-\lambda B)^{-1}\!f,\,(f\in \mathfrak{H}),$  equals to
$$
-\sum\limits_{\xi=1}^{m(q)}\sum\limits_{i=0}^{k(q_{\xi})}e_{q_{\xi}+i}c_{q_{\xi}+i}(t),
$$
where $m(q)$ is a geometrical multiplicity of the $q$-th eigenvalue,  $k(q_{\xi})+1$ is a number of elements in the $q_{\xi}$-th Jourdan chain,
$$
c_{q_{\xi}+j} (t):=e^{-\varphi(\lambda_{q}) t}\sum\limits_{m=0}^{k(q_{\xi})-j} c_{q_{\xi}+j+m} H_{m}(\varphi,\lambda_{q},t ).
$$
\end{lem}
\begin{proof} Consider an integral
$$
 \mathfrak{I}=\frac{1}{2\pi i}\oint\limits_{\vartheta_{q}} e^{-\varphi(\lambda)t}B(I-\lambda B)^{-1}fd\lambda ,\,f\in \mathrm{D}(W),
$$
where the interior of the contour $\vartheta_{q}$ does not contain any poles of the operator $(I-\lambda T)^{-1},$  except of $\lambda_{q}=1/\mu_{q}.$ Assume that  $\mathfrak{N}_{q}$ is a  root  space corresponding to $\lambda_{q}$ and  consider  a Jordan basis  \{$e_{q_{\xi}+i}\},\,i=0,1,...,k(q_{\xi}),\;\xi=1,2,...,m(q)$ in $\mathfrak{N}_{q}.$ Using decomposition of the Hilbert space in the direct sum \eqref{4.9}, we can represent an element
$$
f=f_{1}+f_{2},
$$
where $f_{1}\in \mathfrak{N}_{q},\;f_{2}\in \mathfrak{M}_{q}.$ Note that the operator function $e^{-\varphi(\lambda)t}B(I-\lambda B)^{-1}f_{2}$ is regular in the interior of the contour $\vartheta_{q},$ it follows from the fact that $\mu_{q}$ ia a normal eigenvalue (see the supplementary information). Hence, we have
$$
 \mathfrak{I}=\frac{1}{2\pi i}\oint\limits_{\vartheta_{q}} e^{-\varphi(\lambda)t}B(I-\lambda B)^{-1}f_{1}d\lambda.
$$
Using the formula
$$
B(I-\lambda B)^{-1}=\left\{(I-\lambda B)^{-1}-I   \right\}\frac{1}{\lambda}=\left\{\left(\frac{1}{\lambda}I- B\right)^{-1}-\lambda I   \right\}\frac{1}{\lambda^{2}},
$$
we obtain
$$
\mathfrak{I}=-\frac{1}{2\pi i}\oint\limits_{\tilde{\vartheta}_{q}} e^{-\varphi(\zeta^{-1})t}B(\zeta I-  B)^{-1}f_{1}d\zeta,\,\zeta=1/\lambda.
$$
Now, let us decompose the element $f_{1}$ on the corresponding Jordan basis, we have
\begin{equation}\label{6.3}
f_{1}=\sum\limits_{\xi=1}^{m(q)}\sum\limits_{i=0}^{k(q_{\xi})}e_{q_{\xi}+i}c_{q_{\xi}+i}.
\end{equation}
In accordance with the  definition of the root vector, we have
$$
Be_{q_{\xi}}=\mu_{q}e_{q_{\xi}},\;Be_{q_{\xi}+1}=\mu_{q}e_{q_{\xi}+1}+e_{q_{\xi}},...,Be_{q_{\xi}+k}=\mu_{q}e_{q_{\xi}+k}+e_{q_{\xi}+k-1}.
$$
Using this formula, we can prove the following relation
\begin{equation}\label{6.4}
(\zeta I-  B)^{-1}e_{q_{\xi}+i}=\sum\limits_{j=0}^{i}\frac{e_{q_{\xi}+j}}{(\zeta-\mu_{q})^{i-j+1}}.
\end{equation}
Note that the case $i=0$ is trivial. Consider a case, when $i>0,$ we have
$$
\frac{(\zeta I-  B)e_{q_{\xi}+j}}{(\zeta-\mu_{q})^{i-j+1}}=\frac{\zeta e_{q_{\xi}+j}-Be_{q_{\xi}+j}}{(\zeta-\mu_{q})^{i-j+1}} = \frac{ e_{q_{\xi}+j}}{(\zeta-\mu_{q})^{i-j }}-\frac{e_{q_{\xi}+j-1}}{(\zeta-\mu_{q})^{i-j+1}},\,j>0,
$$
$$
\frac{(\zeta I-  B)e_{q_{\xi} }}{(\zeta-\mu_{q})^{i +1}}=  \frac{ e_{q_{\xi} }}{(\zeta-\mu_{q})^{i }}.
$$
Using these formulas, we obtain
$$
\sum\limits_{j=0}^{i}\frac{(\zeta I-  B)e_{q_{\xi}+j}}{(\zeta-\mu_{q})^{i-j+1}}= \frac{ e_{q_{\xi} }}{(\zeta-\mu_{q})^{i }}+  \frac{ e_{q_{\xi}+1}}{(\zeta-\mu_{q})^{i-1 }}-\frac{e_{q_{\xi} }}{(\zeta-\mu_{q})^{i }}+...
$$
$$
+ \frac{ e_{q_{\xi}+i}}{(\zeta-\mu_{q})^{i-i }}-\frac{e_{q_{\xi}+i-1}}{(\zeta-\mu_{q})^{i-i+1}}=
 \frac{ e_{q_{\xi}+i}}{(\zeta-\mu_{q})^{i-i }},
$$
what gives us the desired result. Now, substituting  \eqref{6.3},\eqref{6.4}, we get
$$
\mathfrak{I}=- \frac{1}{2\pi i}\sum\limits_{\xi=1}^{m(q)}\sum\limits_{i=0}^{k(q_{\xi})}c_{q_{\xi}+i}\sum\limits_{j=0}^{i} e_{q_{\xi}+j} \oint\limits_{\tilde{\vartheta}_{q}}\frac{ e^{-\varphi(\zeta^{-1})t}}{(\zeta-\mu_{q})^{i-j+1}}d\zeta.
$$
Note that the function $\varphi (\zeta^{-1})$ is analytic inside the interior of $\tilde{\vartheta}_{q},$ hence
$$
 \frac{1}{2\pi i}\oint\limits_{\tilde{\vartheta}_{q}}\frac{ e^{-\varphi(\zeta^{-1})t}}{(\zeta-\mu_{q})^{i-j+1}}d\zeta= \frac{1}{(i-j)!}\lim\limits_{\zeta\rightarrow\, \mu_{q}}\frac{d^{i-j}}{d\zeta^{\,i-j}}\left\{ e^{-\varphi(\zeta^{-1})t}\right\}=:e^{-\varphi(\lambda_{q})t}H_{i-j}(\varphi,\lambda_{q},t ).
$$
Changing the indexes, we have
$$
\mathfrak{I}=-  \sum\limits_{\xi=1}^{m(q)}\sum\limits_{i=0}^{k(q_{\xi})}c_{q_{\xi}+i}e^{-\varphi(\lambda_{q})t}\sum\limits_{j=0}^{i} e_{q_{\xi}+j} H_{i-j}(\varphi,\lambda_{q},t )= -\sum\limits_{\xi=1}^{m(q)}\sum\limits_{j=0}^{k(q_{\xi})}e_{q_{\xi}+j}e^{-\varphi(\lambda_{q})t}\sum\limits_{m=0}^{k(q_{\xi})-j} c_{q_{\xi}+j+m} H_{m}(\varphi,\lambda_{q},t )=
$$
$$
=-\sum\limits_{\xi=1}^{m(q)}\sum\limits_{j=0}^{k(q_{\xi})}e_{q_{\xi}+j}   c_{q_{\xi}+j} (t),
$$
where
$$
c_{q_{\xi}+j} (t):=e^{-\varphi(\lambda_{q})t}\sum\limits_{m=0}^{k(q_{\xi})-j} c_{q_{\xi}+j+m} H_{m}(\varphi,\lambda_{q},t ).
$$
The proof is complete.
\end{proof}

Consider a subset of natural numbers  $ \{N_{\nu}\}_{0}^{\infty}\subset  \mathbb{N}$ and define operators
$$
\mathcal{P}_{\nu}(\varphi,t):= \frac{1}{2 \pi i}\int\limits_{\vartheta_{\nu}(B)}e^{-\varphi(\lambda)t}B(I-\lambda B)^{-1} d\lambda,
$$
where  $\vartheta_{\nu}(B)$ is a contour on the complex plain containing only the  eigenvalues $\lambda_{N_{\nu}+1},\lambda_{N_{\nu}+2},...,\lambda_{N_{\nu+1}},$ and no  more eigenvalues, more detailed see Chapter \ref{Ch4}. In accordance with Lemma \ref{L6.1}, we have
\begin{equation*}\label{3h}
 \mathcal{P} _{\nu}(\varphi,t)f:   = \sum\limits_{q=N_{\nu}+1}^{N_{\nu+1}}\sum\limits_{\xi=1}^{m(q)}\sum\limits_{i=0}^{k(q_{\xi})}e_{q_{\xi}+i}c_{q_{\xi}+i}(t),
\end{equation*}
where   $k(q_{\xi})+1$ is a number of elements in the $q_{\xi}$-th Jourdan chain,  $m(q)$ is a geometrical multiplicity of the $q$-th eigenvalue,
\begin{equation*}\label{4h}
c_{q_{\xi}+i}(t)=   e^{ -\varphi(\lambda_{q})  t}\sum\limits_{m=0}^{k(q_{\xi})-i}H_{m}(\varphi, \lambda_{q},t)c_{q_{\xi}+i+m},\,i=0,1,2,...,k(q_{\xi}),
\end{equation*}
$
c_{q_{\xi}+i}= (f,g_{q_{\xi}+k-i}) /(e_{q_{\xi}+i},g_{q_{\xi}+k-i}),
$
$\lambda_{q}=1/\mu_{q}$ is a characteristic number corresponding to $e_{q_{\xi}},$
$$
H_{m}( \varphi,z,t ):=  \frac{e^{ \varphi(z)  t}}{m!} \cdot\lim\limits_{\zeta\rightarrow 1/z }\frac{d^{m}}{d\zeta^{\,m}}\left\{ e^{-\varphi (\zeta^{-1})t}\right\} ,\;m=0,1,2,...\,,\,.
$$
More detailed information on the considered above   Jordan chains can be found in \cite{firstab_lit:1kukushkin2021}.\\

In order to resume the results of Chapter\ref{Ch4}, consider the following statement  \\

\begin{teo}\label{T6.1}
 Assume that the operator function $\varphi$ is defined on the operator argument $W,$
 $|\varphi(\lambda)|\rightarrow \infty,\,|\lambda|\rightarrow\infty,\,\lambda\in \vartheta(B),\,\alpha\geq 1,$
 a sequence of natural numbers $\{N_{\nu}\}_{0}^{\infty}$ can be chosen so that
 \begin{equation}\label{6.5}
 \frac{1}{2\pi i}\int\limits_{\vartheta(B)}e^{-\varphi(\lambda)^{\alpha}t}B(I-\lambda B)^{-1}f d \lambda =\sum\limits_{\nu=0}^{\infty} \mathcal{P}_{\nu}(\varphi^{\alpha},t)f,\;f\in \mathrm{D}(\varphi),
 \end{equation}
 the latter series is absolutely convergent in the sense of the norm,
 \begin{equation}\label{6.6}
 \lim\limits_{t\rightarrow+0}\frac{1}{2\pi i}\int\limits_{\vartheta(B)}e^{-\varphi(\lambda)^{\alpha}t}B(I-\lambda B)^{-1}f d \lambda=f.
\end{equation}
 Then
there exists a solution of the Cauchy problem \eqref{6.1} in the form
\begin{equation}\label{6.7}
u(t)= \sum\limits_{\nu=0}^{\infty} \mathcal{P}_{\nu}(\varphi^{\alpha},t)f.
\end{equation}
Moreover,  if the operator   $\mathfrak{D}^{1-1/\alpha}_{-}\!\varphi(W)$ is accretive then  the existing solution is unique, if the set $\mathrm{D}(\varphi)$ is dense in $\mathfrak{H}$ then    the condition $f\in \mathrm{D}(\varphi)$ can be omitted.
\end{teo}

\begin{proof}
Let us show that
$
u(t)
$
is a solution of the problem \eqref{6.1}. Using the fact that the operator function is defined, we have
$$
  \varphi(W)u(t)=\frac{1}{2\pi i}\int\limits_{\vartheta(B)}\varphi(\lambda) e^{-\varphi(\lambda)^{\alpha}  t} B(I-\lambda B)^{-1}f d\lambda,\,f\in \mathrm{D}(\varphi),\,n\in \mathbb{N}.
$$
We need establish  the following relation
\begin{equation}\label{6.8}
\frac{d}{dt}\!\!\int\limits_{\vartheta(B)} \varphi(\lambda)^{ 1-\alpha }e^{-\varphi^{\alpha}(\lambda)t}B(I-\lambda B)^{-1}f\, d\lambda =
-\!\!\int\limits_{\vartheta(B)}\varphi(\lambda) e^{-\varphi(\lambda)^{\alpha}t}B(I-\lambda B)^{-1}f\, d\lambda,\,f\in \mathfrak{H},
\end{equation}
i.e. we can use a differentiation operation  under the integral. Using  simple  reasonings, we obtain the fact that
 that for an arbitrary
$$
\vartheta_{j}(B):=\left\{\lambda:\;|\lambda|=r>0,\, \theta_{0} \leq\mathrm{arg} \lambda \leq \theta_{1} \right\}\cup\left\{\lambda:\;r<|\lambda|<r_{j},\,\mathrm{arg} \lambda =\theta_{0}  ,\,\mathrm{arg} \lambda =\theta_{1} \right\},
$$
 there exists a limit
$
 (e^{-\varphi^{\alpha}(\lambda)  \Delta t} -1)e^{-\varphi^{\alpha}(\lambda)  t}/ \Delta t \,    \longrightarrow -\varphi^{\alpha}(\lambda) e^{-\varphi^{\alpha}(\lambda)  t}  ,\,\Delta t\rightarrow 0,
$
where convergence in accordance with the  Heine-Cantor theorem, is uniform with respect to $ \lambda\in \vartheta_{j}(B).$
By virtue of the  decomposition on the Taylor series, taking into account the fact that $\mathrm{Re}\varphi^{\alpha}(\lambda) \geq C |\varphi(\lambda)|^{\alpha},\,\lambda\in \vartheta(B),$
we get
$$
\left |\frac{e^{-\varphi^{\alpha}(\lambda)  \Delta t} -1}{ \Delta t}e^{-\varphi^{\alpha}(\lambda)  t}\right | \leq |\varphi(\lambda)|^{\alpha} e^{ |\varphi(\lambda)  |^{\alpha}\Delta t}
e^{-\mathrm{Re}\,\varphi^{\alpha}(\lambda)  t}\leq |\varphi(\lambda)|^{\alpha} e^{(\Delta t-Ct)  |\varphi(\lambda)|^{\alpha}  }, \,\lambda\in \vartheta(B).
$$
Thus, applying the latter estimate, Lemma \ref{L4.7} (Lemma 6 \cite{firstab_lit:1kukushkin2021}), for a sufficiently small value $\Delta t,$ we get
\begin{equation}\label{6.9}
 \left\|\, \int\limits_{\vartheta(B)}\frac{e^{-\varphi^{\alpha}(\lambda)  \Delta t} -1}{ \Delta t}\varphi^{1-\alpha}(\lambda)e^{-\varphi^{\alpha}(\lambda)  t}B(I-\lambda B)^{-1}f d\lambda \right\|_{\mathfrak{H}}
  \leq
   C \|f\|_{\mathfrak{H}} \int\limits_{\vartheta(B)} e^{-C   t|\varphi(\lambda)|^{\alpha}}|\varphi(\lambda)|    |d\lambda|.
\end{equation}
Let us establish the convergence of the last integral. Applying the condition  $|\varphi(\lambda)|> C |\lambda|^{\xi},\,\xi>0,$    we get\\
$$
\int\limits_{\vartheta(B)} e^{-C   t|\varphi(\lambda)|^{\alpha}}|\varphi(\lambda)|    |d\lambda| \leq \int\limits_{\vartheta(B)}  e^{-t    e^{C|\lambda|^{\xi}}  }    e^{ C|\lambda|^{\xi}}    |d\lambda|.
$$
It is clear that the latter integral is convergent for an arbitrary positive value $t,$ what     guaranties that the improper integral at the left-hand side of \eqref{6.9}
 is uniformly convergent with respect to $\Delta t.$ These facts give us an opportunity to claim that   relation \eqref{6.8} holds. Here, we should explain that this conclusion is based  on the generalization of the well-known theorem of the calculus, we left a complete  investigation of the matter to the reader having noted that the  the reasonings are absolutely analogous to the ordinary calculus.

 Applying  the scheme of the proof corresponding to the  ordinary integral calculus,  using the contour $\vartheta_{j}(B),$    applying  Lemma \ref{L4.7}  respectively, we can establish a  formula
\begin{equation*}
\int\limits_{0}^{\infty}x^{-\xi}dx\!\!\int\limits_{\vartheta(B)}e^{-\varphi^{\alpha}(\lambda)(t+x)} B(I-\lambda B)^{-1}f\, d\lambda=\!\!\int\limits_{\vartheta(B)}e^{-\varphi^{\alpha}(\lambda)t} B(I-\lambda B)^{-1}f d\lambda\int\limits_{0}^{\infty}x^{-\xi}e^{-\varphi^{\alpha}(\lambda)x}dx,
\end{equation*}
where $\xi\in(0,1).$
 Applying  the obtained  formulas, taking into account a relation
$$
  \int\limits_{0}^{\infty}x^{-\xi}e^{-\varphi^{\alpha}(\lambda)x}dx=  \Gamma(1-\xi) \varphi^{\alpha(\xi-1)}(\lambda),
$$
we get
\begin{equation*}
\mathfrak{D}^{1/\alpha}_{-}\!\!\! \int\limits_{\vartheta(B)}e^{-\varphi^{\alpha}(\lambda)t}     B(I-\lambda B)^{-1}f= -\frac{d}{dt}\int\limits_{\vartheta(B)}\varphi^{1-\alpha}(\lambda)e^{-\varphi^{\alpha}(\lambda)t}    B(I-\lambda B)^{-1}f\, d\lambda=
$$
$$
=\int\limits_{\vartheta(B)}e^{-\varphi^{\alpha}(\lambda)t}  \varphi(\lambda)  B(I-\lambda B)^{-1}f\, d\lambda,\;\alpha\geq 1.
\end{equation*}
Now, using the theorem conditions, we obtain the fact that $u$ is a solution of the equation   \eqref{6.1}.

 Let us show that the initial condition holds in the sense
$$
u(t)   \xrightarrow[   ]{\mathfrak{H}}  f,\,t\rightarrow+0.
$$
 If $f\in \mathrm{D}(W),$ then it becomes obvious due to the theorem conditions. To establish the fact in  the case   $f\in \mathfrak{H},$ we should   involve the accretive property of the operator composition  $ \mathfrak{D}^{1-1/\alpha}_{-} \varphi(W)   .$
  Consider an operator
$$
S_{t}f=\frac{1}{2\pi i} \int\limits_{\vartheta}e^{-\varphi^{\alpha}(\lambda)t}B(I-\lambda B)^{-1}f\, d\lambda,\,t>0.
$$
It is clear that    $S_{t}:\mathfrak{H}\rightarrow \mathfrak{H},$ more detailed information in this regard  is represented in the preliminary section.
Let us prove  that
$$
\|S_{t}\|_{\mathfrak{H}\rightarrow \mathfrak{H}}\leq1,\;t>0.
$$
We need to establish the following  formula
\begin{equation}\label{6.10}
\mathfrak{D}^{1-1/\alpha}_{-}\mathfrak{D}^{1/\alpha}_{-}u(t)=- \frac{du(t)}{dt}.
\end{equation}
Analogously to the above, we get
\begin{equation*}
\int\limits_{0}^{\infty}x^{1/\alpha-1}dx\int\limits_{\vartheta}\varphi(\lambda)e^{-\varphi^{\alpha}(\lambda)(t+x)} B(I-\lambda B)^{-1}f d\lambda=
$$
$$
=\int\limits_{\vartheta}\varphi(\lambda)e^{-\varphi^{\alpha}(\lambda)t}  B(I-\lambda B)^{-1}f d\lambda\int\limits_{0}^{\infty}x^{1/\alpha-1}e^{-\varphi^{\alpha}(\lambda)x}dx=
\end{equation*}
$$
=\Gamma(1/\alpha)\int\limits_{\vartheta}\varphi(\lambda) \varphi^{-1}(\lambda) e^{-\varphi^{\alpha}(\lambda)t}  B(I-\lambda B)^{-1}f d\lambda = 2\pi i\, \Gamma(1/\alpha) u(t).
$$
Differentiating the latter relation with respect to the time variable, we obtain   formula \eqref{6.10}.

  Assume that   $f\in \mathrm{D}( \varphi).$
Applying  the operator $\mathfrak{D}^{1-1/\alpha}_{-}$ to   both sides of equation  \eqref{6.1}, using   formula \eqref{6.10},
we get
$$
u'+\mathfrak{D}^{1-1/\alpha}_{-} \varphi(W)u=0.
$$
 Let us multiply the both sides of the last    relation  on $u$  in the sense of the inner product, we get
$$
\left(u'_{t},u\right)_{\mathfrak{H}}+(\mathfrak{D}^{1-1/\alpha}_{-} \varphi(W)u,u)_{\mathfrak{H}}=0.
$$
Consider a real part of the latter expression, we have
$$
\mathrm{Re}\left(u'_{t},u\right)_{\mathfrak{H}}+\mathrm{Re}(\mathfrak{D}^{1-1/\alpha}_{-}\varphi(W)u,u)_{\mathfrak{H}}= \left(u'_{t},u\right)_{\mathfrak{H}}/2+ \left(u, u'_{t}\right)_{\mathfrak{H}}/2+\mathrm{Re}(\mathfrak{D}^{1-1/\alpha}_{-}\varphi(W)u,u)_{\mathfrak{H}}.
$$
Having noticed that
$$
\frac{d}{dt}\left\{\|u(t)\|_{\mathfrak{H}}^{2} \right\}=\frac{d}{dt}\left(u(t),u(t)\right)_{\mathfrak{H}}=\left(u'_{t},u\right)_{\mathfrak{H}}+\left(u,u'_{t}\right)_{\mathfrak{H}},
$$
we obtain
$$
\frac{d}{dt}\left\{\|u(t)\|_{\mathfrak{H}}^{2} \right\} = -2\mathrm{Re}(\mathfrak{D}^{1-1/\alpha}_{-}\tilde{W}u,u)_{\mathfrak{H}}\leq 0.$$
 Integrating both sides, we get
$$
  \|u(\tau)\|_{\mathfrak{H}}^{2}-  \|u(0)\|_{\mathfrak{H}}^{2}=\int\limits_{0}^{\tau} \frac{d}{dt}\left\{\|u(t)\|_{\mathfrak{H}}^{2} \right\} dt\leq 0.
$$
The last relation can be rewritten in the form
$$
\|S_{t}f\|_{\mathfrak{H}}\leq \|f\|_{\mathfrak{H}},\,f\in \mathrm{D}(\varphi).
$$
Since $\mathrm{D}(\varphi)$ is a dense set in $ \mathfrak{H},$ then we obviously  obtain  the  desired result, i.e. $\|S_{t}\|_{\mathfrak{H}\rightarrow \mathfrak{H}}\leq 1.$
Now, having assumed that
$$
f_{n}   \xrightarrow[   ]{\mathfrak{H}}  f,\,n\rightarrow \infty,\;\{f_{n}\}\subset \mathrm{D}( \varphi ),\,f\in \mathfrak{H},
$$
consider the following reasonings
$$
\|u(t)-f\|_{\mathfrak{H}}=\|S_{t}f-f\|_{\mathfrak{H}}=\|S_{t}f-S_{t}f_{n}+S_{t}f_{n}-f_{n}+f_{n}-    f\|_{\mathfrak{H}}\leq
$$
$$
\leq\|S_{t}\|\cdot\|f- f_{n}\|_{\mathfrak{H}}+\|S_{t}f_{n}-f_{n}\|_{\mathfrak{H}}+\|f_{n}-    f\|_{\mathfrak{H}}.
$$
Note that
$$
S_{t}f_{n}   \xrightarrow[   ]{\mathfrak{H}}  f_{n},\,t\rightarrow+0.
$$
It is clear that  if we choose $n$ so that  $\|f- f_{n}\|_{\mathfrak{H}}<\varepsilon/3$ and  after that  choose $t$   so that $\|S_{t}f_{n}-f_{n}\|_{\mathfrak{H}}<\varepsilon/3,$ then we obtain
 $\forall\varepsilon>0,\,\exists\, \delta(\varepsilon):\,\|u(t)-f\|_{\mathfrak{H}}<\varepsilon,\,t<\delta.$
   Thus, we can put  $u(0)=f$ and claim that   the initial condition holds in the case $f\in \mathfrak{H}.$       The uniqueness follows easily from the fact that $ \mathfrak{D}^{1-1/\alpha}_{-}\varphi(W) $ is accretive.  In this case, repeating the previous reasonings, we come to
$$
  \|\phi(t)\|_{\mathfrak{H}}^{2} \leq  \|\phi(0)\|_{\mathfrak{H}}^{2},\,t>0,
$$
where $\phi(t) $ is a difference  of   solutions $u_{1}(t) $ and $u_{2}(t)$ corresponding to the given initial condition. Observe  that  by virtue of the initial condition, we have  $\phi(0)=0.$  Therefore,   the last  inequality    can hold only if $\phi(t) =0.$ We complete the proof.
\end{proof}

\section{Entire function case}

Consider an entire function    $\varphi$ that can be represented by its Taylor series about the point zero and consider a compact invertible operator  $B: \mathfrak{H}\rightarrow \mathfrak{H}$  defined in the preliminary section   such that
$
 \Theta(B) \subset \mathfrak{L}_{0}(\theta ),
$
thus we have a formal construction
\begin{equation*}
\varphi(W)=\sum\limits_{n=0}^{\infty}  c_{n}  W^{n},
\end{equation*}
  where $c_{n}$ are the Taylor coefficients corresponding to the function $\varphi.$ In  Chapter \ref{Ch5}
we established  conditions under which being imposed the latter series of operators   converges on some elements of the Hilbert space $\mathfrak{H}$ and coincides with the operator function.

We have the following theorem that gives us conditions of the Cauchy problem \eqref{6.1}  solvability.

\begin{teo}\label{T6.2}
Assume that    the entire function $\varphi$ of the order less than a half,    its zeros with a sufficiently large absolute value   do not belong to the sector $\mathfrak{L}_{0}(\theta),$ in the case $\alpha \neq  [\alpha],$ the operator $B$ is compact, $ \Theta(B) \subset \mathfrak{L}_{0}(\theta ),$ $B\in \mathfrak{S}_{s},\,0<s<\infty,$  then the statement  of Theorem \ref{T6.1}  holds.
 \end{teo}

\begin{proof}
  Let us establish the first  relation \eqref{6.5}.  Consider a contour $\vartheta(B).$ Having fixed $R>0,0<\kappa<1,$ so that $R(1-\kappa)=r,$ consider a monotonically increasing sequence $\{R_{\nu}\}_{0}^{\infty},\,R_{\nu}=R(1-\kappa)^{-\nu+1}.$   Using Lemma 5 \cite{firstab_lit:1kukushkin2021}, we get
$$
\|(I-\lambda B )^{-1}\|_{\mathfrak{H}}\leq e^{\gamma (|\lambda|)|\lambda|^{\sigma}}|\lambda|^{m},\,\sigma>s,\,m=[\sigma],\,|\lambda|=\tilde{R}_{\nu},\;R_{\nu}<\tilde{R}_{\nu}<R_{\nu+1},
$$
where
$$
\gamma(|\lambda|)= \beta ( |\lambda|^{m+1})  +C \beta(|C  \lambda| ^{m+1}),\;\beta(r )= r^{ -\frac{\sigma}{m+1} }\left(\int\limits_{0}^{r}\frac{n_{B^{m+1}}(t)dt}{t }+
r \int\limits_{r}^{\infty}\frac{n_{B^{m+1}}(t)dt}{t^{ 2  }}\right).
$$
Note that  in accordance with Lemma 3 \cite{firstab_lit:1Lidskii}  the following relation holds
\begin{equation}\label{6.11}
\sum\limits_{i=1}^{\infty}\lambda^{\frac{\sigma}{ (m+1)}}_{i}( \tilde{B} )\leq \sum\limits_{i=1}^{\infty}s^{ \,\sigma }_{i}( B )<\infty,\,\varepsilon>0,
\end{equation}
where   $\tilde{B}:=(B^{\ast m+1}A^{m+1})^{1/2}.$  It is clear that   $\tilde{B}\in  \mathfrak{S}_{\upsilon},\,\upsilon<\sigma/(m+1).$
Denote by $\vartheta_{\nu}$ a bound of the intersection of the ring $\tilde{R}_{\nu}<|\lambda|<\tilde{R}_{\nu+1}$ with the interior of the contour $\vartheta(B),$ denote by $N_{\nu}$ a number of the resolvent poles  being   contained  in the set $\mathrm{int }\,\vartheta(B) \,\cap \{\lambda:\,r<|\lambda|<\tilde{R}_{\nu} \}.$ In accordance with Lemma 3 \cite{firstab_lit(axi2022)},  we get
\begin{equation}\label{6.12}
 \frac{1}{2\pi i}\oint\limits_{\vartheta_{\nu}}e^{-\varphi^{\alpha}(\lambda)  t} B(I-\lambda B)^{-1}f d \lambda =\sum\limits_{q=N_{\nu}+1}^{N_{\nu+1}}\sum\limits_{\xi=1}^{m(q)}\sum\limits_{i=0}^{k(q_{\xi})}e_{q_{\xi}+i}c_{q_{\xi}+i}(t),\;f\in \mathfrak{H}.
\end{equation}
Let us estimate the above integral, for this purpose split the contour $\vartheta_{\nu}$ on  terms $\tilde{\vartheta}_{ \nu  }:=\{\lambda:\,|\lambda|=\tilde{R}_{\nu},\, \theta_{0} \leq\mathrm{arg} \lambda  \leq\theta_{1}  \},\,\tilde{\vartheta}_{ \nu+1  },\,\vartheta_{\nu_{-}}:=
\{\lambda:\,\tilde{R}_{\nu}<|\lambda|<\tilde{R}_{\nu+1},\, \mathrm{arg} \lambda  =\theta_{0}  \},\, \vartheta_{\nu_{+}}:=
\{\lambda:\,\tilde{R}_{\nu}<|\lambda|<\tilde{R}_{\nu+1},\, \mathrm{arg} \lambda  =\theta_{1}  \}.$   Applying  the Wieman theorem  (Theorem 30, \S 18, Chapter I \cite{firstab_lit:Eb. Levin}),  we can claim that there exists such a sequence $\{R'_{n}\}_{1}^{\infty},\,R'_{n}\uparrow \infty,\,\tilde{R}_{\nu}<R'_{\nu}<\tilde{R}_{\nu+1}$ that
\begin{equation}\label{6.13}
\forall \varepsilon>0,\,\exists N(\varepsilon):\,e^{- C|\varphi(\lambda )|^{\alpha}t}\leq e^{- C m^{\alpha}_{\varphi}(R'_{\nu})t}\leq e^{- C t[M_{\varphi}(R'_{\nu})]^{(\cos \pi \varrho-\varepsilon)\alpha}},\,\lambda \in \tilde{\vartheta}_{ \nu  },\,\nu> N(\varepsilon),
\end{equation} where $\varrho$ is the order of the entire  function $\varphi.$   We should note that the assumption  $\tilde{R}_{\nu}<R'_{n}<\tilde{R}_{\nu+1}$ has been  made without loss of generality of the reasonings for in the context of the proof we do not care on the accurate arrangement of the contours but need to prove the   existence of an arbitrary one. This inconvenience is based upon the uncertainty in the way of chousing the contours in the Wieman theorem, at the same time  at any rate, we can extract a subsequence of    the sequence  $\{\tilde{R}_{n}\}_{1}^{\infty}$ in the way we need. Thus, using the given reasonings, Applying  Lemma 5 \cite{firstab_lit:1kukushkin2021}, relation \eqref{6.13},      we get
\begin{equation*}
 J_{  \nu  }: =\left\|\,\int\limits_{\tilde{\vartheta}_{ \nu }}e^{-\varphi^{\alpha}(\lambda)  t} B(I-\lambda B)^{-1}f d \lambda\,\right\|_{\mathfrak{H}}\leq \,
 \int\limits_{\tilde{\vartheta}_{ \nu }}e^{-t\mathrm{Re}\,\varphi^{\alpha}(\lambda)  } \left\|B(I-\lambda B)^{-1}f \right\|_{\mathfrak{H}} |d \lambda|\leq
\end{equation*}
$$
 \leq e^{\gamma (|\lambda|)|\lambda|^{\sigma} }|\lambda|^{m+1}   Ce^{- C t[M_{\varphi}(R'_{\nu})]^{(\cos \pi \varrho-\varepsilon)\alpha}} \int\limits_{-\theta-\varsigma}^{\theta+\varsigma}  d \,\mathrm{arg} \lambda,\,|\lambda|=\tilde{R}_{\nu}.
$$
As a result, we get
$$
J_{ \nu } \leq    e^{\gamma (|\lambda|)|\lambda|^{\sigma} }|\lambda|^{m+1}   Ce^{- C t[M_{\varphi}(R'_{\nu})]^{(\cos \pi \varrho-\varepsilon)\alpha}} ,
   $$
   where
   $
   \,m=[\sigma],\,|\lambda|=\tilde{R}_{\nu}.
$
In accordance with  Lemma 2 \cite{firstab_lit:1kukushkin2021}, we have $\gamma (|\lambda|)\rightarrow 0,\,|\lambda|\rightarrow\infty.$

It follows from the definition that if $\varrho$ is the order of the entire function $\varphi(z),$ and if $\varepsilon$ is an arbitrary positive number, then
$
e^{r^{\varrho-\varepsilon}}<M_{\varphi}(r)<e^{r^{\varrho+\varepsilon}},
$
where the inequality on the right-hand side is satisfied for all sufficiently large values of $r,$ and the inequality on the left-hand side holds for some sequence $\{r_{n}\}$ of values of $r,$ tending to infinity, see Chapter \ref{Ch4}. Thus, we can extract a subsequence from the sequence $\{\tilde{R'}_{n}\}_{1}^{\infty}$ and as a result from the sequence  $\{\tilde{R}_{n}\}_{1}^{\infty}$  so that  for a fixed $t$ and  a sufficiently large $\nu,$ we have  $ \gamma (|\tilde{R}_{\nu}|)|\tilde{R}_{\nu}|^{\sigma} - C t[M_{\varphi}(R'_{\nu})]^{(\cos \pi \varrho-\varepsilon)\alpha}     < 0.$ Here, we have not used a subsequence to not complicate the form of writing.
 Therefore, taking into account the above estimates, we can claim that the following series is convergent
 $$
 \sum\limits_{\nu=0}^{\infty}J_{\nu}<\infty.
$$
  Applying Lemma 6 \cite{firstab_lit:1kukushkin2021}, Lemma \ref{L4.5}, we get
$$
 J^{+}_{\nu}: =\left\|\,\int\limits_{\vartheta_{\nu_{+}}}e^{-\varphi^{\alpha}(\lambda)  t} B(I-\lambda B)^{-1}f d \lambda\,\right\|_{\mathfrak{H}}\leq  C\|f\|_{\mathfrak{H}} \cdot C\int\limits_{R_{\nu}}^{R_{\nu+1}}  e^{-  C  t \mathrm{Re}\, \varphi^{\alpha}(\lambda) }   |d   \lambda|\leq
  $$
  $$
\leq C\!\!\!\int\limits_{R_{\nu}}^{R_{\nu+1}}  e^{-  C t |\varphi (\lambda)|^{\alpha} }   |d   \lambda| \leq C  e^{- Cte^{ \alpha H(\theta_{1})R_{\nu}^{\varrho(R_{\nu})}}  }     \int\limits_{R_{\nu}}^{R_{\nu+1}}   |d   \lambda|=
 C  e^{-   C te^{ \alpha H(\theta_{1})R_{\nu}^{\varrho(R_{\nu})}}  }    \{R_{\nu+1}-R_{\nu} \}.
 $$
$$
 J^{-}_{\nu}: =\left\|\,\int\limits_{\vartheta_{\nu_{-}}}e^{-\varphi^{\alpha}(\lambda)  t}B(I-\lambda B)^{-1}f d \lambda\,\right\|_{\mathfrak{H}}\leq   C  e^{-   Cte^{ \alpha H(\theta_{0})R_{\nu}^{\varrho(R_{\nu})}}  }       \int\limits_{R_{\nu}}^{R_{\nu+1}}   |d   \lambda|=
 C  e^{-   Cte^{\alpha  H(\theta_{0})R_{\nu}^{\varrho(R_{\nu})}}  }    \{R_{\nu+1}-R_{\nu} \}.
$$
The obtained results allow us to claim (the proof is omitted) that
$$
  \sum\limits_{\nu=0}^{\infty}J^{+}_{\nu}<\infty,\;\; \sum\limits_{\nu=0}^{\infty}J^{-}_{\nu}<\infty.
$$
Using the formula \eqref{6.12}, the given above decomposition of the contour $\vartheta_{\nu},$  we   obtain the fact of absolute convergence of the series in   \eqref{6.5}. Let us establish equality  \eqref{6.5}, for this purpose, we should note that in accordance with  relation \eqref{6.12}, the properties of the contour integral, we have
 $$
 \frac{1}{2\pi i}\!\!\oint\limits_{\vartheta_{\tilde{R}_{n}}  (B)}\!\!\!\!e^{-\varphi^{\alpha}(\lambda)  t} B(I-\lambda B)^{-1}f \,d \lambda =
   \sum\limits_{\nu=0}^{n-1} \sum\limits_{q=N_{\nu}+1}^{N_{\nu+1}}\sum\limits_{\xi=1}^{m(q)}
  \sum\limits_{i=0}^{k(q_{\xi})}e_{q_{\xi}+i}c_{q_{\xi}+i}(t)
  ,\;f\in \mathfrak{H},\,n\in \mathbb{N},
$$
where
$
 \vartheta_{R}(B):= \mathrm{Fr}\left\{\mathrm{int }\,\vartheta(B) \,\cap \{\lambda:\,r<|\lambda|<R \}\right\}.
$
Using the proved above  fact $J_{ \nu }\rightarrow0,\,\nu\rightarrow\infty,$    we obtain easily
$$
\frac{1}{2\pi i}\!\!\oint\limits_{\vartheta_{\tilde{R}_{n}}  (B)}\!\!\!\!e^{-\varphi^{\alpha}(\lambda)  t} B(I-\lambda B)^{-1}f \,d \lambda\rightarrow \frac{1}{2\pi i}\oint\limits_{\vartheta  (B)} e^{-\varphi^{\alpha}(\lambda)  t} B(I-\lambda B)^{-1}f \,d \lambda,\;n\rightarrow\infty.
$$
 The latter relation
  gives us the following formula
\begin{equation*}
 \frac{1}{2\pi i}\!\!\oint\limits_{\vartheta   (B)}\! e^{-\varphi^{\alpha}(\lambda)  t} B(I-\lambda B)^{-1}f \,d \lambda =
   \sum\limits_{\nu=0}^{\infty} \mathcal{P}_{\nu}(\varphi^{\alpha},t)f,\;f\in \mathfrak{H}.
\end{equation*}
Taking into account     $\mathrm{D}(\varphi)\subset \mathrm{D}(W)$ see Remark \ref{R5.1}, then applying Lemma \ref{L5.3}, we obtain   relation \eqref{6.6}.
\end{proof}

\section{Essential singularity case}

Under assumptions of the previous section regarding the operator argument,
consider an entire function    $\varphi$ that can be represented by its Laurent  series about the point zero, we have  the following formal construction
\begin{equation*}
\varphi(W)=\sum\limits_{n=-\infty}^{k}  c_{n}  W^{n},
\end{equation*}
where $c_{n}$ are the Taylor coefficients corresponding to the function $\varphi.$ In  Chapter \ref{Ch5}
we established  conditions under which being imposed the latter series of operators   converges on some elements of the Hilbert space $\mathfrak{H}$ and coincides with the operator function so the used   notation is not accidental.\\

\begin{teo}\label{T6.3} Assume that $B$ is a compact operator, $\Theta(B) \subset   \mathfrak{L}_{0}(\theta),$
\begin{equation*}
 \varphi(z)=\sum\limits_{n=-\infty}^{s}c_{n}z^{n},\,z\in  \mathbb{C},\,s\in \mathbb{N},\; \max\limits_{n=0,1,...,s}(|\mathrm{arg} c_{n}|+n\theta)<\pi/2\alpha,
\end{equation*}
$B\in \mathfrak{S}_{p},\,0<p<\alpha s,$   moreover  in the case $B\,  \overline{\in}\,  \mathfrak{S}_{\rho}, \,\rho=\inf p$ the additional condition holds
\begin{equation}\label{6.14}
s_{n}(B)=o(n^{-1/\rho}).
 \end{equation}
Then  the statement of Theorem \ref{T6.1} holds.
 \end{teo}

\begin{proof}
  Let us establish relation \eqref{6.5}.  Consider a contour $\vartheta(B).$ Having fixed $R>0,0<\kappa<1,$ so that $R(1-\kappa)=r,$ consider a monotonically increasing sequence $\{R_{\nu}\}_{0}^{\infty},\,R_{\nu}=R(1-\kappa)^{-\nu+1}.$   Using Lemma \ref{L4.11}, we get
$$
\|(I-\lambda B )^{-1}\|_{\mathfrak{H}}\leq e^{\gamma (|\lambda|)|\lambda|^{\rho}}|\lambda|^{m},\,m=[\rho],\,|\lambda|=\tilde{R}_{\nu},\;R_{\nu}<\tilde{R}_{\nu}<R_{\nu+1},
$$
where the function  $\gamma (r)=r^{-\rho}w(r)$ the latter function  is defined in Lemma \ref{L4.11}.
Note that  in accordance with Lemma 3 \cite{firstab_lit:1Lidskii}  the following relation holds
\begin{equation*}
\sum\limits_{i=1}^{\infty}\lambda^{\frac{\rho+\varepsilon}{ (m+1)}}_{i}( \tilde{B} )\leq \sum\limits_{i=1}^{\infty}s^{ \,\rho+\varepsilon }_{i}( B )<\infty,\,\varepsilon>0,
\end{equation*}
where   $\tilde{B}:=(B^{\ast m+1}A^{m+1})^{1/2}.$  It is clear that   $\tilde{B}\in \tilde{\mathfrak{S}}_{\upsilon},\,\upsilon\leq \rho/(m+1).$
Denote by $\vartheta_{\nu}$ a bound of the intersection of the ring $\tilde{R}_{\nu}<|\lambda|<\tilde{R}_{\nu+1}$ with the interior of the contour $\vartheta(B),$ denote by $N_{\nu}$ a number of poles being   contained  in the set $\mathrm{int }\,\vartheta(B) \,\cap \{\lambda:\,r<|\lambda|<\tilde{R}_{\nu} \}.$ In accordance with the conditions imposed upon $\varphi,$ we can apply
Lemma \ref{L6.1},  we get
\begin{equation*}
 \frac{1}{2\pi i}\oint\limits_{\vartheta_{\nu}}e^{-\varphi^{\alpha} (\lambda)t} B(I-\lambda B)^{-1}f d \lambda =\sum\limits_{q=N_{\nu}+1}^{N_{\nu+1}}\sum\limits_{\xi=1}^{m(q)}\sum\limits_{i=0}^{k(q_{\xi})}e_{q_{\xi}+i}c_{q_{\xi}+i}(t),\;f\in \mathfrak{H}.
\end{equation*}
Let us estimate the above integral, for this purpose split the contour $\vartheta_{\nu}$ on  terms $\tilde{\vartheta}_{ \nu  }:=\{\lambda:\,|\lambda|=\tilde{R}_{\nu},\,|\mathrm{arg} \lambda |\leq\theta +\varsigma\},\,\tilde{\vartheta}_{ \nu+1  },\, \vartheta_{\nu_{+}}:=
\{\lambda:\,\tilde{R}_{\nu}<|\lambda|<\tilde{R}_{\nu+1},\, \mathrm{arg} \lambda  =\theta +\varsigma\},\,\vartheta_{\nu_{-}}:=
\{\lambda:\,\tilde{R}_{\nu}<|\lambda|<\tilde{R}_{\nu+1},\, \mathrm{arg} \lambda  =-\theta -\varsigma\}.$ Note that in accordance the theorem conditions we have the following inequality
\begin{equation}\label{6.15}
e^{-\mathrm{Re}\varphi^{\alpha}(\lambda) t}\leq e^{-C|\varphi(\lambda)|^{\alpha} t}\leq   e^{-C|\lambda|^{\alpha s} t},
\end{equation}
for a sufficiently large value $|\lambda|,$ more detailed see relation\eqref{5.19},  Chapter \ref{Ch5}.  Applying \eqref{6.15},   Lemma \ref{4.11},
 we get
\begin{equation*}
 J_{  \nu  }: =\left\|\,\int\limits_{\tilde{\vartheta}_{ \nu }}e^{-\varphi^{\alpha} (\lambda)t} B(I-\lambda B)^{-1}f d \lambda\,\right\|_{\mathfrak{H}}\leq \,
 \int\limits_{\tilde{\vartheta}_{ \nu }}e^{-t\mathrm{Re}\,\varphi^{\alpha}(\lambda)  } \left\|B(I-\lambda B)^{-1}f \right\|_{\mathfrak{H}} |d \lambda|\leq
\end{equation*}
$$
 \leq e^{\gamma (|\lambda|)|\lambda|^{\rho} }|\lambda|^{m+1}   Ce^{-  C|\lambda|^{\alpha s}t } \int\limits_{-\theta-\varsigma}^{\theta+\varsigma}  d \,\mathrm{arg} \lambda,\,|\lambda|=\tilde{R}_{\nu}.
$$
Thus,  we get
$
J_{ \nu } \leq     Ce^{\gamma (|\lambda|)|\lambda|^{\rho}-C|\lambda|^{\alpha s}t   }|\lambda|^{m+1},
   $
   where
   $
   \,m=[\rho],\,|\lambda|=\tilde{R}_{\nu}.
$
Let us show that for a fixed $t$ and  a sufficiently large $|\lambda|,$ we have  $ \gamma (|\lambda|)|\lambda|^{\rho}-C|\lambda|^{\alpha s}t     < 0.$
It follows  directly from Lemma 2 \cite{firstab_lit:1kukushkin2021}, see Lemma \ref{L4.2}, Chapter \ref{Ch4}. We should consider  \eqref{6.11}, in the case when  $B\in \mathfrak{S}_{\rho}$ as well as in the case $B \, \overline{\in} \, \mathfrak{S}_{\rho}$  but here we must involve the additional condition \eqref{6.14}, that gives us due to the reasonings of Theorem \ref{4.6} fulfilment of Lemma \ref{L4.2} conditions.
 Therefore
$$
 \sum\limits_{\nu=0}^{\infty}J_{\nu}<\infty.
$$
Using the analogous estimates, applying  Lemma \ref{L4.7}, we get
$$
 J^{+}_{\nu}: =\left\|\,\int\limits_{\vartheta_{\nu_{+}}}e^{-\varphi^{\alpha}(\lambda) t} B(I-\lambda B)^{-1}f d \lambda\,\right\|_{\mathfrak{H}}\leq  C\|f\|_{\mathfrak{H}} \cdot C\int\limits_{R_{\nu}}^{R_{\nu+1}}  e^{-  C t \mathrm{Re}\, \varphi^{\alpha}(\lambda) }   |d   \lambda|\leq
  $$
  $$
  \leq C  e^{-t CR^{\alpha s}_{\nu}  }     \int\limits_{R_{\nu}}^{R_{\nu+1}}   |d   \lambda|=
 C  e^{-t CR^{\alpha s}_{\nu}  }  \{R_{\nu+1}-R_{\nu} \}.
 $$
$$
 J^{-}_{\nu}: =\left\|\,\int\limits_{\vartheta_{\nu_{-}}}e^{-\varphi^{\alpha}(\lambda) t}B(I-\lambda B)^{-1}f d \lambda\,\right\|_{\mathfrak{H}}\leq
 C  e^{-t CR^{\alpha s}_{\nu}  }     \int\limits_{R_{\nu}}^{R_{\nu+1}}   |d   \lambda|=
 C  e^{-t CR^{\alpha s}_{\nu}  }  \{R_{\nu+1}-R_{\nu} \}.
$$
The obtained results allow us to claim (the proof is left to the reader) that
$$
  \sum\limits_{\nu=0}^{\infty}J^{+}_{\nu}<\infty,\;\; \sum\limits_{\nu=0}^{\infty}J^{-}_{\nu}<\infty.
$$
Using the formula \eqref{6.12}, the given above decomposition of the contour $\vartheta_{\nu},$  we   obtain the fact that the series in the right-hand side of    relation \eqref{6.5} is absolutely convergent in the sense of the norm. Let us establish   \eqref{6.5}, for this purpose, we should note that in accordance with  relation \eqref{6.12}, the properties of the contour integral, we have
 $$
 \frac{1}{2\pi i}\!\!\oint\limits_{\vartheta_{\tilde{R}_{p}}  (B)}\!\!\!\!e^{-\varphi^{\alpha}(\lambda) t} B(I-\lambda B)^{-1}f \,d \lambda =
   \sum\limits_{\nu=0}^{p-1} \sum\limits_{q=N_{\nu}+1}^{N_{\nu+1}}\sum\limits_{\xi=1}^{m(q)}
  \sum\limits_{i=0}^{k(q_{\xi})}e_{q_{\xi}+i}c_{q_{\xi}+i}(t)
  ,\;f\in \mathfrak{H},\,p\in \mathbb{N},
$$
where
$$
\vartheta_{\tilde{R}_{\nu}}(B):= \mathrm{Fr}\left\{\mathrm{int }\,\vartheta(B) \,\cap \{\lambda:\,r<|\lambda|<\tilde{R}_{\nu} \}\right\}.
$$
Using the proved above  fact $J_{ \nu }\rightarrow0,\,\nu\rightarrow\infty,$    we can easily  get
$$
\frac{1}{2\pi i}\!\!\oint\limits_{\vartheta_{\tilde{R}_{p}}  (B)}\!\!\!\!e^{-\varphi^{\alpha}(\lambda) t} B(I-\lambda B)^{-1}f \,d \lambda\rightarrow \frac{1}{2\pi i}\oint\limits_{\vartheta  (B)} e^{-\varphi(\lambda) t} B(I-\lambda B)^{-1}f \,d \lambda,\;p\rightarrow\infty.
$$
 The latter relation
  gives us the desired result \eqref{6.5}. Taking into account that  $\mathrm{D}(\varphi)\subset \mathrm{D}(W)$ see Remark \ref{R5.1},   applying Lemma \ref{L5.3}, we obtain   relation \eqref{6.6}.
\end{proof}
\begin{remark}Note that generally the existence and uniqueness  Theorem \ref{T6.3} is based upon  the Theorem 2 \cite{firstab_lit:1kukushkin2021}. The corresponding analogs based upon the Theorems 3,4 \cite{firstab_lit:1kukushkin2021} can be obtained due to the same scheme and the proofs are not worth representing. At the same time the mentioned analogs  can be useful because of   special conditions imposed upon the operator $B$ such as ones formulated in terms of the operator order \cite{firstab_lit:1kukushkin2021}. Here we should also appeal to an artificially constructed normal operator presented in \cite{firstab_lit:2kukushkin2022}.
\end{remark}

\section{Polynomial case}

In this paragraph, we consider a simplest case of the operator function $\varphi(\lambda)=\lambda^{n},\,n\in \mathbb{N},$ although it is covered by Theorem \ref{T6.3} we produce some technicalities that may represent an interest from the point of view of the Fractional calculus theory.

Consider a Cauchy problem
\begin{equation}\label{6.16}
     \mathfrak{D}^{1/\alpha}_{-}u(t)-\sum\limits_{k=-s}^{s} Q_{k} D^{ k \vartheta}_{a+}u(t)=0,\; \vartheta>0,\;
       u(0,x)=f(x)\in L_{2}(J),
\end{equation}
where $Q_{k}=\mathrm{const},$ in the second term we have a linear combination of the Riemann-Liouville fractional integro-differential operators   acting in $L_{2}(J)$ with respect to the variable $x,$ see \cite{firstab_lit:samko1987}. We will call, analogously to the theory of ordinary differential equations,  the second term of  equation \eqref{6.16}   quasi-polynomial and denote it by $P_{s,\vartheta}u.$ Here, we are dealing with a rather wide class of   fractional integro-differential equations what is undoubtedly relevant from the applied point of view as well as from the theoretical one. Further consideration is devoted to the methods of solving Cauchy problem \eqref{6.16}.

 Consider an operator argument $W$    constructed as a closure of the following operator
\begin{equation}\label{6.17}
W_{0}=\eta D^{2}+\xi D^{\beta}_{a+} ,\;0<\beta<1/n,\,\eta<0,\,\xi >0,
\end{equation}
$$
\mathrm{D}(W_{0})=C_{0}^{\infty}(J),
$$
  Note that the hypotheses  $\mathrm{H}1,\mathrm{H}2$  hold regarding the operator,    if we assume that $\mathfrak{H}:=L_{2}(J),\,\mathfrak{H}_{+}:=H^{1}_{0}(J),$
where $L_{2}(J)$ is a Lebesque space,   $H^{k}_{0}(J),\,k\in \mathbb{N}$ is a    subspace  of the Sobolev space $H^{k}(J)$ defined to be the closure of $C_{0}^{\infty}(J).$
It follows from the strictly accretive property of the fractional differential operator (see \cite{kukushkin2019}) and the estimate
\begin{equation}\label{6.18}
\|D^{\beta}_{a+}f\|_{L_{2}}\leq C \|f\|_{H^{1}_{0}},\;f\in C_{0}^{\infty}(J),\,\beta\in(0,1).
\end{equation}
Let us prove that
\begin{equation}\label{6.19}
 -C (D^{2}f,f)_{L_{2}}\leq\mathrm{Re}(Wf,f)_{L_{2}}\leq -C (D^{2}f,f)_{L_{2}},\;f\in C_{0}^{\infty}(J).
\end{equation}
 Using    relation \eqref{6.18} and the Friedrichs inequality, we obtain
$
\mathrm{Re}(D^{\beta}_{a+}f,f)_{L_{2}}\leq\|f\|_{L_{2}}\|D^{\beta}_{a+}f\|_{L_{2}}\leq C\|f\|^{2}_{H_{0}^{1}}=-C(D^{2}f,f)_{L_{2}},\,f\in C_{0}^{\infty}(J),
$
  what gives us  the upper estimate.
The lower estimate follows easily from the  accretive   property of the fractional differential operator of the order less than one.
 Using   relation \eqref{6.19}, the corollary  of the minimax principle, we get
$
-\lambda_{j}(H)\asymp \lambda_{j}(D^{2}),
$
where  $ H$  is a real part   of the operator $ W.$ Therefore, taking into account the well-known fact
$
 \lambda_{j}(D^{2})= -\pi^{2}j^{2}/(b-a)^{2},
$
we get
$
\lambda_{j}(H)\asymp j^{2},
$
it follows that   $\mu(H)>1.$

Consider a case when the second term  the  equation \eqref{6.16} can be represented as follows
\begin{equation}\label{6.20}
P_{s,\vartheta} u= W^{n}u.
\end{equation}
Now we can  study   the Cauchy  problem \eqref{6.16}  by rewriting it  for  the reader convenience in the form
\begin{equation}\label{6.21}
\mathfrak{D}^{1/\alpha}_{-}u(t)=    (\eta D^{2}+\xi D^{\beta }_{a+})^{n}u(t),\;u(0)=f\in \mathrm{D}( W^{n}),\,\alpha\geq1.
\end{equation}
Note that in the  case $\alpha= n=1$ the condition $n\theta<\pi/2\alpha$ of Theorem \ref{T6.3}   is fulfilled by virtue of the condition  $\mathrm{H}2,$   since in Chapter \ref{Ch4} it is noticed that the latter condition   guarantees  the fact  $\Theta(B)\subset \mathfrak{L}_{0}(\theta),\, \pi/4<\theta<\pi/2,$  hence the     conditions $\mathrm{H}2$ are not sufficient to guaranty a value of the semi-angle less than $\pi/4$ and other cases are not covered. At the same time some relevant   results can be obtained in the case of  sufficiently small values of the semi-angle, it gives us a motivation to consider an additional assumption  H3, see the preliminary section.
 This assumption   guarantees  that  we can choose a sufficiently small value of the semi-angle $\theta,$  see Chapter \ref{Ch4}.
 It can be verified directly that the condition $\mathrm{H}3$  holds, i.e.
$$
|\mathrm{Im}( W f,g)_{L_{2}}|\! \leq \! C \|f\|_{ H^{1}_{0}} \|g\|_{L_{2}},\,f,g\in C_{0}^{\infty}(J),
$$
we should apply   \eqref{6.18} and the Cauchy-Schwartz inequality. In accordance with   Theorems   \ref{T6.3},\ref{T6.1}    we are able to represent a solution of
the problem \eqref{6.21}  as  follows
$$
 u(t)=\frac{1}{2 \pi i} \int\limits_{\Gamma(B)}e^{-\lambda^{\alpha n}t}  B (I-\lambda B )^{-1}fd\lambda,
$$
where  the contour $\Gamma(B)$ is defined in Chapter \ref{Ch4}, Paragraph \ref{S4.4.4}. Thus, we have in the reminder a question how  to weaken conditions imposed upon the function $f$ as well as   whether  the representation \eqref{6.20} holds. To answer the questions consider the following reasonings.
 Further, for the sake of the simplicity, we consider a case when $\eta =-1,\,\xi =1.$   This assumption does not restrict generality of reasonings.
Let us show that
$
\mathrm{D}( W)\subset   H^{2}_{0}(J).
$
Using $\mathrm{H}2,$ we have the implication
$$
 f_{k}  \xrightarrow[      W       ]{}f \Longrightarrow  f_{k} \xrightarrow[              ]{H_{0}^{1}} f ,\;\{f_{k}\}_{1}^{\infty}\subset C_{0}^{\infty}(J).
$$
Applying  \eqref{6.18}, we get
$
  D^{\beta}_{a+}f_{k} \xrightarrow[              ]{L_{2}} D^{\beta}_{a+}f.
$
The following fact can be obtained easily,   the proof is omitted
$$
\{f_{k}  \xrightarrow[      W       ]{}f,\;  D^{\beta}_{a+}f_{k} \xrightarrow[              ]{L_{2}} D^{\beta}_{a+}f \}\Longrightarrow D^{2}f_{k}\xrightarrow[              ]{L_{2}}D^{2}f .
$$
Combining the above implications  we obtain the desired result, i.e. $
\mathrm{D}_{0}( W )\subset   H^{2}_{0}(J).
$
Consider a set
$
H_{0+}^{s}(J):=\{f:\,f\in H^{s}(J),\,f^{(k)}(a)=0,\,k=0,1,...,s-1\},\,s\in \mathbb{N}.
$
It is clear that $H^{s}_{0}(J)\subset H_{0+}^{s}(J),$ thus we can    define the operator $W_{+}$ as the extension  of the operator $W$ on the set $H_{0+}^{2}(J),$    we have $W\subset W_{+}.$  Let us show that
$\mathrm{D}(W^{n}_{+})=H_{0+}^{2n}(J).$ Assume that  $f\in H_{0+}^{2n}(J),$ then $f\in \mathrm{D}(W^{n}_{+}),$ it can be verified directly.
If $f\in \mathrm{D}(W^{n}_{+}),$ then in accordance with the definition, we have  $W^{n-1}_{+}f\in H_{0+}^{2}(J).$ It follows that
$
W_{+}g_{1}\in H_{0+}^{2}(J),$ where  $g_{1}=W^{n-2}_{+}f\in H_{0+}^{2}(J).
$
Hence
$
D^{2}g_{1}+D^{\beta}_{a+}g_{1}\in H_{0+}^{2}(J).
$
Applying the operator $I^{2}_{a+}$ to the both sides of the last relation, we easily  get
\begin{equation}\label{6.22}
 g_{1}+I^{2-\beta}_{a+}g_{1}\in H_{0+}^{4}(J).
\end{equation}
 Using the constructive features of relation \eqref{6.22}, we can conclude firstly  $g_{1}\in H_{0+}^{3}(J)$  and due to the same reasonings establish the fact
  $g_{1}\in H_{0+}^{4}(J)$ secondly, what gives us  $W^{n-2}_{+}f\in H_{0+}^{4}(J).$ Using the absolutely analogous reasonings we prove that
  $W^{n-k}_{+}f\in H_{0+}^{2k}(J),\;k= 1,2,...,n.$ Thus, we obtain the desired result.
  Let us show that
\begin{equation}\label{6.23}
W^{n}_{+}f=\sum\limits_{k=0}^{n}(-1)^{n-k}C_{n}^{k} D^{\beta k+2(n-k)}_{a+}f,\;f\in H_{0+}^{2n}(J).
\end{equation}
We need  establish the formula
$
D^{\beta k}_{a+}D^{2(n-k)}f=D^{2(n-k)} D^{\beta k}_{a+}f= D_{a+}^{\beta k+2(n-k)} f,
 \,f\in H_{0+}^{2n}(J),\,
 (k= 1,2,...,n)
$
  for this purpose, in accordance with the  Theorem 2.5   \cite[p.46]{firstab_lit:samko1987},  we should prove  that
$
f\in I^{\beta k +2(n-k)}_{a+}(L_{1})
$
or
$
f =I^{2(n-k)}_{a+}\varphi \;\,\mathrm{a.e.}, \varphi\in I^{\beta k }_{a+}(L_{1}).
$
It is clear that almost everywhere, we have  $f = I^{2n}_{a+}D^{2n}f = I^{2(n-k)}_{a+}D^{2(n-k)}f=   I^{2(n-k)}_{a+}\varphi,$ where $\varphi:=D^{2(n-k)}f.$
 Note that  the conditions of the Theorem 13.2 \cite[p.229]{firstab_lit:samko1987} hold,   i.e. the Marchaud derivative of the function $\varphi$ belongs to $L_{1}(J).$ Hence   $\varphi\in I^{\beta k }_{a+}(L_{1})$  and we obtain the required formula.
  Using the well-known   formulas for linear operators $(A+B)C\supseteq AC+BC,\,C(A+B)=CA+CB,$   applying the Leibniz formula, we obtain \eqref{6.23}.
Now, combining  the obvious inclusion  $\tilde{W}^{n}\subset W^{n}_{+}$ with  \eqref{6.23},  we get
$$
 W^{n}\subset \sum\limits_{k=0}^{n}(-1)^{n-k}C_{n}^{k} D^{\beta k+2(n-k)}_{a+},
$$
  where $C_{n}^{k}$ are  binomial coefficients.
The next question is whether the operator $ W^{n}$ is accretive.  By direct calculation, we have
$$
\mathrm{Re}( W^{n}f,f)_{L_{2}}= \sum\limits_{k=0}^{n} C_{n}^{k}  \mathrm{Re}\left( D^{ \beta k+(n-k)}_{a+}f,D^{n-k}f   \right)_{L_{2}}=
 \sum\limits_{k=0}^{n} C_{n}^{k}  \mathrm{Re}\left(D_{a+}^{\beta k}  g_{k},g_{k}   \right)_{L_{2}}\!\! \geq 0,
$$
where $g_{k}:=D^{n-k}f,\,f\in \mathrm{D}( W^{n}).$
Note that the last inequality holds by virtue of the strictly accretive property of the fractional differential operator of the order less than one (see \cite{kukushkin2019}).
 Thus, the uniqueness and the opportunity to weaken conditions imposed on $f$  follow    from Theorem  \ref{T6.1}.  Here, we should remark that the last   theorem  gives us the fact that the existing solution is unique in the set $\mathrm{D}(W^{n}),$ we should recall the fact $W^{n}\subset W^{n}_{+}.$ However, using the same method, we can   establish   the  uniqueness of the solution  of the   problem  \eqref{6.16}.   Having known the root vectors of the operator  $B,$ applying    \eqref{6.5}, we can represent the obtained solution as a series. Note that using the same reasonings, we can solve the problem \eqref{6.16}, with the second term
$
P_{s,\vartheta}  = - (\eta D^{2}+\xi I^{\beta}_{a+})^{n} .
$
Using the ordinary properties of the homogenous equations, combining the obtained results, we can consider the problem   \eqref{6.16} with the second term of the mixed type, i.e. it contains fractional integrals as well as fractional derivatives. Here, we should remark  that such technical peculiarities (we left them to the reader)  do  not lie in the scope of this chapter  devoted mostly to methods and their applications.\\

\section{ Applications to  concrete operators and physical processes}

Note that  the method considered above allows   to obtain a solution  for the evolution equation with the operator function in the second term where the operator-argument belongs to a sufficiently wide class of operators.  One can find a lot   of examples  in    \cite{firstab_lit:2kukushkin2022} where such well-known operators as the Riesz potential,  the Riemann-Liouville fractional differential operator, the Kipriyanov operator,  the difference operator are  studied, some interesting  examples that cannot be covered by the results established  in  \cite{firstab_lit:Shkalikov A.} are represented    in the paper  \cite{firstab_lit(arXiv non-self)kukushkin2018}.
 The  general approach, implemented in the paper \cite{kukushkin2021a}, creates a theoretical base   to  produce a more abstract example -- a transform of the  m-accretive operator.   We should stress a significance of the last statement   since the  class    contains the   infinitesimal generator of a strongly continuous    semigroup of contractions. Here,   we recall  that  fractional differential operators of the real order can be expressed in terms of the  infinitesimal generator of the corresponding semigroup    \cite{kukushkin2021a}.
 Application of the obtained results  appeals to  electron-induced kinetics of ferroelectrics polarization
switching as the self-similar memory physical systems. The whole point is that  the mathematical
model of the fractal dynamic system includes a Cauchy  problem for the
differential equation of fractional order   considered in the paper \cite{L. Mor}, where    computational schemes for solving the problem
were constructed using Adams-Bashforth-Moulton type predictor-corrector methods.
The stochastic algorithm based on Monte-Carlo method was proposed to simulate the
domain nucleation process during restructuring domain structure in ferroelectrics.

 At the same time the results discussed  in this chapter allow us not only to solve the problem analytically but consider a whole class of   problems for evolution equations of fractional order. As for the mentioned concrete case \cite{L. Mor}, we just need   consider a suitable functional  Hilbert  space and apply Theorem \ref{T6.1} directly. For instance, it can be the  Lebesgue space of square-integrable functions.  Here, we should note that in the case corresponding to a functional Hilbert space we gain more freedom in constructing the theory, thus some modifications of the method can appear but it is an issue for further more detailed study what is not supposed  in the framework of this paper. However, the following example  may be of interest to the reader.

  Goldstein et al. proved in \cite{firstab_lit:Goldstein} several new results  having  replaced the
Laplacian by the Kolmogorov operator
$$
L=\Delta+\frac{\nabla\rho}{\rho}\cdot\nabla,
$$
here $\rho$ is a probability density on $\mathbb{R}^{N}$ satisfying $\rho\in C^{1+\alpha}_{lok}(\mathbb{R}^{N})$
  for some
$\alpha\in (0, 1),\; \rho(x) > 0$ for all $x\in \mathbb{R}^{N}.$ A reasonable question can appear  - Is there possible connections between
the developed theory   and the operator $L$? Indeed, the mentioned operator gives us an opportunity to show brightly   capacity of the spectral theory methods. First of all, let us note the following relation holds $ L=\rho^{-1} W,$
where
$
 W :=    \mathrm{div} \rho  \nabla  .
$
Thus, from the first glance  the right direction of the issue investigation should be connected with the operator composition $ \rho^{-1} W$ since the operator $W$ is uniformly elliptic and satisfies  the following hypotheses  (see\cite{kukushkin2021a})\\

 \noindent ($ \mathrm{H}1 $) There  exists a Hilbert space $\mathfrak{H}_{+}\subset\subset\mathfrak{ H}$ and a linear manifold $\mathfrak{M}$ that is  dense in  $\mathfrak{H}_{+}.$ The operator $V$ is defined on $\mathfrak{M}.$\\

 \noindent  $( \mathrm{H2} )  \,\left|(Vf,g)_{\mathfrak{H}}\right|\! \leq \! C_{1}\|f\|_{\mathfrak{H}_{+}}\|g\|_{\mathfrak{H}_{+}},\,
      \, \mathrm{Re}(Vf,f)_{\mathfrak{H}}\!\geq\! C_{2}\|f\|^{2}_{\mathfrak{H}_{+}} ,\,f,g\in  \mathfrak{M},\; C_{1},C_{2}>0.
$\\

\noindent Apparently,  the results \cite{firstab_lit(arXiv non-self)kukushkin2018}, \cite{kukushkin2021a}, \cite{firstab_lit:1kukushkin2021} can be applied to the operator after an insignificant modification. A couple of words on the difficulties appearing while we study the operator composition.  Superficially, the problem  looks pretty well  but it is not so for the inverse operator (one need prove that it is a resolvent)  is a composition of an unbounded operator and a resolvent of the operator $W,$  indeed  since $R_{W}W=I,$ then formally, we have
$
L^{-1}f= R_{W}\rho f.
$
Most likely,   the general theory created in the papers \cite{firstab_lit(arXiv non-self)kukushkin2018}, \cite{kukushkin2021a} can be adopted to some operator composition but it is  a tremendous work. Instead of that, we may  find a suitable pair of Hilbert spaces that is also not so easy matter. However, we will see! Bellow, we consider a space $\mathbb{R}^{N}$ endowed with the norm
 $$
 |x|=\sqrt{\sum\limits_{k=1}^{n}|x_{k}|^{2}},\,x=(x_{1},x_{2},...,x_{n} )\in \mathbb{R} ^{N}.
 $$
 Assume that there exists a constant $\lambda>2$ such that  the following condition holds
$$
\left\| \rho^{1/\lambda-1}\nabla\rho \right\|_{L_{\infty}(\mathbb{R}^{N})}<\infty,\;\rho^{1/\lambda}(x)=O(1+|x|).
$$
One can verify easily that this condition is not unnatural for it holds for a function $\rho(x)=(1+|x|)^{\lambda},\,x\in \mathbb{R}^{N},\,\lambda\geq 1.$ Let us define a Hilbert space $\mathfrak{H}_{+}$ as a completion of the  set $C_{0}^{\infty}(\mathbb{R}^{N})$ with the norm
$$
\|f\|^{2}_{\mathfrak{H}_{+}}=\|\nabla f\|^{2}_{L_{2}(\mathbb{R}^{N})}+\|f\|^{2}_{L_{2}(\mathbb{R}^{N},\varphi^{-2})},\,\varphi(x)=(1+|x|),
$$
here one can easily see that it is generated by the corresponding inner product. The following result can be obtained as a consequence of the Adams theorem (see Theorem 1 \cite{firstab_lit:1Adams}).
Using the   formula
$$
 \varphi^{\lambda/2}\nabla f=\nabla(f\varphi^{\lambda/2})-f\nabla \varphi^{\lambda/2},\,f=g\varphi^{ -\lambda/2 },\,g\in C_{0}^{\infty}(\mathbb{R}^{N}),
$$
we can easily obtain
$$
\left(\;\int\limits_{\mathbb{R}^{N}}|\nabla (g\varphi^{-\lambda/2 })|^{2}\varphi^{\lambda }dx\right)^{1/2}\leq C\|g\|_{\mathfrak{H}_{+}}.
$$
It is clear that the latter relation can be expanded to the elements of the space $\mathfrak{H}_{+}.$  Note that
$$
\|g\|_{L_{2}(\mathbb{R}^{N},\varphi^{-\lambda})}\leq\|g\|_{L_{2}(\mathbb{R}^{N},\varphi^{-2})},\,g\in L_{2}(\mathbb{R}^{N},\varphi^{-2}), \,\lambda>2.
$$
This relation gives us the inclusion $ \mathfrak{H}_{+}\subset L_{2}(\mathbb{R}^{N},\varphi^{-\lambda}),$ thus  we conclude that $g\varphi^{ -\lambda/2 }\in L_{2}(\mathbb{R}^{N}),\,g\in \mathfrak{H}_{+}.$
 In accordance with  the Theorem 1 \cite{firstab_lit:1Adams}, we conclude that if a  set is bounded in the sense of the norm  $\mathfrak{H}_{+}$ then it  is compact in the sense of the norm $L_{2}(\mathbb{R}^{N},\varphi^{-\lambda}).$
 Thus, we have created a pair of Hilbert spaces $\mathfrak{H}_{-}:=L_{2}(\mathbb{R}^{N},\varphi^{-\lambda})$ and $\mathfrak{H}_{+}$ satisfying the condition of compact embedding i.e. $\mathfrak{H}_{+}\subset\subset \mathfrak{H}_{-}.$ Let us see how can it help us in studying the operator $L.$ Consider an operator $L':=-L+ \eta\rho^{-2/\lambda}I,\,\eta>0,$   we ought to remark here that we need involve additional summand to apply the methods \cite{kukushkin2021a}. The crucial point is related to how to estimate the second term of the operator $-L$ from bellow. Here, we should point out that some peculiar techniques of the theory of functions can be involved. However, along with this we can consider a simplified case (since we have imposed additional conditions upon the function $\rho$) in order to show how the invented method works. The following reasonings are made under the assumption  that the functions  $f,g\in C_{0}^{\infty}(\mathbb{R}^{N}).$ Using simple reasonings based upon  the Cauchy-Schwarz inequality, we get
$$
\left|\,\int\limits_{\mathbb{R}^{N}}\frac{\nabla\rho}{\rho}\cdot\nabla f\,\bar{g}dx\right|\leq \,\int\limits_{\mathbb{R}^{N}}\left|\frac{\nabla\rho}{\rho}\right| |\nabla f|\,|g|dx  \leq    \left\| \rho^{1/\lambda-1}\nabla\rho \right\|_{L_{\infty}(\mathbb{R}^{N})}\frac{1}{2}\left\{\varepsilon\|\nabla f\|^{2}_{L_{2}(\mathbb{R}^{N})}+ \frac{1}{\varepsilon}\|g\|^{2}_{L_{2}(\mathbb{R}^{N},\rho^{-2/\lambda})}  \right\},
$$
where $\varepsilon>0.$ Therefore
$$
-\mathrm{Re}\left(\frac{\nabla\rho}{\rho}\cdot\nabla f ,f \right)_{L_{2}(\mathbb{R}^{N})}\geq -\left\| \rho^{1/\lambda-1}\nabla\rho \right\|_{L_{\infty}(\mathbb{R}^{N})}\frac{1}{2}\left\{\varepsilon\|\nabla f\|^{2}_{L_{2}(\mathbb{R}^{N})}+ \frac{1}{\varepsilon}\|g\|^{2}_{L_{2}(\mathbb{R}^{N},\rho^{-2/\lambda})}  \right\}.
$$
 Choosing $\eta,\varepsilon,$ we  get easily
$$
\mathrm{Re}\left(L' f ,f \right)_{L_{2}(\mathbb{R}^{N})}\geq C\|f\|^{2}_{\mathfrak{H}_{+}},\,C>0.
$$
Using the above estimates, we obtain
$$
\left|(L' f ,g  )_{L_{2}(\mathbb{R}^{N})}\right|\leq  \|\nabla f\| _{L_{2}(\mathbb{R}^{N})}\|\nabla g\| _{L_{2}(\mathbb{R}^{N})}+\left\| \rho^{1/\lambda-1}\nabla\rho \right\|_{L_{\infty}(\mathbb{R}^{N})}
   \|\nabla f\|_{L_{2}(\mathbb{R}^{N})} \|g\|_{L_{2}(\mathbb{R}^{N},\rho^{-2/\lambda})}
 $$
 $$
  +\,\eta\|f\| _{L_{2}(\mathbb{R}^{N},\rho^{-2/\lambda})}\|g\| _{L_{2}(\mathbb{R}^{N},\rho^{-2/\lambda})}\leq C \|f\|_{\mathfrak{H}_{+}}\|g\|_{\mathfrak{H}_{+}},\,C >0.
$$
Thus, we have a fulfilment of the   hypothesis $\mathrm{H}2$ \cite{kukushkin2021a}. Taking into account a fact  that a  negative space  $L_{2}(\mathbb{R}^{N},\varphi^{-\lambda})$ is involved,
we are being forced  to involve a modification of the hypothesis  $\mathrm{H}1$ \cite{kukushkin2021a} expressed  as follows. There  exist pairs of Hilbert spaces   $ \mathfrak{H}\subset \mathfrak{H}_{-},\,\,\mathfrak{H}_{+}\subset \subset \mathfrak{H}_{-},\,\mathfrak{H}:=L_{2}(\mathbb{R}^{N})$  and a linear manifold $\mathfrak{M}:=C_{0}^{\infty}(\mathbb{R}^{N})$ that is  dense in  $\mathfrak{H}_{+}.$ The operator $L'$ is defined on $\mathfrak{M}.$  However, we can go further and  modify a norm $\mathfrak{H}_{+}$ adding a summand, in this case the considered operator can be changed, we have
$$
\|f\|^{2}_{\mathfrak{H}_{+}}:=\|\nabla f\|^{2}_{L_{2}(\mathbb{R}^{N})}+\|f\|^{2}_{L_{2}(\mathbb{R}^{N},\psi)},\,\psi(x)=(1+|x|)^{-2}+1,\,L':=L+\eta I,\,\eta>0.
$$
Implementing the same reasonings one can prove that in this case the hypothesis $\mathrm{H}2$ \cite{kukushkin2021a} is fulfilled, the modified analog of the hypothesis  $\mathrm{H}1$ \cite{kukushkin2021a} can be formulated as follows.\\

 \noindent ($ \mathrm{H}1^{\ast} $) There  exists a chain  of Hilbert spaces   $\mathfrak{H}_{+}\subset  \mathfrak{H} \subset \mathfrak{H}_{-},\,\mathfrak{H}_{+}\subset\subset \mathfrak{H}_{-} $ and a linear manifold $\mathfrak{M}$ that is  dense in  $\mathfrak{H}_{+}.$ The operator $L'$ is defined on $\mathfrak{M}.$\\

However, we have $\mathfrak{H}_{+}\subset\subset \mathfrak{H}_{-}$ instead of the required inclusion $\mathfrak{H}_{+}\subset\subset \mathfrak{H}.$  This inconvenience can stress a peculiarity of the chosen method, at the same time the central point of the theory - Theorem 1 \cite{kukushkin2021a}  can be reformulated under newly obtained conditions corresponding to both variants of the operator $L'.$      The further step is how to calculate order of the real component $\mathfrak{Re}L':=(L'+L'^{\ast})/2$ (more precise  definition can be seen  see in the paper \cite{kukushkin2021a}).  Formally, we can avoid the appeared difficulties connected with the fact that the set $\mathbb{R}^{N}$ is not bounded since  we can   referee  to  the Fefferman concept  presented in the monograph
\cite[p.47]{firstab_lit:Rosenblum}, in accordance with which we can choose such an unbounded subset of $\mathbb{R}^{N}$ that the relation
 $\lambda_{j}( \mathfrak{Re}  L')\asymp j^{2/N}$ holds, where the symbol $\lambda_{j}$ denotes an eigenvalue.   It gives us  $\mu(\mathfrak{Re} L')=2/N,$ where the symbol $\mu$ denotes   order of the real component of the  operator $L'$(see \cite{kukushkin2021a}).      Thus, we left this question to the reader for  a more detailed  study    and reasonably allow ourselves to  assume that the operator $L'$ has a finite non-zero order. Having obtained analog of Theorem 1 \cite{kukushkin2021a} and order  of the real component of the operator $L'$ we have a key to the theory created in the papers \cite{firstab_lit:1kukushkin2021},\cite{firstab_lit:2kukushkin2022},\cite{firstab_lit(axi2022)}. Now, we can consider a Cauchy problem for the evolution equation with the operator $L'$ in the second term as well as a function of the operator $L'$ in the second term what leads us to the integro-differential evolution equation - it corresponds to an operator function having   finite  principal and major parts of the   Laurent series.

One more example of a non-selfadjoin operator that is not completely subordinated in the sense of forms (see \cite{firstab_lit:Shkalikov A.}, \cite{firstab_lit(arXiv non-self)kukushkin2018}) is given bellow.
   Consider a  differential operator acting in the complex Sobolev  space
$$
\mathcal{L}f := (c_{k}f^{(k)})^{(k)} + (c_{k-1}f^{(k-1)})^{(k-1)}+...+  c_{0}f,
$$
$$
\mathrm{D}(\mathcal{L}) = H^{2k}(I)\cap H_{0}^{k}(I),\,k\in \mathbb{N},
$$
where    $I: = (a, b) \subset \mathbb{R},$ the   complex-valued coefficients
$c_{j}(x)\in C^{(j)}(\bar{I})$ satisfy the condition $  {\rm sign} (\mathrm{Re} c_{j}) = (-1)^{j} ,\, j = 1, 2, ..., k.$
 Consider  a linear combination of   the  Riemann-Liouville  fractional differential   operators
  (see \cite[p.44]{firstab_lit:samko1987})   with the  constant  real-valued  coefficients
$$
\mathcal{D}f:=p_{n}D_{a+}^{\alpha_{n}}+q_{n}D_{b-}^{\beta_{n}}+p_{n-1}D_{a+}^{\alpha_{n-1}}+q_{n-1}D_{b-}^{\beta_{n-1}}+...+
p_{0}D_{a+}^{\alpha_{0}}+q_{0}D_{b-}^{\beta_{0}},
$$
$$
\mathrm{D}(\mathcal{D}) = H^{2k}(I)\cap H_{0}^{k}(I),\,n\in \mathbb{N},
$$
where $\alpha_{j},\beta_{j}\geq 0,\,0 \leq [\alpha_{j}],[\beta_{j}] < k,\, j = 0, 1, ..., n.,\;$
\begin{equation*}
 q_{j}\geq0,\;{\rm sign}\,p_{j}= \left\{ \begin{aligned}
  (-1)^{\frac{[\alpha_{j}]+1}{2}},\,[\alpha_{j}]=2m-1,\,m\in \mathbb{N},\\
\!\!\!\!\!\!\! \!\!\!\!(-1)^{\frac{[\alpha_{j}]}{2}},\;\,[\alpha_{j}]=2m,\,\,m\in \mathbb{N}_{0}   .\\
\end{aligned}
\right.
\end{equation*}
The following result is represented in the paper  \cite{firstab_lit(arXiv non-self)kukushkin2018},    consider the operator
$$
G=\mathcal{L}+\mathcal{D},
$$
$$
\mathrm{D}(G)=H^{2k}(I)\cap H_{0}^{k}(I).
$$
and suppose  $\mathfrak{H} := L_{2}(I),\, \mathfrak{H}^{+} := H_{0}^{k}(I),\,\mathfrak{M}:=C_{0}^{\infty}(I),$ then we have that the operator $G$ satisfies  the
      conditions  H1, H2. Using  the  minimax principle for estimating eigenvalues, we   can easily see that the operator $\mathfrak{Re} G  $  has   non-zero order.
Hence, we can successfully  apply  Theorem 1 \cite{kukushkin2021a}  to the   operator $G,$ in accordance with which the resolvent of the operator $G$ belongs to the Schatten-von Neumann class $\mathfrak{S}_{s}$ with the value of the index $0<s<\infty$ defined by the formula given in Theorem 1 \cite{kukushkin2021a}. Thus, it gives us an opportunity to apply Theorem 3   to the operator.

A couple of words on condition  $ \mathrm{H}1  $  in the context of  operators generating semigroups. Assume that an operator  $-A $ acting in a separable   Hilbert space $\mathfrak{H}$ is   the infinitesimal  generator  of a $C_{0}$ semigroup   such that $A^{-1}$ is compact.  By virtue of   Corollary 2.5 \cite[p.5]{Pasy}, we have that the operator $A$    is    densely defined and closed.
Let us  check fulfilment of    condition $\mathrm{H1},$ consider a   separable  Hilbert   space  $
 \mathfrak{H}_{A}:= \big \{f,g\in \mathrm{D}(A),\,(f,g)_{ \mathfrak{H}_{A}}=(Af,Ag)_{\mathfrak{H} } \big\},
$
the fact that $\mathfrak{H}_{A}$ is separable  follows from   properties of the  energetic space.
Note that         since $A^{-1}$ is compact, then   we conclude that the following   relation holds
 $\|f\|_{\mathfrak{H}}\leq \|A^{-1}\| \cdot \|Af\|_{\mathfrak{H}},\,f\in \mathrm{D}(A) $ and    the embedding provided by this   inequality is compact. Thus, we have obtained in the natural way a pair of Hilbert spaces such that $\mathfrak{H}_{A}\subset\subset \mathfrak{H}.$ We may say that this general property of   infinitesimal  generators  is not so valuable  for requires a rather strong and unnatural  condition of compactness of the inverse operator. However, if we additionaly deal with the semigroup of contractions then   we can formulate a significant result (see Theorem 2 \cite{kukushkin2021a}) allowing us to study spectral properties of the infinitesimal  generator transform
$$
 Z^{\alpha}_{G,F}(A):= A^{\ast}GA+FA^{\alpha},\,\alpha\in [0,1),
$$
where the symbols  $G,F$  denote  operators acting in $\mathfrak{H}.$ Having analyzed the proof of the Theorem   2 \cite{kukushkin2021a} one can easily see that the condition   of contractions can be omitted in the case $\alpha=0.$   The   Theorem 5 \cite{kukushkin2021a} gives us a tool to describe spectral properties  of the  transform $Z^{\alpha}_{G,F}(A).$ Particularly, we can establish the order  of the transform and its belonging to the  Schatten-von Neumann class of the convergence  exponent by virtue of the Theorem 3 \cite{kukushkin2021a}. Thus, having known the index of the Schatten-von Neumann class of the convergence  exponent, we can apply Theorems \ref{T6.1}, \ref{T6.2}, \ref{T6.3} to the transform.

\subsection{Kipriyanov operator}

  We assume that $\Omega$ is a  convex domain of the  $m$ -  dimensional Euclidean space $\mathbb{E}^{m}$, $P$ is a fixed point of the boundary $\partial\Omega,$
$Q(r,\mathbf{e})$ is an arbitrary point of $\Omega;$ we denote by $\mathbf{e}$   a unit vector having a direction from  $P$ to $Q,$ denote by $r=|P-Q|$   the Euclidean distance between the points $P,Q,$ and   use the shorthand notation    $T:=P+\mathbf{e}t,\,t\in \mathbb{R}.$
We   consider the Lebesgue  classes   $L_{p}(\Omega),\;1\leq p<\infty $ of  complex valued functions.  For the function $f\in L_{p}(\Omega),$    we have
\begin{equation}\label{6.24}
\int\limits_{\Omega}|f(Q)|^{p}dQ=\int\limits_{\omega}d\chi\int\limits_{0}^{d(\mathbf{e})}|f(Q)|^{p}r^{m-1}dr<\infty,
\end{equation}
where $d\chi$   is an element of   solid angle of
the unit sphere  surface (the unit sphere belongs to $\mathbb{E}^{m}$)  and $\omega$ is a  surface of this sphere,   $d:=d(\mathbf{e})$  is the  length of the  segment of the  ray going from the point $P$ in the direction
$\mathbf{e}$ within the domain $\Omega.$
Without   loss   of   generality, we consider only those directions of $\mathbf{e}$ for which the inner integral on the right-hand  side of the equality \eqref{6.24} exists and is finite. It is  the well-known fact that  these are almost all directions.
 Denote by  $D_{i}f$  a distributional derivative of the function $f$ with respect to a coordinate variable with index   $1\leq i\leq m.$   For convenience, we use the Einstein   convention    $ a^{i}b_{i}:=\sum^{m}_{i=1}a^{i}b_{i}.$

Here, we study the case $\beta \in (0,1).$ Assume that  $\Omega\subset \mathbb{E}^{m}$ is  a convex domain, with a sufficient smooth boundary ($ C ^{3}$ class)   of the m-dimensional Euclidian space. For the sake of the simplicity, we consider that $\Omega$ is bounded, but  the results  can be extended     to some type of    unbounded domains.
In accordance with the definition given in  the paper  \cite{kukushkin2021a}, we consider the directional  fractional integrals.  By definition, put
$$
 (\mathfrak{I}^{\beta }_{0+}f)(Q  ):=\frac{1}{\Gamma(\beta )} \int\limits^{r}_{0}\frac{f (P+t \mathbf{e} )}{( r-t)^{1-\beta }}\left(\frac{t}{r}\right)^{m-1}\!\!\!\!dt,\,(\mathfrak{I}^{\beta }_{d-}f)(Q  ):=\frac{1}{\Gamma(\beta )} \int\limits_{r}^{d }\frac{f (P+t\mathbf{e})}{(t-r)^{1-\beta }}\,dt,
$$
$$
\;f\in L_{p}(\Omega),\;1\leq p\leq\infty,
$$
  where $\Gamma(\beta )$ is the gamma function.
Also,     we   consider auxiliary operators,   the so-called   truncated directional  fractional derivatives    (see \cite{firstab_lit:1kukushkin2018}).  By definition, put
 \begin{equation*}
 ( \mathfrak{D} ^{\beta}_{0+,\,\varepsilon}f)(Q)=\frac{\beta }{\Gamma(1-\beta )}\int\limits_{0}^{r-\varepsilon }\frac{ f (Q)r^{m-1}- f(P+\mathbf{e}t)t^{m-1}}{(  r-t)^{\beta  +1}r^{m-1}}   dt+\frac{f(Q)}{\Gamma(1-\beta )} r ^{-\beta },\;\varepsilon\leq r\leq d ,
 $$
 $$
 (\mathfrak{D}^{\beta }_{0+,\,\varepsilon}f)(Q)=  \frac{f(Q)}{\varepsilon^{\beta }}  ,\; 0\leq r <\varepsilon ;
\end{equation*}
\begin{equation*}
 ( \mathfrak{D }^{\beta }_{d-,\,\varepsilon}f)(Q)=\frac{\beta }{\Gamma(1-\beta )}\int\limits_{r+\varepsilon }^{d }\frac{ f (Q)- f(P+\mathbf{e}t)}{( t-r)^{\beta  +1}} dt
 +\frac{f(Q)}{\Gamma(1-\beta )}(d-r)^{-\beta },\;0\leq r\leq d -\varepsilon,
 $$
 $$
  ( \mathfrak{D }^{\beta }_{d-,\,\varepsilon}f)(Q)=      \frac{ f(Q)}{\beta } \left(\frac{1}{\varepsilon^{\beta }}-\frac{1}{(d -r)^{\beta } }    \right),\; d -\varepsilon <r \leq d .
 \end{equation*}
  Now, we can  define  the directional   fractional derivatives as follows
 \begin{equation*}
 \mathfrak{D }^{\beta }_{0+}f=\lim\limits_{\stackrel{\varepsilon\rightarrow 0}{ (L_{p}) }} \mathfrak{D }^{\beta }_{0+,\varepsilon} f  ,\;
  \mathfrak{D }^{\beta }_{d-}f=\lim\limits_{\stackrel{\varepsilon\rightarrow 0}{ (L_{p}) }} \mathfrak{D }^{\beta }_{d-,\varepsilon} f ,\,1\leq p\leq\infty.
\end{equation*}
The properties of these operators  are  described  in detail in the papers  \cite{kukushkin2019,firstab_lit:1kukushkin2018}.  We suppose  $\mathfrak{I}^{0}_{0+} =I.$ Nevertheless,   this    fact can be easily proved dy virtue of  the reasonings  corresponding to the one-dimensional case, and is  given in \cite{firstab_lit:samko1987}. We also consider integral operators with a weighted factor (see \cite[p.175]{firstab_lit:samko1987}) defined by the following formal construction
$$
 \left(\mathfrak{I}^{\beta }_{0+}\phi f\right) (Q  ):=\frac{1}{\Gamma(\beta )} \int\limits^{r}_{0}
 \frac{(\phi f) (P+t\mathbf{e})}{( r-t)^{1-\beta }}\left(\frac{t}{r}\right)^{m-1}\!\!\!\!dt,
$$
where $\phi$ is a real-valued  function. Consider a linear combination of the uniformly elliptic operator, which is written in the divergence form, and
  a composition of a   fractional integro-differential  operator, where the fractional  differential operator is understood as the adjoint  operator  regarding  the Kipriyanov operator  (see  \cite{kukushkin2019,firstab_lit:kipriyanov1960,firstab_lit:1kipriyanov1960})
\begin{equation*}
 L  :=-  \mathcal{T}  \, +\mathfrak{I}^{\sigma}_{ 0+}\phi\, \mathfrak{D}  ^{ \beta  }_{d-},
\; \sigma\in[0,1) ,
 $$
 $$
   \mathrm{D}( L )  =H^{2}(\Omega)\cap H^{1}_{0}( \Omega ),
  \end{equation*}
where
$\,\mathcal{T}:=D_{j} ( a^{ij} D_{i}\cdot),\,i,j=1,2,...,m,$
under    the following  assumptions regarding        coefficients
\begin{equation} \label{6.25}
     a^{ij}(Q) \in C^{2}(\bar{\Omega}),\, \mathrm{Re} a^{ij}\xi _{i}  \xi _{j}  \geq   \gamma_{a}  |\xi|^{2} ,\,  \gamma_{a}  >0,\,\mathrm{Im }a^{ij}=0 \;(m\geq2),\,
 \phi\in L_{\infty}(\Omega).
\end{equation}
 Note that in the one-dimensional case, the operator $\mathfrak{I}^{\sigma }_{ 0+} \phi\, \mathfrak{D}  ^{ \beta  }_{d-}$ is reduced to   a  weighted fractional integro-differential operator  composition, which was studied properly  by many researchers
    \cite{firstab_lit:2Dim-Kir,firstab_lit:15Erdelyi,firstab_lit:9McBride,firstab_lit:nakh2003}, more detailed historical review  see  in \cite[p.175]{firstab_lit:samko1987}.  In accordance with the Theorem 9 \cite{kukushkin2021a}, we claim  that the conditions $\mathrm{H}1,\mathrm{H}2$ are fulfilled, if $\gamma_{a}$ is sufficiently large in comparison with $\|\phi\|_{\infty},$  where we put $\mathfrak{M}:=C_{0}^{\infty}(\Omega).$  Note that the order $\mu$ of the operator $H$ can be evaluated easily  through  the order of the regular differential operator   and since the latter can be found by the methods described in \cite{firstab_lit:Rosenblum}. More precisely, we have
$$
C(\mathfrak{Re}\mathcal{T}f,f)_{\mathfrak{H}} \leq(Hf,f)_{\mathfrak{H}}\leq C (\mathfrak{Re}\mathcal{T}f,f)_{\mathfrak{H}},\,f\in C_{0}^{\infty}(\Omega).
$$
Applying the minimax principle, we get $\lambda_{j}(H)\asymp \lambda_{j}(\mathfrak{Re}\mathcal{T}).$
Using  the well-known formula for regular differential operators   $\lambda_{j}(\mathfrak{Re}\mathcal{T})\asymp j^{2/m}$ (see  \cite{firstab_lit:Rosenblum}), we get  $\mu(H)=2/m.$ The results of paragraph \ref{S2.6.1} allow to obtain the Schatten index due to the order of the real component. A special interest may appear by virtue of  the fact    that the   composition of a   fractional integro-differential  operator has been involved.    Since the hypothesis $\mathrm{H}1,\mathrm{H}2$ guarantee the  fulfilment of the condition   $ \Theta(B) \subset   \mathfrak{L}_{0}(\theta),\, \theta n< \pi/2\alpha$   in the case  only $\alpha=n=1,$    then to consider higher powers, we should verify   the fulfilment of the  condition $\mathrm{H}3.$   Firstly,   we should  assume that $a^{ij}=a^{ji}.$ Secondly, we should use relation (27) \cite{kukushkin2021a} in accordance with which, we have
$$
 \left|\left( \mathfrak{I}^{\sigma }_{0+}\phi\,\mathfrak{D}^{\alpha}_{d-}  f,g\right)_{ L _{2 }}\right| \leq C\|f\|_{H_{0}^{1}}\|g\|_{L_{2}},\,f,g\in C^{\infty}_{0}(\Omega).
$$
Thus, the additional condition $\mathrm{H}3$ holds.
  Then in accordance with Theorems \ref{T6.3},\ref{T6.1}   we can claim that  there exists a solution of  the problem \eqref{6.1},
where $W$ is a restriction of $L$ on the set $C_{0}^{\infty}(\Omega),$  the coefficients \eqref{6.25}   are sufficiently smooth to guarantee  the fact  the second term of the equation  \eqref{6.1} (in this case    an operator function of the power type)    has a sense. Note that the solvability of the   uniqueness problem as well as the opportunity to extend the initial condition depends on the accretive property of the operator $ W^{n}.$ The last   problem can be studied by the methods similar to the ones used in the previous paragraphs.      \\

\subsection{ Riesz potential}

Consider a   space $L_{2}(\Omega),\,\Omega:=(-\infty,\infty).$      We denote by $H^{2, \varsigma}_{0}(\Omega)$ the completion of the set  $C^{\infty}_{0}(\Omega)$  with the norm
$$
\|f\|_{H^{2,\varsigma}_{0}(\Omega)}=\left\{\|f\|^{2}_{L_{2}(\Omega) }+\|D^{2}f \|^{2}_{L_{2}(\Omega,\omega^{\varsigma})} \right\}^{1/2},\,  \varsigma\in \mathbb{R},
$$
where
$$
L_{2}(\Omega,\omega^{\varsigma}):=
\left\{f:\;\|f\|_{L_{2}(\Omega,\omega^{\varsigma})}<\infty,
\;\|f\|_{L_{2}(\Omega,\omega^{\varsigma})}:=\left(\int\limits_{\Omega}|f(t)|^{2}\omega^{\varsigma}(t)dt\right)^{1/2} \right\},\,\omega(x):=   1+|x|,
$$
  the above  integral is understood in the Lebesque sense.
Let us notice the following fact  (see Theorem 1 \cite{firstab_lit:1Adams}), if $\varsigma>4,$ then
$
H^{2,\varsigma}_{0}(\Omega)\subset\subset L_{2}(\Omega).
$
Consider a Riesz potential
$$
I^{\beta}f(x)=B_{\beta}\int\limits_{-\infty}^{\infty}f (s)|s-x|^{\beta-1} ds,\,B_{\beta}=\frac{1}{2\Gamma(\beta)  \cos  (\beta \pi / 2)   },\,\beta\in (0,1),
$$
where $f$ is in $L_{p}(\Omega),\,1\leq p<1/\beta.$
It is  obvious that
$
I^{\beta}f= B_{\beta}\Gamma(\beta) (I^{\beta}_{+}f+I^{\beta}_{-}f),
$
where
$$
I^{\beta}_{\pm}f(x)=\frac{1}{\Gamma(\beta)}\int\limits_{0}^{\infty}f (s\mp x) s ^{\beta-1} ds,
$$
these operators are known as fractional integrals on the  whole  real axis   (see \cite[p.94]{firstab_lit:samko1987}). Assume that the following  condition holds
 $ \sigma/2 + 3/4<\beta<1 ,$ where $\sigma$ is a non-negative  constant.
 Following the idea of the   monograph \cite[p.176]{firstab_lit:samko1987},
 consider a sum of a differential operator and  a composition of    fractional integro-differential operators
\begin{equation*}
 W   :=  D^{2} a  D^{2}  +   I^{\sigma}_{+}\,\xi \,I^{2(1-\beta)}D^{2}+\delta I,\;\mathrm{D}(W)=C^{\infty}_{0}(\Omega),
\end{equation*}
where
$
\xi(x)\in L_{\infty}(\Omega),\, a(x)\in L_{\infty}(\Omega)\cap C^{ 2  }( \Omega ),\, \mathrm{Re}\,a(x) >\gamma_{a}(1+|x|)^{5}.
$
Let $\Omega':=[0,\infty),$ consider the functions  $u(t,x),\,t\in \Omega',\,x\in \Omega.$ Similarly to the previous paragraph, we will consider functional spaces with respect to the variable $x$ and we will assume that if $u$ belongs to a functional space then this  fact  holds for all values of the variable $t,$ wherewith all standard topological  properties of a space as completeness, compactness etc. remain correctly defined.
Consider a Cauchy problem \eqref{6.1}  in the corresponding terms.
 Notice that in accordance with the results  \cite{kukushkin2021a}, we can claim that  if $ \min\{\gamma_{a},\delta\}$ is  sufficiently large, then    the  hypothesis $\mathrm{H}1,\mathrm{H}2$ hold  regarding: the operator $\tilde{W} ,$ the set  $C^{\infty}_{0}(\Omega),$ the spaces $L_{2}(\Omega),H^{2,\,5}_{0}(\Omega),$    more precisely we should put
$
\mathfrak{H}:=L_{2}(\Omega),\,\mathfrak{H}_{+}:=H^{2,\,5}_{0}(\Omega),\,              \mathfrak{M}:=C^{\infty}_{0}(\Omega).
$
Thus, in accordance with the condition  $\mathrm{H}2,$ we have
$$
(H f,f)_{L_{2}}=\mathrm{Re}( W  f,f)_{L_{2}}\geq \gamma_{a}\|f\|^{2}_{H^{2,5}_{0}}= C( D^{2}wD^{2} f,f)_{L_{2}}+C( f,f)_{L_{2}},\,f\in C^{\infty}_{0}(\Omega).
$$
where  $w(x)=(1+|x|)^{5}.$ Let us consider the operator $G=D^{2}wD^{2}+\delta I,\, \mathrm{D}(B)=C^{\infty}_{0}(\Omega).$ It is clear that by virtue of the minimax principle, we can estimate the eigenvalues of the operator $W$ via estimating the eigenvalues of the operator $G.$ Hence, we have  come to the problem of estimating the eigenvalues of the singular  operator. Here, we should point out that there exists the Fefferman concept that covers such a kind of problems. For instance, the Rozenblyum result is presented in the monograph
\cite[p.47]{firstab_lit:Rosenblum}, in accordance with which we can choose such an unbounded subset of $\mathbb{R}$ that the relation $\lambda_{j}( G )\asymp j^{4}$ holds.   Thus, we left this question to the reader for  a more detailed  study    and reasonably allow ourselves to  assume that the  condition $\mu(H)=4$ holds. Note that  in the case  $n=1,$ in accordance with the  Theorem 3 \cite{kukushkin2021a}, Theorem \ref{T6.3},  Theorem \ref{T6.1},   we are able to present a solution of  the problem \eqref{6.1}  in the form \eqref{6.5}.
Now, assume additionally that $\mathrm{Im} a=0,$ then $D^{2} a  D^{2}$ is selfadjoint. It follows that the operator $ W$ is selfadjoint.  Applying formula (40)  \cite{kukushkin2021a}, we get
$$
 \|I^{\sigma}_{+}\xi I^{2(1-\alpha)} D^{2}f\|_{L_{2}}\ \leq C\|     f \|_{H^{2,5}_{0}},\,f\in H_{0}^{2,5}(\Omega).
$$
Using this  fact, applying the Cauchy-Schwartz inequality, we conclude that    the condition $\mathrm{H}3$  holds. Moreover, since $ W$ is selfadjoint, then we  can easily prove that
 $
\mathrm{Re}( W^{n} f,f)_{\mathfrak{H}}\geq0,\;n=1,2,...\, .
$
Therefore,   applying Theorem \ref{T6.1},   we can claim that  there exists a solution of the  problem \eqref{6.1},
where   the coefficients  of the operator $W$ are sufficiently smooth to guaranty  the fact  the second term of the equation  \eqref{6.1}   has a sense.
   Moreover,   we can assume  that $f\in \mathfrak{H}$ and claim that the existing solution is unique in the case $\alpha=1.$

\subsection{ Difference operator}

The approach implemented in studying the difference  operator is remarkable due to the appeared opportunity to set   the problem within the framework of the created theory having constructed a  suitable perturbation of the   operator composition.
Consider a   space $L_{2}(\Omega),\,\Omega:=(-\infty,\infty),$   define a family of operators
$$
T_{t}f(x):=e^{-c t}\sum\limits_{k=0}^{\infty}\frac{(c \,t)^{k}}{k!}f(x-d\mu),\,f\in L_{2}(\Omega),\;c,d>0,\; t\geq0,
$$
where convergence is understood in the sense of $L_{2}(\Omega)$ norm. In accordance with the Lemma 6 \cite{kukushkin2021a}, we know that  $T_{t}$ is a $C_{0}$ semigroup of contractions, the corresponding  infinitesimal generator and its adjoint operator are defined by the following expressions
 $$
Yf(x)=c[f(x)-f(x-d)],\,Y^{\ast}f(x)=c[f(x)-f(x+d)],\,f\in L_{2}(\Omega).
$$
Let us find a representation for fractional powers of the operator $Y.$ Using formula  (45) \cite{kukushkin2021a}, we get
$$
   Y^{\beta}f=\sum\limits_{k=0}^{\infty}M_{k}f(x-kd), \,f\in L_{2}(\Omega),
   \,M_{k}= -\frac{\beta\Gamma(k -\beta)}{k!\Gamma(1 -\beta)}c^{\,\beta},\,\beta\in (0,1).
   $$
 We need the following theorem (see the Theorem 14 \cite{kukushkin2021a}).\\

\begin{teo}\label{T5a} Assume that  $Q$ is a   closed operator acting in $L_{2}(\Omega),\,Q^{-1}\in \mathfrak{S}_{\!\infty}(L_{2}),$ the operator $N$ is strictly accretive, bounded, $\mathrm{R}(Q)\subset \mathrm{D}(N).$ Then
a perturbation
$$
L:= Y^{\ast}\!a Y+b Y^{\beta}+ Q^{\ast}N Q ,\;a,b\in L_{\infty}(\Omega),
$$
   satisfies conditions  H1 and H2, if  $\gamma_{N}>\sigma\|Q^{-1}\|^{2},$
where we put $\mathfrak{M}:=\mathrm{D}_{0}(Q),$
$$
 \sigma= 4c\|a\|_{L_{\infty}}+\|b\|_{L_{\infty}}\frac{\beta c^{\,\beta}   }{\Gamma(1 -\beta)}
 \sum\limits_{k=0}^{\infty}\frac{ \Gamma(k -\beta)}{k! }.
$$
\end{teo}
Observe that by virtue of the made assumptions regarding   $Q,$ we have $\mathfrak{H}_{Q}\subset\subset L_{2}(\Omega).$ We have chosen  the space  $L_{2}(\Omega)$ as a space $\mathfrak{H}$ and the space  $\mathfrak{H}_{Q} $ as a space $\mathfrak{H}_{+}.$   Applying  the condition   $\mathrm{H}2,$   we get
$$
C (Q^{\ast}N Qf,f)_{\mathfrak{H}}\leq(Hf,f)_{\mathfrak{H}}\leq C (Q^{\ast}N Qf,f)_{\mathfrak{H}},\;f\in \mathrm{D}_{0}(Q),
$$
where $H$ is a real part of $W.$ Therefore, by virtue of the minimax principle, we get $\lambda_{j}(H)\asymp \lambda_{j}(Q^{\ast}N Q).$
Hence $\mu(H)=\mu(Q^{\ast}N Q).$
Thus, we have naturally  come to the significance  of the operator $Q$ and the remarkable fact   that we can fulfill the conditions of  Theorems  \ref{T6.1},\ref{T6.3} choosing the operator $Q$ in the artificial way.
  Applying Theorem  \ref{T6.1}, we can claim that  there exists a solution of the  problem \eqref{6.1},
where $W$ is a restriction of $L$ on the set $\mathfrak{M}$ (see Introduction),  functions $a,b$ are sufficiently smooth to guaranty  the fact  the right-hand side of \eqref{6.1} has a sense. The extension  of the initial conditions on the whole space $\mathfrak{H},$ as well as  solvability of the  uniqueness problem can be implemented in the case  when the operator $W^{n}$  is accretive. In its own turn, it is clear that the particular methods,   to establish the  accretive property, can differ and may depend on the concrete form of the operator $Q.$\\

\subsection{  Artificially constructed normal operator }

In this paragraph we consider an operator class  which cannot be completely  studied by methods  \cite{firstab_lit:1Lidskii}, at the same time   Theorem 3 \cite{firstab_lit:2kukushkin2022} (Theorem \ref{T4.7}) gives us a rather relevant result. Our aim is to construct a normal operator $N$ being satisfied the  Theorem 3 \cite{firstab_lit:2kukushkin2022}  conditions, such that  $N\in \mathfrak{S}_{p},\,0<p<1$ and at the same time    $N\,\overline{\in}\,  \mathfrak{S}_{\rho},\,\rho=\inf p.$
We use  Example \ref{E5.1}, Chapter \ref{Ch5} as a base and consider a sequence $\mu_{n}=n^{\kappa}\ln^{\kappa} (n+q) \cdot \ln^{\kappa}\ln (n+q),\,q>e^{e}-1,\,n\in \mathbb{N}.$  In Chapter \ref{Ch5} it is shown that
 $$
 (\ln^{\kappa+1}x)'_{\mu_{n} }  =o(  n^{-\kappa}   ),\,  \kappa>0,
$$
and at the same time
$$
\sum\limits_{n=1}^{\infty}\frac{1}{|\mu_{n}|^{1/\kappa}}=\infty.
$$

Consider the abstract separable Hilbert space $\mathfrak{H}$ and an operator $N$ acting in the space as follows
$$
Nf=\sum\limits_{n=1}^{\infty}\lambda_{n}f_{n}e_{n},\;f_{n}=(f,e_{n})_{\mathfrak{H}},\;
 \lambda_{n}=\mu_{n}+i \eta_{n},
$$
where $\{e_{n}\}_{1}^{\infty}\subset \mathfrak{H}$ is an orthonormal basis, the sequence  $\{\mu_{n}\}_{1}^{\infty}$  is defined above, the  sequence $\{\eta_{n}\}_{1}^{\infty}$ is satisfied the following condition $|\eta_{n}| \leq |\lambda_{n}|^{1/2},\; n=1,2,...\,.$
Define a space $\mathfrak{H}_{+}$ as a subset of $\mathfrak{H}$ endowed with a special norm, i.e.
$$
\mathfrak{H}_{+}:=\left\{f\in \mathfrak{H}:\;\|f\|^{2}_{\mathfrak{H}_{+}}:=\sum\limits_{n=1}^{\infty}|\lambda_{n}||f_{n}|^{2}<\infty\right\}.
$$
It is clear that $\mathfrak{H}_{+}$ is dense in $\mathfrak{H},$ since  $\{e_{n}\}_{1}^{\infty}\subset \mathfrak{H}_{+}.$ Let us show that embedding of the spaces $\mathfrak{H}_{+}\subset \mathfrak{H}$ is compact. Consider the operator $G: \mathfrak{H}\rightarrow \mathfrak{H}$ defined as follows
$$
Gf=\sum\limits_{n=1}^{\infty}|\lambda_{n}|^{-1/2}f_{n}e_{n}.
$$
  The fact that  $G$ is a compact operator can be proved easily due to the well-known  criterion of compactness in the Banach space endowed with a basis (we left the prove  to the reader).
Notice that if $f\in \mathfrak{H}_{+},$ then $g\in\mathfrak{H},$  where $g$ is defined by its   Fourier coefficients     $g_{n}=|\lambda_{n}|^{1/2}|f_{n}|.$ By virtue of such a correspondence, we can consider any bounded set in the space $\mathfrak{H}_{+}$ as a bounded set in the space $\mathfrak{H}.$ Applying the operator $G$ to the element $g,$ we get the element $f.$ Since  $G$ is a compact operator, then  we   conclude that the image of the bounded set in the sense of the norm $\mathfrak{H}_{+}$  is a compact set in the sense of the norm $\mathfrak{H}.$
Define the set $\mathfrak{M}$ as a linear manifold generated by  the basis vectors.  Thus, we have obtained the relation $\mathfrak{H}_{+}\subset\subset\mathfrak{H}$ and established the fulfilment of the condition $\mathrm{H}1.$     The first relation of the condition $\mathrm{H}2$ can be obtained  easily due to the application of the Cauchy-Schwarz inequality. To obtain the second one, we should note that  $ \eta_{n} ^{2}+ \mu^{2}_{n} \leq |\lambda_{n}| +\mu^{2}_{n};\, \mu _{n} \geq |\lambda_{n}| (1-|\lambda_{n}|^{-1})^{1/2}\geq M|\lambda_{n}|,\, M=(1-|\mu_{1}|^{-1})^{1/2},\, n=1,2,...\,.$   Therefore
$$
\mathrm{Re}(Nf,f)_{\mathfrak{H}}=\sum\limits_{n=1}^{\infty}\mathrm{Re} \lambda_{n} |f_{n}|^{2} \geq  M  \sum\limits_{n=1}^{\infty}  |\lambda_{n}| |f_{n}|^{2}= M \|f\|^{2}_{\mathfrak{H}_{+}}.
$$
Now, we  conclude that hypotheses $\mathrm{H}1, \mathrm{H}2$ hold. Let us show that the condition $\mathrm{H}3$ is satisfied, we have
$$
\left|\mathrm{Im}(Nf,g)_{\mathfrak{H}}\right|=\left|\sum\limits_{n=1}^{\infty}\mathrm{Im} \lambda_{n}  f_{n} \overline{g}_{n} \right| \leq
 \sum\limits_{n=1}^{\infty}  |\lambda_{n}|^{1/2} | f_{n}|  |g_{n}|  \leq \|f\|_{\mathfrak{H}_{+}}\|g\|_{\mathfrak{H}},
$$
thus we have obtained the desired result.  Consider an abstract  Cauchy problem
\begin{equation}\label{6.26}
\mathfrak{D}^{1/\alpha}_{-}u(t)-N^{1/\kappa}u(t)=0,\;u(0)=f\in \mathrm{D}(N),\;\kappa= 1/2,1/3,...\,,\;\alpha\geq 1,
\end{equation}
where $f$ is supposed to be an  arbitrary element, if the operator $\mathfrak{D}^{1-1/\alpha}_{-}N^{1/\kappa}$ is accretive.
In accordance with       Theorem \ref{T6.3}, Theorem \ref{T6.1} (see in original version Theorem 3 \cite{firstab_lit:2kukushkin2022}), we conclude that there exists a solution of  problem \eqref{6.26} represented by  series \eqref{6.5}. Assume that $\alpha=1$ and  let us establish conditions under which being imposed $N^{1/\kappa}$ is an accretive operator.
Note that using the  relation between  the real and imaginary parts of an eigenvalue, we get
$$
\mu _{n} \geq  (1-|\lambda_{n}|^{-1})^{1/2}|\lambda_{n}|\geq (1-|\lambda_{n}|^{-1})^{1/2} \eta^{2}_{n}\geq (1-|\mu_{1}|^{-1})^{1/2} \eta^{2}_{n}>(1-e^{-\kappa})^{1/2} \eta^{2}_{n}.
$$
Hence, the eigenvalues lie in a parabolic domain $\mathfrak{W}:=\{z\in \mathbb{C},\,\mathrm{Re} z>(1-e^{-\kappa})^{1/2} \mathrm{Im}^{2}z\}.$ Notice that a condition
 \begin{equation}\label{6.27}
|\mathrm{arg}\lambda_{n}|\leq \pi \kappa/2,\,n=1,2,...\,
\end{equation}
   guarantees the following fact
$$
\mathrm{Re }\lambda_{n}^{1/\kappa}=|\lambda_{n}|^{1/\kappa} \cos   \left(\frac{\mathrm{arg}\lambda_{n}}{\kappa}   \right) \geq 0,\,n=1,2,...\,
$$
from what follows the desired accretive property
\begin{equation}\label{6.28}
\mathrm{Re}(N^{1/\kappa}f,f)_{\mathfrak{H}}= \sum\limits_{n=1}^{\infty}\mathrm{Re}\, \lambda^{1/\kappa}_{n} |f_{n}|^{2}\geq 0.
\end{equation}
 It is clear that the   condition \eqref{6.27} holds for the eigenvalues with  sufficiently large numbers of  indexes, since they lie in the parabolic domain $\mathfrak{W}.$ Thus, we see that    a finite number of eigenvalues do not satisfy the condition  \eqref{6.27}. Here, we should note that using simple reasonings (we left them to the reader), we can find the initial number, starting  from which   condition \eqref{6.27} holds. Now, it is clear that if we additionally assume that $|\eta_{n}|/\mu_{n}\leq \tan(\pi \kappa/2),\,n=1,2,...,N,$ where $N$ is a sufficiently large number, then we obtain  \eqref{6.28}. Therefore, in accordance with Theorem \ref{T6.1},  we are capable to  extend the initial condition assuming  that $f\in \mathfrak{H}$ and claim that the existing solution is unique.
      The constructed normal operator indicates the significance  of the made  in Theorem \ref{T4.7} (the original version is represented by  Theorem 3 \cite{firstab_lit:2kukushkin2022})  clarification  of the results \cite{firstab_lit:1Lidskii}.

\section{Review}

In this section, we invented a method to study a Cauchy problem for the  abstract fractional evolution equation  with the operator function in the second term. The considered class corresponding to the operator-argument is rather wide and includes non-selfadjoint unbounded operators.
  As  a main result we  represent  a technique   allowing  to principally weaken conditions imposed upon the second term not containing the time variable. Obviously, the application section of the paper is devoted  to the theory of fractional differential equations.

  The invented method  allows to solve the Cauchy problem for the  abstract fractional evolution equation   analytically what is undoubtedly a great  advantage,
   we used  results of  the spectral theory of non-selfadjoint operators as a base for studying the mathematical objects. Characteristically, that the operator function is defined on a special operator class covering the infinitesimal generator transform (see \cite{kukushkin2021a}), where a corresponding semigroup is assumed  to be a strongly  continuous  semigroup of contractions. The corresponding particular cases  leads us  to a linear composition of differential operators of real order in various senses listed in the introduction section. In connection with this, various types of  fractional integro-differential operators  can be considered what  becomes clear if we involve an operator function represented by the Laurent series with finite principal and regular parts. Moreover, the artificially constructed normal operator \cite{firstab_lit:2kukushkin2022} belonging to the special operator  class indicates that the application part is beyond the class of differential operators of  real order.
Bellow, we represent a comparison analysis to show brightly the main contribution of the paper, particularly  the newly invented method allowing us to consider  an entire function as the operator function. First of all the technique related to the proof of the contour integral convergence is similar to the papers
 \cite{firstab_lit:1Lidskii},\cite{firstab_lit:1kukushkin2021},\cite{firstab_lit:2kukushkin2022},\cite{firstab_lit(axi2022)} one can italicize a similar  scheme of reasonings, but the last one is nothing without the required properties of the considered entire function. Such theorems as the Wieman  theorem, the theorem on the entire function  growth regularity and their applications form the main author's creative  contribution to the paper. To be honest, it was not so easy to find such a condition that makes the contour integral be  convergent on the entire function, we should note that the latter idea in its precise statement  has not been   considered previously. The following fact is also  worth noting -  a suitable algebraic reasonings having noticed by the author and allowing us to involve   a fractional derivative in the first term. This idea  allows to cover many results in the framework of the theory of fractional differential equations. The latter   what is if not  a relevant result!
As for other mathematicians, here the Lidskii name ought to be sounded, however   the peculiarities of  the author's own technique have been shown and discussed  in the papers \cite{firstab_lit(arXiv non-self)kukushkin2018}, \cite{kukushkin2021a}, \cite{firstab_lit:1kukushkin2021} one  can study them properly. We may say that the main concept of the root vector series expansion  jointly  with the method analogous to the Abel's one belong   to Lidskii what is reflected in the name -- Abel-Lidskii sense of the series   convergence. As for the author's contribution to this method it is not so small as one  can observe in the paper \cite{firstab_lit:1kukushkin2021} for  the main  result establishes  clarification of the results by Lidskii. In the framework of the discussion  the following papers by Markus \cite{firstab_lit:2Markus}, Matsaev \cite{firstab_lit:Markus Matsaev}, Shkalikov \cite{firstab_lit:Shkalikov A.} can be undergone to a comparison analysis. The latter represents in the paper \cite{firstab_lit:Shkalikov A.} only  an idea of the proof of Theorem 5.1  even  the statement  of which differs from the statement of Theorem 4 \cite{firstab_lit:1kukushkin2021}
which is provided  with a detailed proof and clarifies  the  Lidskii results  represented in \cite{firstab_lit:1kukushkin2021}. A particular attention can be drown to a special class of operators with   which, due to the author's results  \cite{kukushkin2021a}, the reader  can  successfully  deal. The latter benefit stresses relevance of the results for initially the theoretical results in the framework of the developed  direction of the  spectral theory \cite{firstab_lit(arXiv non-self)kukushkin2018},\cite{kukushkin2021a} originated from the ones \cite{kukushkin2019} devoted to uniformly elliptic non-selfadjoint operators which  can not be covered by the results by   Markus \cite{firstab_lit:2Markus}, Matsaev \cite{firstab_lit:Markus Matsaev}, Shkalikov \cite{firstab_lit:Shkalikov A.}  due to the  absence   of a so-called    complete subordination condition  imposed  upon the operator (a corresponding example is given in the paper\cite{firstab_lit(arXiv non-self)kukushkin2018}).\\

\noindent {\bf Acknowledgment}\\

\noindent The author is sincerely grateful to Professor Boris G. Vakulov for invaluable discussions.


\begin{thebibliography}{99}






















\bibitem{firstab_lit:1Adams} {\sc Adams   R.A.} Compact lmbeddings of Weighted Sobolev Spaces
on Unbounded Domains.
\textit{Journal of Differential Equations}, \textbf{9}   (1971), 325--334.




\bibitem{firstab_lit:2Agmon} {\sc Agmon S.} Lectures on elliptic boundary value problems.\textit{Van Nostrand Math. Stud.,
2, D. Van Nostrand Co., Inc., Princeton, NJ-Toronto-London}, 1965.

\bibitem{firstab_lit:1Agranovich2013} {\sc  Agranovich M.S.}  Sobolev spaces,
their generalization
and elliptic problems
in domains with the smooth
and Lipschitz boundary.  \textit{Moscow: MCNMO}, 2013.





\bibitem{firstab_lit:2Agranovich1994} {\sc  Agranovich M.S.} On series with respect to root vectors of operators associated with forms having symmetric principal part.   \textit{ Functional Analysis and its applications},   \textbf{28} (1994),   151--167.


\bibitem{firstab_lit:2Agranovich2011} {\sc  Agranovich M.S.}  Spectral problems in Lipshitz mapping areas.
 \textit{Modern mathematics, Fundamental direction},   \textbf{39} (2011), 11--35.



\bibitem{firstab_lit:3Agranovich1999} {\sc  Agranovich M.S.,  Katsenelenbaum B.Z., Sivov A.N., Voitovich N.N.}
Generalized method of eigenoscillations in the diffraction theory. \textit{Zbl 0929.65097
Weinheim: Wiley-VCH.}, 1999.


\bibitem{firstab_lit: Ahiezer1966} {\sc  Ahiezer N.I. And. Glazman I.M.} Theory of linear operators in Hilbert space.
\textit{Moscow: Nauka, Fizmatlit},  1966.












\bibitem{firstab_lit:Andronova2017} {\sc   Andronova O.A.,  Voytitsky V.I.} On spectral properties of one boundary value
problem with a surface energy dissipation.   \textit{Ufa Mathematical Journal},   \textbf{9}, No.2  (2017),   3--16.






































\bibitem{firstab_lit:Aronszajne}
 {\sc Aronszajne  N., Smith K.} Invariant subspaces of completely continuos  operators. \textit{  An.  Math.},    \textbf{60} (1954),     345--350.












\bibitem{firstab_lit:1Ashyralyev} {\sc  Ashyralyev   A.} A note on fractional derivatives and fractional powers of operators.
\textit{J. Math. Anal. Appl.},      \textbf{357}   (2009), 232--236.










\bibitem{firstab_lit:Bazhl} {\sc Bazhlekova E.} The abstract Cauchy problem for the fractional evolution
equation.
\textit{Fractional Calculus and Applied Analysis} (1998), \textbf{1}, No.3,  255--270.







\bibitem{firstab_lit: Berezansk1968}
{\sc Berezanskii Yu.M.}
    Expansions in eigenfunctions of selfadjoint operators.
   \textit{Providenve, Rhode Island : American Mathematical Society. Translations of mathematical monographs volume 17},   1968.









\bibitem{firstab_lit:Bernstain}
 {\sc Bernstain   A., Robinson A.} Solutions of invariant subspace problem of K.T. Smith and P.R. Halmos. \textit{  Pacif. J.  Math.},    \textbf{16} (1966),     421--431.



\bibitem{firstab_lit:1Browder} {\sc Browder F.E.}   On the eigenfunctions and eigenvalues of the general linear elliptic
differential operator. \textit{Proc. Nat. Acad. Sci. U.S.A.},   \textbf{39}   (1953), 433--439.


\bibitem{firstab_lit:2Browder}
{\sc Browder F.E.} On the spectral theory of strongly elliptic differential operators.
   \textit{Proc. Nat. Acad. Sci. U.S.A.},   \textbf{45} (1959), 1423--1431.



\bibitem{firstab_lit:Brodsky} {\sc Brodsky M.S.} On one problem of  I.M. Gelfand. \textit{Russian Mathematical Surveys},  \textbf{12} (1957),  129--132.






\bibitem{firstab_lit:1Carleman} {\sc Carleman T.} EUber die asymptotische Verteilung der Eigenwerte partieller Differentialgleichungen.  \textit{Ber. Verh. SEachs. Akad. Leipzig}, \textbf{88} (1936), no. 3, 119--132.










\bibitem{Ph. Cl} {\sc Clement  Ph., Gripenberg  G.,  Londen  S.O.} Holder regularity
for a linear fractional evolution equation.
\textit{Topics in Nonlinear Analysis: the Herbert
Amann Anniversary Volume, Birkh\"{a}user, Basel} (1998),  69--82.



\bibitem{firstab_lit: Courant-Hilbert}{\sc Courant R., Hilbert D.}  Methods of mathematical physics.\textit{Gostekhizdat, Moscow}   1951.







 \bibitem{firstab_lit:Danford}
 {\sc Danford N.} Linear operators, part II.
  \textit{Moscow, Mir}, 1966.


\bibitem{firstab_lit:2Dim-Kir} {\sc Dimovski I.H., Kiryakova V.S.} Transmutations, convolutions and fractional powers of Bessel-type operators via Maijer's  G-function. In:
\textit{Proc."Complex Anall. and Appl-s, Varna' 1983"}, (1985),  45--66.




\bibitem{firstab_lit:Donoghue}
 {\sc Donoghue W.} The lattice of invariant subspaces of completely continous quasi-nilpotent transformations. \textit{  Pacif. J.  Math.},    \textbf{7} (1957),     1031--1035.







\bibitem{firstab_lit:15Erdelyi} {\sc Erdelyi A.} Fractional integrals of generalized functions.
\textit{J. Austral. Math. Soc.},  \textbf{14}, No.1  (1972),  30--37.





\bibitem{firstab_lit:fedosov1964}  {\sc  Fedosov B.V.}  Asymptotic formulas for
eigenvalues of the Laplace operator in the case of a polyhedron. \textit{Reports of the Academy of Sciences of the USSR}, \textbf{157} (1964), No. 3, 536--538.


\bibitem{firstab_lit:1arx Ferrari}  {\sc   Ferrari F.}  Weyl and Marchaud derivatives: A forgotten history.  \textit{arXiv.org, arXiv:1711.08070v2}, [math.AP], 2017.


\bibitem{firstab_lit:fridrichs1958}
 {\sc  Friedrichs K.}  Symmetric positive linear differential equations.  \textit{Comm. Pure Appl. Math.},   \textbf{11} (1958),   238--241.





\bibitem{firstab_lit:1Glazman} {\sc  Glazman I.M.}  On   decomposability  in a system of eigenvectors of dissipative operators.  \textit{ Russian Mathematical Surveys}, \textbf{13} (2016),  Issue  3 (81),  179--181.




\bibitem{firstab_lit:gilbarg} {\sc  Gilbarg D., Trudinger N.S.,}   Eliptic  partial differential equations of second order.  \textit{Second edition,  Springer-Verlag Berlin, Heidelberg, New York, Tokyo}, 1983.




\bibitem{firstab_lit:1Gohberg1965} {\sc   Gohberg I.C., Krein M.G.}  Introduction to the theory of linear non-selfadjoint operators in a Hilbert space.
 \textit{Moscow: Nauka, Fizmatlit},  1965.



\bibitem{firstab_lit:Goldstein} {\sc   Goldstein G.R.,  Goldstein J.A.,  Rhandi A.}  Weighted Hardys inequality
and the Kolmogorov equation perturbed by an inverse-square potential.  \textit{Applicable Analysis}, \textbf{91}, Issue: 11 (2012), DOI:10.1080/00036811.2011.587809.








\bibitem{firstab_lit:Halmos}
 {\sc Halmos P.} Invariant subspaces of polynomially compact operators. \textit{  Pacif. J.  Math.},    \textbf{16} (1966),     433--437.

\bibitem{firstab_lit:Halmos 1}
 {\sc Halmos P.} A Hilbert space problem book.
  \textit{D. Van Nostrand Company, Inc. Princenton, New Jersey, Toronto, London}, 1967.






\bibitem{firstab_lit:Hardy}
 {\sc Hardy G.H.} Divergent series.
  \textit{Oxford University Press, Ely House, London W.}, 1949.




\bibitem{firstab_lit:Haq  2022} {\sc Haq A.} Partial-approximate controllability of semi-linear systems involving two Riemann-Liouville fractional derivatives.
 \textit{Chaos, Solitons and Fractals}   (2022), \textbf{157}, 111923. https://doi.org/10.1016/j.chaos.2022.111923.



\bibitem{firstab_lit:Helson}
 {\sc Helson H.} Lectures on invariant subspaces.
  \textit{Ney York}, 1964.

\bibitem{firstab_lit:H o r n}
{\sc Horn A. }  On the singular values of a product of completely continuous ope
rators. \textit{Proc. Nat. Acad. Sci.},  \textbf{36} (1950),  374--375.



\bibitem{firstab_lit:Sukavanam 2020} {\sc Haq A., Sukavanam N.} Existence and approximate controllability of Riemann-Liouville fractional integrodifferential systems with damping.
 \textit{Chaos, Solitons and Fractals}   (2020), \textbf{139}, 110043.




































\bibitem{firstab_lit:Kalisch}
 {\sc Kalisch G.} On symilarity, reducing manifolds, and unitary equivalence of certain Volterra operators. \textit{  An.  Math.},    \textbf{66} (1957),     481--494.












\bibitem{firstab_lit: Karapetyants N. K. Rubin B. S. 1} {\sc Karapetyants N.K. Rubin B.S.} On the fractional integral operators   in the
 weighted spaces.
\textit{News of Armenian SSR Academy of Sciences Math.},  \textbf{19}, No.1 (1984),  31--43.

\bibitem{firstab_lit: Karapetyants N. K. Rubin B. S. 2} {\sc Karapetyants N.K. Rubin B.S.} Radial Riess potential on the disk  and the fractional integral  operators.
\textit{Reports of the USSR Academy of Sciences},  \textbf{263}, No.6 (1982),  1299--1302.


\bibitem{firstab_lit:1.1kato1961}  {\sc Kato T.}   Fractional powers of dissipative operators.
\textit{J.Math.Soc.Japan}, \textbf{13}, No.3 (1961), 246--274.


\bibitem{firstab_lit:kato1966}{\sc Kato T.} Perturbation theory for linear operators. \textit{Springer-Verlag Berlin, Heidelberg, New York}, 1966.

\bibitem{firstab_lit:kato1980}{\sc Kato T.} Perturbation theory for linear operators. \textit{Springer-Verlag Berlin, Heidelberg, New York}, 1980.

\bibitem{firstab_lit:1Katsnelson} {\sc Katsnelson V.E.} Conditions under which systems of eigenvectors of some classes
of operators form a basis.
\textit{ Funct. Anal. Appl.}, \textbf{1}, No.2   (1967), 122--132.



\bibitem{firstab_lit:1keldysh 1951} {\sc  Keldysh M.V.}  On   eigenvalues and eigenfunctions of some
classes of non-selfadjoint equations. \textit{Reports of the Academy of Sciences of the USSR},  \textbf{77} (1951), no. 1,  11--14.

\bibitem{firstab_lit:2keldysh 1973} {\sc  Keldysh M.V.}   On a Tauberian
theorem. \textit{Amer. Math. Soc. Transl. Ser., Amer. Math. Soc., Providence, RI},  \textbf{102} (1973), no. 2,  133--143.



\bibitem{firstab_lit:3keldysh 1971} {\sc  Keldysh M.V.}   On  completeness of   eigenfunctions of some classes of
non-selfadjoint linear operators. \textit{Russian Math. Surveys},  \textbf{26} (1971), no. 4,  15--44.


\bibitem{firstab_lit:2.1kipriyanov1960}
{\sc Kipriyanov I.A. }    On some properties of fractional derivative in the direction.   \textit{ Proceedings of the universities. Math., USSR},  No.2     (1960),  32--40.




 \bibitem{firstab_lit:kipriyanov1960}
{\sc Kipriyanov I.A. } On spaces of fractionally differentiable functions. \textit{ Proceedings of the Academy of Sciences. USSR},  \textbf{24} (1960),  665--882.

\bibitem{firstab_lit:1kipriyanov1960}
{\sc Kipriyanov I.A. }  The operator of fractional differentiation and powers of the elliptic operators. \textit{  Proceedings of the Academy of Sciences. USSR},    \textbf{131} (1960),     238--241.







\bibitem{firstab_lit:2.2kipriyanov1960} {\sc Kipriyanov I.A. } On   complete continuity of embedding operators  in   spaces of fractionally differentiable functions. \textit{Russian Mathematical Surveys},   \textbf{17}   (1962), 183--189.













\bibitem{firstab_lit: Kolmogoroff  A. N.} {\sc Kolmogoroff  A.N.}  EUber Kompaktheit der Funktionenmengen bei der Konvergenz
im Mittel.
\textit{Nachr. Ges. Wiss. GEottingen},  \textbf{9}    (1931),  60--63.





































\bibitem{firstab_lit:Krasnoselskii M.A.}{\sc Krasnoselskii M.A.,  Zabreiko P.P.,  Pustylnik E.I.,  Sobolevskii P.E.}
  Integral operators in the spaces of summable functions. \textit{ Moscow:  Science,   FIZMATLIT},   1966.


\bibitem{firstab_lit:1Krein} {\sc  Krein M.G.}  Criteria for
completeness of the system of root vectors of a dissipative operator.
 \textit{Amer. Math. Soc. Transl. Ser., Amer. Math. Soc., Providence, RI}, \textbf{26}, No.2 (1963), 221--229.












\bibitem{firstab_lit:(vlad)kukushkin2016} {\sc  Kukushkin M.V.}  On  some inequalities for norms  in    weighted Lebesgue spaces.  \textit{Review of science, the south of Russia, Math. forum, Vladikavkaz},  \textbf{10} (2016), part 1,  181--186.






\bibitem{firstab_lit:1.1kukushkin2017} {\sc  Kukushkin M.V.}   On some qulitative properties of the   Kipriyanov  fractional differential operator.
  \textit{Vestnik of Samara  University, Natural Science Series, Math.},    \textbf{23} (2017),   no. 2,   32--43.








  \bibitem{firstab_lit:2kukushkin2017} {\sc  Kukushkin M.V.} Evaluation of the eigenvalues of the Sturm-Liouville problem for a differential operator with fractional derivative in the lower terms.
  \textit{Belgorod State University Scientific Bulletin, Math. Physics},   46, No.6,  29--35,  2017.

  \bibitem{firstab_lit:2kukushkin2017} {\sc  Kukushkin M.V.} On some qulitative properties of the operator fractional differentiation in Kipriyanov sense.
 \textit{Vestnik of Samara  University, Natural Science Series, Math.},   23, No.2,  32--43,   2017.










\bibitem{firstab_lit(JFCA)kukushkin2018} {\sc  Kukushkin M.V.}  Theorem of existence and uniqueness of  a solution for a  differential equation of   fractional order.  \textit{Journal of Fractional Calculus and Applications},  \textbf{9} (2018), no. 2, 220--228.







\bibitem{firstab_lit:1kukushkin2018} {\sc  Kukushkin M.V.}  Spectral properties of  fractional differentiation operators.  \textit{Electronic Journal of Differential Equations},  \textbf{2018} (2018), No. 29,  1--24.










\bibitem{kukushkin2019}{\sc Kukushkin M.V.} Asymptotics of eigenvalues for differential
operators of fractional order.
 \textit{Fract. Calc. Appl. Anal.}, \textbf{22}, No. 3 (2019), 658-681, arXiv:1804.10840v2 [math.FA]; DOI:10.1515/fca-2019-0037; at https://www.degruyter.com/view/j/fca.


\bibitem{firstab_lit(arXiv non-self)kukushkin2018}   {\sc  Kukushkin M.V.} On One Method of Studying Spectral Properties of Non-selfadjoint Operators. \textit{Abstract and Applied Analysis; Hindawi: London, UK}, \textbf{2020},  (2020); at  https://doi.org/10.1155/2020/1461647.

\bibitem{kukushkin2021a}{\sc Kukushkin M.V.} Abstract fractional calculus for m-accretive operators.
 \textit{International Journal of Applied Mathematics.}, \textbf{34}, Issue: 1 (2021),  DOI: 10.12732/ijam.v34i1.1


\bibitem{firstab_lit:1kukushkin2021}   {\sc  Kukushkin M.V.} Natural lacunae method and Schatten-von Neumann classes of the convergence exponent. \textit{Mathematics},  (2022), \textbf{10}, (13), 2237; https://doi.org/10.3390/math10132237.





\bibitem{firstab_lit:2kukushkin2022}   {\sc Kukushkin M.V.} Evolution Equations in Hilbert Spaces via the Lacunae Method. \textit{Fractal Fract.},  (2022), \textbf{6}, (5), 229; https://doi.org/10.3390/fractalfract6050229.




\bibitem{firstab_lit(axi2022)}   {\sc Kukushkin M.V.} Abstract Evolution Equations with an Operator Function in the Second Term. \textit{Axioms}, (2022), 11, 434; at  https://doi.org/10.3390/axioms
11090434.

\bibitem{firstab_lit KRAUNC}   {\sc Kukushkin M.V.} Note on the spectral theorem for unbounded non-selfadjoint
operators. \textit{Vestnik KRAUNC. Fiz.-mat. nauki.}, (2022), 139, 2, 44--63. DOI: 10.26117/2079-
6641-2022-39-2-44-63




\bibitem{firstab_lit(frac2023)}{\sc  Kukushkin M.V.}  Cauchy Problem for an Abstract Evolution Equation of Fractional Order. \textit{Fractal Fract.}, (2023), 7, 111; at   https://doi.org/10.3390/fractalfract7020111.







\bibitem{firstab_lit:Fan} {\sc  Ky Fan}  Maximum properties and inequalities for the eigenvalues of completely continuous operators.
 \textit{Proc. Nat. Acad. Sci. USA}, \textbf{37}, (1951), 760--766.










\bibitem{firstab_lit:Eb. Levin} {\sc  Levin  B.Ja.}  Distribution of Zeros of Entire Functions.
 \textit{Translations of Mathematical Monographs, American Mathematical Soc, Providence, Rhode Island}, \textbf{5},   1964.

\bibitem{firstab_lit:Eb. LevinE} {\sc  Levin  B.Ja.}  Lectures on  Entire Functions.
 \textit{Translations of Mathematical Monographs, American Mathematical Soc, Providence, Rhode Island},  \textbf{150},  1991.

\bibitem{firstab_lit:1Lidskii} {\sc  Lidskii V.B.}  Conditions for completeness of a system of
root subspaces for non-selfadjoint operators with discrete spectra.  \textit{Amer. Math. Soc. Transl. Ser., Amer. Math. Soc., Providence, RI.}, \textbf{34} (1963), No. 2,  241--281.


\bibitem{firstab_lit:1Lidskii} {\sc  Lidskii V.B.} Summability of series in terms of the principal vectors of non-selfadjoint operators.
 \textit{Tr. Mosk. Mat. Obs.}, \textbf{11},  (1962), 3--35.




\bibitem{firstab_lit:1Livshits}{\sc  Livshits M.S.}   On spectral decomposition of linear non-selfadjoint operators. \textit{Amer. Math. Soc. Transl. Ser., Amer. Math. Soc., Providence, RI.}, \textbf{5} (1957), No. 2,  67--114.









\bibitem{firstab_lit:5Love} {\sc Love E.R.} Two index laws for fractional integrals and derivatives.
\textit{J. Austral. Math. Soc.},  \textbf{14}, No.4  (1972),  385--410.















\bibitem{firstab_lit:Lomonosov}
 {\sc Lomonosov V.I.} On invariant subspaces of    operators family  commuting with compact operator. \textit{ Functional analysis and its aplications},    \textbf{7}, No.3  (1973),     55--56.



\bibitem{firstab_lit:Mainardi F.} {\sc Mainardi F. } The fundamental solutions for the fractional diffusion-wave equation.
\textit{Appl. Math. Lett.}, \textbf{9}, No.6  (1966),  23--28.






\bibitem{firstab_lit:Mamchuev2017} {\sc Mamchuev M.O. } Solutions of the main boundary value problems for the time-fractional telegraph equation by the Green function method.
\textit{Fractional Calculus and Applied Analysis}, \textbf{20}, No.1  (2017),  190--211,   DOI: 10.1515/fca-2017-0010.

\bibitem{firstab_lit:Mamchuev2017a} {\sc Mamchuev M.O.} Boundary value problem for the time-fractional telegraph equation with Caputo derivatives Mathematical Modelling of Natural Phenomena.
\textit{Special functions and analysis of PDEs},  (2017), \textbf{12}, No.3,  82--94. DOI: 10.1051/mmnp/201712303.









\bibitem{firstab_lit: Marz} {\sc  Marcinkiewicz J., Zygmund A.} Some theorems on orthogonal systems.
\textit{Fundamenta Mathematicae},  \textbf{28}, No.1 (1937),   309--335.



\bibitem{firstab_lit:1Markus}  {\sc Markus A.S.}   On   the  basis of root vectors
of a dissipative operator. \textit{Soviet Math. Dokl.},  \textbf{1}   (1960), 599--602.



\bibitem{firstab_lit:2Markus} {\sc Markus A.S.} Expansion in root vectors of a slightly perturbed selfadjoint operator.
\textit{ Soviet Math. Dokl.}, \textbf{3}   (1962), 104--108.






\bibitem{firstab_lit:Markus Matsaev} {\sc Markus A.S.,  Matsaev V.I.} Operators generated by sesquilinear forms and their spectral asymptotics.
\textit{ Linear operators and integral equations, Mat. Issled., Stiintsa, Kishinev}, \textbf{61}   (1981), 86--103.



















\bibitem{firstab_lit:1Matsaev}  {\sc  Matsaev V.I.}   On a class of completely
continuous operators. \textit{Soviet Math. Dokl.},  \textbf{2}   (1961), 972--975.



\bibitem{firstab_lit:3Matsaev}  {\sc    Matsaev V.I.}  A method for   estimation of   resolvents of non-selfadjoint operators. \textit{Soviet Math. Dokl.},  \textbf{5}   (1964), 236--240.

























\bibitem{firstab_lit:9McBride} {\sc McBride A.} A note of the index laws of fractional calculus.
\textit{J. Austral. Math. Soc.}, A \textbf{34}, No.3  (1983),  356--363.

\bibitem{firstab_lit:mihlin1970}{\sc Mihlin S.G.} Variational methods in mathematical physics. \textit{Moscow Science}, 1970.







\bibitem{L. Mor} {\sc  Moroz L., Maslovskaya A.G.} Hybrid stochastic fractal-based approach to modeling the switching kinetics of ferroelectrics in the injection mode.
\textit{ Mathematical Models and Computer Simulations}, \textbf{12}   (2020), 348--356.




\bibitem{firstab_lit:Motovilov} {\sc   Motovilov A.K.,  Shkalikov A.A.}  Preserving of the unconditional basis property under non-self-adjoint perturbations of self-adjoint operators.
 \textit{Funktsional. Anal. i Prilozhen.}, \textbf{53}, Issue 3  (2019),  45--60 (Mi faa3632).































\bibitem{firstab_lit:Muckenhoupt}
 {\sc  Muckenhoupt B.} Mean Convergence of Jacobi Series. \textit{ Proceedings of the American Mathematical Society},    \textbf{23}, No.2  (Nov., 1969),      306--310.




\bibitem{firstab_lit:1Mukminov} {\sc Mukminov B.R.}  Expansion in eigenfunctions of dissipative kernels. \textit{Proceedings of the Academy of Sciences. USSR}, \textbf{99} (1954), No. 4,  499--502.








\bibitem{firstab_lit:Nagy}
 {\sc Nagy B.Sz., Foias C. }  Sur les contractions de l'espace de Hilbert.  \textit{ V Translations bilaterales. Acta Szeged},    \textbf{23} (1962),     106--129.








\bibitem{firstab_lit:1.1Nakhushev1998}
{\sc Nakhushev A.M.}  On the positiveness  property of operators of continuous and discrete differentiation and integration, are
very important in fractional calculus and theory of equations of mixed type. \textit{Differential equations}, \textbf{34},   No.1  (1998),
101-109.




\bibitem{firstab_lit:1Nakhushev1977} {\sc Nakhushev A.M.}  The Sturm-Liouville problem for an ordinary differential equation of the  second order with  fractional derivatives in  lower terms.
 \textit{Proceedings of the Academy of Sciences. USSR},  \textbf{234}, No.2 (1977),  308--311.


\bibitem{firstab_lit:nakh2003}
 {\sc Nakhushev A.M.} Fractional calculus and its application.
  \textit{M.: Fizmatlit}, 2003



\bibitem{firstab_lit:H. Newman} {\sc Newman J.,  Rudin W.} Mean convergence of orthogonal series.
\textit{ Proc. Amer. Math.
Soc.},  \textbf{3}, No.2 (1952),  219--222.





\bibitem{firstab_lit:1Okazawa} {\sc Okazawa N.}    Remarks on linear m-accretive operators in a Hilbert space. \textit{ J. Math. Soc. Japan}, \textbf{27} (1975), No. 1,  160--165.

\bibitem{firstab_lit:2Okazawa} {\sc Okazawa N.} Singular perturbations of m-accretive operators.  \textit{J. Math. Soc. Japan}, \textbf{32} (1980), No. 1,  19--44.

\bibitem{firstab_lit:3Okazawa} {\sc Okazawa N.} On the perturbation of linear operators in Banach and Hilbert spaces. \textit{J. Math. Soc. Japan}, \textbf{34} (1982), No. 4,  677--701.





\bibitem{Pasy} {\sc Pazy A.} Semigroups of Linear Operators and Applications to Partial Differential Equations.
  \textit{Berlin-Heidelberg-New York-Tokyo, Springer-Verlag (Applied Mathematical Sciences V. 44),}  1983.





\bibitem{firstab_lit:J.Peetre}{\sc Peetre J.}On the theory of $\mathcal{L}_{p,\lambda}$ spaces.
\textit{J.Funct.Analysis},  \textbf{4}, No.1 (1969),  71--87.


\bibitem{firstab_lit:H. Pollard 1} {\sc Pollard H.} The mean convergence of orthogonal series I.
\textit{Trans. Amer. Math. Soc.},  \textbf{62}, No.3 (1947),  387--403.

\bibitem{firstab_lit:H. Pollard 2} {\sc Pollard H.} The mean convergence of orthogonal series II.
\textit{Trans. Amer. Math. Soc.},  \textbf{63}, No.2 (1948),  355--367.

\bibitem{firstab_lit:H. Pollard 3} {\sc Pollard H.} The mean convergence of orthogonal series III.
\textit{Duke Math. J.},  \textbf{16}, No.1 (1949),  189--191.









\bibitem{firstab_lit:1Prabhakar} {\sc Prabhakar T.R.} Two singular integral equations involving confluent hypergeometric functions.
\textit{Proc. Cambrige Phil. Soc.},  \textbf{66}, No.1  (1969),  71--89.


















\bibitem{firstab_lit:Pskhu} {Pskhu  A.V.} The fundamental solution of a diffusion-wave equation of fractional order.
\textit{Izvestiya: Mathematics},  \textbf{73}, No.2  (2009),  351--392.



















\bibitem{firstab_lit:Riesz1955} {\sc  Riesz F.,  Nagy B.Sz.} Functional Analysis.
  \textit{Ungar, New York}, 1955.

 \bibitem{firstab_lit:Rosenblum} {\sc  Rozenblyum G.V.,  Solomyak M.Z., Shubin M.A.
} Spectral theory of differential operators.
 \textit{Results of science and technology. Series Modern problems of mathematics
Fundamental directions}, \textbf{64}    (1989),    5--242.





\bibitem{firstab_lit: Rubin} {\sc Rubin B.S.} Fractional integrals in the weighted  Helder spaces  and   operators of the potential type.
\textit{News of Armenian SSR Academy of Sciences Math.},  \textbf{9}, No.4 (1974),  308--324.


\bibitem{firstab_lit: Rubin 2} {\sc Rubin B.S.} One dimensional representation, inverse  and some properties of Riess potentials defined on  radial functions.
\textit{Mathematical Notes},  \textbf{34}, No.4 (1983),  521--523.



\bibitem{firstab_lit: Rubin 1} {\sc Rubin B.S.} Fractional integrals and the Riess potential with the radial density in   spaces with a power weight.
\textit{News of Armenian SSR Academy of Sciences Math.},  \textbf{21}, No.5 (1986),  488.


















\bibitem{firstab_lit:samko1987} {\sc Samko S.G., Kilbas A.A., Marichev O.I.} Fractional Integrals and Derivatives: Theory and Applications.
  \textit{Gordon
and Breach Science Publishers: Philadelphia, PA, USA}, 1993.











\bibitem{firstab_lit: Samko M. Murdaev} {\sc Samko S.G., Murdaev Kh.M.} Weighted Zygmund estimates for fractional differentiation and integration, and their
applications.
\textit{Trudy Matem. Inst. Steklov.},  \textbf{180}   (1987),  197--198. English transl. in Proc. Steklov Inst. Math.  AMS 3, 233--235 (1989).

\bibitem{firstab_lit: Samko Vakulov B.G.} {\sc Samko S.G., Vakulov B.G.} On equivalent norms in fractional order function spaces of continuous functions on the unit sphere.
\textit{Fract. Calc. Appl. Anal.},  \textbf{4}, No.3  (2000),  401--433.




\bibitem{firstab_lit:Shkalikov A.} {\sc  Shkalikov A.A.}  Perturbations of selfadjoint and normal operators with a discrete spectrum.
 \textit{Russian Mathematical Surveys}, \textbf{71}, Issue 5(431) (2016),  113--174.

\bibitem{firstab_lit:Shukla 2020} {\sc Singh A., Shukla A., Vijayakumar V., Udhayakumar R.} Asymptotic stability of fractional order $(1,2]$ stochastic delay differential equations in Banach spaces.
 \textit{Chaos, Solitons and  Fractals}   (2021), \textbf{150}, 111095.https://doi.org/10.1016/j.chaos.2021.111095.



\bibitem{firstab_lit:Smirnov5} {\sc  Smirnov V.I. }  A Course of Higher Mathematics: Integration and Functional Analysis, Volume 5.
 \textit{Pergamon}, 2014.


\bibitem{firstab_lit:Suetin}
 {\sc Suetin P.K.} Classical orthogonal polynomials.
  \textit{Moscow: Nauka, Fizmatlit}, 1979.






 \bibitem{firstab_lit:Sukavanam 2017} {\sc Shukla A., Sukavanam N., Pandey D.N.}  Approximate controllability of semilinear fractional control systems of order $\alpha\in(1,2].$
 \textit{J Dyn Control Syst}   (2017), \textbf{23}, 679--691. https://doi.org/10.1007/s10883-016-9350-7.


\bibitem{firstab_lit:Sukavanam 2018} {\sc Shukla A., Sukavanam N., Pandey D.N.} Approximate controllability of semilinear fractional stochastic control system.
 \textit{Asian-European Journal of Mathematics}   (2018), \textbf{11}, No.6, 1850088. https://doi.org/10.1142/S1793557118500882.

















\bibitem{firstab_lit:1Tamarkin} {\sc Tamarkin I.D.}   On some general problems of the theory of ordinary linear
differential equations and on  expansion of arbitrary functions in  series.   \textit{Type. M. P. Frolova, Petrograd}, 1917.




\bibitem{firstab_lit:Tricomi} {\sc Tricomi F.G.} Integral equations.
  \textit{Interscience Publishers, Inc., New York}, 1957.

\bibitem{firstab_lit: Vaculov}{\sc Vaculov B.G., Samko N.} Spherical fractional and hypersingular integrals of variable order
in generalized HEolder spaces with variable characteristic.
\textit{Math. Nachr.},  \textbf{284}, Issue 2-3 (2011),  355--369.

\bibitem{firstab_lit:Wyss} {\sc  Wyss W.}  The fractional diffusion equation.
 \textit{J. Math. Phys.}, \textbf{27},  No.11  (1986),  2782--2785.

\bibitem{firstab_lit:Yosida}{\sc Yosida K.}
  Functional analysis, sixth edition. \textit{Springer-Verlag Berlin, Heidelberg, New York},   1980.

\bibitem{firstab_lit:Eb. Zeidler} {\sc Zeidler E.}  Applied functional analysis, applications to mathematical physics.
 \textit{Applied mathematical sciences 108, Springer-Verlag, New York},      1995.


 



 \end{thebibliography}
\end{document}